\documentclass[a4paper,11pt,reqno]{amsart}

\usepackage{lmodern}

\usepackage[utf8]{inputenc} 
\usepackage[T1]{fontenc}

\usepackage[english]{babel}

\usepackage{amsfonts}
\usepackage{amssymb}
\usepackage{amscd}
\usepackage{amsthm}
\usepackage{a4wide}
\usepackage{mathrsfs}
\usepackage{amsmath}
\usepackage{amssymb}
\usepackage{amsthm}
\usepackage{textcomp}
\usepackage{graphicx}
\usepackage{enumerate}
\usepackage{mathrsfs}
\usepackage{frcursive}
\usepackage{tikz}
\usepackage[cyr]{aeguill}
\usepackage{xspace}
\usepackage{hyperref}
\usepackage{appendix}
\usepackage{esint}
\usepackage{cancel}
\usepackage{calc}
\usepackage[colorinlistoftodos,prependcaption]{todonotes}
\usepackage{bbm}

\newtheorem{defn}{Definition}[section]
\newtheorem{lemma}[defn]{Lemma}
\newtheorem{prop}[defn]{Proposition}
\newtheorem{theo}[defn]{Theorem}
\newtheorem{coro}[defn]{Corollary}
\newtheorem{conj}[defn]{Conjecture}
\newtheorem{note}[defn]{Note}
\newtheorem{claim}{Claim}
\newtheorem{rk}[defn]{Remark}
\newtheorem{ques}[defn]{Question}
\newtheorem{exmp}[defn]{Example}

\def\RR{\mathbb{R}}
\def\Ric{\mathop{\rm Ric}\nolimits}
\def\Hess {\mathop{\rm Hess}\nolimits}
\def\Rm{\mathop{\rm Rm}\nolimits}
\def\tr{\mathop{\rm tr}\nolimits}
\def\det{\mathop{\rm det}\nolimits}

\def\vol{\mathop{\rm vol}\nolimits}
\def\dim{\mathop{\rm dim}\nolimits}
\def\vol{\mathop{\rm Vol}\nolimits}

\def\inj{\mathop{\rm inj}\nolimits}

\def\div{\mathop{\rm div}\nolimits}

\def\Id{\mathop{\rm Id}\nolimits}

\def\Li{\mathop{\rm \mathscr{L}}\nolimits}

\def\Ric{\mathop{\rm Ric}\nolimits}
\def\Hess {\mathop{\rm Hess}\nolimits}
\def\Rm{\mathop{\rm Rm}\nolimits}
\def\tr{\mathop{\rm tr}\nolimits}
\def\det{\mathop{\rm det}\nolimits}

\def\vol{\mathop{\rm vol}\nolimits}
\def\dim{\mathop{\rm dim}\nolimits}
\def\vol{\mathop{\rm Vol}\nolimits}

\def\inj{\mathop{\rm inj}\nolimits}

\def\div{\mathop{\rm div}\nolimits}

\def\Id{\mathop{\rm Id}\nolimits}

\def\Li{\mathop{\rm \mathscr{L}}\nolimits}

\def\Sym{\mathop{\rm Sym}\nolimits}

\def\R{\mathop{\rm R}\nolimits}

\def\eh {\mathbf{eh}}
\def\oi{\mathbf{o}}

\newsavebox\CBox
\newcommand\hcancel[2][0.5pt]{%
  \ifmmode\sbox\CBox{$#2$}\else\sbox\CBox{#2}\fi%
  \makebox[0pt][l]{\usebox\CBox}%
  \rule[0.5\ht\CBox-#1/2]{\wd\CBox}{#1}}

\numberwithin{equation}{section}

\begin{document}
\title{
Orbifold singularity formation along ancient and immortal Ricci flows}
\date{}
\author{Alix Deruelle}
\address{Universit\'e Paris-Saclay, CNRS, Laboratoire de math\'ematiques d'Orsay, 91405, Orsay, France}
\email{alix.deruelle@universite-paris-saclay.fr}
\author{Tristan Ozuch}
\address{MIT, Dept. of Math., 77 Massachusetts Avenue, Cambridge, MA 02139-4307.}
\email{ozuch@mit.edu}
\maketitle

\begin{abstract}
	In stark contrast to lower dimensions, we produce a plethora of ancient and immortal Ricci flows in real dimension $4$ with Einstein orbifolds as tangent flows at infinity. For instance, for any $k\in\mathbb{N}_0$, we obtain continuous families of non-isometric ancient Ricci flows on $\#k(\mathbb{S}^2\times \mathbb{S}^2)$ depending on a number of parameters growing linearly in $k$, and a family of half-PIC ancient Ricci flows on $\mathbb{CP}^2\#\mathbb{CP}^2$. 

The ancient/immortal dichotomy is determined by a notion of linear stability of orbifold singularities with respect to the expected way for them to appear along Ricci flow: by bubbling off Ricci-flat ALE metrics. We discuss the case of Ricci solitons orbifolds and motivate a conjecture that spherical and cylindrical solitons with orbifold singularities, which are unstable in our sense, should not appear along Ricci flow by bubbling off Ricci-flat ALE metrics.
\end{abstract}

\maketitle

\setcounter{tocdepth}{1}
\tableofcontents

\section{Introduction}

In this article we focus on $4$-dimensional solutions to the Ricci flow developing orbifold isolated singularities either in the past or in the future of the solution depending on whether these are ancient or immortal.
 Dimension $4$ seems to be critical for Ricci flow for the following reasons. From a PDE perspective, the curvature satisfies a reaction diffusion equation with quadratic nonlinearities which makes Ricci flow subcritical in dimensions $2$ and $3$ and supercritical in dimensions higher than or equal to $5$. On the geometric side, the works \cite{Gromov,Nabutovsky} seem to indicate that the quest for finding canonical Einstein metrics on a given closed manifold of dimension at least $5$ is bound to failure without extra symmetry and/or curvature sign assumptions. It is also the first dimension for which finite-time singularities of the Ricci flow can develop isolated singularities as demonstrated in \cite{FIK}.

\textit{Ancient Ricci flows} are key to understand the behavior of Ricci flow close to  \textit{finite-time singularities} in order to extend the flow past singular times at the cost of performing specific topological surgeries. Ideally, once the flow is extended, it becomes either extinct or exists forever. In that second situation, additional topological surgeries are performed in the \textit{long-term behavior} of Ricci flow which corresponds to \textit{immortal Ricci flows}. 

In the present paper, we construct large classes of noncollapsed ancient and immortal Ricci flows spontaneously developing or resolving orbifold singularities depending on a notion of \textit{stability} of orbifold singularities that we introduce. The block decomposition of the curvature tensor into irreducible components in dimension $4$ is crucially exploited here. Among these solutions we construct, we exhibit ancient flows on the same topology but with non-isometric or diffeomorphic tangent flows at $-\infty$, as well as ancient Ricci flows with isometric tangent flows, but with different topologies. We also construct immortal Ricci flows approaching hyperbolic and complex hyperbolic orbifolds with isolated singularities at $+\infty$, confirming that singular Einstein orbifolds can be ``thick'' regions of immortal Ricci flows. However our analysis strongly suggests that orbifold cylinders should not appear as finite-time singularities through the expected process of bubbling-off Ricci-flat ALE metrics, see \cite{LopezOzuch2024} for strong evidence towards this. In both of the above ancient and immortal cases, our techniques handle bubbling off any of the known stable Ricci-flat ALE spaces, and we obtain corresponding continuous families of ancient or immortal Ricci flows.

Until now, these orbifold singularities were ``worst case scenarii'' that have only been proved to be possible in very specific situations \cite{Appleton-EH,Bre-Kap} that do not correspond to finite-time singularities or long-term behaviors of the Ricci flow.

\subsection{Main motivations}
\subsubsection{Noncollapsed Ancient Ricci flows}
Our first motivation in this paper is to understand ancient solutions to the Ricci flow, i.e. solutions defined on $(-\infty,T)$ for some $-\infty<T\leq +\infty$. These are crucial to understand finite-time singularities of the Ricci flow and they naturally arise  by performing a ``parabolic blow up'' that zooms in at the singular time. Perelman \cite{Per-I} has shown that singularity models of a finite-time solution to the Ricci flow on a closed $n$-dimensional manifold are \textit{$\kappa$-noncollapsed ancient solutions} to the Ricci flow defined on $(-\infty, 0 ]$ and those are key to understand the singularity formation. 

The classification of $\kappa$-noncollapsed ancient solutions in dimension $3$ has been achieved in \cite{Ang-Bre-Das-Ses,Bre-Das-Ses}. In higher dimensions, a similar classification has been obtained (under additional curvature assumptions called uniformly PIC and weakly PIC2) in \cite{Bre-Das-Naf-Ses}. In the aforementioned works, the underlying topology is that of the standard $n$-sphere because of the curvature conditions that are imposed. 

In dimension $4$, the possible topologies of ancient solutions to the Ricci flow are not known since the classification of Einstein metrics with positive scalar curvature is open! What is lacking is therefore a sufficiently large class of examples. 
In dimension $4$, it has been made clear in \cite{Bre-Kap} that the question will be much more complicated without further curvature constraints. Indeed, the authors show the existence of an ancient solution to Ricci flow resolving a flat orbifold by gluing Eguchi-Hanson metrics at the orbifold points. Recall that Eguchi-Hanson metrics are Ricci flat Asymptotically Locally Euclidean (ALE) metrics whose tangent cone at infinity is $\RR^4/\mathbb{Z}^2$. 

Bamler's compactness theory \cite{bam1,bam2,bam3} shows that finite-time singularity models admit \textit{tangent orbifold solitons} at time $-\infty$: in dimension $4$, it states roughly that finite-time singularities are modeled on \textit{orbifold gradient shrinking solitons} such as quotients of the sphere $\mathbb{S}^4$ and cylinders $\mathbb{S}^3\times \mathbb{R}$ or $\mathbb{S}^2\times \mathbb{R}^2$.

Our first main Theorem \ref{mainthm1} below provides the first example of a noncollapsed ancient solution to the Ricci flow with a \textit{singular} tangent soliton demonstrating that the ``worst case scenario'' happens, and explaining the mechanism behind it. Notice that the ancient solutions of the Ricci flow constructed in \cite{Bre-Kap} are collapsed.



\subsubsection{Degeneration of Einstein $4$-manifolds and obstructions}

One of our main motivation here are the compactness results of noncollapsed Einstein $4$-manifolds \cite{Ban-Kas-Nak,Anderson,Che-Nab}, where the worst case scenario is that of orbifold singularities and where the degeneration comes from blowing up a Ricci flat ALE metric. In this situation, all noncollapsed Einstein degenerations can be reconstructed in \cite{ozu1,ozu2,ozuthese} by gluing-perturbation techniques, see also \cite{Leb-Sin,Biq-1,Biq-2} for earlier constructions. 

It however turns out that many Einstein $4$-orbifolds cannot be limits of Einstein metrics, even in the rough Gromov-Hausdorff sense \cite{ozu4}. There are indeed obstructions to the desingularization of Einstein $4$-orbifolds that can be computed from the curvature of the orbifold metric at its singular point. For instance, if at the singular point, the sectional curvatures are all equal to a non-zero value, then the Einstein orbifold cannot be a Gromov-Hausdorff limit of smooth Einstein $4$-manifolds. This is the case of the sine-cone over $\mathbb{RP}^3$ which is our most basic Example \ref{ex:orientatble spherical orb}, this is also the case of hyperbolic orbifolds such as the one of Example \ref{ex:hyperbolic orbifold}.

The only known examples of $4$-dimensional Ricci-flat ALE metrics are obtained from hyperkähler ALE metrics whose first examples were obtained in \cite{Eguchi}, \cite{GibbonsHawking1978}, \cite{Hitchin1979}, and extended in \cite{Kronheimer1989ALE}. These metrics and their smooth quotients are classified in \cite{Kronheimer1989Torelli,suv,Wright}. In the present article, we compute explicitly the obstructions to desingularizing an Einstein orbifold using hyperkähler ALE metrics at a singular point $p_o$. This was done in \cite{Biq-1} for the case of Eguchi-Hanson metrics. It turns out that they can be read from the selfdual part of the curvature of the orbifold at its singular point $\mathbf{R}^+_{p_o}$: obstructions correspond to nonzero eigenvalues of $\mathbf{R}^+_{p_o}$. Seeing the curvature tensor as a symmetric endomorphism acting on $2$-forms $\Rm:\Lambda^2\to\Lambda^2$, the selfdual part of the curvature is the restriction to the space of selfdual $2$-forms $\Lambda^+$, i.e. the eigenspace of the Hodge star operator associated to the eigenvalue $1$. 

Recall that a (singular) Einstein metric with positive scalar curvature defines an ancient solution to the Ricci flow outside the singular set by rescaling-in-time the metric. We are naturally led to the question:
\begin{ques}\label{ques-1}
Given a $4$-dimensional Einstein orbifold with positive scalar curvature and isolated singularities that cannot be desingularized by smooth Einstein metrics (e.g. spherical), is there an ancient Ricci flow coming out of it at $t=-\infty$? 
\end{ques}

In order to answer this question, instead of seeing the nonzero eigenvalues of $\mathbf{R}^+_{p_o}$ as \textit{obstructions} to finding Einstein metrics, we instead interpret them as the driving force of producing solutions to the Ricci flow. The obstructions to solving the static, elliptic Einstein problem, lead to interesting dynamical solutions to the parabolic Ricci flow.

Indeed, in the simplest case of Eguchi-Hanson bubbles, corresponding to the kernel of the linearization of the equation, there is a three-dimensional (so called Kuranishi) family of Einstein modulo obstructions metrics $\hat{g}_\zeta$ parametrized by a gluing parameter $\zeta\in\mathbb{R}^3$. Pushing the study further towards a center manifold analysis, it turns out that $t\mapsto \hat{g}_{\zeta(t)}$ approximately solves the Ricci flow equation if and only if $t\mapsto \zeta(t)$ solves the first order linear ODE:
\begin{equation}\label{eq:ODE zeta intro}
	\dot{\zeta} = 2 \mathbf{R}^+_{p_o}(\zeta).
\end{equation}
This leaves the obstructions of \cite{Biq-1} apparent: the Einstein metrics may only be found in the direction of the kernel of $\mathbf{R}^+_{p_o}$.

Moreover, \eqref{eq:ODE zeta intro} tells us that if $\mathbf{R}^+_{p_o}>0$ (as in the case of spherical orbifolds), then one should expect the bubbles to \textit{grow} along the flow, and the flow escapes from the orbifold. On the other hand, if $\mathbf{R}^+_{p_o}<0$ (as in the case of hyperbolic orbifolds), then the bubbles should shrink along the flow, and the flow approaches the orbifold. 

We interpret this as a notion of stability of the orbifold singularity: an orbifold singularity at $p_o$ is said to be strictly \textit{orbifold point unstable} if $\mathbf{R}^+_{p_o}>0$, and strictly \textit{orbifold point stable} if $\mathbf{R}^+_{p_o}<0$.
 This is also consistent with the signs of the eigenvalues of the linearization of the equation at desingularizations as computed in \cite{no1}.  This terminology is chosen to avoid any confusion with the more classical notion of stability for an Einstein metric that requires the linearized operator (called the Lichnerowicz laplacian here) to be nonpositive in the $L^2$-variational sense. We also stress that this is a pointwise condition that holds at the orbifold point only.

In Section \ref{sec:obsttoRF}, we also extend the obstruction analysis to the case of Ricci solitons. The obstructions and dynamics at Ricci soliton metrics involve some \textit{weighted} selfdual curvature $\overline{\mathbf{R}}^+_{p_o}$ at the singular point of the orbifold soliton that depends on the soliton potential.

\subsubsection{Limits of thick immortal Ricci flows}

An analogous question to Question \ref{ques-1} can be asked about immortal flows.
\begin{ques}\label{ques-2}
Given a $4$-dimensional Einstein orbifold with negative scalar curvature and isolated singularities that cannot be desingularized by smooth Einstein metrics (e.g. hyperbolic), is there an immortal Ricci flow converging to it at $t=+\infty$? 
\end{ques}

Another motivation for the search of immortal Ricci flows and their limits is that they are expected to be \textit{core model geometry} of the initial topology. For instance, on $3$-manifolds, the noncollapsed ``thick'' parts immortal Ricci flows approach to the \textit{hyperbolic part} of Thurston's geometrization decomposition. In higher dimensions, in a series of articles \cite{bam1,bam2,bam3}, Bamler developed a compactness theory for Ricci flows which extends (in a weak, so-called $\mathbb{F}$-sense) $3$-dimensional results to dimension $4$ and higher.  In the long-term behavior of the flow, one has a ``thick-thin'' decomposition of the flow into collapsed ``thin'' regions, and noncollapsed ``thick'' regions approaching finite-volume \textit{Einstein $4$-orbifolds} with possible isolated orbifold singularities such as (complex-)hyperbolic orbifolds. 


In this article, Theorem \ref{mainthm2} below shows that this worst case scenario involving orbifold singularities (with singularity group in $\operatorname{SU}(2)$) actually occurs, and we show that hyperbolic and complex hyperbolic orbifolds can be limits of Ricci flow as well as a large class of Einstein orbifolds satisfying a specific curvature condition.

Here, by contrast with the ancient solutions from Theorem \ref{mainthm1}, Ricci flow creates orbifold singularities by contracting bouquets of $(-2)$-spheres, performing the opposite of a minimal resolution.

\subsection{Main results: Dynamical (in)stability of Einstein $4$-orbifolds}

We now state our main results which justify the above notions of linear (in)stability by the associated results of \textit{dynamical} (in)stability. In order to make it clear that we are exhibiting new (un)stable behaviors, we consider flows built from gluing ``stable'' building blocks. Here ``stable'' refers to a more classical notion of stability which means that the underlying linearized operator called the Lichnerowicz Laplacian is non-positive in the $L^2$-variational sense. The interactions of the different pieces however lead to stability or instability according to our above linear stability criterion of orbifold singularities.  

In the case of ancient flows, it turns out that the only known to be stable Einstein orbifolds are spherical, i.e. with constant positive sectional curvature, so we restrict our statements to them. 

\begin{theo}\label{mainthm1}
  Let $\Gamma\subset \operatorname{SU}(2)$ be a finite group acting freely on $\mathbb{S}^3$ and let $(M_o^4,g_o)$ be a metric suspension over $\mathbb{S}^3/\Gamma$ endowed with a metric of constant curvature with Einstein constant $\Lambda$. Let $(N^4,g_b)$ be a hyperk\"ahler ALE metric asymptotic to $\mathbb{C}^2/\Gamma$ and define $k_\Gamma:=\dim_{\RR}H^2_c(N)$. Denote by $M$ the manifold obtained from gluing $(N^4,g_b)$ to $(M_o,g_o)$ at its orbifold points. Then $M$ is either oriented homeomorphic to $\#k_{\Gamma}(\mathbb{S}^2\times\mathbb{S}^2)$ if $k_{\Gamma}>1$ or to $\mathbb{CP}^2\#\mathbb{CP}^2$, $\overline{\mathbb{CP}}^2\#\overline{\mathbb{CP}}^2$, or $\mathbb{S}^2\times \mathbb{S}^2$ depending on the orientations in which the gluings are performed and there exists a $C(k_{\Gamma})$-parameter family of non-isometric solutions $(g(t))_{t\in(-\infty,0)}$ on $M$ to the renormalized Ricci flow $$\partial_tg(t)=-2\Ric(g(t))+2\Lambda g(t),\quad\text{ on $M\times (-\infty,0]$,}$$ such that:
    \begin{enumerate}
\item\label{curv-prop-I}
For $l\in\mathbb{N}$ and $t\leq 0$,
\begin{equation*}
\begin{split}
C^{-1}\,e^{-t}\leq \sup_M|\Rm(g(t))|_{g(t)}\leq C\,e^{-t},&\quad \sup_M|\nabla^{g(t),\,l}\Rm(g(t))|_{g(t)}\leq C_l\,e^{-t},\\
  \sup_M|\Ric(g(t))|_{g(t)}\leq C,&\quad\sup_M|\R_{g(t)}-4\Lambda|\leq Ce^{c t},
\end{split}
\end{equation*}
for some uniform positive constants $C$, $C_l$ and $c\in(0,1)$.

\item \label{curv-prop-V}The metrics $g(t)$ converge locally smoothly on $M_0$ to $g_o$ outside the orbifold points of $M_o$ as $t$ tends to $-\infty$ and Gromov-Hausdorff converge to $(M_o,g_o)$: there exists $\delta\in(0,1)$ and $C>0$ such that for $t\leq 0$, $$d_{GH}((M,d_{g(t)}),(M_o,d_o))\leq Ce^{\delta t}.$$
\end{enumerate}
In particular,   any hyperkähler ALE metric can be bubbled-off that way in the past of an ancient solution to the Ricci flow at $t=-\infty$.

Moreover, if $\Gamma=\mathbb{Z}_2$ and if we choose the orientations on the Eguchi-Hanson instantons $(N,\eh)$ to reach the topology $\overline{\mathbb{CP}}^2\#\overline{\mathbb{CP}}^2$, then there exists a $3$-parameter family of non-isometric solutions $(g(t))_{t\in(-\infty,0)}$ to the renormalized Ricci flow that satisfies on top of the above properties \eqref{curv-prop-I}-\eqref{curv-prop-V}, for $t\leq 0$,
\begin{equation*}
\begin{split}
\sup_M\left|\mathbf{R}^+(g(t))-\frac{\Lambda}{3}\Id_{\Lambda^+}\right|_{g(t)}\leq Ce^{ct}&,\quad \Ric(g(t))\geq C^{-1}g(t),\quad\sup_M|W^+(g(t))|_{g(t)}\leq Ce^{ct},
\end{split}
\end{equation*}
for some uniform positive constants $C$ and $c\in(0,1)$.

In particular, these solutions to the normalized Ricci flow $(g(t))_{t\leq 0}$ satisfy the $\operatorname{PIC}_+$ curvature condition of \cite{RichardSeshadri2016}.

\end{theo}
Here, $C(k_{\Gamma})$ grows linearly in $k_{\Gamma}$ as $k_{\Gamma}$ tends to $+\infty$ and one can show that $C(k_{\Gamma})\leq 6k_{\Gamma}$.

Topologically, this shows that Ricci flow may spontaneously desingularize orbifold isolated singularities: for a singularity $\mathbb{R}^4/\Gamma$ for $\Gamma\subset \operatorname{SU}(2)$ acting freely on $\RR^4\setminus\{0\}$, Ricci flow may spontaneously perform a \textit{minimal resolution} of the singularity. We moreover obtain continuous families of such ancient flows whose dimensions can be arbitrarily large.

Theorem \ref{mainthm1} on $\mathbb{S}^2\times\mathbb{S}^2$ or $\overline{\mathbb{CP}}^2\#\overline{\mathbb{CP}}^2$ provides the existence of ancient solutions $(g(t))_{t\,\leq\, 0}$ that are uniformly Ricci pinched, i.e. that satisfy for all $t\leq 0$, $\Ric(g(t))\geq C^{-1}g(t)$ for $C>0$ uniform. However, those cannot be PIC1 by \cite{Bre-Hui-Sin} since they do not have constant curvature. Moreover, with topology $\overline{\mathbb{CP}}^2\#\overline{\mathbb{CP}}^2$, this flow can be arranged to be arbitrarily close to being anti-selfdual as $t\to-\infty$.


In the case of immortal flows, we have much more flexibility, and all Einstein orbifold satisfying (at every point) $\mathbf{R}^+<0$ are therefore \textit{orbifold point stable}. This is for instance the case of anti-selfdual Einstein metrics with negative scalar curvature such as hyperbolic metrics and complex hyperbolic metrics (in the opposite of their Kähler orientation).
\begin{theo}\label{mainthm2}
Let $(M_o^4,g_o)$ be a compact Einstein orbifold with Einstein constant $\Lambda<0$ with isolated singularities associated to finite subgroups $(\Gamma_i)_i$ of $\operatorname{SU}(2)$. Assume furthermore that $\mathbf{R}^+(g_o)<0$ on $M_o$. For each subgroup $\Gamma_i\subset\operatorname{SU}(2)$, let $(N_i^4,g_i)$ be a Kronheimer instanton  asymptotic to $\mathbb{C}^2/\Gamma_i$ and define $k_i:=\dim_{\RR}H^2_c(N_i)$, and $k:=\sum_ik_i$. Denote by $M$ the manifold obtained from gluing the $(N_i^4,g_i)$ to $(M_o,g_o)$ at its orbifold points. 

    Then, there exists a $C(k)$-parameter family of non-isometric immortal solutions $(g(t))_{t\in(0,+\infty)}$ to the renormalized Ricci flow $$\partial_tg(t)=-2\Ric(g(t))+2\Lambda g(t),\quad\text{ on $M\times (0,+\infty)$,}$$ such that if $\Lambda_0<0$ denotes the largest eigenvalue of $2\mathbf{R}^+$ at the orbifold points,
 \begin{enumerate}
\item\label{curv-prop-I-imm}
For $k\in\mathbb{N}$ and $t\geq 0$,
\begin{equation*}
\begin{split}
C^{-1}\,e^{-\frac{\Lambda_0}{2}t}\leq \sup_M|\Rm(g(t))|_{g(t)}\leq C\,e^{-\frac{\Lambda_0}{2}t},&\quad \sup_M|\nabla^{g(t),\,k}\Rm(g(t))|_{g(t)}\leq C_k\,e^{-\frac{\Lambda_0}{2}t},\\
  \sup_M|\Ric(g(t))|_{g(t)}\leq C,&\quad\sup_M|\R_{g(t)}-4\Lambda|\leq Ce^{\frac{c\Lambda_0}{2}t},
\end{split}
\end{equation*}
for some uniform positive constants $C$, $C_k$ and $c\in(0,1)$.

\item \label{curv-prop-Vimm}The metrics $g(t)$ converge locally smoothly on $M_0$ to $g_o$ outside the orbifold points of $M_o$ as $t$ tends to $+\infty$ and Gromov-Hausdorff converge to $(M_o,g_o)$: there exists $\delta\in(0,1)$ and $C>0$ such that for $t\geq 0$, $$d_{GH}((M,d_{g(t)}),(M_o,d_o))\leq Ce^{\frac{\delta\Lambda_0}{2} t}.$$
\end{enumerate}



\end{theo}
Here, $C(k)$ grows linearly in $k$ as $k$ tends to $+\infty$.
 Similarly to Theorem \ref{mainthm1}, starting with anti-selfdual (complex) hyperbolic orbifolds, we may again ensure that the flows we construct are negatively Ricci pinched and as close as wanted to being anti-selfdual as $t\to+\infty$.

The above flows are perturbations of the approximate Ricci flows $t\mapsto \hat{g}_{\zeta(t)}$ for $\zeta$ solving the linear ordinary differential equation $\dot{\zeta} = 2\mathbf{R}^+_{p_o}(\zeta)$. We actually obtain families of ancient and immortal Ricci flows parametrized by the solutions of the aforementioned first order linear differential equation.

\begin{note}
At this intuitive ODE level, the ancient and immortal cases are exactly the same, up to reversing time. At the actual PDE level, it is of course not possible to reverse time for parabolic flows. Still, except for the results specific to the PDE at hand of Sections \ref{sec:liouville}, \ref{sec:fct spaces} and \ref{sec:linear theory}, most of the technicalities of this article come from careful estimates which are exactly the same for ancient and immortal flows. The reason is simple: for a tensor $T$, the estimates involving $(\partial_t-\Delta)T$ almost always come from estimates falling on $\partial_tT$ and $\Delta T$ \emph{separately}.
\end{note}





\begin{rk}
	The case of cyclic subgroups of $\operatorname{U(2)}$ reached at the infinity of Kähler (not hyperkähler) Ricci-flat ALE metrics classified in \cite{suv,Wright} is easily deduced from our construction since they are finitely covered by hyperkähler ALE metrics with group $\operatorname{SU(2)}$ at infinity.
\end{rk}

\begin{rk}
    The above flows also yield (non renormalized) Ricci flows for which the curvature decays as time goes to $-\infty$ in the ancient case, and $+\infty$ in the immortal case. In both cases, the curvature decays in $|t|^{-\frac{2}{3}}$ corresponding to a size $|t|^{\frac{2}{3}}$ for the bubbles. See Section \ref{sec:not renormalized scales}.
\end{rk}

\begin{rk}
    The above main Theorems \ref{mainthm1} and \ref{mainthm2} applies to large classes of examples of orbifolds, and yield many families of associated flows, see for instance Examples \ref{ex:orientatble spherical orb}, \ref{ex:waterdrop} and \ref{ex:hyperbolic orbifold}.
\end{rk}

\subsection{Related results}

\subsubsection{Brendle-Kapouleas construction: links and differences with our work}

A huge source of inspiration for the analysis developed in the present article is \cite{Bre-Kap}, where an ancient Ricci flow resolving the singularities of a flat orbifold by Eguchi-Hanson metrics is constructed. There are however major differences with our present work.

The first and main difference is the abundance of examples and the flexibility of the construction. {In \cite{Bre-Kap} are glued Eguchi-Hanson metrics to resolve the orbifold singularities whereas Theorem \ref{mainthm1} shows that any Kronheimer's instanton can be used as a building block to resolve orbifold singularities at $t=-\infty$. In \cite{Bre-Kap}, the analysis builds on the work of Biquard \cite{Biq-1}. As Kronheimer's instantons are less explicit than  Eguchi-Hanson metrics, this makes the analysis more delicate. Moreover, as we explained above, \cite{Bre-Kap} imposes symmetries so that the moduli space of Ricci flat metrics is reduced to one element that reflects the effect of scaling only, they obtain a single ancient flow up to time-reparametrization and rescaling. Theorems \ref{mainthm1} and \ref{mainthm2} considers the full moduli space near a given Kronheimer's instanton which not only allows scaling but also rotations of a given background hyperkähler ALE metric. The analysis of the obstructions and that on abstract function spaces is considerably more difficult, and the results of \cite{Biq-1} about the explicit Eguchi-Hanson metric have to be extended to abstract Ricci-flat ALE metrics.}

Since the orbifold reached in \cite{Bre-Kap} is flat and therefore matches optimally the Ricci-flat ALE metrics to be glued to its singularities, no additional preparation is required on this region, and the ``gluing scales'' can be chosen arbitrarily large with little care. This is not our case: our construction requires an additional perturbation of the Ricci-flat ALE metrics to make them match the orbifold metric up to order $4$, and the ``gluing scales'' have to be chosen very carefully.
\\

Additionally, our solutions are constructed as limits of a fast converging fixed point iteration of a contracting map on a Banach space. This has a few advantages compared to the Schauder fixed point theorem used in \cite{Bre-Kap}. Indeed, this gives a quantitative uniqueness to our flows while providing a continuous dependence with respect to the parameters (i.e. the solution $\zeta$ to $\dot{\zeta} = 2\mathbf{R}^+_{p_o}(\zeta)$  for an appropriate given initial condition). 

The precision of our estimates lets us obtain delicate estimates of the curvature of the solutions we construct in Theorems \ref{mainthm1} and \ref{mainthm2} in Section \ref{sec:curvature}. Indeed, while the curvature is becoming infinite either at $t=-\infty$ or at $t=+\infty$, we are still able to estimate specific components of the block decomposition of the curvature, namely the selfdual Weyl curvature, the scalar curvature and the traceless part of the Ricci curvature, and determine which specific parts of the curvature are blowing up.

In the situation of \cite{Bre-Kap}, the orbifold is flat, hence $\mathbf{R}^+_{p_o} = 0$ and our linear (in)stability criteria cannot be applied. It is necessary to determine the higher orders of approximation of curvature which is generally a subtle question. It was determined in \cite{ozu3} that one detects some directions where in a higher order sense, one has $\mathbf{R}^+_{p_o}>0$, explaining why an ancient flow is obtained in \cite{Bre-Kap}. In fact, at every singular point there is only \textit{one} direction where this higher order instability can be detected. This means that the discrete symmetry assumptions of \cite{Bre-Kap} are actually crucial for the existence of the flow while simplifying the analysis. Notice that neither Theorem \ref{mainthm1} nor Theorem \ref{mainthm2} impose any symmetry on the solution.

Our analysis should handle such nongeneric situations with non definite $\mathbf{R}^+_{p_o}$ once the expected ODE satisfied by $\zeta$ is understood. Understanding this ODE is however a task that can only be performed a case at a time.

\subsubsection{Other orbifold Ricci flows}

We briefly mention here some works that study Ricci flow in the presence of isolated orbifold singularities. The paper \cite{Gia-Sch} investigates the possibility of smoothing out instantaneously a closed orbifold with \textit{isolated} singularities modeled on cones over codimension $1$ spheres endowed with a positively curved smooth metric. An equivariant version of their result holds if the link of each cone is not simply connected anymore. There, the key idea is to use forward self-similar solutions (also known as expanding gradient Ricci solitons) as building blocks in order to get a first approximation of the flow as opposed to Ricci flat ALE metrics in this paper. Their choice is dictated by their desire to use the instantaneous smoothing effect of the Ricci flow on such metric singularities. Note that in their construction, the orbifold singularities are preserved while the metric singularity is being smoothed out canonically with the help of a forward self-similar solution. Moreover, for orbifold singularities as considered in the present article, \textit{spin} self-similar expanding solitons cannot exist by \cite{LopezOzuch2024}, so that gluing hyperkähler ALE spaces seems to be the only way to preserve the spin condition.

From a different perspective, \cite{Kro-Mar-Ver} (see also the references therein) considers the existence of a Ricci-DeTurck flow from spaces with isolated conical singularities which preserves the conical structure along the flow, i.e. the analytic setup is designed so that the singular orbifold points of the initial metric are preserved under the flow.

Up to our knowledge, all of the aforementioned works preserve the orbifold singularities. 

\subsection{Further directions}

We believe that the present article leaves open exciting directions of study.
Motivated by Theorem \ref{mainthm1}, we conjecture the following:

\begin{conj}\label{conj-1}
    There exist ancient Ricci flows from orbifold cylinders.
\end{conj}
If Theorem \ref{mainthm1} provides a generic behavior of the Ricci flow on spherical orbifolds with isolated singularities then it is tempting to ask the following:
\begin{ques}
    Can we rule out singular spherical and cylindrical orbifolds as finite-time singularity models of $4$-dimensional Ricci flows?
\end{ques}

In this direction, another crucial question is whether orbifold singularities along Ricci flow are necessarily related to Ricci-flat ALE metrics, see \cite{LopezOzuch2024} for progress in this direction under the topological spin condition:
\begin{ques}\label{ques:RFALE seules sing models}
	If a Ricci flow approaches a soliton orbifold at $t\to\pm\infty$, are Ricci-flat ALE metrics bubbled-off?
\end{ques}
Motivated by Theorem \ref{mainthm2} and echoing Conjecture \ref{conj-1}, we conjecture the following:
\begin{conj}
    There exist immortal Ricci flows approaching any quotient of Bryant's steady soliton by a subgroup of $\operatorname{SU}(2)$ bubbling-off any hyperkähler Ricci-flat ALE metric.
\end{conj}

In the direction of classifying noncollapsed ancient Ricci flows in dimension $4$, following Question \ref{ques:RFALE seules sing models} and  \cite{LopezOzuch2024}, we conjecture the following: 
\begin{conj}
    Up to time reparametrization and diffeomorphism, the ancient solutions to Ricci flow from Theorem \ref{mainthm1} are unique among smooth ancient solutions to Ricci flow having the same underlying topology and tangent solitons at $-\infty$.
\end{conj}
Notice that the proof of Theorem \ref{mainthm1} shows a uniqueness result thanks to the Banach fixed point theorem. However, this requires the solutions to lie in the same unit ball of the function spaces we introduce in Section \ref{sec:fct spaces}.

Our constructions should extend almost verbatim to the (nongeneric) case of orbifold singularities in higher dimensions following the obstructions to gluing Calabi metrics of \cite{Mor-Via}. These situations should even be simpler since the kernel of the linearization of the ALE Calabi metric is only $1$-dimensional, explicit, and identified with the traceless part of a Hessian in \cite{ozu4}. Dealing with more general Calabi-Yau ALE metrics would be more challenging, as the analysis of the obstructions in the present article relies heavily on a $4$-dimensional formalism, and the hyperkähler structure of the metrics.

\begin{conj}
    In any dimension $2n\geq 6$, there exist:
    \begin{itemize}
        \item an ancient Ricci flow from a $\mathbb{Z}_n$ quotient of $\mathbb{S}^{2n}$,
        \item immortal Ricci flows from hyperbolic orbifolds with $\mathbb{R}^{2n}/\mathbb{Z}_n$ singularities.
    \end{itemize}
    In both situations, Calabi metrics are bubbled-off.
\end{conj}

Similarly, if non selfdual Ricci-flat ALE metrics were to exist, 
we explain how one would still obtain ancient and immortal Ricci flows from it in Section \ref{sec:stab general RFALE}. In particular, the instability of spherical orbifolds and stability of hyperbolic orbifolds would be valid for any Ricci-flat ALE metric.

\begin{conj}
    If $(N^4,g_b)$ is a stable Ricci-flat ALE $4$-manifold with tangent cone at infinity $\mathbb{R}^4/\Gamma$, and $(M_o,g_o)$ is a spherical (respectively hyperbolic) orbifold with a singularity $\mathbb{R}^4/\Gamma$, then, there exists an ancient (respectively immortal) renormalized Ricci flow on $M_o\#_{\mathbb{S}^3/\Gamma}N$ with limit $(M_o,g_o)$ at $-\infty$ (respectively $+\infty$).
\end{conj}

\subsection{Discussion of the proofs}

The present article first justifies the new notion of linear stability by studying and interpreting in a dynamical way the obstructions to desingularizing Ricci soliton $4$-orbifolds. This essentially extends the computations of \cite{Biq-1,Biq-2,ozuthese} on Einstein metrics to Ricci solitons, and treats the case of any hyperkähler ALE metric for the first time.

The core of the article deals with the construction of ancient and immortal flows. This is a \textit{parabolic} gluing-perturbation theorem. The gluing is obtained from a suitable approximate solution to the (renormalized) Ricci flow given below in \eqref{eq:almostRF}, and most of the article consists in finding a small perturbation that reaches a normalized Ricci flow.




The main idea of the current paper  is that obstructions to solving the Einstein equation yield interesting behaviors for Ricci flows. In Sections \ref{sec:obsttoRF} and \ref{sec-kro}, we extend \eqref{eq:ODE zeta intro} to the case of general hyperkähler ALE metrics, parametrized for some $k\geq 1$ by $\zeta:=(\zeta_1,\dots,\zeta_k)\in(\mathbb{R}^3)^k$. One also finds that if for every $1\leq i\leq k$, one has $\dot{\zeta}_i= 2\mathbf{R}^+_{p_o}(\zeta_i)$, then, for $\zeta$ small enough, then we show that the curve of metrics $(\hat{g}_{\zeta(t)})_{t}$ approximately solve the renormalized Ricci flow:
\begin{equation}\label{eq:almostRF}
    \partial_t\hat{g}_{\zeta(t)}\approx -2\left(\Ric(\hat{g}_{\zeta(t)})-\Lambda \hat{g}_{\zeta(t)}\right).
\end{equation}

Note that our notion of stability only depends on the value of the curvature at a single point, and is therefore very easily checked on examples. This is in stark contrast with more classical situations where stability depends on the sign of eigenvalues of linear differential operators, which is generally a very subtle question, see \cite{CaoHamiltonIlmanen2004}.

We expect that strictly orbifold point stable orbifolds have a chance to appear as forward limits of some parabolic rescalings of solutions to Ricci flow. On the other hand, strictly orbifold point unstable elements should only appear as backward limits of parabolically rescaled solutions to Ricci flow.

Once a first approximation to the normalized Ricci flow is constructed as above, the next step consists in reducing the proof of Theorems \ref{mainthm1} and \ref{mainthm2} to the existence of a fixed point of a map $\Phi:\mathcal{X}\rightarrow\mathcal{X}$ defined between the same Banach space $\mathcal{X}$ through either Schauder's fixed point theorem (Section \ref{sec-nonlin}) or Picard's fixed point theorem (Section \ref{sec:picard}). The Banach space $\mathcal{X}$ is a function space of H\"older type and takes into account four types of variations in the setting of Theorem \ref{mainthm1}: 
\begin{itemize}
\item  metric variations consisting of symmetric $2$-tensors $k$, 
\item variations $\eta\in \RR^{3k_\Gamma}$ coming from the moduli space of Ricci flat ALE metrics and which are related to the kernel of the Lichnerowicz operator associated to a background Kronheimer's instanton, 
\item variations $\beta\in\RR$ coming from scaling the metric,
\item and variations $\tau$ from conformal diffeomorphisms acting on such a spherical suspension which lie in the kernel of the Lichnerowicz laplacian associated to the background orbifold Einstein metric $g_o$.
\end{itemize}
In the setting of Theorem \ref{mainthm2}, there are no such elements $\tau$ in the kernel of the linearized operator of the orbifold Einstein metric $g_o$: this makes the analytic set-up slightly simplified. 

\begin{note}
	Most of the proofs only deal with the \emph{ancient} case, hence negative times. Their adaptation to the immortal case is only nontrivial when the irreversibility of time in parabolic equations plays a role, i.e. in Sections \ref{sec:liouville}, \ref{sec:fct spaces} and \ref{sec:linear theory}. In the remaining technical sections we restrict ourselves to the ancient case, the immortal case is easily deduced: no parabolic theory is used, and the time derivatives and spatial derivatives are estimated \emph{separately}.
\end{note}

The map $\Phi:(k,\eta,\beta,\tau)\in\mathcal{X}\rightarrow ( h,\xi,\nu,\mu)\in\mathcal{X}$ is then roughly defined as follows: if $(\tilde{g}(t))_{t\,\leq\, T}$, $T\leq 0$, denotes the first approximation of metrics obtained in the first step, we wish to solve the following linear heat equation for some $T\leq 0$ \textit{modulo} obstructions and \textit{orthogonally} to the obstructions:
	\begin{equation}\label{eqn-solve-h-intro}
	\partial_th(t)-\Delta_{L,\tilde{g}(t)}h(t)-2\Lambda h(t)-\psi(t)\in\tilde{\mathbf{O}}(t),\quad h(t)\perp\tilde{\mathbf{O}}(t),\quad t\leq T,
	\end{equation}
where the non-homogeneous term $\psi(t)$ depends non linearly on the data $(k,\eta,\beta,\tau)$ and $(\tilde{g}(t))_{t\,\leq\, T}$. Here $\tilde{\mathbf{O}}(t)$ denotes the approximate finite dimensional kernel coming from i) the $L^2$-kernel of the Lichnerowicz laplacian associated to a background Kronheimer's instanton, ii) the first approximation $\tilde{g}(t)$ and iii) conformal deformations of the orbifold Einstein metric. We refer to Section \ref{sec-first-approx} for a definition of $\tilde{\mathbf{O}}(t)$. 
Once \eqref{eqn-solve-h-intro} is solved, the other components $\xi,\nu,\mu$ of the map $\Phi$ can be defined via the projections of $\partial_th(t)-\Delta_{L,\tilde{g}(t)}h(t)-2\Lambda h(t)-\psi(t)$ on an appropriate basis of $\tilde{\mathbf{O}}(t)$.

In order to solve \eqref{eqn-solve-h-intro}, we first derive  in Section \ref{sec:fct spaces} a priori Schauder estimates on the following linear heat equation:
	\begin{equation*}
	\partial_th(t)-\Delta_{L,\tilde{g}(t)}h(t)-2\Lambda h(t)=\tilde{\psi}(t),\quad h(t)\perp\tilde{\mathbf{O}}(t),\quad t\leq T,
	\end{equation*}
for a possibly different datum $\tilde{\psi}(t)$. This is where the geometry of each geometric pieces of the gluing is revealed through an adequate Liouville theorem proved for each region of the first approximation $\tilde{g}(t)$. This is the purpose of Section \ref{sec:liouville}.

The system \eqref{eqn-solve-h-intro} is then solved by controlling a priori the projections of $\partial_th(t)-\Delta_{L,\tilde{g}(t)}h(t)-2\Lambda h(t)-\psi(t)$ on $\tilde{\mathbf{O}}(t)$ that are estimated in terms of a potential solution $h$ and the datum $\psi$. The desired Schauder estimate is then obtained by absorption. This is the combination of Section \ref{sec:proj kernel} and Section \ref{sec:linear theory} together with Appendix \ref{sec-flot-mod-obst}.

\subsection{Organization of the article}

The first sections \ref{sec:obsttoRF}, \ref{sec-first-approx}, \ref{sec:proj kernel}, \ref{sec:liouville}, \ref{sec:fct spaces}, \ref{sec:linear theory}, \ref{sec-nonlin} of this article only deal with the existence of flows bubbling off Eguchi-Hanson metrics through the use of the Schauder fixed point theorem. The reader only interested in the existence of such Ricci flows should focus on these sections. We indicate in remarks the parts where additional care is required for other Ricci-flat ALE metrics, and treat them in Section \ref{sec-kro}. The case of immortal flows is only explicitly treated where the differences are significant: in Sections \ref{sec:liouville}, \ref{sec:fct spaces} and \ref{sec:linear theory}.

Section \ref{sec:obsttoRF} justifies the notions of stability of Einstein and Ricci solitons orbifolds. This is obtained by computing the obstructions to performing an Einstein gluing-perturbation of Eguchi-Hanson metrics at the singular points. 

Section \ref{sec-first-approx} provides the ``gluing'' step of our parabolic gluing-perturbation and determines how close to a Ricci flow the flow we start from is.

Section \ref{sec:proj kernel} estimates projections of various tensors and will be used all along the article.

Section \ref{sec:liouville} details how the assumption of stability for the pieces of the construction (orbifold and ALE spaces) is used along the article: to prove Liouville theorems for the heat flow associated to the Lichnerowicz Laplacian. 

Section \ref{sec:fct spaces} then defines the function spaces where the analysis of the article will take place. Their properties such as Schauder estimates are verified.

Section \ref{sec:linear theory} deals with the linear theory providing existence and controls for the heat flow associated to the Lichnerowicz Laplacian on our degenerating curve of metrics close to our initial approximate Ricci flow.

Section \ref{sec-nonlin} finally proves the existence of the flows through the use of the Schauder fixed point theorem.

The proofs of additional aspects of the statements are deferred to later sections. In Section \ref{sec:picard}, a Picard fixed point iteration scheme is proved to give the desired solutions to the normalized Ricci flow from either Theorem \ref{mainthm1} or Theorem \ref{mainthm2}. Section \ref{sec:curvature} details the curvature properties of the ancient and immortal flows from Theorems \ref{mainthm1} and \ref{mainthm2}. Section \ref{sec-kro} then extends the results of the article to any known Ricci-flat ALE metric through an analysis of the deformations of Kronheimer's instantons which is of independent interest. The end of Section \ref{sec-kro} is dedicated to extending the previous results to potential Ricci-flat ALE metrics which are not hyperkähler. The (in)stability of spherical and hyperbolic orbifolds singularities is proven for general Ricci-flat ALE metrics using the renormalized volume of \cite{Biq-Hein}. 

The proof of some more technical results are then provided in the Appendices. Appendix \ref{sec:analytic} provides real-analytic controls with respect to their parameters for the metrics and tensors relevant to the analysis. Appendix \ref{sec-flot-mod-obst} proves the existence of ``heat flows modulo obstructions'' used in Section \ref{sec:linear theory}. Appendix \ref{sec:lin curv} provides schematic controls of linear and nonlinear operations that are seen numerous times along the article, and Hölder estimates are deferred to Appendix \ref{sec: annexe holder}.

\subsection*{Acknowledgments}
The authors would like to thank Joel Fine and Bruno Premoselli for explaining to us some subtleties of the constructions in \cite{Fine-Panov-2010} and \cite{FinePremoselli2020}, Simon Brendle for discussions about the results of Appendix \ref{sec-flot-mod-obst}, and John Lott for pointing out some inaccuracies in the immortal cases of some our statements.

The first author is partially supported by grants from the French National Research Agency  ANR-AAPG2020 (Project PARAPLUI) and ANR-24-CE40-0702 (Project OrbiScaR) together with the Charles Defforey Fondation-Institut de France via the project "Syst\`emes dynamiques et flots g\'eom\'etriques". The first author also benefits from a Junior Chair from Institut Universitaire de France. The second named author was partially supported by NSF Grant DMS-2405328 during the project.

\newpage

\section{From obstructions to Ricci flows}\label{sec:obsttoRF} 
\subsection{Basics of $4$-dimensional geometry and notations}\label{section 4d}

Let us start by recalling well-known properties specific to dimension $4$ and fix our notations for the rest of the article.

\subsubsection{Notations}


In dimension $4$, the space of $2$-forms decomposes into selfdual and anti-selfdual $2$-forms which are elements of the eigenspaces of Hodge star operator $*$ (which satisfies $*^2 = \textup{Id}$) respectively associated to the eigenvalues $1$ and $-1$. We denote $\Lambda^+$ and $\Lambda^-$ the associated bundles of eigenspaces.

Denote $(x_1,x_2,x_3,x_4)$ coordinates in an orthonormal basis of $\mathbb{R}^4$, and define $$r:=\sqrt{x_1^2+x_2^2+x_3^2+x_4^2}.$$ We define the $2$-forms 
\begin{equation}\label{def omegai euclidean}
	\omega^\pm_1(\mathbf{e}) := dx^1\wedge dx^2\pm dx^3\wedge dx^4,
\end{equation} 
and similarly $\omega_2^\pm(\mathbf{e})$ and $\omega_3^\pm(\mathbf{e})$ by cyclic permutations of the indices $\{2,3,4\}$. To simplify the notations, in this section only, we will abusively write $\omega^\pm_1(\mathbf{e}) = \omega^\pm_1$.

The $(\omega_i^+)_{1\,\leq\, i\,\leq\, 3}$ form an orthogonal basis of the space of selfdual $2$-forms, $\Lambda^+$, and the $(\omega_i^-)_{1\,\leq\,i\,\leq\,3}$ form an orthogonal basis of the space of anti-selfdual $2$-forms, $\Lambda^-$. 

We define $\alpha_i^\pm:=\omega_i^+(\partial_r/r)$. These form a basis of invariant $1$-forms on the sphere $\mathbb{S}^3$, $(\alpha_1,\alpha_2,\alpha_3)$ and can be written down explicitly: 
\begin{equation}\label{def-alpha}
    \alpha_1^\pm :=\frac{1}{r^2}(x_1dx_2-x_2dx_1\pm x_3dx_4 \mp x_4dx_3),
\end{equation}
and the other forms $\alpha_2$ and $\alpha_3$ are obtained by cyclic permutation of the indices $\{2,3,4\}$. In each orientation, we define the frame $(r\partial_r,Y_1^+,Y_2^+,Y_3^+) $ dual to the coframe $ (dr/r,\alpha_1^\pm,\alpha_2^\pm,\alpha_3^\pm) $ where the $Y_i^\pm$ are Killing vector fields on $\mathbb{R}^4$ fixing the origin. We will have a preferred orientation along the article and we will denote $\alpha_i = \alpha_i^+$.

We also define the following $2$-forms which also form bases of the spaces of selfdual or anti-selfdual $2$-forms:
$$\theta_1^- :=  r dr\wedge \alpha_1 - r^2\alpha_2\wedge\alpha_3,$$ and similarly $\theta_2^-$ and $\theta_3^-$ by cyclic permutations.
\\

 For the above $2$-forms, we have the following formula: for $x = (x_1,x_2,x_3,x_4),$
    \begin{align}
        \theta_i^-(x) = \sum_{j=1}^3 \frac{x^T(\omega_i^+\circ\omega_j^-)x}{|x|^2}\omega_j^- =- \frac{\langle Y_i^+, Y_j^-\rangle}{r^2}\omega_j^-,\label{rotation 2 forms -}
    \end{align}
    where $ \omega_i^+\circ\omega_j^- $ is the symmetric traceless matrix given by the (commuting) product of the antisymmetric matrices associated to $\omega_i^\pm$ and $\omega_j^\mp$, and where $x^T$ is the transpose of $x$.
We also have the following equalities: $\omega_1^+ = r dr\wedge \alpha_1 + r^2\alpha_2\wedge\alpha_3$
    and similar equalities for $\omega_i^+$ for $i \in \{2,3\}$ by cyclic permutations.

\subsubsection{Orbifolds and ALE spaces}

We will be interested in two types of geometries: Einstein orbifolds and Ricci-flat ALE metrics. 

\begin{defn}[Orbifold (with isolated singularities)]\label{orb Ein}
    A metric space $(M_o,g_o)$ is an orbifold of dimension $d\geq 2$ if there exists $\delta_0>0$ and a finite number of points $(p_k)_k$ of $M_o$ called \emph{singular} such that we have the following properties:
    \begin{enumerate}
        \item the space $(M_o\backslash\{p_k\}_k,g_o)$ is a manifold of dimension $d$,
        \item for each singular point $p_k$ of $M_o$, there exists a neighborhood of $p_k$, $ U_k\subset M_o$, a finite subgroup acting freely on $\mathbb{S}^{d-1}$, $\Gamma_k\subset \operatorname{SO}(d)$, and a diffeomorphism $ \Phi_k: B_\mathbf{e}(0,\delta_0)\subset\mathbb{R}^d\slash\Gamma_k \to U_k\subset M_o $ for which, the pull-back of $\Phi_k^*g_o$  on the covering $\mathbb{R}^d$ is smooth.
    \end{enumerate}
\end{defn}

{
\begin{rk}\label{rk: orb Ein local coord}
	Einstein orbifold metrics are smooth up to taking a finite local cover at the singular point as seen in \cite{Ban-Kas-Nak}. In particular, as proved in \cite[Lemma 3.4]{Mor-Via} and with the above abuse of notation, one has coordinates in which there exists a quadratic harmonic $2$-tensor $H_2$ satisfying 
	$$\Phi_k^* g_o = \mathbf{e} + H_2+O(r_o^3), $$
    for small $r_o<\delta_o$ where $\mathbf{e}$ is a Euclidean metric on $\mathbb{R}^d\slash\Gamma_k$, and $r_o$ is the distance to $0\in\mathbb{R}^d\slash\Gamma_k$ on each $\Phi_k^*U_k$ with respect to $\mathbf{e}$ extended as a smooth function smaller than $1$ on the rest of $M_o$. If the group of the singularity is $\mathbb{Z}_2$, the above $O(r_o^3)$ is upgraded to a $O(r_o^4)$. 
\end{rk}

\begin{note}
	We will abusively omit the diffeomorphism $\Phi_k$ when talking about the metric $g_o$ close to its singular point, and we will simply write $g_o = \mathbf{e} + H_2+O(r_o^3)$.
\end{note}}

\begin{defn}[ALE orbifold (with isolated singularities)]\label{def orb ale}
    An ALE orbifold $(N,g_b)$ of dimension $d\geq 2$, is a metric space for which there exists $\delta_0>0$, singular points $(p_k)_k$ and a compact $K\subset N$ for which we have:
    \begin{enumerate}
        \item $(N,g_b)$ is an orbifold of dimension $d$,
        \item there exists a diffeomorphism $\Psi_\infty: (\mathbb{R}^d\slash\Gamma_\infty)\backslash \overline{B_\mathbf{e}(0,\delta_0^{-1})} \to N\backslash K$ such that we have $r^l|\nabla^l(\Psi_\infty^* g_b - \mathbf{e})|_{\mathbf{e}}\leq C_l r^{-d}.$
    \end{enumerate}
\end{defn}

{
\begin{note}
	We will similarly abusively omit the diffeomorphism $\Phi_\infty$ when talking about the metric $g_b$ close to infinity.
\end{note}}

\subsubsection{Curvature of $4$-manifolds}

Thanks to the direct sum of selfdual and anti-selfdual $2$-forms, the symmetric endomorphism on $2$-forms, $\mathbf{R}$ given by the Riemannian curvature decomposes into blocks,
\[
\mathbf{R}=:
  \begin{bmatrix}
     \mathbf{R}^+ & \mathring{\Ric} \\
     \mathring{\Ric} & \mathbf{R}^-
  \end{bmatrix},
\]
 where the tensor $\mathring{\Ric}$ is identified with the traceless part of the Ricci curvature through the isomorphism \eqref{eq:identification traceless} below, and where $\mathbf{R}^\pm$ are the selfdual and anti-selfdual parts of the curvature. One also has the further decomposition of the blocks $\mathbf{R}^\pm$:
 \begin{equation*}
 \mathbf{R}^\pm=W^\pm+\frac{\R}{24},
 \end{equation*}
where $W^\pm$ the selfdual and anti-selfdual parts of the Weyl tensor $ W$ and where $\R$ denotes the scalar curvature of $ \mathbf{R}$.

We will also need the identification of the set of traceless symmetric $2$-tensors $\textup{Sym}_0^2(M)$ with $\Lambda^+\otimes\Lambda^-$ thanks to the map:
\begin{equation}\label{eq:identification traceless}
    \omega^+\otimes\omega^-\in\Lambda^+\otimes\Lambda^- \mapsto \omega^+\circ \omega^- =\omega^-\circ \omega^+\in\textup{Sym}_0^2(M),
\end{equation}
where $\omega^+\circ \omega^-$ is the $2$-tensor associated to the composition of the anti-symmetric endomorphisms of $TM$ associated to $\omega^+$ and $\omega^-$ by the metric. Therefore, any $2$-tensor $h$ on $(\mathbb{R}^4,\mathbf{e})$ has unique decompositions:
$$h = \lambda \mathbf{e} + \sum_i \phi_i^-\circ\omega_i^+ =\lambda \mathbf{e} + \sum_j \phi_j^+\circ\omega_j^-$$
for a scalar function $\lambda$, and the $\phi_i^\pm$ with values in $\Lambda^\pm$.

\textbf{Examples of possible curvature for Einstein metrics}

Einstein metrics have curvature (seen as an endomorphism between $2$-forms decomposed into selfdual and anti-selfdual forms):
$$ \mathbf{R}= \begin{bmatrix}
\mathbf{R}^+&0\\
0&\mathbf{R}^-
\end{bmatrix}. $$
The $\mathbf{R}^\pm$ may take arbitrary values.

Assume that the Einstein metric has $\Ric(g) = \Lambda g$, then
\begin{itemize}
	\item Constant curvature metrics have $\mathbf{R}^\pm = \frac{\Lambda}{3} \operatorname{I}_3$,
	\item Selfdual metrics have $\mathbf{R}^- = \frac{\Lambda}{3} \operatorname{I}_3$,
	\item Hyperkähler metrics have $\mathbf{R}^+=0$ (hence $\Lambda=0$), and $\mathbf{R}^-$ can take arbitrary traceless values,
	\item Kähler metrics have $\mathbf{R}^+ =\begin{bmatrix}
0&0&0\\
0&0&0\\
0&0&\Lambda\\
\end{bmatrix},$
	\item Complex hyperbolic metrics are both Kähler and selfdual,
	\item on Eguchi-Hanson (whose definition is recalled in Section \ref{eh-section} below) in a well-chosen basis, $\mathbf{R}^+=0$ and the matrix of $\mathbf{R}^-$ is pointwise proportional to $\begin{bmatrix}
2&0&0\\
0&-1&0\\
0&0&-1\\
\end{bmatrix},$
	\item The metrics of \cite{FinePremoselli2020} have $\mathbf{R}^+<0$ and $\mathbf{R}^-<0$ when the gluing parameter is large enough. It is unclear if the construction can yield orbifolds.
\end{itemize}

\textbf{Example of possible curvatures of Ricci solitons}

Gradient Ricci solitons with an orbifold singularity at $p$ and $\Ric(g_o) + \Hess_{g_o}f_o = \Lambda g_o$ have a development:
$$g_o = \mathbf{e} + H_2 +... $$
and $f_o = cst + F_2 +...$, where $F_2$ is quadratic and solves $ \Delta_{\mathbf{e}} F_2 = \Lambda n - \R_{g_o} $ because $\nabla f_o$ vanishes at orbifold singularities since there are no nontrivial linear functions on nontrivial quotients of $\mathbb{R}^4$. This implies that 
$$F_2 = \Big(\frac{\Lambda}{2} - \frac{\R_{g_o}(p)}{2n}\Big)r^2 +\frac{1}{2} \sum_{ij}f_{ij}x^T( \omega_i^+\circ\omega_j^-)x,$$
and the Hessian of $F_2$ is:
$$ \Hess_{\mathbf{e}}F_2 = \Big(\Lambda - \frac{\R_{g_o}(p)}{n}\Big)\mathbf{e} + \sum_{ij}f_{ij}\omega_i^+\circ\omega_i^-.$$

However, the metrics are not Einstein so the traceless part of the Ricci curvature $\mathring{\Ric}$ seen as an element of $ \Lambda^+\otimes \Lambda^- $ is not zero, it is of the form $-\sum_{ij}f_{ij}\omega_i^+\circ\omega_i^-$. We further need a ``weighted (anti-)selfdual'' curvature at orbifold points for solitons as in \cite{CaoTran}. It is this weighted curvature which is relevant to the stability of orbifold singularities, and it is defined as:
\begin{equation}
	\overline{\mathbf{R}}^\pm = \mathbf{R}^\pm + \frac{\Delta f}{4}\operatorname{Id}_{\Lambda^+} = \mathbf{R}^\pm + \left(\Lambda - \frac{\R}{4}\right)\operatorname{Id}_{\Lambda^+}.
\end{equation}

This curvature can easily be computed on examples:
\begin{itemize}
	\item at every point, the curvatures $\mathbf{R}^\pm$ of Bryant's rotationally symmetric soliton in a well-chosen basis are multiples of $\operatorname{I}_3$. The multiplication factor is positive for its steady soliton.
	\item On the shrinking (respectively expanding) Gaussian soliton, the potential $f_o$ satisfies $\Delta_{g_o}f_o > 0$ (respectively $\Delta_{g_o}f_o < 0$), and we find
	\begin{equation}
		\overline{\mathbf{R}}^+ >0 \text{ (respectively }\overline{\mathbf{R}}^+ <0 ).
\end{equation}  
	\item on the cylinder $\mathbb{S}^3\times \mathbb{R}$, one has 
	\begin{equation}
		\mathbf{R} = \frac{1}{2}\begin{bmatrix}
I_3&I_3\\
I_3&I_3
\end{bmatrix},
	\end{equation}
	and in particular, since $\Delta_{g_o}f_o > 0$ at the critical points of the potential $f_o$, we find:
	\begin{equation}
		\overline{\mathbf{R}}^+ >0.
\end{equation} 
	\item Similarly, on the cylinder $\mathbb{S}^2\times\mathbb{R}^2$, we find
	\begin{equation}
		\overline{\mathbf{R}}^+ >0.
\end{equation} 
\end{itemize}

\subsection{Eguchi-Hanson metrics}\label{eh-section}

The Eguchi-Hanson metric, which we will denote $\mathbf{eh}$, is defined on $T^*\mathbb{S}^2$. It is asymptotic to the flat cone $\mathbb{R}^4\slash\mathbb{Z}_2$. Outside of the zero section of $T^*\mathbb{S}^2$ represented by $r=0$, the metric $\mathbf{eh}$ has the following expression (with the identification $T^*\mathbb{S}^2\backslash \mathbb{S}^2\approx (\mathbb{R}^4\slash\mathbb{Z}_2) \backslash \{0\}$):
\begin{equation}
    \mathbf{eh}:= \sqrt{\frac{r^4}{1+r^4}}(dr^2+r^2\alpha_1^2) + \sqrt{1+r^4}(\alpha_2^2+\alpha_3^2).\label{eguchihanson}
\end{equation}
The metric extends to $ T^*\mathbb{S}^2 $ by adding the zero section $\mathbb{S}^2$ of self-intersection $(-2)$ with metric $\alpha_2^2+\alpha_3^2$ at $r=0$. We will always denote $\mathbb{S}^2$ to be the zero-section of $ T^*\mathbb{S}^2 $.

\subsubsection{Ricci-flat ALE deformations of Eguchi-Hanson metrics}

    We will need to consider the homothetic deformations of Eguchi-Hanson metrics, and we will parametrize them as follows. For any $\zeta = (x,y,z)\in \mathbb{R}^3\backslash \{0\}$, let us denote $\zeta^+_1 := x\omega_1^+ +y\omega_2^+ +z\omega_3^+$,  and define $\zeta_2^+$ and $\zeta_3^+$ so that $(\zeta_1^+,\zeta_2^+,\zeta_3^+)$ forms an orthogonal basis of constant norm of $\Lambda^+$, note that with these conventions, we find $|\zeta_i^+|_{\Lambda^2(\mathbb{R}^4)}^2 = 2|\zeta|^2_{\mathbb{R}^3}$. Let us also define 
    \begin{equation}\label{def-esp-phi}
    \varepsilon = \sqrt[4]{x^2 + y^2+ z^2},\quad \phi^+_i := \varepsilon^{-2}\zeta_i^+.
    \end{equation}
     We define the metric
    \begin{equation}
        \mathbf{eh}_{\zeta}:= \sqrt{\frac{r^4}{\varepsilon^4+r^4}}\left(dr^2+\phi^+_{1}(\partial_r)^2\right) + \frac{\sqrt{\varepsilon^4+r^4}}{r^2}\left(\phi^+_{2}(\partial_r)^2+\phi^+_{3}(\partial_r)^2\right)\label{eh zeta+}.
    \end{equation}
    
    \begin{defn}\label{def rho -}
        We define for any $x\in \mathbb{R}^4$ the linear transformation $\rho_x$ as follows: 
        \begin{equation*}
    \rho_x:    \omega_i^+ = dr\wedge \alpha_i + \alpha_j\wedge \alpha_k \rightarrow \theta_i^-:=dr\wedge \alpha_i - \alpha_j\wedge \alpha_k,
    \end{equation*}
     for all $(i,j,k)$ cyclic permutation of $(1,2,3)$ and where the $1$-forms $(\alpha_i)_{1\,\leq\,i\,\leq\,3}$ are defined in \eqref{def-alpha}.
    \end{defn}
    The next result recalls the basic asymptotic properties of $\mathbf{eh}_{\zeta}$.
    
    {
    \begin{prop}[\cite{ozu3}]\label{differents zeta}
        Let $\zeta= (x,y,z)\in \mathbb{R}^3\backslash \{0\}$ and denote $\zeta^+ := x\omega_1^+ +y\omega_2^+ +z\omega_3^+$. We have the following properties
        \begin{enumerate}
            \item for $r\geq 1000\sqrt{|\zeta|}$, denoting $\mathbf{e} = \mathbf{eh}_{0} = dr^2+r^2(\alpha_1^2+\alpha_2^2+\alpha_3^2),$ we have
            \begin{equation}
                \mathbf{eh}_{\zeta} = \mathbf{e} - \frac{\rho(\zeta)^-\circ \zeta^+}{2r^4} + O(|\zeta|^4r^{-8}) =:\mathbf{e} + H_\zeta^4 + O(|\zeta|^4r^{-8}),\label{dvp eh}
            \end{equation}
            \item $\mathbf{eh}_{\zeta}$ is isometric to $|\zeta|\cdot\mathbf{eh}$.
        \end{enumerate}
    \end{prop}
    \begin{rk}\label{rk same}
        In particular, \eqref{dvp eh} provides \textit{common} ALE coordinates on $\{ r> 1000c\}$ for all of the metrics $\mathbf{eh}_\zeta$ with $|\zeta| < c^2$ at once. More precisely, the diffeomorphism $\Phi_\infty$ of Definition \ref{def orb ale} is the same independently of $\zeta$ if $|\zeta|\leq c^2$. This makes the construction of the first approximate curve of metrics $\tilde{g}(t)$ less tedious in the next section.
    \end{rk}}
 
    	Let us give a description of the diffeomorphism in the second point of Proposition \ref{differents zeta}: 
{\begin{defn}\label{defn: diffeo s zeta}
    	 We have 
	\begin{equation}\label{Def-attaching-map}
	 \mathbf{eh}_{\zeta} = |\zeta|\cdot (s_{\zeta})_*\mathbf{eh}, 
	 \end{equation}
	 where $s_{\zeta}$ is the composition of the (commuting) reparametrization $r\mapsto \varepsilon r = \sqrt{|\zeta|} r$, and of the rotation sending each $\omega_i^+$ to $\phi_i^+$.
\end{defn}	 }
	 
    \begin{rk}
        These metrics $\mathbf{eh}_{\zeta}$ for $\zeta\in \mathbb{R}^3\backslash\{0\}$ reach \emph{all} of the metrics obtained by \emph{orientation-preserving} rotations and rescaling of $\mathbf{eh}$ \emph{up to isometry}. There are also examples with the other orientation obtained by replacing $\omega_i^+$ by $\omega_i^-$.
    \end{rk}
    
    \subsubsection{Infinitesimal Ricci-flat ALE deformations of Eguchi-Hanson metrics}\label{basics-EH}

        On the metric $\mathbf{eh} = \mathbf{eh}_{(1,0,0)}$, denote $\mathbf{O}(\mathbf{eh})$ the set of infinitesimal Ricci-flat ALE deformations which are in Bianchi gauge, i.e. the set of smooth symmetric $2$-tensors $h\in L^2(\mathbf{eh})$ satisfying:
        \begin{equation*}
        d_{\mathbf{eh} }\Ric(h)=0,\quad B_{\mathbf{eh}}(h):=\div_{\mathbf{eh} }h-\frac{d\tr_{\mathbf{eh} }h}{2}=0.
        \end{equation*}
        Alternatively, by tracing the equation $d_{\mathbf{eh}}\Ric(h)=0$, $\mathbf{O}(\mathbf{eh})$ consists of smooth symmetric $2$-tensors $h\in L^2(\mathbf{eh})$ satisfying:
         \begin{equation*}
        d_{\mathbf{eh} }\Ric(h)=0,\quad \div_{\mathbf{eh} }h=0,\quad\tr_{\mathbf{eh} }h=0.
        \end{equation*}
        An $L^2$-orthogonal basis of $\mathbf{O}(\mathbf{eh})$ with squared norms $2\pi^2$ is given by:
   \begin{enumerate}
       \item $\mathbf{o}_1:=-\frac{r^2}{(1+r^4)^{3/2}}(dr^2+r^2 \alpha_1^2) +\frac{1}{(1+r^4)^{1/2}}(\alpha^2_2+\alpha_3^2)$,\newline
       \item $\mathbf{o}_2:= \frac{1}{1+r^4}(rdr\cdot\alpha_2-r^2\alpha_1\cdot\alpha_3)$,\newline
       \item $\mathbf{o}_3:= \frac{1}{1+r^4}(rdr\cdot\alpha_3+r^2\alpha_1\cdot\alpha_2)$.
   \end{enumerate}
     (we take the same convention as \cite{Bre-Kap} which are essentially the opposite of \cite{Biq-1}). Note that the definition of $\mathbf{o}_1$ matches that of \eqref{eq:def o1 par u} below.
     
        For $\zeta\in \mathbb{R}^3\backslash\{0\}$, let us consider, thanks to the second point of Proposition \ref{differents zeta}, a diffeomorphism $s_\zeta$ such that $s_\zeta^*\mathbf{eh}_{\zeta} = |\zeta|\cdot \mathbf{eh}$. 
        
        Then, the $2$-tensors
        \begin{equation*}
        \mathbf{o}_i(\zeta):=|\zeta|(s_\zeta)_*\mathbf{o}_i,\quad i\in\{1,2,3\},
        \end{equation*}
         forms an orthogonal basis of $\mathbf{O}(\mathbf{eh}_{\zeta})$ with 
         \begin{equation}
         \|\mathbf{o}_i(\zeta)\|^2_{L^2(\mathbf{eh}_{\zeta})}= |\zeta|^2\|\mathbf{o}_i\|^2_{L^2(\mathbf{eh})}=2\pi^2|\zeta|^2,\label{compute-norm-oi-zeta}
         \end{equation}
        since the $L^2$-norm of $2$-tensors is invariant by rescaling in dimension $4$.  We have the following more concrete expression for the $\mathbf{o}_i(\zeta)$. 
        Writing as before $\zeta=\varepsilon^2\phi$ with $|\phi|=|\omega_1^+|$, we have:
    \begin{enumerate}
       \item $\mathbf{o}_1(\zeta):=-\frac{\varepsilon^4r^2}{(\varepsilon^4+r^4)^{3/2}}(dr^2+\phi_1^+(\partial_r)^2) +\frac{\varepsilon^4}{r^2(\varepsilon^4+r^4)^{1/2}}(\phi_2^+(\partial_r)^2+\phi_3^+(\partial_r)^2)$,\newline
       \item $\mathbf{o}_2(\zeta):= \frac{\varepsilon^4}{\varepsilon^4+r^4}(dr\cdot\phi_2^+(\partial_r)-\phi_1^+(\partial_r)\cdot\phi_3^+(\partial_r))$,\newline
       \item $\mathbf{o}_3(\zeta):= \frac{\varepsilon^4}{\varepsilon^4+r^4}(dr\cdot\phi_3^+(\partial_r)+\phi_1^+(\partial_r)^2\cdot\phi_2^+(\partial_r)^2)$.
   \end{enumerate}

        Recalling that in the same coordinates at infinity, the first $r_\mathbf{e}^{-4}$-terms of $|\zeta|(s_\zeta)_*\mathbf{eh}$ and of $\mathbf{eh}_{\zeta}$ coincide, we find, for $\rho$ defined in Definition \ref{def rho -},
$$|\zeta|\cdot(s_\zeta)_*\Big(\frac{\omega_1^+\circ\theta_1^-}{2r_\mathbf{e}^4}\Big) =\frac{\zeta^+\circ\rho(\zeta)^-}{2r_\mathbf{e}^4} = O(|\zeta|^2r_{\mathbf{e}}^{-4})$$
and therefore, denoting $(\zeta^+_{i})_{i\in\{1,2,3\}}$ a basis of the selfdual $2$-forms of constant length as above, the development at infinity of the $(s_\zeta)_*\mathbf{o}_i$ is
\begin{equation}
    \mathbf{o}_i(\zeta)=|\zeta|\cdot (s_\zeta)_*\mathbf{o}_i = \frac{\zeta^+_{i}\circ\rho(\zeta)^-}{r_\mathbf{e}^4} + O(|\zeta|^4r_\mathbf{e}^{-8}).\label{dvp obst avec zeta}
\end{equation}
\begin{rk}\label{rem deformation eh zeta et noyau}
    We check that $\mathbf{o}_i(\zeta)=\partial_{t}\big|_{t=0}\mathbf{eh}_{\zeta+t\zeta_i} = d_{\zeta}\mathbf{eh}(\zeta_i)$, and even more concretely:
    \begin{equation*}
    \mathbf{o}_1(\zeta) = \frac{1}{2}\varepsilon \partial_{\varepsilon}g_{\varepsilon^2\phi}.
    \end{equation*}
	In other words, the variation of the metric $\mathbf{eh}_\zeta$ in the direction $\zeta_i$ at $\zeta$ corresponds to the $2$-tensor $\mathbf{o}_i$. In particular, one has
	\begin{equation}\label{eq representation O(eh) variations}
		\mathbf{O}(\mathbf{eh}_{\zeta}) = d_{\zeta}\mathbf{eh}(\mathbb{R}^3).
	\end{equation}
	This representation will be particularly useful to understand the dynamical content of the obstruction results of \cite{Biq-1}.
\end{rk}

\subsection{Canonical radial vector fields on Ricci flat ALE metrics} \label{sec-can-rad-vec-field}

We take the following conventions: on a Ricci-flat ALE orbifold $(N^4,g_b)$, there exists a unique harmonic vector field $X$ such that $X = r\partial_r +O (r^{-3})$ \cite{Biq-Hein,ozu2}. We also have $X =\nabla^{g_b} u$, where $u$ is the unique solution of 
\begin{equation}\label{poisson-eqn-ALE}
\Delta_{g_b} u = 4,\quad u = \frac{r^2}{2}+O\left((1+r)^{-2+\varepsilon}\right),\quad \text{ for all $\varepsilon>0$ as $r\rightarrow+\infty$}.
\end{equation}
See \cite{Biq-Hein,ozuthese} for a proof.

 We define
 \begin{equation}\label{eq:def o1 par u}
 \mathbf{o}_1 := g_b - \Hess_{g_b} u = g_b - \frac{1}{2} \mathcal{L}_X g_b.
 \end{equation}

We can always define a flow of metrics $ g_{b,\varepsilon} $ isometric to $\varepsilon g_b$ thanks to $\mathbf{o}_{1,\varepsilon} = \frac{\varepsilon}{2}\partial_\varepsilon g_{b,\varepsilon}$.

\begin{exmp}
    In the case when $g_b = \mathbf{eh}_\zeta$ is the Eguchi-Hanson metric as defined in the previous Section \ref{eh-section},
    we have $X = r\partial_r$ and $u = \frac{1}{2}\sqrt{\varepsilon^4+r^4}$.
\end{exmp}

\subsection{Obstruction to desingularizing Einstein orbifolds}

Let $g_o$ be an Einstein orbifold, and $p$ a singular point where from Remark \ref{rk: orb Ein local coord}, one has the development:
\begin{equation}\label{taylor-g0-devt}
g_o = \mathbf{e} + H_2+O_{\mathbf{e}}(r_o^3),
\end{equation}
where $H_2$ is a quadratic $2$-tensor in Bianchi gauge with respect to the flat metric $\mathbf{e}$ (which can always be assumed, see \cite{Mor-Via}), and $r_o$ is the distance to the origin for $\mathbf{e}$. {We consider the metrics $\mathbf{eh}_\zeta$ asymptotic to the same Euclidean metric $\mathbf{e} = dr_o^2+r_o^2(\alpha_1^2+\alpha_2^2+\alpha_3^2)$, and in particular $r=r_o$.}

Denote $\mathbf{R}_p^+$ the selfdual part of the curvature of $g_o$ at $p$ (note that it only depends on $H_2$).

\begin{prop}[\cite{Biq-1}]\label{prop-h2-H2}
    Up to the addition of an element of $\mathbf{O}(\mathbf{eh}_\zeta)$, there exists a unique solution $h_{2,\zeta}$ to: 
    \begin{equation*}
    \begin{split}
        d_{\mathbf{eh}_\zeta}\Ric (h_{2,\zeta})- \Lambda \mathbf{eh}_\zeta  &\in \mathbf{O}(\mathbf{eh}_\zeta),\\
        B_{\mathbf{eh}_\zeta}(h_{2,\zeta}) &= 0,     \\
        h_{2,\zeta} - \chi H_2 &  \in C^{2,\alpha}_{\beta-2},\quad\text{ for all $\beta>0$,}
        \end{split}
    \end{equation*}
    where $\chi$ is a cut-off function supported in $\{r>1000\sqrt{|\zeta|}\}$, $C^{2,\alpha}_{\beta-2}$ is defined in Example \ref{ex: holder a poids ale}, and where one has more precisely 
    \begin{equation}\label{hello-beauty}
    d_{\mathbf{eh}_\zeta}\Ric (h_{2,\zeta})- \Lambda \mathbf{eh}_\zeta = -d_\zeta\mathbf{eh}(\mathbf{R}_p^+(\zeta)) \in \mathbf{O}(\mathbf{eh}_\zeta).
    \end{equation}
    
    In particular, $|\nabla^{\eh _{\zeta},\,k}h_{2,\zeta}|_{\eh _{\zeta}}=O(r_o^{2-k})$ for all $k\geq 0$.
\end{prop}
\begin{proof}
	This result is reproved in Section \ref{sec-kro} for any hyperkähler ALE metric, and in particular, the constants in the above statement are verified precisely. 
	\\ 

	From \cite[Sections $2$ and $3$]{Biq-1},  the result is known for the metric $\mathbf{eh}$, i.e. with $\zeta = (1,0,0)$. We will take advantage of the fact that the other ones are homothetic to $\mathbf{eh}$ as follows. Denote $s_\zeta$ the above diffeomorphism such that $s_\zeta^*\mathbf{eh}_\zeta = |\zeta| \mathbf{eh}$. On the one hand, we have
\begin{equation*}
d_{\mathbf{eh}_\zeta}\Ric(h_{2,\zeta}) = \frac{1}{|\zeta|}(s_\zeta)_*(d_{\mathbf{eh}}\Ric (s_\zeta^*h_{2,\zeta})),
\end{equation*}
	and on the other hand, we have both $\mathbf{eh}_\zeta =  |\zeta|(s_\zeta)_* \mathbf{eh}$ and $\mathbf{o}_{k}(\zeta) = |\zeta|(s_\zeta)_* \mathbf{o}_{k}$. By uniqueness of the solution $h_2$ (up to $\mathbf{O}(\mathbf{eh})$) on $\mathbf{eh}$, and the scaling properties of the tensors with respect to $\zeta$, we therefore find that:
	\begin{equation*}
	\begin{aligned}
	h_{2,\zeta}&=|\zeta|^2(s_\zeta)_*h_2,&\quad \text{ and}\\
	 d_{\mathbf{eh}_\zeta}\Ric (h_{2,\zeta})- \Lambda \mathbf{eh}_\zeta &= -d_\zeta\mathbf{eh}(\mathbf{R}_p^+(\zeta)).
	 \end{aligned}
	 \end{equation*}
\end{proof}
\begin{rk}\label{rk-sum-oi-sol}
	If we slightly abusively call $\mathbf{R}^+_p\in\mathbb{R}^{3\times 3}$ the matrix representation of $\mathbf{R}^+_p$ in the basis $(\omega_i^+)_i$, then the renormalized Ricci flow $\partial_t g = -2(\Ric(g)-\Lambda g)$ approximately evolves the parameter $\zeta\in\mathbb{R}^3$ of the Eguchi-Hanson metric by:
	\begin{equation}
		\dot{\zeta} = 2\mathbf{R}^+_p (\zeta), \label{approximate eq rf}
	\end{equation}
	and therefore for some $\zeta_0\in\mathbb{R}^3$, one has:
	\begin{equation}
	\zeta(t) = \exp\left(2t\mathbf{R}^+_p\right) \zeta_0.\label{approximate Ricci flow}	
	\end{equation}
\end{rk}

This leads us to the following definitions.

\begin{defn}[Linear (in)stability of Einstein orbifold singularities of Einstein]\label{def-linear-instab}
	An Einstein orbifold $(M_o,g_o)$ is said to be \emph{orbifold point stable at a singular point $p$} if $\mathbf{R}^+_{g_o}(p) \leq 0$. It is orbifold point unstable otherwise.
\end{defn}

\begin{defn}[Strict linear (in)stability of Einstein orbifold singularities]\label{def-strict-linear-instab}
	An Einstein orbifold $(M_o,g_o)$ is said to be \emph{strictly orbifold point stable at a singular point $p$} if $\mathbf{R}^+_{g_o}(p) < 0$. It is strictly orbifold point unstable if one has $\mathbf{R}^+_{g_o}(p) > 0$ instead.
\end{defn}

\begin{rk}
	This notion of linear stability also has an eigenvalue interpretation thanks to the computations of \cite{no1}, where the spectrum of the Lichnerowicz Laplacian on a desingularization is shown to be asymptotically read off the curvature $\mathbf{R}^+_p$ of the orbifold.
\end{rk}

	We will spend most of the article constructing an ancient Ricci flow out of the typical example of an Einstein metric with \textit{orbifold point unstable} orbifold singularities: spherical orbifolds which satisfy $\mathbf{R}^+ = \operatorname{I}_{\Lambda^+}>0$.
	
\begin{exmp}[Spherical orbifolds with two singularities]\label{ex:orientatble spherical orb}
	For any finite subgroup $\Gamma$ of $\operatorname{SO(4)}$ acting freely on $\mathbb{S}^3$, a spherical orbifold with two singularities $\mathbb{R}^4/\Gamma$ is given as follows. Consider $\mathbb{S}^4\subset\mathbb{R}^5 = \mathbb{R}^4\times \mathbb{R}$, on which $\Gamma$ acts on the first factor $\mathbb{R}^4$. The resulting quotient of $\mathbb{S}^4$ by this action of $\Gamma$ has two singularities of type $\mathbb{R}^4/\Gamma$ at the fixed points of the action: the two poles.
\end{exmp}	
	
\begin{exmp}[Non orientable spherical orbifold with one singularity]\label{ex:waterdrop}
	One can construct a spherical orbifold with a single singularity $\mathbb{R}^4/\mathbb{Z}_2$ as follows. Consider $\mathbb{S}^4\subset\mathbb{R}^5 = \mathbb{R}^4\times \mathbb{R}$, where we use coordinates $(x,s)\in\mathbb{R}^4\times \mathbb{R}$. We define an equivalence relationship $\sim$ through $(x,s)\sim (-x,s)$ and $(x,s)\sim (x,-s)$ for any $(x,s)\in\mathbb{S}^4\subset\mathbb{R}^4\times \mathbb{R}$. This can be seen as a quotient of the smooth $\mathbb{RP}^4 = \mathbb{S}^4/\{(x,s)\sim(-x,-s)\}$. Indeed, one may quotient $\mathbb{RP}^4$ by the additional relationship $(x,s)\sim (-x,s)$ whose only fixed point on $\mathbb{RP}^4$ is $(0,\pm 1)$. There, one sees a $\mathbb{R}^4/\mathbb{Z}_2$ orbifold singularity. This ``water drop'' orbifold is also a quotient of the ``American football'' by an involution. Note that this additional quotient yields the same analysis as below.
\end{exmp}
Comparatively to spherical orbifolds, hyperbolic orbifolds with isolated singularities (which are orbifold point stable) are much more complicated to exhibit although they are abundant. We describe below a beautiful example due to Fine-Panov \cite{FinePanov2009}:
\begin{exmp}[A hyperbolic orbifold and its desingularizations]\label{ex:hyperbolic orbifold}
	Thanks to Fine-Panov \cite{FinePanov2009}, a hyperbolic orbifold with isolated singularities is obtained from a quotient of the simplest known hyperbolic $4$-manifold, the \textit{Davis manifold} \cite{Davis1985}. The Davis manifold $X$ is a complicated manifold described in \cite{RatcliffeTschantz2001} with $\chi(X) = 26$, $b_1 = 24$ and of course $\sigma(X)=0$. The simply connected quotient $M_o = X/\mathbb{Z}_2$ studied in \cite{Fine-Panov-2010} has $122$ fixed points corresponding to $\mathbb{R}^4/\mathbb{Z}_2$ singularities matching the asymptotic cone of Eguchi-Hanson metrics, let us number these fixed points as $p_1,\dots,p_{122}$.

	Consider $\theta = (\theta_1,\dots,\theta_{122})\in \{+,-\}^{122}$ and consider the manifold $M_\theta$ obtained by gluing Eguchi-Hanson metrics in the $\theta_i\in\{+,-\}$-orientation at $p_i$, there are clearly $2^{122}$ such choices, and we do not attempt to determine which ones lead to diffeomorphic or isometric constructions. Once such a $\theta$ is chosen, there remain $3\times {122}$ gluing parameters $(\zeta_i)_{1\,\leq\, i\,\leq\, 122}\in \mathbb{R}^{3\times 122}$ associated to the Eguchi-Hanson metrics. The resulting manifolds $M_{\theta}$ are simply connected and have the following characteristic numbers: denoting $S^\pm(\theta)\in [0,122]$ the number of $\theta_i$ equal to $\pm$, one finds $\chi(M_\theta) = 196$ and $\sigma(M_\theta) = S^-(\theta)-S^+(\theta)$, as well as $b_1 = 0$, $b_+ = 36+S^-(\theta)$, and $b_- = 36+S^+(\theta)$. 
\end{exmp}

\subsection{Obstruction to desingularizing orbifold Ricci solitons}

We now define a similar notion of (in)stability of more general self-similar solutions of Ricci flow: gradient Ricci solitons. 
\\

Gradient Ricci solitons are triplets $(M^n,g_o,\nabla^{g_o}f_o)$ satisfying $\Ric(g_o) + \Hess_{g_o}f_o = \Lambda g_o$. If such a soliton has an orbifold singularity at $p$ then the following development holds:
$$g_o = \mathbf{e} + H_2 +O(r_o^3), $$
and $f_o = cst + F_2 +...$, where $F_2$ is quadratic and solves $ \Delta_{\mathbf{e}} F_2 = \Lambda n - \R_{g_o} $. This implies that 
$$F_2 = \Big(\frac{\Lambda}{2} - \frac{\R_{g_o}(p)}{2n}\Big)r^2 +\frac{1}{2} \sum_{ij}f_{ij}x^T( \omega_i^+\circ\omega_j^-)x,$$
which implies that the Hessian of $F_2$ is:
$$ \Hess_{\mathbf{e}}F_2 = \Big(\Lambda - \frac{\R_{g_o}(p)}{n}\Big)\mathbf{e} + \sum_{ij}f_{ij}\omega_i^+\circ\omega_i^-.$$

We first prove an analogue of Proposition \ref{prop-h2-H2}.
\begin{prop}
    Up to the addition of an element of $\mathbf{O}(\mathbf{eh}_\zeta)$, there exists a unique solution $(h_{2,\zeta},f_2)$ to: 
    \begin{equation*}
    \begin{split}
        d_{\mathbf{eh}_\zeta}\Ric (h_{2,\zeta})  + \Hess_{{\mathbf{eh}_\zeta}} f_2 - \Lambda \mathbf{eh}_\zeta &= -d_\zeta\mathbf{\mathbf{eh}}(\overline{\mathbf{R}}^+_p(\zeta))\in\mathbf{O}(\mathbf{eh}_\zeta).\\
        B_{\mathbf{eh}_\zeta}(h_{2,\zeta}) &= 0,      \\
        h_{2,\zeta} - \chi H_2 &  \in C^{2,\alpha}_{\beta-2},\quad f_{2,\zeta} - \chi F_2\in \in C^{2,\alpha}_{\beta-2},\quad\text{ for all $\varepsilon>0$,}
        \end{split}
    \end{equation*}
    where $\chi$ is a cut-off function supported in $\{r>1000\sqrt{|\zeta|}\}$, $C^{2,\alpha}_{\beta-2}$ is defined in Example \ref{ex: holder a poids ale}.
\end{prop}
\begin{proof}
We prove the result for $\zeta = (1,0,0)$, i.e. for $\mathbf{eh}_\zeta =\mathbf{eh}$, the general result is obtained as in the proof of Proposition \ref{prop-h2-H2}.

The orbifold Ricci soliton metric has a development: $g_o=\mathbf{e}+ H_2+O(r_o^3)$ at its singular point and the potential: $f_o=f_o(p)+F_2 +O(r_o^3).$

We want to find $h_2 =  H_2 + O(r^{-2+\varepsilon})$ and $f_2 = F_2 + O(r^{-2+\varepsilon})$ for which we have:
$$d_{\mathbf{eh}}\Ric (h_2) +\Hess_{\eh }f_2 = \Lambda \eh , $$
and we will ask that $\Delta_{\eh }f_2 = n\Lambda - \R_{g_o}(p)$ for this, let $\chi$ be a cut-off function supported in the neighborhood of the infinity of $\eh $. 

Let us first find a function $f_2$ as $f_2=\chi (f_o(p)+F_2)+ f'$ with $f' = O(r^{-2+\varepsilon})$ for any $\varepsilon>0$. Since the kernel of $\Delta_{\eh }$ on bounded functions is reduced to constant functions and since $\Delta_{\eh }(\chi (f_o(p)+F_2))-(n\Lambda - \R_{g_o}(p) )= O(r^{-4})$ since $\eh-\mathbf{e}=O(r^{-4})$ and $F_2=O(r^2)$ together with derivatives, it is standard by elliptic theory in weighted Hölder spaces that such an $f'$ exists and is unique.

We are then looking for a symmetric $2$-tensor $h'$ decaying at infinity such that:
$$d_{\mathbf{eh}}\Ric(h')  = -d_{\mathbf{eh}}\Ric (\chi H_2)- \Hess_{\eh }f_2 + \Lambda \eh .$$
As already seen in \cite{Biq-1,ozu2}, this can rarely be solved and one needs to add an element of the cokernel $\mathbf{O}(\eh )$. To determine it, one considers the integration by parts of
$$\left\langle-d_{\mathbf{eh}}\Ric (\chi H_2)- \Hess_{\eh }f_2 + \Lambda \eh , \mathbf{o}_\eh \right\rangle_{L^2(\eh )},$$
for some element $\mathbf{o}_\eh  = O^4_\eh  + O(r^{-5})\in \mathbf{O}(\eh )$, where $O^4_\eh=O(r^{-4})$. Compared to the Einstein case discussed above, the only new term in the integral is 
\begin{align*}
-\left\langle \Hess_\eh f_2, \mathbf{o}_\eh \right\rangle_{L^2(\eh )} &= \left\langle \delta^*_{\eh } \nabla^{\eh } f_2, \mathbf{o}_\eh \right\rangle_{L^2(\eh )}\\
&= \lim_{r\to \infty} \int_{\{r_\eh  = r\}} *_{\eh }(\mathbf{o}_{\eh }(\nabla^{\eh }f_2))\\
&= \lim_{r\to \infty} \int_{\{r_\eh  = r\}} \mathbf{o}^{\eh }(\nabla^{\eh }f_2, n )\,d\sigma_{\{r_\eh  = r\}}\\
&= \int_{\mathbb{S}^3\slash\Gamma} O^4_\eh (\nabla^{\mathbf{e}}F_2, \partial_r ) \,d\sigma_{\mathbb{S}^3\slash\Gamma},
\end{align*}
where we used that the elements of $\mathbf{O}(\mathbf{eh})$ are divergence-free. Note that the trace part of $\Hess_\eh  f_2$ does not contribute since $\tr_\eh \mathbf{o}_\eh =0$. 

Let us analyze this boundary term more carefully and actually show that it vanishes. As seen in \cite{Biq-Hein}, in a well-chosen gauge, $r^4O^4_\eh $ is a linear combination of  $\theta_k^{\mp}\circ \omega_l^{\pm}$ where we recall that $ \theta_k^{\mp} = \sum_m \langle \omega_k^\pm(\partial_r),\omega_m^\mp(\partial_r)  \rangle \omega_m^\mp $.  Similarly as seen at the beginning of the section, $F_2$ is linear combination of $\langle \omega_i^\pm(r\partial_r),\omega_j^\mp(r\partial_r)  \rangle$, hence $\nabla^{\mathbf{e}}F_2$ is a linear combination of $\omega_i^\pm\circ\omega_j^\mp(r\partial_r)$. 

Consequently, $O^4_\eh (\nabla^{\mathbf{e}}F_2, \partial_r )$ is a linear combination of 
$$\langle \omega_k^\pm(\partial_r),\omega_m^\mp(\partial_r)  \rangle\cdot \langle\omega_m^\mp\circ \omega_l^{\pm}\circ\omega_i^\pm\circ\omega_j^\mp(r\partial_r),r\partial_r\rangle.$$
This can only have a nonzero integral on $\mathbb{S}^3$ if $\omega_l^{\pm}\circ\omega_i^\pm=\omega_k^\pm$ and $\omega_m^\mp\circ\omega_j^\mp = \omega_m^\mp $. This last equality is however never satisfied and the boundary term vanishes.
\end{proof}

 We identified non-vanishing obstructions to desingularizing Ricci soliton orbifolds. By analogy with the Einstein case, we conjecture that extending the results of \cite{Biq-1,ozu1,ozu2,ozu4}, one should similarly obtain obstructions to desingularizing many Ricci solitons. Our result essentially recovers Biquard's obstruction in \cite{Biq-1}.
\begin{conj}
	Let $(M_o,g_o)$ be a Ricci soliton that is compact or with controlled asymptotics (cylindrical or conical for instance). Then, if $\det\overline{\mathbf{R}}^+\neq 0$, then it cannot be limit of smooth Ricci solitons in the pointed Gromov-Hausdorff sense. 
\end{conj}

This would rule out the formation of spherical orbifold singularities as well as orbifold cylinder singularities.

\begin{rk}
	The evolution along Ricci flow at leading order should therefore be: $\dot{\zeta}(t) = 2\overline{\mathbf{R}}^+_p(\zeta(t))$.

\end{rk}

\begin{defn}[Linear (in)stability of Ricci soliton orbifold singularities]\label{def-linear-instab-sol}
	A Ricci soliton orbifold $(M_o,g_o)$ is said to be \emph{orbifold point stable at a singular point $p$} if $\overline{\mathbf{R}}^+_{g_o}(p) \leq 0$. It is orbifold point unstable otherwise.
\end{defn}

\begin{defn}[Strict linear (in)stability of orbifold singularities]\label{def-strict-linear-instab-sol}
	A Ricci soliton orbifold $(M_o,g_o)$ is said to be \emph{strictly orbifold point stable at a singular point $p$} if $\overline{\mathbf{R}}^+_{g_o}(p) < 0$. It is strictly orbifold point unstable if one has $\overline{\mathbf{R}}^+_{g_o}(p) > 0$ instead.
\end{defn}

{\begin{exmp}[Orbifold cylinders]
	One can also construct orbifolds with isolated singularities by quotienting cylinders. For instance, let us consider the cylinders $\mathbb{S}^k\times \mathbb{R}^{4-k}= \{(x_1,x_2,x_3,x_4,x_5), x_1^2+\dots+x_{k+1}^2=1\}\subset \mathbb{R}^5 $ for $ k\in\{0,1,2,3,4\}$. We may construct orbifolds with isolated singularities by the identification $(x_1,x_2,x_3,x_4,x_5)\sim (x_1,-x_2,-x_3,-x_4,-x_5)$ with fixed points $ (1,0,0,0) $ and $ (-1,0,0,0) $ as fixed points. 
	
	A further identification of $(x_1,x_2,x_3,x_4,x_5)\sim (-x_1,-x_2,-x_3,-x_4,-x_5)$ leaves only one singularity. These solitons are strictly \textit{orbifold point unstable} at their singularities.
\end{exmp}

\begin{exmp}[Examples of strictly orbifold point stable orbifold Ricci solitons]
	There are further steady soliton orbifolds that can be constructed along Ricci flow as follows. Bryant's $4$-dimensional steady soliton, Bryant's $3$-dimensional steady soliton and the cigar soliton times $\mathbb{R}^2$ are topologically $\mathbb{R}^4$, and the metrics are obviously invariant by the action of $x\mapsto -x$ on $\mathbb{R}^4$. Taking the quotients by this action yields orbifolds with an isolated singularity of type $\mathbb{Z}_2$ which is \textit{orbifold point stable}.
		
	One can similarly construct expanding solitons orbifolds from quotients of $k$-dimensional hyperbolic metrics equipped with an isometric involution time $\mathbb{R}^{4-k}$ for $k\in\{2,3\}$. 
\end{exmp}
}

\begin{rk}
	The analysis of the present article is local in the orbifold regions and can probably be extended to noncompact solitons. We do not construct ancient or immortal flows from solitons. 
\end{rk}


In the table below we summarize some of the $4$-dimensional gradient Ricci soliton orbifolds with isolated singularities that are strictly orbifold point (un)-stable in the sense of Definitions \ref{def-strict-linear-instab} and \ref{def-linear-instab-sol}. This is not exhaustive, and we assume that $\Gamma\subset \operatorname{SU}(2)$. Our main point is that once a soliton is known, computing its curvature and potential at a single (singular) point is sufficient to determine the stability.
\begin{center}
    \begin{tabular}{ |l|l| }
    \hline
 \textbf{strictly orbifold point Stable} ($\overline{\mathbf{R}}^+<0$)  & \textbf{strictly orbifold point Unstable}  ($\overline{\mathbf{R}}^+>0$)\\
 \hline
 (Complex) hyperbolic orbifolds & Spherical orbifolds \\
 Selfdual Einstein and $\operatorname{Scal}<0$ &Selfdual Einstein and $\operatorname{Scal}>0$ \\
$\RR^4/\Gamma$ Gaussian expander & $\RR^4/\Gamma$ Gaussian shrinker\\
$(\mathbb{R}\times\mathbb{H}^3)/\mathbb{Z}_2$ & $(\mathbb{R}\times\mathbb{S}^{3})/\mathbb{Z}_2$\\
$(\mathbb{R}^2\times\mathbb{H}^2)/\mathbb{Z}_2$ &$(\mathbb{R}^2\times\mathbb{S}^{2})/\mathbb{Z}_2$ \\
  Bryant's steady soliton$/\Gamma$ & FIK compact shrinking soliton\\
Bryant/Deruelle/Schulze-Simon expanders&\\
$(\mathbb{R}\times$Bryant's steady$)/\mathbb{Z}_2$ & \\
 \hline
\end{tabular}
\end{center}

\begin{rk}
	Since the constructions of Einstein $4$-manifolds with negative scalar curvature from \cite{Anderson2006,Biquard2006,Bamler2012,FinePremoselli2020} is purely local, it is likely that the constructions of Einstein $4$-manifolds from hyperbolic $4$-manifolds can yield Einstein $4$-orbifolds with isolated singularities if suitable hyperbolic $4$-orbifolds exist. They would be strictly orbifold point stable with $\overline{\mathbf{R}}^+<0$ at its singular points.
\end{rk}

Another large source of expanding soliton is given in \cite{BamlerChen2023}.
\begin{ques}
	Which of the orbifold expanding solitons produced by \cite{BamlerChen2023} are orbifold point unstable?
\end{ques}

\subsection{Extension of tensors on Einstein orbifolds}

The next proposition shows how to extend canonically tensors arising from the Taylor expansion at infinity \eqref{dvp eh} on an Einstein orbifold.
\begin{prop}[{\cite[Section 10]{Biq-1}\cite[Proposition 2.10]{ozu3}}]\label{premiere obst orbifold 1 pt}
    Let $(M_o,g_o)$ be a compact Einstein orbifold with $\Ric(g_o) = \Lambda g_o$, and denote $\mathbf{O}(g_o)$ the kernel of $\Delta_{L,g_o}+2\Lambda$.

    For $\zeta \in \mathbb{R}^3\setminus\{0\}$ and the symmetric $2$-tensor $H^4_{\zeta}$ of \eqref{dvp eh} with $B_\mathbf{e}H^4 = 0$ and $tr_\mathbf{e}H^4=0$, there exists a symmetric $2$-tensor $ {h}^4_\zeta$ on $M_o$ satisfying:
\begin{equation}
  \left\{
      \begin{aligned}
        (\Delta_{L,g_o}+2\Lambda){h}^4_\zeta & = 0,\\
        \tr_{g_o}h^4_\zeta &= 0,\\
        \mathrm{div}_{g_o}h^4_\zeta &= O(|\zeta|^2r_o^{-3-}),\\
        {h}^4_\zeta &= H^4_\zeta  + O(|\zeta|^2r_o^{-2-}) \text{ at the singular point }.
      \end{aligned}
    \right.
\end{equation}
\end{prop}
\begin{note}
	Here and in the rest of the article the notation $\mathrm{div}_{g_o}h^4_\zeta = O(|\zeta|^2r_o^{-3-})$ means that for any $c>0$, there exists $C>0$ such that $|\mathrm{div}_{g_o}h^4_\zeta|\leq C|\zeta|^2r_o^{-3-c}$, see also the Appendix \ref{sec:notations}. 
\end{note}
\begin{proof}
	The analysis of \cite[Section 10]{Biq-1} provides the existence of a tensor ${h}^4_\zeta$ with $(\Delta_{L,g_o}+2\Lambda){h}^4_\zeta = 0$ as well as ${h}^4_\zeta = H^4_\zeta  + O(|\zeta|^2r_o^{-2-})$. It may be obtained more concretely as follows.
	
	Consider $\chi$ to be a cut-off function supported close to the singular point of $M_o$, and $(\chi H^4_\zeta)^0_{g_o}$, the traceless part of $H^4_\zeta$ for the metric $g_o$. It satisfies $(\chi H^4_\zeta)^0_{g_o} = H^4_\zeta + O(|\zeta|^2r_o^{-2})$ at the singular point as well as 
	$$(\Delta_{L,g_o}+2\Lambda)(\chi H^4_\zeta)^0_{g_o} = O(|\zeta|^2 r_o^{-4}),$$
	where the right-hand side is $g_o$-traceless since $(\Delta_{L,g_o}+2\Lambda)$ preserves traceless symmetric $2$-tensors.

	We may now use the analysis of \cite[Section 10]{Biq-1} to find a traceless symmetric $2$-tensor $h' = O(|\zeta|^2r_o^{-2-})$ such that $(\Delta_{L,g_o}+2\Lambda)((\chi H^4_\zeta)^0_{g_o}+h')=0$.
\end{proof}

\section{First approximations}\label{sec-first-approx}

We now present a curve of metrics that will be an approximate solution to the normalized Ricci flow which we will perturb into an actual Ricci flow in Section \ref{sec-fixed-point}. Note that all of the gluing constructions and estimates are performed in the case of the Eguchi-Hanson metric, but easily extends to the case of general Ricci-flat ALE metrics.

\subsection{Definition of the approximate metric}
Let $(M_o^4,g_o)$ be a $4$-dimensional Einstein orbifold. The Einstein constant is denoted by $\Lambda\in\mathbb{R}$, i.e. $\Ric(g_o)=\Lambda g_o$ on the regular part of $M_o$. We assume from now on that $M_o$ has one single orbifold singularity modeled on $\RR^4/\mathbb{Z}_2$ at a point $p_o$. A finite number of singularities can be handled analogously.

In the coordinates of Remark \ref{rk: orb Ein local coord}, at the point $p_o$, for $r_o < \delta_0$,
\begin{equation}\label{eq go at po}
    g_o=dr_o^2 + r_o^2 g_{\mathbb{S}^3/\mathbb{Z}_2}+H_2(r_o)+O(r_o^4).
\end{equation}
\begin{rk}\label{rk-kron-sec} 
Notice that the previous estimate \eqref{eq go at po} is $O(r_o^3)$ if the metric is $\Gamma$-invariant for $\Gamma$ a finite subgroup of $\operatorname{SU}(2)$. In order to reach the same degree of precision with arbitrary hyperkähler ALE metrics, we will need to extend the cubic terms of the expansion in Section \ref{sec-kro}.
\end{rk}

  From now on, $r$ will denote any positive smooth extension of the radial coordinate on $(T^*\mathbb{S}^2=:N,\eh_\zeta =:g_\zeta)$ as in \eqref{eh zeta+} (which is independent of $\zeta$). In such coordinates, thanks to Proposition \ref{differents zeta} for $r>1000 \sqrt{|\zeta|}$ we have 
  \begin{equation}\label{eq gzeta at inft}
      \mathbf{eh}_\zeta = g_\zeta=dr^2 + r^2 g_{\mathbb{S}^3/\mathbb{Z}_2}+|\zeta|^2H^4(r_o)+O(|\zeta|^4r^{-8}),
  \end{equation}
  Again, the above estimate \eqref{eq gzeta at inft} is $O(|\zeta|^\frac{5}{2}r^{-5})$ if the metric is $\Gamma$-invariant for $\Gamma$ a finite subgroup of $\operatorname{SU}(2)$. 
  
  Let $\varepsilon>0$ and $\delta>0$ be small parameters such that to be determined later. We define the following subsets:
  \begin{equation*}
\begin{split}
N^{\delta}&:=\{r< 2\delta^{-1}\}\subset N,\\
M_o^{\delta}&:= \{r_o>\delta\}\subset M_o.
\end{split}
\end{equation*}


{ Let $M^{\varepsilon,\delta}=M$ be the manifold obtained as the union of $N^{\varepsilon/\delta}$ and $M_o^{\delta/2}$ along the identification of the annuli of the asymptotic cones through $1000\sqrt{|\zeta|}<r = r_o<\delta_0$ in \eqref{eq go at po} and \eqref{eq gzeta at inft}, and the identity on ${\mathbb{S}^3/\mathbb{Z}_2}$. }

Let $\chi:\RR\rightarrow[0,1]$ be an even smooth cut-off function satisfying $\chi\equiv 1$ on $(-\infty,1/2]$ and $\chi\equiv 0$ on $[2,+\infty)$. Let $\chi_{\delta(t)}:=\chi(r_o/\delta)$. The first approximation background metric is defined as:
 \begin{equation}\label{defn-tilde-g}
\tilde{g}:=\chi_\delta(g_\zeta + h_{2,\zeta})+(1-\chi_\delta)\left(g_o+h_{\zeta}^4\right),
\end{equation}
where $g_{\zeta}:=\eh _{\zeta}$ is defined in \eqref{Def-attaching-map}, where $h_{2,\zeta}$ is given by Proposition \ref{prop-h2-H2} and where $h^4_{\zeta}$ is ensured by Proposition \ref{premiere obst orbifold 1 pt}. 

\begin{rk}
In the sequel, {when dealing with ancient flows,} $\varepsilon$ and $\delta$ are smooth functions of time $t\in(-\infty,t_0)$ for $t_0$ negative. In that case, we denote $(\tilde{g}(t))_{t<t_0}$ to be the curve of metrics defined for each time $t<t_0$ by $\tilde{g}$ with parameters $\varepsilon(t)$ and $\delta(t)$. {Similarly, when dealing with immortal flows, we will define $\tilde{g}(t)$ in the same way for $t\in(t_0,+\infty)$ for $t_0>0$. Many of our estimates will be proved for both ancient and immortal flows at the same time, since up to flipping the sign of $t$, the proofs are exactly the same: ODE's are time-reversible. We will deal with ancient and immortal flows separately when the proofs are substantially different: parabolic flows are not time-reversible.}
\end{rk}
\textbf{From now on, define 
\begin{equation*}
\varepsilon_0(t):=\exp\left\{\frac{\Lambda_0}{2} t\right\},\quad t\leq T, \quad\text{respectively $\varepsilon_0(t):=\exp\left\{-\frac{\Lambda_0}{2} t\right\},\quad t\geq T,$}
\end{equation*}
 for some $\Lambda_0>0$ to be specified in Section \ref{sec-nonlin} and we will assume the following conditions:
\begin{equation}\label{cond-param}
  \tag{$\mathcal{H}_0$}
  \begin{split}
&C^{-1}\varepsilon_0(t)\leq |\zeta(t)|\leq C\varepsilon_0(t),\quad|\dot{\zeta}(t)|+[\dot{\zeta}]_{\alpha,t,\varepsilon_0(t)}\leq C|\zeta(t)|,\\
 &|\delta'(t)|\leq C\delta(t)\quad\text{and}\quad C^{-1}\varepsilon(t)^{\delta}\leq \delta(t)\leq  C\varepsilon(t)^{\delta},\quad\delta\in\left(\frac{5}{9}\,,\,\frac{3-\sigma}{5-\sigma}\right),
 \end{split}
\end{equation}
for some positive constant $C$ and all $t\leq T$ {in the ancient case, and for all $t\geq T$ for immortal flows} and some fixed $\sigma\in(0,1/2)$. Here the semi-norm $[\,\cdot\,]_{\alpha,t,r}$ is defined in Definition \ref{defn-semi-norm-holder-fct}.
 In particular, since $\varepsilon(t)^2=|\zeta(t)|$, one gets from \eqref{cond-param} that $$|\varepsilon'(t)|+[\varepsilon']_{\alpha,t,\varepsilon_0(t)}\leq C\varepsilon(t).$$
 Finally, define the parabolic weight: $$\rho(x,t):=\varepsilon(t)+r_o(x),\quad (x,t)\in M\times(-\infty,T],\quad\text{respectively $(x,t)\in M\times[T,+\infty)$.}$$}

\begin{note}
 	The intervals $t\leqslant T$ and $t\geqslant T$ respectively refer to the ancient and immortal cases.  
\end{note} 
 
From now on, we will omit the spatial dependency of the weight function defined above but we will most of the time keep the time dependency and will write $\rho(t)$ to denote the function $\rho(\cdot,t)$.

The curve $(\tilde{g}(t))_{t\,\leq \,T}$ satisfies the following estimates:

\begin{lemma}\label{lemma-first-app-est}
For $k\in\mathbb{N}$, there exists $C_k>0$ such that on $\{\delta(t)/2\,\leq\,r_o\leq 2\delta(t)\}$,
\begin{equation*}
\left|\nabla^{g_o,\,k}\left(\tilde{g}(t)-\left(g_{\zeta(t)}+h_{2,\zeta(t)}\right)\right)\right|_{g_o}\leq C_k\,\varepsilon(t)^{5}\delta(t)^{-5-k}.
\end{equation*}
In particular, for all $k\geq 0$, on $\{\delta(t)/2\,\leq\,r_o\leq 2\delta(t)\}$,
\begin{equation*}
\begin{split}
\left|\nabla^{\tilde{g}(t),\,k}\left(\left(g_o+h^4_{\zeta(t)}\right)-\left(g_{\zeta(t)}+h_{2,\zeta(t)}\right)\right)\right|_{\tilde{g}(t)}&\leq C_k\,\varepsilon(t)^{5}\delta(t)^{-5-k},\\
|\nabla^{\tilde{g}(t),\,k}h_{2,\zeta(t)}|_{\tilde{g}(t)}&\leq C_k\,\delta(t)^{2-k}.
\end{split}
\end{equation*}
 
 Finally, for each $k\geq 0$, there exists $C_k>0$ such that on $M\times(-\infty,T]$,
 \begin{equation*}
 \begin{split}
\rho(t)^{2+k}|\nabla^{\tilde{g}(t),\,k}\Rm(\tilde{g}(t))|_{\tilde{g}(t)}\leq C_k.
 \end{split}
\end{equation*}
\end{lemma}
\begin{rk}\label{rk autre singularite a ref}
In the first two estimates of Lemma \ref{lemma-first-app-est}, one gets an error $O(r_o^{3-k})$ if $g_o=\mathbf{e}+H_2+O(r_o^3)$ if the group of the singularity is not even. This requires extending cubic terms of the orbifold metric to the ALE metric for general Ricci-flat ALE metrics in Section \ref{sec-kro}.
\end{rk}
\begin{proof}

Since by [\eqref{dvp eh}, Proposition \ref{differents zeta}] and Proposition \ref{premiere obst orbifold 1 pt},
\begin{equation*}
\begin{split}
g_{\zeta(t)}&=\mathbf{e}+H^4_{\zeta(t)}+O(\varepsilon(t)^5\rho(t)^{-5}),\\
h_{\zeta(t)}^4&=H^4_{\zeta(t)}+O(\varepsilon(t)^4\rho(t)^{-2^{-}}),
\end{split}
\end{equation*}
one has,
\begin{equation*}
\begin{split}
g_{\zeta(t)}&=\mathbf{e}+h_{\zeta(t)}^4+O(\varepsilon(t)^5\rho(t)^{-5}+\varepsilon(t)^4\rho(t)^{-2^{-}}).
\end{split}
\end{equation*}
Therefore, by definition of $\tilde{g}(t)$ given in \eqref{defn-tilde-g} together with Proposition \ref{prop-h2-H2}, the following holds on $\{\delta(t)/2\,\leq\,r_o\,\leq\,2\delta(t)\}$,
\begin{equation*}
\begin{split}
\left|\tilde{g}(t)-\left(g_{\zeta(t)}+h_{2,\zeta(t)}\right)\right|_{\mathbf{e}}&=(1-\chi_{\delta(t)})\left|\left(g_0+h^4_{\zeta(t)}-\left(g_{\zeta(t)}+h_{2,\zeta(t)}\right)\right)\right|_{\mathbf{e}}\\
&\leq (1-\chi_{\delta(t)})\Bigg|\mathbf{e}+H_2+O(\rho(t)^4)+H^4_{\zeta(t)}+O(\varepsilon(t)^4\rho(t)^{-2^-})\\
&\quad-\left(\mathbf{e}+H^4_{\zeta(t)}+O(\varepsilon(t)^5\rho(t)^{-5})+H_2+O(\varepsilon(t)^{4^-}\rho(t)^{-2^-})\right)\Bigg|_{\mathbf{e}}\\
&= O(\delta(t)^4)+O(\varepsilon(t)^4\delta(t)^{-2^-})+O(\varepsilon(t)^5\delta(t)^{-5})\\
&=O(\varepsilon(t)^5\delta(t)^{-5}),
\end{split}
\end{equation*}
where we have used $\delta>5/9$ in the last line. One imposes $\delta<(3-\sigma)/(5-\sigma)<5/7$ to get a better error compared to $|h_{2,\zeta(t)}|_{g_{\zeta(t)}}=O(\rho(t)^2)$.

Higher covariant derivatives with respect to $g_o$ (and then with respect to $\tilde{g}(t)$) can be obtained similarly.

The last estimate on the covariant derivatives of $h_{2,\zeta(t)}$ follows from Proposition \ref{prop-h2-H2}, the previous estimate and the schematic formula for two Riemannian metrics $g_i$, $i=1,2$ and a tensor $T$:
$(\nabla^{g_2}-\nabla^{g_1})T= g_2^{-1}\ast T\ast\nabla^{g_1}(g_2-g_1)$.

As for the curvature estimates, observe that on the set $\{r_o>2\delta(t)\}$, $\tilde{g}(t)=g_o+h^4_{\zeta(t)}$. Since $g_o$ is time independent, its curvature tensor stays bounded as $t$ tends to $-\infty$. Lemma \ref{gal-lemma-diff-curv} together with Proposition \ref{premiere obst orbifold 1 pt} give for $t\leq T$:
\begin{equation*}
\begin{split}
|\Rm(g_o+h_{\zeta(t)}^4)|_{g_o}&\leq |\Rm(g_o)|_{g_o}+C\varepsilon(t)^4\delta(t)^{-6},
\end{split}
\end{equation*}
which is bounded since $\delta<(3-\sigma)/(5-\sigma)<2/3.$

 On the set $\{r_o<\delta(t)/2\}$, $\tilde{g}(t)=g_{\zeta(t)}+h_{2,\zeta(t)}$ so that on the one hand,
\begin{equation*}
\begin{split}
|\Rm(g_{\zeta(t)})|_{g_{\zeta(t)}}&=|\zeta(t)|^{-1}(s_{\zeta(t)})_{*}|\Rm(g)|_{g}\leq |\zeta(t)|^{-1}(s_{\zeta(t)})_{*}\left(\frac{C}{(1+r_b)^4}\right)\\
&\leq \frac{C\varepsilon(t)^2}{(\varepsilon(t)+r_o)^4}\leq  \frac{C}{(\varepsilon(t)+r_o)^2} ,
\end{split}
\end{equation*}
and on the other hand, Lemma \ref{gal-lemma-diff-curv} gives for $t\leq T$:
\begin{equation*}
\begin{split}
|\Rm(g_{\zeta(t)}+h_{2,\zeta(t)})-\Rm(g_{\zeta(t)})|_{\tilde{g}(t)}&\leq C\left(|\nabla^{g_{\zeta(t)},2}h_{2,\zeta(t)}|_{\tilde{g}(t)}+|\nabla^{g_{\zeta(t)}}h_{2,\zeta(t)}|_{\tilde{g}(t)}^2\right)\leq C.
\end{split}
\end{equation*}
Here we have used Proposition \ref{prop-h2-H2} in the last line and $C$ is a positive constant that may vary from line to line. By invoking Lemma \ref{gal-lemma-diff-curv} twice on the regions ${\{r_o<3\delta(t)/2\}}$ and ${\{r_o>\delta(t)\}}$ together with the convergence rates obtained previously on $\tilde{g}(t)-(g_{\zeta(t)}+h_{2,\zeta(t)})$, one gets the expected estimate on $|\Rm(\tilde{g}(t))|_{\tilde{g}(t)}.$ The corresponding estimates on the covariant derivatives of $\Rm(\tilde{g}(t))$ can be obtained similarly.

\end{proof}

The next lemma measures up to first order how far the approximation $\tilde{g}(t)$ is from being Einstein.
\begin{lemma}\label{rough-est-deviation-einstein}

For all $k\geq 0$, there exists $C_k>0$ such that for $t\leq T$,
\begin{equation*}
\begin{split}
&\left|\nabla^{\tilde{g}(t),\,k}\left(-\Ric(\tilde{g}(t))+\Lambda \tilde{g}(t)-\chi_{\delta(t)} d_{\zeta(t)} g(\mathbf{R}^+_p(\zeta(t)))+(1-\chi_{\delta(t)})\frac{1}{2}\Li_{B_{g_o}(h_{\zeta(t)}^4)}(g_o)\right)
\right|_{\tilde{g}(t)}\leq\\
& C_k\,\rho(t)^{2-k}\mathbbm{1}_{\{r_o\,\leq \,2\delta(t)\}}+C_k\,\varepsilon(t)^5\delta(t)^{-7-k}\mathbbm{1}_{\{\delta(t)/2\,\leq\,r_o\,\leq \,2\delta(t)\}}
+C_k\,\varepsilon(t)^8\rho(t)^{-10-k}\mathbbm{1}_{\{2\delta(t)\,\leq\,r_o\}}.
\end{split}
\end{equation*}
\end{lemma}

\begin{rk}
This estimate is reminiscent of \cite[Lemma $4.2$]{Biq-1}. However, one has to refine the contribution of the obstruction tensor $-\Ric(\tilde{g}(t))+\Lambda \tilde{g}(t)$ to be an exact Einstein metric on the orbifold here. The $L^2$-estimates in Section \ref{sec:proj kernel} explains the need for this refinement. This is particularly true in the proofs of Propositions \ref{prop-proj-tilde-g} and \ref{prop-proj-conf}.
\end{rk}

\begin{proof}
Recall the following property shared by the linearization of $-\Ric$ at a metric $g$ along a variation $h$ such that $|h|_g\leq \frac{1}{2}$, as recalled in Lemma \ref{app-C-monster}:
\begin{equation}
\left|\Ric(g+h)-\Ric(g)-d_g\Ric(h)\right|_g\leq C\left(|h|_g|\nabla^{g,2}h|_g+|\nabla^gh|^2_g+|\Rm(g)|_{g}|h|^2_{g}\right),\label{lin-ric-love}
\end{equation}
for some positive constant $C$ depending on the dimension only. In particular, on the region where $\chi_{\delta(t)}\equiv 1$, one gets with the help of \eqref{lin-ric-love} and Lemma \ref{lemma-first-app-est}:

\begin{equation*}
\begin{split}
&\left|-\Ric(\tilde{g}(t))+\Lambda \tilde{g}(t)-\chi_{\delta(t)} d_{\zeta(t)} g(\mathbf{R}^+_p(\zeta(t)))\right|_{\tilde{g}(t)}\\
&=\left|-\Ric(g_{\zeta(t)}+h_{2,\zeta(t)})+\Lambda(g_{\zeta(t)}+h_{2,\zeta(t)})- d_{\zeta(t)} g(\mathbf{R}^+_p(\zeta(t)))\right|_{\tilde{g}(t)}\\
&\leq \left|-\Ric(g_{\zeta(t)})-d_{g_{\zeta(t)}}\Ric(h_{2,\zeta(t)})+\Lambda(g_{\zeta(t)}+h_{2,\zeta(t)})- d_{\zeta(t)} g(\mathbf{R}^+_p(\zeta(t)))\right|_{\tilde{g}(t)}+Cr_o^2\\
&\leq Cr_o^2.
\end{split}
\end{equation*}
Here we have used the Ricci flatness of $g_{\zeta(t)}$ in the last line together with the equation satisfied by $h_{2,\zeta(t)}$ imposed in Proposition \ref{prop-h2-H2} and the fact that $\nabla^{g_{\zeta(t)},\,k} h_{2,\zeta(t)} = O((\varepsilon(t)+r_o)^{2-k})$.

Now, estimates on higher derivatives are obtained by linearizing as above together with Lemma \ref{lemma-first-app-est}. 

In the region $\{\delta(t)/4\,<\,r_o\,<\, \delta(t)\}$, we linearize the Ricci tensor around the metric $g_{\zeta(t)}$ with the help of Lemma  \ref{lemma-first-app-est}: if $\tilde{h}(t):=\tilde{g}(t)-g_{\zeta(t)}$ then \eqref{lin-ric-love} and Proposition \ref{premiere obst orbifold 1 pt} imply,
\begin{equation*}
\begin{split}
-\Ric(\tilde{g}(t))+\Lambda\tilde{g}(t)&=-\Ric(g_{\zeta(t)})+\Lambda g_{\zeta(t)}+\Lambda \tilde{h}(t)-d_{g_{\zeta(t)}}\Ric(\tilde{h}(t))+O(\delta(t)^2)\\
&=\Lambda g_{\zeta(t)}-d_{g_{\zeta(t)}}\Ric (h_{2,\zeta(t)})-d_{g_{\zeta(t)}}\Ric(\tilde{h}(t)-h_{2,\zeta(t)})+O(\delta(t)^{2})\\
&=-d_{g_{\zeta(t)}}\Ric(\tilde{h}(t)-h_{2,\zeta(t)})+O(\delta(t)^2)+O(\varepsilon(t)^4\delta(t)^{-4})\\
&=O(\varepsilon(t)^{5}\delta(t)^{-7})+O(\delta(t)^2)+O(\varepsilon(t)^4\delta(t)^{-4})=O(\varepsilon(t)^5\delta(t)^{-7}).
\end{split}
\end{equation*}
Here we have invoked $2/3>\delta>5/9$ thanks to \eqref{cond-param} and we have  used the fact that $h_{2,\zeta(t)}$ is in Bianchi gauge together with the fact that $|\oi_i(\zeta(t))|_{g_{\zeta(t)}}=O(\varepsilon(t)^4(\varepsilon(t)+r_o)^{-4})$ in the third equality.

Now, on the region $\{r_o> \delta(t)\}$, we linearize the Ricci tensor around the metric $g_o$ with the help of Lemma  \ref{lemma-first-app-est}: if $\tilde{h}(t):=\tilde{g}(t)-g_o$, since 
\begin{equation*}
d_{g_o}(2\Ric-2\Lambda)(h^4_{\zeta(t)})=\Li_{B_{g_o}(h_{\zeta(t)}^4)}(g_o)=O(\varepsilon(t)^4\rho(t)^{-4}),
\end{equation*}
thanks to Proposition \ref{premiere obst orbifold 1 pt}, then \eqref{lin-ric-love} implies,
\begin{equation*}
\begin{split}
-\Ric(\tilde{g}(t))+\Lambda\tilde{g}(t)&=-\Ric(g_o)+\Lambda g_o-d_{g_o}(\Ric-\Lambda)(\tilde{h}(t))+O(\varepsilon(t)^8\rho(t)^{-10})\\
&=-d_{g_o}(\Ric-\Lambda)(\tilde{h}(t)-h^4_{\zeta(t)})-\frac{1}{2}\Li_{B_{g_o}(h_{\zeta(t)}^4)}(g_o)+O(\varepsilon(t)^8\rho(t)^{-10})\\
&=-\frac{1}{2}\Li_{B_{g_o}(h_{\zeta(t)}^4)}(g_o)+O(\varepsilon(t)^{5}\delta(t)^{-7})\mathbbm{1}_{\{\delta(t)/2\,\leq\,r_o\,\leq \,2\delta(t)\}}
\\
&\quad+\varepsilon(t)^8\rho(t)^{-10}\mathbbm{1}_{\{2\delta(t)\,\leq\,r_o\}}.
\end{split}
\end{equation*}
Here we have invoked $1/3<5/9<\delta<2/3$ thanks to \eqref{cond-param}.


The equivalent estimate is obtained thanks to the combination of the previously established estimate together with Remark \ref{rk-sum-oi-sol} and \eqref{cond-param}.
\end{proof}

The last lemma of this section reveals the obstruction of the approximate solution $(\tilde{g}(t))_{t\,\leq \,T}$ to be an exact solution to the normalized Ricci flow:
\begin{lemma}\label{lemma-time-der-metric}
There exists $C>0$ such that the first time variation of $\tilde{g}(t)$ satisfies: 
\begin{equation*}
\begin{split}
&\left|\partial_t\tilde{g}(t)-\chi_{\delta(t)}d_{\zeta(t)}g(\dot{\zeta}(t))-(1-\chi_{\delta(t)})d_{\zeta(t)}h^4(\dot{\zeta})\right|_{\tilde{g}(t)}\leq\\
& C\varepsilon(t)^{4^-}\rho(t)^{-2^{-}}\mathbbm{1}_{\{r_o\,\leq \,2\delta(t)\}}+C\varepsilon(t)^{5}\delta(t)^{-5}\mathbbm{1}_{\{\delta(t)/2\,\leq\,r_o\,\leq \,2\delta(t)\}},
\end{split}
\end{equation*}
and for each $k\geq 1$, there exists $C_k>0$ such that for $t\leq T$,
\begin{equation*}
\left|\nabla^{\tilde{g}(t),\,k}\partial_t\tilde{g}(t)\right|_{\tilde{g}(t)}\leq C_k\,\varepsilon(t)^4\rho(t)^{-4-k}.
\end{equation*}
Moreover,
\begin{equation*}
\begin{split}
\left|\tr_{\tilde{g}(t)}(\partial_t\tilde{g}(t))\right|&\leq C\varepsilon(t)^{4^-}\rho(t)^{-2^{-}}\mathbbm{1}_{\{r_o\,\leq \,2\delta(t)\}}+C\varepsilon(t)^{5}\delta(t)^{-5}\mathbbm{1}_{\{\delta(t)/2\,\leq\,r_o\,\leq \,2\delta(t)\}}\\
&\quad+C\varepsilon(t)^{8}\rho(t)^{-8}\mathbbm{1}_{\{2\delta(t)\,\leq\,r_o\}}.
\end{split}
\end{equation*}
The obstruction to be a solution to the renormalized Ricci flow can be estimated as follows for all $k\geq0$:
\begin{equation*}
\begin{split}
\Bigg|\nabla^{\tilde{g}(t)\,,k}&\Bigg(-\partial_t\tilde{g}(t)-2\Ric(\tilde{g}(t))+2\Lambda\tilde{g}(t)+\chi_{\delta(t)}d_{\zeta(t)}g\left(\dot{\zeta}(t)-2\mathbf{R}_p^+(\zeta(t))\right)\\
&+(1-\chi_{\delta(t)})\left(d_{\zeta(t)}h^4(\dot{\zeta})+\Li_{B_{g_o}(h_{\zeta(t)}^4)}(g_o)\right)\Bigg)\Bigg|_{\tilde{g}(t)}\\
&\leq C\left(\rho(t)^{2-k}+\varepsilon(t)^{4^-}\rho(t)^{-2^{-}-k}\right)\mathbbm{1}_{\{r_o\,\leq \,2\delta(t)\}}+C\varepsilon(t)^{5}\delta(t)^{-7-k}\mathbbm{1}_{\{\delta(t)/2\,\leq\,r_o\,\leq \,2\delta(t)\}}\\
&\quad+C_k\,\varepsilon(t)^8\rho(t)^{-10-k}\mathbbm{1}_{\{2\delta(t)\,\leq\,r_o\}}.
\end{split}
\end{equation*}

\end{lemma}

\begin{proof}
By a straightforward differentiation in time, we see that:
\begin{equation}
\begin{split}\label{first-est-partial-g}
\partial_t\tilde{g}(t)&=\chi'_{\delta(t)}\frac{r_o\delta'(t)}{\delta(t)^2}\left(g_o+h^4_{\zeta(t)}-g_{\zeta(t)}-h_{2,\zeta(t)}\right)+\chi_{\delta(t)}\partial_t\left(g_{\zeta(t)}+h_{2,\zeta(t)}\right)+(1-\chi_{\delta(t)})\partial_th^4_{\zeta(t)}\\
&=\underbrace{\chi'_{\delta(t)}O(\varepsilon(t)^{5}\delta(t)^{-5})}_{\text{supported in $\{\delta(t)/2\leq r_o\leq 2\delta(t)\}$}} +\chi_{\delta(t)}\partial_t\left(g_{\zeta(t)}+h_{2,\zeta(t)}\right)+(1-\chi_{\delta(t)})\partial_th^4_{\zeta(t)},
\end{split}
\end{equation}
where we have used Lemma \ref{lemma-first-app-est} in the last line and the constraint $|\delta'(t)|\leq C\delta(t)$ for all $t$ { in the interval of interest} thanks to \eqref{cond-param}.

Now, by definition of the differential, $\partial_tg_{\zeta(t)}=d_{\zeta(t)}g(\dot{\zeta}(t))$ and thanks to Proposition \ref{prop-analy}, $|\partial_th_{2,\zeta(t)}|_{g_{\zeta(t)}}=O(|\dot{\zeta}(t)|\varepsilon(t)^{2^-}(\varepsilon(t)+r_o)^{-2^{-}})=O(\varepsilon(t)^{4^-}(\varepsilon(t)+r_o)^{-2^{-}}),$ thanks to \eqref{cond-param}. This gives the first estimate on $\partial_t\tilde{g}(t)$. 

The higher covariant derivatives of $\partial_t\tilde{g}(t)$ can be obtained by differentiating the above expression and by invoking Proposition \ref{prop-analy}. Notice here that the terms supported on the annulus are either of the form $O(\delta(t)^{4-k})$ or $O(\varepsilon(t)^4\delta(t)^{-4-k})$ since $|\partial_th^4_{\zeta(t)}|_{g_{\zeta(t)}}=O(|\dot{\zeta}(t)|\varepsilon(t)^{2}(\varepsilon(t)+r_o)^{-4^{}})=O(\varepsilon(t)^4\delta(t)^{-4})$ , $h^4_{\zeta(t)}$ depending quadratically on $|\zeta(t)|$. Condition \eqref{cond-param} ensures that it is $O(\varepsilon(t)^4\delta(t)^{-4-k})$.

As for the trace of $\partial_t\tilde{g}(t)$, it suffices to estimate $\tr_{\tilde{g}(t)}\left(\chi_{\delta(t)}d_{\zeta(t)}g(\dot{\zeta}(t))\right).$ Now, since $d_{\zeta(t)}g(\dot{\zeta}(t))$ lies in the kernel of the Lichnerowicz operator associated to the Ricci flat metric $g_{\zeta(t)}$, $\tr_{g_{\zeta(t)}}\left(d_{\zeta(t)}g(\dot{\zeta}(t))\right)=0$ thanks to Subsection \ref{basics-EH}. Therefore, Lemma \ref{lemma-first-app-est} ensures that 
\begin{equation*}
\tr_{\tilde{g}(t)}\left(\chi_{\delta(t)}d_{\zeta(t)}g(\dot{\zeta}(t))\right)=\chi_{\delta(t)}d_{\zeta(t)}g(\dot{\zeta}(t))\ast \left(\tilde{g}(t)-g_{\zeta(t)}\right)=O\left(|\dot{\zeta}(t)|\varepsilon(t)^2 (\varepsilon(t)+r_o)^{-2}\right),
\end{equation*}
by invoking \eqref{cond-param} once more. Proposition \ref{premiere obst orbifold 1 pt} ensures that $\tr_{g_o}\left(d_{\zeta(t)}h^4(\dot{\zeta})\right)=0$ so that $\tr_{\tilde{g}(t)}((1-\chi_{\delta(t)})d_{\zeta(t)}h^4(\dot{\zeta}))=O(\varepsilon(t)^{8}\rho(t)^{-8})$ on $\{r_o\,\geq\,\delta(t)\}$. This gives the expected estimate on $\tr_{\tilde{g}(t)}(\partial_t\tilde{g}(t)).$

As for the third estimate, it is a combination of Lemma \ref{rough-est-deviation-einstein} and the previous estimate \eqref{first-est-partial-g}.
\end{proof}

\subsection{Definition of the approximate kernel: rescalings and rotations}

We define $\tilde{\oi}_i(t)$ as the localized part of the tensor $ \oi_i$ appropriately rescaled:
\begin{equation}\label{oi-app}
\tilde{\oi}_i(t):=\chi_{\delta(t)}\cdot \oi_i(\zeta(t)).
\end{equation}


This definition follows the work of Biquard \cite{Biq-1} and that of Morteza-Viaclovsky \cite{Mor-Via}.
\begin{rk}
Alternatively, as in \cite{Bre-Kap}, we could have defined the tensor $\tilde{\oi}_1(t)$, for $t\leq T$, to be the trace-free part of the tensor $\frac{1}{2}\varepsilon\frac{\partial}{\partial \varepsilon}\tilde{g}$ at $\varepsilon=\varepsilon(t)$, i.e. the tensor:
\begin{equation*}
\frac{1}{2}\varepsilon\frac{\partial}{\partial \varepsilon}\tilde{g}-\frac{1}{4}\tr_{\tilde{g}(t)}\left(\frac{1}{2}\varepsilon\frac{\partial}{\partial \varepsilon}\tilde{g}\right)\tilde{g}(t).
\end{equation*}
The advantage of the  approximate kernel as defined in \cite{Bre-Kap} is that it is not only defined on a curve but also on the whole moduli space of possible gluings (accounting for the deformations of the ALE spaces and the orbifold) where the curve $(\tilde{g}(t))_t$ lives. Our choice in this paper seems to be handier for direct estimates whereas theirs seems to be better suited for a global understanding of the dynamics of the Ricci flow in the moduli space of glued metrics.
	
\end{rk}

 We make precise the fact that the tensors $\tilde{\oi}_i(t)$ can be interpreted as being in the approximate kernel of the Lichnerowicz Laplacian associated to $\tilde{g}(t)$ for $t\leq T$ in the next lemma:
 
\begin{lemma}\label{estimates variations et laplacien o1}
For each $k\geq 0$, there exists $C_k>0$ such that for $t\leq T$,
\begin{equation*}
\begin{split}
|\tr_{\tilde{g}(t)} \tilde{\oi}_i(t)| &\leq C\varepsilon(t)^4\rho(t)^{-2}\mathbbm{1}_{\{r_o\,\leq \,2\delta(t)\}},\\
\left|\nabla^{\tilde{g}(t)\,k}\Delta_{L,\tilde{g}(t)}\tilde{\oi}_i(t)\right|&\leq C_k\varepsilon(t)^4\rho(t)^{-4-k}\mathbbm{1}_{\{r_o\,\leq \,2\delta(t)\}}+C_k\varepsilon(t)^4\delta(t)^{-6-k}\mathbbm{1}_{\{\delta(t)/2\,\leq\,r_o\,\leq \,2\delta(t)\}},\\
\left| \nabla^{\tilde{g}(t),\,k}\tilde{\oi}_i(t)\right|_{\tilde{g}(t)}&\leq C_k\,\chi_{\delta(t)}\varepsilon(t)^4\rho(t)^{-4-k}\mathbbm{1}_{\{r_o\,\leq \,2\delta(t)\}}+C_k\,\varepsilon(t)^4\delta(t)^{-4-k}\mathbbm{1}_{\{\delta(t)/2\,\leq\,r_o\,\leq \,2\delta(t)\}},\\
 \left| \nabla^{\tilde{g}(t),\,k}\partial_t\tilde{\oi}_i(t)\right|_{\tilde{g}(t)} &\leq C_k\varepsilon(t)^{4^-}\rho(t)^{-4^--k}\mathbbm{1}_{\{r_o\,\leq \,2\delta(t)\}}+C_k\,\varepsilon(t)^{4}\delta(t)^{-4-k}\mathbbm{1}_{\{\delta(t)/2\,\leq\,r_o\,\leq \,2\delta(t)\}},\\
\left|\div_{\tilde{g}(t)}\left( \tilde{\oi}_i(t)\right)\right|_{\tilde{g}(t)}& \leq C\varepsilon(t)^4\rho(t)^{-3}\mathbbm{1}_{\{r_o\,\leq \,2\delta(t)\}}+C\varepsilon(t)^4\delta(t)^{-5}\mathbbm{1}_{\{\delta(t)/2\,\leq\,r_o\,\leq \,2\delta(t)\}},\\
\left|\div_{\tilde{g}(t)}(\partial_t\tilde{\oi}_i(t))\right|_{\tilde{g}(t)}&\leq C\varepsilon(t)^{8^{-}}\rho(t)^{-7^{-}}\mathbbm{1}_{\{r_o\,\leq \,2\delta(t)\}}+C\varepsilon(t)^4\delta(t)^{-5}\mathbbm{1}_{\{\delta(t)/2\,\leq\,r_o\,\leq \,2\delta(t)\}}.
\end{split}
\end{equation*}

\end{lemma}

\begin{proof}
Observe that since $\tr_{g_{\zeta(t)}}\oi_i(\zeta(t))=0$ (see Section \ref{basics-EH}),
\begin{equation*}
\begin{split}
\tr_{\tilde{g}(t)} \tilde{\oi}_i(t)&=\chi_{\delta(t)}\left(\tr_{\tilde{g}(t)}-\tr_{g_{\zeta(t)}}\right)\oi_i(\zeta(t))=\chi_{\delta(t)} O(r_o^2\varepsilon(t)^4(\varepsilon(t)+r_o)^{-4})\\
&=\chi_{\delta(t)} O(\varepsilon(t)^4(\varepsilon(t)+r_o)^{-2}),
\end{split}
\end{equation*}
where we have used Lemma \ref{lemma-first-app-est} in the penultimate estimate.

On $\{r_o\leq 3\delta(t)/2\},$ since $\Delta_{L,g_{\zeta(t)}}\oi_i(\zeta(t))=0$,
\begin{equation*}
\begin{split}
\Delta_{L,\tilde{g}(t)}\oi_i(t)&=\Delta_{L,g_{\zeta(t)}}\oi_i(t)+\sum_{k=0}^2\nabla^{g_{\zeta(t)},\,k}(\tilde{g}(t)-g_{\zeta(t)})\ast \nabla^{g_{\zeta(t)},2-k}\oi_i(t)\\
&=\sum_{k=0}^2O\Big(\Big((\varepsilon(t)+r_o)^{2-k}+\underbrace{O(\varepsilon(t)^4\delta(t)^{-4-k})}_{\text{with support in $\{\delta(t)/2\leq r_o\leq2\delta(t)\}$}}\Big) \varepsilon(t)^4(\varepsilon(t)+r_o)^{-4-(2-k)}\Big)\\
&=O(\varepsilon(t)^4(\varepsilon(t)+r_o)^{-4})+\underbrace{O(\varepsilon(t)^4\delta(t)^{-6})}_{\text{with support in $\{\delta(t)/2\leq r_o\leq2\delta(t)\}$}}.
\end{split}
\end{equation*}
Here we have invoked Lemma \ref{lemma-first-app-est} once more. Note that the estimate $O(\varepsilon(t)^4(\varepsilon(t)+r_o)^{-4})$ is optimal: this is exactly the scale given by the ``small eigenvalues'' of $\Delta_{L,\tilde{g}(t)}$, see \cite{no1}.

By taking account for the cut-off function $\chi_{\delta(t)}$,
\begin{equation*}
\begin{split}
\Delta_{L,\tilde{g}(t)}\tilde{\oi}_i(t)&=\Delta_{\tilde{g}(t)}\chi_{\delta(t)}\cdot \oi_i(\zeta(t))+2\nabla^{\tilde{g}(t)}_{\nabla^{\tilde{g}(t)}\chi_{\delta(t)}}\oi_i(\zeta(t))+\chi_{\delta(t)}\Delta_{L,\tilde{g}(t)}\oi_i(\zeta(t))\\
&=\underbrace{O(\varepsilon(t)^4\delta(t)^{-6})}_{\text{with support in $\{\delta(t)/2\leq r_o\leq2\delta(t)\}$}}+\chi_{\delta(t)}\Delta_{L,\tilde{g}(t)}\oi_i(\zeta(t)).
\end{split}
\end{equation*}
This last estimate combined with the penultimate one leads to the first desired estimate on $\Delta_{L,\tilde{g}(t)}\tilde{\oi}_i(t)$ on $\{r_o\leq 3\delta(t)/2\}$. As for the corresponding estimate on $\{r_o\geq \delta(t)\},$ a similar reasoning by linearizing with respect to $g_o$ or the Euclidean metric $\mathbf{e}$ leads to the desired decay.

The estimate holding on the covariant derivatives of $\tilde{\oi}_i(t)$ and those of $\Delta_{L,\tilde{g}(t)}\tilde{\oi}_i(t)$ follows from the same reasoning.\\

Regarding the time derivative of $\tilde{\oi}_i(t)$: thanks to Proposition \ref{prop-analy}, we have
\begin{equation*}
\begin{split}
\partial_t\tilde{\oi}_i(t)&=\chi_{\delta(t)}\partial_t\oi_i(\zeta(t))-\chi'_{\delta(t)}r_o\delta'(t)\delta(t)^{-2}\oi_i(\zeta(t))\\
&=\chi_{\delta(t)}O(|\dot{\zeta}(t)|\varepsilon(t)^{2^-}(\varepsilon(t)+r_o)^{-4^-})+\chi'_{\delta(t)}O(\varepsilon(t)^4\delta(t)^{-4}).
\end{split}
\end{equation*}
Here we have used the assumption \eqref{cond-param} on $|\dot{\zeta}(t)|$ and $\delta'(t)$ in the last line. The estimate holding on the covariant derivatives of $\partial_t\tilde{\oi}_i(t)$ follows from the same reasoning together with Proposition \ref{prop-analy}.

Finally, observe that since $\div_{g_{\zeta(t)}}\oi_i(\zeta(t))=0$ by Section \ref{basics-EH}, we have the following additional improvement:
\begin{equation*}
\begin{split}
\div_{\tilde{g}(t)}\left( \tilde{\oi}_i(t)\right)&=\chi_{\delta(t)}\left(\div_{\tilde{g}(t)}-\div_{g_{\zeta(t)}}\right)\oi_i(\zeta(t))+\oi_i(\zeta(t))(\nabla^{\tilde{g}(t)}\chi_{\delta(t)},\cdot)\\
&=\chi_{\delta(t)}\tilde{g}(t)^{-1}\ast g_{\zeta(t)}^{-1}\ast \nabla^{\tilde{g}(t)}(\tilde{g}(t)-g_{\zeta(t)})\ast \oi_i(\zeta(t))\\
&\quad+\chi_{\delta(t)}\tilde{g}(t)^{-1}\ast g_{\zeta(t)}^{-1} \ast (\tilde{g}(t)-g_{\zeta(t)})\ast \nabla^{\tilde{g}(t)}\oi_i(\zeta(t))+\chi'_{\delta(t)}O(\varepsilon(t)^4\delta(t)^{-5})\\
&=\chi_{\delta(t)}O(\varepsilon(t)^4(\varepsilon(t)+r_o)^{-3})+\chi'_{\delta(t)}O(\varepsilon(t)^4\delta(t)^{-5}).
\end{split}
\end{equation*}
Here we have used Lemma \ref{lemma-first-app-est} in the last line. A similar (but tedious) reasoning leads to the corresponding estimate on the divergence of the time derivative $\partial_t\tilde{\oi}_i(t)$.
\end{proof}

\subsection{Definition of the approximate kernel: conformal vector fields}\label{conf-vec-field}

 In this section, we define tensors that will only be used when constructing ancient Ricci flows coming out of spherical orbifolds.

Let $v_o$ be an eigenfunction on the orbifold $M_o$ associated to the eigenvalue $-4=-\Lambda-1$, i.e. $\Delta_{g_o}v_o=-(-\Lambda+1)v_o$. In particular, $v_o$ is the restriction of the $x_5$ coordinate to the orbifold in the description of Example \ref{ex:orientatble spherical orb}. In coordinates where $g = d\theta^2 + \sin(\theta)^2g_{\mathbb{S}^3/\Gamma}$, one has $v_o(\theta)=\cos(\theta)$.

In particular, the Lie derivative satisfies $\mathcal{L}_{\nabla^{g_o}v_o}g_o=\frac{1}{4}\Delta_{g_o}v_og_o=-v_og_o.$ The eigenfunction $v_0$ is normalized so that the following asymptotic expansion close to the singular point $o\in M_o$ holds:
\begin{equation}\label{cov-der-v0}
\nabla^{g_o,\,k}v_0=\nabla^{g_o,\,k}\left(1-\frac{1}{2}r_o^2\right)+O_{g_o}(r_o^{4-k}),\quad k\geq 0.
\end{equation}
Let us define the approximate eigenfunction $\tilde{v}(t)$ for $t<0$ as follows:
\begin{equation*}
\tilde{v}(t):=(1-\chi_{\delta(t)})\cdot v_o+\chi_{\delta(t)}\cdot v_{g_{\zeta(t)}}(t),\quad \text{on $M$}.
\end{equation*}
Here the function $v_{g_{\zeta(t)}}$ is defined to match and extend $v_o$ on $N$ by:
\begin{equation*}
v_{g_{\zeta(t)}}(t):=1-\varepsilon^2(t)(s_{\zeta(t)})_{\ast}u,
\end{equation*}
 where $u$ satisfies $\Delta_{g}u=4$, $u=\frac{1}{2}r^2+o((1+r)^{-2})$ together with derivatives as explained in Section \ref{sec-can-rad-vec-field}. In order to lighten the notations, we introduce the function $v(t):=v_{g_{\zeta(t)}}(t)$ so that:
 \begin{equation*}
\tilde{v}(t)=(1-\chi_{\delta(t)})\cdot v_o+\chi_{\delta(t)}\cdot v(t),\quad \text{on $M$}.
\end{equation*}

The associated Lie derivative will be denoted by $\tilde{c}(t)$ (where the letter $c$ stands for conformal) for $t\leq 0$ and defined as:
\begin{equation*}
\tilde{c}(t):=\frac{1}{2}\mathcal{L}_{\nabla^{\tilde{g}(t)}\tilde{v}(t)}\tilde{g}(t)=\nabla^{\tilde{g}(t),\,2}\tilde{v}(t),\quad \text{on $M$}.
\end{equation*}
Let us summarize the properties of $\tilde{c}(t)$ in the following lemma:

\begin{lemma}\label{est-basic-conf}
For each $k\geq 0$, there exists $C>0$ such that for $t\leq T$, the function $\tilde{v}(t)$ and the tensor $\tilde{c}(t)$ satisfy:
\begin{enumerate}
\item On $M$,
\begin{equation*}
\begin{split}
\left|\nabla^{\tilde{g}(t),\,k}\left(\tilde{v}(t)-1\right)\right|_{\tilde{g}(t)}&\leq C_k\rho(t)^{2-k},\\
\left|\nabla^{\tilde{g}(t),\,k}\left(\Delta_{\tilde{g}(t)}\tilde{v}(t)+(\Lambda+1)\tilde{v}(t)\right)\right|_{\tilde{g}(t)}&\leq C_k\rho(t)^{2-k}\mathbbm{1}_{\{r_o\,\leq\,2\delta(t)\}}+C_k\varepsilon(t)^4\rho(t)^{-4-k}),\\
|\nabla^{\tilde{g}(t),\,k}\left(\partial_t\tilde{v}(t)\right)|_{\tilde{g}(t)}&\leq C_k\varepsilon(t)^{4^-}\rho(t)^{-2^--k}\mathbbm{1}_{\{r_o\,\leq \,2\delta(t)\}}+C_k \delta(t)^{4-k}\mathbbm{1}_{\{\delta(t)/2\,\leq\,r_o\,\leq \,2\delta(t)\}}.
\end{split}
\end{equation*}

\item On $M$, 
\begin{equation*}
\begin{split}
\left|\nabla^{\tilde{g}(t),\,k}\left(\Delta_{L,\tilde{g}(t)}\tilde{c}(t)+(\Lambda+1)\,\tilde{c}(t)\right)\right|_{\tilde{g}(t)}&\leq C_k\rho(t)^{-k}\mathbbm{1}_{\{r_o\,\leq\,2\delta(t)\}}+C_k\varepsilon(t)^{5}\delta(t)^{-7-k}\mathbbm{1}_{\{\delta(t)/2\,\leq\,r_o\,\leq\,2\delta(t)\}}\\
&\quad+C_k\varepsilon(t)^4\rho(t)^{-4-k}.
\end{split}
\end{equation*}
\item On $M$, for each $k\geq 0$,
\begin{equation*}
\begin{split}
|\nabla^{\tilde{g}(t),\,k}\partial_t\tilde{c}(t)|_{\tilde{g}(t)}\leq C_k\varepsilon(t)^{4^-}\rho(t)^{-4^--k}+C_k\delta(t)^{2-k}\mathbbm{1}_{\{\delta(t)/2\,\leq\,r_o\,\leq \,2\delta(t)\}}.
\end{split}
\end{equation*}
\end{enumerate}

\end{lemma}

\begin{proof}

The first estimate is straightforward since by definition of $v(t)$,
\begin{equation}
\begin{split}\label{est-diff-v-v0}
\tilde{v}(t)-v_o&=\chi_{\delta(t)}\,\left(v(t)-v_o\right)\\
&=\chi_{\delta(t)}\, O(r_o^{4}+\varepsilon(t)^4 (\varepsilon(t)+r_o)^{-2}).
\end{split}
\end{equation}
 Covariant derivatives of the difference follows by differentiating the previous estimate on the basis of \eqref{cov-der-v0} and Section \ref{sec-can-rad-vec-field}.
 
Since $\Delta_{g_o}v_o+(\Lambda+1)v_o=0$ and thanks to Lemma \ref{lemma-first-app-est} and \eqref{est-diff-v-v0}, 
 \begin{equation*}
 \begin{split}
\Delta_{\tilde{g}(t)}\tilde{v}(t)+(\Lambda+1)\tilde{v}(t)&=\left(\Delta_{\tilde{g}(t)}-\Delta_{g_o}\right)\tilde{v}(t)+\Delta_{g_o}(\tilde{v}(t)-v_o)+(\Lambda+1)\left(\tilde{v}(t)-v_o\right)\\
&=O(\varepsilon(t)^4\rho(t)^{-4})+O(\rho(t)^2)\mathbbm{1}_{\{r_o\,\leq\,2\delta(t)\}},
\end{split}
\end{equation*}
as desired. Covariant derivatives follows by invoking \eqref{est-diff-v-v0} and Lemma \ref{lemma-first-app-est} once more.

  As for the time-derivative of $\tilde{v}(t)$, Proposition \ref{prop-analy} gives:
  \begin{equation*}
  \begin{split}\label{est-time-der-v-tilde}
\partial_t\tilde{v}(t)&=\partial_t\left(\chi_{\delta(t)}(v(t)-v_0)\right)\\
&=\chi_{\delta(t)}\partial_tv(t)-\chi'_{\delta(t)}r_o\delta'(t)\delta(t)^{-2}(v(t)-v_0)\\
&=\chi_{\delta(t)}O(\varepsilon(t)^{4^-}(\varepsilon(t)+r_o)^{-2^-})+\chi'_{\delta(t)}O\left(r_o^{4}+\varepsilon(t)^4 (\varepsilon(t)+r_o)^{-2}\right).
\end{split}
\end{equation*}
Similarly, one gets for the covariant derivatives of $\tilde{v}(t)$:
\begin{equation}\label{est-time-der-v-tilde-bis}
|\nabla^{\tilde{g}(t),\,k}\left(\partial_t\tilde{v}(t)\right)|_{\tilde{g}(t)}=\chi_{\delta(t)}O(\varepsilon(t)^{4^-}(\varepsilon(t)+r_o)^{-2^--k})+\underbrace{O\left(\delta(t)^{4-k}+\varepsilon(t)^4 \delta(t)^{-2-k}\right)}_{\text{with support in $\{\delta(t)/2\,\leq\,r_o\,\leq\, 2\delta(t)\}$}}.
\end{equation}
We conclude by invoking \eqref{cond-param} with $\delta<2/3$.

Regarding the second estimate, recall the following commutation identity that can be found for instance in the proof of \cite[Chapter 2, Lemma $2.33$]{Cho-Lu-Ni-Boo}: 
\begin{equation}\label{eq:Lapl-Hess}
\begin{split}
\Delta_{L,\tilde{g}(t)}\nabla^{\tilde{g}(t),2}\tilde{v}(t)&=\nabla^{\tilde{g}(t),2}\Delta_{\tilde{g}(t)}\tilde{v}(t)+\nabla^{\tilde{g}(t)}\Ric(\tilde{g}(t))\ast \nabla^{\tilde{g}(t)}\tilde{v}(t)\\
&=\nabla^{\tilde{g}(t),2}\Delta_{\tilde{g}(t)}\tilde{v}(t)+\nabla^{\tilde{g}(t)}(\Ric(\tilde{g}(t))-\Lambda\tilde{g}(t))\ast \nabla^{\tilde{g}(t)}\tilde{v}(t),
\end{split}
\end{equation}
on $M.$

By the previous estimate on $\Delta_{\tilde{g}(t)}\tilde{v}(t)+(\Lambda+1)\tilde{v}(t)$ together with its derivatives, we arrive at:
\begin{equation*}
\begin{split}
\Delta_{L,\tilde{g}(t)}\nabla^{\tilde{g}(t),2}\tilde{v}(t)&=-(\Lambda+1)\nabla^{\tilde{g}(t),2}\tilde{v}(t)+O(\varepsilon(t)^4\rho(t)^{-4})+O(\rho(t)^2)\mathbbm{1}_{\{r_o\,\leq\,2\delta(t)\}}\\
&\quad+\nabla^{\tilde{g}(t)}\Ric(\tilde{g}(t))\ast \nabla^{\tilde{g}(t)}\tilde{v}(t)\\
\end{split}
\end{equation*}
Invoking Lemma \ref{rough-est-deviation-einstein} also gives:
\begin{equation*}
\begin{split}
&|\nabla^{\tilde{g}(t)}(\Ric(\tilde{g}(t))-\Lambda\tilde{g}(t))\ast \nabla^{\tilde{g}(t)}\tilde{v}(t)|_{\tilde{g}(t)}=\\
&=O(\varepsilon(t)^4\rho(t)^{-4})+O(\rho(t)^2)\mathbbm{1}_{\{r_o\,\leq\, 2\delta(t)\}}+\varepsilon(t)^5\delta(t)^{-7}\mathbbm{1}_{\{\delta(t)/2\,\leq\,r_o\,\leq\, 2\delta(t)\}}.
\end{split}
\end{equation*}

The combination of the two previous estimates leads to the desired one.

As for the last estimate, the proof of \cite[Chapter 2, Lemma $2.33$]{Cho-Lu-Ni-Boo} gives:
\begin{equation*}
\begin{split}
|\partial_t\tilde{c}(t)|_{\tilde{g}(t)}&=\partial_t\nabla^{\tilde{g}(t),2}\tilde{v}(t)\\
&=\nabla^{\tilde{g}(t),2}\left(\partial_t\tilde{v}(t)\right)+\nabla^{\tilde{g}(t)}(\partial_t\tilde{g}(t))\ast \nabla^{\tilde{g}(t)}\tilde{v}(t)\\
&=O\left(\varepsilon(t)^{4^-}(\varepsilon(t)+r_o)^{-4^-}\right)+\underbrace{O(\delta(t)^{2})}_{\text{with support in $\{\delta(t)/2\,\leq\,r_o\,\leq\, 2\delta(t)\}$}}
\end{split}
\end{equation*}
where we have used Lemma \ref{lemma-time-der-metric} together with \eqref{est-time-der-v-tilde-bis} in the last line. This ends the proof of the desired estimate as expected for $k=0$. Estimates on higher covariant derivatives of $\partial_t\tilde{c}(t)$ can be obtained similarly.

\end{proof}

Finally, we define the approximate kernel associated to the approximate solution $\tilde{g}(t)$ as follows:
\begin{defn}\label{defn-app-kernel}
We define the approximate kernel associated to the curve $\tilde{g}(t)$ as:
\begin{equation*}
\tilde{\mathbf{O}}(t):=\operatorname{vect}((\tilde{\oi}_i(t))_{i},\tilde{g}(t),\tilde{c}(t)),\quad t\leq T,\quad \text{respectively $\tilde{\mathbf{O}}(t):=\operatorname{vect}((\tilde{\oi}_i(t))_{i},\tilde{g}(t)),\quad t\geq T.$}
\end{equation*}
\end{defn}

\subsection{H\"older estimates for the approximate kernel}
The following proposition is the first of a series of estimates on the H\"older semi-norms $[\,\cdot\,]_{\alpha,t,r}$ as defined in Definition \ref{defn-semi-norm-holder-tensor} of the approximate kernel $\tilde{\mathbf{O}}(t)$ defined in Definition \ref{defn-app-kernel} together with those of the time-derivative and Lichnerowicz laplacian acting on $\tilde{\mathbf{O}}(t)$.

\begin{prop}\label{prop-est-semi-norm-time-der-g}
There exists $C>0$ such that for $t\leq T$,
{
\begin{equation*}
\begin{split}
&[\partial_t\tilde{g}]_{\alpha,t,r}\leq C\frac{\varepsilon(t)^{4^{-}}}{\rho(t)^{4^{-}}}\left(r^2+\frac{[\dot{\zeta}]_{\alpha,t,r}}{\varepsilon(t)^2}\right),\\
&[\partial_t\tilde{g}\ast \tilde{\oi}_i]_{\alpha,t,r}\leq C\frac{\varepsilon(t)^{8^{-}}}{\rho(t)^{8^{-}}}\left(r^2+\frac{[\dot{\zeta}]_{\alpha,t,r}}{\varepsilon(t)^2}\right)\mathbbm{1}_{\{r_o\,\leq \,2\delta(t)\}}+Cr^2\varepsilon(t)^{8^-}\delta(t)^{-8^-}\mathbbm{1}_{\{\delta(t)/2\,\leq\,r_o\,\leq \,2\delta(t)\}},\\
&[\partial_t\tilde{g}\ast \tilde{g}]_{\alpha,t,r}+[\partial_t\tilde{g}\ast \tilde{c}]_{\alpha,t,r}\leq  C\frac{\varepsilon(t)^{4^{-}}}{\rho(t)^{4^{-}}}\left(r^2+\frac{[\dot{\zeta}]_{\alpha,t,r}}{\varepsilon(t)^2}\right).
\end{split}
\end{equation*}}
\end{prop}

\begin{proof}
Thanks to [\eqref{first-est-partial-g}, Lemma \ref{lemma-time-der-metric}], Lemma \ref{lemma-first-app-est} and Lemma \ref{lemma-basic-interp}:
\begin{equation*}
\begin{split}
[\partial_t\tilde{g}]_{\alpha,t,r}&=\left[\chi'_{\delta}\frac{r_o\delta'}{\delta^2}\left(g_o+h^4_{\zeta}-g_{\zeta}-h_{2,\zeta}\right)+\chi_{\delta}d_{\zeta}\left(g+h_{2}\right)(\dot{\zeta})+(1-\chi_{\delta})d_{\zeta}h^4(\dot{\zeta})\right]_{\alpha,t,r}\\
&\leq Cr^2\sup_{t-r^2\,\leq \,t'\,\leq\, t}\left|\partial_t\left(\chi'_{\delta}\frac{r_o\delta'}{\delta^2}\left(g_o+h^4_{\zeta}-g_{\zeta}-h_{2,\zeta}\right)\right)\right|_{\tilde{g}(t)}\\
&\quad+\left[\chi_{\delta}d_{\zeta}\left(g+h_{2}\right)(\dot{\zeta})\right]_{\alpha,t,r}+\left[(1-\chi_{\delta})d_{\zeta}h^4(\dot{\zeta})\right]_{\alpha,t,r}\\
&\leq \underbrace{Cr^2\varepsilon(t)^{4^-}\delta(t)^{-4^-}}_{\text{supported in $\{\delta(t)/2\leq r_o\leq 2\delta(t)\}$}}+\left[\chi_{\delta}d_{\zeta}\left(g+h_{2}\right)(\dot{\zeta})\right]_{\alpha,t,r}+\left[(1-\chi_{\delta})d_{\zeta}h^4(\dot{\zeta})\right]_{\alpha,t,r}\\
&\leq \underbrace{Cr^2\varepsilon(t)^{4^-}\delta(t)^{-4^-}}_{\text{supported in $\{\delta(t)/2\leq r_o\leq 2\delta(t)\}$}}+\underbrace{\left[d_{\zeta}\left(g+h_{2}\right)(\dot{\zeta})\right]_{\alpha,t,r}}_{\text{supported in $\{ r_o\leq 2\delta(t)\}$}}+\underbrace{\left[d_{\zeta}h^4(\dot{\zeta})\right]_{\alpha,t,r}}_{\text{supported in $\{ r_o\geq \delta(t)/2\}$}}.
\end{split}
\end{equation*}
Here we have invoked Proposition \ref{prop-analy} in the penultimate and the last  inequality.

 Now, Proposition \ref{prop-analy} gives:
\begin{equation*}
\begin{split}
\left[d_{\zeta}\left(g+h_{2}\right)(\dot{\zeta})\right]_{\alpha,t,r}&\leq Cr^2\sup_{t-r^2\,\leq\,t',t''\,\leq t}|d^2_{\zeta(t'')}(g+h_2)(\dot{\zeta}(t''),\dot{\zeta}(t'))|_{g_{\zeta(t)}}\\
&\quad+C\frac{\varepsilon(t)^{4^{-}}}{(\varepsilon(t)+r_o)^{4^{-}}}\frac{[\dot{\zeta}]_{\alpha,t,r}}{\varepsilon(t)^2}\\
&\leq C\frac{\varepsilon(t)^{4^{-}}}{(\varepsilon(t)+r_o)^{4^{-}}}\left(r^2+\frac{[\dot{\zeta}]_{\alpha,t,r}}{\varepsilon(t)^2}\right).
\end{split}
\end{equation*}
{An analogous estimate holds on $\left[d_{\zeta}h^4(\dot{\zeta})\right]_{\alpha,t,r}$:
\begin{equation*}
\left[d_{\zeta}h^4(\dot{\zeta})\right]_{\alpha,t,r}\leq C\frac{\varepsilon(t)^{4^{-}}}{(\varepsilon(t)+r_o)^{4^{-}}}\left(r^2+\frac{[\dot{\zeta}]_{\alpha,t,r}}{\varepsilon(t)^2}\right).
\end{equation*}
This ends the estimate on $[\partial_t\tilde{g}]_{\alpha,t,r}$.}

By invoking \eqref{semi-holder-diff-tens} together with the estimate on $\partial_t\tilde{\oi}_i$ from Lemma \ref{estimates variations et laplacien o1}:
\begin{equation}\label{semi-holder-oi}
[\tilde{\oi}_i]_{\alpha,t,r}\leq \underbrace{Cr^2\frac{\varepsilon(t)^{4^-}}{(\varepsilon(t)+r_o)^{4^-}}}_{\text{supported in $\{ r_o\leq 2\delta(t)\}$}}+\underbrace{Cr^2\varepsilon(t)^{4}\delta(t)^{-4}}_{\text{supported in $\{\delta(t)/2\leq r_o\leq 2\delta(t)\}$}}.
\end{equation}

Similarly, thanks to Lemma \ref{lemma-time-der-metric},
\begin{equation}\label{semi-holder-g}
\begin{split}
[\tilde{g}]_{\alpha,t,r}&\leq  r^{2}\sup_{t-r^2\,\leq\,\tau\,\leq\,t}|\partial_t\tilde{g}(\tau)|_{\tilde{g}(t)} \leq Cr^2\frac{\varepsilon(t)^4}{(\varepsilon(t)+r_o)^4},
\end{split}
\end{equation}

and Lemma \ref{est-basic-conf} gives:

\begin{equation}\label{semi-holder-c}
[\tilde{c}]_{\alpha,t,r}\leq Cr^2\frac{\varepsilon(t)^{4^-}}{(\varepsilon(t)+r_o)^{4^-}}+\underbrace{Cr^2\delta(t)^{2}}_{\text{supported in $\{\delta(t)/2\leq r_o\leq 2\delta(t)\}$}}.
\end{equation}

We can now use Lemma \ref{lemma-holder-prod} to estimate any contraction of $\partial_t\tilde{g}$ with the approximate kernel $\tilde{\mathbf{O}}(t)$ with the help of estimates \eqref{semi-holder-oi}, \eqref{semi-holder-g} and \eqref{semi-holder-c}. 
\end{proof}

We continue by estimating the H\"older norm of the Lichnerowicz operator of the approximate kernel $\tilde{\mathbf{O}}(t)$:
\begin{prop}\label{prop-semi-norm-lap-app-ker}
There exists $C>0$ such that for $t\leq T$,
{
\begin{equation*}
\begin{split}
[\Delta_{L,\tilde{g}(t)}\tilde{\oi}_i]_{\alpha,t,r}&\leq Cr^2\frac{\varepsilon(t)^{4^-}}{\rho(t)^{6^-}}\mathbbm{1}_{\{r_o\,\leq \,2\delta(t)\}}+Cr^2\varepsilon(t)^{4}\delta(t)^{-6}\mathbbm{1}_{\{\delta(t)/2\,\leq\,r_o\,\leq \,2\delta(t)\}},\\
[\Delta_{L,\tilde{g}(t)}\tilde{g}]_{\alpha,t,r}&=0,\\
[\Delta_{L,\tilde{g}(t)}\tilde{c}]_{\alpha,t,r}&\leq Cr^2\frac{\varepsilon(t)^{4^-}}{\rho(t)^{6^-}}+Cr^2\mathbbm{1}_{\{\delta(t)/2\,\leq\,r_o\,\leq \,2\delta(t)\}}.
\end{split}
\end{equation*}
}
\end{prop}

\begin{proof}
Since $\tilde{g}(t)$ is parallel with respect to $\nabla^{\tilde{g}(t)}$ and that the zeroth order terms of the Lichnerowicz operator preserves pure trace symmetric $2$-tensors, one gets $\Delta_{L,\tilde{g}(t)}\tilde{g}(t)=0.$

Observe that both terms $\Delta_{L,\tilde{g}(t)}\tilde{\oi}_i(t)$ and $\Delta_{L,\tilde{g}(t)}\tilde{c}$ are time-differentiable so that Lemma \ref{lemma-basic-interp}, \eqref{semi-holder-diff-tens} and Lemma \ref{lemma-holder-prod} are applicable.  

Now, if $(h(t))_{t\,\leq\, T}$ is a $C^1$ family of symmetric $2$-tensors on $M$ then: 
\begin{equation}\label{gal-id-time-diff-lic-eqn}
\partial_t\left(\Delta_{L,\tilde{g}(t)} h(t)\right)=\Delta_{L,\tilde{g}(t)} \left(\partial_th(t)\right)+\sum_{j=0}^2\nabla^{\tilde{g}(t),j}h(t)\ast \nabla^{\tilde{g}(t),2-j}(\partial_t\tilde{g}(t)).
\end{equation}

Applying the previous identity \eqref{gal-id-time-diff-lic-eqn} together with Lemmata \ref{lemma-time-der-metric} and \ref{estimates variations et laplacien o1} gives:
\begin{equation*}
\begin{split}
\left|\partial_t\left(\Delta_{L,\tilde{g}(t)} \tilde{\oi}_i(t)\right)\right|_{\tilde{g}(t)}&\leq \underbrace{C\frac{\varepsilon(t)^{4^-}}{(\varepsilon(t)+r_o)^{6^-}}}_{\text{supported in $\{ r_o\leq 2\delta(t)\}$}}+\underbrace{C\varepsilon(t)^{4}\delta(t)^{-6}}_{\text{supported in $\{\delta(t)/2\leq r_o\leq 2\delta(t)\}$}},
\end{split}
\end{equation*}
which gives the expected result on $[\Delta_{L,\tilde{g}(t)}\tilde{\oi}_i]_{\alpha,t,r}$ thanks to \eqref{semi-holder-diff-tens}.

Similarly, applying Lemma \ref{est-basic-conf} and identity \eqref{gal-id-time-diff-lic-eqn},
\begin{equation*}
\begin{split}
\left|\partial_t\left(\Delta_{L,\tilde{g}(t)} \tilde{c}(t)\right)\right|_{\tilde{g}(t)}&\leq C\frac{\varepsilon(t)^{4^-}}{(\varepsilon(t)+r_o)^{6^-}}+\underbrace{C}_{\text{supported in $\{\delta(t)/2\leq r_o\leq 2\delta(t)\}$}}.
\end{split}
\end{equation*}
\end{proof}

We continue this section by estimating the H\"older semi-norms $[\,\cdot\,]_{\alpha,t,r}$ of the time derivatives of the remaining components of the approximate kernel $\tilde{\mathbf{O}}(t)$.
\begin{prop}\label{prop-holder-time-app-ker-oi}
There exists $C>0$ such that for $t\leq T$,
{
\begin{equation*}
\begin{split}
[\partial_t\tilde{\oi}_i]_{\alpha,t,r}&\leq C\frac{\varepsilon(t)^{4^{-}}}{\rho(t)^{4^{-}}}\left(r^2+\frac{[\dot{\zeta}]_{\alpha,t,r}}{\varepsilon(t)^2}\right)\mathbbm{1}_{\{r_o\,\leq \,2\delta(t)\}},\\
[\partial_t\tilde{c}]_{\alpha,t,r}&\leq C\frac{\varepsilon(t)^{4^{-}}}{\rho(t)^{4^{-}}}\left(r^2+\frac{[\dot{\zeta}]_{\alpha,t,r}}{\varepsilon(t)^2}\right).
\end{split}
\end{equation*}}
\end{prop}

\begin{proof}
Recall that 
\begin{equation*}
\begin{split}
\partial_t\tilde{\oi}_i(t)&=\chi_{\delta(t)}\partial_t\oi_i(\zeta(t))-\chi'_{\delta(t)}\frac{r_o\delta'(t)}{\delta(t)^2}\oi_i(\zeta(t))\\
&=\chi_{\delta(t)}d_{\zeta(t)}\oi_i(\dot{\zeta}(t))-\chi'_{\delta(t)}\frac{r_o\delta'(t)}{\delta(t)^2}\oi_i(\zeta(t)).
\end{split}
\end{equation*}
Proposition \ref{prop-analy} and repeated applications of \eqref{semi-holder-diff-tens} give:
\begin{equation*}
[\chi_{\delta}d_{\zeta}\oi_i(\dot{\zeta})]_{\alpha,t,r}\leq C\underbrace{\frac{\varepsilon(t)^{4^{-}}}{(\varepsilon(t)+r_o)^{4^{-}}}\left(r^2+\frac{[\dot{\zeta}]_{\alpha,t,r}}{\varepsilon(t)^2}\right)}_{\text{supported in $\{r_o\leq 2\delta(t)\}$}}+\underbrace{Cr^2\varepsilon(t)^{4^-}\delta(t)^{-4^-}}_{\text{supported in $\{\delta(t)/2\leq r_o\leq 2\delta(t)\}$}}.
\end{equation*}
Since the term $\chi'_{\delta(t)}\frac{r_o\delta'(t)}{\delta(t)^2}\oi_i(\zeta(t))$ is time-differentiable, \eqref{semi-holder-diff-tens} gives:
\begin{equation*}
\left[\chi'_{\delta}\frac{r_o\delta'}{\delta^2}\oi_i(\zeta)\right]_{\alpha,t,r}\leq \underbrace{Cr^2\varepsilon(t)^4\delta(t)^{-4}}_{\text{supported in $\{\delta(t)/2\leq r_o\leq 2\delta(t)\}$}}.
\end{equation*}
The two previous estimates lead to the desired result thanks to \eqref{cond-param} since $\delta<1$.

As for the H\"older semi-norm of $\partial_t\tilde{c}(t)$, recall from the proof of Lemma \ref{est-basic-conf} that:
\begin{equation}\label{take-it-easy}
\partial_t\tilde{c}(t)=\nabla^{\tilde{g}(t),2}(\partial_t\tilde{v}(t))+\nabla^{\tilde{g}(t)}(\partial_t\tilde{g}(t))\ast \nabla^{\tilde{g}(t)}\tilde{v}(t).
\end{equation}
In order to handle the second term on the righthand side of \eqref{take-it-easy}, we apply Lemma \ref{lemma-holder-cov-tensor} to $S(t):=\partial_t\tilde{g}(t)$ and $T(t):=\tilde{v}(t)$ to get:
\begin{equation}
\begin{split}\label{love-holder-est}
[\nabla^{\tilde{g}}(\partial_t\tilde{g})]_{\alpha,t,r}&\leq C\frac{\varepsilon(t)^{4^{-}}}{(\varepsilon(t)+r_o)^{5^{-}}}\left(r^2+\frac{[\dot{\zeta}]_{\alpha,t,r}}{\varepsilon(t)^2}\right),\\
[\nabla^{\tilde{g}}\tilde{v}]_{\alpha,t,r}&\leq Cr^2\underbrace{\frac{\varepsilon(t)^{4^-}}{(\varepsilon(t)+r_o)^{3^-}}}_{\text{supported in $\{r_o\leq 2\delta(t)\}$}}+\underbrace{Cr^2\delta(t)^{3}}_{\text{supported in $\{\delta(t)/2\leq r_o\leq 2\delta(t)\}$}}.
\end{split}
\end{equation}
Here we have used Lemma \ref{lemma-time-der-metric}, \ref{est-basic-conf} and \ref{prop-est-semi-norm-time-der-g} in combination with Proposition \ref{prop-analy}.

As an intermediate result, we get thanks to the two previous estimates \eqref{love-holder-est} together with Lemma \ref{lemma-holder-prod}:
\begin{equation}
\begin{split}\label{tired-c-I}
[\nabla^{\tilde{g}}(\partial_t\tilde{g})\ast\nabla^{\tilde{g}}\tilde{v}]_{\alpha,t,r}&\leq C\frac{\varepsilon(t)^{4^{-}}}{(\varepsilon(t)+r_o)^{4^{-}}}\left(r^2+\frac{[\dot{\zeta}]_{\alpha,t,r}}{\varepsilon(t)^2}\right).
\end{split}
\end{equation}
Now, recall that:
\begin{equation*}
\begin{split}
\partial_t\tilde{v}(t)&=\chi_{\delta(t)}\partial_tv_{g_{\zeta(t)}}-\chi'_{\delta(t)}\frac{r_o\delta'(t)}{\delta(t)^2}(v_{g_{\zeta(t)}}-v_0)\\
&=-\chi_{\delta(t)}d_{\zeta(t)}u(\dot{\zeta}(t))-\chi'_{\delta(t)}\frac{r_o\delta'(t)}{\delta(t)^2}(v_{g_{\zeta(t)}}-v_0).
\end{split}
\end{equation*}
This implies by Lemma \ref{lemma-holder-cov-tensor}:
\begin{equation}
\begin{split}\label{tired-c-II}
[\nabla^{\tilde{g},\,2}(\partial_t\tilde{v})]_{\alpha,t,r}&\leq C\underbrace{\frac{\varepsilon(t)^{4^{-}}}{(\varepsilon(t)+r_o)^{4^{-}}}\left(r^2+\frac{[\dot{\zeta}]_{\alpha,t,r}}{\varepsilon(t)^2}\right)}_{\text{supported in $\{r_o\leq 2\delta(t)\}$}}+\underbrace{Cr^2\varepsilon(t)^{4^-}\delta(t)^{-4^-}}_{\text{supported in $\{\delta(t)/2\leq r_o\leq 2\delta(t)\}$}}.
\end{split}
\end{equation}
Here we have used Lemma \ref{lemma-time-der-metric} and \ref{est-basic-conf} together with Proposition \ref{prop-analy}.

Recalling \eqref{take-it-easy}, estimates \eqref{tired-c-I} and \eqref{tired-c-II} gives the expected result on $[\partial_t\tilde{c}]_{\alpha,t,r}$ thanks to \eqref{cond-param}.
\end{proof}

We end this section by a careful H\"older estimate on the obstruction for the first approximation $\tilde{g}(t)$ to be an exact solution to the normalized Ricci flow:
\begin{lemma}\label{holder-semi-obstruction}
There exists $C>0$ such that for $t\leq T$,
{
\begin{equation*}
\begin{split}
\Bigg[&\partial_t\tilde{g}+2\Ric(\tilde{g})-2\Lambda\tilde{g}-\chi_{\delta}d_{\zeta}g\left(\dot{\zeta}-2\mathbf{R}_p^+(\zeta)\right)-(1-\chi_{\delta})\left(d_{\zeta}h^4(\dot{\zeta})+\Li_{B_{g_o}(h_{\zeta}^4)}(g_o)\right)\Bigg]_{\alpha,t,r}\\
&\leq C\left(r^2+\frac{[\dot{\zeta}]_{\alpha,t,r}}{\varepsilon(t)^2}\right)\varepsilon(t)^{4^-}\rho(t)^{-2^{-}}\mathbbm{1}_{\{r_o\,\leq \,2\delta(t)\}}+C\left(r^2+\frac{[\dot{\zeta}]_{\alpha,t,r}}{\varepsilon(t)^2}\right)\varepsilon(t)^{5^-}\delta(t)^{-7-}\mathbbm{1}_{\{\delta(t)/2\,\leq\,r_o\,\leq \,2\delta(t)\}}\\
&\quad+Cr^2\varepsilon(t)^8\rho(t)^{-10}\mathbbm{1}_{\{2\delta(t)\,\leq\,r_o\}},
\end{split}
\end{equation*}
and,

\begin{equation*}
\begin{split}
&\Big[\partial_t\tilde{g}+2\Ric(\tilde{g})-2\Lambda\tilde{g}-\chi_{\delta}d_{\zeta}g(\dot{\zeta}-2\mathbf{R}_p^+(\zeta))-(1-\chi_{\delta})\left(d_{\zeta}h^4(\dot{\zeta})+\Li_{B_{g_o}(h_{\zeta}^4)}(g_o)\right)\Big]_{\alpha,x,t}\leq\\
& \quad+C\left(\left(1+\frac{[\dot{\zeta}]_{\alpha,t,\rho(x,t)}}{\varepsilon(t)^2}\right)\frac{\varepsilon(t)^{4^-}}{\rho(t)^{2^-}}+\rho(t)^{2}\right)\mathbbm{1}_{\{r_o\,\leq \,2\delta(t)\}}\\
 &\quad+C\left(1+\frac{[\dot{\zeta}]_{\alpha,t,\rho(x,t)}}{\varepsilon(t)^2}\right)\varepsilon(t)^{5^-}\delta(t)^{-7^-}\mathbbm{1}_{\{\delta(t)/2\,\leq\,r_o\,\leq \,2\delta(t)\}}+C\varepsilon(t)^8\rho(t)^{-10}\mathbbm{1}_{\{2\delta(t)\,\leq\,r_o\}}.
\end{split}
\end{equation*}}
\end{lemma}

\begin{proof}
If $T(t):=\partial_t\tilde{g}(t)+2\Ric(\tilde{g}(t))-2\Lambda\tilde{g}(t)$ then thanks to Lemma \ref{app-C-monster}, as in the proofs of Lemma \ref{rough-est-deviation-einstein} and \ref{lemma-time-der-metric}, let us linearize this tensor at $g_{\zeta}$ so that on $\{r_o\leq \delta(t)\}$,
\begin{equation*}
\begin{split}
T&=\chi'_{\delta}\frac{r_o\delta'}{\delta^2}(g_o+h^4_{\zeta}-g_{\zeta}-h_{2,\zeta})+\chi_{\delta}d_{\zeta}(g+h_2)(\dot{\zeta})+(1-\chi_{\delta})d_{\zeta}h^4(\dot{\zeta})\\
&\quad+2\Ric(g_{\zeta})-2\Lambda(g_{\zeta}+h_{2,\zeta}+(1-\chi_{\delta})(g_o+h^4_{\zeta}-g_{\zeta}-h_{2,\zeta}))\\
&\quad-\Delta_{L,g_{\zeta}}(h_{2,\zeta}+(1-\chi_{\delta})(g_o+h^4_{\zeta}-g_{\zeta}-h_{2,\zeta}))+\mathcal{L}_{B_{g_{\zeta}}(h_{2,\zeta}+(1-\chi_{\delta})(g_o+h^4_{\zeta}-g_{\zeta}-h_{2,\zeta})}\tilde{g}\\
&\quad+Q_{g_{\zeta}}(h_{2,\zeta}+(1-\chi_{\delta})(g_o+h^4_{\zeta}-g_{\zeta}-h_{2,\zeta}))\\
&=\chi_{\delta}d_{\zeta}g(\dot{\zeta}-2\mathbf{R}_p^+(\zeta))+(1-\chi_{\delta})d_{\zeta}h^4(\dot{\zeta})\\
&\quad+\left(\chi'_{\delta}\frac{r_o\delta'}{\delta^2}-2\Lambda (1-\chi_{\delta})\right)(g_o+h^4_{\zeta}-g_{\zeta}-h_{2,\zeta})+\chi_{\delta}d_{\zeta}h_2(\dot{\zeta})-2\Lambda h_{2,\zeta}\\
&\quad-\Delta_{L,g_{\zeta}}((1-\chi_{\delta})(g_o+h^4_{\zeta}-g_{\zeta}-h_{2,\zeta}))+\mathcal{L}_{B_{g_{\zeta}}((1-\chi_{\delta})(g_o+h^4_{\zeta}-g_{\zeta}-h_{2,\zeta})}\tilde{g}\\
&\quad+Q_{g_{\zeta}}(h_{2,\zeta}+(1-\chi_{\delta})(g_o+h^4_{\zeta}-g_{\zeta}-h_{2,\zeta})).
\end{split}
\end{equation*}

Here we have been using Proposition \ref{prop-h2-H2} in the second equality.
In particular, on $\{r_o\leq \delta(t)\}$, thanks to \eqref{cond-param}, Lemma \ref{lemma-first-app-est} and 
\begin{equation*}
\begin{split}
[T-\chi_{\delta}d_{\zeta}&g(\dot{\zeta}-2\mathbf{R}_p^+(\zeta))-(1-\chi_{\delta})d_{\zeta}h^4(\dot{\zeta})]_{\alpha,t,r}\leq \\
&C\left(r^2+\frac{[\dot{\zeta}]_{\alpha,t,r}}{\varepsilon(t)^2}\right)\frac{\varepsilon(t)^{4^-}}{(\varepsilon(t)+r_o)^{2^-}}+Cr^2\underbrace{\varepsilon(t)^{5^-}\delta(t)^{-7^-}}_{\text{supported in $\{r_o\geq \delta(t)/2\}$}}.
\end{split}
\end{equation*}
On $\{r_o> \delta(t)\}$, linearizing $T$ at $g_o$ gives:
\begin{equation*}
\begin{split}
T&=\chi'_{\delta}\frac{r_o\delta'}{\delta^2}(g_o+h^4_{\zeta}-g_{\zeta}-h_{2,\zeta})+\chi_{\delta}d_{\zeta}(g+h_2)(\dot{\zeta})+(1-\chi_{\delta})d_{\zeta}h^4(\dot{\zeta})\\
&\quad+2\Ric(g_o)-2\Lambda(g_{o}+h^4_{\zeta}+\chi_{\delta}(g_{\zeta}+h_{2,\zeta}-g_{o}-h^4_{\zeta}))\\
&\quad-\Delta_{L,g_{o}}(h^4_{\zeta}+\chi_{\delta}(g_o+h^4_{\zeta}-g_{\zeta}-h_{2,\zeta}))+\mathcal{L}_{B_{g_{\zeta}}(h_{2,\zeta}+\chi_{\delta}(g_{\zeta}+h_{2,\zeta}-g_{o}-h^4_{\zeta})}\tilde{g}\\
&\quad+Q_{g_{\zeta}}(h^4_{\zeta}+\chi_{\delta}(g_{\zeta}+h_{2,\zeta}-g_{o}-h^4_{\zeta}))\\
&=(1-\chi_{\delta})d_{\zeta}h^4(\dot{\zeta})+\mathcal{L}_{B_{g_{\zeta}}(h^4_{\zeta})}\tilde{g}+\chi_{\delta}d_{\zeta}(g+h_2)(\dot{\zeta})\\
&\quad+\chi'_{\delta}\frac{r_o\delta'}{\delta^2}(g_o+h^4_{\zeta}-g_{\zeta}-h_{2,\zeta})-2\Lambda\chi_{\delta}(g_{\zeta}+h_{2,\zeta}-g_{o}-h^4_{\zeta})\\
&\quad-\Delta_{L,g_{o}}(\chi_{\delta}(g_o+h^4_{\zeta}-g_{\zeta}-h_{2,\zeta}))+\mathcal{L}_{B_{g_{\zeta}}(\chi_{\delta}(g_{\zeta}+h_{2,\zeta}-g_{o}-h^4_{\zeta}))}\tilde{g}\\
&\quad+Q_{g_{\zeta}}(h^4_{\zeta}+\chi_{\delta}(g_{\zeta}+h_{2,\zeta}-g_{o}-h^4_{\zeta})),
\end{split}
\end{equation*}
where we have been using Proposition \ref{premiere obst orbifold 1 pt} in the second equality.
In particular, on $\{r_o\,>\,\delta(t)\}$, thanks to \eqref{cond-param}, Lemma \ref{lemma-first-app-est} and Proposition \ref{prop-analy},
\begin{equation*}
\begin{split}
&\Big[T-(1-\chi_{\delta})d_{\zeta}h^4(\dot{\zeta})-(1-\chi_{\delta})\mathcal{L}_{B_{g_{\zeta}}(h^4_{\zeta})}g_o\Big]_{\alpha,t,r}\leq \\
& C\underbrace{\left(r^2+\frac{[\dot{\zeta}]_{\alpha,t,r}}{\varepsilon(t)^2}\right)\frac{\varepsilon(t)^{4^-}}{\delta(t)^{2^-}}}_{\text{supported in $\{\delta(t)/2\,\leq\,r_o\leq 2\delta(t)\}$}}+Cr^2\underbrace{\varepsilon(t)^{5^-}\delta(t)^{-7^-}}_{\text{supported in $\{\delta(t)/2\,\leq\,r_o\leq 2\delta(t)\}$}}+Cr^2\underbrace{\varepsilon(t)^{8}\rho(t)^{-10}}_{\text{supported in $\{r_o\geq 2\delta(t)\}$}}.
\end{split}
\end{equation*}
Here we have bounded from above the terms $\frac{[\dot{\zeta}]_{\alpha,t,r}}{\varepsilon(t)^2}\varepsilon(t)^4\delta(t)^{-4}$ coming from $\chi_{\delta}d_{\zeta}g(\dot{\zeta})$ by  $\frac{[\dot{\zeta}]_{\alpha,t,r}}{\varepsilon(t)^2}\varepsilon(t)^{5^-}\delta(t)^{-7^-}$.

In order to get the second estimate, one invokes the previous estimate with $r=\rho(x,t)$ together with Lemma \ref{lemma-time-der-metric} with $k=1$.
\end{proof}

\section{Projections on the approximate kernel}\label{sec:proj kernel}

In this whole section, we consider $\tilde{g}(t)$ defined either on $(-\infty,T)$, or on $(T,+\infty)$ towards the construction of ancient and immortal flows separately. The proof are exactly the same in both cases which explains why we state the estimates for the ancient case only.

We start by projecting the obstruction of $\tilde{g}(t)$ to be an exact Ricci flow on $\tilde{g}(t)$:
\begin{prop}(Projection on $\tilde{g}$)\label{prop-proj-tilde-g}
There exists $C>0$ such that for $t\leq T$,
\begin{equation*}
\begin{split}
\left|\langle -\Ric(\tilde{g}(t))+\Lambda\tilde{g}(t),\tilde{g}(t)\rangle_{L^2(\tilde{g}(t))}\right|\leq C\varepsilon(t)^8\delta(t)^{-10},\\
\left|\langle \partial_t\tilde{g}(t),\tilde{g}(t)\rangle_{L^2(\tilde{g}(t))}\right|\leq C\varepsilon(t)^8\delta(t)^{-8}.
\end{split}
\end{equation*}

\end{prop}

\begin{proof}
Thanks to Lemma \ref{rough-est-deviation-einstein}, one gets:
\begin{equation}\label{0-tr-obst}
\begin{split}
&\Big\langle -\Ric(\tilde{g}(t))+\Lambda \tilde{g}(t),\tilde{g}(t)\Big\rangle_{L^2(\tilde{g}(t))}=\\
&\Big\langle\chi_{\delta(t)} d_{\zeta(t)} g(\mathbf{R}^+_p(\zeta(t)))-(1-\chi_{\delta(t)})\frac{1}{2}\Li_{B_{g_o}(h_{\zeta(t)}^4)}(g_o),\tilde{g}(t)\big\rangle_{L^2(\tilde{g}(t))}\\
&+\Big\langle O(\rho(t)^2)\mathbbm{1}_{\{r_o\,\leq\,\delta(t)/2\}}+O(\varepsilon(t)^5\delta(t)^{-7})\mathbbm{1}_{\{\delta(t)/2\,\leq\,r_o\,\leq \,2\delta(t)\}}+\varepsilon(t)^8\rho(t)^{-10}\mathbbm{1}_{\{2\delta(t)\,\leq\,r_o\}},\tilde{g}(t)\Big\rangle_{L^2(\tilde{g}(t))}\\
&=\Big\langle\chi_{\delta(t)} d_{\zeta(t)} g(\mathbf{R}^+_p(\zeta(t)))-(1-\chi_{\delta(t)})\frac{1}{2}\Li_{B_{g_o}(h_{\zeta(t)}^4)}(g_o),\tilde{g}(t)\big\rangle_{L^2(\tilde{g}(t))}\\
&\quad+O(\delta(t)^6)+O(\varepsilon(t)^5\delta(t)^{-3})+O(\varepsilon(t)^8\delta(t)^{-10}).
\end{split}
\end{equation}
Now, since $\tr_{g_{\zeta(t)}}d_{\zeta(t)} g(\mathbf{R}^+_p(\zeta(t)))=0$ and that $\tilde{g}(t)-g_{\zeta(t)}=O(\rho(t)^2)$ on $\{r_o\,\leq\,2\delta(t)\}$ according to Lemma \ref{lemma-first-app-est}, we get that 
\begin{equation}\label{1-tr-obst}
\Big\langle\chi_{\delta(t)} d_{\zeta(t)} g(\mathbf{R}^+_p(\zeta(t))),\tilde{g}(t)\big\rangle_{L^2(\tilde{g}(t))}=O(\varepsilon(t)^4\delta(t)^2).
\end{equation}
Similarly, since $\tr_{g_o}\Li_{B_{g_o}(h_{\zeta(t)}^4)}(g_o)=2\div_{g_o}B_{g_o}(h_{\zeta(t)}^4)$ and that $\tilde{g}(t)-g_{o}=O(\varepsilon(t)^4\rho(t)^{-4})$ on $\{r_o\,\geq\,\delta(t)/2\}$ according to Lemma \ref{lemma-first-app-est}, an integration by parts shows that:
we get that 
\begin{equation}\label{2-tr-obst}
\begin{split}
\Big\langle(1-\chi_{\delta(t)})\Li_{B_{g_o}(h_{\zeta(t)}^4)}(g_o),\tilde{g}(t)\big\rangle_{L^2(\tilde{g}(t))}&=\Big\langle(1-\chi_{\delta(t)})\Li_{B_{g_o}(h_{\zeta(t)}^4)}(g_o),g_o\big\rangle_{L^2(g_o)}\\
&\quad+\int_MO(\varepsilon(t)^4\rho(t)^{-4})|\Li_{B_{g_o}(h_{\zeta(t)}^4)}(g_o)|_{g_o}\,d\mu_{g_o}\\
&=-2\langle B_{g_o}(h_{\zeta(t)}^4),\nabla^{g_o}(1-\chi_{\delta(t)})\rangle_{L^2(g_o)}+O(\varepsilon(t)^8\delta(t)^{-8})\\
&=O(\varepsilon(t)^4)+O(\varepsilon(t)^8\delta(t)^{-8}).
\end{split}
\end{equation}
Here we have used Proposition \ref{premiere obst orbifold 1 pt} in the last line.

Combining \eqref{0-tr-obst}, \eqref{1-tr-obst} and \eqref{2-tr-obst} together with \eqref{cond-param} leads to the expected estimate.


As for the projection on $\tilde{g}(t)$ of $\partial_t\tilde{g}(t)$, Lemma \ref{lemma-time-der-metric} gives:
\begin{equation*}
\begin{split}
\langle \partial_t \tilde{g}(t),\tilde{g}(t)\rangle_{L^2(\{r_o\,\leq\,2\delta(t)\})}&=\int_M O(\varepsilon(t)^{5}\delta(t)^{-5})\mathbbm{1}_{\{\delta(t)/2\,\leq\,r_o\,\leq \,2\delta(t)\}}+O(\varepsilon(t)^{8}\rho(t)^{-8})\mathbbm{1}_{\{2\delta(t)\,\leq\,r_o\}}\,d\mu_{\tilde{g}(t)}\\
&+\int_MO(\varepsilon(t)^{4^-}(\varepsilon(t)+r_o)^{-2^-})\mathbbm{1}_{\{r_o\,\leq \,2\delta(t)\}}\,d\mu_{\tilde{g}(t)}\\
&=O(\varepsilon(t)^5\delta(t)^{-1})+O(\varepsilon(t)^{4^-}\delta(t)^{2^+})+O(\varepsilon(t)^8\delta(t)^{-8})\\
&=O(\varepsilon(t)^8\delta(t)^{-8}),
\end{split}
\end{equation*}
as expected.
\end{proof}
The following proposition estimates the projection of the obstruction of $\tilde{g}(t)$ to be an exact Ricci flow on the family $(\tilde{\oi}_i(t))_i$:

\begin{prop}(Projection on $\tilde{\oi}_i$)\label{prop-o_1-first-app}
There exists $C>0$ such that for $t\leq T$ such that the family $(\tilde{\oi}_i(t))_i$ satisfies:
\begin{equation*}
\begin{split}
\left|\langle\tilde{\oi}_i(t),\tilde{g}(t)\rangle_{L^2(\tilde{g}(t))}\right|&\leq C\varepsilon(t)^4\delta(t)^2,\\
\Bigg|\langle\tilde{\oi}_i(t),\tilde{\oi}_j(t)\rangle_{L^2(\tilde{g}(t))}-2\pi^2\varepsilon(t)^4\delta_{ij}\Bigg|&\leq C\varepsilon(t)^8\delta(t)^{-4},\\
\Bigg|\langle -2\Ric(\tilde{g}(t))+2\Lambda\tilde{g}(t),\tilde{\oi}_i(t)\rangle_{L^2(\tilde{g}(t))}&-2\langle d_{\zeta(t)}g(\mathbf{R}_p^{+}(\zeta(t))) ,\oi_i(\zeta(t))\rangle_{L^2(g_{\zeta(t)})}\Bigg|\leq  C\varepsilon(t)^9\delta(t)^{-7},\\
\Bigg|\langle \partial_t\tilde{g}(t),\tilde{\oi}_i(t)\rangle_{L^2(\tilde{g}(t))}&-\langle d_{\zeta(t)}g(\dot{\zeta}(t)),\oi_i(\zeta(t))\rangle_{L^2(g_{\zeta(t)})}\Bigg|\leq C\varepsilon(t)^8\delta(t)^{-4},\\
\left|\langle\partial_t\tilde{\oi}_i(t),\tilde{\oi}_j(t)\rangle_{L^2(\tilde{g}(t))}\right|&\leq C\varepsilon(t)^{4},\\
\left|\langle\partial_t\tilde{c}(t),\tilde{\oi}_i(t)\rangle_{L^2(\tilde{g}(t))}\right|&\leq C\varepsilon(t)^{4}.
\end{split}
\end{equation*}

\end{prop}

\begin{rk}\label{rk BON oij a ref}
The second estimate in Proposition \ref{prop-o_1-first-app} uses the fact that the basis $(\oi_i)_{1\,\leq\, i\,\leq\, 3}$ of the $L^2$-kernel $\mathbf{O}(\eh)$ of the Lichnerowicz operator associated to the Eguchi-Hanson metric $\eh$ is naturally an orthonormal basis as explained in Section \ref{basics-EH}. If one glues one of Kronheimer's ALE Ricci flat metrics $g_{\zeta}$ instead of $\eh_{\zeta}$, then Section \ref{l2-kronheimer-proj} asks for an intermediate orthonormalization of the basis $(\oi_{ij})_{i,j}$ to reach the analogous results for the corresponding approximate kernel $(\tilde{\oi}_{ij})_{i,j}$.
\end{rk}
\begin{proof}
The estimate on $\langle\tilde{\oi}_i(t),\tilde{g}(t)\rangle_{L^2(\tilde{g}(t))}$ follows by integrating the estimate on $\tr_{\tilde{g}(t)} \tilde{\oi}_i(t)$ from Lemma \ref{estimates variations et laplacien o1}.

Next, \eqref{compute-norm-oi-zeta} from Section \ref{basics-EH} and the fact that $(\oi_i(\zeta(t)))_i$ is an $L^2(g_{\zeta(t)})$-orthogonal basis, $\langle\tilde{\oi}_i(t),\tilde{\oi}_j(t)\rangle_{L^2(\tilde{g}(t))}$ equals $2\pi^2\varepsilon(t)^4\delta_{ij}$ modulo integrals of the form:
\begin{equation*}
\begin{split}
\int_{\{r_o\,\leq\,\delta(t)/2\}}r_o^2\varepsilon(t)^8(\varepsilon(t)+r_o)^{-8}\,d\mu_{\tilde{g}(t)}&=O(\varepsilon(t)^6),\\
\int_{\{\delta(t)/2\,\leq\, r_o\,\leq\,2\delta(t)\}}\varepsilon(t)^8(\varepsilon(t)+r_o)^{-8}\,d\mu_{\tilde{g}(t)}&=O(\varepsilon(t)^8\delta(t)^{-4}),\\
\int_{\{r_o\,>\,\delta(t)/2\}}\varepsilon(t)^8(\varepsilon(t)+r_o)^{-8}\,d\mu_{g_{\zeta(t)}}&=O(\varepsilon(t)^8\delta(t)^{-4}).
\end{split}
\end{equation*}
Here we have used Lemma \ref{lemma-first-app-est} together with Lemma \ref{lemma-diff-compact-scal-prod}. We conclude by invoking \eqref{cond-param}, i.e. $\delta>1/2$.

Similarly to the proof of Proposition \ref{prop-proj-tilde-g}, Lemma \ref{rough-est-deviation-einstein} gives:
\begin{equation*}
\begin{split}
&\langle -\Ric(\tilde{g}(t))+\Lambda \tilde{g}(t),\tilde{\oi}_i(t)\rangle_{L^2(\tilde{g}(t))}=\\
&\langle\chi_{\delta(t)} d_{\zeta(t)} g(\mathbf{R}^+_p(\zeta(t)))+O(\rho(t)^2)+O(\varepsilon(t)^5\delta(t)^{-7})\mathbbm{1}_{\{\delta(t)/2\,\leq\,r_o\,\leq \,2\delta(t)\}},\tilde{\oi}_i(t)\rangle_{L^2(\{r_o\,\leq\,2\delta(t)\})}\\
&=\langle \chi_{\delta(t)} d_{\zeta(t)} g(\mathbf{R}^+_p(\zeta(t))),\tilde{\oi}_i(t)\rangle_{L^2(\tilde{g}(t))}+\int_{\{r_o\,\leq\,2\delta(t)\}}O(\varepsilon(t)^4(\varepsilon(t)+r_o)^{-2})\,d\mu_{\tilde{g}(t)}++O(\varepsilon(t)^9\delta(t)^{-7})\\
&=\langle d_{\zeta(t)} g(\mathbf{R}^+_p(\zeta(t))),\oi_i(t)\rangle_{L^2(g_{\zeta(t)})}+O(\varepsilon(t)^8\delta(t)^{-4})+O(\varepsilon(t)^4\delta(t)^2)++O(\varepsilon(t)^9\delta(t)^{-7})\\
&=\langle d_{\zeta(t)} g(\mathbf{R}^+_p(\zeta(t))),\oi_i(t)\rangle_{L^2(g_{\zeta(t)})}+O(\varepsilon(t)^9\delta(t)^{-7}).
\end{split}
\end{equation*}
Here we have used a similar reasoning that led to the previous estimate on $\langle \tilde{\oi}_j(t),\tilde{\oi}_i(t)\rangle_{L^2(\tilde{g}(t))}$ based on Lemma \ref{lemma-diff-compact-scal-prod} in the last line together with the fact that $\delta>5/9>1/3$ by \eqref{cond-param}. This gives us the desired third estimate.

Now, according to Lemma \ref{lemma-time-der-metric}, one can estimate in a similar way the following $L^2$-projection:
\begin{equation}
\begin{split}\label{proj-oi-partial-I}
\langle \partial_t\tilde{g}(t)&,\tilde{\oi}_i(t)\rangle_{L^2(\tilde{g}(t))}=\\
&\langle\chi'_{\delta(t)}O(\varepsilon(t)^{4}\delta(t)^{-4}) +\chi_{\delta(t)}(d_{\zeta(t)}g(\dot{\zeta})+O(\varepsilon(t)^{4^-}(\varepsilon(t)+r_o)^{-2^{-}})),\tilde{\oi}_i(t)\rangle_{L^2(\tilde{g}(t))}\\
&=O(\varepsilon(t)^8\delta(t)^{-4})+O(\varepsilon(t)^6)+\langle\chi_{\delta(t)}d_{\zeta(t)}g(\dot{\zeta}),\tilde{\oi}_i(t)\rangle_{L^2(\tilde{g}(t))}.
\end{split}
\end{equation}
Thanks to Lemma \ref{lemma-diff-compact-scal-prod} again, we get:
\begin{equation*}
\langle\chi_{\delta(t)}d_{\zeta(t)}g(\dot{\zeta}(t)),\tilde{\oi}_i(t)\rangle_{L^2(\tilde{g}(t))}=\langle d_{\zeta(t)}g(\dot{\zeta}(t)),\oi_i(\zeta(t))\rangle_{L^2(g_{\zeta(t)})}+O(\varepsilon(t)^8\delta(t)^{-4}),
\end{equation*}
which gives the expected fourth estimate once it is combined with \eqref{proj-oi-partial-I} together with the fact that $\delta>1/2$ according to \eqref{cond-param}.

The last two estimates can be proved with the help of Lemma \ref{estimates variations et laplacien o1} and Lemma \ref{est-basic-conf} together with the fact that $|\varepsilon(t)|\leq C|\delta(t)|$ according to \eqref{cond-param}.

\end{proof}
 Finally, we compute the time derivatives of the $L^2$-projections against $\tilde{\oi}_i(t)$ and $\tilde{g}(t)$, we get:
\begin{coro}\label{time-der-gram}
There exists $C>0$ such that for $t\leq T$ such that the time derivatives of the scalar products $\langle\tilde{\oi}_i(t),\tilde{\oi}_j(t)\rangle_{L^2(\tilde{g}(t))}$, $1\leq i,j\leq 3$, and $\langle\tilde{\oi}_i(t),\tilde{g}(t)\rangle_{L^2(\tilde{g}(t))}$ satisfy:
\begin{equation*}
\begin{split}
\Bigg|\frac{d}{dt}\langle\tilde{\oi}_i(t),\tilde{\oi}_j(t)\rangle_{L^2(\tilde{g}(t))}&-4\pi^2\langle \dot{\zeta}(t),\zeta(t)\rangle\delta_{ij}\Bigg|\leq C\varepsilon(t)^8\delta(t)^{-4},\\
\left| \frac{d}{dt}\langle\tilde{\oi}_i(t),\tilde{g}(t)\rangle_{L^2(\tilde{g}(t))}\right|&\leq C\varepsilon(t)^{4^-}\delta(t)^{2^+}.
 \end{split}
\end{equation*}
\end{coro}

\begin{proof}
Similarly to the proof of Proposition \ref{prop-o_1-first-app}, one gets:
\begin{equation*}
\begin{split}
&\frac{d}{dt}\langle\tilde{\oi}_i(t),\tilde{\oi}_j(t)\rangle_{L^2(\tilde{g}(t))}=\frac{d}{dt}\int_M\chi_{\delta(t)}^2\langle \oi_i(\zeta(t)),\oi_j(\zeta(t))\rangle_{\tilde{g}(t)}\,d\mu_{\tilde{g}(t)}\\
&=\frac{d}{dt}\left(\|\oi_i(\zeta(t)\|^2_{L^2(g_{\zeta(t)})}\delta_{ij}-\int_{\{r_o\,\geq \,\delta(t)/2\}}\langle \oi_i(\zeta(t)),\oi_j(\zeta(t))\rangle_{g_{\zeta(t)}}\,d\mu_{g_{\zeta(t)}}\right)\\
&+\frac{d}{dt}\left(\int_M\chi_{\delta(t)}^2\langle \oi_i(\zeta(t)),\oi_j(\zeta(t))\rangle_{\tilde{g}(t)}\,d\mu_{\tilde{g}(t)}-\int_{M}\mathbbm{1}_{\{r_o\,\leq \,\delta(t)/2\}}\langle \oi_i(\zeta(t)),\oi_j(\zeta(t))\rangle_{g_{\zeta(t)}}\,d\mu_{g_{\zeta(t)}}\right).
\end{split}
\end{equation*}
By taking into account the contribution of the boundary integrals, one  gets the expected estimate.

As in the proof of Lemma \ref{estimates variations et laplacien o1}, 
\begin{equation*}
\begin{split}
\frac{\partial}{\partial t}\tr_{\tilde{g}(t)} \tilde{\oi}_i(t)&=\chi'_{\delta(t)}O(\varepsilon(t)^4(\varepsilon(t)+r_o)^{-2})+\chi_{\delta(t)}\frac{\partial}{\partial t}\left[\left(\tr_{\tilde{g}(t)}-\tr_{g_{\zeta(t)}}\right)\oi_i(\zeta(t))\right]\\
&=\chi'_{\delta(t)}O(\varepsilon(t)^4(\varepsilon(t)+r_o)^{-2})+\chi_{\delta(t)}\frac{\partial}{\partial t}\left[\tilde{g}(t)^{-1}\ast g_{\zeta(t)}^{-1}\ast (\tilde{g}(t)-g_{\zeta(t)})\ast \oi_i(\zeta(t))\right]\\
&=\chi'_{\delta(t)}O(\varepsilon(t)^4\delta(t)^{-2})+\chi_{\delta(t)} O(\varepsilon(t)^{4^-}(\varepsilon(t)+r_o)^{-2^-}+\varepsilon(t)^{8^-}(\varepsilon(t)+r_o)^{-6^{-}})\\
&=\chi'_{\delta(t)}O(\varepsilon(t)^4\delta(t)^{-2})+\chi_{\delta(t)} O(\varepsilon(t)^{4^-}(\varepsilon(t)+r_o)^{-2^-}),
\end{split}
\end{equation*}
where we have used Lemma \ref{lemma-time-der-metric} together with Proposition \ref{prop-analy} in the penultimate line. 

By integrating the previous estimate on $M$ gives the expected result on the time derivative of $\langle\tilde{\oi}_i(t),\tilde{g}(t)\rangle_{L^2(\tilde{g}(t))}$ since 
\begin{equation*}
\begin{split}
\tr_{\tilde{g}(t)} \tilde{\oi}_i(t)\tr_{\tilde{g}(t)}\partial_t\tilde{g}(t)&=\chi_{\delta(t)}'O(\varepsilon(t)^8\delta(t)^{-6})+\chi_{\delta(t)}O(\varepsilon(t)^{8^-}(\varepsilon(t)+r_o)^{-4^{-}})\\
&=\chi_{\delta(t)}'O(\varepsilon(t)^8\delta(t)^{-8})+\chi_{\delta(t)}O(\varepsilon(t)^{4^-}(\varepsilon(t)+r_o)^{-2^{-}}),
\end{split}
\end{equation*}
and $\varepsilon(t)^8\delta(t)^{-4}\leq C \varepsilon(t)^{4^-}\delta(t)^{2^+}$ according to the choice of $\delta<2/3$ in \eqref{cond-param}. 

\end{proof}

The next proposition of this section estimates the projection of the obstruction of $\tilde{g}(t)$ to be an exact Ricci flow on the Lie derivative $\tilde{c}(t)$:
\begin{prop}(Projection on $\tilde{c}(t)$)\label{prop-proj-conf}
There exists $C>0$ such that for $t\leq T$,
\begin{equation*}
\begin{split}
\left|\langle \partial_t\tilde{g}(t)+2\Ric(\tilde{g}(t))-2\Lambda \tilde{g}(t),\tilde{c}(t)\rangle_{L^2(\tilde{g}(t))}\right|&\leq  C\varepsilon(t)^8\delta(t)^{-10},\\
\left|\langle \tilde{\oi}_i(t),\tilde{c}(t)\rangle_{L^2(\tilde{g}(t))}\right|&\leq C\varepsilon(t)^4,\\
\left|\frac{d}{dt}\langle \tilde{\oi}_i(t),\tilde{c}(t)\rangle_{L^2(\tilde{g}(t))}\right|&\leq C\varepsilon(t)^4,\\
\langle \tilde{g}(t),\tilde{c}(t)\rangle_{L^2(\tilde{g}(t))}&=0,\\
\left|\langle \partial_t\tilde{c}(t),\tilde{c}(t)\rangle_{L^2(\tilde{g}(t))}\right|&\leq C\varepsilon(t)^{4^-}\delta(t)^{-4^-},\\
\left|\frac{d}{dt}\langle\tilde{c},\tilde{c}\rangle_{L^2(\tilde{g}(t))}\right|&\leq C\varepsilon(t)^{4^-}\delta(t)^{-4^-}.
\end{split}
\end{equation*}
\end{prop}
\begin{proof}
Regarding the first estimate, observe that the integrand is supported in $\{r_o\,\leq\,2\delta(t)\}$. Lemma \ref{lemma-time-der-metric} and the fact that $\tilde{c}(t)$ is bounded thanks to Lemma \ref{est-basic-conf} give:
\begin{equation}\label{0-obst-c}
\begin{split}
&\langle \partial_t\tilde{g}(t)+2\Ric(\tilde{g}(t))-2\Lambda \tilde{g}(t),\tilde{c}(t)\rangle_{L^2(\tilde{g}(t))}=\\
&\int_{\{r_o\,\leq\,2\delta(t)\}}O\left(\varepsilon(t)^4(\varepsilon(t)+r_o)^{-4}+(\varepsilon(t)+r_o)^2+\varepsilon(t)^{4^-}(\varepsilon(t)+r_o)^{-2^{-}}\right)\,d\mu_{\tilde{g}(t)}\\
&+\int_{\{\delta(t)/2\,\leq\,r_o\,\leq\,2\delta(t)\}}O(\varepsilon(t)^5\delta(t)^{-7})\,d\mu_{\tilde{g}(t)}+\int_{\{\delta(t)/2\,\leq\,r_o\}}O(\varepsilon(t)^8\delta(t)^{-10})\,d\mu_{\tilde{g}(t)}\\
&+\Big\langle(1-\chi_{\delta(t)})\left(d_{\zeta(t)}h^4(\dot{\zeta})+\Li_{B_{g_o}(h_{\zeta(t)}^4)}(g_o)\right),\tilde{c}(t)\Big\rangle_{L^2(\tilde{g}(t))}\\
&=O(\varepsilon(t)^4\log(1+\delta(t)/\varepsilon(t)))+O(\delta(t)^6)+O(\varepsilon(t)^{4^-}\delta(t)^{2^+})+O(\varepsilon(t)^5\delta(t)^{-3})+O(\varepsilon(t)^8\delta(t)^{-10})\\
&+\Big\langle(1-\chi_{\delta(t)})\left(d_{\zeta(t)}h^4(\dot{\zeta})+\Li_{B_{g_o}(h_{\zeta(t)}^4)}(g_o)\right),\tilde{c}(t)\Big\rangle_{L^2(\tilde{g}(t))}\\
&=O(\varepsilon(t)^8\delta(t)^{-10})+\Big\langle(1-\chi_{\delta(t)})\left(d_{\zeta(t)}h^4(\dot{\zeta})+\Li_{B_{g_o}(h_{\zeta(t)}^4)}(g_o)\right),\tilde{c}(t)\Big\rangle_{L^2(\tilde{g}(t))},
\end{split}
\end{equation}
according to \eqref{cond-param}.

Now, on the one hand, since $\nabla^{\tilde{g}(t)}\tilde{v}(t)=O(\rho(t))$ by Lemma \ref{est-basic-conf}, an integration by parts shows:
 \begin{equation}\label{1-obst-c}
\begin{split}
 \Big\langle(1-\chi_{\delta(t)})d_{\zeta(t)}h^4(\dot{\zeta}),\tilde{c}(t)\Big\rangle_{L^2(\tilde{g}(t))}&=-\Big\langle\div_{\tilde{g}(t)}\left((1-\chi_{\delta(t)})d_{\zeta(t)}h^4(\dot{\zeta})\right),\nabla^{\tilde{g}(t)}\tilde{v}(t)\Big\rangle_{L^2(\tilde{g}(t))}\\
 &=O(\varepsilon(t)^4)-\Big\langle(1-\chi_{\delta(t)})\div_{\tilde{g}(t)}\left(d_{\zeta(t)}h^4(\dot{\zeta})\right),\nabla^{\tilde{g}(t)}\tilde{v}(t)\Big\rangle_{L^2(\tilde{g}(t))}\\
 &=O(\varepsilon(t)^4)-\Big\langle(1-\chi_{\delta(t)})\div_{g_o}\left(d_{\zeta(t)}h^4(\dot{\zeta})\right),\nabla^{\tilde{g}(t)}\tilde{v}(t)\Big\rangle_{L^2(\tilde{g}(t))}\\
 &\quad+O(\varepsilon(t)^4\delta(t)^{-8})\\
 &=O(\varepsilon(t)^4)+O(\varepsilon(t)^4\delta(t)^{-2})+O(\varepsilon(t)^8\delta(t)^{-8})\\
 &=O(\varepsilon(t)^4\delta(t)^{-2}).
 \end{split}
\end{equation}
Here we have used Proposition \ref{premiere obst orbifold 1 pt} in the second line and the fourth equality to derive that $\div_{g_o}d_{\zeta(t)}h^4(\dot{\zeta})=O(\varepsilon(t)^4\rho(t)^{-3})$. Lemma \ref{lemma-first-app-est} is used in the third equality.

On the other hand, since $\nabla^{\tilde{g}(t)}\tilde{v}(t)=O(\rho(t))$ by Lemma \ref{est-basic-conf}, an integration by parts shows:
 \begin{equation*}
\begin{split}
& \Big\langle(1-\chi_{\delta(t)})\Li_{B_{g_o}(h_{\zeta(t)}^4)}(g_o),\tilde{c}(t)\Big\rangle_{L^2(\tilde{g}(t))}=-\Big\langle\div_{\tilde{g}(t)}\left((1-\chi_{\delta(t)})\Li_{B_{g_o}(h_{\zeta(t)}^4)}(g_o)\right),\nabla^{\tilde{g}(t)}\tilde{v}(t)\Big\rangle_{L^2(\tilde{g}(t))}\\
&=-\Big\langle\div_{g_o}\left(\Li_{B_{g_o}(h_{\zeta(t)}^4)}(g_o)\right),(1-\chi_{\delta(t)})\nabla^{\tilde{g}(t)}\tilde{v}(t)\Big\rangle_{L^2(\tilde{g}(t))}+O(\varepsilon(t)^8\delta(t)^{-8}).
 \end{split}
\end{equation*}
Let us recall the following Bochner formula that holds for any $C^1$ vector field $X$:
$$\div_{g_o}\mathcal{L}_X(g_o)=d\div_{g_o}X+g_o(\Delta_{g_o}X,\cdot)+\Ric(g_o)(X,\cdot).$$
Let us apply this formula to $X=B_{g_o}(h_{\zeta(t)}^4)$ to derive that:
\begin{equation*}
\begin{split}
\div_{g_o}\mathcal{L}_{B_{g_o}(h_{\zeta(t)}^4)}(g_o)&=d(\div_{g_o}B_{g_o}(h_{\zeta(t)}^4))+g_o(\Delta_{g_o}B_{g_o}(h_{\zeta(t)}^4),\cdot)+\Lambda g_o(B_{g_o}(h_{\zeta(t)}^4),\cdot)\\
&=d(\div_{g_o}B_{g_o}(h_{\zeta(t)}^4))-\Lambda g_o(B_{g_o}(h_{\zeta(t)}^4),\cdot),
\end{split}
\end{equation*}
where we have used the Bochner formula $\div_{g_o}(\Delta_{L,g_o}h)=\Delta_{g_o}(\div_{g_o}h)$ due to Lichnerowicz \cite{Lic-Pro} for any symmetric $2$-tensor $h$ on an Einstein manifold together with Proposition \ref{premiere obst orbifold 1 pt} in the last line. 
To summarize,
 \begin{equation*}
\begin{split}
 \Big\langle(1-\chi_{\delta(t)})\Li_{B_{g_o}(h_{\zeta(t)}^4)}(g_o),\tilde{c}(t)\Big\rangle_{L^2(\tilde{g}(t))}&=-\Big\langle \nabla^{\tilde{g}(t)}(\div_{g_o}(B_{g_o}(h_{\zeta(t)}^4)(g_o))),(1-\chi_{\delta(t)})\nabla^{\tilde{g}(t)}\tilde{v}(t)\Big\rangle_{L^2(\tilde{g}(t))}\\
&\quad+\Lambda \Big\langle B_{g_o}(h^4_{\zeta(t)}),(1-\chi_{\delta(t)})\nabla^{\tilde{g}(t)}\tilde{v}(t)\Big\rangle_{L^2(\tilde{g}(t))}+O(\varepsilon(t)^8\delta(t)^{-8})\\
&=\Big\langle \div_{g_o}(B_{g_o}(h_{\zeta(t)}^4)(g_o))),\div_{\tilde{g}(t)}\left((1-\chi_{\delta(t)})\nabla^{\tilde{g}(t)}\tilde{v}(t)\right)\Big\rangle_{L^2(\tilde{g}(t))}\\
&\quad +O(\varepsilon(t)^4\delta(t)^{-2})+O(\varepsilon(t)^8\delta(t)^{-8})\\
&=\Big\langle \div_{g_o}(B_{g_o}(h_{\zeta(t)}^4)(g_o))),(1-\chi_{\delta(t)})\Delta_{\tilde{g}(t)}\tilde{v}(t)\Big\rangle_{L^2(\tilde{g}(t))}\\
&\quad +O(\varepsilon(t)^4)+O(\varepsilon(t)^4\delta(t)^{-2})+O(\varepsilon(t)^8\delta(t)^{-8}).
 \end{split}
\end{equation*}
The last term on the righthand side of the previous estimates can be taken care of as follows:
\begin{equation*}
\begin{split}
&\Big\langle \div_{g_o}(B_{g_o}(h_{\zeta(t)}^4)(g_o))),(1-\chi_{\delta(t)})\Delta_{\tilde{g}(t)}\tilde{v}(t)\Big\rangle_{L^2(\tilde{g}(t))}=\\
&\Big\langle \div_{g_o}(B_{g_o}(h_{\zeta(t)}^4)(g_o))),(1-\chi_{\delta(t)})\left(\Delta_{\tilde{g}(t)}\tilde{v}(t)+(\Lambda+1)\tilde{v}(t)\right)\Big\rangle_{L^2(\tilde{g}(t))}\\
&\quad +O(\varepsilon(t)^4\delta(t)^{-2})+O(\varepsilon(t)^8\delta(t)^{-8})\\
&=O(\varepsilon(t)^8\delta(t)^{-8}).
 \end{split}
\end{equation*}
Here we have used the fact that $\tilde{v}(t)=O(\rho(t)^2)$ and $$\Delta_{\tilde{g}(t)}\tilde{v}(t)+(\Lambda+1)\tilde{v}(t)=O(\rho(t)^2)\mathbbm{1}_{\{\rho(t)\,\leq\,2\delta(t)\}}+O(\varepsilon(t)^4\rho(t)^{-4}),$$ thanks to Lemma \ref{est-basic-conf}.
As an intermediate conclusion, \eqref{cond-param} ensures,
 \begin{equation}\label{2-obst-c}
\begin{split}
 \Big\langle(1-\chi_{\delta(t)})\Li_{B_{g_o}(h_{\zeta(t)}^4)}(g_o),\tilde{c}(t)\Big\rangle_{L^2(\tilde{g}(t))}=O(\varepsilon(t)^4\delta(t)^{-2}).
  \end{split}
\end{equation}
The combination of \eqref{0-obst-c}, \eqref{1-obst-c} and \eqref{2-obst-c} leads to the desired estimate.

For the second estimate, integrate by parts to get:
\begin{equation*}
\langle \tilde{\oi}_i(t),\tilde{c}(t)\rangle_{L^2(\tilde{g}(t))}=-\langle \div_{\tilde{g}(t)}\tilde{\oi}_i(t),\nabla^{\tilde{g}(t)}\tilde{v}(t)\rangle_{L^2(\tilde{g}(t))}.
\end{equation*}
Therefore, since $|\nabla^{\tilde{g}(t)}\tilde{v}(t)|_{\tilde{g}(t)}=O((\varepsilon(t)+r_o))$ by Lemma \ref{est-basic-conf}, the previous estimate combined with Lemma \ref{estimates variations et laplacien o1} allows us to conclude that:
\begin{equation*}
\langle \tilde{\oi}_i(t),\tilde{c}(t)\rangle_{L^2(\tilde{g}(t))}=O(\varepsilon(t)^4\delta(t)^2)+O(\varepsilon(t)^4)=O(\varepsilon(t)^4),
\end{equation*}
as desired.
The estimate on the time-derivative of $\langle \tilde{\oi}_i(t),\tilde{c}(t)\rangle_{L^2(\tilde{g}(t))}$ is obtained by differentiating (Lemma \ref{easy-time-der-abs-2-tensors-prelim}):
\begin{equation*}
\begin{split}
\frac{d}{dt}\langle \tilde{\oi}_i(t),\tilde{c}(t)\rangle_{L^2(\tilde{g}(t))}&=\langle \partial_t\tilde{\oi}_i(t),\tilde{c}(t)\rangle_{L^2(\tilde{g}(t))}+\langle \tilde{\oi}_i(t),\partial_t\tilde{c}(t)\rangle_{L^2(\tilde{g}(t))}+\langle \tilde{\oi}_i(t),\tilde{c}(t)\ast \partial_t\tilde{g}(t)\rangle_{L^2(\tilde{g}(t))}\\
&=O(\varepsilon(t)^4)-\langle\div_{\tilde{g}(t)}(\partial_t \tilde{\oi}_i(t)),\nabla^{\tilde{g}(t)}\tilde{v}(t)\rangle_{L^2(\tilde{g}(t))}\\
&=O(\varepsilon(t)^4),
\end{split}
\end{equation*}
where we have used Lemmata \ref{lemma-time-der-metric} and \ref{est-basic-conf} and Proposition \ref{prop-o_1-first-app} together with an integration by parts in the second line. In the last line, we invoke Lemma \ref{estimates variations et laplacien o1}.

The antepenultimate estimate $\langle \tilde{g}(t),\tilde{c}(t)\rangle_{L^2(\tilde{g}(t))}=0$ is straightforward since $\langle \tilde{g}(t),\tilde{c}(t)\rangle_{\tilde{g}(t)}=\Delta_{\tilde{g}(t)}\tilde{v}(t)$. The penultimate estimate is obtained thanks to Lemma \ref{est-basic-conf} and the last estimate can be proved by combining the previous one together with Lemma \ref{lemma-time-der-metric} and Lemma \ref{easy-time-der-abs-2-tensors-prelim}.
\end{proof} 

The next proposition establishes H\"older estimates on the scalar products between two elements of the approximate kernel.

\begin{prop}\label{prop-easy-l2-proj-holder}
There exists $C>0$ such that for $t\leq T$, 
{
\begin{equation*}
\begin{split}
[\langle \tilde{\oi}_i(t),\tilde{\oi}_j(t)\rangle_{L^2(\tilde{g}(t))}]_{\alpha,t,r}&\leq Cr^2\left(\varepsilon(t)^{4}\delta_{ij}+\varepsilon(t)^{8}\delta(t)^{-4}\right),\\
[\langle \tilde{\oi}_i(t),\tilde{g}(t)\rangle_{L^2(\tilde{g}(t))}]_{\alpha,t,r}&\leq Cr^2\varepsilon(t)^{4^-}\delta(t)^{2^+},\\
[\langle \tilde{\oi}_i(t),\tilde{c}(t)\rangle_{L^2(\tilde{g}(t))}]_{\alpha,t,r}&\leq Cr^2\varepsilon(t)^{4},\\
[\langle \tilde{g}(t),\tilde{g}(t)\rangle_{L^2(\tilde{g}(t))}]_{\alpha,t,r}&\leq Cr^2\varepsilon(t)^{8}\delta(t)^{-8},\\
[\langle \tilde{g}(t),\tilde{c}(t)\rangle_{L^2(\tilde{g}(t))}]_{\alpha,t,r}&=0,\\
[\langle \tilde{c}(t),\tilde{c}(t)\rangle_{L^2(\tilde{g}(t))}]_{\alpha,t,r}&\leq Cr^2\varepsilon(t)^{4-}\delta(t)^{-4^-}.
\end{split}
\end{equation*}}
\end{prop}
\begin{proof}
Apply \eqref{semi-holder-diff-tens} together with the time-derivative estimates from Lemma \ref{lemma-time-der-metric}, Proposition \ref{prop-proj-tilde-g}, Corollary \ref{time-der-gram} and Proposition \ref{prop-proj-conf}.
\end{proof}
Finally, we give a careful H\"older estimate of the $L^2$-projections of the obstruction for the first approximation $\tilde{g}(t)$ to be an exact solution to the normalized Ricci flow:
\begin{lemma}\label{lemma-holder-proj-ob}
There exists $C>0$ such that for $t\leq T$, for $\eta:=\dot{\zeta}-2\mathbf{R}_p^+(\zeta)$,
{
\begin{equation*}
\begin{split}
\left[\langle \partial_t\tilde{g}+2\Ric(\tilde{g})-2\Lambda\tilde{g},\tilde{\oi}_i\rangle_{L^2(\tilde{g})}-\langle d_{\zeta}g(\eta) ,\oi_i(\zeta)\rangle_{L^2(g_{\zeta})}\right]_{\alpha,t,r}&\leq C\left(r^2+\frac{[\dot{\zeta}]_{\alpha,t,r}}{\varepsilon(t)^2}\right)\varepsilon(t)^{9^-}\delta(t)^{-7^-},\\
\left[\langle \partial_t\tilde{g}+2\Ric(\tilde{g})-2\Lambda \tilde{g},\tilde{g}\rangle_{L^2(\tilde{g})}\right]_{\alpha,t,r}&\leq C\left(r^2+\frac{[\dot{\zeta}]_{\alpha,t,r}}{\varepsilon(t)^2}\right)\varepsilon(t)^{8}\delta(t)^{-10},\\
\left[\langle \partial_t\tilde{g}+2\Ric(\tilde{g})-2\Lambda \tilde{g},\tilde{c}\rangle_{L^2(\tilde{g})}\right]_{\alpha,t,r}&\leq C\left(r^2+\frac{[\dot{\zeta}]_{\alpha,t,r}}{\varepsilon(t)^2}\right)\varepsilon(t)^{8}\delta(t)^{-10}.
\end{split}
\end{equation*}}
\end{lemma}
\begin{proof}
As a first step, let us notice that Lemmata \ref{lemma-time-der-metric}, \ref{estimates variations et laplacien o1} and \ref{holder-semi-obstruction} and Lemma \ref{lemma-scal-prod-holder} give:
\begin{equation*}
\begin{split}
\Bigg[\langle \partial_t\tilde{g}&+2\Ric(\tilde{g})-2\Lambda\tilde{g}-\chi_{\delta}d_{\zeta}g(\eta)-(1-\chi_{\delta})\left(d_{\zeta}h^4(\dot{\zeta})+\Li_{B_{g_o}(h_{\zeta}^4)}(g_o)\right),\tilde{\oi}_i\rangle_{L^2(\tilde{g})}\Bigg]_{\alpha,t,r}\\
&\leq C\left(r^2+\frac{[\dot{\zeta}]_{\alpha,t,r}}{\varepsilon(t)^2}\right)\left(\varepsilon(t)^{4^-}\delta(t)^{2^+}+\varepsilon(t)^{9^+}\delta(t)^{-7^-}\right)\\
&\leq C\left(r^2+\frac{[\dot{\zeta}]_{\alpha,t,r}}{\varepsilon(t)^2}\right)\varepsilon(t)^{9^+}\delta(t)^{-7^-}.
\end{split}
\end{equation*}
Here we have used that $\delta>5/9$ according to \eqref{cond-param}.

Now, thanks to Proposition \ref{premiere obst orbifold 1 pt},  Lemma \ref{estimates variations et laplacien o1} and Lemma \ref{lemma-scal-prod-holder}, we get:
\begin{equation*}
\begin{split}
\Bigg[\langle (1-\chi_{\delta})\left(d_{\zeta}h^4(\dot{\zeta})+\Li_{B_{g_o}(h_{\zeta}^4)}(g_o)\right),\tilde{\oi}_i\rangle_{L^2(\tilde{g})}\Bigg]_{\alpha,t,r}&\leq  C\left(r^2+\frac{[\dot{\zeta}]_{\alpha,t,r}}{\varepsilon(t)^2}\right)\varepsilon(t)^{8^{-}}\delta(t)^{-4^-}.
\end{split}
\end{equation*}

To conclude the first estimate, it suffices to invoke \eqref{cond-param} and Lemma \ref{lemma-diff-compact-scal-prod} to tensors that are (difference quotients of) $d_{\zeta}g(\eta)$ and $\oi_i(t)$:
\begin{equation*}
\begin{split}
\left[\langle \chi_{\delta}d_{\zeta}g(\eta),\tilde{\oi}_i\rangle_{L^2(\tilde{g})}-\langle d_{\zeta}g(\eta) ,\oi_i(\zeta)\rangle_{L^2(g_{\zeta})}\right]_{\alpha,t,r}&\leq C\left(r^2+\frac{[\eta]_{\alpha,t,r}}{\varepsilon(t)^2}\right)\varepsilon(t)^{8^-}\delta(t)^{-4^-}.
\end{split}
\end{equation*}
The last two projections against $\tilde{g}(t)$ and $\tilde{c}(t)$ can be estimated in the same spirit as for the previous one: invoking Lemma \ref{holder-semi-obstruction} and Lemma \ref{lemma-scal-prod-holder} give,
\begin{equation*}
\begin{split}
\Big[\langle \partial_t\tilde{g}+2\Ric(\tilde{g})&-2\Lambda \tilde{g},\tilde{g}\rangle_{L^2(\tilde{g})}\Big]_{\alpha,t,r}\leq\\
& \left[\langle \chi_{\delta}d_{\zeta}g(\eta),\tilde{g}\rangle_{L^2(\tilde{g})}\right]_{\alpha,t,r}+\Big[\langle (1-\chi_{\delta})\left(d_{\zeta}h^4(\dot{\zeta})+\Li_{B_{g_o}(h_{\zeta}^4)}(g_o)\right),\tilde{g}\rangle_{L^2(\tilde{g})}\Big]_{\alpha,t,r}\\
&\quad+C\left(r^2+\frac{[\dot{\zeta}]_{\alpha,t,r}}{\varepsilon(t)^2}\right)\varepsilon(t)^{8}\delta(t)^{-10}\\
&\leq \left[\langle \chi_{\delta}d_{\zeta}g(\eta),\chi_{\delta}g_{\zeta}\rangle_{L^2(\tilde{g})}\right]_{\alpha,t,r}+\left[\langle \chi_{\delta}d_{\zeta}g(\eta),\chi_{\delta}h_{2,\zeta}+(1-\chi_{\delta})g_o\rangle_{L^2(\tilde{g})}\right]_{\alpha,t,r}\\
&\quad+\Big[\langle (1-\chi_{\delta})\left(d_{\zeta}h^4(\dot{\zeta})+\Li_{B_{g_o}(h_{\zeta}^4)}(g_o)\right),\tilde{g}\rangle_{L^2(\tilde{g})}\Big]_{\alpha,t,r}\\
&\quad+C\left(r^2+\frac{[\dot{\zeta}]_{\alpha,t,r}}{\varepsilon(t)^2}\right)\varepsilon(t)^{8}\delta(t)^{-10}.
\end{split}
\end{equation*}
Now, Proposition \ref{prop-analy} and Lemma \ref{lemma-diff-compact-scal-prod}  and  the fact that $\tr_{g_{\zeta}}d_{\zeta}g(\eta)=0$ give:
\begin{equation*}
\begin{split}
\left[\langle \chi_{\delta}d_{\zeta}g(\eta),\chi_{\delta}g_{\zeta}\rangle_{L^2(\tilde{g})}\right]_{\alpha,t,r}&\leq C\left(r^2+\frac{[{\dot{\zeta}}]_{\alpha,t,r}}{\varepsilon(t)^2}\right)\left(\varepsilon(t)^{4^-}\delta(t)^{2^+}+\varepsilon(t)^4\right),\\
\left[\langle \chi_{\delta}d_{\zeta}g(\eta),\chi_{\delta}h_{2,\zeta}+(1-\chi_{\delta})g_o\rangle_{L^2(\tilde{g})}\right]_{\alpha,t,r}&\leq C\left(r^2+\frac{[{\dot{\zeta}}]_{\alpha,t,r}}{\varepsilon(t)^2}\right)\left(\varepsilon(t)^{4^-}\delta(t)^{2^+}+\varepsilon(t)^4\right).
\end{split}
\end{equation*}
Here we have used that $|\tr_{g_o}(d_{\zeta}g(\eta))|\leq C\varepsilon(t)^4\delta(t)^{-2}$ on the annulus $\{\delta(t)/2<r_o<2\delta(t)\}$ thanks to Lemma \ref{lemma-first-app-est}.

Similarly, since $\tr_{g_o}h^4_{\zeta}=0$ and $\tr_{g_o}\left(\Li_{B_{g_o}(h_{\zeta}^4)}(g_o)\right)=2\div_{g_o}B_{g_o}(h_{\zeta}^4)$, according to Proposition \ref{premiere obst orbifold 1 pt},
\begin{equation*}
\begin{split}
&\Big[\langle (1-\chi_{\delta})\left(d_{\zeta}h^4(\dot{\zeta})+\Li_{B_{g_o}(h_{\zeta}^4)}(g_o)\right),\tilde{g}\rangle_{L^2(\tilde{g})}\Big]_{\alpha,t,r}\leq\\
&\Big[\langle (1-\chi_{\delta})\left(d_{\zeta}h^4(\dot{\zeta})+\Li_{B_{g_o}(h_{\zeta}^4)}(g_o)\right), (1-\chi_{\delta})g_o\rangle_{L^2(\tilde{g})}\Big]_{\alpha,t,r}+C\left(r^2+\frac{[{\dot{\zeta}}]_{\alpha,t,r}}{\varepsilon(t)^2}\right)\varepsilon(t)^{8^-}\delta(t)^{-8^-}\\
&\leq\Big[\int_M (1-\chi_{\delta})\left(d_{\zeta}h^4(\dot{\zeta})+\Li_{B_{g_o}(h_{\zeta}^4)}(g_o)\right)\ast (\tilde{g}-g_o)\,d\mu_{\tilde{g}}\Big]_{\alpha,t,r}+C\left(r^2+\frac{[{\dot{\zeta}}]_{\alpha,t,r}}{\varepsilon(t)^2}\right)\varepsilon(t)^{8^-}\delta(t)^{-8^-}\\
&\leq C\left(r^2+\frac{[{\dot{\zeta}}]_{\alpha,t,r}}{\varepsilon(t)^2}\right)\varepsilon(t)^{8^-}\delta(t)^{-8^-},
\end{split}
\end{equation*}
where we have invoked Proposition \ref{lemma-time-der-metric} in the last line. This leads to the second desired estimate.

As for the final estimate, it can be proved along the same lines by using Lemma \ref{holder-semi-obstruction}. 
\end{proof}

\section{Liouville theorems}\label{sec:liouville}

We now present the Liouville theorems that will let us deal with the approximate kernel $\textbf{O}(t)$ of the linearized operator (i.e. the Lichnerowicz Laplacian associated to $\tilde{g}(t)$). 

\begin{rk}\label{rk:ALE instable}
	For simplicity, and because this is the most interesting case, we assume that all of the orbifolds and ALE spaces are stable. The proofs also work if there were other unstable directions, by forcing the tensors to be $L^2$-orthogonal to all of the nonnegative eigentensors (of finite dimension) of the Lichnerowicz Laplacian---not just the zero eigenspace.
\end{rk}

\subsection{On stable ALE}

In this section, we state a Liouville-type result for the parabolic Lichnerowicz equation where the background metric is a \textit{stable} Ricci flat ALE metric:
\begin{defn}\label{defn-sta-ale}
A Ricci flat ALE metric $(N,g_b)$ is \textit{stable} if the associated Lichnerowicz operator is non-positive in the $L^2$ sense: $\langle \Delta_{L,g_b}h,h\rangle_{L^2(g_b)}\leq 0$ for all $h\in L^2(g_b)$ with equality if and only if $h\in\ker_{L^2}\Delta_{L,g_b}$.
\end{defn}

Let us notice that Eguchi-Hanson metrics and Kronheimer's instantons of \cite{Kronheimer1989ALE} are all stable in the aforementioned sense: see Proposition \ref{linear-sta-instantons}.

{ \subsubsection{For ancient flows}}

\begin{prop}\label{liouville-ALE}
Let $(M^4,g_b)$ be a stable Ricci flat ALE metric in the sense of Definition \ref{defn-sta-ale}. 
Let $(h(t))_{t\,<\,0}$ be an ancient solution to the homogeneous heat equation $\partial_th=\Delta_{L,g_b}h$ on $M\times(-\infty,0)$ such that on $M\times(-\infty,0)$, $|h(t)|_{g_b}\leq C(1+r_b)^{-\sigma}$ for some uniform-in-time positive constants $C$ and $\sigma$.

If $h(t)$ is orthogonal to the $L^2$-kernel of $\Delta_{L,g_b}$ for each time $t<0$ then $h(t)\equiv 0$ for $t<0$.
\end{prop}

\begin{rk}\label{rk-pairing-ALE}
Notice that Proposition \ref{liouville-ALE} does not assume the solution $h(t)$ to lie in $L^2$ for each time $t<0$ since $\sigma$ is an arbitrary positive number. In particular, the orthogonality of $h(t)$ to an element of $\ker_{L^2}\Delta_{L,g_b}$ is not well-defined a priori. We make a slight abuse of terminology thanks to the following remark: any element of $\ker_{L^2}\Delta_{L,g_b}$ can be shown to decay as fast as $r_b^{-4}$ at infinity \cite{Ban-Kas-Nak} which makes the integrand $\langle h(t),\oi\rangle_{g_b}$ Lebesgue integrable for any $\oi\in \ker_{L^2}\Delta_{L,g_b}$ by assumption on the pointwise decay of $h(t)$. 
\end{rk}

\begin{proof}[Proof of Proposition \ref{liouville-ALE}]
The proof goes verbatim as in the proof of \cite[Proposition $5.2$]{Bre-Kap}.
\end{proof}

{\subsubsection{For immortal flows}}

We now recall a general uniqueness theorem for solutions of the linear heat flow for a complete background metric due to Karp and Li:
\begin{theo}(\cite[Theorem $7.39$]{Cho-Lu-Ni-Boo})\label{theo-karp-li}
Let $(M^n,g)$ be a complete Riemannian manifold. If $u$ is a smooth subsolution of the heat equation $\partial_tu\leq \Delta_gu$ on $M\times [0,T]$ with $u(\cdot,0)\leq 0$ and if there exists $C>0$ such that
\begin{equation*}
\int_0^T\int_Mu_+^2(x,t)e^{-C d_g(p,x)^2}\,d\mu_{g}dt<+\infty,
\end{equation*}
then $u\leq 0$ on $M\times[0,T]$.
\end{theo}
As a consequence of Theorem \ref{theo-karp-li}, we get a Liouville theorem for solutions of the linear heat flow on metrics with nonnegative Ricci curvature with initial conditions. 

\begin{prop}\label{liouville-ALE-initial}
Let $(M^n,g_b)$ be a complete Riemannian metric with nonnegative Ricci curvature and bounded curvature, i.e. $\sup_M|\Rm(g_b)|_{g_b}<+\infty$. 
Let $(h(t))_{t\in [0,T]}$ be a solution to the homogeneous heat equation $\partial_th=\Delta_{L,g_b}h$ on $M\times[0,T]$ such that on $M\times[0,T]$, $|h(t)|_{g_b}\leq e^{C_0(d_{g_b}(p\,,\,\cdot)+1)^2}$ for some uniform-in-time positive constant $C_0>0$.

If $h(0)\equiv 0$, then $h(t)\equiv 0$ for $t\in[0,T]$.
\end{prop}
\begin{rk}
The assumption on the Ricci curvature is convenient here because we will apply Proposition \ref{liouville-ALE-initial} to a Ricci flat ALE metric. An inspection of the proof shows however that assuming the Ricci curvature to be bounded from below or even satisfying $\Ric(g)\geq -C(1+r_p^2)g$ for some positive constant $C$ would be sufficient: see for instance the proof of \cite[Corollary $7.40$]{Cho-Lu-Ni-Boo}.
\end{rk}
\begin{proof}
 Let us consider the norm $|h(t)|^2_{g_b}$ for each $t\in[0,T]$ and observe that it satisfies schematically:
 \begin{equation*}
\frac{\partial}{\partial t}|h(t)|^2_{g_b}\leq\Delta_{g_b}|h(t)|^2_{g_b}+c(n)|\Rm(g_b)|_{g_b}|h(t)|^2_{g_b}\leq \Delta_{g_b}|h(t)|^2_{g_b}+c(n)C|h(t)|^2_{g_b},
\end{equation*}
       where $C:=\sup_M|\Rm(g_b)|_{g_b}.$ Then $u(x,t):=e^{-c(n)Ct}|h(t)|^2_{g_b}$ is a subsolution of the heat equation with background metric $g_b$ on $M^n\times[0,T]$ and by assumption on the Ricci curvature and the growth of $h(t)$:
       \begin{equation*}
\int_0^T\int_Mu^2(x,t)e^{-C_1 d_{g_b}(p,x)^2}\,d\mu_{g_b}dt<+\infty,
\end{equation*}
as long as $C_1>C_0$ since the volume growth is polynomial thanks to Bishop's comparison theorem. Theorem \ref{theo-karp-li} lets us conclude that $h\equiv 0$ on $M\times[0,T]$ as expected.
\end{proof}

\subsection{On stable Einstein orbifolds}

Let us now prove Liouville theorems for the other main building blocks of our construction, Einstein orbifolds.
 \subsubsection{Liouville theorems for ancient flows}


Let $(M_o^4,g_o)$ be a $4$-dimensional Einstein orbifold with isolated singularities and assume that it has constant sectional curvature $1$. We still denote the Ricci constant by $\Lambda$ which equals $3$ here. Then we investigate the triviality of ancient solutions to the heat flow associated to the linearized operator of $-2\Ric(g)+2\Lambda g$ on such spaces. Let us recall that such an orbifold has at most two isolated singularities: see \cite[Lemma $5.1$]{Chen-Tang-Zhu} for instance. In case it has two, $(M_o,g_o)$ has two opposite singularities and is isometric to a suspension over $\mathbb{S}^3/\Gamma$ endowed with a metric of constant curvature $1$ where $\Gamma\neq\{\Id\}$ is a finite group of $\operatorname{SO(4)}$ acting freely on $\mathbb{S}^3$.

We first recall the $L^2$-spectrum of (minus) the Laplacian acting on functions on a metric suspension that lies below the value $2\Lambda$:

\begin{lemma}\label{lemma-spec-susp}
Let $(M_o^4,g_o)$ be a metric suspension over $\mathbb{S}^3/\Gamma$ endowed with a metric of constant curvature $1$ with Einstein constant $\Lambda$. Then the $L^2$-spectrum of $-\Delta_{g_o}$ acting on functions that lies below $2\Lambda$ consists of $\{0,\Lambda+1\}$. Moreover, $\ker(-\Delta_{g_o})=\RR$ and $\ker(-\Delta_{g_o}+\Lambda+1)=\RR v_o$ where $v_o(\cdot)=\cos(\cdot)$ outside the two isolated singularities of $M_o$ and $\nabla^{g_o}v_o$ is a conformal vector field such that $\Li_{\nabla^{g_o}v_o}(g_o)=-2v_og_o.$
\end{lemma}

\begin{proof}
If $M^4$ is a spherical suspension over an arbitrary Einstein manifold $(N^{3},\tilde{g})$ with $\Ric(\tilde{g})=(\Lambda-1)\tilde{g}$ then according to \cite[Theorem $2.1$]{Kro-Sine} for nonnegative indices $i$, $j$:
\begin{equation}
\begin{split}\label{ugly-lambda}
\lambda_{i,j}(g)&=\left(\sqrt{\frac{(\Lambda-1)^2}{4}+\lambda_i(\tilde{g})}-\frac{\Lambda-1}{2}+j\right)\left(\sqrt{\frac{(\Lambda-1)^2}{4}+\lambda_i(\tilde{g})}-\frac{\Lambda-1}{2}+j+\Lambda\right).
\end{split}
\end{equation}
In particular, $\lambda_{0,0}(g)=0<2\Lambda$ and corresponds to $\lambda_0(\tilde{g})=0$ whose eigenspace consists of constants.


  
Now, if $N^{3}$ is isometric to $\mathbb{S}^{3}/\Gamma$ endowed with a metric of constant curvature equal to $1$ then $\lambda_1(\tilde{g})\geq2(\Lambda+1)$. Combining this fact with the previous formula \eqref{ugly-lambda} leads to $\lambda_1(g)=\lambda_{0,1}(g)=\Lambda+1< 2\Lambda$ whose eigenspace is spanned by the function $\cos(\theta)$ where $g_o=d\theta^2+\sin(\theta)^2g_{\mathbb{S}^3/\Gamma}$ for $\theta\in(0,\pi)$. The property on $\nabla^{g_o}v_o$ can be deduced from a straightforward computation. Finally, observe that $\lambda_{1,0}(g)>2\Lambda.$
\end{proof} 
 
\begin{prop}\label{prop-liouville-orb}
Let $(M_o^4,g_o)$ be a metric suspension over $\mathbb{S}^3/\Gamma$ endowed with a metric of constant curvature $1$ with Einstein constant $\Lambda$ and let $(h(t))_{t< 0}$ be a solution to $\partial_th=\Delta_{L,g_o}h+2\Lambda h$ on $M_o\times(-\infty,0]$ such that on $M_o\times(-\infty,0]$, $|h(t)|\leq Cr_o^{-\sigma}$, for some uniform-in-time positive constants $C$ and $\sigma\in(0,1)$. 

If $\int_M\langle h(t),g_o\rangle_{g_o}\,d\mu_{g_o}=0$ and $\int_M\langle h(t) ,\Li_{\nabla^{g_o} v_o}(g_o)\rangle_{g_o}\,d\mu_{g_o}=0$ then $h(t)= 0$ on $M_o\times(-\infty,0)$.
\end{prop}
\begin{rk}
Lemma \ref{lemma-spec-susp} shows that the second orthogonality condition holding on $h(t)$ is equivalent in this setting to $\int_M\tr_{g_o} h(t) \cdot v_o\,d\mu_{g_o}=0$.
\end{rk}
\begin{rk}
Unlike Remark \ref{rk-pairing-ALE}, the orthogonality is understood in the usual $L^2$ sense. Indeed, since $h(t)$ blows up at most like $r_o^{-\sigma}$ with $\sigma\in(0,2)$, $h(t)$ lies in $L^2$ and so are $g_o$ and $\Li_{\nabla^{g_o} v_o}(g_o)$ since they are bounded on $M_o$.
\end{rk}
Before we prove Theorem \ref{prop-liouville-orb}, we state an algebraic lemma:
\begin{lemma}\label{lemma-alg-lichnero}
Under the assumptions of Theorem \ref{prop-liouville-orb},
\begin{equation*}
\partial_th=\Delta_{g_o}h+2\Rm(g_o)\ast h \quad \Longleftrightarrow\quad \left\{
\begin{aligned}
\partial_t\overset{\circ}{h}&=\Delta_{g_o}\overset{\circ}{h}-2 \overset{\circ}{h},\\
\partial_t\tr_{g_o}h&=\Delta_{g_o}\tr_{g_o}h+2\Lambda\tr_{g_o}h.
\end{aligned}
\right.
\end{equation*}
\end{lemma}

\begin{proof}[Proof of Lemma \ref{lemma-alg-lichnero}]
The decoupling is simply based on the fact that $\Rm(g_o)\ast h=(\tr_{g_o}h) g_o-h$ since the curvature operator of $g_o$ satisfies in coordinates $\Rm(g_o)_{ijkl}=(g_o)_{il}(g_o)_{jk}-(g_o)_{ik}(g_o)_{jl}.$
\end{proof}

\begin{proof}[Proof of Proposition \ref{prop-liouville-orb}]
The assumption on the exponent $\sigma$ ensures the local (hence the global) integrability of such tensor $h(t)$ on $M_o$ for each $t<0$. Moreover, by parabolic regularity, it guarantees that $\nabla^{g_o,i}h$ for $i=0,1,2$ belong to $L^2$.

Since $\partial_t\overset{\circ}{h}=\Delta_{g_o}\overset{\circ}{h}-2 \overset{\circ}{h}$ by Lemma \ref{lemma-alg-lichnero},
multiplying across by $\overset{\circ}{h}$ and integrating by parts give us:
\begin{equation*}
\frac{\partial}{\partial t}\|\overset{\circ}{h}\|_{L^2}^2+2\|\nabla^{g_o}\overset{\circ}{h}(t)\|_{L^2}^2=-4\|\overset{\circ}{h}(t)\|^2_{L^2},
\end{equation*}
where the integration by parts is legitimated by the aforementioned remark on the integrability of the covariant derivatives of $h(t)$.
In particular, the function $e^{4 t}\|\overset{\circ}{h}(t)\|^2_{L^2}$ is decreasing in time. By assumption, there exists a constant $C$ such that $\|\overset{\circ}{h}(t)\|^2_{L^2}\leq C$ for all $t< 0$. Therefore, for $s\leq t< 0$, $\|\overset{\circ}{h}(t)\|^2_{L^2}\leq e^{-4 t}\cdot e^{4 s}\|\overset{\circ}{h}(s)\|^2_{L^2}\leq Ce^{-4 t}e^{4 s}$. Letting $s$ go to $-\infty$ in the previous estimate gives $\overset{\circ}{h}(t)=0$ for all $t< 0$.\\

 Therefore, it suffices to show that $\tr_{g_o}h(t)=0$ for all $t< 0$ to end the proof. \\

Recall from Lemma \ref{lemma-alg-lichnero} that $\partial_t\tr_{g_o}h=\Delta_{g_o}\tr_{g_o}h+2\Lambda\tr_{g_o}h$. By the boundedness assumption on $h(t)$ for $t< 0$, $\tr_{g_o}h(t)$ belongs to the finite dimensional eigenspaces associated to eigenvalues with respect to $-\Delta_{g_o}$ less than or equal to $2\Lambda$. By assumption, $h(t)$ is orthogonal to $g_0$ and $\Li_{\nabla^{g_o} v_o}(g_o)$, a fact that gives us the expected result thanks to Lemma \ref{lemma-spec-susp}, i.e. $h(t)=0$ on $(-\infty,0)\times M_o$ since pointwise, $\langle h(t),g_o\rangle_{g_o}=\tr_{g_o}h(t)$ and $\langle h(t) ,\Li_{\nabla^{g_o} v_o}(g_o)\rangle_{g_o}=-2v_o\tr_{g_o}h(t)$.
\end{proof}

\subsubsection{For immortal flows}

Let us now consider an Einstein $4$-orbifold $(M_o,g_o)$ satisfying $\Ric(g_o) = \Lambda g_o$ with $\Lambda<0$, as well as $\mathbf{R}^+_{g_o} < 0$ pointwise everywhere on $M_o$ as a bilinear form on $\Lambda^+_{g_o}$.

\begin{exmp}
	Hyperbolic and complex hyperbolic (in the orientation opposite to that of the complex structure) orbifold are obvious examples of manifolds satisfying $\mathbf{R}^+<0$, this is also satisfied by the Einstein manifolds constructed by Fine-Premoselli \cite{FinePremoselli2020}.
\end{exmp}

\begin{prop}\label{Liouville-negative-einstein-initial}
   Let $(h(t))_{t\in [0,T]}$ be a solution to the homogeneous heat equation $\partial_th=(\Delta_{L,g_o}+2\Lambda)h$ on $M_o\times[0,T]$ with initial condition $h(0)=0$ such that on $M_o\times[0,T]$, $|h(t)|_{g_o}\leq C r_o^{-\sigma}$ for some uniform-in-time positive constants $C>0$ and $\sigma\in(0,1)$.

    If $h(0)\equiv 0$ on $\mathbb{R}^4\backslash\{0\}$, then $h(t)\equiv 0$ for $t\in[0,T]$.
\end{prop}
\begin{proof}
Proposition \ref{Liouville-negative-einstein-initial} can be proved along the same lines as Proposition \ref{liouville-ALE-initial} by invoking Theorem \ref{theo-karp-li}. The only subtlety here is that Theorem \ref{theo-karp-li} only applies to complete metrics and smooth subsolutions to the heat equation. A straightforward adaptation of its proof shows that it still holds for subsolutions to the heat equation on $M_o$ satisfying the growth $|h(t)|_{g_o}\leq C r_o^{-\sigma}$ with $\sigma<\frac{n-2}{2}=1$ here. Note that we do not need the assumption $\mathbf{R}^+<0$ here.

\end{proof}

We also have a similar statement for ancient flow whose proof is a little more complicated relying on the following formalism. Recall the canonical identification between $\Lambda^-\otimes \Lambda^+$ and $ \operatorname{Sym}^2_0(M_o) $ through $ \alpha^-\otimes\beta^+\mapsto \alpha^-\circ\beta^- $ where $(\alpha^-\circ\beta^-)_{ij} = {(\alpha^-)_{i}}^k(\beta^+)_{kj}$ is the usual composition from \eqref{eq:identification traceless}. Denote $d_-:\Lambda^1\otimes \Lambda^+\to \Lambda^-\otimes \Lambda^+$ the composition of the covariant exterior derivative on $\Lambda^1\otimes \Lambda^+$ and the projection $\Lambda^2\otimes\Lambda^+\to\Lambda^-\otimes\Lambda^+$ on the anti-selfdual part. The action of the selfdual curvature $\mathbf{R}^+: \Lambda^+\to\Lambda^+$ extends to an action on $\operatorname{Sym}^2_0(M_o)$ through $\mathbf{R}^+(\alpha^-\circ \beta^+) = \alpha^-\circ \mathbf{R}^+(\beta^+)$.

\begin{prop}\label{Liouville-negative-einstein-ancien}
    Let $(h(t))_{t\in (-\infty,0]}$ be a solution to the homogeneous heat equation $\partial_th=(\Delta_{L,g_o}+2\Lambda)h$ on $M_o\times(-\infty,0]$ such that on $M_o\times(-\infty,0]$, $|h(t)|_{g_o}\leq C r_o^{-\sigma}$ for some uniform-in-time positive constants $C>0$ and $\sigma\in(0,1)$.

    Then $h(t)\equiv 0$ for $t\in(-\infty,0]$.
\end{prop}
\begin{proof}
    Since the pure trace part is easily dealt with (since $\Lambda<0$), and $\Delta_{L,g_o}+2\Lambda$ preserves the pure trace and traceless decomposition, we will focus on the case where $h$ is traceless.
    
    The key point here is that a Weitzenböck formula from \cite[4.21]{Biq-Rol-ACH}, (see also \cite[(4)]{BiquardOzuch2023} and \cite{no1}) states that on a traceless symmetric $2$-tensor $h \in \operatorname{Sym}^2_0(M_o)$, one has 
    \begin{equation}\label{eq-stab-orb-immortal}
        (\Delta_{L,g_o}+2\Lambda)h = -2(d_-d_-^*-\mathbf{R}^+)h, 
    \end{equation}
    hence for any $2$-tensor $h$ on $M_o$. Now, if $\sigma\in(0,1)$, then one has:
    \begin{equation}\label{eq-decroissance-chaleur-orbifold}
        \frac{1}{2} \frac{d}{dt} \int_{M_o} |h(t)|^2_{g_o} \,d\mu_{g_o} = \int_{M_o} \langle (\Delta_{L,g_o}+2\Lambda)h,\,h \rangle_{g_o}\,d\mu_{g_o}\leq 2\lambda \int_{M_o} |h(t)|^2_{g_o} \,d\mu_{g_o} \leq 0,
    \end{equation}
    for $\lambda<0$ the least negative eigenvalue of $\mathbf{R}^+$ where the integration by parts is allowed as long as $|\nabla^{g_o,\,l}h(t)|_{g_o}\leq C r_o^{-\sigma-l}$ for $\sigma<1$. 

    The quantity $A(t):=\int_{M_o} |h(t)|^2_{g_o} \,d\mu_{g_o}$ therefore satisfies $A'(t) \leq 4\lambda A(t)$. Note that in the case of an initial condition $h(0)=0$, this proves the result since $A(0) = 0$ and $A(t)\geq 0$. For an ancient flow $(h(t))_{t\in(-\infty,0]}$ it implies that $A(t_0)e^{4\lambda t_0}\leq A(t)e^{4\lambda t}\leq Ce^{4\lambda t}$ for $t\leq t_0$ by the uniform assumption $|h(t)|_{g_o}\leq C r_o^{-\sigma}$. This implies that $A(t_0) = 0$, i.e. $h(t_0)=0$ for all $t_0\leq 0$ by sending $t$ to $-\infty$.
\end{proof}





\subsection{On Euclidean space}

We state a Liouville-type result for the homogenous heat equation on Euclidean space proved in \cite[Proposition $5.3$]{Bre-Kap}:
\begin{prop}\label{prop-liouville-ancien-euc}
Let $(h(t))_{t\,<\, 0}$ be a solution to $\partial_th=\Delta_{\mathbf{e}}h$ on $\RR^4\setminus\{0\}\times(-\infty,0)$ such that on $\RR^4\setminus\{0\}\times(-\infty,0)$, $|h(x,t)|\leq C|x|^{-\sigma}$, for some uniform-in-time positive constants $C$ and $\sigma\in(0,2)$. 

 Then $h(t)= 0$ on $\RR^4\setminus\{0\}\times(-\infty,0)$.
\end{prop}

The Liouville-type result for the homogenous heat equation on Euclidean space proved in \cite[Proposition $5.3$]{Bre-Kap} by interpolation with the heat kernel extends directly to the simpler case of initial condition zero and yields the following statement.
\begin{prop}\label{prop-liouville-immortel-euc}
Let $(h(t))_{t\in[0,T] }$ be a solution to $\partial_th=\Delta_{\mathbf{e}}h$ on $\RR^4\setminus\{0\}\times[0,T] $ such that on $\RR^4\setminus\{0\}\times[0,T]$, $|h(x,t)|\leq C|x|^{-\sigma}$, for some uniform-in-time positive constants $C$ and $\sigma\in(0,2)$. 

 If $h(0)\equiv 0$, then $h(t)= 0$ on $\RR^4\setminus\{0\}\times[0,T] $.
\end{prop}

\section{Linear theory: the unconstrained case}\label{sec:fct spaces}

\subsection{Definition and first properties}\label{subsec-holder}

We now define the weighted Hölder spaces which will be used along the article. We start by defining their $C^0$ counterpart.
\begin{itemize}
\item For $T<0$, a smooth function $\gamma:(-\infty,T]\rightarrow\RR_+$, and $\sigma>0$,
\begin{equation*}
\|h\|_{C^{0}_{\gamma,\sigma,T}}:=\sup\left\{\gamma(t)\rho(x,t)^{\sigma}|h(x,t)|_{\tilde{g}(t)}\quad|\quad t\in(-\infty,T],\,x\in M\right\}.
\end{equation*}

The use of $\varepsilon(t)$ or $\varepsilon_0(t)$ leads to uniform equivalent norms thanks to $\eqref{cond-param}$.

\item Let us now define our weighted Hölder norms.\\

For a tensor $(h(t))_{t\,\leq\, T}$, define the following semi-norm:
$$[h]_{\alpha,x,t}:=\rho(x,t)^{2\alpha}\sup_{(y,t')\in P_{\rho}(x,t)\setminus\{(x,t)\}}\frac{|(P_{x,y}^t)^*h(y,t')-h(x,t)|_{\tilde{g}(t)}}{(d_{\tilde{g}(t)}(y,x)^2+|t'-t|)^{\alpha}},$$ where $P_{\rho}(x,t):=B_{\tilde{g}(t)}(x,\rho_0(x,t))\times(t-\rho_0(x,t)^2,t]$ and $\rho_0(x,t):=\iota_0\rho(x,t)=\iota_0(\varepsilon(t)+r_o(x))$
with $\iota_0>0$ a positive constant such that $1/4\inj_x(\tilde{g}(t))\leq \iota_0(\varepsilon(t)+r_o(x))\leq 1/2\inj_x(\tilde{g}(t))$ for all $(x,t)\in M\times(-\infty,T]$. Here $P_{x,y}^t$ denotes parallel transport along the unique minimizing geodesic with respect to $\tilde{g}(t)$ joining $x$ to $y$. For ease of reading, we will avoid the dependency of $P_{\rho}(x,t)$ on $(x,t)$ and will simply denote such a parabolic neighborhood by $P_{\rho}$ whenever there is no source of confusion.\\

\item This allows us to define the space $C^{2,\alpha}_{\gamma,\sigma,T}$:
\begin{equation*}
\begin{split}
\|h\|_{C^{0,\alpha}_{\gamma,\sigma,T}}:=&\|h\|_{C^{0}_{\gamma,\sigma,T}} + \sup_{(x,t)\in M\times(-\infty,T]}\gamma(t)(\varepsilon(t)+r_o(x))^{\sigma}[h]_{\alpha,x,t}.\\
\|h\|_{C^{2,\alpha}_{\gamma,\sigma,T}}:=&\sum_{0\,\leq\, j\,\leq\, 2}\|\nabla^{\tilde{g},\,j}h\|_{C^{0,\alpha}_{\gamma,j+\sigma,T}} + \|\partial_th\|_{C^{0,\alpha}_{\gamma,\sigma+2,T}}.
\end{split}
\end{equation*}

\end{itemize}

 By a slight abuse of notation, we will use the same notation for norms on immortal flows obtained with the same definition, up to flipping the sign $t\mapsto -t$, $T\mapsto -T$. This will make some statements more streamlined as they will be valid for both ancient and immortal flows. When it matters, it will be clear from the context which of the spaces is used from the sign of $t$ and $T$. 
 
More precisely, 
\begin{itemize}
\item the first norm extends as follows to tensors defined on $[T,+\infty)\times M$.

For $T>0$, a smooth function $\gamma:[T,+\infty)\rightarrow\RR_+$, and $\sigma>0$,
\begin{equation*}
\|h\|_{C^{0}_{\gamma,\sigma,T}}:=\sup\left\{\gamma(t)\rho(x,t)^{\sigma}|h(x,t)|_{\tilde{g}(t)}\quad|\quad t\in[T,+\infty),\,x\in M\right\}.
\end{equation*}

\item For a tensor $(h(t))_{t\,\geq\, T}$, define the following semi-norm:
$$[h]_{\alpha,x,t}:=\rho(x,t)^{2\alpha}\sup_{(y,t')\in P_{\rho}(x,t)\setminus\{(x,t)\}}\frac{|(P_{x,y}^t)^*h(y,t')-h(x,t)|_{\tilde{g}(t)}}{(d_{\tilde{g}(t)}(y,x)^2+|t'-t|)^{\alpha}},$$ where $P_{\rho}(x,t):=B_{\tilde{g}(t)}(x,\rho_0(x,t))\times(\max\{0,t-\rho_0(x,t)^2\},t]$ and $\rho_0(x,t):=\iota_0\rho(x,t)=\iota_0(\varepsilon(t)+r_o(x))$ with $\iota_0>0$ a positive constant such that $1/4\inj_x(\tilde{g}(t))\leq \iota_0(\varepsilon(t)+r_o(x))\leq 1/2\inj_x(\tilde{g}(t))$ for all $(x,t)\in M\times[T,+\infty)$. Here $P_{x,y}^t$ denotes parallel transport along the unique minimizing geodesic with respect to $\tilde{g}(t)$ joining $x$ to $y$.
\end{itemize}

\textbf{From now on, we assume $\gamma: (-\infty,T]\rightarrow \RR_+$ (respectively $\gamma: [T,+\infty)\rightarrow \RR_+$)
\begin{equation}\label{cond-param-gamma}
  \tag{$\mathcal{H}_{\gamma}$}
  \begin{split}
\gamma(t):=\varepsilon_0(t)^{-\gamma_0},\quad t\leq T, \quad{\text{respectively $t\geq T$},} 
\end{split}
\end{equation}
for some $\gamma_0>0$ which will be chosen in Section \ref{section-non-hom-term}.}

\subsection{Weighted Schauder estimates}


We now prove Schauder estimates in the above weighted Hölder spaces.

\begin{prop}\label{Schauder nos normes}
    Let $\alpha\in(0,1/2)$, $\sigma\in(0,2)$ and $\gamma$ satisfying \eqref{cond-param-gamma}. Then there exists $C>0$ such that if $h\in C^{2,\alpha}_{\gamma,\sigma,T}$ and $\psi\in C^{0,\alpha}_{\gamma,\sigma,T}$ are $2$-tensors on $(-\infty,T]\times M$ satisfying
\begin{equation}\label{primitive-par-eqn}
 \partial_th - \Delta_{L,\tilde{g}(t)}h-2\Lambda h = \psi, 
 \end{equation}
then, we have the following control
    \begin{equation*}
     \|h\|_{C^{2,\alpha}_{\gamma,\sigma,T}} \leq C \Big(\|\psi\|_{C^{0,\alpha}_{\gamma,\sigma+2,T}}+\|h\|_{C^{0}_{\gamma,\sigma,T}}\Big). 
     \end{equation*}
\end{prop}

\begin{rk}
The constraint on the weight $\gamma$ is consistent with the setting of the present paper.  An inspection of the proof of Proposition \ref{Schauder nos normes} shows that it suffices to assume that there is $C>0$ such that $|\gamma'(t)|\leq C\gamma(t)$ for $t\leq T$.
\end{rk}
\begin{proof}
 We prove the result in the ancient case first.

    Fix $(x_0,t_0)\in M\times(-\infty,T]$ and define $\rho_0:=\iota_0(\varepsilon(t_0)+r_o(x_0))$. Consider the rescaled tensor for $\sigma>0$ and $(y,s)\in B_{\tilde{g}(t_0)}(x_0,\rho_0)\times(-1,0]$:
    \begin{equation*}
    \tilde{g}_0(y,s):=\rho_0^{-2}\tilde{g}(y,\rho_0^2s+t_0),\quad h_0(y,s):=\rho_0^{\sigma-2}h(y,\rho_0^2s+t_0),\quad \psi_0(y,s):=\rho_0^{\sigma}\psi(y,\rho_0^2s+t_0).
    \end{equation*}
    Thanks to \eqref{primitive-par-eqn} and the scaling properties of the curvature tensor, observe that the above triple of tensors satisfies on $B_{\tilde{g}_0(0)}(x_0,1)\times(-1,0]$:
    \begin{equation}\label{primitive-resc-para-eqn}
    \partial_sh_0-\Delta_{\tilde{g}_0(s)\,,\,L}h_0-2\Lambda \rho_0^2h_0=\psi_0.
    \end{equation}
    Now, the curvature estimates of Lemma \ref{lemma-first-app-est} and the fact that $\inj_{x_0}(\tilde{g}_0(0))\geq 2>0$ ensures that the metric $\tilde{g}_0(0)$ is uniformly equivalent to the Euclidean metric through the exponential map $\exp_{x_0}^{  \tilde{g}_0(0)}.$ The same holds for higher (Euclidean) derivatives of the metric $\tilde{g}_0(0)$ in this chart. Moreover, the estimates on the covariant derivatives of the time derivative $\partial_t\tilde{g}(t)$ from Lemma \ref{lemma-time-der-metric} shows that the coefficients of the parabolic equation \eqref{primitive-resc-para-eqn} are uniformly $C^1$ in time. These regularity properties allow us to apply standard parabolic Schauder interior estimates that can be found in \cite[Section $8.11$]{Kry-Boo}: for $\alpha\in(0,1/2)$, there exists a positive constant $C(n,\alpha)>0$ such that 
    \begin{equation}\label{loc-std-schauder}
    \|h_0\|_{C^{2,\alpha}(B_{\mathbf{e}}(0,1/2)\times(-1/2,0])}\leq C \left(\|h_0\|_{C^{0}(B_{\mathbf{e}}(0,1)\times(-1,0])}+\|\psi_0\|_{C^{0,\alpha}(B_{\mathbf{e}}(0,1)\times(-1,0])}\right),
    \end{equation}
    where $B_{\mathbf{e}}(0,r)$ denotes the Euclidean ball centered at the origin of radius $r>0.$ Here, the norms are the standard H\"older parabolic norms as defined in  \cite[Section $8.5$]{Kry-Boo}. Unravelling the definitions of the rescaled tensors $\tilde{g}_0(s)$, $h_0(s)$ and $\psi_0(s)$ gives the expected Schauder estimate once \eqref{loc-std-schauder} gets multiplied by $\gamma(t_0)$ and once the supremum is taken on $(x_0,t_0)\in M\times(-\infty,T]$. This ends the proof in the ancient case.
 
The proof of the corresponding estimate is totally analogous in the immortal case: one has to invoke \cite[Section $9.2$]{Kry-Boo} instead to get the standard Schauder estimate for parabolic equations with the desired regularity on the coefficients.
        \end{proof}

\subsection{Linear controls of solutions orthogonal to the approximate kernel in the ancient case}

The main result of this section is the following a priori $C^0_{\gamma,\sigma,T}$- estimate on any solution to a non-homogeneous heat equation with data lying in $C^{2,\alpha}_{\gamma,\sigma+2,T}$.

\begin{prop}[Linear estimates in weighted Hölder spaces]\label{6.2 chez nous}
  Under the assumption and notations of \eqref{cond-param}, let $(\tilde{g}(t))_{t\,\leq \,T}$ be the associated curve of metrics defined with respect to $\zeta(t)\in\RR^3\setminus\{0\}$ and $\delta(t)\in\RR_+$. Let $\alpha\in(0,1/2)$, $\sigma\in(0,1)$ and $\gamma$ as in \eqref{cond-param-gamma}. Then there exists $T<0$ and $C>0$ such that the following holds. Assume that
\begin{enumerate}
    \item $h\in C^{2,\alpha}_{\gamma,\sigma,T}$ and $\psi\in C^{0,\alpha}_{\gamma,\sigma+2,T}$ satisfy on $M\times(-\infty,T]$
    \begin{equation*}
     (\partial_t- \Delta_{L,\tilde{g}(t)}-2\Lambda) h(t) =\psi(t),
     \end{equation*}
    \item \label{ortho-cond} $h(t)$ is $L^2(\tilde{g}(t))$-orthogonal to the approximate kernel $\tilde{\mathbf{O}}(t)$ introduced in Definition \ref{defn-app-kernel} for all $t\in(-\infty,T]$.
\end{enumerate}
Then one has the a priori estimate:
\begin{equation}\label{linear control heat}
    \|h\|_{C^0_{\gamma,\sigma,T}}\leq C\|\psi\|_{C^{0,\alpha}_{\gamma,\sigma+2,T}}.
\end{equation}
\end{prop}
\begin{rk}\label{rk: zeta close to initial direction}
    An important additional assumption when dealing with other Kronheimer's instantons is that we will want $\zeta(t)/|\zeta(t)|$ to stay close to ${\zeta_0}/|{\zeta_0}|$ to prevent the formation of further orbifold singularities.
    
    Indeed, if the parameter $\zeta(t)$ is uncontrolled, new orbifold singularities may form and make the analysis more difficult. Luckily, we will see in Lemma \ref{lemma-control-zeta-bis} from Section \ref{sec-nonlin} that we can always choose gluing directions in which the assumptions of this proposition are satisfied.
\end{rk}
\begin{proof}
    Let us argue by contradiction and assume that for all $i\in\mathbb{N}$, there is a sequence $T_i\to-\infty$ so that we have:
    \begin{itemize}
    \item $h_i\in C^{2,\alpha}_{\gamma,\sigma,T_i}$ and $\psi_i\in C^{0,\alpha}_{\gamma,\sigma+2,T_i}$ satisfy on $M\times(-\infty,T_i]$,
    \begin{equation*}
     (\partial_t- \Delta_{L,\tilde{g}(t)}-2\Lambda) h_i(t) =\psi_i(t),
     \end{equation*}    \item $h_i(t)$ is $L^2(\tilde{g}(t))$-orthogonal to the approximate kernel $\tilde{\mathbf{O}}(t)$ for all $t\in (-\infty,T_i]$, \\
    \item $\|h_i\|_{C^0_{\gamma,\sigma,T_i}}=1$ and $\lim_{i\rightarrow+\infty}\|\psi_i\|_{C^{0,\alpha}_{\gamma,\sigma+2,T_i}}=0$.
\end{itemize}

By that last point, we can pick $(x_i,t_i) \in M \times (-\infty,T_i]$ such that 
\begin{equation}
    \gamma(t_i)(\varepsilon_0(t_i)+r_o(x_i))^\sigma |h_i(x_i,t_i)|_{\tilde{g}(t_i)}\geq \frac{1}{2}.
\end{equation}
There are three situations to consider depending on where a subsequence of $(x_i,t_i)$ ends up compared to our different regions of $(M,\tilde{g}(t))_{t\,\leq\, T_i}$: 
\begin{enumerate}
\item $\lim_{i\rightarrow +\infty}(\varepsilon_0(t_i))^{-1}r_o(x_i)<+\infty$,\\
\item $\lim_{i\rightarrow +\infty}(\varepsilon_0(t_i))^{-1}r_o(x_i)=+\infty$ and $\lim_{i\rightarrow +\infty} r_o(x_i)=0$,\\
\item  or $\lim_{i\rightarrow +\infty}(\varepsilon_0(t_i))^{-1}r_o(x_i)=+\infty$ and $\lim_{i\rightarrow +\infty} r_o(x_i)>0$.
\\
\end{enumerate}

\textit{First case: $\lim_{i\rightarrow +\infty}\varepsilon(t_i)^{-1}r_o(x_i)  <\infty$.} 
In that case, we define the following parabolic rescalings for $t\in \left(-\infty, \frac{T_i-t_i}{\varepsilon(t_i)^2}\right]$ by recalling that $\varepsilon(t)^2:=|\zeta(t)|$,
\begin{itemize}
    \item $ \tilde{x}_i := s_{\zeta(t_i)}^{-1}(x_i)$,
    \item $ \tilde{g}_i(t) := |\varepsilon(t_i)|^{-2}s_{\zeta(t_i)}^*\tilde{g}(t_i+ \varepsilon(t_i)^2t) $,
    \item $\tilde{h}_i(t) := \gamma(t_i)\varepsilon(t_i)^{\sigma-2}s_{\zeta(t_i)}^*h_i(t_i+ \varepsilon(t_i)^2t)$, and
    \item $\tilde{\psi}_i(t) := \gamma(t_i)\varepsilon(t_i)^{\sigma}s_{\zeta(t_i)}^*\psi_i(t_i+ \varepsilon(t_i)^2t)$.
\end{itemize}
These satisfy: $(\partial_t- \Delta_{L,\tilde{g}_i(t)}-2\Lambda \varepsilon(t_i)^{2} ) \tilde{h}_i(t) =\tilde{\psi}_i(t),$ as well as the suitably rescaled orthogonality conditions.
\\

 After taking a limit as $i\to +\infty$, up to the action of a diffeomorphism, $\tilde{g}_i(t) \to g_b$
 in (space-time) $C^{2,\alpha}_{loc}$ by Lemma \ref{lemma-first-app-est}. Now, the tensors $\tilde{\psi}_i$ are uniformly bounded in $C^{0,\alpha}_{loc}$, so that by parabolic regularity, the solutions $\tilde{h}_i$ are uniformly bounded in (space-time) $C^{2,\alpha}_{loc}$ for any $0<\alpha<1/2$.  By compactness, we therefore find a time-dependent $2$-tensor $\hat{h}$ in $C^{2,\alpha}_{loc}$ on $N\times \left(-\infty, 0\right]\subset N\times \left(-\infty,\lim_i\frac{T_i-t_i}{\varepsilon(t_i)^2}\right]$ satisfying 
\begin{equation*}
(1 +r_b(\hat{x}))^\sigma |\hat{h}(\hat{x},0)|_{g_b}\geq\frac{1}{2},
\end{equation*}
and for all $(x,t)$ in $N\times \left(-\infty, 0\right]$, 
\begin{equation}
(1 +r_b(x))^\sigma |\hat{h}(x,t)|_{g_b}\leq 1.\label{unif-decay-h-hat}
\end{equation}
Indeed, since $r_o=\varepsilon(t)(s_{\zeta(t)})_{\ast}r_b$ and $\gamma(t)(\varepsilon_0(t)+r_o(x))^\sigma |h_i(x,t)|_{\tilde{g}(t)}\leq 1$, we therefore have the following control over the rescaled $\tilde{h}_i$ thanks to the definition of $\tilde{g}_i(t)$: for $t\leq 0$,
    \begin{align*}
        \left|\tilde{h}_i(x,t)\right|_{\tilde{g}_i(t)} &= \gamma(t_i) \varepsilon(t_i)^{\sigma-2}\left|s_{\zeta(t_i)}^*h_i(x,t_i+  \varepsilon(t_i)^2t)\right|_{\tilde{g}_i(t)}\\
        &= \gamma(t_i) \varepsilon(t_i)^{\sigma-2}s_{\zeta(t_i)}^*\left|h_i(x,t_i+  \varepsilon(t_i)^2t)\right|_{\varepsilon(t_i)^{-2}\tilde{g}(t_i+  \varepsilon(t_i)^2t)}\\
        &=\gamma(t_i) \varepsilon(t_i)^{\sigma}s_{\zeta(t_i)}^*\left|h_i(x,t_i+  \varepsilon(t_i)^2t)\right|_{\tilde{g}(t_i+  \varepsilon(t_i)^2t)}\\
        &\leq \frac{\gamma(t_i) }{\gamma(t_i+  \varepsilon(t_i)^2t)(1+r_b(x))^\sigma}\leq \frac{1 }{(1+r_b(x))^\sigma}.
    \end{align*}
Here we have used \eqref{cond-param-gamma} in the last inequality.

Moreover, it satisfies on $N\times \left(-\infty, 0\right]$,
\begin{equation}
(\partial_t- \Delta_{L,g_b}) \hat{h}(t) =0.\label{h-hat-lic-eqn}
\end{equation}

By the very construction of $\tilde{\mathbf{O}}(t)$, any rescaled element from the $(\tilde{\oi}_j(t))_j$ decays like $(1+r_b)^{-4}$ uniformly along the sequence with respect to the rescaled metrics $\tilde{g}_i(t)$ while $\tilde{h}_i = O((1+r_b)^{-\sigma})$ uniformly along the sequence as checked in \eqref{unif-decay-h-hat}. In particular, this implies that the orthogonality conditions (\ref{ortho-cond}) are preserved at the limit, i.e. the Lebesgue dominated convergence theorem implies that for all $j\in \{1,2,3\}$:
\begin{equation*}
\int_{N}\langle \hat{h}(t),\oi_{j}\rangle_{g_b}\,d\mu_{g_b}=0,\quad t\leq 0.
\end{equation*}
This means that $\hat{h}$ is $L^2(g_b)$-orthogonal to the $L^2$-kernel of $\Delta_{L,g_b}$.  This contradicts our Liouville Theorem, Proposition \ref{liouville-ALE}.
\\


\textit{Second case: $\lim_{i\rightarrow +\infty}\varepsilon_0(t_i)^{-1}r_o(x_i)=\infty$ and $\lim_{i\rightarrow +\infty} r_o(x_i)=0$.}
In that case, we define the following parabolic rescalings: for $t\in \left(-\infty, \frac{T_i-t_i}{r_o(x_i)^2}\right]$, {denote $\phi_s:\mathbb{R}^4/\Gamma\mapsto\mathbb{R}^4/\Gamma$ the diffeomorphism sending $(r,x)$ to $(sr,x)$,
\begin{itemize}
    \item $ \tilde{x}_i := \phi_{r_o(x_i)}^{-1}(x_i)$,
    \item $ \tilde{g}_i(t) := r_o(x_i)^{-2}\phi_{r_o(x_i)}^*\tilde{g}(t_i+ r_o(x_i)^2t) $,
    \item $\tilde{h}_i(t) := \gamma(t_i)r_o(x_i)^{\sigma-2}\phi_{r_o(x_i)}^*h_i(t_i+ r_o(x_i)^2t)$, and
    \item $\tilde{\psi}_i(t) := \gamma(t_i)r_o(x_i)^{\sigma}\phi_{r_o(x_i)}^*\psi_i(t_i+ r_o(x_i)^2t)$.
\end{itemize}}
These satisfy $(\partial_t- \Delta_{L,\tilde{g}_i(t)}-2\Lambda r_o(x_i)^2) \tilde{h}_i(t) =\tilde{\psi}_i(t).$ 

In this case, after taking a limit as $i\to +\infty$, $\tilde{g}_i(t) \to \mathbf{e}$ smoothly locally and we find a time-dependent $2$-tensor $\hat{h}$ on $\mathbb{R}^4\slash\mathbb{Z}_2\times \left(-\infty, \lim_i\frac{T_i-t_i}{\varepsilon(t_i)^2}\right]$ satisfying 
\begin{equation*}
r^\sigma |\hat{h}(\hat{x},0)|_{\mathbf{e}}\geq\frac{1}{2},
\end{equation*}
and for all $(x,t)$ in $\mathbb{R}^4\slash\mathbb{Z}_2\times \left(-\infty, \lim_i\frac{T_i-t_i}{\varepsilon(t_i)^2}\right]$, 
$$r^\sigma |\hat{h}(x,t)|_{\mathbf{e}}\leq 1.$$

This contradicts the Liouville Theorem from Proposition \ref{prop-liouville-ancien-euc} because $\sigma\in(0,2)$.\\
\\
\textit{Third case: $\lim_{i\rightarrow +\infty} r_o(x_i)>0$.}
In that case,  we do not need to rescale. Indeed, we only need to translate time by $t_i$, i.e. we define $\tilde{g}_i(t):=\tilde{g}(t_i+t)$ and similarly for the solution $h$ and the tensor $\psi$.

As $i\to +\infty$, $\tilde{g}_i(t) \to g_o$ for all $t\in(-\infty,0]$ smoothly locally on $M_o$ outside the orbifold points and we find a time-dependent $2$-tensor $\hat{h}$ on $M_o\times \left(-\infty, 0\right]$ satisfying 
$$r^\sigma |\hat{h}(\hat{x},0)|_{g_o}\geq\frac{1}{2},$$ 
and for all $(x,t)$ in $M_o\times \left(-\infty, 0\right]$, 
\begin{equation}\label{dec-h-hat-o}
r^\sigma |\hat{h}(x,t)|_{g_o}\leq 1.
\end{equation}
As in the first case, an application of the dominated convergence theorem shows that the orthogonality conditions are satisfied. Indeed, the elements of the $L^2$-kernel of the orbifold 
like $\tilde{c}$ are bounded close to the singularity and $\hat{h}$ satisfies \eqref{dec-h-hat-o} with $ \sigma<1<4 $.
This contradicts the Liouville Theorem from Proposition \ref{prop-liouville-orb} and proves the proposition.
\end{proof}


\subsection{Linear control of solutions orthogonal to the kernel in the immortal case}

\begin{prop}[Linear estimates in weighted Hölder spaces]\label{6.2 chez nous immortel}
 Under the assumption and notations of \eqref{cond-param}, let $(\tilde{g}(t))_{t\geq T}$ be the associated curve of metrics defined with respect to $\zeta(t)\in\RR^3\setminus\{0\}$ and $\delta(t)\in\RR_+$. Let $\alpha\in(0,1/2)$, $\sigma\in(0,1)$ and $\gamma$ as in \eqref{cond-param-gamma}. Then there exists $T>0$ and $C>0$ such that the following holds. Assume that
\begin{enumerate}
    \item $h\in C^{2,\alpha}_{\gamma,\sigma,T}$ and $\psi\in C^{0,\alpha}_{\gamma,\sigma+2,T}$ satisfy on $M\times[T,+\infty)$
    \begin{equation*}
     (\partial_t- \Delta_{L,\tilde{g}(t)}-2\Lambda) h(t) =\psi(t), \quad h(T)\equiv 0,
     \end{equation*}
    \item \label{ortho-cond-imm} $h(t)$ is $L^2(\tilde{g}(t))$-orthogonal to the approximate kernel $\tilde{\mathbf{O}}(t)$ introduced in Definition \ref{defn-app-kernel} for all $t\in[T,+\infty)$.
\end{enumerate}
Then one has the a priori estimate:
\begin{equation}\label{linear control heat-imm}
    \|h\|_{C^0_{\gamma,\sigma,T}}\leq C\|\psi\|_{C^{0,\alpha}_{\gamma,\sigma+2,T}}.
\end{equation}
   \end{prop}
\begin{proof}
    Let us argue by contradiction and assume that for all $i\in\mathbb{N}$, there is a sequence $T_i\to+\infty$ so that we have:
    \begin{itemize}
    \item $h_i\in C^{2,\alpha}_{\gamma,\sigma,T_i}$ with $h_i(T_i)=0$ and $\psi_i\in C^{0,\alpha}_{\gamma,\sigma+2,T_i}$ satisfy on $M\times[T_i,+\infty)$,
    \begin{equation*}
     (\partial_t- \Delta_{L,\tilde{g}(t)}-2\Lambda) h_i(t) =\psi_i(t),
     \end{equation*}    \item $h_i(t)$ is $L^2(\tilde{g}(t))$-orthogonal to the approximate kernel $\tilde{\mathbf{O}}(t)$ for all $t\in [T_i,+\infty)$, \\
    \item $\|h_i\|_{C^0_{\gamma,\sigma,T_i}}=1$ and $\lim_{i\rightarrow+\infty}\|\psi_i\|_{C^{0,\alpha}_{\gamma,\sigma+2,T_i}}=0$.
\end{itemize}

By that last point, we can pick $(x_i,t_i) \in M \times [T_i,+\infty)$ such that 
\begin{equation}
    \gamma(t_i)(\varepsilon(t_i)+r_o(x_i))^\sigma |h_i(x_i,t_i)|_{\tilde{g}(t)}\geq \frac{1}{2}.
\end{equation}
As in Proposition \eqref{6.2 chez nous} which dealt with the ancient case, there are three situations to consider depending on where a subsequence of $(x_i,t_i)$ ends up compared to our different regions of $(M,g_i(t))$: 
\begin{enumerate}
\item $\lim_{i\rightarrow +\infty}(\varepsilon(t_i))^{-1}r_o(x_i)<+\infty$,\\
\item $\lim_{i\rightarrow +\infty}(\varepsilon(t_i))^{-1}r_o(x_i)=+\infty$ and $\lim_{i\rightarrow +\infty} r_o(x_i)=0$,\\
\item  or $\lim_{i\rightarrow +\infty}(\varepsilon(t_i))^{-1}r_o(x_i)=+\infty$ and $\lim_{i\rightarrow +\infty} r_o(x_i)>0$.
\\
\end{enumerate}

In each situation, we will take the same rescalings, reparametrizations and limits as in the proof of Proposition \eqref{6.2 chez nous}.

The difference with the ancient case is that we will also have further cases: 
\begin{itemize}
    \item either $t_i$ stays comparable to $T_i$ in the sense that $\frac{t_i-T_i}{\varepsilon(t_i)^2}\geq 0$ stays bounded, hence we still ``see'' the initial condition $h_i(T_i)=0$ and may use the Liouville Theorems from Propositions from Propositions \ref{liouville-ALE-initial}, \ref{Liouville-negative-einstein-initial}, \ref{prop-liouville-immortel-euc}, 
    \item or $t_i$ grows much faster than $T_i$, and we obtain an ancient (and even eternal) flow for which we may use the Liouville Theorems from Propositions \ref{liouville-ALE}, \ref{Liouville-negative-einstein-ancien}, \ref{prop-liouville-ancien-euc}.
\end{itemize}

\textit{First case: $\lim\varepsilon(t_i)^{-1}r_o(x_i)  <\infty$.} 
As in the proof of Proposition \ref{6.2 chez nous} that corresponds to the ancient case, after parabolic rescaling, we find a time-dependent $2$-tensor $\hat{h}$ in $C^{2,\alpha}_{loc}$ on $N\times \left(\lim_i\frac{T_i-t_i}{\varepsilon(t_i)^2}, +\infty\right)$ satisfying 
\begin{equation*}
(1 +r_b(\hat{x}))^\sigma |\hat{h}(\hat{x},0)|_{g_b}\geq\frac{1}{2},
\end{equation*}
 and for all $(x,t)$ in $N\times \left(\lim_i\frac{T_i-t_i}{\varepsilon(t_i)^2}, +\infty\right)$, 
\begin{equation}
(1 +r_b(x))^\sigma |\hat{h}(x,t)|_{g_b}\leq 1.\label{unif-decay-h-hat-imm}
\end{equation}
Moreover, it satisfies on $N\times \left(\lim_i\frac{T_i-t_i}{\varepsilon(t_i)^2},+\infty\right)$,
\begin{equation}
(\partial_t- \Delta_{L,g_b}) \hat{h}(t) =0.\label{h-hat-lic-eqn-imm}
\end{equation}

First, if $T_\infty:=\lim_i\frac{T_i-t_i}{\varepsilon(t_i)^2}>-\infty$, observe that by Lemma \ref{lemma-time-der-metric}, for $t\geq \frac{T_i-t_i}{\varepsilon(t_i)^2}$,
\begin{equation}\label{der-time-est-resc-met}
|\partial_t\tilde{g}_i(t)|_{\tilde{g}_i(t)}=\left(s_{\zeta(t_i)}^*|\partial_t\tilde{g}|_{\tilde{g}}\right)(\varepsilon(t_i)^2t+t_i)\leq C.
\end{equation} 
Here we have not used the full decay in space of Lemma \ref{lemma-time-der-metric}. Estimate \eqref{der-time-est-resc-met} implies in particular that for $t\geq \frac{T_i-t_i}{\varepsilon(t_i)^2}$:
\begin{equation}
\begin{split}\label{control-time-der-resc-h}
\partial_t|\tilde{h}_i|^2_{\tilde{g}_i}&\leq 2|\partial_t\tilde{h}_i|_{\tilde{g}_i}|\tilde{h}_i|_{\tilde{g}_i}+C|\tilde{h}_i|_{\tilde{g}_i}^2\\
&\leq \frac{C}{(1+r_b)^{2+\sigma}}|\tilde{h}_i|_{\tilde{g}_i}+C|\tilde{h}_i|_{\tilde{g}_i}^2\\
&\leq C,
\end{split}
\end{equation}
 where $C$ denotes a positive constant that may vary from line to line. The penultimate inequality is obtained by unravelling the definitions of the rescaled tensors $\tilde{h}_i$ together with Proposition \ref{Schauder nos normes}. The last inequality is simply obtained by invoking the boundedness of $\tilde{h}_i$ on the region under consideration. Therefore, integrating \eqref{control-time-der-resc-h} implies the following uniform estimate (with respect to the index $i$):
 \begin{equation*}
|\tilde{h}_i(t)|_{\tilde{g}_i(t)}\leq C\left(t-\frac{T_i-t_i}{\varepsilon(t_i)^2}\right),\quad t\geq \frac{T_i-t_i}{\varepsilon(t_i)^2}.
\end{equation*}
As $i$ goes to $+\infty$, we therefore get 
\begin{equation}\label{cond-init-cond-preser}
|\hat{h}(t)|_{g_b}\leq C(t-T_{\infty}),\quad\text{ for $t\geq T_{\infty}$ uniformly in space on $N$.}
\end{equation}
 
Therefore the growth \eqref{cond-init-cond-preser} together with \eqref{h-hat-lic-eqn-imm} enables us to use the Liouville Theorem from Proposition \ref{liouville-ALE-initial} to derive that $\hat{h}=0$ which leads to a contradiction.

 Second, if $\lim_i\frac{T_i-t_i}{\varepsilon(t_i)^2}=-\infty$, then we have an ancient solution to \eqref{h-hat-lic-eqn-imm}.
By the construction of $\tilde{\mathbf{O}}(t)$, any rescaled element decays like $(1+r_b)^{-4}$ uniformly along the sequence with respect to the rescaled metrics $\tilde{g}_i(t)$ while $|\tilde{h}_i(t)|_{\tilde{g}_i(t)} = O((1+r_b)^{-\sigma})$ uniformly along the sequence as checked in \eqref{unif-decay-h-hat-imm}. In particular, this implies that the orthogonality conditions are preserved at the limit, i.e. the Lebesgue dominated convergence theorem implies for all $j\in \{1,2,3\}$:
\begin{equation*}
\int_{N}\langle \hat{h}(t),\oi_{j}(t)\rangle_{g_b}\,d\mu_{g_b}=0,\quad t\geq \lim_i\frac{T_i-t_i}{\varepsilon(t_i)^2}.
\end{equation*}
in this region  we also similarly satisfy the orthogonality conditions against the kernel of $\Delta_{L,g_b}$. 

These facts contradict the Liouville Theorem from Proposition \ref{liouville-ALE}.
\\

The other cases are dealt with in a similar way: the blow-ups are as in the proof of Proposition \ref{6.2 chez nous} and the limits of $\lim_i\frac{T_i-t_i}{\varepsilon(t_i)^2}$ determine whether we use the Liouville Theorems with initial conditions from Propositions \ref{Liouville-negative-einstein-initial} and \ref{prop-liouville-immortel-euc} or the Liouville Theorem for ancient solutions from Propositions \ref{Liouville-negative-einstein-ancien} and \ref{prop-liouville-ancien-euc}.
\end{proof}

\subsection{Estimates on Hölder norms of elements of the approximate kernel}
We end this section by recalling how to estimate the previously introduced H\"older norm in Section \ref{subsec-holder} for a tensor times a function and contractions of tensors. The proof is straightforward from the definition and Lemma \ref{lemma-prod-rule-semi-norm}.
\begin{lemma}\label{lemma-gal-control-semi-norm}
Let $\alpha\in(0,1/2)$, $(S(t))_{t\,\leq\, T}$ be a tensor on $M$ and let $f:(-\infty,T]\rightarrow \RR$ be a function. Then, 
 if $\sigma>0$,
\begin{equation*}
\begin{split}
\|f\cdot S\|_{C^{0,\alpha}_{\gamma,\sigma,T}}&\leq \|f\cdot S\|_{C^{0}_{\gamma,\sigma,T}}+\sup_{(x,t)\in M\times(-\infty,T]}\gamma(t)\sup_{t-\rho_0(x,t)^2\,\leq \,t'\,\leq \,t}|f(t')|\rho(x,t)^{\sigma}[S]_{\alpha,x,t}\\
&\quad+\sup_{(x,t)\in M\times(-\infty,T]}\gamma(t)[f]_{\alpha,t,\rho}\rho(x,t)^{\sigma}|S(x,t)|_{\tilde{g}(t)}.
\end{split}
\end{equation*}
Finally, if $(S_i(t))_{t\,\leq\, T}$, $i=1,2$, denotes two families of tensors such that $S_i\in C^{0,\alpha}_{\gamma,\sigma_i,T}$ for $\sigma_i>0$ then:
\begin{equation*}
\begin{split}
\|S_1 \ast S_2\|_{C^{0,\alpha}_{\gamma^2,\sigma_1+\sigma_2,T}}&\leq C\|S_1\|_{C^{0}_{\gamma,\sigma_1,T}}\|S_2\|_{C^{0,\alpha}_{\gamma,\sigma_2,T}}+C\|S_1\|_{C^{0,\alpha}_{\gamma,\sigma_1,T}}\|S_2\|_{C^{0}_{\gamma,\sigma_2,T}}.
\end{split}
\end{equation*}
\end{lemma}

The next proposition estimates the H\"older norms of each element of the approximate kernel $\tilde{\mathbf{O}}(t)$:
\begin{prop}\label{prop-holder-semi-norm-app-ker}
Let $\sigma\in(0,1)$, $\alpha\in(0,1/2)$ and let $f:(-\infty,T]\rightarrow \RR$ be a function. Then the following general $C^{0,\alpha}_{\gamma,\sigma,T}$-estimates hold:
\begin{equation*}
\begin{split}
\|f\cdot\tilde{\oi}_i\|_{C^{0,\alpha}_{\gamma,\sigma+2,T}}&\leq \|f\cdot \tilde{\oi}_i\|_{C^{0}_{\gamma,\sigma+2,T}}+\sup_{t\in (-\infty,T]}\gamma(t)\varepsilon(t)^{2+\sigma}\sup_{t-4\delta(t)^2\,\leq \,t'\,\leq \,t}|f(t')|\\
&\quad+\sup_{t\,\leq\,T}\gamma(t)\varepsilon(t)^{4}\sup_{\rho\,\leq\,2\delta(t)}\rho^{-2+\sigma}[f]_{\alpha,t,\rho},\\
\|f\cdot\tilde{g}\|_{C^{0,\alpha}_{\gamma,\sigma+2,T}}&\leq\|f\cdot \tilde{g}\|_{C^{0}_{\gamma,\sigma+2,T}}+\sup_{t\,\leq\,T}\gamma(t)\varepsilon(t)^{4}\sup_{t-c_0\,\leq \,t'\,\leq \,t}|f(t')|\\
&\quad+\sup_{t\,\leq\,T}\gamma(t)\sup_M\rho^{2+\sigma}[f]_{\alpha,t,\rho},\\
\|f\cdot\tilde{c}\|_{C^{0,\alpha}_{\gamma,\sigma+2,T}}&\leq \|f\cdot \tilde{c}\|_{C^{0}_{\gamma,\sigma+2,T}}+\sup_{t\,\leq\, T}\gamma(t)\sup_{t-c_0\,\leq \,t'\,\leq \,t}|f(t')|\\
&\quad+\sup_{t\,\leq\,T}\gamma(t)\sup_{ M}\rho^{2+\sigma}[f]_{\alpha,t,\rho},
\end{split}
\end{equation*}
for some uniform positive constant $c_0$.
\end{prop}

\begin{proof}
The proof follows from Lemma \ref{lemma-gal-control-semi-norm} in combination with Lemma \ref{estimates variations et laplacien o1} (estimates holding on $\nabla^{\tilde{g}(t)}\tilde{\oi}_i(t)$ and $\partial_t\tilde{\oi}_i(t)$), Lemma \ref{lemma-time-der-metric} (estimate holding on $\partial_t\tilde{g}(t)$) and Lemma \ref{est-basic-conf} (estimates holding on $\nabla^{\tilde{g}(t)}\tilde{c}(t)$ and $\partial_t\tilde{c}(t)$).

\end{proof}
\section{Linear theory: the constrained case}\label{sec:linear theory}



\subsection{Decoupling the linearized equation}\label{sec-decoupling}

The purpose of this section is to explain our strategy to solve at a heuristic level and discuss different possible strategies that we have not explored but which might sound plausible. 
We start with the following straightforward but crucial lemma:
\begin{lemma}\label{easy-time-der-abs-2-tensors}
Let $(M,g(t))_t$ be a one-parameter $C^1$ family of $C^2$ Riemannian metrics on a closed manifold $M$ and let $(S(t))_t$ and $(T(t))_{t}$ be two families of $C^2$ tensors on $M$. Then:
 \begin{equation*}
\begin{split}
 \frac{d}{dt}&\langle S(t),T(t)\rangle_{L^2(g(t))}=\int_M\left\langle\left( \partial_t-\Delta_{L,g(t)}\right)S(t),T(t)\right\rangle_{g(t)}\,d\mu_{g(t)}\\
&+ \int_M\left\langle S(t),\left(\partial_t+\Delta_{L,\,g(t)}+\frac{\tr_{g(t)}\partial_tg(t)}{2}\operatorname{Id}-\Sym(\partial_tg(t)\circ \cdot)\right)T(t)\right\rangle_{g(t)}\,d\mu_{g(t)},
 \end{split}
 \end{equation*}
 {where in coordinates, one has $\Sym(a\circ b)_{ij} = g^{pq}a_{ip}b_{qj} + g^{pq}b_{ip}a_{qj}$.}
 
 In particular, if $S(t)$ satisfies $\partial_tS(t)-\Delta_{L,g(t)}S(t)=\psi(t)$ for some non-homogeneous data $\psi(t)$ then:
 \begin{equation*}
\begin{split}
 \frac{d}{dt}&\langle S(t),T(t)\rangle_{L^2(g(t))}=\int_M\left\langle\psi(t),T(t)\right\rangle_{g(t)}\,d\mu_{g(t)}\\
&+ \int_M\left\langle S(t),\left(\partial_t+\Delta_{L,\,g(t)}+\frac{\tr_{g(t)}\partial_tg(t)}{2}\operatorname{Id}-\Sym(\partial_tg(t)\circ \cdot)\right)T(t)\right\rangle_{g(t)}\,d\mu_{g(t)}.
 \end{split}
 \end{equation*}

\end{lemma}

In the next section, we will estimate projections of a solution $h\in C^{2,\alpha}_{\gamma,\sigma,T}$ orthogonal to the approximate kernel $\tilde{\mathbf{O}}(t)$ for each $t\leq T$ as defined in Definition \ref{defn-app-kernel} and satisfying for  $t\leq T$,
 \begin{equation}\label{para-eqn}
  (\partial_t - \Delta_{L,\tilde{g}(t)}-2\Lambda)h(t) - \psi(t) \in \tilde{\mathbf{O}}(t),  
  \end{equation}
A couple of remarks following the previous steps.
\begin{itemize}
\item Motivated by Lemma \ref{easy-time-der-abs-2-tensors}, we could solve alternatively a backward Cauchy problem:
\begin{equation*}
\left\{
\begin{aligned}
0&=
\left(\partial_t+\Delta_{L,\tilde{g}(t)}+\frac{\tr_{g(t)}\partial_tg(t)}{2}\operatorname{Id}-\Sym(\partial_tg(t)\circ \cdot)\right)e_i(t),\\
e_i(t_0)&=\tilde{\oi}_i(t_0),\quad \text{approximate "geometric" tensor at time $t_0$}.
\end{aligned}
\right.
\end{equation*}
in order to define the approximate kernel $\tilde{\mathbf{O}}(t)$ coming from the $L^2$-kernel of the Lichnerowicz operator of the Eguchi-Hanson metric (or more generally a hyperkähler ALE metric). \\

One advantage of this choice over the original one \eqref{oi-app} is their canonical aspect. One drawback is that we will have to estimate these solutions and compare them with the first choice anyway.\\

\item	Another final idea is that we could take different bases $(e_i(t))_i$ and $(f_i(t))_i$ for the kernel and for the cokernel: meaning that we would project $\partial_t\tilde{g}(t)$ and the abstract solution $h(t)$ on some basis $(e_i(t))_i$ but we would then project $\psi$, $\Ric(\tilde{g}(t))$ or $\Delta_{L,\tilde{g}(t)} h(t)$ against some different basis $(f_i(t))_i$. Formally speaking, the kernel seems to naturally evolve by $\partial_t e_i(t) = \Delta_{L,\tilde{g}(t)}^{-1}(Q_{\tilde{g}(t)}(\partial_t \tilde{g}(t),e_i(t)))$ and the cokernel satisfies: for all $h$,
	$$ \langle \Delta_{L,\tilde{g}(t)}h, f_i(t)\rangle_{\tilde{g}(t)} = 0. $$
\end{itemize}
\subsection{Orthogonal projections for abstract data: pointwise estimates}\label{Orthogonal projections for abstract data: pointwise estimates}

The aim of this section is to establish pointwise estimates in time on orthogonal projections of the non-homogeneous data of the parabolic equation \eqref{para-eqn} against the elements of the approximate kernel $\tilde{\mathbf{O}}(t).$ 
This is echoing \cite[Proposition $6.5$]{Bre-Kap} in our setting.

{ We will deal with both the case of ancient solutions, the immortal case being completely analogous.}
\subsubsection{Orthogonal projections against $\tilde{\oi}_i(t)$}
We start with a useful property between integral and pointwise norms:
\begin{lemma}\label{lemma-easy-peasy}
Let $\sigma\in(0,4)$ and $\psi\in C^0_{\gamma,\sigma,T}$. There exists $C>0$ such that one has the following domination property in terms of the $C^0_{\gamma,\sigma,T}$ norm:
\begin{equation*}
\left|\langle \psi(t),\tilde{\oi}_i(t)\rangle_{L^2(\tilde{g}(t))}\right|\leq C\gamma(t)^{-1}\varepsilon(t)^{4-\sigma}\|\psi\|_{C^0_{\gamma,\sigma,T}}.
\end{equation*}
Moreover the $L^2(\tilde{g}(t))$-orthogonal projection of a symmetric $2$-tensor $\psi\in {C^0_{\gamma,\sigma,T}}$ on the line $\RR\cdot \tilde{\oi}_i(t)$ and denoted by $L_i(t) \tilde{\oi}_{i}(t)$, satisfies
\begin{equation*}
|L_i(t)| := \frac{\left|\langle \psi(t),\tilde{\oi}_i(t)\rangle_{L^2(\tilde{g}(t))}\right|}{\|\tilde{\oi}_{i}(t)\|^2_{L^2(\tilde{g}(t))}}\leq \gamma(t)^{-1}\varepsilon(t)^{-\sigma}\|\psi\|_{C^0_{\gamma,\sigma,T}},
\end{equation*}
and,
\begin{equation}
\|L_i\cdot\tilde{\oi}_{i}\|_{C^0_{\gamma,\sigma,T}} \leq C \|\psi\|_{C^0_{\gamma,\sigma,T}}. \label{orthogonal projection o1}
\end{equation}
\end{lemma}
\begin{proof}
Using the very definition of the norm $C^{0}_{\gamma,\sigma,T}$ and the fact that $|\tilde{\oi}_i(t)|_{\tilde{g}(t)}=O(\varepsilon(t)^{4}(\varepsilon(t)+r_o)^{-4})$, one gets immediately the first expected estimate. Since $C^{-1}\varepsilon(t)^4\leq \|\tilde{\oi}_i(t)\|_{L^2(\tilde{g}(t)}^2\leq C\varepsilon(t)^4$ thanks to Proposition \ref{prop-o_1-first-app}, one also gets the corresponding estimate on $L_i(t)$.

Finally, 
\begin{equation*}
\gamma(t)(\varepsilon(t)+r_o)^{\sigma}|L_i(t)||\tilde{\oi}_{i}(t)|_{\tilde{g}(t)}\leq C\varepsilon(t)^{4-\sigma}(\varepsilon(t)+r_o)^{-4+\sigma}\|\psi\|_{C^0_{\gamma,\sigma,T}}\leq C\|\psi\|_{C^0_{\gamma,\sigma,T}},
\end{equation*}
where we have used the previous estimate on $L_i(t)$. This proves \eqref{orthogonal projection o1}.
\end{proof}

\begin{lemma}\label{pre-partial-h}
Let $h\in C^2_{\gamma,\sigma,T}$, { $\sigma\in(0,2)$}, such that $h(t)$ is $L^2(\tilde{g}(t))$-orthogonal to the approximate kernel $\tilde{\mathbf{O}}(t)$ for each $t\leq T$. Then,
\begin{equation*}
\begin{split}
\left|\langle \partial_th,\tilde{\oi}_i(t)\rangle_{L^2(\tilde{g}(t))}\right|&\leq C\|h\|_{C^0_{\gamma,\sigma,T}}\gamma(t)^{-1}\varepsilon(t)^{4-\sigma}.
\end{split}
\end{equation*}
\end{lemma}
\begin{proof}
{ The proof is the same for ancient and immortal flows.}

By differentiating the orthogonality condition $\langle h(t),\tilde{\oi}_i(t)\rangle_{L^2(\tilde{g}(t))}=0$, one gets thanks to Lemma \ref{easy-time-der-abs-2-tensors-prelim}:
\begin{equation*}
\begin{split}
\left|\langle \partial_th(t),\tilde{\oi}_i(t)\rangle_{L^2(\tilde{g}(t))}\right|&=\left|\Big\langle h(t),\Big(\partial_t+\frac{\tr_{\tilde{g}(t)}\partial_t\tilde{g}(t)}{2}\operatorname{Id}-\Sym(\partial_t\tilde{g}(t)\circ \cdot)\Big)\tilde{\oi}_i(t)\Big\rangle_{L^2(\tilde{g}(t))}\right|\\
&\leq C\int_{M}| h(t)|_{\tilde{g}(t)}\left(\left|\partial_t\tilde{\oi}_i(t)\right|_{\tilde{g}(t)}+|\partial_t\tilde{g}(t)|_{\tilde{g}(t)}|\tilde{\oi}_i(t)|_{\tilde{g}(t)}\right)\,d\mu_{\tilde{g}(t)}.
\end{split}
\end{equation*}
Since the tensor $\tilde{\oi}_i(t)$ is supported in $\{r_o\,\leq\,2\delta(t)\}$ by definition, the previous integral on the righthand side can be estimated as follows. We first consider terms supported in $\{\delta(t)/2\,\leq\,r_o\,\leq\,2\delta(t)\}$ and we arrive at the following bounds
\begin{equation*}
\begin{split}
 C\left(\varepsilon(t)^8\delta(t)^{-4-\sigma}+\varepsilon(t)^{4^-}\delta(t)^{-\sigma^-}\right)\leq C\varepsilon(t)^{4-\sigma},
\end{split}
\end{equation*}
since $\delta<1$ by \eqref{cond-param}. Here we have used Lemmata \ref{lemma-time-der-metric} and \ref{estimates variations et laplacien o1}.
Similarly, terms supported in $\{r_o\,\leq\,2\delta(t)\}$ can be bounded from above by:
\begin{equation*}
\begin{split}
C\int_{\{r_o\,\leq\,2\delta(t)\}}\varepsilon(t)^4(\varepsilon(t)+r_o)^{-4-\sigma}\,d\mu_{\tilde{g}(t)}\leq C\varepsilon(t)^{4-\sigma}.
\end{split}
\end{equation*}

\end{proof}

The next proposition estimates a priori the $L^2$-projection of the non-homogeneous data of the parabolic equation \eqref{para-eqn} against the elements $\tilde{\oi}_i(t)$ of the approximate kernel $\tilde{\mathbf{O}}(t).$

\begin{prop}\label{prop 6.5 E(t)}
Let $\psi\in C^{0}_{\gamma,\sigma+2,T}$ and let $h\in C^{2,\alpha}_{\gamma,\sigma,T}$ such that $h(t)$ is $L^2(\tilde{g}(t))$-orthogonal to the approximate kernel $\tilde{\mathbf{O}}(t)$ for each $t\leq T$ (respectively $t\geq T$ in the immortal case) and satisfying for all $t\leq T$ (respectively $t\geq T$ in the immortal case),
 \begin{equation*}
  (\partial_t - \Delta_{L,\tilde{g}(t)}-2\Lambda)h(t) - \psi(t) \in \tilde{\mathbf{O}}(t),  
  \end{equation*}
  i.e. 
   \begin{equation*}
  (\partial_t - \Delta_{L,\tilde{g}(t)}-2\Lambda)h(t) = \psi(t) +\lambda_i(t)\tilde{\oi}_i(t)+\nu(t)\tilde{g}(t)+\mu(t)\tilde{c}(t).  
  \end{equation*}
Then there exists a uniform positive constant $C$ such that for $t\leq T$ (respectively $t\geq T$ in the immortal case),
\begin{equation}
\begin{split}\label{why-is-it-so-painful}
\left|\langle \psi(t)+\lambda_i(t)\tilde{\oi}_i(t),\tilde{\oi}_i(t)\rangle_{L^2(\tilde{g}(t))}\right|&\leq C\gamma(t)^{-1}\|h\|_{C^0_{\gamma,\sigma,T}}\varepsilon(t)^{4}\delta(t)^{-2-\sigma}+C\gamma(t)^{-1}\varepsilon(t)^{4-\sigma}\|\psi\|_{C^{0}_{\gamma,\sigma+2,T}}\\
&\quad+C\varepsilon(t)^4\delta(t)^2|\nu(t)|+C\varepsilon(t)^4|\mu(t)|,\\
\varepsilon(t)^4|\lambda_i(t)|&\leq C\gamma(t)^{-1}\|h\|_{C^0_{\gamma,\sigma,T}}\varepsilon(t)^{4}\delta(t)^{-2-\sigma}+C\gamma(t)^{-1}\varepsilon(t)^{2-\sigma}\|\psi\|_{C^{0}_{\gamma,\sigma+2,T}}\\
&\quad+C\varepsilon(t)^4\delta(t)^2|\nu(t)|+C\varepsilon(t)^4|\mu(t)|.
\end{split}
\end{equation}

In particular,  there exists some uniform positive constant $C$ such that:
\begin{equation*}
\begin{split}
\|\lambda_i\cdot\tilde{\oi}_i\|_{C^0_{\gamma,\sigma+2,T}}&\leq o_T(1)\|h\|_{C^0_{\gamma,\sigma,T}}+C\|\psi\|_{C^0_{\gamma,\sigma+2,T}}\\
&\quad+C\sup_{t\,\leq \,T}\gamma(t)\varepsilon(t)^{2+\sigma}|\mu(t)|+C\sup_{t\,\leq \,T}\gamma(t)\varepsilon(t)^{2+\sigma}\delta(t)^2|\nu(t)|.
\end{split}
\end{equation*}

\end{prop}

\begin{proof}
{We provide a proof for ancient flows, but the statement also holds for immortal flows by flipping the sign of time. Indeed, the controls are always made on $\partial_t$ and $\Delta_{L,\tilde{g}(t)}$ separately so looking at $\partial_t+\Delta_{L,\tilde{g}(t)}$ or $\partial_t-\Delta_{L,\tilde{g}(t)}$ yields the same estimate.}

 Since $h(t)$ is orthogonal to $\tilde{\mathbf{O}}(t)$ for each $t\leq T$, a spatial integration by parts yields:
\begin{equation*}
\begin{split}
\langle \psi(t)+\lambda_j(t)\tilde{\oi}_j(t)&+\nu(t)\tilde{g}(t)+\mu(t)\tilde{c}(t),\tilde{\oi}_i(t)\rangle_{L^2(\tilde{g}(t))}=\Big\langle(\partial_t-\Delta_{L,\tilde{g}(t)} )h(t),\tilde{\oi}_i(t)\Big\rangle_{L^2(\tilde{g}(t))}\\
&=\Big\langle \partial_th(t),\tilde{\oi}_i(t)\Big\rangle_{L^2(\tilde{g}(t))}-\Big\langle h(t),\Delta_{L,\tilde{g}(t)} \tilde{\oi}_i(t)\Big\rangle_{L^2(\tilde{g}(t))}.
\end{split}
\end{equation*}
In particular, Lemmata \ref{pre-partial-h} and \ref{estimates variations et laplacien o1} give:
\begin{equation}
\begin{split}\label{prelim-proj-est-oi}
&\left|\langle \psi(t)+\lambda_i(t)\tilde{\oi}_i(t),\tilde{\oi}_i(t)\rangle_{L^2(\tilde{g}(t))}\right|\leq \left|\langle \partial_th,\tilde{\oi}_i(t)\rangle_{L^2(\tilde{g}(t))}\right|\\
&+C\int_{M}| h(t)|_{\tilde{g}(t)}|\Delta_{L,\tilde{g}(t)}\tilde{\oi}_i(t)|_{\tilde{g}(t)}\,d\mu_{\tilde{g}(t)}+|\nu(t)|\left|\langle \tilde{\oi}_i(t),\tilde{g}(t)\rangle_{L^2(\tilde{g}(t))}\right|\\
&+|\mu(t)|\left|\langle \tilde{\oi}_i(t),\tilde{c}(t)\rangle_{L^2(\tilde{g}(t))}\right|+\sum_{j\neq i}|\lambda_j(t)|\left|\langle \tilde{\oi}_i(t),\tilde{\oi}_j(t)\rangle_{L^2(\tilde{g}(t))}\right|\\
&\leq C\|h\|_{C^0_{\gamma,\sigma,T}}\gamma(t)^{-1}\varepsilon(t)^4(\delta(t)^{-2-\sigma}+\varepsilon(t)^{-\sigma})\\
&\quad+C\varepsilon(t)^4\delta(t)^2|\nu(t)|+C\varepsilon(t)^4|\mu(t)|+C\varepsilon(t)^8\delta(t)^{-4}\sum_{j\neq i}|\lambda_j(t)|\\
&\leq C\|h\|_{C^0_{\gamma,\sigma,T}}\gamma(t)^{-1}\varepsilon(t)^4\delta(t)^{-2-\sigma}\\
&\quad+C\varepsilon(t)^4\delta(t)^2|\nu(t)|+C\varepsilon(t)^4|\mu(t)|+C\varepsilon(t)^8\delta(t)^{-4}\sum_{j\neq i}|\lambda_j(t)|,
\end{split}
\end{equation}
where $C$ is a time-independent positive constant that may vary from line to line. Here we have invoked Proposition \ref{prop-o_1-first-app} and Proposition \ref{prop-proj-conf} 
 in the penultimate inequality. Finally, \eqref{cond-param} is used in the last inequality to ensure that $(2+\sigma)\delta\geq \sigma$ since $\sigma\in(0,1)\subset (0,2)$ and $\delta\geq 1/2$.


Now, Lemma \ref{lemma-easy-peasy} applied to $\psi(t)$ with $\sigma+2$ in lieu of $\sigma$ leads to:
\begin{equation*}
\begin{split}
|\lambda_i(t)|\|\tilde{\oi}_i(t)\|^2_{L^2(\tilde{g}(t))}&\leq C\gamma(t)^{-1}\|h\|_{C^0_{\gamma,\sigma,T}}\varepsilon(t)^{4}\delta(t)^{-2-\sigma}+C\varepsilon(t)^4\delta(t)^2|\nu(t)|\\
&\quad+C\varepsilon(t)^4|\mu(t)|+C\varepsilon(t)^8\delta(t)^{-4}\sum_{j\neq i}|\lambda_j(t)|+C\gamma(t)^{-1}\varepsilon(t)^{2-\sigma}\|\psi\|_{C^0_{\gamma,\sigma+2,T}}.
\end{split}
\end{equation*}
Since $\|\tilde{\oi}_i(t)\|^2_{L^2(\tilde{g}(t))}$ is proportional to $\varepsilon(t)^4$ thanks to Proposition \ref{prop-o_1-first-app}, one gets:
\begin{equation}
\begin{split}\label{prelim-est-lambdai}
|\lambda_i(t)|&\leq C\gamma(t)^{-1}\|h\|_{C^0_{\gamma,\sigma,T}}\delta(t)^{-2-\sigma}+C\delta(t)^2|\nu(t)|+C|\mu(t)|\\
&\quad+C\varepsilon(t)^4\delta(t)^{-4}\sum_{j\neq i}|\lambda_j(t)|+C\gamma(t)^{-1}\varepsilon(t)^{-2-\sigma}\|\psi\|_{C^0_{\gamma,\sigma+2,T}}.
\end{split}
\end{equation}
By summing \eqref{prelim-est-lambdai} over $i$ and by absorption, one gets the second half of the first desired estimates, i.e.
\begin{equation}
\begin{split}\label{prelim-est-lambdai-bis} 
\varepsilon(t)^4|\lambda_i(t)|&\leq C\gamma(t)^{-1}\|h\|_{C^0_{\gamma,\sigma,T}}\varepsilon(t)^{4}\delta(t)^{-2-\sigma}\\
&\quad+C\varepsilon^4(t)\delta(t)^2|\nu(t)|+C\varepsilon(t)^4|\mu(t)|+\sum_j\left|\langle \psi(t),\tilde{\oi}_j(t)\rangle_{L^2(\tilde{g}(t))}\right|\\
&\leq C\gamma(t)^{-1}\|h\|_{C^0_{\gamma,\sigma,T}}\varepsilon(t)^{4}\delta(t)^{-2-\sigma}\\
&\quad+C\varepsilon^4(t)\delta(t)^2|\nu(t)|+C\varepsilon(t)^4|\mu(t)|+C\gamma(t)^{-1}\varepsilon(t)^{2-\sigma}\|\psi\|_{C^0_{\gamma,\sigma+2,T}}.
\end{split}
\end{equation}
By inserting \eqref{prelim-est-lambdai-bis} back to the righthand side of \eqref{prelim-proj-est-oi}, one gets the first half of the first desired estimate
in \eqref{why-is-it-so-painful}.

Now, since the tensor $\tilde{\oi}_i(t)$ is supported in $\{r_o\,\leq\,2\delta(t)\}$, one also gets the following pointwise estimate from \eqref{prelim-est-lambdai-bis} :
\begin{equation*}
\begin{split}
\|\lambda_i\cdot\tilde{\oi}_i\|_{C^0_{\gamma,\sigma+2,T}}&\leq C\|h\|_{C^0_{\gamma,\sigma,T}}\varepsilon(T)^{2+\sigma}\delta(T)^{-2-\sigma}+C\sup_{t\,\leq \,T}\gamma(t)\varepsilon(t)^{2+\sigma}\delta(t)^2|\nu(t)|\\
&\quad+C\sup_{t\,\leq \,T}\gamma(t)\varepsilon(t)^{2+\sigma}|\mu(t)|+C\|\psi\|_{C^0_{\gamma,\sigma+2,T}}\\
&\leq o_T(1)\|h\|_{C^0_{\gamma,\sigma,T}}+C\sup_{t\,\leq \,T}\gamma(t)\varepsilon(t)^{2+\sigma}\delta(t)^2|\nu(t)|\\
&\quad+C\sup_{t\,\leq \,T}\gamma(t)\varepsilon(t)^{2+\sigma}|\mu(t)|+C\|\psi\|_{C^0_{\gamma,\sigma+2,T}},
\end{split}
\end{equation*}
where $C$ is a time-independent positive constant that may vary from line to line. Here we have used that $\oi_i$ decays like $\varepsilon(t)^4(\varepsilon(t)+r_o)^{-4}$ and the assumption $\sigma\in(0,2)$ in the first line together with the choice of $\gamma(t)$ in the last line.
\end{proof}

\subsubsection{Orthogonal projections against $\tilde{g}(t)$}
We continue with an analogous property established in Lemma \ref{lemma-easy-peasy} between integral and pointwise norms, here we use the essential fact that the function $(\varepsilon(t)+r_o)^{-\sigma}$ is in $L^1(\tilde{g}(t))$ uniformly in time if $\sigma\in(0,4)$.
\begin{lemma}\label{lemma-easy-peasy-bis}
One has the following domination property in terms of the $C^0_{\gamma,\sigma,T}$ norm if $\sigma\in(0,4)$:
\begin{equation*}
\left|\langle \psi(t),\tilde{g}(t)\rangle_{L^2(\tilde{g}(t))}\right|\leq \gamma(t)^{-1}\|\psi\|_{C^0_{\gamma,\sigma,T}}.
\end{equation*}
Moreover the $L^2(\tilde{g}(t))$-orthogonal projection of a symmetric $2$-tensor $\psi\in {C^0_{\gamma,\sigma,T}}$ on the line $\RR\cdot \tilde{g}(t)$ denoted by $M(t) \tilde{g}(t)$, satisfies
\begin{equation}
\|M\cdot\tilde{g}\|_{C^0_{\gamma,\sigma,T}} \leq C \|\psi\|_{C^0_{\gamma,\sigma,T}}, \label{orthogonal projection g}
\end{equation}
for some uniform positive constant $C$.
\end{lemma}
The proof of Lemma \ref{lemma-easy-peasy-bis} is along the same lines as those of the proof of Lemma \ref{lemma-easy-peasy}.

\begin{lemma}\label{pre-partial-h-2}
Let $h\in C^2_{\gamma,\sigma,T}$ such that $h(t)$ is $L^2(\tilde{g}(t))$-orthogonal to the approximate kernel $\tilde{\mathbf{O}}(t)$ for each $t\leq T$. Then,
\begin{equation*}
\begin{split}
\left|\langle \partial_th,\tilde{g}(t)\rangle_{L^2(\tilde{g}(t))}\right|&\leq C\|h\|_{C^0_{\gamma,\sigma,T}}\gamma(t)^{-1}\varepsilon(t)^{4}\delta(t)^{-4-\sigma}.
\end{split}
\end{equation*}

\end{lemma}
\begin{proof}
By differentiating the orthogonality condition $\langle h(t),\tilde{g}(t)\rangle_{L^2(\tilde{g}(t))}=0$, one gets thanks to Lemma \ref{easy-time-der-abs-2-tensors-prelim}:
\begin{equation*}
\begin{split}
\left|\langle \partial_th(t),\tilde{g}(t)\rangle_{L^2(\tilde{g}(t))}\right|&=\left|\Big\langle h(t),\Big(\partial_t+\frac{\tr_{\tilde{g}(t)}\partial_t\tilde{g}(t)}{2}\operatorname{Id}-\Sym(\partial_t\tilde{g}(t)\circ \cdot)\Big)\tilde{g}(t)\Big\rangle_{L^2(\tilde{g}(t))}\right|\\
&\leq C\int_{M}| h(t)|_{\tilde{g}(t)}|\partial_t\tilde{g}(t)|_{\tilde{g}(t)}\,d\mu_{\tilde{g}(t)}\\
&\leq C\gamma(t)^{-1}\|h\|_{C^0_{\gamma,\sigma,T}}\int_{M}(\varepsilon(t)+r_o)^{-\sigma}\left|\partial_t\tilde{g}(t)\right|_{\tilde{g}(t)}\,d\mu_{\tilde{g}(t)},
\end{split}
\end{equation*}
where $C$ is a time-independent positive constant that may vary from line to line. By invoking Lemma \ref{lemma-time-der-metric}, terms supported in $\{\delta(t)/2\,\leq\,r_o\,\leq\,2\delta(t)\}$ can be estimated by:
\begin{equation*}
\begin{split}
\int_M(\varepsilon(t)+r_o)^{-\sigma}\varepsilon(t)^4\delta(t)^{-4}\mathbbm{1}_{\{\delta(t)/2\,\leq\,r_o\,\leq\,2\delta(t)\}}\,d\mu_{\tilde{g}(t)}\leq C\varepsilon(t)^4\delta(t)^{-\sigma}.
\end{split}
\end{equation*}
Terms supported in $\{r_o\,\leq\,2\delta(t)\}$ can be bounded by
\begin{equation*}
\begin{split}
\int_{M}(\varepsilon(t)+r_o)^{-\sigma}\varepsilon(t)^4(\varepsilon(t)+r_o)^{-4}\mathbbm{1}_{\{r_o\,\leq\,2\delta(t)\}}\,d\mu_{\tilde{g}(t)}&\leq  C\varepsilon(t)^{4-\sigma},
\end{split}
\end{equation*}
where we have used \eqref{cond-param} in the last inequality. 

Finally, terms supported in $\{r_o\,\geq\,2\delta(t)\}$ can be bounded by
\begin{equation*}
\begin{split}
\int_{M}(\varepsilon(t)+r_o)^{-\sigma}\varepsilon(t)^4(\varepsilon(t)+r_o)^{-4}\mathbbm{1}_{\{r_o\,\geq\,2\delta(t)\}}\,d\mu_{\tilde{g}(t)}&\leq  C\varepsilon(t)^{4}\delta(t)^{-4-\sigma},
\end{split}
\end{equation*}Once again, the proof is the same for the ancient and immortal situations.
\end{proof}

The next proposition estimates a priori the $L^2$-projection of the non-homogeneous data of the parabolic equation \eqref{para-eqn} against the elements $\tilde{g}(t)$ of the approximate kernel $\tilde{\mathbf{O}}(t).$

\begin{prop}\label{prop 6.5 F(t)}
Let $\psi\in C^{0}_{\gamma,\sigma+2,T}$ and $h\in C^{2,\alpha}_{\gamma,\sigma,T}$ such that $h(t)$ is $L^2(\tilde{g}(t))$-orthogonal to the approximate kernel $\tilde{\mathbf{O}}(t)$ for each $t\leq T$ (respectively $t\geq T$ in the immortal case) and satisfying for $t\leq T$ (respectively $t\geq T$ in the immortal case),
 \begin{equation*}
  (\partial_t - \Delta_{L,\tilde{g}(t)}-2\Lambda)h(t) - \psi(t) \in \tilde{\mathbf{O}}(t),  
  \end{equation*}
  i.e. 
   \begin{equation*}
  (\partial_t - \Delta_{L,\tilde{g}(t)}-2\Lambda)h(t) = \psi(t) +\lambda_i(t)\tilde{\oi}_i(t)+\nu(t)\tilde{g}(t)+\mu(t)\tilde{c}(t).  
  \end{equation*}
Then there exists a uniform positive constant $C$ such that for $t\leq T$ { (respectively if $t\geq T$ in the immortal case)},
\begin{equation*}
\begin{split}
\left|\langle \psi(t)+\lambda_i(t)\tilde{\oi}_i(t)+\nu(t)\tilde{g}(t),\tilde{g}(t)\rangle_{L^2(\tilde{g}(t))}\right| &\leq C\gamma(t)^{-1}\varepsilon(t)^{4}\delta(t)^{-4-\sigma}\|h\|_{C^0_{\gamma,\sigma,T}},\quad\text{and,}\\
\left|\langle \psi(t)+\nu(t)\tilde{g}(t),\tilde{g}(t)\rangle_{L^2(\tilde{g}(t))}\right|&\leq C\gamma(t)^{-1}\varepsilon(t)^{4}\delta(t)^{-4-\sigma}\|h\|_{C^0_{\gamma,\sigma,T}}\\
&\quad+C\varepsilon(t)^4\delta(t)^2\sum_{i}|\lambda_i(t)|.
\end{split}
\end{equation*}

In particular, there exists some uniform positive constant $C$ such that:
\begin{equation*}
\begin{split}
\|\nu\cdot\tilde{g}\|_{C^0_{\gamma,\sigma+2,T}}&\leq  o_T(1)\|h\|_{C^0_{\gamma,\sigma,T}}+C\|\psi\|_{C^0_{\gamma,\sigma+2,T}}+C\sup_{t\,\leq \,T}\gamma(t)\varepsilon(t)^4\delta(t)^2\sum_i|\lambda_i(t)|,
\end{split}
\end{equation*}

\end{prop}

\begin{proof}
{ Since $\Delta_{L,\tilde{g}(t)}\tilde{g}(t)=0$, the proofs are exactly the same for the ancient and immortal situations.} By differentiating the orthogonality condition $\langle h(t),\tilde{g}(t)\rangle_{L^2(\tilde{g}(t))}=0$ and Lemma \ref{easy-time-der-abs-2-tensors}, one gets:
\begin{equation*}
\begin{split}
\langle \psi(t)+&\lambda_i(t)\tilde{\oi}_i(t)+\nu(t)\tilde{g}(t)+\mu(t)\tilde{c}(t),\tilde{g}(t)\rangle_{L^2(\tilde{g}(t))}=\\
&\left\langle (\partial_t-\Delta_{L,\tilde{g}(t)})h(t),\tilde{g}(t)\right\rangle_{L^2(\tilde{g}(t))}=\left\langle \partial_th(t),\tilde{g}(t)\right\rangle_{L^2(\tilde{g}(t))},
\end{split}
\end{equation*}
since $\Delta_{L,\tilde{g}(t)}\tilde{g}(t)=0$. In particular, since $\langle \tilde{c}(t),\tilde{g}(t)\rangle_{L^2(\tilde{g}(t)}=0$ by Proposition \ref{prop-proj-conf}, Lemma \ref{pre-partial-h-2} ensures:
\begin{equation*}
\begin{split}
&\left|\langle \psi(t)+\lambda_i(t)\tilde{\oi}_i(t)+\nu(t)\tilde{g}(t),\tilde{g}(t)\rangle_{L^2(\tilde{g}(t))}\right|\leq C\|h\|_{C^0_{\gamma,\sigma,T}}\gamma(t)^{-1}\varepsilon(t)^{4}\delta(t)^{-4-\sigma}.
\end{split}
\end{equation*}

The triangular inequality leads to:
\begin{equation*}
\begin{split}
\left|\langle \psi(t)+\nu(t)\tilde{g}(t),\tilde{g}(t)\rangle_{L^2(\tilde{g}(t))}\right|&\leq C\gamma(t)^{-1}\|h\|_{C^0_{\gamma,\sigma,T}}\varepsilon(t)^{4}\delta(t)^{-4-\sigma}+|\lambda_i(t)|\left|\langle \tilde{\oi}_i(t),\tilde{g}(t)\rangle_{L^2(\tilde{g}(t))}\right|\\
&\leq C\gamma(t)^{-1}\|h\|_{C^0_{\gamma,\sigma,T}}\varepsilon(t)^{4}\delta(t)^{-4-\sigma}+C\varepsilon(t)^4\delta(t)^2\sum_i|\lambda_i(t)|,
\end{split}
\end{equation*}
as claimed. Here we have used Proposition \ref{prop-o_1-first-app} in the last line.

Finally, Lemma \ref{lemma-easy-peasy-bis} applied to $\psi(t)$ leads to the following pointwise estimate:
\begin{equation*}
\begin{split}
\|\nu\cdot\tilde{g}\|_{C^0_{\gamma,\sigma+2,T}}&\leq C\varepsilon(T)^{4}\delta(T)^{-4-\sigma}\|h\|_{C^0_{\gamma,\sigma,T}}+C\sup_{t\,\leq \,T}\gamma(t)\varepsilon(t)^4\delta(t)^2\sum_i|\lambda_i(t)|+C\|\psi\|_{C^0_{\gamma,\sigma+2,T}},
\end{split}
\end{equation*}
which implies the desired estimate. 
\end{proof}

\subsubsection{Orthogonal projections against $\tilde{c}(t)$}

\begin{lemma}\label{lemma-easy-peasy-ter}
One has the following domination property in terms of the $C^0_{\gamma,\sigma+2,T}$ norm:
\begin{equation*}
\left|\langle \psi(t),\tilde{c}(t)\rangle_{L^2(\tilde{g}(t))}\right|\leq C\gamma(t)^{-1}\|\psi\|_{C^0_{\gamma,\sigma+2,T}},
\end{equation*}
for some uniform positive constant $C$.
Moreover the $L^2(\tilde{g}(t))$-orthogonal projection of a symmetric $2$-tensor $\psi\in {C^0_{\gamma,\sigma+2,T}}$ denoted by $N(t)\tilde{c}(t)$, satisfies

\begin{equation}
\|N(t)\tilde{c}(t)\|_{C^0_{\gamma,\sigma+2,T}} \leq C \|\psi\|_{C^0_{\gamma,\sigma+2,T}}, \label{orthogonal projection c}
\end{equation}
for some uniform positive constant $C$.
\end{lemma}

\begin{proof}
Lemma \ref{est-basic-conf} ensures that the tensor $\tilde{c}(t)$ is bounded uniformly in time with respect to $\tilde{g}(t)$. The desired estimate follows immediately from this fact.
\end{proof}

\begin{lemma}\label{pre-partial-h-3}
Let $h\in C^2_{\gamma,\sigma,T}$, orthogonal to the approximate kernel $\tilde{\mathbf{O}}(t)$ for each $t\leq T$. Then,
\begin{equation*}
\begin{split}
\left|\langle \partial_th,\tilde{c}(t)\rangle_{L^2(\tilde{g}(t))}\right|&\leq C\|h\|_{C^0_{\gamma,\sigma,T}}\gamma(t)^{-1}\varepsilon(t)^4\delta(t)^{-4-\sigma}.
\end{split}
\end{equation*}

\end{lemma}
\begin{proof}
By differentiating the orthogonality condition $\langle h(t),\tilde{c}(t)\rangle_{L^2(\tilde{g}(t))}=0$, one gets thanks to Lemma \ref{easy-time-der-abs-2-tensors-prelim}:
\begin{equation*}
\begin{split}
&\left|\langle \partial_th(t),\tilde{c}(t)\rangle_{L^2(\tilde{g}(t))}\right|=\left|\Big\langle h(t),\Big(\partial_t+\frac{\tr_{\tilde{g}(t)}\partial_t\tilde{g}(t)}{2}\operatorname{Id}-\Sym(\partial_t\tilde{g}(t)\circ \cdot)\Big)\tilde{c}(t)\Big\rangle_{L^2(\tilde{g}(t))}\right|\\
&\leq C\int_{M}| h(t)|_{\tilde{g}(t)}\left(\left|\partial_t\tilde{g}(t)\right|_{\tilde{g}(t)}|\tilde{c}(t)|_{\tilde{g}(t)}+\left|\partial_t\tilde{c}(t)\right|_{\tilde{g}(t)}\right)\,d\mu_{\tilde{g}(t)},
\end{split}
\end{equation*}
where $C$ is a time-independent positive constant that may vary from line to line. By invoking Lemma \ref{est-basic-conf}, for $t\leq T$,
\begin{equation*}
\begin{split}
\int_{M}| h(t)|_{\tilde{g}(t)}\left|\partial_t\tilde{c}(t)\right|_{\tilde{g}(t)}\,d\mu_{\tilde{g}(t)}&\leq C\gamma(t)^{-1}\|h\|_{C^0_{\gamma,\sigma,T}}\left(\varepsilon(t)^{4-\sigma}+\varepsilon(t)^4\delta(t)^{-4-\sigma}\right)\\
&\leq C\gamma(t)^{-1}\|h\|_{C^0_{\gamma,\sigma,T}}\varepsilon(t)^4\delta(t)^{-4-\sigma},
\end{split}
\end{equation*}
for some uniform positive constant $C$ thanks to \eqref{cond-param} that ensures that $\delta<(3-\sigma)/(5-\sigma)<(4-\sigma)/(6-\sigma)$. Finally, similarly to the proof of Proposition \ref{prop 6.5 F(t)} based on Lemma \ref{lemma-time-der-metric}, for $t\leq T\leq 0$,
\begin{equation*}
\int_{M}| h(t)|_{\tilde{g}(t)}\left|\partial_t\tilde{g}(t)\right|_{\tilde{g}(t)}|\tilde{c}(t)|_{\tilde{g}(t)}\,d\mu_{\tilde{g}(t)}\leq C\gamma(t)^{-1}\|h\|_{C^0_{\gamma,\sigma,T}}\varepsilon(t)^4\delta(t)^{-4-\sigma},
\end{equation*}
for some uniform positive constant $C$.
\end{proof}

The next proposition estimates a priori the $L^2$-projection of the non-homogeneous data of the parabolic equation \eqref{para-eqn} against the elements $\tilde{c}(t)$ of the approximate kernel $\tilde{\mathbf{O}}(t).$

\begin{prop}
\label{prop 6.5 M(t)}
Let $\psi\in C^{0}_{\gamma,\sigma+2,T}$ and $h\in C^{2,\alpha}_{\gamma,\sigma,T}$ such that $h(t)$ is $L^2(\tilde{g}(t))$-orthogonal to $\tilde{\mathbf{O}}(t)$ for each $t\leq T$ satisfying for $t\leq T$,
 \begin{equation*}
  (\partial_t - \Delta_{L,\tilde{g}(t)}-2\Lambda)h(t) - \psi(t) \in \tilde{\mathbf{O}}(t),  
  \end{equation*}
  i.e. 
   \begin{equation*}
  (\partial_t - \Delta_{L,\tilde{g}(t)}-2\Lambda)h(t) = \psi(t) +\lambda_i(t)\tilde{\oi}_i(t)+\nu(t)\tilde{g}(t)+\mu(t)\tilde{c}(t).  
  \end{equation*}
Then there exists a uniform positive constant $C$ such that for $t\leq T$,
\begin{equation*}
\begin{split}
\left|\langle \psi(t)+\mu(t)\tilde{c}(t),\tilde{c}(t)\rangle_{L^2(\tilde{g}(t))}\right|&\leq C\gamma(t)^{-1}\varepsilon(t)^4\delta(t)^{-4-\sigma}\|h\|_{C^0_{\gamma,\sigma,T}}+C\varepsilon(t)^4\sum_i|\lambda_i(t)|.
\end{split}
\end{equation*}
In particular, 
there exists some uniform positive constant $C$ such that:
\begin{equation*}
\begin{split}
  \|\mu\cdot\tilde{c}\|_{C^0_{\gamma,\sigma+2,T}}&\leq o_T(1)\|h\|_{C^0_{\gamma,\sigma,T}}+C\|\psi\|_{C^0_{\gamma,\sigma+2,T}}+C\sup_{t\,\leq \,T}\gamma(t)\varepsilon(t)^4\sum_i|\lambda_i(t)|.
\end{split}
\end{equation*}

\end{prop}

\begin{proof}
Lemma \ref{easy-time-der-abs-2-tensors} gives:
\begin{equation*}
\begin{split}
&\langle \psi(t)+\mu(t)\tilde{c}(t),\tilde{c}(t)\rangle_{L^2(\tilde{g}(t))}\\
&=\left\langle (\partial_t-\Delta_{L,\tilde{g}(t)})h(t),\tilde{c}(t)\right\rangle_{L^2(\tilde{g}(t))}-\lambda_i(t)\langle \tilde{\oi}_i(t),\tilde{c}(t)\rangle_{L^2(\tilde{g}(t))}-\nu(t)\langle \tilde{g}(t),\tilde{c}(t)\rangle_{L^2(\tilde{g}(t))}\\
&=\left\langle \partial_t h(t),\tilde{c}(t)\right\rangle_{L^2(\tilde{g}(t))}-\left\langle h(t),\Delta_{L,\tilde{g}(t)}\tilde{c}(t)\right\rangle_{L^2(\tilde{g}(t))}-\lambda_i(t)\langle \tilde{\oi}_i(t),\tilde{c}(t)\rangle_{L^2(\tilde{g}(t))},
\end{split}
\end{equation*}
since $\tr_{\tilde{g}(t)}\tilde{c}(t)=\Delta_{\tilde{g}(t)}\tilde{v}(t)$ has zero mean value as observed in Proposition \ref{prop-proj-conf}.

According to Lemma \ref{est-basic-conf},  one has
\begin{equation*}
\begin{split}
&\left\langle h(t),\Delta_{L,\tilde{g}(t)}\tilde{c}(t)\right\rangle_{L^2(\tilde{g}(t))}=-(\Lambda+1)\Big\langle h(t),\tilde{c}(t)\Big\rangle_{L^2(\tilde{g}(t))}\\
&+\left\langle h(t),O(\rho(t)^{2})\mathbbm{1}_{\{r_o\,\leq\,2\delta(t)\}}+O(\varepsilon(t)^{5}\delta(t)^{-7})\mathbbm{1}_{\{\delta(t)/2\,\leq\,r_o\,\leq\,2\delta(t)\}}+O(\varepsilon(t)^4\rho(t)^{-4}) \right\rangle_{L^2(\tilde{g}(t))}.
\end{split}
\end{equation*}
In particular, since $h(t)$ is $L^2$-orthogonal to $\tilde{c}(t)$ by assumption, we get the following estimate:
\begin{equation*}
\left|\left\langle h(t),\Delta_{L,\tilde{g}(t)}\tilde{c}(t)\right\rangle_{L^2(\tilde{g}(t))}\right|\leq C\gamma(t)^{-1}\delta(t)^{4-\sigma}\|h\|_{C^{0}_{\gamma,\sigma,T}},\quad t\leq T,
\end{equation*}
for some uniform positive constant $C$.
 In order to handle the $L^2$-projection of $\tilde{c}(t)$ on $\tilde{\oi}_i(t)$, we invoke Proposition \ref{prop-proj-conf}:
\begin{equation*}
\begin{split}
\left|\langle \tilde{\oi}_i(t),\tilde{c}(t)\rangle_{L^2(\tilde{g}(t))}\right|
&\leq C\varepsilon(t)^4,
\end{split}
\end{equation*}
for some uniform positive constant $C$. To summarize our set of estimates, we have obtained, thanks to Lemma \ref{pre-partial-h-3}:
\begin{equation*}
\begin{split}
\left|\langle \psi(t)+\mu(t)\tilde{c}(t),\tilde{c}(t)\rangle_{L^2(\tilde{g}(t))}\right|&\leq C\gamma(t)^{-1}\left(\delta(t)^{4-\sigma}+\varepsilon(t)^4\delta(t)^{-4-\sigma}\right)\|h\|_{C^{0}_{\gamma,\sigma,T}}\\
&\quad+C\varepsilon(t)^4\sum_i|\lambda_i(t)|,
\end{split}
\end{equation*}
which leads directly to the desired first estimate since $\delta>5/9>1/2$ by \eqref{cond-param}.

As for the second estimate, Lemma \ref{lemma-easy-peasy-ter} applied to $\psi(t)$  leads to the following pointwise estimate:
\begin{equation*}
\begin{split}
\|\mu\cdot\tilde{c}\|_{C^0_{\gamma,\sigma+2,T}}&\leq o_T(1)\|h\|_{C^0_{\gamma,\sigma,T}}+C\|\psi\|_{C^0_{\gamma,\sigma+2,T}}+C\sup_{t\,\leq \,T}\gamma(t)\varepsilon(t)^4\sum_i|\lambda_i(t)|,
\end{split}
\end{equation*}
which implies the desired estimate by assumption on $\gamma(t)$ with respect to $\gamma(t)$.

\end{proof}

\subsubsection{Orthogonal projections: pointwise estimates}

By combining Propositions \ref{prop 6.5 E(t)}, \ref{prop 6.5 F(t)} and \ref{prop 6.5 M(t)}, we get by for $T$ small enough by absorption:
\begin{coro}\label{coro-intermed-est-coeff}
There exists a uniform positive constant $C$ and $T<0$ negative enough {(respectively large enough)} such that the following holds: if $h\in C^{2,\alpha}_{\gamma,\sigma,T}$ is such that $h(t)$ is $L^2(\tilde{g}(t))$-orthogonal to $\tilde{\mathbf{O}}(t)$ for each $t\leq T$ and satisfies for $t\leq T$,
 \begin{equation*}
  (\partial_t - \Delta_{L,\tilde{g}(t)}-2\Lambda)h(t) - \psi(t) \in \tilde{\mathbf{O}}(t),  
  \end{equation*}
  i.e. 
   \begin{equation*}
  (\partial_t - \Delta_{L,\tilde{g}(t)}-2\Lambda)h(t) = \psi(t) +\lambda_i(t)\tilde{\oi}_i(t)+\nu(t)\tilde{g}(t)+\mu(t)\tilde{c}(t),
  \end{equation*}
then,
\begin{equation*}
\begin{split}
\varepsilon(t)^4\sum_j|\lambda_j(t)|&\leq C\gamma(t)^{-1}\varepsilon(t)^4\delta(t)^{-2-\sigma}\|h\|_{C^0_{\gamma,\sigma,T}}+C\gamma(t)^{-1}\varepsilon(t)^{2-\sigma}\|\psi\|_{C^0_{\gamma,\sigma+2,T}},\\
|\nu(t)|&\leq C\gamma(t)^{-1}\varepsilon(t)^{4}\delta(t)^{-4-\sigma}\|h\|_{C^0_{\gamma,\sigma,T}}+C\gamma(t)^{-1}\|\psi\|_{C^0_{\gamma,\sigma+2,T}},\\
|\mu(t)|&\leq C\gamma(t)^{-1}\varepsilon(t)^{4}\delta(t)^{-4-\sigma}\|h\|_{C^0_{\gamma,\sigma,T}}+C\gamma(t)^{-1}\|\psi\|_{C^0_{\gamma,\sigma+2,T}},
\end{split}
\end{equation*}

and,
\begin{equation*}
\sum_j\|\lambda_j\cdot\tilde{\oi}_j\|_{C^0_{\gamma,\sigma+2,T}}+ \|\nu\cdot\tilde{g}\|_{C^0_{\gamma,\sigma+2,T}}+ \|\mu\cdot\tilde{c}\|_{C^0_{\gamma,\sigma+2,T}}\leq o_T(1)\|h\|_{C^0_{\gamma,\sigma,T}}+C\|\psi\|_{C^0_{\gamma,\sigma+2,T}}.
\end{equation*}
\end{coro}
\subsection{Orthogonal projections for abstract data: H\"older estimates}

The purpose of this section is to obtain corresponding H\"older estimates of the orthogonal projections considered in the previous Section \ref{Orthogonal projections for abstract data: pointwise estimates}. 
\subsubsection{Orthogonal projections against $\tilde{\oi}_i(t)$: H\"older estimates}

\begin{lemma}\label{gal-est-h-holder}
Let $h\in C^{0,\alpha}_{\gamma,\sigma,T}$. Then for $t\leq T$, 
\begin{equation*}
\left[\langle h,\tilde{\oi}_i\rangle_{L^2(\tilde{g})}\right]_{\alpha,t,r}\leq C\gamma(t)^{-1}r^{2\alpha}\varepsilon(t)^{4-\sigma-2\alpha}\|h\|_{C^{0,\alpha}_{\gamma,\sigma,T}}.
\end{equation*}
\end{lemma}

\begin{proof}
According to Lemma \ref{lemma-semi-norm-diff-scales} and Lemma \ref{lemma-scal-prod-holder}:
\begin{equation*}
\begin{split}
\left[\langle h,\tilde{\oi}_i\rangle_{L^2(\tilde{g})}\right]_{\alpha,t,r}&\leq C\gamma(t)^{-1}\left(r^{2\alpha}\varepsilon(t)^{4-\sigma-2\alpha}+r^2\varepsilon(t)^{4-\sigma}+r^2\varepsilon(t)^4\delta(t)^{-\sigma}\right)\|h\|_{C^{0,\alpha}_{\gamma,\sigma,T}}\\
&\leq C\gamma(t)^{-1}r^{2\alpha}\varepsilon(t)^{4-\sigma-2\alpha}\|h\|_{C^{0,\alpha}_{\gamma,\sigma,T}},
\end{split}
\end{equation*}
as expected. Here we have used Lemma \ref{estimates variations et laplacien o1} in the first inequality together with \eqref{semi-holder-diff-tens} applied to $\tilde{\oi}_i$.
\end{proof}

The next lemma is the H\"older counterpart of Lemma \ref{pre-partial-h}:
\begin{lemma}\label{pre-partial-h-holder}
Let $h\in C^2_{\gamma,\sigma,T}$ such that $h(t)$ is $L^2(\tilde{g}(t))$-orthogonal to $\tilde{\mathbf{O}}(t)$ for each $t\leq T$ (respectively for $t\geq T$). Then,
\begin{equation*}
\begin{split}
\left[\langle \partial_th,\tilde{\oi}_i\rangle_{L^2(\tilde{g})}\right]_{\alpha,t,r}&\leq  C\gamma(t)^{-1}\left(\left(\frac{r}{\varepsilon(t)}\right)^{2\alpha}+r^2+\frac{[\dot{\zeta}]_{\alpha,t,r}}{\varepsilon(t)^2}\right)\varepsilon(t)^{4-\sigma}\|h\|_{C^{0,\alpha}_{\gamma,\sigma,T}}.
\end{split}
\end{equation*}
\end{lemma}
\begin{proof}
The proof is the same for ancient and immortal flows.

Starting from the proof of Lemma \ref{pre-partial-h} and thanks to Lemma \ref{lemma-semi-norm-diff-scales} and Lemma \ref{lemma-scal-prod-holder}:
\begin{equation*}
\begin{split}
\left[\langle \partial_th,\tilde{\oi}_i\rangle_{L^2(\tilde{g})}\right]_{\alpha,t,r}&=\left[\Big\langle h(t),\Big(\partial_t+\frac{\tr_{\tilde{g}(t)}\partial_t\tilde{g}(t)}{2}\operatorname{Id}-\Sym(\partial_t\tilde{g}(t)\circ \cdot)\Big)\tilde{\oi}_i(t)\Big\rangle_{L^2(\tilde{g}(t))}\right]_{\alpha,t,r}\\
&\leq C\gamma(t)^{-1}\|h\|_{C^{0,\alpha}_{\gamma,\sigma,T}}\int_{\{r_o\,\leq\, 2\delta(t)\}}\frac{r^{2\alpha}\varepsilon(t)^{4}}{(\varepsilon(t)+r_o)^{4+\sigma+2\alpha}}\,d\mu_{\tilde{g}(t)}\\
&\quad+ C\gamma(t)^{-1}\|h\|_{C^{0,\alpha}_{\gamma,\sigma,T}}\max\{(r/\delta(t))^{2\alpha};1\}\varepsilon(t)^{4^-}\delta(t)^{-\sigma^-}\\
&\quad+C\gamma(t)^{-1}\|h\|_{C^{0}_{\gamma,\sigma,T}}\int_M\frac{1}{(\varepsilon(t)+r_o)^{\sigma}}\left([\partial_t\tilde{g}\ast \tilde{\oi}_i]_{\alpha,t,r}+[\partial_t \tilde{\oi}_i]_{\alpha,t,r}\right)\,d\mu_{\tilde{g}(t)}\\
&\leq C\gamma(t)^{-1}\|h\|_{C^{0,\alpha}_{\gamma,\sigma,T}}\left(r^{2\alpha}\varepsilon(t)^{4-\sigma-2\alpha}+\max\{(r/\delta(t))^{2\alpha};1\}\varepsilon(t)^{4^-}\delta(t)^{-\sigma^-}\right)\\
&\quad+C\gamma(t)^{-1}\left(r^2+\frac{[\dot{\zeta}]_{\alpha,t,r}}{\varepsilon(t)^2}\right)\|h\|_{C^{0}_{\gamma,\sigma,T}}\int_{\{r_o\,\leq\, 2\delta(t)\}}\frac{\varepsilon(t)^{4-}}{(\varepsilon(t)+r_o)^{4^{-}+\sigma}}\,d\mu_{\tilde{g}(t)}\\
&\quad+C\left(r^2+\frac{[\dot{\zeta}]_{\alpha,t,r}}{\varepsilon(t)^2}\right)\gamma(t)^{-1}\|h\|_{C^{0}_{\gamma,\sigma,T}}\varepsilon(t)^{8^-}\delta(t)^{-4^--\sigma^-}\\
&\leq C\gamma(t)^{-1}\left(\left(\frac{r}{\varepsilon(t)}\right)^{2\alpha}+r^2+\frac{[\dot{\zeta}]_{\alpha,t,r}}{\varepsilon(t)^2}\right)\varepsilon(t)^{4-\sigma}\|h\|_{C^{0,\alpha}_{\gamma,\sigma,T}}.
\end{split}
\end{equation*}
Here we have used Lemmata \ref{lemma-time-der-metric} and \ref{estimates variations et laplacien o1} in the second line together with Propositions \ref{prop-est-semi-norm-time-der-g} and \ref{prop-holder-time-app-ker-oi} in the third inequality. The choice of parameters from \eqref{cond-param} has also been invoked in the last line. The result follows by invoking \eqref{cond-param} once more.
\end{proof}

\begin{prop}\label{prop 6.5 E(t)-holder}
Let $\psi\in C^{0}_{\gamma,\sigma+2,T}$ and let $h\in C^{2,\alpha}_{\gamma,\sigma,T}$ such that $h(t)$ is $L^2(\tilde{g}(t))$-orthogonal to $\tilde{\mathbf{O}}(t)$ for each $t\leq T$ (respectively $t\geq T$ in the immortal case) and satisfying for all $t\leq T$ (respectively $t\geq T$ in the immortal case),
 \begin{equation*}
  (\partial_t - \Delta_{L,\tilde{g}(t)}-2\Lambda)h(t) - \psi(t) \in \tilde{\mathbf{O}}(t),  
  \end{equation*}
  i.e. 
   \begin{equation*}
  (\partial_t - \Delta_{L,\tilde{g}(t)}-2\Lambda)h(t) = \psi(t) +\lambda_i(t)\tilde{\oi}_i(t)+\nu(t)\tilde{g}(t)+\mu(t)\tilde{c}(t).  
  \end{equation*}
Then there exists a uniform positive constant $C$ such that for $t\leq T$ (respectively $t\geq T$ in the immortal case),
\begin{equation}
\begin{split}\label{why-is-it-so-painful-holder}
&\left[\langle \psi+\lambda_i\tilde{\oi}_i,\tilde{\oi}_i\rangle_{L^2(\tilde{g})}\right]_{\alpha,t,r}\leq C\varepsilon(t)^4\delta(t)^2[\nu]_{\alpha,t,r}+C\varepsilon(t)^4[\mu]_{\alpha,t,r}\\
&\quad+C\gamma(t)^{-1}\left(r^{2\alpha}\varepsilon(t)^4\delta(t)^{-2-\sigma-2\alpha}+r^2\varepsilon(t)^{2-\sigma}\right)\|h\|_{C^{0,\alpha}_{\gamma,\sigma,T}}\\
&\quad+C\gamma(t)^{-1}\varepsilon(t)^{6-\sigma}\delta(t)^{-4}\left(1+r^{2\alpha}\varepsilon(t)^{-2\alpha}\right)\|\psi\|_{C^{0,\alpha}_{\gamma,\sigma+2,T}}.
\end{split}
\end{equation}
and,
\begin{equation*}
\begin{split}
\sum_i[\lambda_i]_{\alpha,t,r}&\leq C\delta(t)^2[\nu]_{\alpha,t,r}+C[\mu]_{\alpha,t,r}\\
&\quad+C\gamma(t)^{-1}\left(\left(1+r^{2\alpha}\delta(t)^{-2\alpha}\right)\delta(t)^{-2-\sigma}+r^2\varepsilon(t)^{-2-\sigma}\right)\|h\|_{C^{0,\alpha}_{\gamma,\sigma,T}}\\&\quad+C\gamma(t)^{-1}\left(1+r^{2\alpha}\varepsilon(t)^{-2\alpha}\right)\varepsilon(t)^{-2-\sigma}\|\psi\|_{C^{0,\alpha}_{\gamma,\sigma+2,T}}.
\end{split}
\end{equation*}

In particular,  there exists some uniform positive constant $C$ such that:
\begin{equation*}
\begin{split}
\|\lambda_i\cdot\tilde{\oi}_i\|_{C^{0,\alpha}_{\gamma,\sigma+2,T}}&\leq o_T(1)\|h\|_{C^{0,\alpha}_{\gamma,\sigma,T}}+C\|\psi\|_{C^{0,\alpha}_{\gamma,\sigma+2,T}} +C\sup_{t\,\leq \,T}\gamma(t)\varepsilon(t)^{4}\delta(t)^2\sup_{\rho\,\leq\, 2\delta(t)}\rho^{-2+\sigma}[\nu]_{\alpha,t,\rho}\\
&\quad+C\sup_{t\,\leq \,T}\gamma(t)\varepsilon(t)^{4}\sup_{\rho\,\leq\, 2\delta(t)}\rho^{-2+\sigma}[\mu]_{\alpha,t,\rho}.
\end{split}
\end{equation*}

\end{prop}

\begin{proof}
 We provide a proof for ancient flows, but the statement also holds for immortal flows by flipping the sign of time. Indeed, the controls are always made on $\partial_t$ and $\Delta_{L,\tilde{g}(t)}$ separately so looking at $\partial_t+\Delta_{L,\tilde{g}(t)}$ or $\partial_t-\Delta_{L,\tilde{g}(t)}$ yields the same estimate.

 Since $h(t)$ is orthogonal to $\tilde{\mathbf{O}}(t)$ for each $t\leq T$, a spatial integration by parts yields:
\begin{equation}
\begin{split}\label{prelim-proj-est-oi-semi-norm}
\left[\langle \psi+\lambda_i\tilde{\oi}_i,\tilde{\oi}_i\rangle_{L^2(\tilde{g})}\right]_{\alpha,t,r}&\leq \left[\langle \partial_th,\tilde{\oi}_i\rangle_{L^2(\tilde{g})}\right]_{\alpha,t,r}+\left[\langle h,\Delta_{L,\tilde{g}}\tilde{\oi}_i\rangle_{L^2(\tilde{g})}\right]_{\alpha,t,r}+\left[\nu\cdot\langle \tilde{\oi}_i,\tilde{g}\rangle_{L^2(\tilde{g})}\right]_{\alpha,t,r}\\
&\quad+\left[\mu\cdot\langle \tilde{\oi}_i,\tilde{c}\rangle_{L^2(\tilde{g})}\right]_{\alpha,t,r}+\sum_{j\neq i}\left[\lambda_j\cdot\langle \tilde{\oi}_i,\tilde{\oi}_j\rangle_{L^2(\tilde{g})}\right]_{\alpha,t,r}.
\end{split}
\end{equation}
Now, Lemma \ref{pre-partial-h-holder} and \eqref{cond-param} give:
\begin{equation}
\begin{split}\label{prelim-proj-est-oi-semi-norm-1}
\left[\langle \partial_th,\tilde{\oi}_i\rangle_{L^2(\tilde{g})}\right]_{\alpha,t,r}&\leq  C\gamma(t)^{-1}\left(\left(\frac{r}{\varepsilon(t)}\right)^{2\alpha}+r^2+\frac{[\dot{\zeta}]_{\alpha,t,r}}{\varepsilon(t)^2}\right)\varepsilon(t)^{4-\sigma}\|h\|_{C^{0,\alpha}_{\gamma,\sigma,T}}\\
&\leq C\gamma(t)^{-1}\left(\left(\frac{r}{\varepsilon(t)}\right)^{2\alpha}+1\right)\varepsilon(t)^{4-\sigma}\|h\|_{C^{0,\alpha}_{\gamma,\sigma,T}},
\end{split}
\end{equation}
and Lemma \ref{estimates variations et laplacien o1}, Proposition \ref{prop-semi-norm-lap-app-ker} and Lemma \ref{lemma-scal-prod-holder} give in turn:
\begin{equation}
\begin{split}\label{prelim-proj-est-oi-semi-norm-2}
\left[\langle h,\Delta_{L,\tilde{g}}\tilde{\oi}_i\rangle_{L^2(\tilde{g})}\right]_{\alpha,t,r}&\leq C\gamma(t)^{-1}\|h\|_{C^{0,\alpha}_{\gamma,\sigma,T}}\left(r^{2\alpha}\varepsilon(t)^{4-2\alpha-\sigma}+r^{2\alpha}\varepsilon(t)^4\delta(t)^{-2-\sigma-2\alpha}+\varepsilon(t)^{4-\sigma}\right)\\
&\quad+C\gamma(t)^{-1}\|h\|_{C^{0,\alpha}_{\gamma,\sigma,T}}\left(r^2\varepsilon(t)^{2-\sigma}+r^2\varepsilon(t)^{4}\delta(t)^{-2-\sigma}\right)\\
&\leq C\gamma(t)^{-1}\left(\left(1+r^{2\alpha}\delta(t)^{-2\alpha}\right)\varepsilon(t)^4\delta(t)^{-2-\sigma}+r^2\varepsilon(t)^{2-\sigma}\right)\|h\|_{C^{0,\alpha}_{\gamma,\sigma,T}},
\end{split}
\end{equation}
where $C$ is a time-independent positive constant. 
Here we have invoked \eqref{cond-param} to ensure that $\delta>(\sigma+2\alpha)/(2+\sigma+2\alpha)$. Similarly, repeated applications of Lemma \ref{lemma-scal-prod-holder} gives:

\begin{equation}
\begin{split}\label{prelim-proj-est-oi-semi-norm-3}
\left[\nu\cdot\langle \tilde{\oi}_i,\tilde{g}\rangle_{L^2(\tilde{g})}\right]_{\alpha,t,r}&\leq C\varepsilon(t)^4\delta(t)^2[\nu]_{\alpha,t,r}\\
&\quad+Cr^2\varepsilon(t)^{4^-}\delta(t)^{2^+}\left(\gamma(t)^{-1}\varepsilon(t)^{4}\delta(t)^{-4-\sigma}\|h\|_{C^{0,\alpha}_{\gamma,\sigma,T}}+\gamma(t)^{-1}\|\psi\|_{C^{0,\alpha}_{\gamma,\sigma+2,T}}\right).
\end{split}
\end{equation}
Here we have used Proposition \ref{prop-o_1-first-app}, Proposition \ref{prop-easy-l2-proj-holder} and Corollary \ref{coro-intermed-est-coeff}.

A similar reasoning based on Proposition \ref{prop-proj-conf}, Proposition \ref{prop-easy-l2-proj-holder} and Corollary \ref{coro-intermed-est-coeff} gives:
\begin{equation}
\begin{split}\label{prelim-proj-est-oi-semi-norm-4}
\left[\mu\cdot\langle \tilde{\oi}_i,\tilde{c}\rangle_{L^2(\tilde{g})}\right]_{\alpha,t,r}&\leq C\varepsilon(t)^4[\mu]_{\alpha,t,r}\\
&\quad+Cr^2\varepsilon(t)^{4}\left(\gamma(t)^{-1}\varepsilon(t)^4\delta(t)^{-4-\sigma}\|h\|_{C^{0,\alpha}_{\gamma,\sigma,T}}+\gamma(t)^{-1}\|\psi\|_{C^{0,\alpha}_{\gamma,\sigma+2,T}}\right),
\end{split}
\end{equation}
and for $i\neq j$,
\begin{equation}
\begin{split}\label{prelim-proj-est-oi-semi-norm-5}
&\left[\lambda_j\cdot\langle \tilde{\oi}_i,\tilde{\oi}_j\rangle_{L^2(\tilde{g})}\right]_{\alpha,t,r}\leq C\varepsilon(t)^8\delta(t)^{-4}[\lambda_j]_{\alpha,t,r}\\
&\quad+Cr^2\varepsilon(t)^{8}\delta(t)^{-4}\left(\gamma(t)^{-1}\delta(t)^{-2-\sigma}\|h\|_{C^{0,\alpha}_{\gamma,\sigma,T}}+C\gamma(t)^{-1}\varepsilon(t)^{-2-\sigma}\|\psi\|_{C^{0,\alpha}_{\gamma,\sigma+2,T}}\right).
\end{split}
\end{equation}
Here we have used Proposition \ref{prop-o_1-first-app} and \eqref{cond-param} together with Proposition \ref{prop-easy-l2-proj-holder} and Corollary \ref{coro-intermed-est-coeff}.

The combination of estimates \eqref{prelim-proj-est-oi-semi-norm-1}, \eqref{prelim-proj-est-oi-semi-norm-2}, \eqref{prelim-proj-est-oi-semi-norm-3}, \eqref{prelim-proj-est-oi-semi-norm-4}, \eqref{prelim-proj-est-oi-semi-norm-5}  gives thanks to \eqref{prelim-proj-est-oi-semi-norm} and \eqref{cond-param}:

\begin{equation}
\begin{split}\label{prelim-proj-est-oi-semi-norm-6}
&\left[\langle \psi+\lambda_i\tilde{\oi}_i,\tilde{\oi}_i\rangle_{L^2(\tilde{g})}\right]_{\alpha,t,r}\leq C\varepsilon(t)^8\delta(t)^{-4}\sum_{j\neq i}[\lambda_j]_{\alpha,t,r}+C\varepsilon(t)^4\delta(t)^2[\nu]_{\alpha,t,r}+C\varepsilon(t)^4[\mu]_{\alpha,t,r}\\
&+C\gamma(t)^{-1}\left(\left(1+r^{2\alpha}\delta(t)^{-2\alpha}\right)\varepsilon(t)^4\delta(t)^{-2-\sigma}+r^2\varepsilon(t)^{2-\sigma}\right)\|h\|_{C^{0,\alpha}_{\gamma,\sigma,T}}\\&+Cr^2\gamma(t)^{-1}\varepsilon(t)^{6-\sigma}\delta(t)^{-4}\|\psi\|_{C^{0,\alpha}_{\gamma,\sigma+2,T}}.
\end{split}
\end{equation}
Now,  Lemma \ref{lemma-semi-norm-diff-scales} and Lemma \ref{lemma-scal-prod-holder} applied to $\psi(t)$ and $\tilde{\oi}_i(t)$ lead to: 
\begin{equation*}
\begin{split}
[\langle \psi,\tilde{\oi}_i\rangle_{L^2(\tilde{g})}]_{\alpha,t,r}&\leq C(1+r^{2\alpha}\varepsilon(t)^{-2\alpha})\gamma(t)^{-1}\varepsilon(t)^{2-\sigma}\|\psi\|_{C^{0,\alpha}_{\gamma,\sigma+2,T}}.
\end{split}
\end{equation*}
Here we have used \eqref{semi-holder-oi}.

By summing \eqref{prelim-proj-est-oi-semi-norm-6} over $i$ and by absorption, the previous estimate combined with \eqref{prelim-proj-est-oi-semi-norm-6} gives in turn:
\begin{equation}
\begin{split}\label{prelim-proj-est-oi-semi-norm-7}
\sum_i[\lambda_i]_{\alpha,t,r}&\leq C\delta(t)^2[\nu]_{\alpha,t,r}+C[\mu]_{\alpha,t,r}\\
&\quad+C\gamma(t)^{-1}\left(\left(1+r^{2\alpha}\delta(t)^{-2\alpha}\right)\delta(t)^{-2-\sigma}+r^2\varepsilon(t)^{-2-\sigma}\right)\|h\|_{C^{0,\alpha}_{\gamma,\sigma,T}}\\&\quad+C\gamma(t)^{-1}\left(1+r^{2\alpha}\varepsilon(t)^{-2\alpha}\right)\varepsilon(t)^{-2-\sigma}\|\psi\|_{C^{0,\alpha}_{\gamma,\sigma+2,T}},
\end{split}
\end{equation}
thanks to \eqref{cond-param}
 and the fact that $\|\tilde{\oi}_i(t)\|^2_{L^2(\tilde{g}(t))}$ is proportional to $\varepsilon(t)^4$ thanks to Proposition \ref{prop-o_1-first-app}. Here we have invoked Corollary \ref{coro-intermed-est-coeff} together with Proposition \ref{prop-easy-l2-proj-holder} to handle the terms involving $|\lambda_i|\Big[\|\tilde{\oi}_i(t)\|^2_{L^2(\tilde{g}(t))}\Big]_{\alpha,t,r}$.

By inserting \eqref{prelim-proj-est-oi-semi-norm-7} back to the righthand side of \eqref{prelim-proj-est-oi-semi-norm-6}, one gets the first half of the first desired estimates:
\begin{equation*}
\begin{split}\label{prelim-proj-est-oi-semi-norm-8}
&\left[\langle \psi+\lambda_i\tilde{\oi}_i,\tilde{\oi}_i\rangle_{L^2(\tilde{g})}\right]_{\alpha,t,r}\leq C\varepsilon(t)^4\delta(t)^2[\nu]_{\alpha,t,r}+C\varepsilon(t)^4[\mu]_{\alpha,t,r}\\
&\quad+C\gamma(t)^{-1}\left(\left(1+r^{2\alpha}\delta(t)^{-2\alpha}\right)\varepsilon(t)^4\delta(t)^{-2-\sigma}+r^2\varepsilon(t)^{2-\sigma}\right)\|h\|_{C^{0,\alpha}_{\gamma,\sigma,T}}\\
&\quad+C\gamma(t)^{-1}\varepsilon(t)^{6-\sigma}\delta(t)^{-4}\left(1+r^{2\alpha}\varepsilon(t)^{-2\alpha}\right)\|\psi\|_{C^{0,\alpha}_{\gamma,\sigma+2,T}}.
\end{split}
\end{equation*}

Now, since the tensor $\tilde{\oi}_i(t)$ is supported in $\{r_o\,\leq\,2\delta(t)\}$, one also gets the following pointwise estimate from Propositions \ref{prop-holder-semi-norm-app-ker} and \ref{prop 6.5 E(t)}, Corollary \ref{coro-intermed-est-coeff} and \eqref{prelim-proj-est-oi-semi-norm-7}:
\begin{equation*}
\begin{split}
\|\lambda_i\cdot\tilde{\oi}_i\|_{C^{0,\alpha}_{\gamma,\sigma+2,T}}&\leq \|\lambda_i\cdot\tilde{\oi}_i\|_{C^{0}_{\gamma,\sigma+2,T}}+C\sup_{t\,\leq\,T}\gamma(t)\varepsilon(t)^{2+\sigma}\sup_{t-4\delta(t)^2\,\leq \,t'\,\leq \,t}|\lambda_i(t')|\\
&\quad+C\sup_{t\,\leq \,T}\gamma(t)\varepsilon(t)^{4}\sup_{\rho\,\leq\,2\delta(t)}\rho^{-2+\sigma}[\lambda_i]_{\alpha,t,\rho}\\
&\leq o_T(1)\|h\|_{C^0_{\gamma,\sigma,T}}+C\|\psi\|_{C^{0}_{\gamma,\sigma+2,T}}\\
&\quad+C\sup_{t\,\leq \,T}\gamma(t)\varepsilon(t)^{4}\delta(t)^2\sup_{\rho\,\leq\,2\delta(t)}\rho^{-2+\sigma}[\nu]_{\alpha,t,\rho}+C\sup_{t\,\leq \,T}\gamma(t)\varepsilon(t)^{4}\sup_{\rho\,\leq\,2\delta(t)}\rho^{-2+\sigma}[\mu]_{\alpha,t,\rho}\\
&\quad+C\sup_{t\,\leq \,T}\left((\varepsilon(t)/\delta(t))^{2+\sigma}+(\delta(t)/\varepsilon(t))^{\sigma}\varepsilon(t)^2\right)\|h\|_{C^{0,\alpha}_{\gamma,\sigma,T}}+C\|\psi\|_{C^{0,\alpha}_{\gamma,\sigma+2,T}}\\
&\leq o_T(1)\|h\|_{C^{0,\alpha}_{\gamma,\sigma,T}}+C\|\psi\|_{C^{0,\alpha}_{\gamma,\sigma+2,T}}\\
&\quad+C\sup_{t\,\leq \,T}\gamma(t)\varepsilon(t)^{4}\delta(t)^2\sup_{\rho\,\leq\,2\delta(t)}\rho^{-2+\sigma}[\nu]_{\alpha,t,\rho}+C\sup_{t\,\leq \,T}\gamma(t)\varepsilon(t)^{4}\sup_{\rho\,\leq\,2\delta(t)}\rho^{-2+\sigma}[\mu]_{\alpha,t,\rho},
\end{split}
\end{equation*}
where $C$ is a time-independent positive constant that may vary from line to line. 

\end{proof}
\subsubsection{Orthogonal projections against $\tilde{g}(t)$ and $\tilde{c}(t)$: H\"older estimates}
\begin{lemma}\label{gal-est-h-holder-g-c}
Let $h\in C^{0,\alpha}_{\gamma,\sigma,T}$. Then for $t\leq T$, 
\begin{equation*}
\begin{split}
\left[\langle h,\tilde{g}\rangle_{L^2(\tilde{g})}\right]_{\alpha,t,r}+\left[\langle h,\tilde{c}\rangle_{L^2(\tilde{g})}\right]_{\alpha,t,r}&\leq C\gamma(t)^{-1}\|h\|_{C^{0,\alpha}_{\gamma,\sigma,T}}.
\end{split}
\end{equation*}
\end{lemma}

\begin{proof}
According to Lemma \ref{lemma-semi-norm-diff-scales} and Lemma \ref{lemma-scal-prod-holder}:
\begin{equation*}
\begin{split}
\left[\langle h,\tilde{g}\rangle_{L^2(\tilde{g})}\right]_{\alpha,t,r}&\leq C\gamma(t)^{-1}\left(1+r^{4-\sigma}+r^2\varepsilon(t)^{4-\sigma}+r^2\varepsilon(t)^4\delta(t)^{-\sigma}\right)\|h\|_{C^{0,\alpha}_{\gamma,\sigma,T}}\\
&\leq C\gamma(t)^{-1}\|h\|_{C^{0,\alpha}_{\gamma,\sigma,T}},
\end{split}
\end{equation*}
as expected. Here we have used Lemma \ref{lemma-time-der-metric} (respectively Lemma \ref{est-basic-conf}
) in the first (respectively second) inequality together with \eqref{semi-holder-diff-tens} applied to $\tilde{g}$. 
A similar proof works for $\tilde{c}$.
\end{proof}

\begin{lemma}\label{pre-partial-h-2-holder}
Let $h\in C^2_{\gamma,\sigma,T}$ such that $h(t)$ is $L^2(\tilde{g}(t))$-orthogonal to $\tilde{\mathbf{O}}(t)$ for each $t\leq T$ (respectively for $t\geq T$). Then,
{
\begin{equation*}
\begin{split}
\left[\langle \partial_th,\tilde{g}\rangle_{L^2(\tilde{g})}\right]_{\alpha,t,r}&\leq C\gamma(t)^{-1}\left(\left(\frac{r}{\delta(t)}\right)^{2\alpha}+r^2+\frac{[\dot{\zeta}]_{\alpha,t,r}}{\varepsilon(t)^2}\right)\varepsilon(t)^{4}\delta(t)^{-4-\sigma}\|h\|_{C^{0,\alpha}_{\gamma,\sigma,T}},\\
\left[\langle \partial_th,\tilde{c}\rangle_{L^2(\tilde{g})}\right]_{\alpha,t,r}&\leq C\gamma(t)^{-1}\left(\left(\frac{r}{\delta(t)}\right)^{2\alpha}+r^2+\frac{[\dot{\zeta}]_{\alpha,t,r}}{\varepsilon(t)^2}\right)\varepsilon(t)^{4}\delta(t)^{-4-\sigma}\|h\|_{C^{0,\alpha}_{\gamma,\sigma,T}}.
\end{split}
\end{equation*}}

\end{lemma}
The proof of Lemma \ref{pre-partial-h-2-holder} goes along the same lines as those of the proof of Lemma \ref{pre-partial-h-holder}.

\begin{prop}\label{prop 6.5 F(t)-holder}
Let $\psi\in C^{0,\alpha}_{\gamma,\sigma+2,T}$ and $h\in C^{2,\alpha}_{\gamma,\sigma,T}$ such that $h(t)$ is $L^2(\tilde{g}(t))$-orthogonal to $\tilde{\mathbf{O}}(t)$ for each $t\leq T$ (respectively $t\geq T$ in the immortal case) and satisfying for all $t\leq T$ (respectively $t\geq T$ in the immortal case),
 \begin{equation*}
  (\partial_t - \Delta_{L,\tilde{g}(t)}-2\Lambda)h(t) - \psi(t) \in \tilde{\mathbf{O}}(t),  
  \end{equation*}
  i.e. 
   \begin{equation*}
  (\partial_t - \Delta_{L,\tilde{g}(t)}-2\Lambda)h(t) = \psi(t) +\lambda_i(t)\tilde{\oi}_i(t)+\nu(t)\tilde{g}(t)+\mu(t)\tilde{c}(t).  
  \end{equation*}
Then there exists a uniform positive constant $C$ such that for $t\leq T$ {(respectively if $t\geq T$ in the immortal case)},
{
\begin{equation*}
\begin{split}
&\left[\langle \psi+\lambda_i\tilde{\oi}_i+\nu\tilde{g},\tilde{g}\rangle_{L^2(\tilde{g})}\right]_{\alpha,t,r} \leq C\gamma(t)^{-1}\left(r^2+\left(\frac{r}{\varepsilon(t)}\right)^{2\alpha}\right)\varepsilon(t)^{4}\delta(t)^{-4-\sigma}\|h\|_{C^{0,\alpha}_{\gamma,\sigma,T}},\quad\text{and,}\\
\Big[\langle \psi+&\nu\tilde{g},\tilde{g}\rangle_{L^2(\tilde{g})}\Big]_{\alpha,t,r}\leq C\gamma(t)^{-1}\left(r^2+\left(\frac{r}{\varepsilon(t)}\right)^{2\alpha}\right)\varepsilon(t)^{4}\delta(t)^{-4-\sigma}\|h\|_{C^{0,\alpha}_{\gamma,\sigma,T}}\\
&\quad+C\gamma(t)^{-1}r^2\varepsilon(t)^{2^+-\sigma}\delta(t)^{2^-}\|\psi\|_{C^{0}_{\gamma,\sigma+2,T}}+C\varepsilon(t)^{4}\delta(t)^2\sum_i[\lambda_i]_{\alpha,t,r},
\end{split}
\end{equation*}
and,
\begin{equation*}
\begin{split}
\left[ \nu\right]_{\alpha,t,r}&\leq C\gamma(t)^{-1}\left(1+\left(\frac{r}{\varepsilon(t)}\right)^{2\alpha}\right)\varepsilon(t)^{4}\delta(t)^{-4-\sigma}\|h\|_{C^{0,\alpha}_{\gamma,\sigma,T}}\\
&\quad+C\gamma(t)^{-1}\|\psi\|_{C^{0}_{\gamma,\sigma+2,T}}+C\varepsilon(t)^{4}\sum_i[\lambda_i]_{\alpha,t,r}.
\end{split}
\end{equation*}

In particular, 
there exists some uniform positive constant $C$ such that:
\begin{equation*}
\begin{split}
\|\nu\cdot\tilde{g}\|_{C^{0,\alpha}_{\gamma,\sigma+2,T}}&\leq  o_T(1)\|h\|_{C^{0,\alpha}_{\gamma,\sigma,T}}+C\|\psi\|_{C^{0,\alpha}_{\gamma,\sigma+2,T}}+C\sup_{t\,\leq \,T}\gamma(t)\varepsilon(t)^{4}\sup_M\rho^{2+\sigma}\sum_i[\lambda_i]_{\alpha,t,\rho}.
\end{split}
\end{equation*}}

\end{prop}

\begin{proof}
 Since $\Delta_{L,\tilde{g}(t)}\tilde{g}(t)=0$, the proofs are exactly the same for the ancient and immortal situations.
 
  By differentiating the orthogonality condition $\langle h(t),\tilde{g}(t)\rangle_{L^2(\tilde{g}(t))}=0$ and Lemma \ref{easy-time-der-abs-2-tensors}, one gets:
\begin{equation*}
\begin{split}
\langle \psi(t)+&\lambda_i(t)\tilde{\oi}_i(t)+\nu(t)\tilde{g}(t)+\mu(t)\tilde{c}(t),\tilde{g}(t)\rangle_{L^2(\tilde{g}(t))}=\\
&\left\langle (\partial_t-\Delta_{L,\tilde{g}(t)})h(t),\tilde{g}(t)\right\rangle_{L^2(\tilde{g}(t))}=\left\langle \partial_th(t),\tilde{g}(t)\right\rangle_{L^2(\tilde{g}(t))},
\end{split}
\end{equation*}
since $\Delta_{L,\tilde{g}(t)}\tilde{g}(t)=0$. In particular, since $\langle \tilde{c}(t),\tilde{g}(t)\rangle_{L^2(\tilde{g}(t))}=0$ by Proposition \ref{prop-proj-conf}, Lemma \ref{pre-partial-h-2-holder} ensures that:
\begin{equation*}
\begin{split}
&\left[\langle \psi+\lambda_i\cdot\tilde{\oi}_i+\nu\cdot\tilde{g},\tilde{g}\rangle_{L^2(\tilde{g})}\right]_{\alpha,t,r}\leq C\gamma(t)^{-1}\left(\left(\frac{r}{\delta(t)}\right)^{2\alpha}+r^2+\frac{[\dot{\zeta}]_{\alpha,t,r}}{\varepsilon(t)^2}\right)\varepsilon(t)^{4}\delta(t)^{-4-\sigma}\|h\|_{C^{0,\alpha}_{\gamma,\sigma,T}}.
\end{split}
\end{equation*}

The triangular inequality leads to:
\begin{equation}
\begin{split}\label{g-tilde-psi-hold}
&\left[\langle \psi+\nu\cdot\tilde{g},\tilde{g}\rangle_{L^2(\tilde{g})}\right]_{\alpha,t,r}\leq\\
& \leq C\gamma(t)^{-1}\left(\left(\frac{r}{\delta(t)}\right)^{2\alpha}+r^2+\frac{[\dot{\zeta}]_{\alpha,t,r}}{\varepsilon(t)^2}\right)\varepsilon(t)^{4}\delta(t)^{-4-\sigma}\|h\|_{C^{0,\alpha}_{\gamma,\sigma,T}}+\sum_i\left[\lambda_i\langle \tilde{\oi}_i,\tilde{g}\rangle_{L^2(\tilde{g})}\right]_{\alpha,t,r}\\
&\leq C\gamma(t)^{-1}\left(\left(\frac{r}{\delta(t)}\right)^{2\alpha}+r^2+\frac{[\dot{\zeta}]_{\alpha,t,r}}{\varepsilon(t)^2}\right)\varepsilon(t)^{4}\delta(t)^{-4-\sigma}\|h\|_{C^{0,\alpha}_{\gamma,\sigma,T}}\\
&\quad+Cr^2\varepsilon(t)^{4^+}\delta(t)^{2^-}\sum_i|\lambda_i(t)|+C\varepsilon(t)^{4}\delta(t)^2\sum_i[\lambda_i]_{\alpha,t,r}\\
&\leq C\gamma(t)^{-1}\left(\left(\frac{r}{\delta(t)}\right)^{2\alpha}+r^2+\frac{[\dot{\zeta}]_{\alpha,t,r}}{\varepsilon(t)^2}\right)\varepsilon(t)^{4}\delta(t)^{-4-\sigma}\|h\|_{C^{0,\alpha}_{\gamma,\sigma,T}}\\
&\quad+C\gamma(t)^{-1}r^2\varepsilon(t)^{2^+-\sigma}\delta(t)^{2^-}\|\psi\|_{C^{0}_{\gamma,\sigma+2,T}}+C\varepsilon(t)^{4}\delta(t)^2\sum_i[\lambda_i]_{\alpha,t,r},
\end{split}
\end{equation}
as claimed once one observes that $\frac{[\dot{\zeta}]_{\alpha,t,r}}{\varepsilon(t)^2}\leq C(r/\varepsilon(t))^{2\alpha}$ thanks to \eqref{cond-param} and Lemma \ref{lemma-semi-norm-diff-scales}. Here we have used Proposition \ref{prop-o_1-first-app}, Proposition \ref{prop-easy-l2-proj-holder} in the penultimate line and Corollary \ref{coro-intermed-est-coeff} in the last line. 

As for the second estimate, Lemma \ref{lemma-semi-norm-diff-scales} and Lemma \ref{lemma-scal-prod-holder} applied to $\psi(t)$ and $\tilde{g}(t)$ lead to:  

\begin{equation*}
\left[\langle \psi,\tilde{g}\rangle_{L^2(\tilde{g})}\right]_{\alpha,t,r}\leq C\gamma(t)^{-1}\|\psi\|_{C^{0,\alpha}_{\gamma,\sigma+2,T}}.
\end{equation*}
In particular, one gets the following intermediate estimate on the basis of \eqref{g-tilde-psi-hold} and the fact that $\frac{[\dot{\zeta}]_{\alpha,t,r}}{\varepsilon(t)^2}\leq C(r/\varepsilon(t))^{2\alpha}$ thanks to \eqref{cond-param} and Lemma \ref{lemma-semi-norm-diff-scales}:
\begin{equation*}
\begin{split}
\left[ \nu\right]_{\alpha,t,r}&\leq C\gamma(t)^{-1}\left(1+\left(\frac{r}{\varepsilon(t)}\right)^{2\alpha}\right)\varepsilon(t)^{4}\delta(t)^{-4-\sigma}\|h\|_{C^{0,\alpha}_{\gamma,\sigma,T}}\\
&\quad+C\gamma(t)^{-1}\left(1+r^2\varepsilon(t)^{2^+-\sigma}\delta(t)^{2^-}\right)\|\psi\|_{C^{0}_{\gamma,\sigma+2,T}}\\&\quad+C\varepsilon(t)^{4}\delta(t)^2\sum_i[\lambda_i]_{\alpha,t,r}.
\end{split}
\end{equation*}
Here we have used the fact that $\|\tilde{g}(t)\|^2_{L^2(\tilde{g}(t))}$ is uniformly bounded from below. We have also invoked Corollary \ref{coro-intermed-est-coeff} together with Proposition \ref{prop-easy-l2-proj-holder} to handle the terms involving $|\nu|\Big[\|\tilde{g}\|^2_{L^2(\tilde{g})}\Big]_{\alpha,t,r}$. This leads to the third estimate.

Then one gets the following pointwise estimate thanks to Proposition \ref{prop-holder-semi-norm-app-ker} and Corollary \ref{coro-intermed-est-coeff}:
\begin{equation*}
\begin{split}
\|\nu\cdot\tilde{g}\|_{C^{0,\alpha}_{\gamma,\sigma+2,T}}&\leq \|\nu\cdot\tilde{g}\|_{C^{0}_{\gamma,\sigma+2,T}}+C\sup_{t\,\leq\,T}\gamma(t)\varepsilon(t)^{4}\sup_{t-c_0\,\leq \,t'\,\leq \,t}|\nu(t')|\\
&\quad+C\sup_{t\,\leq \,T}\gamma(t)\sup_M\rho^{2+\sigma}[\nu]_{\alpha,t,\rho}\\
&\leq o_T(1)\|h\|_{C^0_{\gamma,\sigma,T}}+C\|\psi\|_{C^0_{\gamma,\sigma+2,T}}+C\sup_{t\,\leq \,T}\gamma(t)\sup_M\rho^{2+\sigma}[\nu]_{\alpha,t,\rho}\\
&\leq \left(o_T(1)+C\sup_{t\,\leq\, T}\left(\left(\frac{1}{\varepsilon(t)}\right)^{2\alpha}+1\right)\varepsilon(t)^{4}\delta(t)^{-4-\sigma}\right)\|h\|_{C^{0,\alpha}_{\gamma,\sigma,T}}+C\|\psi\|_{C^{0,\alpha}_{\gamma,\sigma+2,T}}\\
&\quad+C\sup_{t\,\leq \,T}\gamma(t)\varepsilon(t)^{4}\sup_M\rho^{2+\sigma}\sum_i[\lambda_i]_{\alpha,t,\rho}\\
&\leq o_T(1)\|h\|_{C^{0,\alpha}_{\gamma,\sigma,T}}+C\|\psi\|_{C^{0,\alpha}_{\gamma,\sigma+2,T}}+C\sup_{t\,\leq \,T}\gamma(t)\varepsilon(t)^{4}\sup_M\rho^{2+\sigma}\sum_i[\lambda_i]_{\alpha,t,\rho},
\end{split}
\end{equation*}
since $4-2\alpha>(4+\sigma)\delta>0$ according to \eqref{cond-param} and $\alpha\in(0,1/2)$.

This implies the desired last estimate. 
\end{proof}
The next proposition is the counterpart of the previous Proposition \ref{prop 6.5 F(t)-holder} where the metric $\tilde{g}$ is replaced by $\tilde{c}$.

\begin{prop}
\label{prop 6.5 M(t)-holder}
Let $\psi\in C^{0,\alpha}_{\gamma,\sigma+2,T}$ and $h\in C^{2,\alpha}_{\gamma,\sigma,T}$ such that $h(t)$ is $L^2(\tilde{g}(t))$-orthogonal to $\tilde{\mathbf{O}}(t)$ for each $t\leq T$ (respectively $t\geq T$ in the immortal case) and satisfying for all $t\leq T$ (respectively $t\geq T$ in the immortal case),
 \begin{equation*}
  (\partial_t - \Delta_{L,\tilde{g}(t)}-2\Lambda)h(t) - \psi(t) \in \tilde{\mathbf{O}}(t),  
  \end{equation*}
  i.e. 
   \begin{equation*}
  (\partial_t - \Delta_{L,\tilde{g}(t)}-2\Lambda)h(t) = \psi(t) +\lambda_i(t)\tilde{\oi}_i(t)+\nu(t)\tilde{g}(t)+\mu(t)\tilde{c}(t).  
  \end{equation*}
Then there exists a uniform positive constant $C$ such that for $t\leq T\leq 0$,
{
\begin{equation*}
\begin{split}
\left[\langle \psi+\mu\cdot\tilde{c},\tilde{c}\rangle_{L^2(\tilde{g})}\right]_{\alpha,t,r}&\leq C\gamma(t)^{-1}\left(1+\left(\frac{r}{\varepsilon(t)}\right)^{2\alpha}\right)\varepsilon(t)^4\delta(t)^{-4-\sigma}\|h\|_{C^{0,\alpha}_{\gamma,\sigma,T}}\\
&\quad+Cr^2\gamma(t)^{-1}\varepsilon(t)^{2-\sigma}\|\psi\|_{C^0_{\gamma,\sigma+2,T}}+C\varepsilon(t)^4\sum_i[\lambda_i]_{\alpha,t,r},
\end{split}
\end{equation*}
and,
\begin{equation*}
\begin{split}
\left[ \mu\right]_{\alpha,t,r}&\leq C\gamma(t)^{-1}\left(1+\left(\frac{r}{\varepsilon(t)}\right)^{2\alpha}\right)\varepsilon(t)^4\delta(t)^{-4-\sigma}\|h\|_{C^{0,\alpha}_{\gamma,\sigma,T}}\\
&\quad+C\gamma(t)^{-1}\|\psi\|_{C^0_{\gamma,\sigma+2,T}}+C\varepsilon(t)^4\sum_i[\lambda_i]_{\alpha,t,r}.
\end{split}
\end{equation*}
In particular, 
there exists some uniform positive constant $C$ such that:
\begin{equation*}
\begin{split}
  \|\mu\cdot\tilde{c}\|_{C^{0,\alpha}_{\gamma,\sigma+2,T}}&\leq o_T(1)\|h\|_{C^{0,\alpha}_{\gamma,\sigma,T}}+C\|\psi\|_{C^{0,\alpha}_{\gamma,\sigma+2,T}}+C\sup_{t\,\leq \,T}\gamma(t)\varepsilon(t)^{4}\sup_M\rho^{2+\sigma}\sum_i[\lambda_i]_{\alpha,t,\rho}.
\end{split}
\end{equation*}}
\end{prop}

\begin{proof}
From the proof of Proposition \ref{prop 6.5 M(t)},
\begin{equation*}
\begin{split}
&\langle \psi(t)+\mu(t)\tilde{c}(t),\tilde{c}(t)\rangle_{L^2(\tilde{g}(t))}\\
&=\left\langle \partial_t h(t),\tilde{c}(t)\right\rangle_{L^2(\tilde{g}(t))}+\left\langle h(t),(\Lambda+1)\,\tilde{c}(t)+\Delta_{L,\tilde{g}(t)}\tilde{c}(t)\right\rangle_{L^2(\tilde{g}(t))}-\lambda_i(t)\langle \tilde{\oi}_i(t),\tilde{c}(t)\rangle_{L^2(\tilde{g}(t))}.
\end{split}
\end{equation*}

In particular, we have obtained, thanks to Lemma \ref{pre-partial-h-2-holder}:
\begin{equation}
\begin{split}\label{est-hold-c-psi}
\left[\langle \psi+\mu\cdot\tilde{c},\tilde{c}\rangle_{L^2(\tilde{g})}\right]_{\alpha,t,r}&\leq C\gamma(t)^{-1}\left(\left(\frac{r}{\delta(t)}\right)^{2\alpha}+r^2+\frac{[\dot{\zeta}]_{\alpha,t,r}}{\varepsilon(t)^2}\right)\varepsilon(t)^{4}\delta(t)^{-4-\sigma}\|h\|_{C^{0,\alpha}_{\gamma,\sigma,T}} \\
&\quad+C\gamma(t)^{-1}\left(\left(1+r^{2\alpha}\delta(t)^{-2\alpha}\right)\varepsilon(t)^4\delta(t)^{-4-\sigma}+r^2\varepsilon(t)^{2-\sigma}+r^2\delta(t)^{4-\sigma}\right)\|h\|_{C^{0,\alpha}_{\gamma,\sigma,T}}\\
&\quad+Cr^2\varepsilon(t)^4\sum_i|\lambda_i(t)|+C\varepsilon(t)^4\sum_i[\lambda_i]_{\alpha,t,r}\\
&\leq C\gamma(t)^{-1}\left(\left(1+\left(\frac{r}{\varepsilon(t)}\right)^{2\alpha}\right)\varepsilon(t)^4\delta(t)^{-4-\sigma}+\left(\frac{r}{\varepsilon(t)}\right)^{2}\varepsilon(t)^{4-\sigma}\right)\|h\|_{C^{0,\alpha}_{\gamma,\sigma,T}}\\
&\quad+Cr^2\gamma(t)^{-1}\varepsilon(t)^{2-\sigma}\|\psi\|_{C^0_{\gamma,\sigma+2,T}}+C\varepsilon(t)^4\sum_i[\lambda_i]_{\alpha,t,r}\\
&\leq C\gamma(t)^{-1}\left(1+\left(\frac{r}{\varepsilon(t)}\right)^{2\alpha}\right)\varepsilon(t)^4\delta(t)^{-4-\sigma}\|h\|_{C^{0,\alpha}_{\gamma,\sigma,T}}\\
&\quad+Cr^2\gamma(t)^{-1}\varepsilon(t)^{2-\sigma}\|\psi\|_{C^0_{\gamma,\sigma+2,T}}+C\varepsilon(t)^4\sum_i[\lambda_i]_{\alpha,t,r}.
\end{split}
\end{equation}
Here we have invoked Lemma \ref{est-basic-conf} and Proposition \ref{prop-semi-norm-lap-app-ker} in the second line, Propositions \ref{prop-proj-conf} and \ref{prop-easy-l2-proj-holder} in the third line and Corollary \ref{coro-intermed-est-coeff} in the second inequality.

This leads directly to the desired first estimate.

As for the second estimate, Lemma \ref{lemma-semi-norm-diff-scales} and Lemma \ref{lemma-scal-prod-holder} applied to $\psi(t)$ and $\tilde{c}(t)$ lead to:  

\begin{equation*}
\left[\langle \psi,\tilde{c}\rangle_{L^2(\tilde{g})}\right]_{\alpha,t,r}\leq C\gamma(t)^{-1}\|\psi\|_{C^{0,\alpha}_{\gamma,\sigma+2,T}}.
\end{equation*}
In particular, one gets the following intermediate estimate on the basis of \eqref{est-hold-c-psi}:
\begin{equation*}
\begin{split}
\left[ \mu\right]_{\alpha,t,r}&\leq C\gamma(t)^{-1}\left(1+\left(\frac{r}{\varepsilon(t)}\right)^{2\alpha}\right)\varepsilon(t)^4\delta(t)^{-4-\sigma}\|h\|_{C^{0,\alpha}_{\gamma,\sigma,T}}\\
&\quad+C\gamma(t)^{-1}\|\psi\|_{C^0_{\gamma,\sigma+2,T}}+C\varepsilon(t)^4\sum_i[\lambda_i]_{\alpha,t,r}.
\end{split}
\end{equation*}
Here we have used the fact that $\|\tilde{c}(t)\|^2_{L^2(\tilde{g}(t))}$ is uniformly bounded from below. We have also invoked Corollary \ref{coro-intermed-est-coeff} together with Proposition \ref{prop-easy-l2-proj-holder} to handle the terms involving $|\mu|\Big[\|\tilde{c}\|^2_{L^2(\tilde{g})}\Big]_{\alpha,t,r}$.
Then one gets the following pointwise estimate thanks to Proposition \ref{prop-holder-semi-norm-app-ker} and Corollary \ref{coro-intermed-est-coeff}:

\begin{equation*}
\begin{split}
\|\mu\cdot\tilde{c}\|_{C^{0,\alpha}_{\gamma,\sigma+2,T}}&\leq \|\mu\cdot\tilde{c}\|_{C^{0}_{\gamma,\sigma+2,T}}+\sup_{t\,\leq\, T}\gamma(t)\sup_{t-c_0\,\leq \,t'\,\leq \,t}|\mu(t')|\\
&\quad+\sup_{t\,\leq\,T}\gamma(t)\sup_{ M}\rho^{2+\sigma}[\mu]_{\alpha,t,\rho}\\
&\leq o_T(1)\|h\|_{C^0_{\gamma,\sigma,T}}+C\|\psi\|_{C^0_{\gamma,\sigma+2,T}}+\sup_{t\,\leq\,T}\gamma(t)\sup_{ M}\rho^{2+\sigma}[\mu]_{\alpha,t,\rho}\\
&\leq o_T(1)\|h\|_{C^{0,\alpha}_{\gamma,\sigma,T}}+C\|\psi\|_{C^{0,\alpha}_{\gamma,\sigma+2,T}}+C\sup_{t\,\leq \,T}\gamma(t)\varepsilon(t)^{4}\sup_M\rho^{2+\sigma}\sum_i[\lambda_i]_{\alpha,t,\rho}.
\end{split}
\end{equation*}


\end{proof}

By combining Propositions \ref{prop 6.5 E(t)-holder}, \ref{prop 6.5 F(t)-holder} and \ref{prop 6.5 M(t)-holder} and using the fact that $2>\sigma+2\alpha$ if $\sigma\in(0,1)$ and $\alpha\in(0,1/2)$, we get for $T$ small enough by absorption:
\begin{coro}\label{coro-intermed-est-coeff-holder}
There exists a uniform positive constant $C$ and $T<0$ negative enough {(respectively large enough)} such that the following holds: if $h\in C^{2,\alpha}_{\gamma,\sigma,T}$ is such that $h(t)$ is $L^2(\tilde{g}(t))$-orthogonal to $\tilde{\mathbf{O}}(t)$ for each $t\leq T$ and satisfies for $t\leq T$,
 \begin{equation*}
  (\partial_t - \Delta_{L,\tilde{g}(t)}-2\Lambda)h(t) - \psi(t) \in \tilde{\mathbf{O}}(t),  
  \end{equation*}
  i.e. 
   \begin{equation*}
  (\partial_t - \Delta_{L,\tilde{g}(t)}-2\Lambda)h(t) = \psi(t) +\lambda_i(t)\tilde{\oi}_i(t)+\nu(t)\tilde{g}(t)+\mu(t)\tilde{c}(t),
  \end{equation*}
then,
{
\begin{equation*}
\begin{split}
\sum_j[\lambda_j]_{\alpha,t,r}&\leq C\gamma(t)^{-1}\left(\left(1+r^{2\alpha}\delta(t)^{-2\alpha}\right)\delta(t)^{-2-\sigma}+r^2\varepsilon(t)^{-2-\sigma}\right)\|h\|_{C^{0,\alpha}_{\gamma,\sigma,T}}\\
&\quad+C\gamma(t)^{-1}\varepsilon(t)^{-2-\sigma}(1+r^{2\alpha}\varepsilon(t)^{-2\alpha})\|\psi\|_{C^{0,\alpha}_{\gamma,\sigma+2,T}},\\
[\nu]_{\alpha,t,r}&\leq C\gamma(t)^{-1}\varepsilon(t)^{4}\delta(t)^{-4-\sigma}\left(1+r^{2\alpha}\varepsilon(t)^{-2\alpha}\right)\|h\|_{C^{0,\alpha}_{\gamma,\sigma,T}}+C\gamma(t)^{-1}\|\psi\|_{C^{0,\alpha}_{\gamma,\sigma+2,T}},\\
[\mu]_{\alpha,t,r}&\leq C\gamma(t)^{-1}\varepsilon(t)^{4}\delta(t)^{-4-\sigma}\left(1+r^{2\alpha}\varepsilon(t)^{-2\alpha}\right)\|h\|_{C^{0,\alpha}_{\gamma,\sigma,T}}+C\gamma(t)^{-1}\|\psi\|_{C^{0,\alpha}_{\gamma,\sigma+2,T}},
\end{split}
\end{equation*}}
and,
\begin{equation*}
\sum_j\|\lambda_j\cdot\tilde{\oi}_j\|_{C^{0,\alpha}_{\gamma,\sigma+2,T}}+ \|\nu\cdot\tilde{g}\|_{C^{0,\alpha}_{\gamma,\sigma+2,T}}+ \|\mu\cdot\tilde{c}\|_{C^{0,\alpha}_{\gamma,\sigma+2,T}}\leq o_T(1)\|h\|_{C^{0,\alpha}_{\gamma,\sigma,T}}+C\|\psi\|_{C^{0,\alpha}_{\gamma,\sigma+2,T}}.
\end{equation*}
\end{coro}

\subsection{A parabolic Lyapunov-Schmidt reduction}

\begin{prop}\label{prop-not-so-easy-bis}
Let $\psi\in C^{0,\alpha}_{\gamma,\sigma+2,T}$. Then there exists a unique $h\in C^{2,\alpha}_{\gamma,\sigma,T}$ orthogonal to the approximate kernel $\tilde{\mathbf{O}}(t)$ for each $t\leq T$ {(respectively for each $t\geq T$ in the immortal case such that $h(T)\equiv 0$)} and such that 
 \begin{equation*}
  (\partial_t - \Delta_{L,\tilde{g}(t)}-2\Lambda)h(t) -\psi(t) \in \tilde{\mathbf{O}}(t),
  \end{equation*}
  i.e. there exists real coefficients $\lambda_i$,  $\nu$ and $\mu$ defined on $(-\infty,T]$ {(respectively $[T,+\infty)\to \mathbb{R}$ in the immortal case} such that
   \begin{equation*}
  (\partial_t - \Delta_{L,\tilde{g}(t)}-2\Lambda)h(t) =\psi(t)+ \lambda_i(t) \tilde{\oi}_{i}(t)+ \nu(t)\tilde{g}(t)+\mu(t) \tilde{c}(t) .  
  \end{equation*}
	Moreover there exists a uniform constant $C$ such that:
	\begin{equation*}
	\begin{split}
	\|h\|_{C^{2,\alpha}_{\gamma,\sigma,T}}&\leq C\|\psi\|_{C^{0,\alpha}_{\gamma,\sigma+2,T}},\\
	\sum_{i}\|\lambda_i\cdot\tilde{\oi}_i\|_{C^{0,\alpha}_{\gamma,\sigma+2,T}}+ \|\nu\cdot\tilde{g}\|_{C^{0,\alpha}_{\gamma,\sigma+2,T}}+ \|\mu\cdot\tilde{c}\|_{C^{0,\alpha}_{\gamma,\sigma+2,T}}&\leq o_T(1)\|h\|_{C^{0,\alpha}_{\gamma,\sigma,T}}+C\|\psi\|_{C^{0,\alpha}_{\gamma,\sigma+2,T}}.
	\end{split}
	\end{equation*}
\end{prop}
\begin{proof}
	Let us explain why the core of the proof is to show the existence of solutions on intervals of finite size $[T',T]$ which will be treated in Appendix \ref{sec-flot-mod-obst}.
	
	First, the uniqueness and controls of the solution come from Propositions \ref{6.2 chez nous} and \ref{6.2 chez nous immortel}, so we are left with proving the existence of $h$, together with that of the coefficients $\lambda_i(t)$, $\mu(t)$ and $\nu(t)$. Moreover, by the a priori controls of Propositions \ref{6.2 chez nous} and \ref{6.2 chez nous immortel}, we may approximate $\psi$ by compactly supported tensors $\psi_k:=\chi_k\psi$, for $\chi_k$ a cut-off function supported on $[-k,T]$. From Proposition \ref{prop-flot-mod-obst} in Appendix \ref{sec-flot-mod-obst}, there exist associated $h_k(t)$, $\lambda_{i,k}(t)$, $\mu_k(t)$ and $\nu_k(t)$ so that 
	\begin{equation*}
  (\partial_t - \Delta_{L,\tilde{g}(t)}-2\Lambda)h_k(t) = \chi_k\psi(t) + \lambda_{i,k}(t) \tilde{\oi}_{i}(t)+\nu_k(t)\tilde{g}(t)+\mu_k(t) \tilde{c}(t).  
  \end{equation*}
  
  Then, by the a priori controls of Proposition \ref{6.2 chez nous} (and \ref{6.2 chez nous immortel} in the immortal case) together with Corollary \ref{coro-intermed-est-coeff-holder}, the pointwise limit $h:= \lim_k h_k$ exists on $(-\infty,T]\times M$ and lies in $C^{2,\alpha}_{\gamma,\sigma+2,T}$ and the pointwise limits of the coefficients $\lambda_i(t) := \lim_k\lambda_{i,k}(t)$, $\mu(t) := \lim_k \mu_k(t)$ and $\nu(t):=\lim_k \nu_k(t)$ exist for each $t\leq T$ and satisfy (after absorption): 
  \begin{equation*}
  \begin{split}
\|h\|_{ C^{2,\alpha}_{\gamma,\sigma+2,T}}+ \sum_{i}\|\lambda_i\cdot\tilde{\oi}_i\|_{C^{0,\alpha}_{\gamma,\sigma+2,T}}+ \|\nu\cdot\tilde{g}\|_{C^{0,\alpha}_{\gamma,\sigma+2,T}}+ \|\mu\cdot\tilde{c}\|_{C^{0,\alpha}_{\gamma,\sigma+2,T}}&\leq C\|\psi\|_{C^{0,\alpha}_{\gamma,\sigma+2,T}}.
	\end{split}
	\end{equation*}
 \end{proof}

\section{Non-linear analysis}\label{sec-nonlin}


\subsection{Set-up for ancient solutions}\label{section-def}
In this section, we consider a $4$-dimensional Einstein orbifold $(M_o^4,g_o)$ with Einstein constant $\Lambda\in\mathbb{R}$, i.e. $\Ric(g_o)=\Lambda g_o$ on the regular part of $M_o$. We assume from now on that $M_o$ has one single orbifold singularity modeled on $\RR^4/\mathbb{Z}_2$ at a point $p_o$. The case where the orbifold singularity is modeled on $\RR^4/\Gamma$ where $\Gamma\subset \operatorname{SU(2)}$ is finite and acts freely on $\mathbb{C}^2\setminus\{0\}$ is handled in Section \ref{sec-kro}. A finite number of orbifold singularities is allowed as explained in Remark \ref{rk-finite-nb-sing}.


We start by setting up the notations together with the heuristics explaining each term. Let $T\leq 0$ and $\alpha\in(0,1/2)$, and consider $\zeta_0\in\mathbb{R}^3\setminus\{0\}$ interpreted as a ``gluing parameter''. Denote the smallest eigenvalue of $2\mathbf{R}^+_p$ by $\Lambda_{0}$. Define also an approximate bound for the size of bubbles by 
\begin{equation}\label{linear-inf-sup-bubbles}
\varepsilon_{0}(t):=e^{\frac{\Lambda_{0}}{2}t},\quad t\leq T.
\end{equation}
\begin{itemize}
\item The family of symmetric $2$-tensors $(k(t))_{t\,\leq\, T}$ is thought of as a metric perturbation of the first approximation $\tilde{g}(t)$ (which is not defined yet) lying in $C^{2,\alpha}_{\gamma, \sigma, T}$. \\

\item The vector $\eta\in\RR^3$ is related to the approximate orthogonal projection of our desired solution against $\tilde{\oi}_{i}(t)$ and it determines the size of the bubble parametrized by $\zeta$ in the next approximation. It lies in the following function space: 
\begin{equation*}
\|\eta\|_{C^{0,\alpha}_{\eta_0,T}}:=\sup_{t\,\leq\, T}\left\{e^{-\eta_0 t}\left(|\eta(t)|+[\eta]_{\alpha,t,\varepsilon_0(t)}\right)\right\}<+\infty,
\end{equation*}
where $\eta_0>0$ will be chosen appropriately in Lemma \ref{lemma-control-zeta-bis} and [\eqref{choice-parameters}, Section \ref{section-non-hom-term}].\\

\item The function $\beta$ is related to the approximate orthogonal projection of the metric perturbation against the background metric $\tilde{g}(t)$ (which is not defined yet). It lies in the following function space: 
\begin{equation*}
\|\beta\|_{C^{0,\alpha}_{\beta_0,T}}:=\sup_{t\,\leq\, T}\left\{e^{-\beta_0 t}\left(|\beta(t)|+[\beta]_{\alpha,t,1}\right)\right\}<+\infty,
\end{equation*}
where $\beta_0>0$ will be chosen in [\eqref{choice-parameters}, Section \ref{section-non-hom-term}].\\

\item The function $\tau$ denotes the approximate orthogonal projection of the metric perturbation against the approximate conformal deformation $\tilde{c}(t)$ (which is not defined yet).
It lies in the following function space: 
\begin{equation*}
\|\tau\|_{C^{0,\alpha}_{\tau_0,T}}:=\sup_{t\,\leq\, T}\left\{e^{-\tau_0 t}\left(|\tau(t)|+[\tau]_{\alpha,t,1}\right)\right\}<+\infty,
\end{equation*}
where $\tau_0>0$ will be chosen in [\eqref{choice-parameters}, Section \ref{section-non-hom-term}].\\
\end{itemize}


We define a function space $\mathcal{X}$ as follows:
\begin{equation}
\begin{split}\label{defn-fct-space-X-imm}
\mathcal{X}&:=C^{2,\alpha}_{\gamma, \sigma, T}\times C^{0,\alpha}_{\eta_0,T}\times C^{0,\alpha}_{\beta_0,T}\times C^{0,\alpha}_{\tau_0,T},\\
\\
\|(k,\eta,\beta,\tau)\|_{\mathcal{X}}&:=\|k\|_{C^{2,\alpha}_{\gamma, \sigma, T}}+\|\eta\|_{C^{0,\alpha}_{\eta_0,T}}+\|\beta\|_{C^{0,\alpha}_{\beta_0,T}}+\|\tau\|_{C^{0,\alpha}_{\tau_0,T}}.
\end{split}
\end{equation}
The normed space $\mathcal{X}$ is a Banach space, its unit ball will be denoted by $\mathcal{B}.$ \\

For a given $\zeta_0\in\RR^3\setminus \{0\}$ and a quadruplet $(k,\eta,\beta,\tau)\in \mathcal{B}$, we formally define a map $\Phi (k,\eta,\beta,\tau)=: (h,\xi,\nu,\mu)$ whose fixed point set is made of ancient solutions to the DeTurck-Ricci flow only. More precisely,  
\begin{enumerate}
	\item Given $\eta\in C^{0,\alpha}_{\eta_0,T}$ such that $\|\eta\|_{C^{0,\alpha}_{\eta_0,T}}\leq1$, we define $\zeta(t)$ through the following differential equation:
	\begin{equation}\label{def-varep}
	\dot{\zeta}(t)=2\mathbf{R}_p^+(\zeta(t))+\eta(t),\quad t\leq T,\quad \zeta(T):=\exp(2T\mathbf{R}_p^+)\zeta_0\in\RR^3\setminus \{0\}.
	\end{equation}
Equivalently, thanks to the Duhamel principle, one gets the following more explicit formula for $\zeta(t)$:
	\begin{equation}\label{def-varep-bis}
	\zeta(t)=\exp\left(2t\mathbf{R}_p^+\right)\zeta_0-\int_t^T\exp\left(2(t-s)\mathbf{R}_p^+\right)\eta(s)\,ds,\quad t\leq T.
	\end{equation}
We need to ensure that $\zeta(t)$ stays controlled in norm and in direction close to that of the initial condition $\zeta_0$ in order to make sense of $g_{\zeta(t)}$. Denote by $\operatorname{sp}(2\mathbf{R}_p^+)=:(\Lambda_i)_i$ the spectrum of $2\mathbf{R}_p^+$: since $2\mathbf{R}_p^+$ is a symmetric matrix with real entries, $\operatorname{sp}(2\mathbf{R}_p^+)$ is contained in $\RR.$ 

	\begin{lemma}\label{lemma-control-zeta-bis}
Assume $\eta_0\in\RR_+\setminus \operatorname{sp}(2\mathbf{R}_p^+)$. Then there exists a constant $C>0$ such that for $t\leq T\leq 0$:
\begin{equation}\label{crucial-est-zeta}
\left|\zeta(t)-\exp\left(2t\mathbf{R}_p^+\right)\zeta_0\right|\leq C\left(\sum_{\eta_0>\Lambda_i}\frac{e^{\Lambda_i t}}{\eta_0-\Lambda_i}+\sum_{\eta_0<\Lambda_i}\frac{e^{\eta_0 t}}{\Lambda_i-\eta_0}\right).
\end{equation}
In particular, \\
\begin{enumerate}
\item \label{iteam-a-ode}if $\operatorname{sp}(2\mathbf{R}_p^+)=:\{\Lambda_0\}$, $\eta_0>\Lambda_0>0$ and $\zeta_0\in \RR^3\setminus\{0\}$ such that $(\eta_0-\Lambda_0)|\zeta_0|>c_0$ for some universal positive constant $c_0$ then there exist positive constants $C_1$ and $C_2$ depending on $|\zeta_0|$ and $\eta_0-\Lambda_{0}$ only such that for all $t\leq T\leq 0$:
\begin{equation*}
C_1\varepsilon_0(t)^2=C_1e^{\Lambda_0 t}\leq |\zeta(t)|\leq C_2e^{\Lambda_0 t}=C_2\varepsilon_0(t)^2.
\end{equation*}
Moreover, if $\eta_0\in(\Lambda_0,2\Lambda_0]$, there exists $C>0$ depending on  $\Lambda_0$ only such that
\begin{equation*}
[\dot{\zeta}]_{\alpha,t,\varepsilon_0(t)}\leq C\varepsilon_0(t)^{2\frac{\eta_0}{\Lambda_0}},\quad t\leq T.
\end{equation*}

\item \label{iteam-b-ode}if $\operatorname{sp}(2\mathbf{R}_p^+)$ contains at least two distinct eigenvalues $0<\Lambda_0<\Lambda_1$, $\eta_0\in(\Lambda_0,\Lambda_1)$ and $\zeta_0\in \RR^3\setminus\{0\}$ is such that its projection on $\ker (2\mathbf{R}_p^+-\Lambda_1\Id)$ is non zero and such that $(\eta_0-\Lambda_0)|\zeta_0|>c_0$ for some universal positive constant $c_0$  then there exist positive constants $C_1$ and $C_2$ depending on $|\zeta_0|$, $\eta_0-\Lambda_{0}$ and $\Lambda_{1}-\eta_0$ only such that for all $t\leq T\leq 0$:
\begin{equation*}
C_1\varepsilon_0(t)^2\leq |\zeta(t)|\leq C_2\varepsilon_0(t)^2.
\end{equation*}
Moreover, if $\eta_0\in(\Lambda_0,\min\{\Lambda_1, 2\Lambda_0\}]$, there exists $C>0$ depending on the operator norm $\|\mathbf{R}_p^+\|$ only such that
\begin{equation*}
[\dot{\zeta}]_{\alpha,t,\varepsilon_0(t)}\leq C\varepsilon_0(t)^{2\frac{\eta_0}{\Lambda_0}},\quad t\leq T.
\end{equation*}

\end{enumerate}
	\end{lemma}
	
	\begin{rk}\label{rk ref zeta controlled Kro}
		This lemma also applies and is crucial in the case of hyperkähler ALE metrics, where it is even more important to make sure that $\zeta(t)/|\zeta(t)|$ stays close to $\zeta_0/|\zeta_0|$. This ensures that other orbifold singularities are not created along the flow, see Remark \ref{rk: zeta close to initial direction}. 
	\end{rk}
	\begin{rk}\label{rk-finite-nb-sing}
	The setting we consider allows a finite number of orbifold singularities $(p_i)_{i\in I}\subset M_o$ on the condition that all the corresponding ``gluing'' parameters introduced previously $\zeta_0^i\in\RR^3\setminus\{0\}$ and $\Lambda_0^i>0$ satisfy: $\Lambda_0^i=\Lambda_0$ and $\zeta_0^i$ have a non zero projection on $\ker (2\mathbf{R}_p^+-\Lambda_0\Id)$ for all $i\in I$.   
	\end{rk}
	\begin{proof}
	The proof of \eqref{crucial-est-zeta} is essentially based on formula \eqref{def-varep-bis}. Indeed, let $(\mathbf{e}_i)_{1\leq i\leq 3}$ be an orthonormal basis of $\RR^3$ such that $2\mathbf{R}_p^+\mathbf{e}_i=\Lambda_i \mathbf{e}_i$ for $i=1,2,3$. Then write $\eta(t)=\sum_{i=1}^3\eta_i(t)\cdot \mathbf{e}_i$ to estimate:
	\begin{equation*}
	\begin{split}
\left|\zeta(t)-\exp\left(2t\mathbf{R}_p^+\right)\zeta_0\right|&\leq C\sum_{i=1}^3\int_t^Te^{\Lambda_i (t-s)}|\eta_i(s)|\,ds\\
&\leq C\|\eta\|_{C^0_{\eta_0,T}}\sum_{i=1}^3e^{\Lambda_i t}\int_t^Te^{(\eta_0-\Lambda_i)s}\,ds\\
&\leq C\left(\sum_{\eta_0>\Lambda_i}\frac{e^{\Lambda_i t}}{\eta_0-\Lambda_i}+\sum_{\eta_0<\Lambda_i}\frac{e^{\eta_0 t}}{\Lambda_i-\eta_0}\right).
\end{split}
\end{equation*}
Here we have used that $\|\eta\|_{C^0_{\eta_0,T}}\leq 1$ by assumption and $C$ is a universal positive constant.

We turn to the proof of item \ref{iteam-a-ode}.

Now, if $\operatorname{sp}(2\mathbf{R}_p^+)=:\{\Lambda_0\}$, $\eta_0>\Lambda_0>0$ and $\zeta_0\in \RR^3\setminus\{0\}$ then the triangular inequality together with estimate \eqref{crucial-est-zeta} gives:
\begin{equation*}
\begin{split}
 |\zeta(t)|&\geq e^{\Lambda_0 t}|\zeta_0|-e^{\Lambda_0 t}\frac{C}{\eta_0-\Lambda_0}=\left(|\zeta_0|-\frac{C}{\eta_0-\Lambda_0}\right)e^{\Lambda_0 t}>0,
 \end{split}
\end{equation*}
if $|\zeta_0|>\frac{C}{\eta_0-\Lambda_0}$. Here we have used that $T\leq 0$. A similar upper bound holds.

Observe that we get an intermediate upper bound for the speed $|\dot{\zeta}(t)|$:
\begin{equation}\label{easy-upp-bd-der}
|\dot{\zeta}(t)|\leq \left(\Lambda_0+1\right)e^{\Lambda_{0}t}=\left(\Lambda_0+1\right)\varepsilon_0(t)^2,\quad t\leq T,
\end{equation}
since again $\|\eta\|_{C^{0,\alpha}_{\eta_0,T}}\leq 1$ and $\eta_0>\Lambda_{0}$. 

Finally, the mean value theorem ensures for $t\leq T\leq 0$ and $t- \varepsilon_{0}(t)^2\leq t'<t$:
\begin{equation*}
\begin{split}
\frac{|\dot{\zeta}(t)-\dot{\zeta}(t')|}{|t-t'|^{\alpha}}&\leq \Lambda_0\frac{|\zeta(t)-\zeta(t')|}{|t-t'|^{\alpha}}+\frac{|\eta(t)-\eta(t')|}{|t-t'|^{\alpha}}\\
&\leq \Lambda_0\max_{s\in(t,t')}|\dot{\zeta}(s)||t-t'|^{1-\alpha}+\varepsilon_{0}(t)^{-2\alpha}e^{\eta_0t}\|\eta\|_{C^{0,\alpha}_{\eta_0,T}}\\
&\leq \Lambda_0\left(\Lambda_0+1\right)e^{\Lambda_{0}(t+\varepsilon_{0}(t))}\varepsilon_{0}(t)^{2(1-\alpha)}+\varepsilon_{0}(t)^{-2\alpha}e^{\eta_0t}\\
&\leq C\varepsilon_{0}(t)^{2\frac{\eta_0}{\Lambda_0}-2\alpha},
\end{split}
\end{equation*}
for some constant $C$ that depends on $\Lambda_0$ only. Here we have used the assumption $\eta_0\leq 2\Lambda_0$.

The proof of item \ref{iteam-b-ode} is similar and is therefore left to the reader.
	\end{proof}
	
	Thanks to Lemma \ref{lemma-control-zeta-bis}, we can define $\zeta(t)$ and hence $\tilde{g}(t)$ whenever $\|\eta\|_{C^{0,\alpha}_{\eta_0,T}}\leq 1$ and $\zeta_0\in \RR^3$ lies outside a sufficiently large ball whose radius depends on  $\eta_0$ and the spectrum of $2\mathbf{R}_p^+$.

		\item Non-linear terms coming from the linearization of the Ricci tensor involving $k\in C^{2,\alpha}_{\gamma,\sigma,T}$ with $|k(t)|_{\tilde{g}(t)}\leq 1/2$ for $t\leq T$ are incorportated in the data $\psi$:
	\begin{equation}\label{defn-psi}
	\psi(t):=-Q_{\tilde{g}(t)}(k(t))+\beta(t)k(t)+\frac{1}{2}\tau(t)\Li_{\nabla^{\tilde{g}(t)}\tilde{v}(t)}k(t)-\left(\partial_t\tilde{g}(t)+2\Ric(\tilde{g}(t))-2\Lambda\tilde{g}(t)\right),
	\end{equation}
	where $Q_{\tilde{g}(t)}(k(t))$ is defined by
	\begin{equation*}
	\begin{split}
	Q_{\tilde{g}(t)}(k(t))&:=2\Ric(\tilde{g}(t)+k(t))-2\Ric(\tilde{g}(t))+\Delta_{L,\tilde{g}(t)}k(t)-\mathcal{L}_{B_{\tilde{g}(t)}(k(t))}(\tilde{g}(t)+k(t)),\\
	B_{\tilde{g}(t)}(k(t))&:=\div_{\tilde{g}(t)}k(t)-\frac{1}{2}d\tr_{\tilde{g}(t)}k(t).
	\end{split}
	\end{equation*}
	Recall that if $|k(t)|_{\tilde{g}(t)}\leq 1/2$ for $t\leq T$, a fact  that will be guaranteed by the choice of parameters made below in [\eqref{choice-parameters}, Section \ref{section-non-hom-term}] if $T\leq 0$ is negative enough, then:
	\begin{equation}\label{prop-Q-0}
	\begin{split}
	Q_{\tilde{g}(t)}(k(t))&=g(t)^{-1}\ast k(t)\ast \nabla^{\tilde{g}(t),\,2}k(t)+g(t)^{-1}\ast g(t)^{-1}\ast \nabla^{\tilde{g}(t)}k(t)\ast\nabla^{\tilde{g}(t)}k(t)\\
	&\quad+g(t)^{-1}\ast \Rm(\tilde{g}(t))\ast k(t)\ast k(t).
	\end{split}
	\end{equation}
	
	\item The first component $h$ and the last two components $(\nu,\mu)$ of the quadruplet $(h,\xi,\nu,\mu)$ are then defined as follows: the tensor $h$ is the unique solution to the linear (Lichnerowicz) heat equation modulo obstructions with the non-homogeneous term $\psi$ defined previously:
	\begin{equation}\label{eqn-solve-h}
	\partial_th(t)-\Delta_{L,\tilde{g}(t)}h(t)-2\Lambda h(t)=\psi(t)+\lambda_i(t)\tilde{\oi}_i(t)+\nu(t)\tilde{g}(t)+\mu(t)\tilde{c}(t),\quad t\leq T,
	\end{equation}
	whose existence is ensured by Proposition \ref{prop-not-so-easy-bis} once $\psi$ lies in $C^{0,\alpha}_{\gamma,\sigma+2,T}$, a fact that will be checked in Proposition \ref{prop-final-est-proj}.\\
	\item The vectorial function $\xi(t)\in\RR^3$ is then obtained by taking into account obstructions coming from $\tilde{\oi}_i(t)$  which are formally encoded by $\eta$ and $(\lambda_i)_i$:
	\begin{equation}\label{def-xi}
	d_{\zeta(t)}g(\xi(t))=d_{\zeta(t)}g(\eta(t))-\sum_i\lambda_i(t)\|\tilde{\oi}_i(t)\|^2_{L^2(\tilde{g}(t))}\frac{\oi_i(\zeta(t))}{\|\oi_i(\zeta(t))\|^2_{L^2(g_{\zeta(t)})}},\quad t\leq T.
	\end{equation}

\end{enumerate}

\begin{rk}
In case there are no rotations involved, i.e. if the symmetries one would impose were implying $\lambda_i(t)=0$ for $i=2,3$ then $\zeta(t)$, $\eta(t)$ and $\xi(t)$ would consist of real functions and if we denote $\operatorname{R}^+_{11}:=\frac{1}{2}\langle\mathbf{R}^+_p{\omega_1^+},\omega_1^+\rangle$,
\begin{equation*}
	\eta(t)=:2\left(\frac{\varepsilon'(t)}{\varepsilon(t)}+\operatorname{R}^+_{11}\right),\quad t\leq T,\quad \varepsilon(T)=\exp(-\lambda_1T).
	\end{equation*}
	Equivalently, $\varepsilon(t)=\exp\left(-\operatorname{R}^+_{11}t+\int_{t}^T\eta(s)\,ds\right)$, $t\leq T$ and,

	\begin{equation*}
	\xi(t):=\eta(t)-(2\pi^2)^{-1}\varepsilon(t)^{-4}\operatorname{R}^+_{11}\|\tilde{\oi}_1(t)\|^2_{L^2(\tilde{g}(t))},\quad t\leq T.
	\end{equation*}
	\end{rk}
\begin{rk}
In \cite{Bre-Kap}, the definition of $\varepsilon(t)^2=|\zeta(t)|$ is: $\eta(t)=4\varepsilon'(t)\varepsilon(t)^{-5}-32\omega$ for $t\leq T$, $\varepsilon(T)^{-4}=32\omega(-T)$, where $\omega$ is a numerical constant. The presence of an additional $\varepsilon(t)^{-4}$ factor is explained by the degenerate curvature of the orbifold $\mathbb{T}^4/\mathbb{Z}_2$ which satisfies $\mathbf{R}^+$. One needs a higher order of precision to detect $\mathbf{R}^+\propto\varepsilon^4$ as computed in \cite{ozu3}, and symmetry assumptions are necessary to ``force'' the flow to stay in the direction of the unique positive eigenvalue of $\mathbf{R}^+$.
\end{rk}

 \subsection{Set-up for immortal solutions}\label{sec-imm-sol}

We now develop the notations in the immortal case with $0<T\leq t<+\infty$, the heuristics are essentially the same as in the ancient case. Let $T\geq 0$ and $\alpha\in(0,1/2)$. Denote the largest eigenvalue of $2\mathbf{R}^+_p$ by $\Lambda_{0}$ assumed to be negative. Define also an approximate bound for the size of bubbles by 
\begin{equation}\label{linear-inf-sup-bubbles-imm}
\varepsilon_{0}(t):=e^{\frac{\Lambda_{0}}{2}t},\quad t\geq T.
\end{equation}
\begin{itemize}
\item The family of symmetric $2$-tensors $(k(t))_{t\geq T}$ is thought of as a metric perturbation of the first approximation $\tilde{g}(t)$ (which is not defined yet) lying in $C^{2,\alpha}_{\gamma, \sigma, T}$, such that $h(T)\equiv 0$. \\

\item The vector $\eta\in\RR^3$ lies in the following function space: 
\begin{equation*}
\|\eta\|_{C^{0,\alpha}_{\eta_0,T}}:=\sup_{t\,\geq\, T}\left\{e^{\eta_0 t}\left(|\eta(t)|+[\eta]_{\alpha,t,\varepsilon_0(t)}\right)\right\}<+\infty,
\end{equation*}
where $\eta_0>0$ will be chosen appropriately in Lemma \ref{lemma-control-zeta-bis}.\\

\item The function $\beta$ lies in the following function space: 
\begin{equation*}
\|\beta\|_{C^{0,\alpha}_{\beta_0,T}}:=\sup_{t\,\geq\, T}\left\{e^{\beta_0 t}\left(|\beta(t)|+[\beta]_{\alpha,t,1}\right)\right\}<+\infty,
\end{equation*}
where $\beta_0>0$ will be chosen in [\eqref{choice-parameters}, Section \ref{section-non-hom-term}].\\
\end{itemize}


We define a function space $\mathcal{X}$ as follows:
\begin{equation}
\begin{split}\label{defn-fct-space-X}
\mathcal{X}&:=C^{2,\alpha}_{\gamma, \sigma, T}\times C^{0,\alpha}_{\eta_0,T}\times C^{0,\alpha}_{\beta_0,T},\\
\\
\|(k,\eta,\beta)\|_{\mathcal{X}}&:=\|k\|_{C^{2,\alpha}_{\gamma, \sigma, T}}+\|\eta\|_{C^{0,\alpha}_{\eta_0,T}}+\|\beta\|_{C^{0,\alpha}_{\beta_0,T}}.
\end{split}
\end{equation}
The normed space $\mathcal{X}$ is a Banach space, its unit ball will be denoted by $\mathcal{B}.$ \\

For $(k,\eta,\beta)\in \mathcal{B}$, we formally define a map $\Phi (k,\eta,\beta)=: (h,\xi,\nu)$ whose fixed point set is made of ancient solutions to the DeTurck Ricci flow only. More precisely,  
\begin{enumerate}
	\item Given $\eta\in C^{0,\alpha}_{\eta_0,T}$ such that $\|\eta\|_{C^{0,\alpha}_{\eta_0,T}}\leq1$, we define $\zeta(t)$ through the following differential equation:
	\begin{equation}\label{def-varep-imm}
	\dot{\zeta}(t)=2\mathbf{R}_p^+(\zeta(t))+\eta(t),\quad t\geq T,\quad \zeta(T):=\exp(2T\mathbf{R}_p^+)\zeta_0\in\RR^3\setminus \{0\},
	\end{equation}
or equivalently through a more explicit formula for $\zeta(t)$:
	\begin{equation}\label{def-varep-bis-imm}
	\zeta(t)=\exp\left(2t\mathbf{R}_p^+\right)\zeta_0+\int_T^t\exp\left(2(t-s)\mathbf{R}_p^+\right)\eta(s)\,ds,\quad t\geq T.
	\end{equation}
We need to ensure that $\zeta(t)$ stays in $\RR^3\setminus\{0\}$ in order to make sense of $g_{\zeta(t)}$. 
	As in the ancient case, flipping the signs of $t$ and $T$, we use Lemma \ref{lemma-control-zeta-bis}, to define $\zeta(t)$ and hence $\tilde{g}(t)$ whenever $\|\eta\|_{C^{0,\alpha}_{\eta_0,T}}\leq 1$ and $\zeta_0\in \RR^3$ lies outside a sufficiently large ball whose radius depends on   $\eta_0$ and the spectrum of $2\mathbf{R}_p^+$.  \\
		\item Non-linear terms coming from the linearization of the Ricci tensor involving $k\in C^{2,\alpha}_{\gamma,\sigma,T}$ with $|k(t)|_{\tilde{g}(t)}\leq 1/2$ for $t\geq T$ are incorportated in the data $\psi$:
	\begin{equation}\label{defn-psi-imm}
	\psi(t):=-Q_{\tilde{g}(t)}(k(t))+\beta(t)k(t)-\left(\partial_t\tilde{g}(t)+2\Ric(\tilde{g}(t))-2\Lambda\tilde{g}(t)\right),
	\end{equation}
	where $Q_{\tilde{g}(t)}(k(t))$ is defined and controlled as in the ancient case. 
	
	Here notice that the underlying Einstein orbifold has no non-trivial kernel as proved in Proposition \ref{Liouville-negative-einstein-ancien}. In particular, there is an additional simplification compared to the ancient case that needs to take conformal diffeomorphisms into account.\\
	\item The first component $h$  and the last component $\nu$ from the triplet $(h,\xi,\nu)$ are then defined as follows: the tensor $h$ is the unique solution to the linear (Lichnerowicz) heat equation modulo obstructions which vanishes at $t=T$ with the non-homogeneous term $\psi$ defined previously:
	\begin{equation}\label{eqn-solve-h-imm}
	\partial_th(t)-\Delta_{L,\tilde{g}(t)}h(t)-2\Lambda h(t)=\psi(t)+\lambda_i(t)\tilde{\oi}_i(t)+\nu(t)\tilde{g}(t),\quad t\geq T,
	\end{equation}
	whose existence is ensured by Proposition \ref{prop-not-so-easy-bis} once $\psi$ lies in $C^{0,\alpha}_{\gamma,\sigma+2,T}$, a fact that will be checked in Proposition \ref{prop-final-est-proj}.\\
	\item The vectorial function $\xi(t)\in\RR^3$ is then obtained by taking into account obstructions coming from $\tilde{\oi}_i(t)$  which are formally encoded by $\eta$ and $(\lambda_i)_i$:
	\begin{equation}\label{def-xi-imm}
	d_{\zeta(t)}g(\xi(t))=d_{\zeta(t)}g(\eta(t))-\sum_i\lambda_i(t)\|\tilde{\oi}_i(t)\|^2_{L^2(\tilde{g}(t))}\frac{\oi_i(\zeta(t))}{\|\oi_i(\zeta(t))\|^2_{L^2(g_{\zeta(t)})}},\quad t\geq T.
	\end{equation}

\end{enumerate}

\subsection{Choice of the parameters and estimates of the non-homogeneous term}\label{section-non-hom-term}
In this section, we restrict ourselves to the ancient case, that is the setting of Section \ref{section-def}.

The purpose of this section is to estimate each element of the image of a vector in $\mathcal{B}$ under the map $\Phi$ defined in the previous Section \ref{section-def}.

Before doing so, we make a choice of  all the parameters introduced in Section \ref{section-def}.\\



\textbf{We consider $T\leq 0$ (or $T\geq 0$ in the immortal case), $\alpha\in(0,1/2)$, $\sigma\in(0,1/2)$. Let $\eta_0\in\RR$, $\iota\in(0,2\sigma^{-1}-4)$ and $\delta\in\left(\frac{5}{9},\frac{3-(1+\iota)\sigma}{5-\sigma}\right)$ such that:
\begin{equation}\label{choice-parameters}
\tag{$\mathcal{H}_1$}
0<\iota\sigma<2(\eta_0/\Lambda_0-1)<(5-\sigma)(1-\delta)-2.
\end{equation}
According to [\ref{iteam-a-ode}, Lemma \ref{lemma-control-zeta-bis}], let us consider $\zeta_0\in\RR^3\setminus\{0\}$ such that $(\eta_0-\Lambda_0)|\zeta_0|>c_0$ for some positive universal constant $c_0$. Then define: 
\begin{equation*}
  \begin{split}
&\delta(t):=\varepsilon_0(t)^{\delta}, \quad \beta_0\in(6-10\delta,2]\ni\tau_0,\quad\gamma(t):=\varepsilon_0(t)^{-2-(1+\iota) \sigma},
\end{split}
\end{equation*}
so that $|\beta(t)|$ and $|\tau(t)|$ converges to $0$ as $t$ goes to $-\infty$ at an exponential rate $\Lambda_0$ and are bounded from above by $\varepsilon_0(t)^2$. Moreover, $\tau_0$ and $\beta_0$ will be less than $8-10\delta$.}\\
Observe that the conditions \eqref{cond-param} are satisfied under the choices made in \eqref{choice-parameters}.  Moreover, observe that any $k\in\mathcal{B}$ converges to $0$ together with its covariant derivatives (up to order $2$) and its time derivative since $\gamma(t)^{-1}\varepsilon_0(t)^{-2-\sigma}=\varepsilon_0(t)^{\iota \sigma}$. The upper bound on $\eta_0$ will be made clearer in Proposition \ref{prop-prelim-schauder} from Section \ref{sec-fixed-point}.

\begin{prop}\label{prop-final-est-proj}
For a quadruplet $(k,\eta,\beta,\tau)\in \mathcal{B}$, recall the non-homogenous term:
\begin{equation*}
\psi(t):=-Q_{\tilde{g}(t)}(k(t))+\beta(t)k(t)+\tau(t)\Li_{\nabla^{\tilde{g}(t)}\tilde{v}(t)}k(t)-\left(\partial_t\tilde{g}(t)+2\Ric(\tilde{g}(t))-2\Lambda\tilde{g}(t)\right).
\end{equation*}
Then $\|\psi\|_{C^{0}_{\gamma,\sigma+2,T}}\leq C\varepsilon_0(T)^{2\left(\frac{\eta_0}{\Lambda_0}-1\right)-\iota\sigma}$ and,
\begin{equation*}
\begin{split}
|\langle \psi(t),\tilde{\oi}_i(t)\rangle_{L^2(\tilde{g}(t))}&+\langle d_{\zeta(t)}g(\eta(t)),\oi_i(\zeta(t))\rangle_{L^2(g_{\zeta(t)})}|\\
&\leq C\left(\gamma(t)^{-2}\varepsilon(t)^{2-2\sigma}+\gamma(t)^{-1}\varepsilon(t)^{4-\sigma+\beta_0}+\gamma(t)^{-1}\varepsilon(t)^{4-\sigma+\tau_0}+\varepsilon(t)^9\delta(t)^{-7}\right),\\
&\leq C\varepsilon_0(t)^{9-7\delta},\\
\left|\langle \psi(t),\tilde{g}(t)\rangle_{L^2(\tilde{g}(t))}\right|&\leq C\left(\varepsilon(t)^{\beta_0}\gamma(t)^{-1}+\varepsilon(t)^{\tau_0}\gamma(t)^{-1}+\gamma(t)^{-2}+\varepsilon(t)^8\delta(t)^{-10}\right)\\
&\leq C\varepsilon_0(t)^{8-10\delta},\\
\left|\langle \psi(t),\tilde{c}(t)\rangle_{L^2(\tilde{g}(t))}\right|&\leq C\left(\varepsilon(t)^{\beta_0}\gamma(t)^{-1}+\varepsilon(t)^{\tau_0}\gamma(t)^{-1}+\gamma(t)^{-2}+\varepsilon(t)^8\delta(t)^{-10}\right)\\
&\leq C\varepsilon(t)^{8-10\delta}.
\end{split}
\end{equation*}
\end{prop}

\begin{proof}
{ We prove the result for the ancient case, the proof in the immortal case is analogous.}
\begin{claim}\label{claim-psi-T}
$\|\psi\|_{C^{0}_{\gamma,\sigma+2,T}}\leq C\varepsilon_0(T)^{2\left(\frac{\eta_0}{\Lambda_0}-1\right)-\iota\sigma}.$
\end{claim}
\begin{proof}[Proof of Claim \ref{claim-psi-T}]
Since $\|k\|_{C^{2,\alpha}_{\gamma,\sigma,T}}\leq 1$, 
\begin{equation*}
|Q_{\tilde{g}(t)}(k(t))|_{\tilde{g}(t)}\leq C(n)\left(|k(t)|_{\tilde{g}(t)}|\nabla^{\tilde{g}(t),2}k(t)|_{\tilde{g}(t)}+|\nabla^{\tilde{g}(t)}k(t)|^2_{\tilde{g}(t)}\right),
\end{equation*}
 pointwise. Multiplying across the previous estimate with $\gamma(t)(\varepsilon(t)+r_o)^{2+\sigma}$ gives:
\begin{equation*}
\begin{split}
\|Q_{\tilde{g}}(k)\|_{C^{0}_{\gamma,\sigma+2,T}}&\leq \sup_{t\,\leq\, T}\gamma(t)\gamma(t)^{-2}(\varepsilon(t)+r_o)^{-\sigma}\|k\|^2_{C^{2,\alpha}_{\gamma,\sigma,T}}\\
&\leq \sup_{t\,\leq\, T}\gamma(t)^{-1}\varepsilon(t)^{-\sigma},
\end{split}
\end{equation*}
since again, $k$ belongs to the unit ball with respect to the norm of $C^{2,\alpha}_{\gamma,\sigma,T}$.
By choice of $\gamma$ with respect to $\varepsilon$ made in \eqref{choice-parameters}, this gives 
\begin{equation*}
\|Q_{\tilde{g}}(k)\|_{C^{0}_{\gamma,\sigma+2,T}}\leq C\varepsilon_0(T)^{2+\iota\sigma}.
\end{equation*}
Similarly,
\begin{equation*}
\|\beta\cdot k\|_{C^{0}_{\gamma,\sigma+2,T}}\leq C\sup_{t\,\leq \,T}|\beta(t)|.
\end{equation*}
 Since $\beta(t)$ is $O(\varepsilon_0(t)^{\beta_0})$, this gives 
\begin{equation*}
\|\beta\cdot k\|_{C^{0}_{\gamma,\sigma+2,T}}\leq C\varepsilon_0(T)^{\beta_0}.
\end{equation*}

A similar reasoning applies to $\tau(t)\Li_{\nabla^{\tilde{g}(t)}\tilde{v}(t)}k(t)$. Indeed, thanks to Lemma \ref{est-basic-conf}, 
\begin{equation}
\begin{split}\label{k-est-lie-der}
|\Li_{\nabla^{\tilde{g}(t)}\tilde{v}(t)}k(t)|_{\tilde{g}(t)}&\leq C\left(|\nabla^{\tilde{g}(t)}k(t)|_{\tilde{g}(t)}|\nabla^{\tilde{g}(t)}\tilde{v}(t)|_{\tilde{g}(t)}+|k(t)|_{\tilde{g}(t)}|\nabla^{\tilde{g}(t),2}\tilde{v}(t)|_{\tilde{g}(t)}\right)\\
&\leq C\|k\|_{C^{2,\alpha}_{\gamma,\sigma+2,T}}\gamma(t)^{-1}\left((\varepsilon(t)+r_o)^{-\sigma-1}(\varepsilon(t)+r_o)+(\varepsilon(t)+r_o)^{-\sigma}\right)\\
&\leq C\|k\|_{C^{2,\alpha}_{\gamma,\sigma+2,T}}\gamma(t)^{-1}(\varepsilon(t)+r_o)^{-\sigma},
\end{split}
\end{equation}
one easily gets:
\begin{equation*}
\|\tau\cdot\Li_{\nabla^{\tilde{g}}\tilde{v}}k\|_{C^{0}_{\gamma,\sigma+2,T}}\leq C\sup_{t\,\leq \,T}|\tau(t)|.
\end{equation*}
This allows us to conclude that $\|\tau\cdot\Li_{\nabla^{\tilde{g}}\tilde{v}}k\|_{C^{0}_{\gamma\gamma,\sigma+2,T}}\leq C\varepsilon_0(T)^{\tau_0}$ by choice of $\tau$ in \eqref{choice-parameters}.

Now, thanks to Lemma \ref{lemma-time-der-metric},
\begin{equation*}
\begin{split}
\|\partial_t\tilde{g}+2\Ric(\tilde{g})-2\Lambda \tilde{g}\|_{C^{0}_{\gamma,\sigma+2,T}}&\leq \sup_{t\,\leq \,T}\gamma(t)\Bigg(\varepsilon(t)^{2\frac{\eta_0}{\Lambda_0}+\sigma}+\delta(t)^{4+\sigma}+\varepsilon(t)^{4^-}\delta(t)^{\sigma^+}\\
&\quad+\varepsilon(t)^5\delta(t)^{-5+\sigma}+\varepsilon(t)^8\delta(t)^{-8+\sigma}\Bigg)\\
&\leq \sup_{t\,\leq \,T}\gamma(t)\varepsilon(t)^{2\frac{\eta_0}{\Lambda_0}+\sigma},
\end{split}
\end{equation*}
 thanks to \eqref{choice-parameters} which imposes $0<\iota\sigma<2(\eta_0/\Lambda_0-1)<(5-\sigma)(1-\delta)-2$. 
By choice of $\gamma$ with respect to $\delta$ in \eqref{choice-parameters}, we also get 
\begin{equation*}
\|\partial_t\tilde{g}+2\Ric(\tilde{g})-2\Lambda \tilde{g}\|_{C^{0}_{\gamma,\sigma+2,T}}\leq C\varepsilon_0(T)^{2\left(\frac{\eta_0}{\Lambda_0}-1\right)-\iota\sigma}.
\end{equation*}

which leads to the first conclusion, i.e. $\lim_{T\rightarrow+\infty}\|\psi\|_{C^{0}_{\gamma,\sigma+2,T}}=0$.
\end{proof}

\begin{claim}\label{proj-oi-final}
\begin{equation*}
\begin{split}
|\langle \psi(t),\tilde{\oi}_i(t)&\rangle_{L^2(\tilde{g}(t))}+\langle d_{\zeta(t)}g(\dot{\zeta}(t)-2\mathbf{R}_p^{+}(\zeta(t))),\oi_i(\zeta(t))\rangle_{L^2(g_{\zeta(t)})}| \\
&\leq C\left(\gamma(t)^{-2}\varepsilon(t)^{2-2\sigma}+\gamma(t)^{-1}\varepsilon(t)^{4-\sigma+\beta_0}+\gamma(t)^{-1}\varepsilon(t)^{4-\sigma+\tau_0}+\varepsilon(t)^9\delta(t)^{-7}\right)\\
&\leq C\varepsilon_0(t)^{9-7\delta}.
\end{split}
\end{equation*}
\end{claim}

\begin{proof}[Proof of Claim \ref{proj-oi-final}]
Let us now estimate the projection of $\psi$ against $\tilde{\oi}_i(t)$.
By definition of the norm $C^{2,\alpha}_{\gamma,\sigma+2,T}$, one gets:
\begin{equation*}
\begin{split}
\left|\int_M\langle Q_{\tilde{g}(t)}(k(t)),\tilde{\oi}_i(t)\rangle_{\tilde{g}(t)}\,d\mu_{\tilde{g}(t)}\right|&\leq C\gamma(t)^{-2}\int_M(\varepsilon(t)+r_o)^{-2-2\sigma}|\tilde{\oi}_i(t)|_{\tilde{g}(t)}\,d\mu_{\tilde{g}(t)}\\
&\leq C\gamma(t)^{-2}\varepsilon(t)^{2-2\sigma}.
\end{split}
\end{equation*}

Similarly, since $\sigma>0$, 
\begin{equation*}
\begin{split}
\left|\int_M\langle \beta(t)k(t),\tilde{\oi}_i(t)\rangle_{\tilde{g}(t)}\,d\mu_{\tilde{g}(t)}\right|\leq C|\beta(t)|\gamma(t)^{-1}\varepsilon(t)^{4-\sigma}.
\end{split}
\end{equation*}
And similarly, based on \eqref{k-est-lie-der},
\begin{equation*}
\begin{split}
\left|\int_M\langle \tau(t)\Li_{\nabla^{\tilde{g}(t)}\tilde{v}(t)}k(t),\tilde{\oi}_i(t)\rangle_{\tilde{g}(t)}\,d\mu_{\tilde{g}(t)}\right|\leq C|\tau(t)|\gamma(t)^{-1}\varepsilon(t)^{4-\sigma}.
\end{split}
\end{equation*}
As an intermediate step, 
\begin{equation*}
\begin{split}
&\left|\int_M\langle \psi(t)+\left(\partial_t\tilde{g}(t)+2\Ric(\tilde{g}(t))-2\Lambda\tilde{g}(t)\right),\tilde{\oi}_i(t)\rangle_{\tilde{g}(t)}\,d\mu_{\tilde{g}(t)}\right| \\&\leq C\left(\gamma(t)^{-2}\varepsilon(t)^{2-2\sigma}+\gamma(t)^{-1}\varepsilon(t)^{4-\sigma+\beta_0}+\gamma(t)^{-1}\varepsilon(t)^{4-\sigma+\tau_0}\right).
\end{split}
\end{equation*}
Proposition \ref{prop-o_1-first-app} gives in turn:
\begin{equation*}
\begin{split}
&\left|\int_M\langle \psi(t),\tilde{\oi}_i(t)\rangle_{\tilde{g}(t)}\,d\mu_{\tilde{g}(t)}+ \langle d_{\zeta(t)}g(\dot{\zeta}(t)-2\mathbf{R}_p^{+}(\zeta(t))),\oi_i(\zeta(t))\rangle_{L^2(g_{\zeta(t)})}\right|\leq \\
&C\left(\gamma(t)^{-2}\varepsilon(t)^{2-2\sigma}+\gamma(t)^{-1}\varepsilon(t)^{4-\sigma+\beta_0}+\gamma(t)^{-1}\varepsilon(t)^{4-\sigma+\tau_0}+\varepsilon(t)^9\delta(t)^{-7}\right),
\end{split}
\end{equation*}
which is exactly the first desired estimate by definition of $\zeta$ in \eqref{def-varep}.

The second estimate follows directly from this first estimate together with the choice of parameters made in \eqref{choice-parameters} since $2+\iota\sigma>2>2\delta$.
\end{proof}

Let us now estimate the projection of $\psi$ against $\tilde{g}(t)$.
Since $\sigma\in(0,1/2)\subset(0,2)$, the function $(\varepsilon(t)+r_o)^{-2\sigma-2}$ is integrable on $M$ uniformly in time and,
\begin{equation*}
\begin{split}
\left|\langle Q_{\tilde{g}(t)}(k(t)),\tilde{g}(t)\rangle_{L^2(d\mu_{\tilde{g}(t)})}\right|\leq C\gamma(t)^{-2}.
\end{split}
\end{equation*}
Similarly,
\begin{equation*}
\begin{split}
\left|\langle \beta(t)k(t),\tilde{g}(t)\rangle_{L^2(d\mu_{\tilde{g}(t)})}\right|\leq C|\beta(t)|\gamma(t)^{-1},
\end{split}
\end{equation*}
and 
\begin{equation*}
\begin{split}
\left|\int_M\langle \tau(t)\Li_{\nabla^{\tilde{g}(t)}\tilde{v}(t)}k(t),\tilde{g}(t)\rangle_{\tilde{g}(t)}\,d\mu_{\tilde{g}(t)}\right|\leq C|\tau(t)|\gamma(t)^{-1}.
\end{split}
\end{equation*}
To summarize the three previous estimates, we get:
\begin{equation*}
\begin{split}
&\left|\int_M\langle \psi(t)+\left(\partial_t\tilde{g}(t)+2\Ric(\tilde{g}(t))-2\Lambda \tilde{g}(t)\right),\tilde{g}(t)\rangle_{\tilde{g}(t)}\,d\mu_{\tilde{g}(t)}\right| \\
&\leq C\left(\gamma(t)^{-2}+\gamma(t)^{-1}\varepsilon(t)^{\beta_0}+\gamma(t)^{-1}\varepsilon(t)^{\tau_0}\right).
\end{split}
\end{equation*}
Based on Proposition \ref{prop-proj-tilde-g}, we obtain the desired estimate: 
\begin{equation*}
\begin{split}
&\left|\int_M\langle \psi(t),\tilde{g}(t)\rangle_{\tilde{g}(t)}\,d\mu_{\tilde{g}(t)}\right|\leq C\left(\gamma(t)^{-1}\varepsilon(t)^{\beta_0}+\gamma(t)^{-1}\varepsilon(t)^{\tau_0}+\gamma(t)^{-2}+\varepsilon(t)^{10}\delta(t)^{-8}\right).
\end{split}
\end{equation*}

We are finally in a position to estimate the projection of $\psi$ against $\tilde{c}(t)$. Since $\tilde{c}(t)$ is bounded uniformly in time by Lemma \ref{est-basic-conf},
\begin{equation*}
\begin{split}
\left|\int_M\langle Q_{\tilde{g}(t)}(k(t)),\tilde{c}(t)\rangle_{\tilde{g}(t)}\,d\mu_{\tilde{g}(t)}\right|&\leq C\gamma(t)^{-2}\int_M(\varepsilon(t)+r_o)^{-2-2\sigma}|\tilde{c}(t)|_{\tilde{g}(t)}\,d\mu_{\tilde{g}(t)}\\
&\leq C\gamma(t)^{-2}.
\end{split}
\end{equation*}
Similarly, 
\begin{equation*}
\begin{split}
\left|\int_M\langle \beta(t)k(t),\tilde{c}(t)\rangle_{\tilde{g}(t)}\,d\mu_{\tilde{g}(t)}\right|&\leq C|\beta(t)|\gamma(t)^{-1},\\
\left|\int_M\langle \tau(t)\Li_{\nabla^{\tilde{g}(t)}\tilde{v}(t)}k(t),\tilde{c}(t)\rangle_{\tilde{g}(t)}\,d\mu_{\tilde{g}(t)}\right|&\leq C|\tau(t)|\gamma(t)^{-1}.
\end{split}
\end{equation*}
Based on Proposition \ref{prop-proj-conf} and \eqref{choice-parameters} (i.e. $\beta_0$ and $\tau_0$ less than $8-10\delta$), we get the desired estimate: 
\begin{equation*}
\begin{split}
&\left|\int_M\langle \psi(t),\tilde{c}(t)\rangle_{\tilde{g}(t)}\,d\mu_{\tilde{g}(t)}\right|\leq C\left(\gamma(t)^{-1}\varepsilon(t)^{\beta_0}+\gamma(t)^{-1}\varepsilon(t)^{\tau_0}+\gamma(t)^{-2}+\varepsilon(t)^{8}\delta(t)^{-10}\right)\\
&\leq C\varepsilon_0(t)^{8-10\delta}.
\end{split}
\end{equation*}

\end{proof}

Next, we establish the corresponding H\"older estimates of Proposition \ref{prop-final-est-proj}:
\begin{prop}\label{prop-final-est-proj-holder}
For a quadruplet $(k,\eta,\beta,\tau)\in \mathcal{B}$, recall the non-homogenous term:
\begin{equation}\label{defn-psi-bis}
\psi(t):=-Q_{\tilde{g}(t)}(k(t))+\beta(t)k(t)+\tau(t)\Li_{\nabla^{\tilde{g}(t)}\tilde{v}(t)}k(t)-\left(\partial_t\tilde{g}(t)+2\Ric(\tilde{g}(t))-2\Lambda\tilde{g}(t)\right).
\end{equation}
{
Then $\|\psi\|_{C^{0,\alpha}_{\gamma,\sigma+2,T}}\leq C\varepsilon_0(T)^{2\left(\frac{\eta_0}{\Lambda_0}-1\right)-\iota\sigma}$ and,

\begin{equation*}
\begin{split}
&\left[\langle \psi,\tilde{\oi}_i\rangle_{L^2(\tilde{g})}+\langle d_{\zeta}g(\eta)),\oi_i(\zeta)\rangle_{L^2(g_{\zeta})}\right]_{\alpha,t,r}\\
&\leq C\left(\gamma(t)^{-2}\varepsilon(t)^{2-2\sigma}+\left(r^2+\frac{[\dot{\zeta}]_{\alpha,t,r}}{\varepsilon(t)^2}\right)\varepsilon(t)^{9^-}\delta(t)^{-7^-}\right),\\
&\left[\langle \psi,\tilde{g}\rangle_{L^2(\tilde{g})}\right]_{\alpha,t,r}\leq C\gamma(t)^{-1}\varepsilon(t)^{\min\{\beta_0,\tau_0\}}+C\left(r^2+r^{2\alpha}\varepsilon(t)^{-2\alpha}\right)\,\varepsilon(t)^{8}\delta(t)^{-10},\\
&\left[\langle \psi,\tilde{c}\rangle_{L^2(\tilde{g})}\right]_{\alpha,t,r}\leq C\gamma(t)^{-1}\varepsilon(t)^{\min\{\beta_0,\tau_0\}}+C\left(r^2+r^{2\alpha}\varepsilon(t)^{-2\alpha}\right)\,\varepsilon(t)^{8}\delta(t)^{-10}.
\end{split}
\end{equation*}}
\end{prop}

\begin{proof}
Let us first estimate the H\"older norm of $\psi$. Regarding the quadratic term $Q_{\tilde{g}(t)}(k(t))$, we invoke Lemma \ref{app-C-monster} together with Lemma \ref{lemma-gal-control-semi-norm} and Proposition \ref{prop-final-est-proj} to get:
\begin{equation}
\begin{split}\label{est-holder-psi-pt-1}
\|Q_{\tilde{g}}(k)\|_{C^{0,\alpha}_{\gamma,\sigma+2,T}}\leq C\varepsilon_0(T)^{2+\iota\sigma}.
\end{split}
\end{equation}

Now, Lemma \ref{lemma-gal-control-semi-norm} and the fact that $(k,\eta,\beta,\tau)\in \mathcal{B}$ imply:
\begin{equation*}
\begin{split}
\|\beta\cdot k\|_{C^{0,\alpha}_{\gamma,\sigma+2,T}}&\leq C\varepsilon_0(T)^{\beta_0}+\sup_{t\,\leq \,T}\varepsilon(t)^{\beta_0}\|\beta\|_{C^0_{\beta_0,T}}\|k\|_{C^{0,\alpha}_{\gamma,\sigma,T}}\\
&\quad+\sup_{t\,\leq \,T}\varepsilon(t)^{2}\|\beta\|_{C^{0,\alpha}_{\beta_0,T}}\|k\|_{C^{0}_{\gamma,\sigma,T}}\\
&\leq C\varepsilon_0(T)^{\beta_0}.
\end{split}
\end{equation*}
A similar estimate on $\|\tau\cdot\Li_{\nabla^{\tilde{g}}\tilde{v}}k\|_{C^{0,\alpha}_{\gamma,\sigma+2,T}}$ holds:
\begin{equation}
\begin{split}\label{est-holder-psi-pt-2}
\|\tau\cdot\Li_{\nabla^{\tilde{g}}\tilde{v}}k\|_{C^{0,\alpha}_{\gamma,\sigma+2,T}}&\leq C\varepsilon_0(T)^{\tau_0}.
\end{split}
\end{equation}
Invoking Lemma \ref{holder-semi-obstruction} gives:
\begin{equation}
\begin{split}\label{est-holder-psi-pt-3}
\|\partial_t\tilde{g}+2\Ric(\tilde{g})-2\Lambda\tilde{g}\|_{C^{0,\alpha}_{\gamma,\sigma+2,T}}&\leq C\varepsilon_0(T)^{2\left(\frac{\eta_0}{\Lambda_0}-1\right)-\iota\sigma}.
\end{split}
\end{equation}
The combination of estimates \eqref{est-holder-psi-pt-1}, \eqref{est-holder-psi-pt-2} and \eqref{est-holder-psi-pt-3} leads to the expected estimate thanks to \eqref{choice-parameters} that ensures that $2\geq\min\{\beta_0,\tau_0\}>6-10\delta>(5-\sigma)(1-\delta)-2>2(\eta_0/\Lambda_0-1).$ 

As for the H\"older norm of the orthogonal projections under consideration, observe first that thanks to Lemma \ref{lemma-scal-prod-holder}, the fact that $\|k\|_{C^{2,\alpha}_{\gamma,\sigma,T}}\leq1$ and Lemma \ref{estimates variations et laplacien o1},
\begin{equation*}
\begin{split}
\left[\int_M\langle Q_{\tilde{g}}(k),\tilde{\oi}_i\rangle_{\tilde{g}}\,d\mu_{\tilde{g}}\right]_{\alpha,t,r}
&\leq C\gamma(t)^{-2}\left(1+r^{2\alpha}\varepsilon(t)^{-2\alpha}\right)\varepsilon(t)^{2-2\sigma}.
\end{split}
\end{equation*}

Similarly, since $\sigma>0$, the fact that $\|k\|_{C^{2,\alpha}_{\gamma,\sigma,T}}\leq1$,  and Lemma \ref{estimates variations et laplacien o1},
\begin{equation*}
\begin{split}
\left[\int_M\langle \beta\cdot k,\tilde{\oi}_i \rangle_{\tilde{g}}\,d\mu_{\tilde{g}}\right]_{\alpha,t,r}\leq C\gamma(t)^{-1}\varepsilon(t)^{\beta_0+4-\sigma}\left(1+r^{2\alpha}\varepsilon(t)^{-2\alpha}\right)\|\beta\|_{C^{0,\alpha}_{\beta_0,T}}.
\end{split}
\end{equation*}
And similarly, based on \eqref{k-est-lie-der},
\begin{equation*}
\begin{split}
\left[\int_M\langle \tau\cdot\Li_{\nabla^{\tilde{g}}\tilde{v}}k,\tilde{\oi}_i\rangle_{\tilde{g}}\,d\mu_{\tilde{g}}\right]_{\alpha,t,r}\leq C\gamma(t)^{-1}\varepsilon(t)^{\tau_0+4-\sigma}\left(1+r^{2\alpha}\varepsilon(t)^{-2\alpha}\right)\|\tau\|_{C^{0,\alpha}_{\beta_0,T}}.
\end{split}
\end{equation*}

As an intermediate step, the definition of $\psi$ recalled in \eqref{defn-psi-bis} together with the three previous estimates give,
\begin{equation*}
\begin{split}
&\left[\int_M\langle \psi+\left(\partial_t\tilde{g}+2\Ric(\tilde{g})-2\Lambda\tilde{g}\right),\tilde{\oi}_i\rangle_{\tilde{g}}\,d\mu_{\tilde{g}}\right]_{\alpha,t,r} \\&\leq C\left(\gamma(t)^{-2}\varepsilon(t)^{2-2\sigma}+\gamma(t)^{-1}\varepsilon(t)^{\min\{\beta_0,\tau_0\}+4-\sigma}\right)\left(1+r^{2\alpha}\varepsilon(t)^{-2\alpha}\right)\\
&\leq C\gamma(t)^{-2}\varepsilon(t)^{2-2\sigma}\left(1+r^{2\alpha}\varepsilon(t)^{-2\alpha}\right).
\end{split}
\end{equation*}
Here we have used the fact that $\iota\sigma<\min\{\beta_0,\tau_0\}$ in the last inequality.

Thanks to Lemma \ref{lemma-holder-proj-ob}, \eqref{choice-parameters} and the triangular inequality, we get the desired first estimate since $\gamma(t)^{-2}\varepsilon(t)^{2-2\sigma}<\varepsilon(t)^8\delta(t)^{-4}$.

Let us now estimate the projection of $\psi$ against $\tilde{g}(t)$.
Since $\sigma+\alpha\in(0,1)$ thanks to \eqref{choice-parameters}, the function $(\varepsilon(t)+r_o)^{-2-2\sigma-2\alpha}$ is integrable on $M$ uniformly in time and since $\|k\|_{C^{2,\alpha}_{\gamma,\sigma,T}}\leq1$,
\begin{equation*}
\begin{split}
\left[\langle Q_{\tilde{g}}(k),\tilde{g}\rangle_{L^2(\tilde{g})}\right]_{\alpha,t,r}\leq C\gamma(t)^{-2}.
\end{split}
\end{equation*}
Similarly,
\begin{equation*}
\begin{split}
\left[\langle \beta\cdot k,\tilde{g}\rangle_{L^2(\tilde{g})}\right]_{\alpha,t,r}\leq C\gamma(t)^{-1}\varepsilon(t)^{\beta_0}\|\beta\|_{C^{0,\alpha}_{\beta_0,T}},
\end{split}
\end{equation*}
and 
\begin{equation*}
\begin{split}
\left[\int_M\langle \tau\cdot\Li_{\nabla^{\tilde{g}}\tilde{v}}k,\tilde{g}\rangle_{\tilde{g}}\,d\mu_{\tilde{g}}\right]_{\alpha,t,r}\leq C\gamma(t)^{-1}\varepsilon(t)^{\tau_0}\|\tau\|_{C^{0,\alpha}_{\beta_0,T}}.
\end{split}
\end{equation*}
To summarize the three previous estimates, we get thanks to the definition of $\gamma$ in \eqref{choice-parameters}:
\begin{equation*}
\begin{split}
&\left[\int_M\langle \psi+\left(\partial_t\tilde{g}+2\Ric(\tilde{g})-2\Lambda \tilde{g}\right),\tilde{g}\rangle_{\tilde{g}}\,d\mu_{\tilde{g}}\right]_{\alpha,t,r} \\
&\leq C\gamma(t)^{-1}\left(\gamma(t)^{-1}+\varepsilon(t)^{\min\{\beta_0,\tau_0\}}\right)\leq C\gamma(t)^{-1}\varepsilon(t)^{\min\{\beta_0,\tau_0\}}.
\end{split}
\end{equation*}
Based on Lemma \ref{lemma-holder-proj-ob}, we obtain the desired estimate: 
\begin{equation*}
\begin{split}
&\left[\int_M\langle \psi,\tilde{g}\rangle_{\tilde{g}}\,d\mu_{\tilde{g}}\right]_{\alpha,t,r}\leq C\gamma(t)^{-1}\varepsilon(t)^{\min\{\beta_0,\tau_0\}}+C\left(r^2+r^{2\alpha}\varepsilon(t)^{-2\alpha}\right)\,\varepsilon(t)^{8}\delta(t)^{-10}.
\end{split}
\end{equation*}

We are finally in a position to estimate the projection of $\psi$ against $\tilde{c}(t)$. Since $\tilde{c}(t)$ is bounded uniformly in time by Lemma \ref{est-basic-conf}, and based on Lemma \ref{lemma-holder-proj-ob} and \eqref{choice-parameters}, we get the desired estimate: 
\begin{equation*}
\begin{split}
&\left[\int_M\langle \psi,\tilde{c}\rangle_{\tilde{g}}\,d\mu_{\tilde{g}}\right]_{\alpha,t,r}\leq C\gamma(t)^{-1}\varepsilon(t)^{\min\{\beta_0,\tau_0\}}+C\left(r^2+r^{2\alpha}\varepsilon(t)^{-2\alpha}\right)\,\varepsilon(t)^{8}\delta(t)^{-10}.
\end{split}
\end{equation*}

\end{proof}

Next, we analyse each element of the image of a quadruplet $(k,\eta,\beta,\tau)$ under the map $\Phi$ in view of the previous estimates. This is echoing \cite[Corollary $7.2$]{Bre-Kap}.
\begin{coro}\label{coro-final-est}
If $(k,\eta,\beta,\tau)$ lies in the ball $\mathcal{B}$, then $(h,\xi,\nu,\mu):=\Phi(k,\eta,\beta,\tau)$ satisfies:
\begin{enumerate}
\item \label{coro-final-est-h}The $C^{2,\alpha}_{\gamma,\sigma,T}$-norm of $h$ is arbitrarily small as $T$ tends to $-\infty$: $\lim_{T\rightarrow-\infty}\|h\|_{C^{2,\alpha}_{\gamma,\sigma,T}}=0$.\\
\item\label{coro-final-est-lambda}
The coefficients $(\lambda_i)_i$ satisfy for $t\leq T$,
\begin{equation*}
\sum_i|\lambda_i(t)|\leq C\varepsilon_0(t)^{2\left(\frac{\eta_0}{\Lambda_0}-1\right)}.\\
\end{equation*}
\item \label{coro-final-est-xi}The vector $\xi$ satisfies for $t\leq T$,
\begin{equation*}
\begin{split}
&|\xi(t)|\leq C\sum_i\left|\left\langle d_{\zeta(t)}g(\xi(t)),\frac{\oi_i(\zeta(t))}{\|\oi_i(\zeta(t))\|_{L^2(g_{\zeta(t)})}}\right\rangle_{L^2(g_{\zeta(t)})}\right|\leq C\varepsilon_0(t)^{2+(2+\sigma)(1-\delta)+\iota\sigma}.
 \end{split}
\end{equation*}
\item \label{coro-final-est-nu}The function $\nu$ satisfies for $t\leq T$,
\begin{equation*}
\begin{split}
|\nu(t)|&\leq 
 C\varepsilon_0(t)^{8-10\delta}.
\end{split}
\end{equation*} 
\item \label{coro-final-est-mu}The function $\mu$ satisfies for $t\leq T$,
\begin{equation*}
|\mu(t)|\leq C\varepsilon_0(t)^{8-10\delta}.
\end{equation*} 
\end{enumerate}

\end{coro}

\begin{proof}
The estimate on $h$ is straightforward by combining Propositions \ref{prop-not-so-easy-bis} together with Proposition \ref{prop-final-est-proj-holder}:
\begin{equation}
	\|h\|_{C^{0,\alpha}_{\gamma,\sigma,T}}\leq C\|\psi\|_{C^{0,\alpha}_{\gamma,\sigma+2,T}}\rightarrow 0,\quad\text{as $T\rightarrow-\infty$.}\label{bded-h-psi}
\end{equation}

Regarding the penultimate estimate, the combination of Propositions \ref{prop 6.5 F(t)} and \ref{prop-final-est-proj} gives by adding and subtracting the data $\psi$ in the scalar product:
\begin{equation*}
\begin{split}
|\nu(t)|\langle\tilde{g}(t),\tilde{g}(t)\rangle_{L^2(\tilde{g}(t))}&\leq C\gamma(t)^{-1}\varepsilon(t)^{4}\delta(t)^{-4-\sigma}\|h\|_{C^0_{\gamma,\sigma,T}}+C\delta(t)^2\gamma(t)^{-1}\|h\|_{C^0_{\gamma,\sigma,T}}\varepsilon(t)^{4}\delta(t)^{-2-\sigma}\\
&\quad+C\gamma(t)^{-1}\varepsilon(t)^{2-\sigma}\delta(t)^2\|\psi\|_{C^{0}_{\gamma,\sigma+2,T}}+C\varepsilon(t)^{8}\delta(t)^{-10}\\
&\leq C\varepsilon(t)^{8}\delta(t)^{-10},
\end{split}
\end{equation*}
since $\delta>1/2$ thanks to \eqref{choice-parameters}. Here we have used the uniform boundedness of $h$ proved in \eqref{bded-h-psi} and that of $\psi$ established in Proposition \ref{prop-final-est-proj}. We also have made use of Corollary \ref{coro-intermed-est-coeff}
in the first inequality.


As for the estimate on $\mu(t)$, Propositions \ref{prop 6.5 M(t)} and \ref{prop-final-est-proj} together with Corollary \ref{coro-intermed-est-coeff} give by adding and subtracting $\psi$ similarly to the previous estimate on $|\nu(t)|$:
\begin{equation*}
\begin{split}
|\mu(t)|&\langle\tilde{c}(t),\tilde{c}(t)\rangle_{L^2(\tilde{g}(t))}\leq C\gamma(t)^{-1}\varepsilon(t)^4\delta(t)^{-4-\sigma}+C\varepsilon(t)^4\sum_i|\lambda_i(t)|+C\varepsilon(t)^{8}\delta(t)^{-10}\\
&\leq C\gamma(t)^{-1}\varepsilon(t)^4\delta(t)^{-4-\sigma}+C\gamma(t)^{-1}\varepsilon(t)^{4}\delta(t)^{-2-\sigma}+C\gamma(t)^{-1}\varepsilon(t)^{2-\sigma}+C\varepsilon(t)^8\delta(t)^{-10}\\
&\leq C\delta(t)^{8-10\delta},
\end{split}
\end{equation*}
since $\delta>1/2$ thanks to \eqref{choice-parameters}.


As for the estimate on $\xi$, the combination of Propositions \ref{prop 6.5 E(t)} and \ref{prop-final-est-proj} leads to:
\begin{equation}
\begin{split}\label{est-xi-prelim}
\Big|\lambda_i(t)&\|\tilde{\oi}_i(t)\|^2_{L^2(\tilde{g}(t))}-\langle d_{\zeta(t)}g(\dot{\zeta}(t)-2\mathbf{R}_p^{+}(\zeta(t))),\oi_i(\zeta(t))\rangle_{L^2(g_{\zeta(t)})}\Big|\leq\\
& C\gamma(t)^{-1}\|h\|_{C^0_{\gamma,\sigma,T}}\varepsilon(t)^{4}\delta(t)^{-2-\sigma}+C\varepsilon^4(t)\delta(t)^2|\nu(t)|+C\varepsilon(t)^4|\mu(t)|\\
&+C\gamma(t)^{-1}\varepsilon(t)^{4-\sigma}\|\psi\|_{C^{0}_{\gamma,\sigma+2,T}}+C\varepsilon(t)^{9}\delta(t)^{-7}\\
&\leq C\gamma(t)^{-1}\varepsilon(t)^4\delta(t)^{-2-\sigma}+C\gamma(t)^{-1}\varepsilon(t)^{4-\sigma}+C\varepsilon^4(t)\delta(t)^2|\nu(t)|+C\varepsilon(t)^4|\mu(t)|+C\varepsilon(t)^{9}\delta(t)^{-7},
\end{split}
\end{equation}
where we have used the uniform boundedness of $h$ proved in \eqref{bded-h-psi} and that of $\psi$ established in Proposition \ref{prop-final-est-proj}.

Recalling the definition of $\xi$ given in \eqref{def-xi} with respect to $\eta$ given in \eqref{def-varep}, the previously established estimates on $\nu(t)$ and $\mu(t)$ let us conclude:
\begin{equation*}
\begin{split}
\left|\langle d_{\zeta(t)}g(\xi(t)),\oi_i(\zeta(t))\rangle_{L^2(g_{\zeta(t)})}\right|&=\left|\langle d_{\zeta(t)}g(\eta(t)),\oi_i(\zeta(t))\rangle_{L^2(g_{\zeta(t)})}-\lambda_i(t)\|\tilde{\oi}_i(t)\|^2_{L^2(\tilde{g}(t))}\right|\\
&\leq C\varepsilon(t)^4\left(\gamma(t)^{-1}\delta(t)^{-2-\sigma}+\varepsilon(t)^8\delta(t)^{-10}+\varepsilon(t)^{5}\delta(t)^{-7}\right).
\end{split}
\end{equation*}

To show that $\left|\langle d_{\zeta(t)}g(\xi(t)),\oi_i(\zeta(t))\rangle_{L^2(g_{\zeta(t)})}\right|\leq C\varepsilon_0(t)^{4+(2+\sigma)(1-\delta)+\iota\sigma}$, observe that by \eqref{choice-parameters}, the smaller exponent of $\varepsilon_0(t)$ in the previous estimate is $(2+\sigma)(1-\delta)+\iota\sigma<5-7\delta<8-10\delta$ since $(5-\sigma)\delta<3-(1+\iota)\sigma.$

Finally, the estimate on the coefficients $\lambda_i(t)$ follows from the previous estimate on $\xi(t)$ and the triangular inequality:
\begin{equation*}
\begin{split}
\varepsilon_0(t)^4|\lambda_i(t)|&\leq \left|\langle d_{\zeta(t)}g(\xi(t)),\oi_i(\zeta(t))\rangle_{L^2(g_{\zeta(t)})}\right|+\left|\langle d_{\zeta(t)}g(\eta(t)),\oi_i(\zeta(t))\rangle_{L^2(g_{\zeta(t)})}\right|\\
&\leq C\varepsilon_0(t)^{4+(2+\sigma)(1-\delta)+\iota\sigma}+C\varepsilon_0(t)^{4+2\left(\frac{\eta_0}{\Lambda_0}-1\right)}\leq C\varepsilon_0(t)^{4+2\left(\frac{\eta_0}{\Lambda_0}-1\right)},
\end{split}
\end{equation*}
as expected. Here we have made use of the constraint  and $(2+\sigma)(1-\delta)+\iota\sigma>(2+\sigma)(1-\delta)>(5-\sigma)(1-\delta)-2>2(\eta_0/\Lambda_0-1)$ from \eqref{choice-parameters} in the last estimate. 

\end{proof}

\begin{coro}\label{coro-final-est-holder}
If $(k,\eta,\beta,\tau)$ lies in the ball $\mathcal{B}$, then $(h,\xi,\nu,\mu):=\Phi(k,\eta,\beta,\tau)$ satisfies:
\begin{enumerate}
\item\label{coro-final-est-lambda-holder}
The coefficients $(\lambda_i)_i$ satisfy for $t\leq T$,
\begin{equation*}
\sum_i[\lambda_i]_{\alpha,t,\varepsilon}\leq C\varepsilon_0(t)^{2\left(\frac{\eta_0}{\Lambda_0}-1\right)}.\\
\end{equation*}
\item \label{coro-final-est-xi-holder}The vector $\xi$ satisfies for $t\leq T$,
\begin{equation*}
\begin{split}
&[\xi]_{\alpha,t,\varepsilon(t)}\leq C\varepsilon_0(t)^{2+(2+\sigma+2\alpha)(1-\delta)+\iota\sigma}.
\end{split}
\end{equation*}
\item \label{coro-final-est-nu-holder}The function $\nu$ satisfies for $t\leq T$,
\begin{equation*}
\begin{split}
[\nu]_{\alpha,t,1}&\leq C\varepsilon_0(t)^{8-(10+2\alpha)\delta}.
\end{split}
\end{equation*} 
\item \label{coro-final-est-mu-holder} The function $\mu$ satisfies for $t\leq T$,
\begin{equation*}
[\mu]_{\alpha,t,1}\leq  C\varepsilon_0(t)^{8-(10+2\alpha)\delta}.
\end{equation*} 
\end{enumerate}

\end{coro}

\begin{proof}

Regarding the penultimate estimate, the combination of Proposition \ref{prop 6.5 F(t)-holder}, Corollary \ref{coro-intermed-est-coeff-holder} and Proposition \ref{prop-final-est-proj-holder} gives by adding and subtracting the data $\psi$ in the scalar product:
\begin{equation*}
\begin{split}
[\nu(t)\vol_{\tilde{g}(t)}M]_{\alpha,t,1}&\leq C\gamma(t)^{-1}\varepsilon(t)^{4-2\alpha}\delta(t)^{-4-\sigma}\|h\|_{C^{0,\alpha}_{\gamma,\sigma,T}}+C\gamma(t)^{-1}\varepsilon(t)^{2-\sigma}\delta(t)^{2}\|\psi\|_{C^{2,\alpha}_{\gamma,\sigma+2,T}}\\
&\quad+C\varepsilon(t)^{8}\delta(t)^{-10-2\alpha}\\
&\leq C\varepsilon(t)^{8}\delta(t)^{-10-2\alpha},
\end{split}
\end{equation*}
where we have used the uniform boundedness of $h$ proved in \eqref{bded-h-psi} and that of $\psi$ established in Proposition \ref{prop-final-est-proj-holder} together with \eqref{choice-parameters}. 

Similarly, Proposition \ref{prop 6.5 M(t)-holder}, Corollary \ref{coro-intermed-est-coeff-holder} and Proposition \ref{prop-final-est-proj-holder} give:
\begin{equation*}
\begin{split}
[\mu\cdot\langle\tilde{c},\tilde{c}\rangle_{L^2(\tilde{g})}]_{\alpha,t,1}&\leq C\gamma(t)^{-1}\varepsilon(t)^{4-2\alpha}\delta(t)^{-4-\sigma}\|h\|_{C^{0,\alpha}_{\gamma,\sigma,T}}+C\gamma(t)^{-1}\varepsilon(t)^{2-\sigma-2\alpha}\|\psi\|_{C^{0,\alpha}_{\gamma,\sigma+2,T}}\\
&\quad+C\varepsilon(t)^{8}\delta^{-10-2\alpha}\\
&\leq C\varepsilon_0(t)^{8-(10+2\alpha)\delta}.
\end{split}
\end{equation*}


As for the estimate on $\xi$, the combination of Propositions \ref{prop 6.5 E(t)-holder} and \ref{prop-final-est-proj-holder} leads to:
\begin{equation}
\begin{split}\label{est-xi-prelim-holder}
&\left[\lambda_i\|\tilde{\oi}_i\|^2_{L^2(\tilde{g})}-\langle d_{\zeta}g(\eta),\oi_i(\zeta)\rangle_{L^2(g_{\zeta})}\right]_{\alpha,t,\varepsilon}\leq C\left(\gamma(t)^{-2}\varepsilon(t)^{2-2\sigma}+\frac{[\dot{\zeta}]_{\alpha,t,\varepsilon(t)}}{\varepsilon(t)^2}\varepsilon(t)^{9^-}\delta(t)^{-7^-}\right)\\
&\quad+C\gamma(t)^{-1}\left(\varepsilon(t)^{2\alpha}\varepsilon(t)^4\delta(t)^{-2-\sigma-2\alpha}+\varepsilon(t)^{4-\sigma}\right)\|h\|_{C^{0,\alpha}_{\gamma,\sigma,T}}\\
&\quad+C\gamma(t)^{-1}\varepsilon(t)^{6-\sigma}\delta(t)^{-4}\|\psi\|_{C^{0,\alpha}_{\gamma,\sigma+2,T}}+C\varepsilon(t)^4\delta(t)^2[\nu]_{\alpha,t,\varepsilon(t)}+C\varepsilon(t)^4[\mu]_{\alpha,t,\varepsilon(t)}\\
&\leq C\varepsilon_0(t)^{4+(2+\sigma)(1-\delta)+\iota\sigma}+C\varepsilon(t)^4\delta(t)^2[\nu]_{\alpha,t,\varepsilon(t)}+C\varepsilon(t)^4[\mu]_{\alpha,t,\varepsilon(t)},
\end{split}
\end{equation}
where we have used the uniform boundedness of $h$ proved in \eqref{bded-h-psi} and that of $\psi$ established in Proposition \ref{prop-final-est-proj-holder} together with \eqref{choice-parameters} since $9-7\delta>4+(2+\sigma)(1-\delta)+\iota\sigma$. Indeed, this condition is equivalent to $(5-\sigma)\delta<(3-\sigma-\iota\sigma).$

Recalling the definition of $\xi$ given in \eqref{def-xi} with respect to $\eta$ given in \eqref{def-varep}, the previously established estimates on $\nu(t)$ and $\mu(t)$ let us conclude:
\begin{equation*}
\begin{split}
&\left[\langle d_{\zeta}g(\xi),\oi_i(\zeta)\rangle_{L^2(g_{\zeta})}\right]_{\alpha,t,\varepsilon(t)}=\left[\langle d_{\zeta}g(\eta),\oi_i(\zeta)\rangle_{L^2(g_{\zeta})}-\lambda_i\|\tilde{\oi}_i\|^2_{L^2(\tilde{g})}\right]_{\alpha,t,\varepsilon(t)}\leq C\varepsilon_0(t)^{4 +(2+\sigma)(1-\delta)+\iota\sigma},
\end{split}
\end{equation*}
which gives the first desired estimate once we divide by the $L^2$ norm of $\oi_i(\zeta(t))$.

Finally, the H\"older estimate on the coefficients $\lambda_i(t)$ follows from the previous H\"older estimate on $\xi(t)$ and the triangular inequality:
\begin{equation*}
\begin{split}
\left[\lambda_i\|\tilde{\oi}_i\|^2_{L^2(\tilde{g})}\right]_{\alpha,t,\varepsilon(t)}&\leq \left[\langle d_{\zeta}g(\xi),\oi_i(\zeta)\rangle_{L^2(g_{\zeta})}\right]_{\alpha,t,\varepsilon(t)}+\left[\langle d_{\zeta}g(\eta),\oi_i(\zeta)\rangle_{L^2(g_{\zeta})}\right]_{\alpha,t,\varepsilon(t)}\\
&\leq C\varepsilon_0(t)^{4+(2+\sigma)(1-\delta)+\iota\sigma}+C\varepsilon_0(t)^{4+2\left(\frac{\eta_0}{\Lambda_0}-1\right)}\leq C\varepsilon_0(t)^{4+2\left(\frac{\eta_0}{\Lambda_0}-1\right)}.
\end{split}
\end{equation*}
 Here we have used that $[\eta]_{\alpha,t,\varepsilon(t)}\leq C \varepsilon_0(t)^{2\frac{\eta_0}{\Lambda_0}}$ by definition of $\|\eta\|_{C^{0,\alpha}_{\eta_0,T}}$. Here also, we have made use of the constraint $2(\eta_0/\Lambda_0-1)<(5-\sigma)(1-\delta)-2<(2+\sigma)(1-\delta)+\iota\sigma$ from \eqref{choice-parameters} in the last estimate. 

Once the previous estimate is divided by $\|\tilde{\oi}_i(t)\|^2_{L^2(\tilde{g}(t))}$, one gets the expected estimate by invoking Corollary \ref{coro-final-est} and Lemma \ref{prop-easy-l2-proj-holder} on $\big[\|\tilde{\oi}_i\|^2_{L^2(\tilde{g})}\big]_{\alpha,t,\varepsilon(t)}$.
\end{proof}

\subsection{A fixed point argument and choice of the parameters}\label{sec-fixed-point}

In this section, we restrict ourselves to the ancient case, that is the setting of Section \ref{section-def}.

More precisely, we reformulate the previous estimates obtained in Section \ref{section-non-hom-term} in terms of the function space $\mathcal{X}$ and the map $\Phi$ introduced in Section \ref{section-def}.
 Our main result Proposition \ref{prop-prelim-schauder} of this section is that a fixed point of the map $\Phi$ is an ancient solution to the Ricci flow up to rescalings and reparametrizations by a family of diffeomorphisms.

 If $0<\iota'<\iota$, $\alpha'\in(0,\alpha)$, $\beta_0'\in(0,\beta_0)$, and $\eta'_0\in(\Lambda_0,\eta_0)$, denote by $\mathcal{X}'$ the function space defined as in \eqref{defn-fct-space-X}.
 
 \begin{prop}\label{prop-prelim-schauder}
 Let $\zeta_0\in\RR^3\setminus\{0\}$, $\eta_0>\Lambda_0$ such that $(\eta_0-\Lambda_0)|\zeta_0|>c_0$ for some universal positive constant $c_0$ and let $\mathcal{X}$ be defined with parameters satisfying \eqref{choice-parameters}.
The spaces $\mathcal{B}\subset\mathcal{X}$, $\mathcal{X}'$ and the map $\Phi$ satisfy:
\begin{enumerate}
\item The identity map $\Id:\mathcal{X}\rightarrow \mathcal{X}'$ is a compact embedding. In particular, $\mathcal{B}$ endowed with the topology induced by that of $\mathcal{X}$ is a compact subset of $\mathcal{X}'$.\\
\item There exists $T\leq 0$ sufficiently small such that $\Phi(\mathcal{B})\subset \mathcal{B}$. \\
\item The map $\Phi$ is continuous on $\mathcal{B}$ with respect to the topology induced by that of $\mathcal{X}'$.
\end{enumerate}
In particular, there exists a fixed point $(k,\eta,\beta,\tau)$ of $\Phi$ in $\mathcal{B}$, i.e. there exists an ancient solution $g(t):=\tilde{g}(t)+k(t)$ to the following approximate DeTurck Ricci flow:
\begin{equation*}
\begin{split}
\partial_tg(t)&=-2\Ric(g(t))+(2\Lambda +\beta(t))g(t)+\mathcal{L}_{V(t)}(g(t)),\\
g(t)(V(t),\cdot)&:=\frac{\tau(t)}{2}g(t)(\nabla^{\tilde{g}(t)}\tilde{v}(t),\cdot)+B_{\tilde{g}(t)}(g(t)-\tilde{g}(t)).
\end{split}
\end{equation*}

\end{prop}
\begin{proof}
Let $(k_i,\eta_i,\beta_i,\tau_i)_i\in \mathcal{X}$ be a bounded sequence. By rescaling, we can assume that $(k_i,\eta_i,\beta_i,\tau_i)_i$ lies in $ \mathcal{B}$, the unit ball of $\mathcal{X}:=C^{2,\alpha}_{\gamma, \sigma, T}\times C^{0,\alpha}_{\eta_0,T}\times C^{0,\alpha}_{\beta_0,T}\times C^{0,\alpha}_{\tau_0,T}$. Proving the first point is equivalent to showing that each component of the vector $(k_i,\eta_i,\beta_i,\tau_i)_i$ admits a subsequence converging with respect to the corresponding norm. We only give a proof for $C^{0,\alpha}_{\eta_0,T}$, the extraction of a convergent subsequence with respect to the other norms is either identical or can be adapted easily.

  Recall that for $T'<T$, the embedding $\Id:C^{0,\alpha}([T',T])\rightarrow C^{0,\alpha'}([T',T])$ is compact for $\alpha'<\alpha$ so that by a diagonal argument, there exists a subsequence still denoted by $(\eta_i)_i$ converging to a function $\eta$ in $C_{loc}^{0,\alpha'}((-\infty,T])$. Moreover, the function $\eta$ lies in $\mathcal{B}$ by pointwise convergence. It remains to show that $(\eta_i)_i$ converges to $\eta$ with respect to the norm $\|\cdot\|_{C^{0,\alpha'}_{\eta_0',T}}$. Indeed, there exists a positive constant $C$ such that for $T'<2T$, 

  \begin{equation*}
  \begin{split}
  \|\eta-\eta_i\|_{C^{0,\alpha'}_{\eta_0',T'}}&\leq\|\eta\|_{C^{0,\alpha'}_{\eta_0',T'}}+\|\eta_i\|_{C^{0,\alpha'}_{\eta_0',T'}} \leq C \exp^{(\eta_0-\eta'_0)T'}\left(\|\eta\|_{C^{0,\alpha'}_{\eta_0,T}}+\|\eta_i\|_{C^{0,\alpha'}_{\eta_0,T}}\right)\\
  &\leq C \exp^{(\eta_0-\eta'_0)T'}\left(\|\eta\|_{C^{0,\alpha}_{\eta_0,T}}+\|\eta_i\|_{C^{0,\alpha}_{\eta_0,T}}\right)\\
  &\leq 2C \exp^{(\eta_0-\eta'_0)T'}.
  \end{split}
  \end{equation*}
In particular, if $\epsilon>0$ is given, there exists $T'<2T$ such that  $\|\eta-\eta_i\|_{C^{0,\alpha'}_{\eta_0',T'}}\leq \epsilon$
since $\eta_0'<\eta_0$. 

Then there exists an index $i_0$ depending on $T'$ such that for $i\geq i_0$, 
\begin{equation*}
\sup_{t\in[T',T]}\left\{e^{-\eta_0' t}\left(|\eta(t)-\eta_i(t)|+\varepsilon_0(t)^{2\alpha'}\sup_{0\,<\, t-t'\,\leq\, \varepsilon_{0}(t)^2}\frac{|(\eta(t)-\eta(t'))-(\eta_i(t)-\eta_i(t'))|}{|t-t'|^{\alpha'}}\right)\right\}\leq \epsilon,
\end{equation*}
  so that $\|\eta-\eta_i\|_{C^{0,\alpha'}_{\eta_0',T'}}\leq 2\epsilon$ for $i\geq i_0$ as expected.
  
Let us now show that there exists $T\leq 0$ sufficiently small such that $\Phi$ preserves the unit ball $\mathcal{B}$ in order to prove the second point. Let $(k,\eta,\beta,\tau)\in \mathcal{B}$ and denote $\Phi(k,\eta,\beta,\tau)=:(h,\xi,\nu,\mu)$ as defined in Section \ref{section-def}. Item \ref{coro-final-est-h} from Corollary \ref{coro-final-est} ensures that there exists $T_0$ sufficiently small such that for $T\leq T_0$, $\|h\|_{C^{2,\alpha}_{\gamma,\sigma,T}}\leq 1/4.$

Now, [\eqref{coro-final-est-xi}, Corollary \ref{coro-final-est}] and [\eqref{coro-final-est-xi-holder}, Corollary \ref{coro-final-est-holder}] give by definition of the norm $C^{0,\alpha}_{\eta_0,T}$:
\begin{equation*}
\|\xi\|_{C^{0,\alpha}_{\eta_0,T}}\leq C\varepsilon_0(T)^{2+(2+\sigma)(1-\delta)+\iota\sigma-2\frac{\eta_0}{\Lambda_0}}\rightarrow 0,\quad \text{as $T\rightarrow -\infty$.}
\end{equation*}
Here we have made use of the constraint $(2+\sigma)(1-\delta)+\iota\sigma>2\eta_0/\Lambda_0-2$ from \eqref{choice-parameters}. Again, let us consider $T_1\leq T_0$ sufficiently small such that for $T\leq T_1$, $\|\xi\|_{C^{0,\alpha}_{\eta_0,T}}\leq 1/4$.

Finally, [\eqref{coro-final-est-nu}, \eqref{coro-final-est-mu}, Corollary \ref{coro-final-est}] and [\eqref{coro-final-est-nu-holder}, \eqref{coro-final-est-mu-holder}, Corollary \ref{coro-final-est-holder}] give by definition of the norms $C^{0,\alpha}_{\beta_0,T}$ and $C^{0,\alpha}_{\tau_0,T}$:
\begin{equation*}
\|\nu\|_{C^{0,\alpha}_{\beta_0,T}}+\|\mu\|_{C^{0,\alpha}_{\tau_0,T}}\leq C\varepsilon_0(T)^{\min\{8-10\delta-\max\{\beta_0,\tau_0\}\}}\rightarrow 0,\quad \text{as $T\rightarrow -\infty$.}
\end{equation*}
Therefore, there is some $T_2\leq T_1$ sufficiently small such that $\|\nu\|_{C^{0,\alpha}_{\beta_0,T}}+\|\mu\|_{C^{0,\alpha}_{\tau_0,T}}\leq 1/2$ if $T\leq T_2$.

We have then shown that $\|\Phi(k,\eta,\beta,\tau)\|_{\mathcal{X}}\leq 1$ as promised for $T\leq T_2\leq 0$.

Finally, we show the third point: the map $\Phi$ is continuous on $\mathcal{X}$ with respect to the topology induced by that of $\mathcal{X}'$. Let $(k,\eta,\beta,\tau)$ in $\mathcal{B}$ and a sequence $(k_i,\eta_i,\beta_i,\tau_i)_i\in \mathcal{B}$ converging to $(k,\eta,\beta,\tau)$ with respect to $\|\cdot\|_{\mathcal{X}'}.$ Then we know that $(\Phi(k_i,\eta_i,\beta_i,\tau_i))_i$ lies in $\mathcal{B}$ since $\Phi(\mathcal{B})\subset\mathcal{B}$. Since $\Id:\mathcal{X}\rightarrow\mathcal{X}'$ is a compact embedding, there is a subsequence still denoted by $(\Phi(k_i,\eta_i,\beta_i,\tau_i))_i$ that converges to some element $(h,\xi,\nu,\mu)$ lying in $\mathcal{B}$. It then suffices to show that $\Phi(k,\eta,\beta,\tau)=(h,\xi,\nu,\mu)$ since in that case, the sequence $(\Phi(k_i,\eta_i,\beta_i,\tau_i))_i$ will have a unique accumulation point. 

If $(h_i,\xi_i,\nu_i,\mu_i):=\Phi(k_i,\eta_i,\beta_i,\tau_i)$ then let $\tilde{g}_i(t)$ be the corresponding background metric associated to $\eta_i$ with parameter $\zeta_0\in\RR^3\setminus\{0\}$ fixed once and for all. The sequence of metrics $\tilde{g}_i(t)$ converges to $\tilde{g}(t)$ corresponding to the parameter $\eta$ in $C^{2,\alpha'}_{loc}(M\times(-\infty,T])$. Similarly, denote the corresponding approximate kernel by $(\tilde{\oi}^i_j(t))_{1\,\leq\, j\,\leq\, 3}$ and $\tilde{c}_i(t):= \frac{1}{2}\Li_{\nabla^{\tilde{g}_i(t)}\tilde{v}_i(t)}\tilde{g}_i(t)$. Then by construction, the approximate kernel converges to that defined with respect to $\tilde{g}(t)$ in $C^{0,\alpha'}_{loc}(M\times(-\infty,T])$. Let $h_i$ be the unique solution in $C^{2,\alpha}_{\gamma,\sigma,T}$ to
\begin{equation}
\begin{split}\label{eqn-h-psi-sequ}
	&\partial_th_i(t)-\Delta_{L,\tilde{g}_i(t)}h_i(t)-2\Lambda h_i(t)=\sum_j\lambda^i_j(t)\tilde{\oi}^i_j(t)+\nu_i(t)\tilde{g}_i(t)+\mu_i(t)\tilde{c}_i(t)+\psi_i(t),\quad t\leq T,\\
	&\psi_i(t):=-Q_{\tilde{g}_i(t)}(k_i(t))+\beta_i(t)k_i(t)+\frac{1}{2}\tau_i(t)\Li_{\nabla^{\tilde{g}_i(t)}\tilde{v}_i(t)}k_i(t)-\left(\partial_t\tilde{g}_i(t)+2\Ric(\tilde{g}_i(t))-2\Lambda\tilde{g}_i(t)\right).
	\end{split}
	\end{equation}
While the lefthand side of the equation satisfied by $h_i$ converges to $\partial_th(t)-\Delta_{L,\tilde{g}(t)}h(t)-2\Lambda h(t)$ in $C^{0,\alpha'}_{loc}(M\times(-\infty,T])$, the data $\psi_i(t)$ converges to the data $\psi(t)$ defined with respect to $(k,\eta,\beta,\tau)$ in $C^{0,\alpha'}_{loc}(M\times(-\infty,T])$. Moreover, $\psi$ lies in $C^{0,\alpha}_{\gamma,\sigma+2,T}$ thanks to Proposition \ref{prop-final-est-proj} and Proposition \ref{prop-final-est-proj-holder}. Finally, the definition of $\xi_i$ in \eqref{def-xi} in Section \ref{section-def} forces the convergence in $C^{0,\alpha'}_{loc}((-\infty,T])$ of the coefficients $\lambda_j^i(t)$ to those defining $\xi$ with respect to $(k,\eta,\beta,\tau).$ Finally, since each $h_i$ is orthogonal to the approximate kernel defined with respect to $\tilde{g}_i(t)$, $h$ is also orthogonal to the approximate kernel defined with respect to $\tilde{g}(t)$. 
Therefore $h\in C^{2,\alpha}_{\gamma,\sigma,T}$ solves \eqref{eqn-h-psi-sequ} with data $\psi\in C^{0,\alpha}_{\gamma,\sigma+2,T}$ and is orthogonal to the approximate kernel defined with respect to $\tilde{g}(t)$: it must be the unique solution provided by Proposition \ref{prop-not-so-easy-bis} which lets us conclude that $(h,\xi,\nu,\mu):=\Phi(k,\eta,\beta,\tau)$ as expected.

Finally, the Schauder fixed point theorem ensures the existence of a fixed point $(k,\eta,\beta,\tau)$ of $\Phi$ in $\mathcal{B}$. In particular, observe that it implies that $\lambda_i(t)=0$ for all $i$ and $t\leq T$ and that $k$ solves:
\begin{equation*}
\begin{split}
	\partial_tk(t)-\Delta_{L,\tilde{g}(t)}k(t)-2\Lambda k(t)&=\psi(t)+\beta(t)\tilde{g}(t)+\tau(t)\tilde{c}(t)\\
	&=-Q_{\tilde{g}(t)}(k(t))+\beta(t)k(t)+\frac{1}{2}\tau(t)\Li_{\nabla^{\tilde{g}(t)}\tilde{v}(t)}k(t)\\
	&\quad-\left(\partial_t\tilde{g}(t)+2\Ric(\tilde{g}(t))-2\Lambda\tilde{g}(t)\right)+\beta(t)\tilde{g}(t)+\tau(t)\tilde{c}(t)\\
	&=-2\Ric(\tilde{g}(t)+k(t))+2\Ric(\tilde{g}(t))-\Delta_{L,\tilde{g}(t)}k(t)\\
	&\quad+\mathcal{L}_{B_{\tilde{g}(t)}(k(t))}(\tilde{g}(t)+k(t))+\beta(t)k(t)+\frac{1}{2}\tau(t)\Li_{\nabla^{\tilde{g}(t)}\tilde{v}(t)}k(t)\\
	&\quad-\left(\partial_t\tilde{g}(t)+2\Ric(\tilde{g}(t))-2\Lambda\tilde{g}(t)\right)+\beta(t)\tilde{g}(t)+\tau(t)\tilde{c}(t)\\
	&=-\Delta_{L,\tilde{g}(t)}k(t)-\partial_t\tilde{g}(t)-2\Ric(\tilde{g}(t)+k(t))+2\Lambda\tilde{g}(t)\\
	&\quad +\beta(t)(\tilde{g}(t)+k(t))+\mathcal{L}_{V(t)}(\tilde{g}(t)+k(t)),
	\end{split}
	\end{equation*}
	where the vector field $V(t)$ is defined as in the statement of this proposition. This leads to the expected modified normalized DeTurck-Ricci flow.
\end{proof}

\section{Solutions through Picard's fixed point theorem}\label{sec:picard}

This section ensures not only the existence of a fixed point for the map $\Phi$ in the ball $\mathcal{B}$ but also a fast iteration scheme through Picard's fixed point theorem together with a continuous dependence on the parameter $\zeta_0\in\RR^3\setminus\{0\}.$ 

\begin{theo}\label{prop-prelim-picard}
Let $\zeta_0\in\RR^3\setminus\{0\}$, $\eta_0>\Lambda_0$ such that $(\eta_0-\Lambda_0)|\zeta_0|>c_0$ for some universal positive constant $c_0$ and let $\mathcal{X}$ be defined with parameters satisfying \eqref{choice-parameters}.

Then there exists $T$ sufficiently large such that $\Phi:\mathcal{B}\rightarrow \mathcal{B}$ is a contraction: there exists $K\in(0,1)$ such that for every $(k_i,\eta_i,\tau_i,\beta_i)\in \mathcal{B}$, $i=1,2$,
\begin{equation*}
\left\|\Phi(k_2,\eta_2,\beta_2,\tau_2)-\Phi(k_1,\eta_1,\beta_1,\tau_1)\right\|_{\mathcal{X}}\leq K\left\|(k_2,\eta_2,\beta_2,\tau_2)-(k_1,\eta_1,\beta_1,\tau_1)\right\|_{\mathcal{X}}.
\end{equation*}
In particular, there exists $T$ sufficiently large such that $\Phi$ admits a unique fixed point in $\mathcal{B}$, i.e. there exists $T$ sufficiently large such that there exists a solution to the normalized DeTurck Ricci flow with background metric $\tilde{g}(t)$.
\end{theo}

\subsection{Gram matrix of the approximate kernel}\label{sec-gram-app-ker}


We describe the $L^2(\tilde{g}(t))$-projection of an arbitrary symmetric $2$-tensor on the approximate kernel $\tilde{\mathbf{O}}(t)$: we restrict our discussion in the ancient case when $t<T<0$ since the proof is completely analogous in the immortal case.

Given $h(t)\in L^2(\tilde{g}(t))$, define the coefficients $L_i(t)$, $M(t)$ and $N(t)$ to be:
\begin{equation*}
\pi(h(t))=:L_i(t)\tilde{\oi}_i(t)+M(t)\tilde{g}(t)+N(t)\tilde{c}(t),
\end{equation*}
where $\pi$ denotes the $L^2$-projection on the approximate kernel $\tilde{\mathbf{O}}(t)$ with respect to $\tilde{g}(t).$
Given $h\in C^{0}_{\gamma,\sigma,T}$, we recall first how to express the coefficients $L_i(t)$, $M(t)$ and $N(t)$ in terms of $h$ and the approximate kernel spanned by $(\tilde{\oi}_i(t))_i$, $\tilde{g}(t)$ and $\tilde{c}(t)$. Let us denote by $G(t)$ the Gram matrix of the set of vectors $(\tilde{\oi}_i(t))_i$ and $\tilde{g}(t)$ so that for $t\leq T$:
\begin{equation}\label{link-G-L-scal}
G(t)\begin{pmatrix}L_1(t)\\L_2(t)\\\vdots\\M(t)\\N(t)
\end{pmatrix}=\begin{pmatrix}\langle h,\tilde{\oi}_1(t)\rangle_{L^2(\tilde{g}(t))}\\\langle h,\tilde{\oi}_2(t)\rangle_{L^2(\tilde{g}(t))}\\\vdots\\\langle h,\tilde{g}(t)\rangle_{L^2(\tilde{g}(t))}\\\langle h,\tilde{c}(t)\rangle_{L^2(\tilde{g}(t))}
\end{pmatrix}.
\end{equation}
The next lemma gathers properties of the above Gram matrix:

\begin{lemma}\label{gram-matrix-expression-lemma}
The Gram matrix satisfies:
\begin{equation}
\begin{split}\label{gram-matrix-expression}
&G(t)=\\
&\begin{pmatrix}
2\pi^2\varepsilon(t)^4+O(\varepsilon(t)^8\delta(t)^{-4})  & O(\varepsilon(t)^8\delta(t)^{-4}) & \dots  &O(\varepsilon(t)^4\delta(t)^2)&O(\varepsilon(t)^4) \\
    O(\varepsilon(t)^8\delta(t)^{-4}) & 2\pi^2\varepsilon(t)^4+O(\varepsilon(t)^8\delta(t)^{-4})  &   \dots &O(\varepsilon(t)^4\delta(t)^2) &O(\varepsilon(t)^4)\\
    \vdots & \vdots & \ddots & \vdots&\vdots \\
    O(\varepsilon(t)^4\delta(t)^2) & O(\varepsilon(t)^4\delta(t)^2) & \dots & 4\vol_{\tilde{g}(t)}M&0\\
    O(\varepsilon(t)^4)&O(\varepsilon(t)^4)&\dots&0&\|\tilde{c}(t)\|_{L^2(\tilde{g}(t)}^2\end{pmatrix}.
    \end{split}
    \end{equation}
    \end{lemma}
    \
    \begin{proof}
The proof of \eqref{gram-matrix-expression} is a combination of Proposition \ref{prop-o_1-first-app} and Proposition \ref{prop-proj-conf}.
\end{proof}
\begin{rk}\label{rk hk ALE sec 9}
	If the elements $(\oi_i(\zeta))_i$ of $\ker_{L^2}(\Delta_{L,g_{\zeta}})$ are not $L^2(g_{\zeta})$-orthogonal, as in the case of general hyperkähler ALE metrics (see Section \ref{sec-kro} for their description), the resulting matrix may be more complicated. We would need to orthonormalize the appropriate natural basis before applying the techniques of this section to more general Ricci-flat ALE metrics.
\end{rk}
    The next lemma gathers properties of the inverse of the Gram matrix:

    \begin{lemma}\label{est-inverse-Gram-lemma}
    The inverse of the Gram matrix $G(t)^{-1}$ satisfies:
    \begin{equation}
\begin{split}\label{est-inverse-Gram}
&G(t)^{-1}=\\
&\begin{pmatrix}\frac{1+O(\varepsilon(t)^4\delta(t)^{-4})}{\|\tilde{\oi}_1(t)\|^{2}_{L^2(\tilde{g}(t))}}&O(\delta(t)^{-4})&\dots&O(\delta(t)^2)&O(1)\\
O(\delta(t)^{-4})&\frac{1+O(\varepsilon(t)^4\delta(t)^{-4})}{\|\tilde{\oi}_2(t)\|^{2}_{L^2(\tilde{g}(t))}}&\ddots&O(\delta(t)^2)&O(1)\\
 \vdots & \vdots & \ddots & \vdots &\vdots\\
    O(\delta(t)^2) & O(\delta(t)^2) & \dots & \frac{1+O(\varepsilon(t)^4\delta(t)^4)}{4\vol_{\tilde{g}(t)}M}&O(\varepsilon(t)^4\delta(t)^2)\\
    O(1)&O(1)&\dots&O(\varepsilon(t)^4\delta(t)^2)& \frac{1+O(\varepsilon(t)^4)}{\|\tilde{c}(t)\|^2_{L^2(\tilde{g}(t))}},
 \end{pmatrix}
\end{split}
   \end{equation}
   or equivalently, 
\begin{equation}
\begin{split}\label{est-inverse-Gram-pre}
L_i(t)&=(1+O(\varepsilon(t)^4\delta(t)^{-4}))\frac{\langle h,\tilde{\oi}_i(t)\rangle_{L^2(\tilde{g}(t))}}{\|\tilde{\oi}_i(t)\|^2_{L^2(\tilde{g}(t))}}+O(\varepsilon(t)^4\delta(t)^{-4})\sum_{j\neq i}\frac{\langle h,\tilde{\oi}_j(t)\rangle_{L^2(\tilde{g}(t))}}{\|\tilde{\oi}_j(t)\|^2_{L^2(\tilde{g}(t))}}\\
&\quad+O(\delta(t)^2)\langle h,\tilde{g}(t)\rangle_{L^2(\tilde{g}(t))}+O(1)\langle h,\tilde{c}(t)\rangle_{L^2(\tilde{g}(t))},\\
M(t)&=(1+O(\varepsilon(t)^4\delta(t)^4))\frac{\langle h,\tilde{g}(t)\rangle_{L^2(\tilde{g}(t))}}{4\vol_{\tilde{g}(t)}M}+O(\varepsilon(t)^4\delta(t)^2)\sum_{ i}\frac{\langle h,\tilde{\oi}_i(t)\rangle_{L^2(\tilde{g}(t))}}{\|\tilde{\oi}_i(t)\|^2_{L^2(\tilde{g}(t))}}\\&\quad+O(\varepsilon(t)^4\delta(t)^2)\langle h,\tilde{c}(t)\rangle_{L^2(\tilde{g}(t))},\\
N(t)&=(1+O(\varepsilon(t)^4)\frac{\langle h,\tilde{c}(t)\rangle_{L^2(\tilde{g}(t))}}{\|\tilde{c}(t)\|^2_{L^2(\tilde{g}(t))}}+O(\varepsilon(t)^4)\sum_{ i}\frac{\langle h,\tilde{\oi}_i(t)\rangle_{L^2(\tilde{g}(t))}}{\|\tilde{\oi}_i(t)\|^2_{L^2(\tilde{g}(t))}}\\&\quad+O(\varepsilon(t)^4\delta(t)^2)\langle h,\tilde{g}(t)\rangle_{L^2(\tilde{g}(t))}.
\end{split}
\end{equation}
In particular, 
\begin{equation}\label{control-L^2-proj}
\varepsilon(t)^2\sum_i|L_i(t)|+|M(t)|+|N(t)|\leq C\|\pi(h)\|_{L^2(\tilde{g}(t))},
\end{equation}
and,
\begin{equation}
\begin{split}\label{est-Li-g}
\sum_i|L_i(t)|&\leq C\left(\varepsilon(t)^{-4}\sum_{j}|\langle h,\tilde{\oi}_j(t)\rangle_{L^2(\tilde{g}(t))}|+|\langle h,\tilde{g}(t)\rangle_{L^2(\tilde{g}(t))}|+|\langle h,\tilde{c}(t)\rangle_{L^2(\tilde{g}(t))}|\right)\\
&\leq C\gamma(t)^{-1}(\varepsilon(t)^{-\sigma}+1)\|h\|_{C^0_{\gamma,\sigma,T}}.
\end{split}
\end{equation}

\end{lemma}
\begin{proof}

Thanks to \eqref{gram-matrix-expression-lemma}, one gets:
\begin{equation*}
\begin{split}
L_i(t)&=\frac{\langle h,\tilde{\oi}_i(t)\rangle_{L^2(\tilde{g}(t))}}{\|\tilde{\oi}_i(t)\|^2_{L^2(\tilde{g}(t))}}+O(\varepsilon(t)^4\delta(t)^{-4})\sum_{j\neq i}L_j(t)+O(\delta(t)^2)M(t)+O(1)N(t),\\
M(t)&=\frac{\langle h,\tilde{g}(t)\rangle_{L^2(\tilde{g}(t))}}{4\vol_{\tilde{g}(t)}M}+O(\varepsilon(t)^4\delta(t)^2)\sum_{ i}L_i(t),\\
N(t)&=\frac{\langle h,\tilde{c}(t)\rangle_{L^2(\tilde{g}(t))}}{\|\tilde{c}(t)\|^2_{L^2(\tilde{g}(t))}}+O(\varepsilon(t)^4)\sum_{ i}L_i(t).
\end{split}
\end{equation*}
This implies \eqref{est-inverse-Gram-pre} which can be summarized as \eqref{est-inverse-Gram}. As for \eqref{control-L^2-proj}, it is a direct consequence of \eqref{est-inverse-Gram-pre}.

Regarding the final estimate, once the previous estimate \eqref{est-inverse-Gram} on the inverse of $G(t)$ is invoked, the result follows by using Lemmata \ref{lemma-easy-peasy}, \ref{lemma-easy-peasy-bis} and \ref{lemma-easy-peasy-ter}. 
\end{proof}
The next lemma gathers properties of the derivative of the Gram matrix:
\begin{lemma}\label{est-der-Gram-matrix-lemma}
The first derivative of the Gram matrix satisfies:
\begin{equation}
\begin{split}\label{est-der-Gram-matrix}
&G'(t)=\\
&\begin{pmatrix}
4\pi^2\langle\dot{\zeta}(t),\zeta(t)\rangle+O\left(\frac{\varepsilon(t)^8}{\delta(t)^{4}}\right)  & O\left(\frac{\varepsilon(t)^8}{\delta(t)^{4}}\right) & \dots  &O(\varepsilon(t)^{4^-}\delta(t)^{2^+})&O(\varepsilon(t)^4) \\
    O\left(\frac{\varepsilon(t)^8}{\delta(t)^{4}}\right) & 4\pi^2\langle\dot{\zeta}(t),\zeta(t)\rangle+O\left(\frac{\varepsilon(t)^8}{\delta(t)^{4}}\right)  &   \dots &O(\varepsilon(t)^{4^-}\delta(t)^{2^+}) &O(\varepsilon(t)^4)\\
    \vdots & \vdots & \ddots & \vdots&\vdots \\
    O(\varepsilon(t)^{4^-}\delta(t)^{2^+}) & O(\varepsilon(t)^{4^-}\delta(t)^{2^+}) & \dots & O\left(\frac{\varepsilon(t)^8}{\delta(t)^{8}}\right)&0\\
    O(\varepsilon(t)^4)&O(\varepsilon(t)^4)&\dots&0&O\left(\frac{\varepsilon(t)^{4^-}}{\delta(t)^{4^-}}\right)
    \end{pmatrix}.
    \end{split}
\end{equation}
\end{lemma}
\begin{proof}
Observe that Corollary \ref{time-der-gram}, Proposition \ref{prop-proj-tilde-g}, Proposition \ref{prop-o_1-first-app} and Proposition \ref{prop-proj-conf} lead to the proof of \eqref{est-der-Gram-matrix}.
\end{proof}


The last technical lemma of this section is stated for clarity of the formulas that will be needed in the next section. It can be skipped until it is being invoked.
\begin{lemma}\label{hello-ugly-gram-lemma}
The following estimates hold:
\begin{equation}
\begin{split}\label{hello-ugly-gram}
\frac{d}{dt}&\langle h,\tilde{\oi}_i\rangle_{L^2(\tilde{g})}=\|\tilde{\oi}_i\|^2_{L^2(\tilde{g}(t))}\dot{L_i}(t)+\left\langle \tilde{g}(t),\tilde{\oi}_i(t)\right\rangle_{L^2(\tilde{g}(t))}\dot{M}(t)+\left\langle \tilde{c}(t),\tilde{\oi}_i(t)\right\rangle_{L^2(\tilde{g}(t))}\dot{N}(t)\\
&\quad+O(\varepsilon(t)^8\delta(t)^{-4})\sum_{j\neq i}|\dot{L_j}(t)|+O(\varepsilon(t)^4)\sum_i|L_i(t)|+O(\varepsilon(t)^{4^-}\delta(t)^{2^+})|M(t)|+O(\varepsilon(t)^4)|N(t)|,\\
\frac{d}{dt}&\langle h,\tilde{g}\rangle_{L^2(\tilde{g})}=\left(4\vol_{\tilde{g}(t)}M\right)\dot{M}(t)+O(\varepsilon(t)^4\delta(t)^2)\sum_i|\dot{L_i}(t)|\\
&\quad+O(\varepsilon(t)^{4^-}\delta(t)^{2^+})\sum_i|L_i(t)|+O(\varepsilon(t)^8\delta(t)^{-8})|M(t)|,\\
\frac{d}{dt}&\langle h,\tilde{c}\rangle_{L^2(\tilde{g})}=\|\tilde{c}(t)\|_{L^2(\tilde{g}(t)}^2\dot{N}(t)+O(\varepsilon(t)^4)\sum_i|\dot{L_i}(t)|\\
&\quad+O(\varepsilon(t)^{4})\sum_i|L_i(t)|+O(\varepsilon(t)^{4^-}\delta(t)^{-4^-})|N(t)|,
  \end{split}
\end{equation}
and,
\begin{equation}
\begin{split}\label{hello-ugly-gram-bis-a}
&\Bigg|\frac{d}{dt}\langle h,\tilde{\oi}_i\rangle_{L^2(\tilde{g})}-\dot{L_i}(t)\|\tilde{\oi}_i(t)\|^2_{L^2(\tilde{g}(t))}-\dot{M}(t)\left\langle \tilde{g}(t),\tilde{\oi}_i(t)\right\rangle_{L^2(\tilde{g}(t))}-\dot{N}(t)\left\langle \tilde{c}(t),\tilde{\oi}_i(t)\right\rangle_{L^2(\tilde{g}(t))}\Bigg|\leq\\
&\quad+C\varepsilon(t)^4\delta(t)^{-4}\left|\frac{d}{dt}\sum_i\langle h,\tilde{\oi}_i\rangle_{L^2(\tilde{g})}\right|+C\varepsilon(t)^8\delta(t)^{-2}\left|\frac{d}{dt}\langle h,\tilde{g}\rangle_{L^2(\tilde{g})}\right|\\
&\quad+C\varepsilon(t)^8\delta(t)^{-4}\left|\frac{d}{dt}\langle h,\tilde{c}\rangle_{L^2(\tilde{g})}\right|+C\varepsilon(t)^2\|\pi(h(t))\|_{L^2(\tilde{g}(t))}.
  \end{split}
\end{equation}

\end{lemma}
\begin{proof}
The estimates in \eqref{hello-ugly-gram} can be proved by differentiating the vector-valued relation \eqref{link-G-L-scal} together with [\eqref{est-der-Gram-matrix}, Lemma \ref{est-der-Gram-matrix-lemma}].

Summing the first line of \eqref{hello-ugly-gram} over $i$ gives by absorption:
\begin{equation*}
\begin{split}\label{hello-ugly-gram-bis}
&\sum_i|\dot{L_i}(t)|\|\tilde{\oi}_i(t)\|^2_{L^2(\tilde{g}(t))}\leq \left|\frac{d}{dt}\sum_i\langle h,\tilde{\oi}_i\rangle_{L^2(\tilde{g})}\right|+|\dot{M}(t)|\sum_i|\left\langle \tilde{g}(t),\tilde{\oi}_i(t)\right\rangle_{L^2(\tilde{g}(t))}|\\
&+|\dot{N}(t)|\sum_i|\left\langle \tilde{c}(t),\tilde{\oi}_i(t)\right\rangle_{L^2(\tilde{g}(t))}|+ C\varepsilon(t)^4\sum_i|L_i(t)|+C\varepsilon(t)^{4^-}\delta(t)^{2^+}|M(t)|+C\varepsilon(t)^4|N(t)|.
  \end{split}
\end{equation*}
Now, according to Proposition \ref{prop-o_1-first-app}, Proposition \ref{prop-proj-conf} and [\eqref{control-L^2-proj}, Lemma \ref{est-inverse-Gram-lemma}], 
\begin{equation*}
\begin{split}\label{hello-ugly-gram-bis-0}
\sum_i|\dot{L_i}(t)|\|\tilde{\oi}_i\|^2_{L^2(\tilde{g}(t))}&\leq \left|\frac{d}{dt}\sum_i\langle h,\tilde{\oi}_i\rangle_{L^2(\tilde{g})}\right|+C\varepsilon(t)^4\delta(t)^2|\dot{M}(t)|\\
&\quad+C\varepsilon(t)^4|\dot{N}(t)|+ C\varepsilon(t)^2\|\pi(h(t))\|_{L^2(\tilde{g}(t))}.
  \end{split}
\end{equation*}
Therefore, feeding back the previous estimate with \eqref{hello-ugly-gram}, one gets by absorption, 
\begin{equation}
\begin{split}\label{hello-ugly-gram-bis-1}
\varepsilon(t)^4\sum_i|\dot{L_i}(t)|&\leq C\left|\frac{d}{dt}\sum_i\langle h,\tilde{\oi}_i\rangle_{L^2(\tilde{g})}\right|+C\varepsilon(t)^4\delta(t)^2\left|\frac{d}{dt}\langle h,\tilde{g}\rangle_{L^2(\tilde{g})}\right|+C\varepsilon(t)^4\left|\frac{d}{dt}\langle h,\tilde{c}\rangle_{L^2(\tilde{g})}\right|\\
&\quad+C\varepsilon(t)^2\|\pi(h(t))\|_{L^2(\tilde{g}(t))}.
  \end{split}
\end{equation}
Finally, invoking the first estimate of \eqref{hello-ugly-gram} together with \eqref{hello-ugly-gram-bis-1} leads to the proof of \eqref{hello-ugly-gram-bis-a}.
\end{proof}
\subsection{A stability ODE lemma and analytic estimates on the approximate kernel}
We begin this section by stating the following stability lemma whose proof is along the same lines as those of Lemma \ref{lemma-control-zeta-bis} and is therefore omitted:
	
	\begin{lemma}[Stability ODE lemma]\label{sta-ode-lemma}
Assume either the setting of \ref{iteam-a-ode} or that of \ref {iteam-b-ode} of Lemma \ref{lemma-control-zeta-bis}. Let $\eta_i$, $i=1,2$ be in $C^{0,\alpha}_{\eta_0,T}$. For each $i=1,2$, denote by $\zeta_i$ the corresponding solution to \eqref{def-varep-bis} with the same initial condition $\zeta_0$. Then there exists $C>0$ such that for all $t\leq T\leq 0$:
\begin{equation*}
\begin{split}
\|\zeta_2-\zeta_1\|_{C^{0}_{\Lambda_0,T}}+\|\dot{\zeta_2}-\dot{\zeta_1}\|_{C^{0,\alpha}_{\Lambda_0,T}}&\leq C\varepsilon_0(T)^{2\frac{\eta_0}{\Lambda_0}-2}\|\eta_2-\eta_1\|_{C^{0,\alpha}_{\eta_0,T}}.
\end{split}
\end{equation*}
\end{lemma}
Notice that a similar stability result to Lemma \ref{sta-ode-lemma} holds as well for forward solutions in the setting of Section \ref{sec-imm-sol}.

The following result gives pointwise estimates for the covariant and time derivatives of the difference of basis of two approximate kernels $\tilde{\mathbf{O}}_i(t):=\operatorname{vect}((\tilde{\oi}_j^i(t))_{j},\tilde{g}_i(t),\tilde{c}_i(t))$, $t\leq T$, (see Definition \ref{defn-app-kernel}) defined with respect to two parameters $\zeta_i\in\RR^3$, $i=1,2$ close to a fixed background $\zeta_0\in \RR^3\setminus\{0\}$. Its proof is a combination of Proposition \ref{prop-analy} and Lemma \ref{sta-ode-lemma}.

\begin{prop}\label{prop-est-diff-analy-app-ker}
Under the setting and notations of Theorem \ref{prop-prelim-picard}, let $\tilde{g}_0(t)$ be defined with respect to $\zeta_0(t):=\exp(2\mathbf{R}_p^+t)\zeta_0$ for $t\leq T$, $l\in\mathbb{N}$ and $m=0,1$,
\begin{equation*}
\begin{split}
&|\nabla^{\tilde{g}_0(t),\,l}\partial_t^{(m)}(\tilde{g}_2(t)-\tilde{g}_1(t))|_{\tilde{g}_0(t)}\leq C\varepsilon_0(T)^{2\frac{\eta_0}{\Lambda_0}-2}\|\eta_2-\eta_1\|_{C^{0,\alpha}_{\eta_0,T}}\frac{\varepsilon_0(t)^{4^-}}{\rho(t)^{4^-+l}},\\
&|\nabla^{\tilde{g}_0(t),\,l}\partial_t^{(m)}(\tilde{c}_2(t)-\tilde{c}_1(t))|_{\tilde{g}_0(t)}\leq C\varepsilon_0(T)^{2\frac{\eta_0}{\Lambda_0}-2}\|\eta_2-\eta_1\|_{C^{0,\alpha}_{\eta_0,T}}\frac{\varepsilon_0(t)^{4^-}}{\rho(t)^{4^-+l}},\\
&|\nabla^{\tilde{g}_0(t),\,l}\partial_t^{(m)}(\tilde{\oi}_2(t)-\tilde{\oi}_1(t))|_{\tilde{g}_0(t)}\leq C\varepsilon_0(T)^{2\frac{\eta_0}{\Lambda_0}-2}\|\eta_2-\eta_1\|_{C^{0,\alpha}_{\eta_0,T}}\frac{\varepsilon_0(t)^{4^-}}{\rho(t)^{4^-+l}}\mathbbm{1}_{\{r_o\,\leq\,2\delta(t)\}},\\
&|\Delta_{L,\tilde{g}_2(t)}\tilde{\oi}_2(t)-\Delta_{L,\tilde{g}_1(t)}\tilde{\oi}_1(t))|_{\tilde{g}_0(t)}\leq \\
&C\varepsilon_0(T)^{2\frac{\eta_0}{\Lambda_0}-2}\|\eta_2-\eta_1\|_{C^{0,\alpha}_{\eta_0,T}}\left(\frac{\varepsilon_0(t)^{4^-}}{\rho(t)^{4^-}}\mathbbm{1}_{\{r_o\,\leq\,2\delta(t)\}}+\frac{\varepsilon_0(t)^{4^-}}{\delta(t)^{6^-}}\mathbbm{1}_{\{\delta(t)/2\,\leq\,r_o\,\leq\,2\delta(t)\}}\right).
\end{split}
\end{equation*}
\end{prop}
\subsection{Proof of Theorem \ref{prop-prelim-picard}}

We are now in a position to prove Theorem \ref{prop-prelim-picard} through a series of claims. \begin{proof}[Proof of Theorem \ref{prop-prelim-picard}]
Let $(k_i,\eta_i,\beta_i,\tau_i)\in \mathcal{B}$ and let $(h_i,\xi_i,\nu_i,\mu_i):=\Phi(k_i,\eta_i,\beta_i,\tau_i)$  for $i=1,2$. Observe that there is a conflict of notations with the map $\zeta\rightarrow h_{2,\zeta}$ from Proposition \ref{prop-h2-H2}. However, since we do not use formally this map in the proof below, there is no source of confusion.

With the same notations as those in the proof of Proposition \ref{prop-prelim-schauder}, observe that $h_2-h_1$ satisfies:
\begin{equation*}
\begin{split}
&\left(\partial_t-\Delta_{L,\tilde{g}_1(t)}-2\Lambda\right)(h_2-h_1)(t)=\\
&\Psi(t)+\sum_j\left(\lambda^2_j(t)-\lambda^1_j(t)\right)\tilde{\oi}_j^1(t)
+(\nu_2(t)-\nu_1(t))\tilde{g}_1(t)+(\mu_2(t)-\mu_1(t))\tilde{c}_1(t),\\
\Psi(t)&:=\psi_2(t)-\psi_1(t)+\left[\Delta_{L,\tilde{g}_2(t)}-\Delta_{L,\tilde{g}_1(t)}\right]h_2(t)\\
&\quad+\sum_j\lambda_j^2(t)(\tilde{\oi}_j^2(t)-\tilde{\oi}_j^1(t))
 +\nu_2(t)(\tilde{g}_2(t)-\tilde{g}_1(t))+\mu_2(t)(\tilde{c}_2(t)-\tilde{c}_1(t)).
\end{split}
\end{equation*}
If $\Psi$ lied in $C^{0,\alpha}_{\gamma,\sigma+2,T}$ then Proposition \ref{prop-not-so-easy-bis} would ensure that the norm of $h_2-h_1$ is well-controlled in the $C^{ 2,\alpha}_{\gamma,\sigma,T}$ norm \textit{on the condition} that $h_2$ is orthogonal to the approximate kernel defined with respect to $\tilde{g}_1(t)$. To remedy this issue, denote the $L^2(\tilde{g}_1(t))$-orthogonal projection of $h_2$ on the approximate kernel defined with respect to $\tilde{g}_1(t)$ by $\pi_1(h_2)$ so that:
\begin{equation}
\begin{split}\label{def-psi-Psi}
&\left(\partial_t-\Delta_{L,\tilde{g}_1(t)}-2\Lambda\right)(h_2-\pi_1(h_2)-h_1)(t)=\\
&\psi(t)+\sum_j\left(\lambda^2_j(t)-\lambda^1_j(t)\right)\tilde{\oi}_j^1(t)+(\nu_2(t)-\nu_1(t))\tilde{g}_1(t)+(\mu_2(t)-\mu_1(t))\tilde{c}_1(t),\\
\psi(t)&:=\Psi(t)-\left(\partial_t-\Delta_{L,\tilde{g}_1(t)}-2\Lambda\right)\pi_1(h_2).
\end{split}
\end{equation}
\begin{claim}\label{claim-control-pi1h2}
The projection $\pi_1(h_2)$ lies in $C^{0,\alpha}_{\gamma,\sigma+2,T}$ and,
\begin{equation*}
\begin{split}
\|\pi_1(h_2)\|_{C^{2,\alpha}_{\gamma,\sigma,T}}&\leq C\varepsilon_0(T)^{2\frac{\eta_0}{\Lambda_0}-2}\left\|(k_2,\eta_2,\beta_2,\tau_2)-(k_1,\eta_1,\beta_1,\tau_1)\right\|_{\mathcal{X}} ,\\
\left\|\left(\partial_t-\Delta_{L,\tilde{g}_1(t)}-2\Lambda\right)\pi_1(h_2)\right\|_{C^{0,\alpha}_{\gamma,\sigma+2,T}}&\leq C\varepsilon_0(T)^{2\frac{\eta_0}{\Lambda_0}-2}\left\|(k_2,\eta_2,\beta_2,\tau_2)-(k_1,\eta_1,\beta_1,\tau_1)\right\|_{\mathcal{X}}.
\end{split}
\end{equation*}
In particular,  if $K\in(0,1)$, there exist $T\leq 0$ such that,
\begin{equation*}
\left\|\left(\partial_t-\Delta_{L,\tilde{g}_1(t)}-2\Lambda\right)\pi_1(h_2)\right\|_{C^{0,\alpha}_{\gamma,\sigma+2,T}}\leq K\left\|(k_2,\eta_2,\beta_2,\tau_2)-(k_1,\eta_1,\beta_1,\tau_1)\right\|_{\mathcal{X}}.
\end{equation*}
\end{claim}

\begin{proof}[Proof of Claim \ref{claim-control-pi1h2}]
Let us first start by estimating the scalar products of $h_2$ against the approximate kernel defined with respect to $\tilde{g}_1(t)$ by using that $h_2$ is orthogonal to the approximate kernel defined with respect to $\tilde{g}_2(t)$. 
\begin{equation}
\begin{split}\label{scal-oj-h2}
\left|\langle h_2(t),\tilde{\oi}^1_j(t)\rangle_{L^2(\tilde{g}_1(t))}\right|&\leq \left|\langle h_2(t),\tilde{\oi}^1_j(t)-\tilde{\oi}^2_j(t)\rangle_{L^2(\tilde{g}_1(t))}\right|\\
&\quad+\left|\langle h_2(t),\tilde{\oi}^2_j(t)\rangle_{L^2(\tilde{g}_1(t))}-\langle h_2(t),\tilde{\oi}^2_j(t)\rangle_{L^2(\tilde{g}_2(t))}\right|\\
&\leq C\varepsilon_0(T)^{2\frac{\eta_0}{\Lambda_0}-2}\|\eta_2-\eta_1\|_{C^0_{\eta_0,T}}\gamma(t)^{-1}\varepsilon_0(t)^{4-\sigma},
\end{split}
\end{equation}
where we have used that $h_2$ lies in $\mathcal{B}$ thanks to Proposition \ref{prop-prelim-schauder} and where we have used Lemma \ref{conv-lemma-diff-scal-prod}. Here, the terms $\tilde{g}_2(t)-\tilde{g}_1(t)$ and $\tilde{\oi}_2(t)-\tilde{\oi}_1(t)$ have been estimated with the help of Proposition \ref{prop-est-diff-analy-app-ker}. Similarly, invoking Proposition \ref{prop-est-diff-analy-app-ker} leads to:
\begin{equation}\label{scal-g-c-h2}
\left|\langle h_2(t),\tilde{g}_1(t)\rangle_{L^2(\tilde{g}_1(t))}\right|+\left|\langle h_2(t),\tilde{c}_1(t)\rangle_{L^2(\tilde{g}_1(t))}\right|\leq C\varepsilon_0(T)^{2\frac{\eta_0}{\Lambda_0}-2}\|\eta_2-\eta_1\|_{C^0_{\eta_0,T}}\gamma(t)^{-1}\varepsilon_0(t)^{4-\sigma}.
\end{equation}

Thanks to estimates [\eqref{est-inverse-Gram-pre}, Lemma \ref{est-inverse-Gram-lemma}], \eqref{scal-oj-h2} and \eqref{scal-g-c-h2}:

 \begin{equation*}
 \|\pi_1(h_2)\|_{C^0_{\gamma,\sigma,T}}\leq C\varepsilon_0(T)^{2\frac{\eta_0}{\Lambda_0}-2}\|\eta_2-\eta_1\|_{C^0_{\eta_0,T}},
 \end{equation*}
 since $\frac{\eta_0}{\Lambda_0}-1>0$ by choice of $\eta_0$ in \eqref{choice-parameters}.
 Moreover, 
 \begin{equation}\label{claim-moreover-1}
 \|\pi_1(h_2)\|_{C^0_{\gamma,\sigma+2,T}}\leq C\varepsilon_0(T)^{2\frac{\eta_0}{\Lambda_0}-2}\varepsilon_0(T)^2\|\eta_2-\eta_1\|_{C^0_{\eta_0,T}},
 \end{equation}
 since again $2\left(\frac{\eta_0}{\Lambda_0}-1\right)>0$ by \eqref{choice-parameters}. 
 
 Since the coefficients are space-independent, the corresponding estimates will hold for covariant derivatives of $\pi_1(h_2)$. Indeed, if $\pi_1(h_2)=:L_i(t)\tilde{\oi}_i^1(t)+M_1(t)\tilde{g}_1(t)+N_1(t)\tilde{c}_1(t)$, Lemma \ref{estimates variations et laplacien o1}, Lemma \ref{est-basic-conf} and [\eqref{est-Li-g}, Lemma \ref{est-inverse-Gram-lemma}] together with \eqref{scal-g-c-h2} give for $k=0,1,2$:
 \begin{equation*}
 \begin{split}
& \gamma(t)(\varepsilon(t)+r_o)^{k+\sigma}\left|\nabla^{\tilde{g}_1(t),\,k}\pi_1(h_2)\right|_{\tilde{g}_1(t)}\\
&=\gamma(t)(\varepsilon(t)+r_o)^{k+\sigma}\left|\nabla^{\tilde{g}_1(t),\,k}\left(L_i(t)\tilde{\oi}^1_i(t)+N_1(t)\tilde{c}_1(t)\right)\right|_{\tilde{g}_1(t)}\\
&\leq C\varepsilon_0(T)^{2\frac{\eta_0}{\Lambda_0}-2}\left(1+\varepsilon(t)^{4-\sigma}\delta(t)^{-4+\sigma}\right)\|\eta_2-\eta_1\|_{C^0_{\eta_0,T}}\\
&\leq C \varepsilon_0(T)^{2\frac{\eta_0}{\Lambda_0}-2}\|\eta_2-\eta_1\|_{C^0_{\eta_0,T}}.
\end{split}
\end{equation*}
A similar reasoning leads to
\begin{equation}\label{claim-moreover-2}
 \|\Delta_{L,\tilde{g}_1}\pi_1(h_2)\|_{C^0_{\gamma,\sigma+2,T}}\leq C\varepsilon_0(T)^{2\frac{\eta_0}{\Lambda_0}-2}\|\eta_2-\eta_1\|_{C^0_{\eta_0,T}}.
 \end{equation}
 
 We are therefore left with estimating time-derivatives of $\pi_1(h_2)$. Let us take care of the term $L_i(t)\partial_t\tilde{\oi}_i^1(t)+M_1(t)\partial_t\tilde{g}_1(t)+N_1(t)\partial_t\tilde{c}_1(t)$ first. Invoking Lemma \ref{lemma-time-der-metric}, Lemma \ref{estimates variations et laplacien o1} and Lemma \ref{est-basic-conf}  together with \eqref{scal-oj-h2} and \eqref{scal-g-c-h2} give:
 \begin{equation}
 \begin{split}\label{partial-time-1-part}
 \gamma(t)(\varepsilon(t)+r_o)^{2+\sigma}&\left|L_i(t)\partial_t\tilde{\oi}_i^1(t)+M_1(t)\partial_t\tilde{g}_1(t)+N_1(t)\partial_t\tilde{c}_1(t)\right|_{\tilde{g}_1(t)}\\
 &\leq C\varepsilon_0(T)^{2\frac{\eta_0}{\Lambda_0}-2}\varepsilon_0(t)^{2}\|\eta_2-\eta_1\|_{C^0_{\eta_0,T}}\\
 &\leq C\varepsilon_0(T)^{2\frac{\eta_0}{\Lambda_0}}\|\eta_2-\eta_1\|_{C^0_{\eta_0,T}},\quad t\leq T.
 \end{split}
 \end{equation}
 Here in the second line, we have only kept the minimum exponent of $\varepsilon_0$.
Let us estimate $\dot{L_i}(t)\tilde{\oi}_i^1(t)+\dot{M_1}(t)\tilde{g}_1(t)+\dot{N_1}(t)\tilde{c}_1(t)$ through the Gram matrix $G(t)$ introduced in \eqref{link-G-L-scal}:
\begin{equation}\label{time-der-L'}
\begin{pmatrix}\dot{L_1}(t)\\\dot{L_2}(t)\\\vdots\\\dot{M_1}(t)\\\dot{N_1}(t)
\end{pmatrix}=G(t)^{-1}\frac{d}{dt}\begin{pmatrix}\langle h_2(t),\tilde{\oi}^1_1(t)\rangle_{L^2(\tilde{g}_1(t))}\\\langle h_2(t),\tilde{\oi}^1_2(t)\rangle_{L^2(\tilde{g}_1(t))}\\\vdots\\\langle h_2(t),\tilde{g}_1(t)\rangle_{L^2(\tilde{g}_1(t))}\\\langle h_2(t),\tilde{c}_1(t)\rangle_{L^2(\tilde{g}_1(t))}
\end{pmatrix}-G(t)^{-1}G'(t)\begin{pmatrix}L_1(t)\\L_2(t)\\\vdots\\M_1(t)\\N_1(t)
\end{pmatrix}, \quad t\leq T.
\end{equation}

Observe first that thanks to [\eqref{est-inverse-Gram}, Lemma \ref{est-inverse-Gram-lemma}] and [\eqref{est-der-Gram-matrix}, Lemma \ref{est-der-Gram-matrix-lemma}], the coefficients of the matrix $G(t)^{-1}G'(t)$ are at most uniformly bounded since $\|\tilde{\oi}_i^j(t)\|_{L^2(\tilde{g}_j(t))}$ are in $\varepsilon(t)^4$ for $j\in\{1,2\}$. Therefore, the second term of the righthand side of \eqref{time-der-L'} can be estimated pointwise in time as follows:

\begin{equation}\label{intermed-norm-matrix-der}
\begin{split}
\left|G(t)^{-1}G'(t)\begin{pmatrix}L_1(t)\\L_2(t)\\\vdots\\M_1(t)\\N_1(t)
\end{pmatrix}\right|&\leq C\sum_j|L_j(t)|+|N_1(t)|+|M_1(t)|\\
&\leq C\varepsilon_0(T)^{2\frac{\eta_0}{\Lambda_0}-2}\|\eta_2-\eta_1\|_{C^0_{\eta_0,T}}\gamma(t)^{-1}\varepsilon_0(t)^{-\sigma}.
\end{split}
\end{equation}
Here we have used \eqref{scal-oj-h2} and \eqref{scal-g-c-h2} in the second inequality.

In order to estimate the time derivatives of the scalar products $\langle h_2(t),\tilde{\oi}^1_i(t)\rangle_{L^2(\tilde{g}_1(t))}$, we proceed as in \eqref{scal-oj-h2}. Schematically, by Lemma \ref{easy-time-der-abs-2-tensors-prelim},
\begin{equation*}
\begin{split}
&\left|\frac{d}{dt}\langle h_2,\tilde{\oi}^1_i\rangle_{L^2(\tilde{g}_1)}\right|=\left|\frac{d}{dt}\left(\langle h_2,\tilde{\oi}^1_i\rangle_{L^2(\tilde{g}_1)}-\langle h_2,\tilde{\oi}^2_i\rangle_{L^2(\tilde{g}_2)}\right)\right|\leq\\ &C\int_M|h_2(t)|_{\tilde{g}_1(t)}|\tilde{\oi}^2_i(t)|_{\tilde{g}_1(t)}\left(|\partial_t\tilde{g}_1(t)|_{\tilde{g}_1(t)}+|\partial_t\tilde{g}_2(t)|_{\tilde{g}_1(t)}\right)|\tilde{g}_2(t)-\tilde{g}_1(t)|_{\tilde{g}_1(t)}\,d\mu_{\tilde{g}_1(t)}\\
&+C\int_M|h_2(t)|_{\tilde{g}_1(t)}|\tilde{\oi}^2_i(t)|_{\tilde{g}_1(t)}|\partial_t(\tilde{g}_2(t)-\tilde{g}_1(t))|_{\tilde{g}_1(t)}\,d\mu_{\tilde{g}_1(t)}\\
&+C\int_M|h_2(t)|_{\tilde{g}_1(t)}|\partial_t\tilde{\oi}^2_i(t)|_{\tilde{g}_1(t)}|\tilde{g}_2(t)-\tilde{g}_1(t)|_{\tilde{g}_1(t)}\,d\mu_{\tilde{g}_1(t)}\\
&+C\int_M|\partial_th_2(t)|_{\tilde{g}_1(t)}\left(|\tilde{\oi}^2_i(t)|_{\tilde{g}_1(t)}|\tilde{g}_2(t)-\tilde{g}_1(t)|_{\tilde{g}_1(t)}+|\tilde{\oi}^2_i(t)-\tilde{\oi}^1_i(t)|_{\tilde{g}_1(t)}\right)\,d\mu_{\tilde{g}_1(t)}\\
&+C\int_M|h_2(t)|_{\tilde{g}_1(t)}\left(|\partial_t(\tilde{\oi}^2_i(t)-\tilde{\oi}^1_i(t))|_{\tilde{g}_1(t)}+|\partial_t\tilde{g}_1(t)|_{\tilde{g}_1(t)}|\tilde{\oi}^2_i(t)-\tilde{\oi}^1_i(t)|_{\tilde{g}_1(t)}\right)\,d\mu_{\tilde{g}_1(t)}.
\end{split}
\end{equation*} 
The ``worse'' term above is the term in the fourth integral on the righthand side involving the time derivative $\partial_th_2$.

Using that $h_2$ lies in $\mathcal{B}$ together with estimates from Lemma \ref{lemma-time-der-metric} and Lemma \ref{estimates variations et laplacien o1} combined with Proposition \ref{prop-est-diff-analy-app-ker}:
\begin{equation}
\begin{split}\label{time-der-h2-oi}
&\left|\frac{d}{dt}\langle h_2,\tilde{\oi}^1_i\rangle_{L^2(\tilde{g}_1)}\right|\leq C\varepsilon_0(T)^{2\left(\frac{\eta_0}{\Lambda_0}-1	\right)}\gamma(t)^{-1}\varepsilon_0(t)^{2-\sigma}\|\eta_2-\eta_1\|_{C^0_{\eta_0,T}}.
\end{split}
\end{equation}
 Similarly, by invoking Lemma \ref{lemma-time-der-metric} and Lemma \ref{est-basic-conf}  successively together with Proposition \ref{prop-est-diff-analy-app-ker}, 
 \begin{equation}
\begin{split}\label{time-der-h2-g-c}
\left|\frac{d}{dt}\langle h_2,\tilde{g}_1\rangle_{L^2(\tilde{g}_1)}\right|&+\left|\frac{d}{dt}\langle h_2,\tilde{c}_1\rangle_{L^2(\tilde{g}_1)}\right|\leq C\varepsilon_0(T)^{2\left(\frac{\eta_0}{\Lambda_0}-1	\right)}\gamma(t)^{-1}\varepsilon_0(t)^{4^-}\delta(t)^{-6^--\sigma}\|\eta_2-\eta_1\|_{C^0_{\eta_0,T}}.
\end{split}
\end{equation}
The difference of estimate \eqref{time-der-h2-g-c} with \eqref{time-der-h2-oi} lies in the fact that the (difference of) tensors are bot supported on the bubble region $\{r_o\,\leq\,2\delta(t)\}$ anymore.


Then the expression of $G(t)^{-1}$ given in [\eqref{est-inverse-Gram}, Lemma \ref{est-inverse-Gram-lemma}], in combination with \eqref{intermed-norm-matrix-der} and \eqref{time-der-h2-oi} and \eqref{time-der-h2-g-c}  lead to:
\begin{equation}
\begin{split}\label{partial-time-2-part}
\gamma(t)&(\varepsilon(t)+r_o)^{2+\sigma}\left|\dot{L_i}(t)\tilde{\oi}_i^1(t)+\dot{M_1}(t)\tilde{g}_1(t)+\dot{N_1}(t)\tilde{c}_1(t)\right|_{\tilde{g}_1(t)}\\
&\leq C\left(\left(\varepsilon_0(t)^{\sigma-2}+1\right)\varepsilon_0(t)^{2-\sigma}+\varepsilon_0(t)^{4^-}\delta(t)^{-6^--\sigma}\right)\varepsilon_0(t)^{2\left(\frac{\eta_0}{\Lambda_0}-1	\right)}\|\eta_2-\eta_1\|_{C^0_{\eta_0,T}}\\
&\leq C\varepsilon_0(t)^{2\left(\frac{\eta_0}{\Lambda_0}-1	\right)}\|\eta_2-\eta_1\|_{C^0_{\eta_0,T}},
\end{split}
\end{equation}
thereby proving the desired intermediate estimate, i.e. if $K\in(0,1)$ there exists $ T_0\leq 0$ such that for $T\leq T_0$, $\|\pi_1(h_2)\|_{C^{2}_{\gamma,\sigma,T}}\leq K\|\eta_2-\eta_1\|_{C^0_{\eta_0,T}}$.

A similar reasoning leads to the existence of $T_0\leq 0$ for a given $K\in(0,1)$ such that for $T\leq T_0$, $\|\partial_t\pi_1(h_2)\|_{C^{0}_{\gamma,\sigma+2,T}}\leq K\|\eta_2-\eta_1\|_{C^0_{\eta_0,T}}$ thanks to the choice of parameters made in \eqref{choice-parameters}.

Moreover, the combination of \eqref{claim-moreover-1}, \eqref{claim-moreover-2}, \eqref{partial-time-1-part} and \eqref{partial-time-2-part} gives the quantitative decay on $\left\|\left(\partial_t-\Delta_{L,\tilde{g}_1(t)}-2\Lambda\right)\pi_1(h_2)\right\|_{C^{0}_{\gamma,\sigma+2,T}}$ by inspecting the choice of parameters \eqref{choice-parameters}. Corresponding H\"older estimates hold and are left to the reader.
\end{proof}
\begin{claim}\label{claim-control-psi}
The data $\psi$ lies in $C^{0,\alpha}_{\gamma,\sigma+2,T}$ and 
\begin{equation*}
\|\psi\|_{C^{0,\alpha}_{\gamma,\sigma+2,T}}\leq \varepsilon_0(T)^{2\left(\frac{\eta_0}{\Lambda_0}-1\right)-\iota\sigma}\left\|(k_2,\eta_2,\beta_2,\tau_2)-(k_1,\eta_1,\beta_1,\tau_1)\right\|_{\mathcal{X}}.
\end{equation*}
In particular, if $K\in(0,1)$, there exist $T\leq 0$ such that,
\begin{equation*}
\|\psi\|_{C^{0,\alpha}_{\gamma,\sigma+2,T}}\leq K\left\|(k_2,\eta_2,\beta_2,\tau_2)-(k_1,\eta_1,\beta_1,\tau_1)\right\|_{\mathcal{X}}.
\end{equation*}
\end{claim}

\begin{proof}[Proof of Claim \ref{claim-control-psi}]
According to Claim \ref{claim-control-pi1h2} and the definition \eqref{def-psi-Psi} of $\psi$, it suffices to estimate $\Psi$ in the $C^{0,\alpha}_{\gamma,\sigma+2,T}$ norm. We start with estimating the difference $\psi_2-\psi_1$ appearing in the expression of $\Psi$. To do so, observe that:
\begin{equation*}
\begin{split}
\|\beta_2\cdot k_2&-\beta_1\cdot k_1\|_{C^{0,\alpha}_{\gamma,\sigma+2,T}}\\
&\leq \sup_{t\,\leq \,T}\left(\varepsilon_0(t)^{\beta_0}\right)\left(\|\beta_2\|_{C^{0,\alpha}_{\beta_0,T}}\|k_2-k_1\|_{C^{0,\alpha}_{\gamma,\sigma,T}}+\|k_1\|_{C^{0,\alpha}_{\gamma,\sigma,T}}\|\beta_2-\beta_1\|_{C^{0,\alpha}_{\beta_0,T}}\right)\\
&\leq C\varepsilon_0(T)^{\beta_0}\left\|(k_2,\eta_2,\beta_2,\tau_2)-(k_1,\eta_1,\beta_1,\tau_1)\right\|_{\mathcal{X}}.
\end{split}
\end{equation*}
Here we have used Lemma \ref{lemma-gal-control-semi-norm} and that $(k_1,\eta_1,\beta_1,\tau_1)$ lies in $\mathcal{B}$.

Similarly,
\begin{equation*}
\begin{split}
\|\tau_2\cdot \Li_{\nabla^{\tilde{g}_2}\tilde{v}_2}k_2-\tau_1\cdot \Li_{\nabla^{\tilde{g}_1}\tilde{v}_1}k_1\|_{C^{0,\alpha}_{\gamma,\sigma+2,T}}&\leq \|(\tau_2-\tau_1)\cdot \Li_{\nabla^{\tilde{g}_2}\tilde{v}_2}k_2\|_{C^{0,\alpha}_{\gamma,\sigma+2,T}}\\
&\quad+\|\tau_1\cdot\Li_{\nabla^{\tilde{g}_2}\tilde{v}_2-\nabla^{\tilde{g}_1}\tilde{v}_1}k_2\|_{C^{0,\alpha}_{\gamma,\sigma+2,T}}\\
&\quad+\|\tau_1 \Li_{\nabla^{\tilde{g}_1}\tilde{v}_1}(k_2-k_1)\|_{C^{0,\alpha}_{\gamma,\sigma+2,T}}\\
&\leq C\varepsilon_0(T)^{\tau_0}\left\|(k_2,\eta_2,\beta_2,\tau_2)-(k_1,\eta_1,\beta_1,\tau_1)\right\|_{\mathcal{X}}.
\end{split}
\end{equation*}
Now, the difference $Q_{\tilde{g}_2(t)}(k_2(t))-Q_{\tilde{g}_1(t)}(k_1(t))$ can be handled schematically as:
\begin{equation*}
\begin{split}
\|Q_{\tilde{g}_2(t)}(k_2(t))&-Q_{\tilde{g}_1(t)}(k_1(t))\|_{C^{0,\alpha}_{\gamma,\sigma+2,T}}\\
&\leq C\sup_{t\,\leq \,T}\left(\gamma(t)\gamma(t)^{-2}\varepsilon_0(t)^{-\sigma}\right)\left(\|k_1\|_{C^{2,\alpha}_{\gamma,\sigma,T}}+\|k_2\|_{C^{2,\alpha}_{\gamma,\sigma,T}}\right)\|k_2-k_1\|_{C^{2,\alpha}_{\gamma,\sigma,T}}\\
&\quad+C\sup_{t\,\leq \,T}\left(\gamma(t)\gamma(t)^{-2}\varepsilon_0(t)^{-\sigma}\right)\left(\|k_1\|_{C^{2,\alpha}_{\gamma,\sigma,T}}+\|k_2\|_{C^{2,\alpha}_{\gamma,\sigma,T}}\right)^2\|\tilde{g}_2-\tilde{g}_1\|_{C^{2,\alpha}_{1,0,T}}\\
&\leq C\varepsilon_0(T)^{2+\iota\sigma}\|k_2-k_1\|_{C^{2,\alpha}_{\gamma,\sigma,T}}+C\varepsilon_0(T)^{2+\iota\sigma+2\frac{\eta_0}{\Lambda_0}-2}\|\eta_2-\eta_1\|_{C^{0,\alpha}_{\eta_0,T}}\\
&\leq C\varepsilon_0(T)^{2+\iota\sigma}\left(\|k_2-k_1\|_{C^{2,\alpha}_{\gamma,\sigma,T}}+\|\eta_2-\eta_1\|_{C^{0,\alpha}_{\eta_0,T}}\right).
\end{split}
\end{equation*}
Here we have used Lemma \ref{app-C-monster} in the first inequality together with Proposition \ref{prop-est-diff-analy-app-ker} and \eqref{choice-parameters} in the last line. 

In the same vein, the difference $\Delta_{L,\tilde{g}_2(t)}h_1(t)-\Delta_{L,\tilde{g}_1(t)}h_1(t)$ can be estimated as:
\begin{equation*}
\begin{split}
\|\Delta_{L,\tilde{g}_2}h_2-\Delta_{L,\tilde{g}_1}h_2\|_{C^{0,\alpha}_{\gamma,\sigma+2,T}}&\leq C\varepsilon_0(T)^{2\left(\frac{\eta_0}{\Lambda_0}-1\right)}\|h_2\|_{C^{2,\alpha}_{\gamma,\sigma,T}}\|\tilde{g}_2-\tilde{g}_1\|_{C^{2,\alpha}_{1,0,T}} \\
&\leq C\varepsilon_0(T)^{2\left(\frac{\eta_0}{\Lambda_0}-1\right)}\|\eta_2-\eta_1\|_{C^{0,\alpha}_{\eta_0,T}}.
\end{split}
\end{equation*}

Here again we have used that $2\left(\frac{\eta_0}{\Lambda_0}-1\right)>0$ thanks to \eqref{choice-parameters} together with Proposition \ref{prop-est-diff-analy-app-ker}.

It remains to estimate the difference of the two approximate kernels weighted by the coefficients $(\lambda_j^2)_j$, $\nu_2$ and $\mu_2$. Thanks to [\eqref{coro-final-est-lambda}, Corollary  \ref{coro-final-est}]
 and Proposition \ref{prop-est-diff-analy-app-ker},
\begin{equation*}
\begin{split}
\left\|\sum_j\lambda_j^2(\tilde{\oi}_j^2-\tilde{\oi}_j^1)\right\|_{C^{0}_{\gamma,\sigma+2,T}}&\leq C\sup_{t\,\leq \,T}\left(\gamma(t)\varepsilon_0(t)^{2+\sigma+2\left(\frac{\eta_0}{\Lambda_0}-1\right)}\right)\varepsilon_0(T)^{2\left(\frac{\eta_0}{\Lambda_0}-1\right)}\|\eta_2-\eta_1\|_{C^0_{\eta_0,T}}\\
&\leq C\varepsilon_0(T)^{4\left(\frac{\eta_0}{\Lambda_0}-1\right)-\iota\sigma}\|\eta_2-\eta_1\|_{C^0_{\eta_0,T}}\\
&\leq C\varepsilon_0(T)^{2\left(\frac{\eta_0}{\Lambda_0}-1\right)}\|\eta_2-\eta_1\|_{C^0_{\eta_0,T}},
\end{split}
 \end{equation*}
 since $2\left(\frac{\eta_0}{\Lambda_0}-1\right)>\iota\sigma$ by \eqref{choice-parameters}.
 
 Along the same lines, based on [\eqref{coro-final-est-nu}, Corollary  \ref{coro-final-est}] together with Proposition \ref{prop-est-diff-analy-app-ker},
 \begin{equation*}
\begin{split}
\left\|\nu_2\cdot(\tilde{g}_2-\tilde{g}_1)\right\|_{C^{0}_{\gamma,\sigma+2,T}}&\leq C\sup_{t\,\leq \,T}\left(\gamma(t)\varepsilon_0(t)^{8-10\delta+2+\sigma}\right)\varepsilon_0(T)^{2\left(\frac{\eta_0}{\Lambda_0}-1\right)}\|\eta_2-\eta_1\|_{C^0_{\eta_0,T}}\\
&\leq C\varepsilon_0(T)^{8-10\delta-\iota\sigma+2\left(\frac{\eta_0}{\Lambda_0}-1\right)}\|\eta_2-\eta_1\|_{C^0_{\eta_0,T}}.
\end{split}
 \end{equation*}
 A similar estimate holds for the term $\mu_2(t)(\tilde{c}_2(t)-\tilde{c}_1(t))$.
 Since $\iota\sigma< 8-10\delta$ by \eqref{choice-parameters}, the result follows. 
 
 Finally, invoking Proposition \ref{prop-analy} together with the proofs of Lemmata \ref{lemma-time-der-metric} and \ref{holder-semi-obstruction} lead to:
\begin{equation*}
\Big\|\partial_t(\tilde{g}_2-\tilde{g}_1)+2(\Ric(\tilde{g}_2)-\Ric(\tilde{g}_1))-2\Lambda(\tilde{g}_2-\tilde{g}_1)\Big\|_{C^{0}_{\gamma,\sigma+2,T}}\leq C\varepsilon_0(T)^{2\left(\frac{\eta_0}{\Lambda_0}-1\right)-\iota\sigma}\|\eta_2-\eta_1\|_{C^0_{\eta_0,T}}.
\end{equation*}
 
 Corresponding H\"older estimates hold and are left to the reader.
 \end{proof}
The combination of Claims \ref{claim-control-pi1h2} and \ref{claim-control-psi} together with Proposition \ref{prop-not-so-easy-bis} give us $T\leq 0$ such that:
\begin{equation}
\begin{split}\label{quantitative-est-h-psi}
\|h_2-h_1\|_{C^{2,\alpha}_{\gamma,\sigma,T}}&\leq C\|\psi\|_{C^{0,\alpha}_{\gamma,\sigma+2,T}}+\|\pi_1(h_2)\|_{C^{2,\alpha}_{\gamma,\sigma,T}}\\
&\leq C\varepsilon_0(T)^{2\left(\frac{\eta_0}{\Lambda_0}-1\right)-\iota\sigma}\left\|(k_2,\eta_2,\beta_2,\tau_2)-(k_1,\eta_1,\beta_1,\tau_1)\right\|_{\mathcal{X}},
\end{split}
\end{equation}
thanks to the triangular inequality used in the first inequality.

We are left with proving that the remaining components $\xi$, $\nu$ and $\mu$ are $K$-Lipschitz for any $K\in(0,1)$ if $T\leq 0$ is small enough. Thanks to Corollary \ref{coro-intermed-est-coeff}, the choice of parameters made in \eqref{choice-parameters} and estimates \eqref{quantitative-est-h-psi} and Claims \ref{claim-control-pi1h2} and \ref{claim-control-psi} again, we have:
\begin{equation}
\begin{split}\label{est-mu-nu-2}
\|\nu_2-\nu_1\|_{C^0_{\eta_0,T}}&\leq C\varepsilon_0(T)^{6+\sigma-(4+\sigma)\delta+\iota\sigma-\frac{2\eta_0}{\Lambda_0}}\|h_2-\pi_1(h_2)-h_1\|_{C^0_{\gamma,\sigma,T}}\\
&\quad+C\varepsilon_0(T)^{2+(1+\iota)\sigma-\frac{2\eta_0}{\Lambda_0}}\|\psi\|_{C^0_{\gamma,\sigma+2,T}}\\
&\leq C\varepsilon_0(T)^{(4+\sigma)(1-\delta)}\left\|(k_2,\eta_2,\beta_2,\tau_2)-(k_1,\eta_1,\beta_1,\tau_1)\right\|_{\mathcal{X}}\\
&\quad+C\varepsilon_0(T)^{2+(1+\iota)\sigma-\frac{2\eta_0}{\Lambda_0}}\|\psi\|_{C^0_{\gamma,\sigma+2,T}}\\
&\leq C\varepsilon_0(T)^{\sigma}\left\|(k_2,\eta_2,\beta_2,\tau_2)-(k_1,\eta_1,\beta_1,\tau_1)\right\|_{\mathcal{X}},\\
\|\mu_2-\mu_1\|_{C^0_{\eta_0,T}}&\leq C\varepsilon_0(T)^{6+\sigma-(4+\sigma)\delta+\iota\sigma-\frac{2\eta_0}{\Lambda_0}}\|h_2-\pi_1(h_2)-h_1\|_{C^0_{\gamma,\sigma,T}}\\
&\quad+C\varepsilon_0(T)^{2+(1+\iota)\sigma-\frac{2\eta_0}{\Lambda_0}}\|\psi\|_{C^0_{\gamma,\sigma+2,T}}\\
&\leq C\varepsilon_0(T)^{\sigma}\left\|(k_2,\eta_2,\beta_2,\tau_2)-(k_1,\eta_1,\beta_1,\tau_1)\right\|_{\mathcal{X}}.
\end{split}
\end{equation}
Recall that we expect that $\nu_i$, $\mu_i$, $i=1,2$ belongs to $C^{0,\alpha}_{\beta_0,T}$ and $C^{0,\alpha}_{\tau_0,T}$ respectively. The previous estimate \eqref{est-mu-nu-2} shows in particular that the maps $(k,\eta,\beta,\tau)\in\mathcal{B}\rightarrow \nu\in C^0_{\beta_0,T}$ and $(k,\eta,\beta,\tau)\in\mathcal{B}\rightarrow \nu\in C^0_{\tau_0,T}$ are $K$-contracting maps provided $T\leq 0$ is small enough:
\begin{equation}
\begin{split}\label{est-mu-nu-2-sum}
\|\nu_2-\nu_1\|_{C^0_{\beta_0,T}}&\leq C\varepsilon_0(T)^{2\left(\frac{\eta_0}{\Lambda_0}-1\right)+2\left(1-\frac{\beta_0}{\Lambda_0}\right)+\sigma}\left\|(k_2,\eta_2,\beta_2,\tau_2)-(k_1,\eta_1,\beta_1,\tau_1)\right\|_{\mathcal{X}}\\
&\leq C\varepsilon_0(T)^{2\left(\frac{\eta_0}{\Lambda_0}-1\right)+\sigma}\left\|(k_2,\eta_2,\beta_2,\tau_2)-(k_1,\eta_1,\beta_1,\tau_1)\right\|_{\mathcal{X}},\\
\|\mu_2-\mu_1\|_{C^0_{\tau_0,T}}&\leq C\varepsilon_0(T)^{2\left(\frac{\eta_0}{\Lambda_0}-1\right)+2\left(1-\frac{\tau_0}{\Lambda_0}\right)+\sigma}\left\|(k_2,\eta_2,\beta_2,\tau_2)-(k_1,\eta_1,\beta_1,\tau_1)\right\|_{\mathcal{X}}\\
&\leq C\varepsilon_0(T)^{2\left(\frac{\eta_0}{\Lambda_0}-1\right)+\sigma}\left\|(k_2,\eta_2,\beta_2,\tau_2)-(k_1,\eta_1,\beta_1,\tau_1)\right\|_{\mathcal{X}}.
\end{split}
\end{equation}
Finally, for the difference $\xi_2(t)-\xi_1(t)$, recall that by definition of each $\xi_i$ given in \eqref{def-xi},
\begin{equation*}
\begin{split}
&d_{\zeta_1(t)}g(\xi_2(t)-\xi_1(t))=\\
&d_{\zeta_1(t)}g(\eta_2(t)-\eta_1(t))-\sum_j\left(\lambda_j^2(t)-\lambda_j^1(t)\right)\frac{\|\tilde{\oi}_j^1(t)\|^2_{L^2(\tilde{g}_1(t))}}{\|\oi_j(\zeta_1(t))\|_{L^2(g_{\zeta_1(t)})}^2}\oi_j(\zeta_1(t))\\
&-\sum_j\lambda_j^2(t)\left(\frac{\|\tilde{\oi}_j^2(t)\|^2_{L^2(\tilde{g}_2(t))}}{\|\oi_j(\zeta_2(t))\|_{L^2(g_{\zeta_2(t)})}^2}\oi_j(\zeta_2(t))-\frac{\|\tilde{\oi}_j^1(t)\|^2_{L^2(\tilde{g}_1(t))}}{\|\oi_j(\zeta_1(t))\|_{L^2(g_{\zeta_1(t)})}^2}\oi_j(\zeta_1(t))\right)\\
&+\left(d_{\zeta_1(t)}g-d_{\zeta_2(t)}g\right)(\xi_2(t)-\eta_2(t))\\
&=:d_{\zeta_1(t)}g(\eta_2(t)-\eta_1(t))-\sum_j\left(\lambda_j^2(t)-\lambda_j^1(t)\right)\frac{\|\tilde{\oi}_j^1(t)\|^2_{L^2(\tilde{g}_1(t))}}{\|\oi_j(\zeta_1(t))\|_{L^2(g_{\zeta_1(t)})}^2}\oi_j(\zeta_1(t))+\mathcal{E}(t),
\end{split}
\end{equation*}
where the term $\mathcal{E}(t)$ has to be interpreted as an error term. These heuristics are shown to be true in the following claim:
\begin{claim}\label{claim-monstrosity-xi}
There exists $C>0$ such that for $t\leq T\leq 0$,

\begin{equation*}
\sum_j\left|\left\langle \mathcal{E}(t),\frac{\oi_j(\zeta_1(t))}{\|\oi_j(\zeta_1(t))\|_{L^2(g_{\zeta_1(t)})}}\right\rangle_{L^2(g_{\zeta_1(t)})}\right|\leq C\varepsilon_0(T)^{2\left(\frac{\eta_0}{\Lambda_0}-1\right)}\varepsilon_0(t)^{\frac{2\eta_0}{\Lambda_0}}\|\eta_2-\eta_1\|_{C^{0}_{\eta_0,T}}.
\end{equation*}
\end{claim}
\begin{proof}[Proof of Claim \ref{claim-monstrosity-xi}]
Let us start with the last term defining $\mathcal{E}(t)$: for $t\leq T$,
\begin{equation*}
\begin{split}
&\left|\left\langle\left(d_{\zeta_1(t)}g-d_{\zeta_2(t)}g\right)(\xi_2(t)-\eta_2(t)),\frac{\oi_j(\zeta_1(t))}{\|\oi_j(\zeta_1(t))\|_{L^2(g_{\zeta_1(t)})}}\right\rangle_{L^2(g_{\zeta_1(t)})}\right|\\
&\leq C\frac{|\zeta_2(t)-\zeta_1(t)|}{\varepsilon_0(t)^2}(|\xi_2(t)|+|\eta_2(t)|)
\leq C\varepsilon_0(T)^{2\left(\frac{\eta_0}{\Lambda_0}-1\right)}\varepsilon_0(t)^{\frac{2\eta_0}{\Lambda_0}}\|\eta_2-\eta_1\|_{C^{0}_{\eta_0,T}}.
\end{split}
\end{equation*}
Here we have been using Proposition \ref{prop-analy} to estimate $d_{\zeta_1(t)}g-d_{\zeta_2(t)}g$, the term $|\zeta_2(t)-\zeta_1(t)|$ has been estimated with the help of Lemma \ref{sta-ode-lemma} and Corollary \ref{coro-final-est} is invoked to estimate $|\xi_2(t)|$.

Finally, the first term defining $\mathcal{E}(t)$ can be handled as follows: 
\begin{equation*}
\begin{split}
&\left|\left\langle\mathcal{E}(t)-\left(d_{\zeta_1(t)}g-d_{\zeta_2(t)}g\right)(\xi_2(t)-\eta_2(t)),\frac{\oi_j(\zeta_1(t))}{\|\oi_j(\zeta_1(t))\|_{L^2(g_{\zeta_1(t)})}}\right\rangle_{L^2(g_{\zeta_1(t)})}\right|\\
&\leq C\sum_{j}|\lambda_j^2(t)|\frac{|\zeta_2(t)-\zeta_1(t)|}{\varepsilon_0(t)^2}\|\oi_j(\zeta_1(t))\|_{L^2(g_{\zeta_1(t)})}\\
&\leq C\varepsilon_0(T)^{2\left(\frac{\eta_0}{\Lambda_0}-1\right)}\varepsilon_0(t)^{\frac{2\eta_0}{\Lambda_0}}\|\eta_2-\eta_1\|_{C^{0}_{\eta_0,T}},
\end{split}
\end{equation*}
where we have used Corollary \ref{coro-final-est} to estimate the coefficients $\lambda_j^2(t)$ together with the fact that $\|\oi_j(\zeta_1(t))\|_{L^2(g_{\zeta_1(t)})} = O(\varepsilon_0(t)^2)$.
\end{proof}
Now, invoking \eqref{why-is-it-so-painful} from Proposition \ref{prop 6.5 E(t)},
\begin{equation*}
\begin{split}
\varepsilon_0(t)^{-2}&\left|\langle \psi(t)+(\lambda_i^2(t)-\lambda_i^1(t))\tilde{\oi}^1_i(t),\tilde{\oi}^1_i(t)\rangle_{L^2(\tilde{g}_1(t))}\right|\\
&\leq C\varepsilon_0(t)^{2+2(1-\delta)+(1-\delta+\iota)\sigma}\|h_2-\pi_1(h_2)-h_1\|_{C^0_{\gamma,\sigma,T}}\\
&\quad+C\varepsilon_0(t)^{2+2\delta}|\nu_2(t)-\nu_1(t)|+C\varepsilon_0(t)^2|\mu_2(t)-\mu_1(t)|+C\varepsilon_0(t)^{4+\iota\sigma}\|\psi\|_{C^{0}_{\gamma,\sigma+2,T}}\\
&\leq C\varepsilon_0(T)^{2\frac{\eta_0}{\Lambda_0}-2-\iota\sigma}\varepsilon_0(t)^{2+2(1-\delta)+(1-\delta+\iota)\sigma}\left\|(k_2,\eta_2,\beta_2,\tau_2)-(k_1,\eta_1,\beta_1,\tau_1)\right\|_{\mathcal{X}}\\
&\quad+C\varepsilon_0(T)^{\sigma}\varepsilon_0(t)^{2\frac{\eta_0}{\Lambda_0}+2}\left\|(k_2,\eta_2,\beta_2,\tau_2)-(k_1,\eta_1,\beta_1,\tau_1)\right\|_{\mathcal{X}}\\
&\quad+C\varepsilon_0(T)^{2\frac{\eta_0}{\Lambda_0}-2-\iota\sigma}\varepsilon_0(t)^{4+\iota\sigma}\left\|(k_2,\eta_2,\beta_2,\tau_2)-(k_1,\eta_1,\beta_1,\tau_1)\right\|_{\mathcal{X}}.
\end{split}
\end{equation*}
Here we have used \eqref{quantitative-est-h-psi} in the anti-penultimate line, \eqref{est-mu-nu-2-sum}  in the penultimate line and Claim \ref{claim-control-psi} in the last line. By carefully inspecting each power with the help of the choice of parameters \eqref{choice-parameters}, observe that $2(\frac{\eta_0}{\Lambda_0}-1)<2<2+2(1-\delta)+(1-\delta+\iota)\sigma$ implies:
\begin{equation}
\begin{split}\label{psi-oi-xi}
\varepsilon_0(t)^{-2}&\left|\langle \psi(t)+(\lambda_i^2(t)-\lambda_i^1(t))\tilde{\oi}^1_i(t),\tilde{\oi}^1_i(t)\rangle_{L^2(\tilde{g}(t))}\right|\\
&\leq C\varepsilon_0(T)^{2\left(\frac{\eta_0}{\Lambda_0}-1\right)}\varepsilon_0(t)^{2\frac{\eta_0}{\Lambda_0}}\left\|(k_2,\eta_2,\beta_2,\tau_2)-(k_1,\eta_1,\beta_1,\tau_1)\right\|_{\mathcal{X}}.
\end{split}
\end{equation}
Here we have also used that $2+\sigma>2(\eta_0/\Lambda_0-1)$ and $4+\iota\sigma>2\eta_0/\Lambda_0$ from \eqref{choice-parameters}.

 In particular, we get the following intermediate conclusion thanks to Claim \ref{claim-monstrosity-xi} and \eqref{psi-oi-xi}:
\begin{equation}
\begin{split}\label{bis-monstro-0}
&\|\xi_2-\xi_1\|_{C^0_{\eta_0,T}}\leq C\varepsilon_0(T)^{2\left(\frac{\eta_0}{\Lambda_0}-1\right)-\iota\sigma}\left\|(k_2,\eta_2,\beta_2,\tau_2)-(k_1,\eta_1,\beta_1,\tau_1)\right\|_{\mathcal{X}}\\
&+\left\|\varepsilon_0^{-2}\sum_i\left|\left\langle d_{\zeta_1}g(\eta_2-\eta_1),\oi_i(\zeta_1)\right\rangle_{L^2(g_{\zeta_1})}+\langle \psi,\tilde{\oi}^1_i\rangle_{L^2(\tilde{g}_1)}\right|\right\|_{C^0_{\eta_0,T}}.
\end{split}
\end{equation}
Estimating the last term on the righthand side of \eqref{bis-monstro-0} is similar to the proof of \ref{coro-final-est-xi} from Corollary \ref{coro-final-est} modulo the additional terms in the definition of $\psi$ given in \eqref{def-psi-Psi}. This is the purpose of the following claims.

\begin{claim}\label{claim-nearly-there}
There exists $C>0$ such that for $t\leq T\leq 0$,
\begin{equation}
\begin{split}\label{bis-monstro}
\varepsilon_0(t)^{-2}\sum_i&\left|\left\langle \sum_j\lambda_j^2(t)(\tilde{\oi}_j^2(t)-\tilde{\oi}_j^1(t))
 +\nu_2(t)(\tilde{g}_2(t)-\tilde{g}_1(t))+\mu_2(t)(\tilde{c}_2(t)-\tilde{c}_1(t)),\tilde{\oi}^1_i(t)\right\rangle_{L^2(\tilde{g}_1(t))}\right|\\
&\leq C\varepsilon_0(t)^{\frac{2\eta_0}{\Lambda_0}}\varepsilon_0(T)^{2\left(\frac{\eta_0}{\Lambda_0}-1\right)}\left\|(k_2,\eta_2,\beta_2,\tau_2)-(k_1,\eta_1,\beta_1,\tau_1)\right\|_{\mathcal{X}}.
\end{split}
\end{equation}
\end{claim}

\begin{proof}[Proof of Claim \ref{claim-nearly-there}]
Let us estimate term by term as follows:
\begin{equation*}
\begin{split}
\varepsilon_0(t)^{-2}\sum_j\left|\langle\lambda_j^2(t)(\tilde{\oi}_j^2(t)-\tilde{\oi}_j^1(t)),\tilde{\oi}^1_i(t)\rangle_{L^2(\tilde{g}_1(t))}\right|\leq C\varepsilon_0(t)^{\frac{2\eta_0}{\Lambda_0}}\varepsilon_0(T)^{2\left(\frac{\eta_0}{\Lambda_0}-1\right)}\left\|\eta_2-\eta_1\right\|_{C^0_{\eta_0,T}},\\
 \varepsilon_0(t)^{-2}\left|\langle\nu_2(t)(\tilde{g}_2(t)-\tilde{g}_1(t)),\tilde{\oi}^1_i(t)\rangle_{L^2(\tilde{g}_1(t))}\right|\leq C\varepsilon_0(t)^{10-10\delta}\varepsilon_0(T)^{2\left(\frac{\eta_0}{\Lambda_0}-1\right)}\left\|\eta_2-\eta_1\right\|_{C^0_{\eta_0,T}},\\
 \varepsilon_0(t)^{-2}\left|\langle\mu_2(t)(\tilde{c}_2(t)-\tilde{c}_1(t)),\tilde{\oi}^1_i(t)\rangle_{L^2(\tilde{g}_1(t))}\right|\leq C\varepsilon_0(t)^{10-10\delta}\varepsilon_0(T)^{2\left(\frac{\eta_0}{\Lambda_0}-1\right)}\left\|\eta_2-\eta_1\right\|_{C^0_{\eta_0,T}}.
\end{split}
\end{equation*}
In the first line, we have used \eqref{coro-final-est-lambda} from Corollary \ref{coro-final-est} together with Proposition \ref{prop-est-diff-analy-app-ker}. In the second and third lines, [\eqref{coro-final-est-nu}, Corollary \ref{coro-final-est}] and [\eqref{coro-final-est-mu}, Corollary \ref{coro-final-est}] together with Proposition \ref{prop-est-diff-analy-app-ker} have been used.
Since $2(\eta_0/\Lambda_0-1)<10(1-\delta)-2$ by \eqref{choice-parameters}, the expected estimate \eqref{bis-monstro} follows for the above terms.
\end{proof}
Up to the term $\psi_2(t)-\psi_1(t)$ in the definition of $\psi(t)$ given in \eqref{def-psi-Psi}, we are left with estimating:
\begin{equation*}
\begin{split}
\left\langle\left[\Delta_{L,\tilde{g}_2(t)}-\Delta_{L,\tilde{g}_1(t)}\right]h_2(t)-\left(\partial_t-\Delta_{L,\tilde{g}_1(t)}-2\Lambda\right)\pi_1(h_2),\tilde{\oi}^1_i(t)\right\rangle_{L^2(\tilde{g}_1(t))}.
\end{split}
\end{equation*}
It turns out that estimating each term as we did in the proof of Claim \ref{claim-control-pi1h2} and \ref{claim-control-psi} is not enough to give the desired control: the main issue are the terms decaying like $(\varepsilon(t)+r_o)^{-2-\sigma}$. Now, an integration by parts shows that
\begin{equation*}
\begin{split}
\varepsilon_0(t)^{-2}&\left|\left\langle\Delta_{L,\tilde{g}_1(t)}\pi_1(h_2),\tilde{\oi}^1_i(t)\right\rangle_{L^2(\tilde{g}_1(t))}\right|=\varepsilon_0(t)^{-2}\left|\left\langle\pi_1(h_2),\Delta_{L,\tilde{g}_1(t)}\tilde{\oi}^1_i(t)\right\rangle_{L^2(\tilde{g}_1(t))}\right|\\
&\leq C\gamma(t)^{-1}\left(\varepsilon_0(t)^{2-\sigma}+\varepsilon_0(t)^{2-(2+\sigma)\delta}\right)\varepsilon(T)^{2\left(\frac{\eta_0}{\Lambda_0}-1\right)}\left\|(k_2,\eta_2,\beta_2,\tau_2)-(k_1,\eta_1,\beta_1,\tau_1)\right\|_{\mathcal{X}}\\
&\leq C\varepsilon_0(t)^{4-2\delta+(1-\delta+\iota)\sigma}\varepsilon(T)^{2\left(\frac{\eta_0}{\Lambda_0}-1\right)}\left\|(k_2,\eta_2,\beta_2,\tau_2)-(k_1,\eta_1,\beta_1,\tau_1)\right\|_{\mathcal{X}}\\
&\leq C\varepsilon_0(t)^{2\frac{\eta_0}{\Lambda_0}}\varepsilon(T)^{2\left(\frac{\eta_0}{\Lambda_0}-1\right)}\left\|(k_2,\eta_2,\beta_2,\tau_2)-(k_1,\eta_1,\beta_1,\tau_1)\right\|_{\mathcal{X}},
\end{split}
\end{equation*}
thanks to Claim \ref{claim-control-pi1h2} and Lemma \ref{estimates variations et laplacien o1} and the fact that $\sigma<(2+\sigma)\delta$ together with $(2+\sigma)(1-\delta)>(5-\sigma)(1-\delta)-2>2(\eta_0\Lambda_0-1)$ as imposed by \eqref{choice-parameters}. 
So that we only need to estimate
\begin{equation*}
\begin{split}
\left\langle\left[\Delta_{L,\tilde{g}_2(t)}+2\Lambda-\Delta_{L,\tilde{g}_1(t)}\right]h_2(t)-\partial_t(\pi_1(h_2)-\pi_2(h_2)),\tilde{\oi}^1_i(t)\right\rangle_{L^2(\tilde{g}_1(t))},
\end{split}
\end{equation*}
since $\pi_2(h_2)=0$.

\begin{claim}\label{claim-comm-der-time-proj}
There exists $C>0$ such that for $t\leq T\leq 0$,
\begin{equation}
\begin{split}
\varepsilon_0(t)^{-2}&\left|\left\langle (\pi_1-\pi_2)(\partial_th_2)-\partial_t((\pi_1-\pi_2)(h_2)),\tilde{\oi}^1_i(t)\right\rangle_{L^2(\tilde{g}_1(t))}\right|\\
&\leq C\varepsilon_0(T)^{2\left(\frac{\eta_0}{\Lambda_0}-1\right)}\varepsilon_0(t)^{2\frac{\eta_0}{\Lambda_0}}\left\|(k_2,\eta_2,\beta_2,\tau_2)-(k_1,\eta_1,\beta_1,\tau_1)\right\|_{\mathcal{X}}.
\end{split}
\end{equation}
\end{claim}

\begin{proof}[Proof of Claim \ref{claim-comm-der-time-proj}]
Define for $j=1,2$, $\pi_j(h_2)(t)=:\sum_iL_i^j(t)\tilde{\oi}_i^j(t)+M_j(t)\tilde{g}_j(t)+N_j(t)\tilde{c}_j(t).$
Then,
\begin{equation*}
\begin{split}
&\left|\left\langle \partial_t((\pi_1-\pi_2)(h_2))-\sum_j\dot{L_j^1}(t)\tilde{\oi}_j^1(t)-\dot{M_1}(t)\tilde{g}_1(t)-\dot{N_1}(t)\tilde{c}_1(t),\tilde{\oi}^1_i(t)\right\rangle_{L^2(\tilde{g}_1(t))}\right|\\
&\leq \sum_j|L_j^1(t)||\langle\partial_t\tilde{\oi}_j^1(t),\tilde{\oi}^1_i(t)\rangle_{L^2(\tilde{g}_1(t))}|\\
&\quad+|M_1(t)||\langle\partial_t\tilde{g}_1(t),\tilde{\oi}^1_i(t)\rangle_{L^2(\tilde{g}_1(t))}|+|N_1(t)||\langle\partial_t\tilde{c}_1(t),\tilde{\oi}^1_i(t)\rangle_{L^2(\tilde{g}_1(t))}|\\
&\leq \varepsilon_0(t)^4\left(\sum_j|L_j^1(t)|+|M_1(t)|+|N_1(t)|\right)\\
&\leq \sum_i|\langle h_2(t),\tilde{\oi}^1_i(t)\rangle_{L^2(\tilde{g}_1(t))}|+\varepsilon_0(t)^4\left(|\langle h_2(t),\tilde{g}_1(t)\rangle_{L^2(\tilde{g}_1(t))}|+|\langle h_2(t),\tilde{c}_1(t)\rangle_{L^2(\tilde{g}_1(t))}|\right).
\end{split}
\end{equation*}
Here we have been using Proposition \ref{prop-o_1-first-app} in the second inequality and [\eqref{est-inverse-Gram-pre}, Lemma \ref{est-inverse-Gram-lemma}] in the last inequality.
Since $\pi_2(h_2)=0$, each scalar product of $h_2$ and an element of the approximate kernel with respect to $\tilde{g}_2(t)$ is $0$. By linearizing each term of the righthand side of the previous estimate leads to the following intermediate result:
\begin{equation}
\begin{split}\label{claim-comm-der-time-proj-1}
&\varepsilon_0(t)^{-2}\left|\left\langle \partial_t((\pi_1-\pi_2)(h_2))-\sum_j\dot{L_j^1}(t)\tilde{\oi}_j^1(t)-\dot{M_1}(t)\tilde{g}_1(t)-\dot{N_1}(t)\tilde{c}_1(t),\tilde{\oi}^1_i(t)\right\rangle_{L^2(\tilde{g}_1(t))}\right|\\
&\leq C\gamma(t)^{-1}\varepsilon_0(T)^{2\left(\frac{\eta_0}{\Lambda_0}-1\right)}\left(\varepsilon(t)^{2-\sigma}+\varepsilon(t)^{6^-}\delta(t)^{-4^--\sigma}\right)\|\eta_2-\eta_1\|_{C^0_{\eta_0,T}}\\
&\leq C\varepsilon_0(T)^{2\left(\frac{\eta_0}{\Lambda_0}-1\right)}\varepsilon(t)^{4+\iota\sigma}\|\eta_2-\eta_1\|_{C^0_{\eta_0,T}}\\
&\leq C\varepsilon_0(T)^{2\left(\frac{\eta_0}{\Lambda_0}-1\right)}\varepsilon_0(t)^{2\frac{\eta_0}{\Lambda_0}}\|\eta_2-\eta_1\|_{C^0_{\eta_0,T}}.
\end{split}
\end{equation}
We have used \eqref{choice-parameters} in the last two lines here.

Now, invoking [\eqref{est-inverse-Gram-pre}, Lemma \ref{est-inverse-Gram-lemma}] again, the difference $\left\langle\pi_2( \partial_th_2),\tilde{\oi}_i^1(t)\right\rangle_{L^2(\tilde{g}_1(t))}-\left\langle \partial_th_2,\tilde{\oi}_i^2(t)\right\rangle_{L^2(\tilde{g}_2(t))}$ is dominated by a sum of the following terms: 
\begin{equation*}
\begin{split}
& \left|\langle \partial_th_2,\tilde{\oi}_j^2(t)\rangle_{L^2(\tilde{g}_2(t))}\right|\left|\frac{\left\langle\tilde{\oi}_i^2(t),\tilde{\oi}_i^1(t)\right\rangle_{L^2(\tilde{g}_1(t))}}{\|\tilde{\oi}_i^2(t)\|^2_{L^2(\tilde{g}_2(t))}}-1\right|,\\
&\varepsilon_0(t)^2\sum_{j\neq i}\left|\langle \partial_th_2,\tilde{\oi}_j^2(t)\rangle_{L^2(\tilde{g}_2(t))}\right|\left|\frac{\left\langle\tilde{\oi}_i^2(t),\tilde{\oi}_i^1(t)\right\rangle_{L^2(\tilde{g}_1(t))}}{\|\tilde{\oi}_i^2(t)\|^2_{L^2(\tilde{g}_2(t))}}-1\right|,\\
&\left|\langle \partial_th_2,\tilde{g}_2(t)\rangle_{L^2(\tilde{g}_2(t))}\right|\left|\left\langle\tilde{g}_2(t),\tilde{\oi}_i^1(t)\right\rangle_{L^2(\tilde{g}_1(t))}-\left\langle\tilde{g}_2(t),\tilde{\oi}_i^2(t)\right\rangle_{L^2(\tilde{g}_2(t))}\right|,\\
&\left|\langle \partial_th_2,\tilde{c}_2(t)\rangle_{L^2(\tilde{g}_2(t))}\right|
\left|\left\langle\tilde{c}_2(t),\tilde{\oi}_i^1(t)\right\rangle_{L^2(\tilde{g}_1(t))}-\left\langle\tilde{c}_2(t),\tilde{\oi}_i^2(t)\right\rangle_{L^2(\tilde{g}_2(t))}\right|.
\end{split}
\end{equation*}
Each of these terms once multiplied by $\varepsilon_0(t)^{-2}$ can be estimated from above by $$C\varepsilon_0(T)^{2\left(\frac{\eta_0}{\Lambda_0}-1\right)}\varepsilon_0(t)^{2\frac{\eta_0}{\Lambda_0}}\|\eta_2-\eta_1\|_{C^0_{\eta_0,T}},$$ at least thanks to Proposition \ref{prop-est-diff-analy-app-ker}, Lemma \ref{conv-lemma-diff-scal-prod}. Notice that this bound is not sharp by inspecting the decays of each of these terms but those turn out to be sufficient.  Here we have used that $\pi_2(h_2)=0$ and Lemma \ref{pre-partial-h}, \ref{pre-partial-h-2} and \ref{pre-partial-h-3}.
A second intermediate estimate is then:
\begin{equation}
\begin{split}\label{claim-comm-der-time-proj-2}
&\varepsilon_0(t)^{-2}\left|\left\langle(\pi_2-\pi_1)( \partial_th_2),\tilde{\oi}_i^1(t)\right\rangle_{L^2(\tilde{g}_1(t))}-\left\langle \partial_th_2,\tilde{\oi}_i^2(t)\right\rangle_{L^2(\tilde{g}_2(t))}+\left\langle \partial_th_2,\tilde{\oi}_i^1(t)\right\rangle_{L^2(\tilde{g}_1(t))}\right|\\
&\leq C\varepsilon_0(T)^{2\left(\frac{\eta_0}{\Lambda_0}-1\right)}\varepsilon_0(t)^{2\frac{\eta_0}{\Lambda_0}}\|\eta_2-\eta_1\|_{C^0_{\eta_0,T}}.
\end{split}
\end{equation}
Similarly, invoking Lemma \ref{easy-time-der-abs-2-tensors-prelim}, Proposition \ref{prop-est-diff-analy-app-ker} and Lemma \ref{conv-lemma-diff-scal-prod},
\begin{equation}
\begin{split}\label{claim-comm-der-time-proj-3}
&\varepsilon_0(t)^{-2}\left|\frac{d}{dt}\left(\left\langle h_2,\tilde{\oi}_i^1(t)\right\rangle_{L^2(\tilde{g}_1(t))}-\left\langle h_2,\tilde{\oi}_i^2(t)\right\rangle_{L^2(\tilde{g}_2(t))}\right)-\left\langle \partial_th_2,\tilde{\oi}_i^2(t)\right\rangle_{L^2(\tilde{g}_2(t))}+\left\langle \partial_th_2,\tilde{\oi}_i^1(t)\right\rangle_{L^2(\tilde{g}_1(t))}\right|\\
&\leq C\varepsilon_0(T)^{2\left(\frac{\eta_0}{\Lambda_0}-1\right)}\varepsilon_0(t)^{2\frac{\eta_0}{\Lambda_0}}\|\eta_2-\eta_1\|_{C^0_{\eta_0,T}}.
\end{split}
\end{equation}
Finally, [\eqref{hello-ugly-gram-bis-a}, Lemma \ref{hello-ugly-gram}] applied to $h_2$ and the approximate kernel defined with respect to $\tilde{g}_1(t)$ gives:

\begin{equation*}
\begin{split}
\Bigg|\frac{d}{dt}&\bigg(\left\langle h_2,\tilde{\oi}_i^1(t)\right\rangle_{L^2(\tilde{g}_1(t))}-\left\langle h_2,\tilde{\oi}_i^2(t)\right\rangle_{L^2(\tilde{g}_2(t))}\Bigg)-\\
&\left\langle\sum_j\dot{L_j^1}(t)\tilde{\oi}_j^1(t)+\dot{M_1}(t)\tilde{g}_1(t)+\dot{N_1}(t)\tilde{c}_1(t),\tilde{\oi}^1_i(t)\right\rangle_{L^2(\tilde{g}_1(t))}\bigg|\\
&\leq C\varepsilon(t)^4\delta(t)^{-4}\left|\frac{d}{dt}\sum_i\langle h_2,\tilde{\oi}_i\rangle_{L^2(\tilde{g})}\right|+C\varepsilon(t)^8\delta(t)^{-2}\left|\frac{d}{dt}\langle h_2,\tilde{g}\rangle_{L^2(\tilde{g})}\right|\\
&\quad+C\varepsilon(t)^8\delta(t)^{-4}\left|\frac{d}{dt}\langle h_2,\tilde{c}\rangle_{L^2(\tilde{g})}\right|+C\varepsilon(t)^2\|\pi_1(h_2(t))\|_{L^2(\tilde{g}_1(t))}.
\end{split}
\end{equation*}
We now invoke \eqref{scal-oj-h2}, \eqref{scal-g-c-h2} and \eqref{time-der-h2-oi},  \eqref{time-der-h2-g-c} from the proof of Claim \ref{claim-control-pi1h2} to infer that the lefthand side of the previous estimate is bounded from above by the expected quantity. More precisely, one has
\begin{equation}
\begin{split}\label{claim-comm-der-time-proj-4}
\varepsilon_0(t)^{-2}\Bigg|\frac{d}{dt}&\bigg(\left\langle h_2,\tilde{\oi}_i^1(t)\right\rangle_{L^2(\tilde{g}_1(t))}-\left\langle h_2,\tilde{\oi}_i^2(t)\right\rangle_{L^2(\tilde{g}_2(t))}\bigg)-\\
&\left\langle\sum_j\dot{L_j^1}(t)\tilde{\oi}_j^1(t)+\dot{M_1}(t)\tilde{g}_1(t)+\dot{N_1}(t)\tilde{c}_1(t),\tilde{\oi}^1_i(t)\right\rangle_{L^2(\tilde{g}_1(t))}\bigg|\\
&\leq C\varepsilon_0(T)^{2\left(\frac{\eta_0}{\Lambda_0}-1\right)}\varepsilon_0(t)^{2\frac{\eta_0}{\Lambda_0}}\left\|(k_2,\eta_2,\beta_2,\tau_2)-(k_1,\eta_1,\beta_1,\tau_1)\right\|_{\mathcal{X}}.
\end{split}
\end{equation}
Here we have invoked that $2(\eta_0/\Lambda_0-1)<2+\iota\sigma$ and $2(\eta_0/\Lambda_0-1)<4(1-\delta)+\iota\sigma$ thanks to \eqref{choice-parameters}.
The claim is proved by invoking the triangular inequality and estimates \eqref{claim-comm-der-time-proj-1}, \eqref{claim-comm-der-time-proj-2}, \eqref{claim-comm-der-time-proj-3} and \eqref{claim-comm-der-time-proj-4}.
\end{proof}

\begin{claim}\label{claim-8}
There exists $C>0$ such that for $t\leq T\leq 0$,
\begin{equation*}
\begin{split}
\varepsilon_0(t)^{-2}&\left|\left\langle\left[\Delta_{L,\tilde{g}_2(t)}+2\Lambda-\Delta_{L,\tilde{g}_1(t)}\right]h_2(t)-(\pi_1-\pi_2)(\partial_th_2)),\tilde{\oi}^1_i(t)\right\rangle_{L^2(\tilde{g}_1(t))}\right|\\
&\leq C\varepsilon_0(T)^{2\left(\frac{\eta_0}{\Lambda_0}-1\right)}\varepsilon_0(t)^{2\frac{\eta_0}{\Lambda_0}}\left\|(k_2,\eta_2,\beta_2,\tau_2)-(k_1,\eta_1,\beta_1,\tau_1)\right\|_{\mathcal{X}}.
\end{split}
\end{equation*}
\end{claim}

\begin{proof}[Proof of Claim \ref{claim-8}]
Since $h_2$ solves \eqref{eqn-solve-h} and $\pi_2(h_2)=0$,
\begin{equation*}
\begin{split}
&\left\langle\left[\Delta_{L,\tilde{g}_2(t)}+2\Lambda-\Delta_{L,\tilde{g}_1(t)}\right]h_2(t)-(\pi_1-\pi_2)(\partial_th_2)),\tilde{\oi}^1_i(t)\right\rangle_{L^2(\tilde{g}_1(t))}\\
&=\left\langle(\Delta_{L,\tilde{g}_2(t)}+2\Lambda-\partial_t)h_2(t)-\Delta_{L,\tilde{g}_1(t)}h_2(t)+\pi_2(\partial_th_2)),\tilde{\oi}^1_i(t)\right\rangle_{L^2(\tilde{g}_1(t))}\\
&=\left\langle\pi_2(\psi_2)-\psi_2+\pi_2(\Delta_{L,\tilde{g}_2(t)}h_2(t))-\Delta_{L,\tilde{g}_1(t)}h_2(t),\tilde{\oi}^1_i(t)\right\rangle_{L^2(\tilde{g}_1(t))}.
\end{split}
\end{equation*}
This identity lets us get the desired estimate. 

On the one hand, we estimate all of the terms in the definition of $\psi_2$ given by \eqref{defn-psi} one by one. Using the fact that $\|k_2\|_{C^{2,\alpha}_{\gamma,\sigma,T}}\leq 1$ and Proposition \ref{prop-est-diff-analy-app-ker} together with Lemma \ref{conv-lemma-diff-scal-prod} gives:
\begin{equation*}
\begin{split}
\varepsilon_0(t)^{-2}&\left|\left\langle Q_{\tilde{g}_2(t)}(k_2(t)),\tilde{\oi}^1_i(t)\right\rangle_{L^2(\tilde{g}_1(t))}-\left\langle Q_{\tilde{g}_2(t)}(k_2(t)),\tilde{\oi}^2_i(t)\right\rangle_{L^2(\tilde{g}_2(t))}\right|\\
&\leq C\varepsilon_0(T)^{2\left(\frac{\eta_0}{\Lambda_0}-1\right)}\gamma(t)^{-2}\varepsilon_0(t)^{-2\sigma}\|\eta_2-\eta_1\|_{C^0_{\eta_0,T}}\\
&\leq C\varepsilon_0(T)^{2\left(\frac{\eta_0}{\Lambda_0}-1\right)}\varepsilon_0(t)^{4+2\iota\sigma}\|\eta_2-\eta_1\|_{C^0_{\eta_0,T}}\\
&\leq C\varepsilon_0(T)^{2\left(\frac{\eta_0}{\Lambda_0}-1\right)}\varepsilon_0(t)^{2\frac{\eta_0}{\Lambda_0}}\|\eta_2-\eta_1\|_{C^0_{\eta_0,T}}.
\end{split}
\end{equation*}
Here we have used \eqref{choice-parameters} in the last line.

By definition of the weights $\beta_0$ and $\tau_0$ given in Section \ref{section-def} and \eqref{choice-parameters},
\begin{equation*}
\begin{split}
\varepsilon_0(t)^{-2}&\bigg|\left\langle\beta_2(t)k_2(t)+\frac{1}{2}\tau_2(t)\Li_{\nabla^{\tilde{g}_2(t)}\tilde{v}_2(t)}k_2(t),\tilde{\oi}^1_i(t)\right\rangle_{L^2(\tilde{g}_1(t))}\\
&-\left\langle \beta_2(t)k_2(t)+\frac{1}{2}\tau_2(t)\Li_{\nabla^{\tilde{g}_2(t)}\tilde{v}_2(t)}k_2(t),\tilde{\oi}^2_i(t)\right\rangle_{L^2(\tilde{g}_2(t))}\bigg|\\
&\leq C\varepsilon_0(T)^{2\left(\frac{\eta_0}{\Lambda_0}-1\right)}\varepsilon_0(t)^{4+2\frac{\beta_0}{\Lambda_0}+\iota\sigma}\|\eta_2-\eta_1\|_{C^0_{\eta_0,T}}\leq C\varepsilon_0(T)^{2\left(\frac{\eta_0}{\Lambda_0}-1\right)}\varepsilon_0(t)^{2\frac{\eta_0}{\Lambda_0}}\|\eta_2-\eta_1\|_{C^0_{\eta_0,T}},\\
\varepsilon_0(t)^{-2}&\bigg|\left\langle\partial_t\tilde{g}_2(t)+2\Ric(\tilde{g}_2(t))-2\Lambda\tilde{g}_2(t)
,\tilde{\oi}^1_i(t)\right\rangle_{L^2(\tilde{g}_1(t))}\\
&-\left\langle \partial_t\tilde{g}_2(t)+2\Ric(\tilde{g}_2(t))-2\Lambda\tilde{g}_2(t)
,\tilde{\oi}^2_i(t)\right\rangle_{L^2(\tilde{g}_2(t))}\bigg|\leq C\varepsilon_0(T)^{2\left(\frac{\eta_0}{\Lambda_0}-1\right)}\varepsilon_0(t)^{2\frac{\eta_0}{\Lambda_0}}\|\eta_2-\eta_1\|_{C^0_{\eta_0,T}}.
\end{split}
\end{equation*}
Here we have used Lemma \ref{lemma-time-der-metric} in the last line together with Lemma \ref{conv-lemma-diff-scal-prod}. Notice that the previous estimate is sharp with respect to the choice of parameters.

One can then conclude that:
\begin{equation}\label{intermed-proj-psi2-oi-lic}
\begin{split}
\varepsilon_0(t)^{-2}\left|\left\langle\psi_2(t),\tilde{\oi}^2_i(t)\right\rangle_{L^2(\tilde{g}_2(t))}-\left\langle\psi_2(t),\tilde{\oi}^1_i(t)\right\rangle_{L^2(\tilde{g}_1(t))}\right|\leq C\varepsilon_0(T)^{2\left(\frac{\eta_0}{\Lambda_0}-1\right)}\varepsilon_0(t)^{2\frac{\eta_0}{\Lambda_0}}\|\eta_2-\eta_1\|_{C^0_{\eta_0,T}}.
\end{split}
\end{equation}

Next, we estimate $\pi_2(\psi_2)$ thanks to Proposition \ref{prop-final-est-proj} as follows:
\begin{equation}
\begin{split}\label{intermed-proj-psi2-g-c-lic}
\left|\left\langle\pi_2(\psi_2(t)),\tilde{\oi}_i^2(t)\right\rangle_{L^2(\tilde{g}_2(t))}\right|&\leq C\varepsilon(t)^{4+2\left(\frac{\eta_0}{\Lambda_0}-1\right)},\\
\left\langle\pi_2(\psi_2(t)),\tilde{g}_2(t)\right\rangle_{L^2(\tilde{g}_2(t))}&\leq C\varepsilon(t)^{8-10\delta},\\
\left\langle\pi_2(\psi_2(t)),\tilde{c}_2(t)\right\rangle_{L^2(\tilde{g}_2(t))}&\leq C\varepsilon(t)^{8-10\delta}.
\end{split}
\end{equation}
In particular, one gets thanks to [\eqref{est-inverse-Gram}, Lemma \ref{est-inverse-Gram-lemma}] that:
\begin{equation}\label{intermed-proj-2-psi2}
\pi_2(\psi_2(t))=O\left(\varepsilon(t)^{2\left(\frac{\eta_0}{\Lambda_0}-1\right)}\right)\tilde{\oi}_i^2(t)+O\left(\varepsilon(t)^{8-10\delta}\right)\tilde{g}_2(t)+O\left(\varepsilon(t)^{8-10\delta}\right)\tilde{c}_2(t).
\end{equation}

Proposition \ref{prop-est-diff-analy-app-ker} together with the previous estimate \eqref{intermed-proj-2-psi2} give:
\begin{equation}\label{intermed-proj-psi2-oi-diff-lic}
\varepsilon_0(t)^{-2}\left|\left\langle\pi_2(\psi_2(t)),\tilde{\oi}^1_i(t)-\tilde{\oi}^2_i(t)\right\rangle_{L^2(\tilde{g}_1(t))}\right|\leq C\varepsilon_0(T)^{2\left(\frac{\eta_0}{\Lambda_0}-1\right)}\varepsilon_0(t)^{2\frac{\eta_0}{\Lambda_0}}\|\eta_2-\eta_1\|_{C^0_{\eta_0,T}}.
\end{equation}

Now, writing:
\begin{equation}
\begin{split}\label{intermed-proj-psi-2-oi-diff-lic-bis}
&\left\langle\pi_2(\psi_2(t))-\psi_2(t),\tilde{\oi}^1_i(t)\right\rangle_{L^2(\tilde{g}_1(t))}=\left\langle\pi_2(\psi_2(t)),\tilde{\oi}^1_i(t)-\tilde{\oi}^2_i(t)\right\rangle_{L^2(\tilde{g}_1(t))} \\
&+\left\langle\pi_2(\psi_2(t)),\tilde{\oi}^2_i(t)\right\rangle_{L^2(\tilde{g}_1(t))}-\left\langle\pi_2(\psi_2(t)),\tilde{\oi}^2_i(t)\right\rangle_{L^2(\tilde{g}_2(t))}\\
&+\left\langle\psi_2(t),\tilde{\oi}^2_i(t)\right\rangle_{L^2(\tilde{g}_2(t))}-\left\langle\psi_2(t),\tilde{\oi}^1_i(t)\right\rangle_{L^2(\tilde{g}_1(t))},
\end{split}
\end{equation}
estimates \eqref{intermed-proj-psi2-oi-lic} (for the third line), \eqref{intermed-proj-psi2-oi-diff-lic} and Lemma \ref{conv-lemma-diff-scal-prod} applied to the second line of \eqref{intermed-proj-psi-2-oi-diff-lic-bis} together with \eqref{intermed-proj-2-psi2} give the desired intermediate result, i.e.
\begin{equation}\label{intermed-monster-proj-psi2}
\begin{split}
\varepsilon_0(t)^{-2}\bigg|&\left\langle\pi_2(\psi_2(t))-\psi_2(t),\tilde{\oi}^1_i(t)\right\rangle_{L^2(\tilde{g}_1(t))}\bigg|\leq C\varepsilon_0(T)^{2\left(\frac{\eta_0}{\Lambda_0}-1\right)}\varepsilon_0(t)^{2\frac{\eta_0}{\Lambda_0}}\|\eta_2-\eta_1\|_{C^0_{\eta_0,T}}.
\end{split}
\end{equation}

On the other hand, an integration by parts and Lemma \ref{estimates variations et laplacien o1} shows:
\begin{equation}
\begin{split}\label{intermed-proj-h2-g-c-lic}
\left|\left\langle\pi_2(\Delta_{L,\tilde{g}_2(t)}h_2(t)),\tilde{\oi}_i^2(t)\right\rangle_{L^2(\tilde{g}_2(t))}\right|&\leq C\gamma(t)^{-1}\varepsilon(t)^{4}\delta(t)^{-2-\sigma},\\
\left\langle\pi_2(\Delta_{L,\tilde{g}_2(t)}h_2(t)),\tilde{g}_2(t)\right\rangle_{L^2(\tilde{g}_2(t))}&=\left\langle h_2(t),\Delta_{L,\tilde{g}_2(t)}\tilde{g}_2(t)\right\rangle_{L^2(\tilde{g}_2(t))}=0,\\
\left\langle\pi_2(\Delta_{L,\tilde{g}_2(t)}h_2(t)),\tilde{c}_2(t)\right\rangle_{L^2(\tilde{g}_2(t))}&=\left\langle h_2(t),\Delta_{L,\tilde{g}_2(t)}\tilde{c}_2(t)\right\rangle_{L^2(\tilde{g}_2(t))}\\
&=-(\Lambda+1)\left\langle h_2(t),\tilde{c}_2(t)\right\rangle_{L^2(\tilde{g}_2(t))}+O(\gamma(t)^{-1}\delta(t)^{6-\sigma})\\
&\quad+O(\gamma(t)^{-1}\varepsilon(t)^5\delta(t)^{-3-\sigma})+O(\gamma(t)^{-1}\varepsilon(t)^4\delta(t)^{-4-\sigma})\\
&=O(\gamma(t)^{-1}\varepsilon(t)^{4}\delta(t)^{-4-\sigma}).
\end{split}
\end{equation}
Here we have used Lemma \ref{est-basic-conf} in the penultimate line together with the fact that $\pi_2(h_2)=0$ by construction.
In particular, one gets thanks to [\eqref{est-inverse-Gram}, Lemma \ref{est-inverse-Gram-lemma}] that:
\begin{equation}\label{intermed-proj-2-lic-lap-h2}
\pi_2(\Delta_{L,\tilde{g}_2(t)}h_2(t))=O(\gamma(t)^{-1}\delta(t)^{-2-\sigma})\tilde{\oi}_i^2(t)+O(\gamma(t)^{-1}\varepsilon(t)^{4}\delta(t)^{-4-\sigma})\tilde{c}_2(t).
\end{equation}

Now, an integration by parts and Proposition \ref{prop-est-diff-analy-app-ker} show:
\begin{equation}
\begin{split}\label{intermed-proj-h2-oi-lic}
&\varepsilon_0(t)^{-2}\left|\left\langle\Delta_{L,\tilde{g}_2(t)}h_2(t),\tilde{\oi}^2_i(t)\right\rangle_{L^2(\tilde{g}_2(t))}-\left\langle\Delta_{L,\tilde{g}_1(t)}h_2(t),\tilde{\oi}^1_i(t)\right\rangle_{L^2(\tilde{g}_1(t))}\right|\\
&=\varepsilon_0(t)^{-2}\left|\left\langle h_2(t),\Delta_{L,\tilde{g}_2(t)}\tilde{\oi}^2_i(t)\right\rangle_{L^2(\tilde{g}_2(t))}-\left\langle h_2(t),\Delta_{L,\tilde{g}_1(t)}\tilde{\oi}^1_i(t)\right\rangle_{L^2(\tilde{g}_1(t))}\right|\\
&\leq C\left(\gamma(t)^{-1}\varepsilon(t)^{2-\sigma}+\gamma(t)^{-1}\varepsilon(t)^{2^{-}}\delta(t)^{-2^--\sigma}\right)\varepsilon_0(T)^{2\left(\frac{\eta_0}{\Lambda_0}-1\right)}\varepsilon_0(t)^{2\frac{\eta_0}{\Lambda_0}}\\
&\leq C\varepsilon_0(T)^{2\left(\frac{\eta_0}{\Lambda_0}-1\right)}\varepsilon_0(t)^{2\frac{\eta_0}{\Lambda_0}}\|\eta_2-\eta_1\|_{C^0_{\eta_0,T}}.
\end{split}
\end{equation}
Here we have used \eqref{choice-parameters} to ensure that $(2+\sigma)(1-\delta)>(5-\sigma)(1-\delta)-2$ in the last line.


Thanks to \eqref{intermed-proj-h2-g-c-lic}, the projection $\pi_2(\Delta_{L,\tilde{g}_2(t)}h_2(t))$ can be estimated as follows with the help of [\eqref{est-inverse-Gram-pre}, Lemma \ref{est-inverse-Gram-pre}]:
\begin{equation}
\begin{split}\label{intermed-proj-h2-oi-diff-lic}
\varepsilon_0(t)^{-2}\bigg|&\left\langle\pi_2(\Delta_{L,\tilde{g}_2(t)}h_2(t)),\tilde{\oi}_i^2(t)-\tilde{\oi}_i^1(t)\right\rangle_{L^2(\tilde{g}_2(t))}\bigg|\\
&\leq C\varepsilon_0(T)^{2\left(\frac{\eta_0}{\Lambda_0}-1\right)}\gamma(t)^{-1}\varepsilon_0(t)^{-2}\left(\varepsilon(t)^4\delta(t)^{-2-\sigma}+\varepsilon(t)^{4-\sigma}\right)\|\eta_2-\eta_1\|_{C^0_{\eta_0,T}}\\
&\leq C\varepsilon_0(T)^{2\left(\frac{\eta_0}{\Lambda_0}-1\right)}\varepsilon_0(t)^{(1+\iota)\sigma}\left(\varepsilon(t)^4\delta(t)^{-2-\sigma}+\varepsilon(t)^{4-\sigma}\right)\|\eta_2-\eta_1\|_{C^0_{\eta_0,T}}\\
&\leq C\varepsilon_0(T)^{2\left(\frac{\eta_0}{\Lambda_0}-1\right)}\varepsilon_0(t)^{2\frac{\eta_0}{\Lambda_0}}\|\eta_2-\eta_1\|_{C^0_{\eta_0,T}},
\end{split}
\end{equation}
where we have invoked $2-2\delta+(1-\delta+\iota)\sigma>2(\eta_0/\Lambda_0-1)$ from \eqref{choice-parameters}.

Writing:
\begin{equation}
\begin{split}\label{intermed-proj-h2-oi-diff-lic-bis}
&\left\langle\pi_2(\Delta_{L,\tilde{g}_2(t)}h_2(t))-\Delta_{L,\tilde{g}_1(t)}h_2(t),\tilde{\oi}^1_i(t)\right\rangle_{L^2(\tilde{g}_1(t))}=\left\langle\pi_2(\Delta_{L,\tilde{g}_2(t)}h_2(t)),\tilde{\oi}^1_i(t)-\tilde{\oi}^2_i(t)\right\rangle_{L^2(\tilde{g}_1(t))} \\
&+\left\langle\pi_2(\Delta_{L,\tilde{g}_2(t)}h_2(t)),\tilde{\oi}^2_i(t)\right\rangle_{L^2(\tilde{g}_1(t))}-\left\langle\pi_2(\Delta_{L,\tilde{g}_2(t)}h_2(t)),\tilde{\oi}^2_i(t)\right\rangle_{L^2(\tilde{g}_2(t))}\\
&+\left\langle\Delta_{L,\tilde{g}_2(t)}h_2(t),\tilde{\oi}^2_i(t)\right\rangle_{L^2(\tilde{g}_2(t))}-\left\langle\Delta_{L,\tilde{g}_1(t)}h_2(t),\tilde{\oi}^1_i(t)\right\rangle_{L^2(\tilde{g}_1(t))},
\end{split}
\end{equation}
estimates \eqref{intermed-proj-h2-oi-lic}, \eqref{intermed-proj-h2-oi-diff-lic} and Lemma \ref{conv-lemma-diff-scal-prod} applied to the second line of \eqref{intermed-proj-h2-oi-diff-lic-bis} together with \eqref{intermed-proj-2-lic-lap-h2} give the second intermediate desired result, i.e.
\begin{equation*}
\begin{split}
\varepsilon_0(t)^{-2}\bigg|&\left\langle\pi_2(\Delta_{L,\tilde{g}_2(t)}h_2(t))-\Delta_{L,\tilde{g}_1(t)}h_2(t),\tilde{\oi}^1_i(t)\right\rangle_{L^2(\tilde{g}_1(t))}\bigg|\\
&\leq C\varepsilon_0(T)^{2\left(\frac{\eta_0}{\Lambda_0}-1\right)}\varepsilon_0(t)^{2\frac{\eta_0}{\Lambda_0}}\|\eta_2-\eta_1\|_{C^0_{\eta_0,T}}.
\end{split}
\end{equation*}
This ends the proof of this claim when combined with \eqref{intermed-monster-proj-psi2}.
\end{proof}

\begin{claim}\label{claim-final-picard}
There exists $C>0$ such that for $t\leq T\leq 0$,
\begin{equation}
\begin{split}\label{bis-monstro-final}
\varepsilon_0(t)^{-2}\sum_i&\left|\langle \psi_2(t)-\psi_1(t),\tilde{\oi}^1_i(t)\rangle_{L^2(\tilde{g}_1(t))}+\langle d_{\zeta_1(t)}g(\eta_2(t)-\eta_1(t)),\oi_i(\zeta_1(t))\rangle_{L^2(g_{\zeta_1(t)})}\right|\\
&\leq C\varepsilon_0(T)^{2\left(\frac{\eta_0}{\Lambda_0}-1\right)}\varepsilon_0(t)^{\frac{2\eta_0}{\Lambda_0}}\left\|(k_2,\eta_2,\beta_2,\tau_2)-(k_1,\eta_1,\beta_1,\tau_1)\right\|_{\mathcal{X}}.
\end{split}
\end{equation}
\end{claim}
\begin{proof}[Proof of Claim \ref{claim-final-picard}]
The proof of this claim is virtually identical to that of [\eqref{coro-final-est-xi}, Corollary \ref{coro-final-est}] and is therefore omitted: the only difference being that instead of having one parameter $(k,\eta,\beta,\tau)$, the dependence on the difference $(k_2,\eta_2,\beta_2,\tau_2)-(k_1,\eta_1,\beta_1,\tau_1)$ must be taken into account via Proposition \ref{prop-est-diff-analy-app-ker}.
\end{proof}

The proof of the fact that the second projection of the map $\Phi$ is a contraction for $T$ small enough is now a combination of Claims \ref{claim-monstrosity-xi}, \ref{claim-nearly-there}, \ref{claim-comm-der-time-proj} and \ref{claim-final-picard} up to corresponding H\"older estimates for the map $\Phi$ that can be proved along the same lines and that are therefore omitted.
\end{proof}

\section{Topology and curvature estimates}\label{sec:curvature}

\subsection{Topology of some examples}
In this section, we prove the topological aspect of Theorem \ref{mainthm1}.
As stated in Theorem \ref{mainthm1}, we restrict ourselves to the homeotype of our ancient flows, we believe that they will also coincide with their diffeotype.
\begin{prop}
	The ancient flows provided by Theorem \ref{mainthm1} with tangent soliton a spherical suspension over $\mathbb{S}^3/\Gamma$ where $\Gamma$ is a finite subgroup of $\operatorname{SU}(2)$ acting freely, are oriented homeomorphic to:
	\begin{enumerate}
		\item (Case $\Gamma = \mathbb{Z}_2$) $\mathbb{CP}^2\#\mathbb{CP}^2$, $\overline{\mathbb{CP}}^2\#\overline{\mathbb{CP}}^2$, or $\mathbb{S}^2\times \mathbb{S}^2$ depending on the orientations in which the gluings are performed, and
		\item (Case $\Gamma\subset\{A_k,D_k,E_k\}$, $k>2$)
		$\#k(\mathbb{S}^2\times \mathbb{S}^2)$.
	\end{enumerate}
\end{prop}
\begin{rk}
	In particular, the flows constructed in Theorem \ref{mainthm1} show that the following are possible:
	\begin{itemize}
		\item finding two homeomorphic ancient Ricci flows, one with smooth tangent Ricci soliton $\mathbb{S}^2\times\mathbb{S}^2$, and one with the singular suspension over $\mathbb{RP}^3$, 
		\item finding two smooth ancient Ricci flows with the same singular tangent Ricci soliton, the singular suspension over $\mathbb{RP}^3$, which are not homeomorphic, and
		\item families of homeomorphic smooth ancient Ricci flows with different singular tangent Ricci soliton.
	\end{itemize}
\end{rk}
\begin{proof}
When we start with the suspension over $\mathbb{RP}^3$, which is a $\mathbb{Z}_2$ quotient of $\mathbb{S}^4$, there are three possible (oriented) topologies for desingularizations by Eguchi-Hanson metrics. These depend on the orientation in which the gluing of each of the two Eguchi-Hanson metrics is performed: we obtain different oriented and simply connected $4$-manifolds,
\begin{enumerate}
	\item gluing both in their anti-selfdual (or hyperkähler) orientation, leading to a $4$-manifold $M_{--}$, which we will see is homeomorphic to $\overline{\mathbb{CP}}^2\#\overline{\mathbb{CP}}^2$,
	\item gluing both in their selfdual orientation, leading to a $4$-manifold $M_{++}$, which we will see is homeomorphic to ${\mathbb{CP}}^2\#{\mathbb{CP}}^2$, or
	\item gluing one in each orientation, leading to a $4$-manifold $M_{+-}$, which we will see is homeomorphic to $\mathbb{S}^2\times\mathbb{S}^2$.
\end{enumerate}
In order to reach these conclusions, we appeal to Freedman's classification of $4$-manifolds up to homeomorphism by their intersection form. First, note that $b_+(M_{--}) = 0$ and $b_-(M_{--}) = 2$, $b_+(M_{++}) = 2$ and $b_-(M_{++}) = 0$, and $b_+(M_{+-}) = 1$ and $b_-(M_{+-}) = 1$, which already confirms the homeotype of $M_{--}$ and $M_{++}$. As for $M_{+-}$, this leaves two options: $\mathbb{S}^2\times\mathbb{S}^2$ or $\mathbb{CP}^2\#\overline{\mathbb{CP}}^2$. The homeomorphism of $M_{+-} \approx \mathbb{S}^2\times\mathbb{S}^2$ can be obtained by different techniques, but we will treat the general case of subgroups of $\operatorname{SU(2)}$.

In the case of other quotients $\mathbb{S}^3/\Gamma$ for $\Gamma\subset \operatorname{SU(2)}$, the manifolds we are considering are ``doublings'' as follows. Denoting $X_\Gamma$ the manifold with boundary $\mathbb{S}^3/\Gamma$ obtained from cutting the end $(\mathbb{R}^4\slash\Gamma)\backslash \overline{B_\mathbf{e}(0,\varepsilon_0^{-1})}$ in the ALE coordinates of a hyperkähler ALE manifold, we define $M_{\Gamma}$ as the gluing of $X_\Gamma$ to $-X_\Gamma$ along their common boundary $\mathbb{S}^3/\Gamma$. In the case of $\Gamma = \mathbb{Z}_2$, the obtained $4$-manifold is the above $M_{+-}$.

The manifolds with boundary $X_\Gamma$ and $-X_\Gamma$ are spin manifolds whose induced spin structures on the boundary $\mathbb{S}^3/\Gamma$ coincide by construction. This is exactly the condition that ensures that the manifold $M_\Gamma$ admits a spin structure. Being spin and having zero signature, and $b_+ = b_- = k_\Gamma$ for $k_{A_k}=k_{D_k}=k_{E_k}=k$, we conclude that $M_{\Gamma}$ is homeomorphic to $\#k_\Gamma (\mathbb{S}^2\times \mathbb{S}^2)$.

Through Mayer-Vietoris sequence, the intersection form of these manifolds can also be computed and compared to those of $\#k(\mathbb{S}^2\times\mathbb{S}^2)$. As an illustration, an interesting case is that of the group $E_8$. Indeed, the boundary $\mathbb{S}^3/E_8$ is the Poincaré homology sphere, hence the intersection form of $M_{E_8}$ is simply the direct sum of the intersection forms of $X_{E_8}$ and $-X_{E_8}$ which is $E_8\#(-E_8) \approx \#8(\mathbb{S}^2\times \mathbb{S}^2)$.



\end{proof}

\begin{rk}
	The examples obtained from desingularizing the ``water drop'' orbifold from Example \ref{ex:waterdrop} with one singularity are not orientable, and are $\mathbb{Z}_2$ quotients of the above manifolds $M_\Gamma$. Other specific quotients of toric flows on $M_{A_{dn}}$ for $d\geq 1$ and $n\geq 2$, by a cyclic group of order $n$ yield additional ancient flows with an orbifold tangent soliton with two singularities associated to cyclic subgroups of $\operatorname{U(2)}$.
\end{rk}
\subsection{Curvature estimates}
The main result of this section for ancient flows is the following theorem which rephrases Theorem \ref{mainthm1} in terms of the scale $\varepsilon_0$ and parameters introduced in \eqref{choice-parameters}.
\begin{theo}\label{theo-exist-RF}
Under the setting of Proposition \ref{prop-prelim-schauder}, there exists $T<0$ and an ancient solution $(g(t))_{t\,\leq \,T}$  to the normalized Ricci flow $\partial_tg = -2\Ric(g)+\Lambda g$ on $M$ satisfying:
\begin{enumerate}
\item(Curvature tensor)
For $t\leq T$,
\begin{equation*}
C_1\,\varepsilon_0(t)^{-2}\leq \sup_M|\Rm(g(t))|_{g(t)}\leq C_2\,\varepsilon_0(t)^{-2}\to +\infty \text{ as }t\to -\infty,
\end{equation*}
for some positive uniform constants $C_1$ and $C_2$.\\
\item (Covariant derivatives of the curvature tensor)
For $k\geq 0$, there exists $C_k>0$ such that for $t\leq T$:
\begin{equation*}
\sup_M|\nabla^{g(t),\,k}\Rm(g(t))|_{g(t)}\leq C_k\,\varepsilon_0(t)^{-2}.
\end{equation*}

\item(Ricci curvature)
    There exists $C$ such that for $t\leq T$, 
    \begin{equation*}
    \sup_M|\Ric(g(t))|_{g(t)}\leq C.
    \end{equation*}
   \item (Scalar curvature) There exists $C>0$ such that for $t\leq T$,
\begin{equation*}
\sup_M|\R_{g(t)}-4\Lambda|\leq C\varepsilon_0(t)^{\iota\sigma}\to 0 \text{ as }t\to -\infty.
\end{equation*}
\item (Local Ricci pinching) For $\theta>0$ sufficiently small with respect to the supremum of the diameters of $(M,\tilde{g}(t))_{t\,\leq\, T}$, there exists $C(\theta)>0$ such that for $t\leq T$:
    \begin{equation*}
\sup_{\{r_o\,\geq\, \theta\delta(t)\}}|\Ric(g(t))-\Lambda g(t)|_{g(t)}\leq C(\theta)\varepsilon_0(t)^{\iota\sigma}\to 0 \text{ as }t\to -\infty.
\end{equation*}

\item The metrics $g(t)$ converge locally smoothly on $M_0$ to $g_o$ outside the singular set and for $t\leq T$, $d_{GH}(M,d_{g(t)}),(M_o,d_o))\leq C\varepsilon_0(t)^{\delta}.$
\end{enumerate}
Moreover, if $\Gamma=\mathbb{Z}_2$ and if we choose the orientations on the Eguchi-Hanson instantons $(N,\eh)$ such that they become anti-selfdual,
\begin{enumerate}
\item ($\mathbf{R}^+$-pinching) There exists $C>0$ such that for $t\leq T$,
\begin{equation*}
\sup_M\left|\mathbf{R}^+(g(t))-\frac{\Lambda}{3}\Id_{\Lambda^+}\right|_{g(t)}\leq C\varepsilon_0(t)^{\iota\sigma}\to 0 \text{ as }t\to -\infty.
\end{equation*}
In particular, the ancient solution $(g(t))_{t\,\leq\,T}$ satisfies the $\operatorname{PIC}_+$ curvature condition.

\item (Global Ricci pinching)
There exists $C>0$ such that for $t\leq T$,
\begin{equation*}
\Ric(g(t))\geq \left(\Lambda-1-C\varepsilon_0(t)^{\iota\sigma}\right)g(t).
\end{equation*}
In particular, there exists $c>0$ such that for $t\leq T$,
\begin{equation*}
\Ric(g(t))\geq c\R_{g(t)}g(t).
\end{equation*}
\item (Almost selfduality of $g(t)$)
For $t\leq T$,
\begin{equation*}
\sup_M|W^+(g(t))|_{g(t)}\leq C\varepsilon_0(t)^{\iota\sigma}\to 0 \text{ as }t\to -\infty,
\end{equation*}
for some uniform positive constant $C$.

    \end{enumerate}

\begin{rk}
	The above almost ``selfdual metrics of constant scalar curvature'' provided by the second part of Theorem \ref{theo-exist-RF} are reminiscent of the behavior of Poon-LeBrun anti-selfdual metrics \cite{Poon1986,LeBrun1991} with constant scalar curvature on $\overline{\mathbb{CP}}^2\#\overline{\mathbb{CP}}^2$. The instability detected here is similar to that of \cite{Viaclovsky} in the setting of the Yamabe problem: among the finite-dimensional space of symmetric Poon-LeBrun metrics on $\overline{\mathbb{CP}}^2\#\overline{\mathbb{CP}}^2$, the sine-cone over $\mathbb{RP}^3$ is shown to be \textit{minimizing} the Yamabe constant, making it an \textit{unstable} critical point in the Yamabe setting.
\end{rk}
\end{theo}

\begin{rk}
In the main result of \cite{Bre-Kap} is shown that the Ricci curvature converges to $0$ as $t\rightarrow-\infty$ on the whole of $M$. Theorem \ref{theo-exist-RF} only ensures the convergence of $\Ric(g(t))-\Lambda g(t)$ to $0$ away from the singular set of $M_o$. This is explained by the absence of obstructions in \cite{Bre-Kap}, which formally speaking amounts to having the righthand side of \eqref{hello-beauty} to be $0$ in Proposition \ref{prop-h2-H2}.
\end{rk}

\begin{rk}
 Theorem \ref{theo-exist-RF} applies to finitely many singular points provided the sizes of the Eguchi-Hanson or other hyperkähler ALE metrics glued near those points are comparable and provided the underlying Einstein orbifold $(M_o^4,g_o)$ is stable.
\end{rk}

The proof of Theorem \ref{theo-exist-RF} is given at the end of this section. We start with estimating each curvature conditions on the perturbed bubble then on the background metric $\tilde{g}(t)$ and then proceed to analyzing the solution to the Ricci-DeTurck  flow with background metric $\tilde{g}(t)$ and the resulting solution to the normalized Ricci flow.\\

\begin{rk}
The main statement and its proof about the corresponding immortal flows reaching a hyperbolic orbifold from Theorem \ref{mainthm2} are virtually the same as those of Theorem \ref{theo-exist-RF}
 by imposing $\Lambda=-3$ and by flipping the sign of the time coordinate. Therefore we omit the proof of the curvature estimates stated in Theorem \ref{mainthm2}.
\end{rk}





\subsubsection{Curvature estimates on the perturbed bubble}

On Eguchi-Hanson $(N^4,\eh )$, we can compute the exact norm of the obstructions in order to show that the Ricci curvature stays pinched.

Recall the metric
\begin{equation}
    \mathbf{eh}:= \sqrt{\frac{r^4}{1+r^4}}(dr^2+r^2\alpha_1^2) + \sqrt{1+r^4}(\alpha_2^2+\alpha_3^2),\label{eguchihanson2}
\end{equation}
and the infinitesimal Ricci-flat ALE deformations:
       \begin{enumerate}
       \item $\mathbf{o}_1:=-\frac{r^2}{(1+r^4)^{3/2}}(dr^2+r^2 \alpha_1^2) +\frac{1}{(1+r^4)^{1/2}}(\alpha^2_2+\alpha_3^2)$,
       \item $\mathbf{o}_2:= \frac{1}{1+r^4}(rdr\cdot\alpha_2-r^2\alpha_1\cdot\alpha_3)$,
       \item $\mathbf{o}_3:= \frac{1}{1+r^4}(rdr\cdot\alpha_3+r^2\alpha_1\cdot\alpha_2)$.
   \end{enumerate}

\begin{lemma}\label{lemma-pt-norm-oi}
    We have the following pointwise norms for $i=1,2,3$:
    \begin{equation*}
    |\mathbf{o}_i|_{\mathbf{eh}} = \frac{2}{1+r_b^4}.
    \end{equation*}
\end{lemma}
\begin{proof}
    Consider the orthonormal basis of $1$-forms $e^0 := \Big(\frac{r^4}{1+r^4}\Big)^\frac{1}{4}dr$, $e^1:=\Big(\frac{r^4}{1+r^4}\Big)^\frac{1}{4}r\alpha_1$, $e^2:= (1+r^4)^\frac{1}{4}\alpha_2$ and $e^3:= (1+r^4)^\frac{1}{4}\alpha_3$ for $\mathbf{eh}$. We then notice that
    \begin{enumerate}
       \item $\mathbf{o}_1=\frac{1}{1+r^4}(-(e^0)^2-(e^1)^2+(e^2)^2+(e^3)^2)$,
       \item $\mathbf{o}_2= \frac{1}{1+r^4}(e^0e^2 +e^2e^0 -e^1e^3-e^3e^1)$
       \item $\mathbf{o}_3= \frac{1}{1+r^4}(e^0e^3+e^3e^0+e^1e^2+e^2e^2)$.
   \end{enumerate}
   In particular, we see that 
   $$|\mathbf{o}_i|_{\mathbf{eh}} = \frac{2}{1+r^4},$$
   as desired.
\end{proof}

We are in a good position to state and prove curvature estimates at the linear level on the perturbed bubble:
\begin{prop}\label{prop-lin-est-curv}
    If $\mathbf{R}^+_p = \pm\Id_{\Lambda^+}$ then one has for $\zeta\in\RR^3\setminus\{0\}$:
    \begin{enumerate}
    \item \label{first-prop-lin-est-curv}$d_{\mathbf{eh}_{\zeta}}\mathring{\Ric}(h_{2,\zeta}) = \sum_i \lambda_i \oi_i(\zeta) = \mp\oi_1(\zeta)$,
    \item $d_{\mathbf{eh}_{\zeta}}\mathbf{R}^+(h_{2,\zeta})= \pm\Id_{\Lambda^+}$, and consequently
    \item $d_{\mathbf{eh}_{\zeta}}W^+(h_{2,\zeta}) = 0$.
\end{enumerate}
In particular, on $(N^4,\eh _{\zeta})$,
\begin{equation*}
\left|d_{\mathbf{eh}_{\zeta}}{\Ric}(h_{2,\zeta}) - \Lambda\mathbf{eh}_{\zeta}\right|_{\eh _{\zeta}} = \frac{2|\zeta|^2}{|\zeta|^2+r_o^4},
\end{equation*}
and the following pinching holds:
\begin{equation*}
\left(\Lambda-\frac{|\zeta|^2}{|\zeta|^2+r_o^4}\right)\mathbf{eh}_{\zeta}\leq d_{\mathbf{eh}_{\zeta}}{\Ric}(h_{2,\zeta}) \leq \left(\Lambda+\frac{|\zeta|^2}{|\zeta|^2+r_o^4}\right)\mathbf{eh}_{\zeta}.
\end{equation*}
\end{prop}
\begin{proof}
Let us consider $h_2$ an extension of $H_2$ as in Proposition \ref{prop-h2-H2}. At the linear level, it induces the following curvature components on $\mathbf{eh}$, see \cite{Biq-2}. 
\begin{itemize}
    \item $d_{\mathbf{eh}}\mathring{\Ric}(h_2) = \sum_i \lambda_i \mathbf{o}_i = -\sum_iR^+_{1i}\mathbf{o}_i$, where $R^+_{1i}$ are the components of the matrix of the curvature $\mathbf{R}^+_p:\Lambda^+\to\Lambda^+$ in the orthonormal basis $(\omega_i^+)_i$ of $\Lambda^+=\Lambda^+(\mathbb{R}^4)$, i.e. $\mathbf{R}^+_p(\omega_i^+) =  \sum_{j}R_{ji}^+\omega_j^+=\sum_{j}R_{ij}^+\omega_j^+$.
    \item $d_{\mathbf{eh}}\mathbf{R}^+(h_2)(\omega_i^+) = \sum_{ij} R_{ij}^+  \omega_j^+$. 
\end{itemize}
This implies the three first expected properties since $\eh _{\zeta}=|\zeta|(s_{\zeta})_{*}\eh $ and $h_{2,\zeta}=|\zeta|^2(s_{\zeta})_{*}h_2$.

Moreover, Lemma \ref{lemma-pt-norm-oi} for $i=1$ ensures that 
\begin{equation*}
\left|d_{\mathbf{eh}_{\zeta}}{\Ric}(h_{2,\zeta}) - 3\,\mathbf{eh}_{\zeta}\right|_{\eh _{\zeta}} =|\oi_1(\zeta)|_{\eh _{\zeta}}= \frac{2|\zeta|^2}{|\zeta|^2+r_o^4},
\end{equation*}
as desired since $\Lambda=3$.
Finally, 
since $-\frac{1}{1+r_b^4}\mathbf{eh}\leq\mathbf{o}_1\leq\frac{1}{1+r_b^4}\mathbf{eh}$, we also have the corresponding pinching estimate on $d_{\eh _{\zeta}}\Ric(h_{2,\zeta})$ thanks to item [\eqref{first-prop-lin-est-curv}, Proposition \ref{prop-lin-est-curv}] we just proved above.
\end{proof}
\subsubsection{Curvature estimates of the approximate solution}

The goal of this section is to prove curvature estimates similar to those of Proposition \ref{prop-lin-est-curv} at a non-linear level for the approximate solution $(\tilde{g}(t))_{t\,\leq \,T}$. More precisely:

\begin{prop}\label{prop-app-est-curv} 
Under the assumption and notations of Theorem \ref{theo-exist-RF}, there exists $T\leq 0$ such that,
\begin{enumerate}
\item(Curvature tensor)
For $t\leq T$,
\begin{equation*}
C_1|\zeta(t)|^{-1}\leq \sup_M|\Rm(\tilde{g}(t))|_{\tilde{g}(t)}\leq C_2|\zeta(t)|^{-1},
\end{equation*}
for some positive uniform constants $C_1$ and $C_2$ and there exists $C>0$ such that for $t\leq T$,
\begin{equation*}
\sup_M|\mathbf{R}^+(\tilde{g}(t))-\Id_{\Lambda^+}|_{\tilde{g}(t)}\leq C\varepsilon_0(t)^{2\delta}.
\end{equation*}
    \item(Ricci curvature)
    There exists $C>0$ such that for $t\leq T$, 
    \begin{equation*}
    \sup_M|\Ric(\tilde{g}(t))|_{\tilde{g}(t)}\leq C.
    \end{equation*}
    For $\theta>0$ sufficiently small with respect to the supremum of the diameters of $(M,\tilde{g}(t))_{t\,\leq\, T}$, there exists $C(\theta)>0$ such that for $t\leq T$:
    \begin{equation*}
\sup_{\{r_o\,\geq\, \theta\delta(t)\}}|\Ric(\tilde{g}(t))-\Lambda \tilde{g}(t)|_{\tilde{g}(t)}\leq C(\theta)\varepsilon_0(t)^{4-6\delta}.
\end{equation*}
\item (Scalar curvature) There exists $C>0$ such that for $t\leq T$,
\begin{equation*}
\sup_M|\R_{\tilde{g}(t)}-4\Lambda|\leq C\varepsilon_0(t)^{4-6\delta}.
\end{equation*}

\item (Global Ricci pinching)
There exists $C>0$ such that for $t\leq T$,
\begin{equation*}
\Ric(\tilde{g}(t))\geq \left(\Lambda-1-C\varepsilon_0(t)^{4-6\delta}\right)\tilde{g}(t).
\end{equation*}
\item (Almost selfduality of $\tilde{g}(t)$)
For $t\leq T$,
\begin{equation*}
\sup_M|W^+(\tilde{g}(t))|_{\tilde{g}(t)}\leq C\varepsilon_0(t)^{4-6\delta},
\end{equation*}
for some uniform positive constant $C$.
    \end{enumerate}
\end{prop}

\begin{proof}
The upper bound on the curvature tensor of $\tilde{g}(t)$ is a straightforward consequence of Lemma \ref{lemma-first-app-est}. The lower bound is obtained by estimating the curvature on the bubble region $\{r_o\,\leq\,\varepsilon(t)\}$ from below the same way we did for the upper bound in the proof of Lemma \ref{lemma-first-app-est}.
 Therefore, for $t\leq T$,
\begin{equation*}
C_1|\zeta(t)|^{-1}\leq \sup_{M}|\Rm(\tilde{g}(t))|_{\tilde{g}(t)}\leq C_2|\zeta(t)|^{-1},
\end{equation*}
 for some positive time-independent constants $C_1$ and $C_2$. 
 
Now, $\Ric(\tilde{g}(t))-\Lambda\tilde{g}(t)=0$ on $\{r_o>2\delta(t)\}$ by construction of $\tilde{g}(t)$ and choice of $g_o$. On $\{r_o<\delta(t)/2\}$,
\begin{equation}
\begin{split}\label{lin-ric-model-back-g}
\Ric(\tilde{g}(t))-\Lambda\tilde{g}(t)&=\Ric(\eh _{\zeta(t)}+h_{2,\zeta(t)})-\Lambda(\eh _{\zeta(t)}+h_{2,\zeta(t)})\\
&=d_{\eh _{\zeta(t)}}\Ric(h_{2,\zeta(t)})-\Lambda \eh _{\zeta(t)}+\frac{1}{2}\mathcal{L}_{B_{\eh _{\zeta(t)}}(h_{2,\zeta(t)})}\tilde{g}(t)\\
&\quad+Q_{\eh _{\zeta(t)}}(h_{2,\zeta(t)})-\Lambda h_{2,\zeta(t)}\\
&=d_{\eh _{\zeta(t)}}\Ric(h_{2,\zeta(t)})-\Lambda \eh _{\zeta(t)}+Q_{\eh _{\zeta(t)}}(h_{2,\zeta(t)})-\Lambda h_{2,\zeta(t)},
\end{split}
\end{equation}
where we have used that $h_{2,\zeta(t)}$ is in Bianchi gauge by Proposition \ref{prop-h2-H2}.

Proposition \ref{prop-lin-est-curv} then gives on $\{r_o\geq \theta \delta(t)\}$,
\begin{equation*}
\begin{split}
|\Ric(\tilde{g}(t))-\Lambda\tilde{g}(t)|_{\tilde{g}(t)}&\leq C(\varepsilon(t)^4\theta^{-4}\delta(t)^{-4}+\delta(t)^2)\leq C(\theta)\varepsilon_0(t)^{\min\{4(1-\delta),\,2\delta\}}\\
&\leq C(\theta)\varepsilon_0(t)^{2\delta}
\end{split}
\end{equation*}
as desired since $\delta<2/3$ by \eqref{cond-param}. Here Proposition \ref{prop-h2-H2} is used again to handle the terms $Q_{\eh _{\zeta(t)}}(h_{2,\zeta(t)})$ and  $h_{2,\zeta(t)}$.

As for the global pinching, it follows from \eqref{lin-ric-model-back-g} and from Proposition \ref{prop-lin-est-curv} that:
\begin{equation*}
\begin{split}
\Ric(\tilde{g}(t))-\Lambda\tilde{g}(t)&\geq -\eh _{\zeta(t)}-C\delta(t)^2\tilde{g}(t)\\
&\geq -\left(1+C\delta(t)^2\right)\tilde{g}(t),
\end{split}
\end{equation*}
as expected on $\{r_o<\delta(t)/2\}$. Here we have invoked Lemma \ref{lemma-first-app-est} in the last line. In order to globalize the above pinching estimate, observe first that it holds trivially on $\{r_o>2\delta(t)\}$. Next, by linearizing $\Ric(\tilde{g}(t))-\Lambda\tilde{g}(t)$ over ${\{r_o<3\delta(t)/2\}}$ (respectively on ${\{r_o>\delta(t)\}}$) with respect to $\eh _{\zeta(t)}$ (respectively to $g_o$), the same estimate holds thanks to Lemma \ref{lemma-first-app-est} up to an error of magnitude $\varepsilon(t)^4\delta(t)^{-6}$. The boundedness of the Ricci tensor can be proved along a similar reasoning.

Regarding the behavior of the scalar curvature $\R_{\tilde{g}(t)}$, it suffices to look at that of $\mathbf{R}^+(\tilde{g}(t))$ since $M$ has dimension $4$. Proposition \ref{prop-lin-est-curv} gives on $\{r_o<\delta(t)/2\}$:
\begin{equation*}
\mathbf{R}^+(\tilde{g}(t))=d_{\mathbf{eh}_{\zeta}}\mathbf{R}^+(h_{2,\zeta})+O(\delta(t)^2)= \Id_{\Lambda^+}+O(\delta(t)^2).
\end{equation*}
On $\{r_o>2\delta(t)\}$, $\mathbf{R}^+(\tilde{g}(t))=\mathbf{R}^+(g_o)=\Id_{\Lambda^+}$ by assumption on $g_o$. By linearizing $\mathbf{R}^+(\tilde{g}(t))$ over ${\{r_o<3\delta(t)/2\}}$ (respectively on ${\{r_o>\delta(t)\}}$) with respect to $\eh _{\zeta(t)}$ (respectively to $g_o$), the same estimate holds thanks to Lemma \ref{lemma-first-app-est} again up to an error of magnitude $\varepsilon(t)^4\delta(t)^{-6}$.

As for the decay of $W^+(\tilde{g}(t))$, notice that thanks to Proposition \ref{prop-lin-est-curv}, the following estimate holds on ${\{r_o<3\delta(t)/2\}}$:
\begin{equation*}
\begin{split}
W^+(\tilde{g}(t))&=W^+(\eh _{\zeta(t)}+h_{2,\zeta(t)})+O(\varepsilon(t)^4\delta(t)^{-6})\\
&=\int_0^1d^2_{\eh _{\zeta(t)+sh_{2,\zeta(t)}}}W^+(h_{2,\zeta(t)},h_{2,\zeta(t)})(1-s)\,ds+O(\varepsilon(t)^4\delta(t)^{-6})\\
&=O(\delta(t)^2)+O(\varepsilon(t)^4\delta(t)^{-6})=O(\varepsilon(t)^4\delta(t)^{-6}),
\end{split}
\end{equation*}
since $\delta>1/2$ thanks to \eqref{cond-param}.
Here we have used Lemma \ref{gal-lemma-diff-curv} together with Lemma \ref{lemma-first-app-est} in the first line together with a Taylor expansion in the second line. Observe the necessity of the property $d_{\eh_{\zeta(t)}}W^+(h_{2,\zeta(t)})=0$ coming from $W^+_p(g_o)=0$ here.
 A similar result holds when linearizing on ${\{r_o>\delta(t)\}}$ with respect to $g_o$ since $W^+(g_o)=0$.
 \end{proof}
\subsubsection{Curvature estimates of the solution}
We first recall and prove the following lemma that gives a well-known one-to-one correspondance between a Ricci flow and a Ricci flow with extra linear terms and a Lie derivative term,  
\begin{lemma}\label{lemma-dictionary}
For $T\leq 0$ and $\beta\in L^1(-\infty,T)$, let $(g(t))_{t\,\leq \,T}$ solve 
\begin{equation*}
\partial_tg(t)=-2\Ric(g(t))+(\beta(t)+2\Lambda)g(t)+\mathcal{L}_{V(t)}g(t),\quad\text{on $M\times(-\infty,T)$},
\end{equation*}
for some $C^1$ vector field $V(t)$. Let $(\varphi_t)_{t\,\leq\, T}$ be the flow of $-V(t)$.
 Then the metrics $\bar{g}(\bar{t}):=c(t)\varphi_t^*g(t)$ satisfies 
 \begin{equation*}
 \partial_{\bar{t}}\bar{g}(\bar{t})=-2\Ric(\bar{g}(\bar{t}))+2\Lambda\bar{g}(\bar{t}),\quad\text{on $M\times(-\infty,0)$},
 \end{equation*}
where $c(t):=\exp\left(-\int_{-\infty}^t\beta(s)\,ds\right)$ and $\bar{t}(t):=-\int_t^Tc(s)\,ds.$
\end{lemma}
\begin{proof}
It is a straightforward computation: define $\bar{g}(\bar{t}):=c(t)\varphi_t^*g(t)$ for some smooth positive function $c(\cdot)$ and a smooth function $\bar{t}$ to be defined later. The family of metrics $\bar{g}(\bar{t})$ satisfies:
\begin{equation*}
\begin{split}
\bar{t}'\partial_{\bar{t}}\bar{g}(\bar{t})&=\left(c'(t)+c(t)(\beta(t)+2\Lambda)\right)\varphi_t^*g(t)-2c(t)\Ric(\varphi_t^*g(t))\\
&=\left(c'(t)+c(t)(\beta(t)+2\Lambda)\right)\varphi_t^*g(t)-2c(t)\Ric(\bar{g}(\bar{t}))\\
&=(c(t)\left(-2\Ric(\bar{g}(\bar{t}))+2\Lambda \bar{g}(\bar{t})\right),
\end{split}
\end{equation*}
provided $\bar{t}'(t)=c(t)$ and $c'(t)+c(t)\beta(t)=0$ for all $t\in(-\infty,T)$. Here we have used the invariance of the Ricci curvature under scalings and diffeomorphisms of the metric in the second equality.

Therefore, if we choose $c(t):=\exp\left(-\int_{-\infty}^t\beta(s)\,ds\right)$ and $\bar{t}(t)$ to be $-\int_{t}^{T}c(s)\,ds$, we reach the desired conclusion.
\end{proof}

We are now in a position to prove Theorem \ref{theo-exist-RF}:
\begin{proof}[Proof of Theorem \ref{theo-exist-RF}]
Observe that since $g(t)$ satisfies 
\begin{equation*}
\partial_tg(t)=-2\Ric(g(t))+(2\Lambda+\beta(t)) g(t)+\mathcal{L}_{V(t)}(g(t)),
\end{equation*}
where $V(t)$ is defined as in Proposition \ref{prop-prelim-schauder},
 the solution to the normalized Ricci flow $\bar{g}(\bar{t})$ associated to $g(t)$ through Lemma \ref{lemma-dictionary} solving  $\partial_{\bar{t}}\bar{g}(\bar{t})=-2\Ric(\bar{g}(\bar{t}))+2\Lambda \bar{g}(\bar{t})$ only differs from $g(t)$ by rescaling in space (by a  factor $\exp\left(-\int_{-\infty}^t\beta(s)\,ds\right)=1+O(\varepsilon_0(t)^2)$), time-reparametrization and diffeomorphisms of $M$. In particular, it suffices to prove the desired curvature estimates for $g(t)$ only.
 
 Now, linearizing the curvature operator of $g(t)$ around $\tilde{g}(t)$ via Lemma \ref{gal-lemma-diff-curv} gives:
 \begin{equation*}
 \begin{split}
|\Rm(g(t))-\Rm(\tilde{g}(t))|_{g(t)}&\leq C\left(|\nabla^{\tilde{g}(t),2}k(t)|_{g(t)}+|\nabla^{\tilde{g}(t)}k(t)|^2_{g(t)}+|\Rm(\tilde{g}(t))|_{g(t)}|k(t)|^2_{g(t)}\right)\\
&\leq C\left(\gamma(t)^{-1}\varepsilon_0(t)^{-2-\sigma}+\gamma(t)^{-2}\varepsilon_0(t)^{-2-2\sigma}\right)\\
&\leq C\varepsilon_0(t)^{\iota\sigma},
\end{split}
\end{equation*}
so that the triangular inequality leads to the expected estimate through Proposition \ref{prop-app-est-curv}.

Similarly, linearizing the Ricci curvature of $g(t)$ around $\tilde{g}(t)$ gives:
\begin{equation}
\begin{split}\label{lin-ric-tilde-g-g}
\Ric(g(t))-\Lambda g(t)&=\Ric(\tilde{g}(t))-\Lambda \tilde{g}(t)-\frac{1}{2}\Delta_{L,\tilde{g}(t)}k(t)+\frac{1}{2}\Li_{B_{\tilde{g}(t)}(k(t))}\tilde{g}(t)+Q_{\tilde{g}(t)}(k(t))-\Lambda k(t)\\
&=\Ric(\tilde{g}(t))-\Lambda \tilde{g}(t)+O(\gamma(t)^{-1}\varepsilon_0(t)^{-2-\sigma})+O(\gamma(t)^{-2}\varepsilon_0(t)^{-2-2\sigma})\\
&=\Ric(\tilde{g}(t))-\Lambda \tilde{g}(t)+O(\varepsilon_0(t)^{\iota\sigma}),
\end{split}
\end{equation}
which gives the expected result by invoking Proposition \ref{prop-app-est-curv} and \eqref{choice-parameters} to ensure that $4-6\delta>\iota\sigma$. Here we have used the fact that $|\Rm(\tilde{g}(t))|_{\tilde{g}(t)}=O(\varepsilon_0(t)^{-2})$ ensured by Lemma \ref{lemma-first-app-est}. As for the pinching of the Ricci curvature, we use \eqref{lin-ric-tilde-g-g} together with Proposition \ref{prop-app-est-curv} to get the expected result.

Now, linearizing the component $\mathbf{R}^+(g(t))$ of the curvature tensor of $g(t)$ around $\tilde{g}(t)$ gives thanks to Lemma \ref{gal-lemma-diff-curv}:
\begin{equation*}
|\mathbf{R}^+(g(t))-\mathbf{R}^+(\tilde{g}(t))|_{g(t)}\leq C\left(\gamma(t)^{-1}\varepsilon_0(t)^{-2-\sigma}+\gamma(t)^{-2}\varepsilon_0(t)^{-2-2\sigma}\right)\leq C\varepsilon_0(t)^{\iota\sigma},
\end{equation*}
so that the triangular inequality together with Proposition \ref{prop-app-est-curv} and \eqref{choice-parameters} to ensure that $4-6\delta>\iota\sigma$ lead to the desired decay.

Finally, linearizing the selfdual component of the Weyl tensor of $g(t)$ around $\tilde{g}(t)$ gives thanks to Lemma \ref{gal-lemma-diff-curv}:
\begin{equation*}
|W^+(g(t))-W^+(\tilde{g}(t))|_{g(t)}\leq C\left(\gamma(t)^{-1}\varepsilon_0(t)^{-2-\sigma}+\gamma(t)^{-2}\varepsilon_0(t)^{-2-2\sigma}\right)\leq C\varepsilon_0(t)^{\iota\sigma},
\end{equation*}
so that the triangular inequality together with Proposition \ref{prop-app-est-curv} and \eqref{choice-parameters} to ensure that $4-6\delta>\iota\sigma$ lead to the desired decay. Indeed, the constraints from \eqref{choice-parameters} $(6-\sigma)\delta<4-\sigma$ and $\iota<1/3$ imply $4-6\delta>\iota\sigma$.

Regarding the covariant derivatives of $\Rm(g(t))$, it suffices to invoke Shi's estimates for a solution $(\bar{g}(\bar{t}))_{\bar{t}\,\leq\, 0}$ to the Ricci flow that we now recall: if $k\geq 0$, there exists $C_k>0$ such that if $\bar{T}< 0$,
\begin{equation}\label{shi-est}
\sup_M|\nabla^{\bar{g}(\bar{T}),\,k}\Rm(\bar{g}(\bar{T}))|_{\bar{g}(\bar{T})}\leq \frac{C_k}{(-\bar{T})^{\frac{k}{2}}}\sup_{M\times[2\bar{T},\bar{T}]}|\Rm(\bar{g}(\bar{t}))|_{\bar{g}(\bar{t})}.
\end{equation}
Now, by adapting the proof of Lemma \ref{lemma-dictionary}, $\bar{g}(\bar{t})=c(t)\varphi_t^*g(t)$ with $c'(t)+(2\Lambda+\beta(t))c(t)=0$ and $\bar{t}'(t)=c(t)$ for $t\leq T$ and $\partial_t\varphi_t=-V(t)\circ\varphi_t$, so that \eqref{shi-est} gives for $\bar{t}(t)=\bar{T}$:
\begin{equation*}
\begin{split}
\sup_M|\nabla^{g(t),\,k}\Rm(g(t))|_{g(t)}&\leq \frac{C_kc(t)^{\frac{k}{2}+1}}{(\int_t^Tc(s)\,ds)^{\frac{k}{2}}}\sup_{M\times[\bar{t}^{-1}(2\bar{T}),\,\bar{t}^{-1}(\bar{T})]}c(s)^{-1}|\Rm(g(s))|_{g(s)}\\
&\leq C_k\sup_{M\times[\bar{t}^{-1}(2\bar{T}),\,\bar{t}^{-1}(\bar{T})]}|\Rm(g(s))|
_{g(s)}\leq C_k\varepsilon_0(t)^{-2},
\end{split}
\end{equation*}
since $\varepsilon_0$ is decreasing in time. Here we have used the curvature estimate proved above. Notice that implicit here is the consequence from Shi's estimates that $g(t)$ hence $k(t)$ is smooth in space.

Finally, by standard interpolation inequalities, on $M_o\setminus \{r_o\leq r\}$ for a fixed $r>0$ and any $a>0$ small enough with respect to the parameters $\delta$ and $\iota\sigma$, there exists $T<0$ such that for $t\leq T$,
\begin{equation}\label{cov-der-einstein}
|\nabla^{g(t),\,k}(\Ric(g(t))-2\Lambda g(t))|_{g(t)}\leq C_k(a,r)\varepsilon_0(t)^{\iota\sigma-a},
\end{equation}
for some positive constant $C_k(a,r)$. In particular, the same holds true for the solution to the normalized Ricci flow associated to $g(t)$. Indeed since here the estimate holds away from the singular set we only need to observe that the diffeomorphisms $\varphi_t$ generated by $-V(t)$ preserve uniformly in time any neighborhood of a point that lies outside the singular set of $M_o$. To convince ourselves of this fact, let us compute the time derivative of $\varphi_t^*r_o$:
\begin{equation}\label{gron-pull-back}
\begin{split}
|\partial_t\varphi_t^*r_o|\leq |V(t)|_{g(t)}\circ \varphi_t&\leq C\gamma(t)^{-1}\varepsilon_0(t)^{-1-\sigma}+C\varepsilon_0(t)^2\varphi_t^*r_o\\
&\leq C\varepsilon_0(t)^2\varphi_t^*r_o+C\varepsilon_0(t)^{1+\iota\sigma},
\end{split}
\end{equation}
where we have used the fact that $\nabla^{\tilde{g}(t)}\tilde{v}(t)=O(\rho(t))$ uniformly in time by Lemma \ref{est-basic-conf}. Gr\"onwall inequality applied to \eqref{gron-pull-back} implies the desired result if we impose $\varphi_t\big|_{t=T}=\Id_M$.

Now that we know that \eqref{cov-der-einstein} holds for the associated solution to the normalized Ricci flow, integration in time gives smooth local convergence of this solution to the orbifold metric $g_o$ away from the singularities at an exponential rate.

We are left with proving the global convergence of  the associated solution to the normalized Ricci flow  to $g_o$ in the Gromov-Hausdorff topology: by Lemma \ref{lemma-dictionary}, the associated normalized Ricci flow  differs from $g(t)$ by diffeomorphisms, scaling (by a factor converging to $1$ as fast as $\int^t_{-\infty}\beta(s)\,ds=O(\varepsilon_0(t)^2)$ and a time-reparametrization so that it suffices to prove that $g(t)$ converges to $g_o$ in the Gromov-Hausdorff topology as $t$ goes to $-\infty$. Now, since 
\begin{equation*}
(1-C\gamma(t)^{-1}\varepsilon_0(t)^{-\sigma})\tilde{g}(t)\leq g(t)\leq (1+C\gamma(t)^{-1}\varepsilon_0(t)^{-\sigma})\tilde{g}(t),\quad t\leq T,
\end{equation*}
we know that $d_{GH}((M,d_{\tilde{g}(t)}), (M,d_{g(t)}))\leq C\varepsilon_0(t)^{1+\iota\sigma/2}$ for $t\leq T$. In particular, it suffices to show that $\tilde{g}(t)$ converges to $g_o$ in the Gromov-Hausdorff topology: since the diameter of the bubbles are $O(\delta(t))$, this immediately implies the expected result, i.e. 
\begin{equation*}
\begin{split}
d_{GH}((M,d_{g_o}), (M,d_{g(t)}))&\leq d_{GH}((M,d_{g_o}), (M,d_{\tilde{g}(t)}))+d_{GH}((M,d_{\tilde{g}(t)}), (M,d_{g(t)}))\\
&\leq C\varepsilon_0(t)^{\delta}+C\varepsilon_0(t)^{1+\iota\sigma/2}\leq C\varepsilon_0(t)^{\delta}.
\end{split}
\end{equation*}
Here we use the map $\Pi_t:M\rightarrow M_o$ defined for each $t\leq T$ such that if $x\in\{r_o>2\delta(t)\}$, $\Pi_t(x)=x$ and otherwise, $\Pi_t(x)=o$ the singular point of $M_o$ as an $C\delta(t)$-isometry.
\end{proof}

\subsection{Scales along the non renormalized Ricci flow}\label{sec:not renormalized scales}

Let us finally determine the various scales along the original Ricci flow, i.e. along the unnormalized Ricci flow. Thanks to Lemma \ref{lemma-dictionary}, we see that to go from $\Bar{g}(\bar{t})$ satisfying $\partial_{\bar{t}}\Bar{g}(\bar{t}) = -2(\Ric(\Bar{g}(\bar{t}))-\Lambda \Bar{g}(\bar{t}))$ for $\Lambda\in\mathbb{R}$ to a flow $\hat{g}(\hat{t})$ satisfying $\partial_{\hat{t}}\hat{g}(\hat{t}\,) = -2\Ric(\hat{g}(\hat{t}\,))$, we may perform the following changes of variables:
$$ \hat{t} := -\frac{e^{-2\Lambda \bar{t}}}{2\Lambda},\quad \text{ and }\quad \hat{g}(\hat{t}\,) :=e^{-2\Lambda \bar{t}} \,\,\Bar{g}(\bar{t}\,) = \big(-2\Lambda \hat{t}\big) \,\,\Bar{g}_{-\frac{\log(-2\Lambda \hat{t}\,)}{2\Lambda}}.$$

Now, as expected, this tells us that the resulting metric in the orbifold regions $r_o>\delta_0>0$ evolves at leading order like an honest Einstein metric: $$\hat{g}(\hat{t}\,) = \big(-2\Lambda \hat{t} \,\big)g_o $$ as $-2\Lambda \hat{t}$ tends to $+\infty$. Similarly, one may look at the size of the Eguchi-Hanson metric (or any other Ricci-flat ALE metric) through Kronheimer's period map. At leading order, we have $\bar{\zeta}(\bar{t}\,) = e^{2\bar{t}\mathbf{R}^+} \zeta_0$, which in the ``hat'' variables becomes
$$ \hat{\zeta}(\hat{t}\,) = (-2\Lambda \hat{t}\,) \times \exp\Big(\log(-2\Lambda\hat{t}\,){\frac{2\mathbf{R}^+}{-2\Lambda}}\Big) \zeta_0 .$$

\begin{exmp}
    Let us detail the various scales on the main flows we construct:
    \begin{itemize}
        \item For the unit round sphere, with $\Lambda = 3$ and $\mathbf{R}^+ = \operatorname{I}_3$, we find that the orbifold metric evolves by $\hat{g}(\hat{t}\,) = (-6 \hat{t}\,) g_{\mathbb{S}^4} $ while the ALE region evolves by 
        $$ \hat{\zeta}(\hat{t}\,) = (-6\hat{t}\,)^{1-\frac13} \zeta_0 = (-6\hat{t}\,)^{\frac23} \zeta_0. $$
        Along the actual Ricci flow, the curvature is therefore \textit{decaying} as well as $\hat{t}\to -\infty$: it is in $(-6\hat{t}\,)^{-\frac23}$ while the curvature on the orbifold is in $(-6\hat{t}\,)^{-1}$.\\
        \item For hyperbolic and complex hyperbolic orbifolds (in the orientation opposite to the Kähler one), with $\Lambda = -3$ and $\mathbf{R}^+ = -\operatorname{I}_3$, we find that the orbifold metric evolves by $\hat{g}(\hat{t}\,) = (6 \hat{t}\,) g_{\mathbb{H}^4} $ while the ALE region evolves by 
        $$ \hat{\zeta}(\hat{t}\,) = (6\hat{t}\,)^{\frac23} \zeta_0. $$
        Along the actual Ricci flow, the curvature is therefore \textit{decaying} like $(6\hat{t}\,)^{-\frac23}$ as $\hat{t}\to +\infty$.\\
    \end{itemize}

\end{exmp}

As computed in \cite{ozu3}, a higher order approximation of the flat orbifold $\mathbb{T}^4/\mathbb{Z}_2$ in the configuration considered in \cite{Bre-Kap} with $\zeta(t)$ parallel to $(1,0,0)$ is interpreted as the orbifold behaving like a non-flat space at higher order, and satisfying $$\mathbf{R}^\pm_{\zeta(t)} = c|\zeta(t)|^2 \begin{bmatrix}
    2&0&0\\
    0&-1&0\\
    0&0&-1
\end{bmatrix}$$ at its singular points for some positive constant $c$. The positivity of $c|\zeta(t)|^2\times 2$ in the first coefficient of the matrix explains the \textit{growth} of the bubbles along the ancient flow of Brendle and Kapouleas, just like positive eigenvalues of $\mathbf{R}^+_p(g_o)$ in the present article. In terms of producing an approximate Ricci flow as we did along this article, if $\zeta(t)$ stays in the direction $(1,0,0)$, then the ODE is 
\begin{equation}
    \dot{\zeta}(t) = \mathbf{R}^\pm_{\zeta(t)} \zeta(t) = 2c|\zeta|^2\zeta,
\end{equation}
and in particular $\varepsilon'(t) = c\varepsilon(t)^5$ using $\varepsilon(t) =|\zeta|^\frac{1}{2}$. This matches the ODE obtained in \cite{Bre-Kap}, where the condition $\zeta(t)$ in the direction $(1,0,0)$ is ensured by imposing a set of discrete symmetries at every step of the construction and along the Ricci flow. Note that it is crucial for their construction, to use these discrete symmetries to avoid negative eigenvalues $-c|\zeta(t)|^2$ of the selfdual curvature. This is also present in our work since we need to assume that $\mathbf{R}^+_p$ has eigenvalues of the same sign to run our arguments.

\section{General Ricci-flat ALE metrics}\label{sec-kro}
In this section, we mainly extend the results of the previous sections to other hyperkähler Ricci-flat ALE metrics. We then justify the instability of spherical singularities, and the stability of hyperbolic singularities for potential non Kähler Ricci-flat ALE metrics in Section \ref{sec:stab general RFALE}.

\subsection{Kronheimer's period map}\label{sec:kro per map}

As a transition with the next section, the hyperkähler 2-forms on the Eguchi-Hanson metric which are asymptotic to the $2$-forms $\omega_i^+(\mathbf{e}) = dr\wedge r \alpha_i+r^2\alpha_j\wedge\alpha_k$ for $(i,j,k)$ cyclic permutation of $(1,2,3)$ are:
\begin{equation*}
\begin{split}
 \omega^+_1(\mathbf{eh})=\frac{r^3}{\sqrt{1+r^4}} dr\wedge \alpha_1 &+ \sqrt{1+r^4}\alpha_2\wedge\alpha_3, \quad \omega^+_2(\mathbf{eh})= r dr\wedge \alpha_2 + r^2\alpha_3\wedge\alpha_1,\\
 & \omega^+_3(\mathbf{eh})= r dr\wedge \alpha_3 + r^2\alpha_1\wedge\alpha_2,
 \end{split}
 \end{equation*}
which are asymptotic to the $2$-forms $\omega_i^+(\mathbf{e})$. 

The key point of the period map description of the moduli space of hyperkähler ALE metrics is that denoting $\Sigma$ the $2$-sphere added at $\{r=0\}$, one has the following formula:
$$\int_\Sigma \omega_1^+(\mathbf{eh}) = \operatorname{vol}(\mathbb{S}^2) =4\pi \quad \text{ and for } i\in\{2,3\}, \quad\int_\Sigma \omega_i^+(\eh) =0.$$

Now, if the Eguchi-Hanson metric is rotated as described in \eqref{eh zeta+} for $\zeta=(x,y,z)\in\mathbb{R}^3$, then, the $2$-forms $\omega_i^+(\mathbf{eh}_\zeta)$ asymptotic to the $\omega_i^+(\mathbf{e})$ on $\mathbf{eh}_\zeta$ satisfy:
$$ \int_{\mathbb{S}^2} \omega_1^+(\mathbf{eh}_\zeta) = 4\pi x, \quad \int_{\mathbb{S}^2} \omega_2^+(\mathbf{eh}_\zeta) = 4\pi y, \quad \int_{\mathbb{S}^2} \omega_3^+(\mathbf{eh}_\zeta) = 4\pi z. $$

Up to a factor $4\pi$ (that simplifies the computations in the case of an Eguchi-Hanson metric), in the rest of the section, we extend this parametrization to all of the known examples of Ricci-flat ALE spaces, and compute their evolutions along Ricci flow in what follows.

In general, according to Kronheimer's period map \cite{Kronheimer1989Torelli}, hyperkähler ALE metrics may essentially be parametrized by $k$ elements $\zeta_i\in\mathbb{R}^3$ for $i\in\{1,...,k\}$ for some $k\geq 1$ depending on the group at infinity. In the rest of the section, we will compute the ``leading'' terms in $r^{-4}$ of the elements of $\mathbf{O}(g_\zeta)$. This will exhibit the link between the orthogonality of the leading terms of the elements of $\mathbf{O}(g_\zeta)$ at infinity and their $L^2$-orthogonality.

This study is the key to explicitly understand the obstructions to gluing-perturbing to Einstein metrics and therefore understand the dynamics of Ricci flow close to an orbifold with arbitrary singularities with group in $\operatorname{SU}(2)$.

For a codimension $3$ set of parameters $\mathcal{D}\subset \mathbb{R}^{3k}$, the associated metrics are singular and have additional orbifold singularities. This set is a cone over a codimension $3$ set in $\mathbb{S}^{3k-1}$ by homogeneity of the construction. We will always avoid this ``forbidden'' set $\mathcal{D}$ for simplicity in this article.

\begin{rk}
    Our construction should extend to trees of Ricci-flat ALE metrics. However, for Kähler Ricci-flat ALE metrics, this should not lead to new behaviors. Indeed, our analysis applies to flows starting arbitrarily close to the forbidden set $\mathcal{D}$. In other words, instead of developing trees of singularities at different scales, Ricci flow instead flows the whole tree of singularities at the same speed in our examples.
\end{rk}


\subsection{Stability of hyperkähler ALE metrics}

To lighten the notations, we denote $\omega_i^+ = \omega_i^+(g_b)$ the hyperkähler forms of a given hyperkähler metric $g_b$. 

The following result is well-known, but we add a short proof for completeness that will be helpful in the following sections.
\begin{prop}\label{linear-sta-instantons}
    Let $g_b$ be a hyperkähler ALE metric. Then, for any $h\in L^2$, one has $\langle \Delta_{L,g_b}h,h\rangle_{L^2}\leq 0$ with equality if and only if $h\in\ker_{L^2}\Delta_{L,g_b}$, where $$\ker_{L^2}\Delta_{L,g_b} = \operatorname{Span}\Big\{\Omega_j^-\circ\omega_i^+,\quad\Omega_j^-\in\mathcal{H}^-_{L^2}(g_b),\quad i\in\{1,2,3\}\Big\},$$ where $\mathcal{H}_{L^2}^-(g_b)$ is the space of harmonic, anti-selfdual $2$-forms in $L^2(g_b)$.
\end{prop}
\begin{proof}
    It is well known that the Lichnerowicz Laplacian preserves pure-trace and traceless part on Einstein metrics. Moreover, for any smooth function $f$, one has 
    $$ \Delta_{L,g_b} (f\cdot g_b) = (\Delta_{g_b} f) \cdot g_b, $$
    so it is enough to prove the result for a traceless $2$-tensor $h$.

    On traceless $2$-tensors, we may use the identification \eqref{eq:identification traceless} to see that any traceless $2$-tensor $h$ is equal to $h=\sum_i \phi_i^-\circ \omega_i^+$ for $\phi_i^-$ anti-selfdual $2$-forms, and $\omega_i^+$ a basis of Kähler forms of $g_b$. It is known from \cite[Formula (4.6)]{Biq-Rol-ACH} that in this decomposition, when the integration by parts is allowed (e.g. on a sequence of smooth compactly supported $2$-tensors approaching a given tensor in $L^2$), one has 
    \begin{equation*}
     \langle \Delta_{L,g_b}h,h\rangle_{L^2} = -4\sum_i\langle dd^*\phi_i^-,\phi_i^-\rangle_{L^2} =  -4\sum_i\|d^*\phi_i^-\|^2_{L^2}.
     \end{equation*}
    This proves the inequality $\langle \Delta_{L,g_b}h,h\rangle_{L^2}\leq 0$, and we moreover have equality if and only if the $\phi_i^-$ are all \textit{harmonic}, in which case $\Delta_{L,g_b}h=0$.
\end{proof}

\begin{prop}
    The elements of $\ker_{L^2}\Delta_{L,g_b}$ decay as fast as $(1+r)^{-4}$ at infinity.
\end{prop}
\begin{proof}
    Let $h\in\ker_{L^2}\Delta_{L,g_b}$, then, since the curvature is $o(r^{-2})$ at infinity (it is actually in $O(r^{-6})$), it is asymptotic to a harmonic tensor $h_0$ on the asymptotic flat model $\mathbb{R}^4/\Gamma$. It turns out that there are no critical weights for $\Delta_{L,g_b}$ between $0$ and $4$ \cite[Proposition 2.1]{Biq-1}.
\end{proof}

\subsection{Infinitesimal deformations of hyperkähler ALE metrics}\label{sec-inf-def-hyp}


Let us briefly describe the space of infinitesimal Ricci-flat ALE deformations of hyperkähler ALE metrics. As explained in Proposition \ref{linear-sta-instantons}, they can all be written as a linear combination of:
$$\mathbf{o}_{ij} = \Omega_j^- \circ \omega_i^+,$$
where $\omega_i^+ = \omega_i^+(g_b)$ are the three hyperkähler forms of the manifold asymptotic to the $2$-forms $\omega_i^+(\mathbf{e})=rdr\wedge\alpha_i+r^2\alpha_j\wedge\alpha_k$ on the Euclidean end (we drop the tilde for this section), and the $\Omega_j^-$ form a basis of the $L^2$-cohomology of the manifold, that is the space of closed $L^2$ anti-selfdual $2$-forms. In the next few sections, we give a somewhat geometric description of  a basis of $\Omega_j^-$ of the $L^2$-cohomology.

\subsubsection{Kronheimer's period map and homology}\label{period-map-homology}

On a $4$-manifold $N$ given by the diffeomorphic structure of the minimal resolution of $\mathbb{R}^4/\Gamma\approx\mathbb{C}^2/\Gamma$, consider $H^2_c(N)$ the cohomology group with compact support, and $k$ its dimension. Kronheimer's Torelli theorem \cite{Kronheimer1989Torelli} provides a bijection between hyperkähler ALE metrics and elements of $H^2_c(N)\otimes \mathbb{R}^3\approx (\mathbb{R}^3)^k$, up to the action of Diff$_0$ and of $\operatorname{SO}(3)$ (corresponding to the various bases of orthogonal triples of Kähler forms with norm $\sqrt{2}$). We will denote our parameter by $\zeta = (\zeta_1,...,\zeta_k)\in (\mathbb{R}^3)^k$ which determines a unique hyperkähler metric (up to the action of $\operatorname{SO}(3)$). These coefficients $\zeta$ can be chosen to correspond to $k$ codimension $2$ spheres $(\Sigma_j)_{j\in\{1,...,k\}}$ generating the homology of the manifold in the following manner:
$$\zeta_j = \Big(\int_{\Sigma_j}\omega_1^+, \; \int_{\Sigma_j}\omega_2^+, \; \int_{\Sigma_j}\omega_3^+\Big)\in \mathbb{R}^3.$$
\begin{rk}
	Up to a factor $4\pi$ present for simplicity, this corresponds to the $\zeta$ parameter we have used on an Eguchi-Hanson metric at the beginning of Section \ref{sec:kro per map}. 
\end{rk}
\begin{rk}
 As in the case of an Eguchi-Hanson metric, the action of $(\varepsilon,\phi)\in \mathbb{R}_+\times \operatorname{SO}(3)$ corresponding to $\zeta = (\zeta_1,...,\zeta_k)\mapsto \varepsilon^2(\phi^*\zeta_1,...,\phi^*\zeta_k) =: (\varepsilon,\phi)^*\zeta$ corresponds to a homothety: there exists a diffeomorphism $s_{(\varepsilon,\phi)}$ so that as in \eqref{Def-attaching-map}, one has:
 \begin{equation}\label{eq homothety kro}
 	g_{(\varepsilon,\phi)^*\zeta} = \varepsilon^2 (s_{(\varepsilon,\phi)})_* g_{\zeta}.
 \end{equation}
\end{rk}

The intersection of these $2$-spheres follows the Cartan matrix associated to the Dynkin diagram of the singularity. Denote $a_{ij}$ the intersection number of $\Sigma_i$ and $\Sigma_j$ and let $A=(a_{ij})_{ij}$. For instance, the matrix whose $i,j$ component is the intersection of $ \Sigma_i $ and $\Sigma_j$ for a singularity $\mathbb{R}^4/\mathbb{Z}_{5}$, the corresponding matrix is:
$$A=\begin{bmatrix}
    -2&1&0&0\\
    1&-2&1&0\\
    0&1&-2&1\\
    0&0&1&-2
\end{bmatrix}.$$

\subsubsection{From homology to cohomology}

A common procedure that lets one construct the elements $\Omega_j^-$ of the cohomology is as follows. Consider first $\gamma_j$, a representative of the Thom class of the normal bundle of $\Sigma_j$ supported on the $\varepsilon$-neighborhood of $\Sigma_j$ for $\varepsilon>0$. It is a closed $2$-form that is supported on a compact arbitrarily close to $\Sigma_j$.

Since $H^1_c(N)=\{0\}$, by standard weighted elliptic theory (or a specific case of the general results of \cite{Hausel-Hunsicker-Mazzeo}) there exists a unique $1$-form $\eta_j$ solving $d^*d\eta_j = -d^*\gamma_j$ decaying at infinity in $r^{-3}$, see also Proposition \ref{prop: inv linearizations Hodge}. We finally define:
$$\Omega_j^-:= \gamma_j+d\eta_j. $$

An integration by parts and the intersection of the $\Sigma_i$ lets us compute the $L^2$-products of our $\Omega_i^-$:
\begin{align*}
    \int_M\langle\Omega_i^-,\Omega_j^-\rangle_{g_{\zeta}} \,d\mu_{g_{\zeta}}&=-\int_M\Omega_i^-\wedge\Omega_j^- \\
&=-\int_M(\gamma_i+d\eta_i)\wedge(\gamma_j+d\eta_j)\\
    &= - \int_M\gamma_i\wedge\gamma_j \\
    &= - a_{ij}.
\end{align*}

Note that this is also equal to $-\int_{\Sigma_i}\Omega_j^-$. In particular, we find $\int_M|\Omega_i^-|_{g_{\zeta}}^2 \,d\mu_{g_{\zeta}} = -a_{ii} = 2$.

We denote $\mathcal{H}^2_{L^2}(g_b)$ the space of harmonic anti-selfdual $2$-forms (with respect to $g_b$) spanned by the $\Omega_j^-$.

\subsubsection{$L^2$-projections on the infinitesimal deformations}\label{l2-kronheimer-proj}

From the previous remarks and constructions, we define $ \mathbf{o}_{ij} := \Omega_j^-\circ \omega_i^+ $ which form a basis of the infinitesimal Ricci-flat deformation of hyperkähler ALE metrics. They are traceless, divergence free $2$-tensors decaying at least as $O(r^{-4})$ at infinity.

From the previous discussions, we moreover have:
\begin{equation}\label{eq:product oij kro}
\int_N \langle \mathbf{o}_{ij},\mathbf{o}_{kl} \rangle_{g_{\zeta}} \,d\mu_{g_{\zeta}} = - 2 \delta_{ik} a_{jl}, 
\end{equation}
and in particular $\int_N| \mathbf{o}_{ij}|_{g_{\zeta}}^2\,d\mu_{g_{\zeta}} = 4$ in the present conventions. 

Let us now consider $h$ a $2$-tensor and express its $L^2$-projection on the space of infinitesimal deformations in terms of the $L^2$-products 
$H_{ij}=\int_N \langle \mathbf{o}_{ij}, h \rangle_{g_{\zeta}} \,d\mu_{g_{\zeta}}.$ Denoting $ A^{-1} =(b_{ij})_{ij}$ the inverse of $A$,  and $\pi_{\mathbf{O}(g_\zeta)}h$ the $L^2$-orthogonal projection of $h$ on $\mathbf{O}(g_\zeta)$, from \eqref{eq:product oij kro} we have:
\begin{equation}
    \pi_{\mathbf{O}(g_\zeta)}h = -\frac{1}{2}\sum_{ijk} b_{jk} H_{ik} \mathbf{o}_{ij}.\label{eq L2 proj}
\end{equation}

\subsubsection{Expansion at infinity of the infinitesimal deformations and obstructions}
\label{exp-inf-def-obs}
In order to compute the first term at infinity of the above infinitesimal Ricci-flat ALE deformations. We deduce this term thanks to an integration by parts as follows.

Since $\Lambda^+$ and $\Lambda^-$ are \textit{pointwise} orthogonal, we have:
\begin{align}
    0&=\int_{N}\Omega_j^-\wedge \omega_i^+ \nonumber\\
    &=\int_{N}(\gamma_j+d\eta_j)\wedge \omega_i^+\nonumber\\
    &=\int_{N}\gamma_j\wedge \omega_i^+ +\lim_{s\to \infty}\int_{\{r=s\}}\eta_j\wedge \omega_i^+.\label{ipp Omegaj omegai}
\end{align}
Now, since by choosing $\varepsilon>0$ arbitrarily small, $\gamma_j$ is arbitrarily concentrated at $\Sigma_j$, the first integral $\int_{N}\gamma_j\wedge \omega_i^+$ is arbitrarily close to the integral of $\omega_i^+$ on $\Sigma_j$. This is exactly the $i$-th coordinate of $\zeta_j$ denoted $\zeta_j^i$. Since varying this $\varepsilon$ does not affect the second integral, $\lim_{s\to \infty}\int_{\{r=s\}}\eta_j\wedge \omega_i^+$, we always have $\int_{N}\gamma_j\wedge \omega_i^+=\zeta_j^i$.

On the other hand, we know that $\eta_j$ is harmonic in a neighborhood of infinity so by \cite{Biq-Hein} its leading term in $r^{-3}$ must be of the form 
$\frac{\langle L(x), \cdot \rangle}{r^4}$, for $x\in\mathbb{R}^4$ and $|x|= r$ for $L$ a $4\times 4$ matrix. Since $\omega_i^+(\mathbf{e}) = r dr\wedge \alpha_i + r^2\alpha_j\wedge\alpha_k $ for $(i,j,k)$ cyclic permutation of $(1,2,3)$ is anti-symmetric, we see that the symmetric part of $L$ will not contribute to the boundary term. The only part that will contribute is that proportional to $L=\omega_i^+(\mathbf{e})$ for which, since $\omega_i^+(\mathbf{e})(x)=\omega_i^+(\mathbf{e})(r\partial_r) = r^2 \alpha_i$ we find
$$\int_{\{r=s\}}\frac{\omega_i^+(\mathbf{e})(x)}{r^4}\wedge r^2\alpha_j\wedge\alpha_k= \int_{\{r=s\}}\frac{r^2\alpha_i}{r^4}\wedge r^2\alpha_j\wedge\alpha_k = \int_{\{r=s\}}\alpha_i\wedge \alpha_j\wedge\alpha_k = \frac{\vol(\mathbb{S}^3)}{|\Gamma|}. $$
Since by construction $d\alpha_i = 2\alpha_j\wedge\alpha_k$, we find the associated $2$-form:
$$d\Big(\frac{\alpha_i}{r^2}\Big) = -2 \frac{rdr\wedge\alpha_i -r^2\alpha_j\wedge\alpha_k}{r^4} = -2 \frac{\rho(\omega_i^+(\mathbf{e}))}{r^4},$$
where $\rho(\omega_i^+(\mathbf{e}))=r dr\wedge \alpha_i - r^2\alpha_j\wedge\alpha_k$

By \eqref{ipp Omegaj omegai}, we therefore find 
$$\eta_j = -\frac{|\Gamma|}{\vol(\mathbb{S}^3)}\sum_i\zeta_j^i \frac{\alpha_i}{r^2} + O(r^{-4}) = -\frac{|\Gamma|}{\vol(\mathbb{S}^3)}\frac{\zeta_j^+(\partial_r,.)}{r^3} + O(r^{-4}).$$
The $r^{-4}$ terms of $\Omega_j^-$ are consequently exactly:
$$ \Omega_j^- = d\eta_j = 2\frac{|\Gamma|}{\vol(\mathbb{S}^3)}\frac{\rho(\zeta_j^+)}{r^4} + O(r^{-5}) \text{ in a neighborhood of infinity}. $$
The $r^{-4}$ terms of the infinitesimal deformation $\mathbf{o}_{ij}=\Omega_j^-\circ \omega_i^+$ are:
\begin{equation}\label{dec-oij}
\mathbf{o}_{ij} = 2\frac{|\Gamma|}{\vol(\mathbb{S}^3)}\frac{\rho(\zeta_j^+)\circ \omega_i^+(\mathbf{e})}{r^4} + O(r^{-5}).
\end{equation}
This is reminiscent of \eqref{dvp obst avec zeta} in the case of an Eguchi-Hanson metric.

\subsubsection{From obstructions to linearized Ricci flows}\label{sec-lin-stab-orb}

\textbf{Extending quadratic terms.}

Let $H_2$ be a quadratic $2$-tensor on $(\mathbb{R}^4/\Gamma,\mathbf{e})$ in Bianchi gauge with respect to the flat metric $\mathbf{e}$ and denote $\mathbf{R}^+_p=\mathbf{R}_\mathbf{e}^{+,1}(H_2)$ the linearization of the selfdual curvature at the Euclidean metric in the direction $H_2$. 

We finally reach an extension of Proposition \ref{prop-h2-H2} to any hyperkähler ALE metric.

\begin{prop}\label{prop-h2-H2-kro}
    Up to the addition of an element of $\mathbf{O}(g_\zeta)$, there exists a unique solution $h_{2,\zeta}$ to: 
    \begin{equation*}
    \begin{split}
        d_{g_\zeta}\Ric (h_{2,\zeta})- \Lambda g_\zeta  &\in \mathbf{O}(g_\zeta),\\
        B_{g_\zeta}(h_{2,\zeta}) &= 0,      \\
        h_{2,\zeta} - \chi H_2 &  \in C^{2,\alpha}_{\beta-2},\quad\text{ for all $\beta>0$,}
        \end{split}
    \end{equation*}
    where $\chi$ is a cut-off function supported in a neighborhood of infinity, $C^{2,\alpha}_{\beta-2}$ is defined in Example \ref{ex: holder a poids ale}, and where one has more precisely 
    \begin{equation}\label{hello-beauty-2}
    d_{g_\zeta}\Ric (h_{2,\zeta})- \Lambda g_\zeta = -d_\zeta g(\mathbf{R}_p^+(\zeta)) \in \mathbf{O}(g_\zeta).
    \end{equation}
    where $\mathbf{R}^+(\zeta_1,...,\zeta_k) = (\mathbf{R}^+(\zeta_1), ... ,\mathbf{R}^+(\zeta_k))$.
    
    In particular, one has $|\nabla^{g _{\zeta},\,k}h_{2,\zeta}|_{g _{\zeta}}=O(r_o^{2-k})$ for all $k\geq 0$.


\end{prop}

\begin{rk}
    Exactly like for the Eguchi-Hanson metric, the evolution along Ricci flow is approximately: 
    $$\dot{\zeta} = 2\mathbf{R}^+(\zeta).$$
\end{rk}


\begin{proof}
 Decompose the traceless part $\mathring{h}_2$ of $h_2$ as: $\mathring{h}_2 = \sum_i \phi_i^-\circ\omega_i^+$.
 
    Slightly extending the computations of \cite{Biq-2}, we have thanks to \cite[Formula (4.6)]{Biq-Rol-ACH} the following values for the $L^2$-products:
\begin{equation}\label{comp-int-dric}
\begin{split}
    \int_{N} \Big\langle d_{g_\zeta}\Ric (h_2), \mathbf{o}_{ij} \Big\rangle_{g_\zeta} \,d\mu_{g_{\zeta}}&=-|\omega_i^+|^2_{g_\zeta}\int_{N} d_-d_-^* \phi_i^- \wedge \Omega_j^-\\
    &=-2\int_{N} d d^* \phi_i^- \wedge \Omega_j^-\\
    &=-2\lim_{s\to \infty}\int_{\{r=s\}} d^* \phi_i^- \wedge \Omega_j^-\\
    &=-2\lim_{s\to \infty}\int_{\{r=s\}} dd^* \phi_i^- \wedge \eta_j^-\\ 
    &=-2\lim_{s\to \infty}\int_{\{r=s\}}\mathbf{R}^+(\omega_i^+) \wedge \eta_j^-\\
    &=-\frac{2|\Gamma|}{ \vol(\mathbb{S}^3)}\sum_k\zeta_j^k \lim_{s\to \infty}\int_{\{r=s\}} \mathbf{R}^+(\omega_i^+) \wedge \frac{\alpha_k}{r^2}\\
    &=-\left\langle\mathbf{R}^+(\omega_i^+), \zeta_j^+\right\rangle=-\left\langle\mathbf{R}^+(\zeta_j^+), \omega_i^+\right\rangle.
    \end{split}
\end{equation}

We deduce from \eqref{eq L2 proj} that one has the following $L^2(g_\zeta)$-projection of $d_{g_\zeta}\Ric (\chi H_2)$ on $\mathbf{O}(g_{\zeta})$:
\begin{equation}\label{proj h2 kro}
\begin{split}
    \pi_{\mathbf{O}(g_{\zeta})}\left(d_{g_\zeta}\Ric (h_2)\right) &= \frac{1}{2}\sum_{ijk} b_{jk}\left\langle \mathbf{R}^+(\zeta_k^+), \omega_i^+ \right\rangle \mathbf{o}_{ij} \\
    &= \frac{1}{2}\sum_{ijk} b_{jk}\left\langle \mathbf{R}^+(\zeta_k^+), \omega_i^+ \right\rangle \Omega_j^-\circ \omega_i^+.
    \end{split}
\end{equation}


Let us finally compute the variation of $\zeta$ associated to the infinitesimal variation \eqref{proj h2 kro}. For this, recall that:
$$\zeta_k^l = \int_{\Sigma_k}\omega_l^+.$$
From \cite[Section 2]{Biq-2} or \cite{no1} for instance, the variation of $\omega_l^+$ associated to the infinitesimal variation $\Omega_j^-\circ \omega_i^+$ is $-\delta_{il}\Omega_j^-$. We deduce that the variation of $\zeta_k^l$ in the direction $\Omega_j^-\circ \omega_i^+$ is:
$$-\delta_{il}\int_{\Sigma_k}\Omega_j^- = -\delta_{il} a_{jk}. $$

The deformation $ \pi_{\mathbf{O}(g_{\zeta})}\left(d_{g_\zeta}\Ric (h_2)\right) =\frac{1}{2}\sum_{ijk} b_{jk}\langle \mathbf{R}^+(\zeta_k^+), \omega_i^+ \rangle \Omega_j^-\circ \omega_i^+ $ therefore corresponds to the following variation of $\zeta_m^n$:
\begin{align*}
    -\frac{1}{2}\sum_{ijk} b_{jk}\langle \mathbf{R}^+(\zeta_k^+), \omega_i^+ \rangle \delta_{in} a_{jm}&= -\frac{1}{2}\sum_{ik} \delta_{km}\delta_{in}\langle \mathbf{R}^+(\zeta_k^+), \omega_i^+ \rangle \\
    &= -\frac{1}{2}\langle \mathbf{R}^+(\zeta_m^+), \omega_n^+ \rangle
\end{align*}

Putting these together proves the result since $|\omega_i^+|^2 = 2$.
\end{proof}


\begin{rk}
    Another way to reach the same conclusion as that of the computation \eqref{comp-int-dric} in the proof of Proposition \ref{prop-h2-H2-kro}
consists in focusing on what happens at the $2$-spheres $\Sigma_j$ rather than at infinity. Recalling that $\Omega_j^-=(\gamma_j+d\eta_j)$, we compute on the basis of \cite[Formula (4.6)]{Biq-Rol-ACH},
    \begin{align*}
    \int_{N} \Big\langle d_{g_\zeta}\Ric (h_2), \mathbf{o}_{ij} \Big\rangle_{g_\zeta} \,d\mu_{g_{\zeta}} &=-2\int_{N} d_- d^* \phi_i^- \wedge (\gamma_j+d\eta_j)=-2\int_{N} d_- d^* \phi_i^- \wedge \gamma_j
\end{align*}
by integration by parts since $d_- d_-^* \phi_i^-\in \mathcal{H}^2_{L^2}(g_{\zeta})$ is a decaying closed $2$-form. Now, since $\gamma_j$ is compactly supported and closed, one has $ \int_N d_-d^*\phi_i\wedge \gamma_j+\int_N d_+d^*\phi_i\wedge \gamma_j=\int_N dd^*\phi_i\wedge \gamma_j = \int_N d(d^*\phi_i\wedge \gamma_j)=0 $. Consequently, we compute
\begin{align*}
    \int_{N} \Big\langle d_{g_\zeta}\Ric (h_2), \mathbf{o}_{ij} \Big\rangle_{g_\zeta} \,d\mu_{g_{\zeta}}
    &=2\int_{N} d_+ d^* \phi_i^- \wedge \gamma_j
    = -2\int_{N} \mathbf{R}^+(\omega_i^+) \wedge \gamma_j
    = -\left\langle\mathbf{R}^+(\omega_i^+),\zeta_j^+\right\rangle.
\end{align*}
\end{rk}

\begin{rk}
    All of the computations of obstructions to gluing hyperkähler ALE metrics to Einstein orbifolds extend to the case of Ricci soliton orbifolds.
\end{rk}

\textbf{Extending cubic terms.}

Now, for some other groups than $\mathbb{Z}_2$, instead of a development $g_o=\mathbf{e}+H_2+O(r^4)$, one has coordinates in which
\begin{equation}
    g_o=\mathbf{e}+H_2+H_3+O(r_o^4),
\end{equation}
for $H_3$ some cubic harmonic $2$-tensor satisfying $d_{\mathbf{e}}\Ric(H_3) = 0$, $B_{\mathbf{e}}H_3 = 0$. This damages the estimates of the first sections of the article. Fortunately, we may also extend the cubic terms $H_3$ to the orbifold like it was done in Proposition \ref{prop-h2-H2-kro}.
﻿
\begin{prop}\label{prop-h3-H3-kro}
    Up to the addition of an element of $\mathbf{O}(g_\zeta)$, there exists a unique solution $h_{3,\zeta}$ to: 
    \begin{equation*}
    \begin{split}
        d_{g_\zeta}\Ric (h_{3,\zeta}) &\in \mathbf{O}(g_\zeta),\\
        B_{g_\zeta}(h_{3,\zeta}) &= 0,      \\
        h_{3,\zeta} - \chi H_3 &  \in C^{2,\alpha}_{\beta-1},\quad\text{ for all $\beta>0$,}
        \end{split}
    \end{equation*}
    where $\chi$ is a cut-off function supported in a neighborhood of infinity, $C^{2,\alpha}_{\beta-1}$ is defined in Example \ref{ex: holder a poids ale}, and where one has more precisely 
    \begin{equation}\label{hello-beauty-3}
    d_{g_\zeta}\Ric (h_{3,\zeta}) = O(\varepsilon^5\rho^{-4})\in \mathbf{O}(g_\zeta).
    \end{equation}
    
    In particular, one has $|\nabla^{g _{\zeta},\,k}h_{3,\zeta}|_{g _{\zeta}}=O(\rho^{3-k})$ for all $k\geq 0$.
\end{prop}
\begin{proof}
    Consider $\chi$ a cut-off function supported at infinity of $(N,g_\zeta)$. One automatically gets $d_{g_\zeta}\Ric(\chi H_3) = O(r_b^{-3})$ on $N$ since ${g_\zeta}-\mathbf{e}=O(r_b^{-4})$. From \cite[Lemma 4.3]{ozu2}, it is always possible to find $h'$ such that $d_{g_\zeta}\Ric(h')+d_{g_\zeta}\Ric(\chi H_3)\in \mathbf{O}(g_\zeta)$. 
﻿
    Next, we estimate the obstruction by integration by parts of $\Delta_{L,g_{\zeta}}(\chi H_3)$ against any element $\mathbf{o}_{ij} = O^4+O^5+O(\varepsilon^6r^{-6})$, with $O^l = O(\varepsilon^lr^{-l})$. The first boundary term is proportional to $ \lim_{r\to\infty} r \int_{\mathbb{S}^3/\Gamma}\langle H_3, O^4 \rangle_{g_{\mathbb{S}^3/\Gamma}}dv_{g_{\mathbb{S}^3/\Gamma}}$, but the coefficients of $H_3$ are harmonic in $r^3$ while the coefficients of $O^4$ are harmonic in $r^{-4}$, hence their restrictions on $\mathbb{S}^3/\Gamma$ belong to different eigenspaces of the spherical Laplacian and therefore are orthogonal:
    $$\int_{\mathbb{S}^3/\Gamma}\langle H_3, O^4 \rangle_{g_{\mathbb{S}^3/\Gamma}}dv_{g_{\mathbb{S}^3/\Gamma}}=0.$$
    The next term then comes from the $r^{-5}$ term of $\mathbf{o}_{ij}$, up to a multiplicative constant independent on $\zeta$, this is 
    $$\int_{\mathbb{S}^3/\Gamma}\langle H_3, O^5\rangle_{g_{\mathbb{S}^3/\Gamma}}dv_{g_{\mathbb{S}^3/\Gamma}}=O(\varepsilon^5).$$
\end{proof}
﻿
﻿
\subsection{Adapting the proofs of the main theorems to the case of general hyperkähler ALE}

The proofs of the main Theorems \ref{mainthm1} and \ref{mainthm2} have only been given for the case of Eguchi-Hanson in the first 9 sections of the article. We now explain how the result of the present section together with minor modifications yield the same result for any other hyperkähler ALE spaces.

The main difficulty was to extend the obstruction computations \ref{prop-h2-H2} to the case of the hyperkähler ALE metrics. This was obtained in \ref{prop-h2-H2-kro}. The rest of the modifications are minor ones, and in fact most of the proofs are written with a generic metric $g_\zeta$ instead of $\mathbf{eh}_\zeta$ for this reason.
\begin{itemize}
	\item The estimates of the orbifold metric might be slightly worse if the singularity is $\mathbb{R}^4/\Gamma$, it can be as bad as $O(r^3)$ compared to the $O(r^4)$ when $\Gamma=\mathbb{Z}_2$, see Remark \ref{rk-kron-sec}. It requires extending the cubic terms of the orbifold metric to the ALE metric using Proposition \ref{prop-h3-H3-kro}. Since this is a lower order alteration of the metric in the ALE region alone, and it only involves lower order obstructions, this does not affect the analysis of any section.
	\item The most natural basis of $\ker_{L^2}\Delta_{L,g_\zeta}$ is not orthonormal, and has to be orthonormalized in order to run the arguments using the Gram matrix, see Remarks \ref{rk BON oij a ref} and \ref{rk hk ALE sec 9}. This can of course always be done by Gram-Schmidt procedure.
	\item Another potential difficulty is that the metrics obtained for $\zeta\in\mathbb{R}^{3k}\backslash \{0\}$ through Kronheimer's period map are not all \textit{manifolds}: a forbidden set $\mathcal{D}$ has to be avoided for the estimates to hold, see Remark \ref{rk: zeta close to initial direction}. This is fortunately the case in our present applications, see Remark \ref{rk ref zeta controlled Kro}.
\end{itemize}

\subsection{Linear (in)stability with respect to other Ricci-flat ALE spaces}\label{sec:stab general RFALE}

A major open question is that of the classification of Ricci-flat ALE metrics. The only known examples are obtained from hyperkähler metrics and are all treated along the present article, but there might exist other Ricci-flat ALE metrics. 

In this section, we discuss our construction and stability condition  applied to general Ricci-flat ALE metrics: spherical orbifolds are still orbifold point \textit{unstable}, and hyperbolic orbifolds are still orbifold point \textit{stable}. This, of course, recovers the case of hyperkähler ALE metrics.

\begin{note}
    In this section, we do not attempt to compute the exact constants in our evolutions. We are simply interested in the signs of the derivatives of meaningful quantities.
\end{note}

\subsubsection{General properties of Ricci-flat ALE metrics}

We will assume that we are given a Ricci-flat ALE metric $(N,g_b)$ asymptotic to $\mathbb{R}^4/\Gamma$ for some finite group $\Gamma\subset \operatorname{O}(4)$ acting freely on $\RR^4\setminus\{0\}$, we do not assume anything else in this section only.\\

\textbf{Infinitesimal Ricci-flat ALE deformations.}

Let us denote $\mathbf{O}(g_b)$ the space of traceless and divergence-free infinitesimal Ricci-flat deformations decaying at infinity, they automatically decay in $r^{-4}$ at infinity.

In this space $\mathbf{O}(g_b)$, there is a special space consisting of deformations obtained by homotheties of the whole Ricci-flat ALE metric. They are linear combinations of tensors of the form:
\begin{enumerate}
    \item $\mathcal{L}_{Y}g_b$, for $Y$ a harmonic but not Killing vector field asymptotic to a Killing vector field of $\mathbb{R}^4/\Gamma$ or,
    \item $\mathbf{o}_1:=(\Hess_{g_b}u)^\circ$ for $u$ the unique solution of $\Delta_{g_b} u = 4$, $u = \frac{r_b^2}{2} + O(r_b^{-2})$ at infinity. Note that this infinitesimal deformation is \textit{nontrivial} on any non-flat Ricci-flat ALE metric. It also moreover always decays exactly at the slowest $r^{-4}$ rate.
\end{enumerate}
Note that the Killing vector fields of $\mathbb{R}^4/\Gamma$ are  $\Gamma$-invariant linear combinations of the Killing vector fields $\omega_i^\pm(\mathbf{e})(r\partial_r)$ fixing zero on $\mathbb{R}^4$. 

In this section, we focus on the effect of the deformation $\mathbf{o}_1$ compared to the rest, and consider the splitting:
\begin{equation}
    \mathbf{O}(g_b) =: \mathbb{R}\mathbf{o}_1\stackrel{\perp}{\oplus}\mathbf{O}_0(g_b) \text{ in $L^2(g_b)$}.
\end{equation}

We will justify the interpretation of $\mathbf{o}_1$ as a ``rescaling'' deformation below.\\

\textbf{Gauges at infinity and renormalized volume.}

In \cite{Biq-Hein}, a specific gauge is introduced for Ricci-flat ALE metrics. It is based on a Constant Mean Curvature  foliation of infinity, from which Biquard and Hein obtain specific ALE coordinates as in Definition \ref{def orb ale}, $\Phi_0: N\backslash K\to \mathbb{R}^4/\Gamma\backslash B_\mathbf{e}(R)/\Gamma$ in which one has:
\begin{equation}\label{eq: expansion ALE}
    \Phi_0^*g_b = \mathbf{e} + H^4 + O(r^{-5}),
\end{equation}
and $H^4$ is a harmonic symmetric $2$-tensor satisfying $H^4(\partial_r,\cdot) = 0$, $\tr_\mathbf{e}H^4= 0$ and $\div_\mathbf{e}H^4 =0$. We will call any such ALE coordinates \textit{strongly radial}.

A measure of the ``size'' of the ALE metric, or of its distance from its flat asymptotic cone $\mathbb{R}^4/\Gamma$ is the \textit{renormalized volume} computed as follows: denote $U_r$ the interior of the hypersurface $\partial U_r = \Phi_0^{-1}(r\mathbb{S}^3/\Gamma)$,
\begin{equation}
    \mathcal{V}(g_b) := \lim_{r\to\infty} \left(\vol_{g_b} U_r - \vol_\mathbf{e} B_\mathbf{e}(r)/|\Gamma|\right).
\end{equation}
This quantity is nonpositive and vanishes \textit{only} if $g_b$ is flat. It is interpreted as the ``size'' of the Ricci-flat ALE metric $g_b$, it is a geometric quantity, independent on a choice of coordinates and scales like a volume, as follows: for $s>0$,
\begin{equation}\label{eq:rescaling renormalized volume}
	\mathcal{V}(s^2g_b) = s^4 \mathcal{V}(g_b).
\end{equation}

In order to simplify later proofs, let us first note that the renormalized volume can be computed in more general coordinates than the above strictly radial gauge, and that it does not require the manifold to be Ricci-flat, but simply to have its Ricci curvature decaying fast enough to yield an expansion as in \eqref{eq: expansion ALE}.
\begin{lemma}\label{lem: vol red not Ricci flat}
    Let $(N,g)$ be an ALE manifold, and assume that there exist ALE coordinates $\Phi_1$ so that
    \begin{equation}\label{eq: expansion ALE-2}
    \Phi_1^*g = \mathbf{e} + H^4 + O(r^{-5}),
\end{equation}
and $H^4$ is a harmonic symmetric $2$-tensor satisfying $H^4(\partial_r,\partial_r) = 0$ (only!), $\tr_\mathbf{e}H^4= 0$ and $\div_\mathbf{e}H^4 =0$. We call such coordinates \emph{weakly radial}.

Then, the renormalized volume of $g$ is also computed as follows from $ U^1_r$ the interior of $\Phi_1^{-1}(r\mathbb{S}^3/\Gamma)$,
$$\mathcal{V}(g) = \lim_{r\to\infty} \left(\vol_{g} U^1_r - \vol_\mathbf{e} B_\mathbf{e}(r)/|\Gamma|\right),$$
and $\mathcal{V}(g)$ is independent of the choice of diffeomorphism $\Phi_1$ with the above properties.
\end{lemma}
\begin{proof}
    With the exact same proof as in \cite{Biq-Hein}, even if $g$ is not Ricci-flat, we obtain the existence of a CMC foliation and that of a diffeomorphism $\Phi_0$ for which 
    \begin{equation}\label{eq: expansion ALE-3}
    \Phi_0^*g = \mathbf{e} + H^4 + O(r^{-5}),
\end{equation}
and $H^4$ is a harmonic symmetric $2$-tensor satisfying $H^4(\partial_r,\cdot) = 0$, $\tr_\mathbf{e}H^4= 0$ and $\div_\mathbf{e}H^4 =0$. The existence of the limit defining $\mathcal{V}(g)$ also follows from \cite{Biq-Hein}.

In \cite{Biq-Hein}, the renormalized volume is not computed from a CMC foliation, this is simply a way to make the value canonical. Instead, its value is computed in some strictly radial system of coordinates. All such coordinate systems lead to hypersurfaces $\Phi_0^{-1}(r\mathbb{S}^3/\Gamma)$ which are close to the actual CMC foliation at a fast enough rate $r^{-5}$, see \cite{Biq-Hein}. This ensures that the value of the renormalized volume is independent of the \textit{strongly radial} system of coordinates chosen.

What remains to be proved is that if one instead has ALE coordinates from $\Phi_1$ with only $H^4(\partial_r,\partial_r) = 0$ instead of $H^4(\partial_r,\cdot) = 0$, then the value of $\mathcal{V}(g)$ is unchanged. The idea is that from \cite[(2.21)]{Biq-Hein} one may go from $\Phi_1$ to some strongly radial system of coordinates $\Phi_0$ through a diffeomorphism corresponding to following the flow of $X_2:=\Phi_1^{-1}(\frac{L_2x}{r^4})$ for $L_2$ a traceless symmetric matrix. Let us denote $(\Phi_t)_{t\in(0,1)}$ the associated family of diffeomorphisms, $\Sigma^t_r:= \Phi_t^{-1}(r)$, and $ U^t_r\subset N$ the interior of $\Sigma^t_r$. 

Then, the variation evolution of the volume of $ U^t_r$ is computed as follows:
$$\frac{d}{dt} \vol_{g_b} U^t_r = \frac{1}{r^4}\int_{\Sigma^t_r}  \langle X_2 , \nu_{\Sigma^t_r}\rangle_{g_b} \, d\sigma_{\Sigma^t_r}=\int_{\mathbb{S}^3/\Gamma}\langle L_2x , x \rangle_{g_b}  \,d\sigma_{\mathbb{S}^3/\Gamma} + \gamma(r,t), $$
where $\nu_{\Sigma^t_r} = \frac{x}{|x|} + O(r^{-4})$ is the exterior $g_b$-normal to $\Sigma^t_r$, and where $\sup_{t\in[0,1]}\gamma(r,t)\to 0$ as $r\to +\infty$. Now, since $L_2$ is traceless, the leading term $\int_{\mathbb{S}^3/\Gamma}\langle L_2x , x \rangle_{g_b}  \,d\sigma_{\mathbb{S}^3/\Gamma}$ vanishes, and the value of the limit 
$$\lim_{r\to +\infty}\vol_{g_b} U^t_r -\vol_\mathbf{e} B_\mathbf{e}(r)/|\Gamma|$$
is independent of $t$.
\end{proof}

In such weakly radial coordinates produced by Lemma \ref{lem: vol red not Ricci flat}, one may also intriguingly compute the renormalized volume from the expansion of $ u $, the unique solution to $\Delta_{g_b} u =4$, $u = \frac{r^2}{2} + o(1)$. In \cite{Biq-Hein}, the following result is only stated for Ricci-flat ALE metrics, but extends verbatim to the setting of Lemma \ref{lem: vol red not Ricci flat}:

\begin{prop}[{\cite{Biq-Hein}}]\label{prop: exp BH}
    Under the assumptions, and in the gauge of Lemma \ref{lem: vol red not Ricci flat}, the above function $u$ has the following development:
    \begin{equation}
        \Phi_1^*u = \frac{r^2}{2} -\frac{2|\Gamma|}{|\mathbb{S}^3|} \frac{\mathcal{V}(g)}{r^2}+ o(r^{-2}).
    \end{equation}
\end{prop}
In particular, the value of the $r^{-2}$-term in the expansion of $u$ determines completely the renormalized volume in the coordinates given by $\Phi_1$.

\subsubsection{Variations of the renormalized volume}

Let us explain why the deformation $\mathbf{o}_1$ is interpreted as a ``rescaling'' and detects the variations of the renormalized volume.
\\

Let $\mathbf{o}\in \mathbf{O}(g_b)$ be an infinitesimal Ricci-flat ALE deformation of $g_b$. We will compute the renormalized volume of $g_t:= g_b+t\mathbf{o}$ at first order, and we will denote with a ``dot'' the derivatives with respect to $t$ at $ t=0$ of various tensors. Our goal is to prove the following statement.
\begin{prop}
     If $g_t = g_b+t\mathbf{o}_1$, then one has:
     \begin{equation}\label{eq: volume var in O1}
        \frac{d}{dt} \mathcal{V}(g_t) = 2 \mathcal{V}(g_b) \text{ at $t=0$}.
    \end{equation}
    
    On the other hand, let $\mathbf{o}_0\in\mathbf{O}_0(g_b)$ and if $g_t = g_b+t\mathbf{o}_0$, then, one has
    \begin{equation}\label{eq: volume cst in O0}
        \frac{d}{dt} \mathcal{V}(g_t) = 0 \text{ at $t=0$}.
    \end{equation}
     
\end{prop}
In other words, $\mathbf{o}_1$ is the only element of $\mathbf{O}(g_b)$ that can vary the renormalized volume of a Ricci-flat ALE metric.

Let us first prove the simpler \eqref{eq: volume var in O1}. The deformation $\mathbf{o}_1 = g_b-\frac{1}{2}\mathcal{L}_{\nabla^{g_b}u_b}g_b$ infinitesimally corresponds to the commuting composition of the action by the diffeomorphism generated by the vector field $-\frac{1}{2}\nabla^{g_b}u_b$ and the rescaling $(1+t)g_b$. The diffeomorphism does not affect the value of the renormalized volume while the rescaling $(1+t)g_b$ modifies the value of the reduced volume according to the rescaling property \eqref{eq:rescaling renormalized volume}, hence the formula:
$$\frac{d}{dt} \mathcal{V}(g_t) = \frac{d}{dt} \mathcal{V}((1+t)g_b) = 2 \mathcal{V}(g_b).$$

We now deal with the more subtle case of \eqref{eq: volume cst in O0} through a series of partial results.

\begin{lemma}\label{lem:expansion variation u}
    Denote $u_t$ the unique solution of $\Delta_{g_t} u_t = 4$, $u_t = \frac{r_b^2}{2} + o(1)$ in the same ALE coordinates as those of $g_b$. Then we find the following development:
    $$\dot{u} = \frac{a}{r^2} + o(r^{-2}),$$
    where 
    $$a= -\frac{|\Gamma|}{2|\mathbb{S}^3|}\int_N \langle \mathbf{o},\mathbf{o}_1\rangle_{g_b} \,d\mu_{g_b}. $$
\end{lemma}
\begin{proof}
    Let us take a look at the equation satisfied by $\dot{u}$. Using the fact that $\mathbf{o}$ is traceless and divergence-free, the (classical) formula for the variation of a Laplacian helps us find that:
    \begin{equation}\label{eq: eq variation u}
        \Delta_{g_b}\Dot{u} = \langle \mathbf{o},\mathbf{o}_1 \rangle_{g_b},
    \end{equation}
    and there is a unique such solution decaying like $r^{-2+\varepsilon}$ at infinity for $\varepsilon>0$ by standard elliptic theory in weighted Hölder spaces. It actually decays like $r^{-2}$: its leading term is a multiple of the Green function $r^{-2}$ of $\mathbb{R}^4/\Gamma$ because of the decay of the righthand-side. Let us compute this leading coefficient $a\in\mathbb{R}$ of $\dot{u} = \frac{a}{r^2} + O(r^{-2.1})$.
    
    In particular, integrating the equation \eqref{eq: eq variation u} by parts over the whole manifold we find:
    \begin{align*}
        -2a|\mathbb{S}^3/\Gamma| = \lim_{s\to +\infty}\int_{\{r_b = s\}} \partial_{r_b}\dot{u} \,d\sigma_{S_{g_b}(p,s)}= \int_N \langle \mathbf{o},\mathbf{o}_1 \rangle_{g_b} \,d\mu_{g_b}.
    \end{align*}
\end{proof}

    In particular, we notice that if $\mathbf{o}$ is $L^2(g_b)$-orthogonal to $\mathbf{o}_1$, then $a = 0$. This is however not enough to ensure that the first variation of the renormalized volume vanishes from Proposition \ref{prop: exp BH}: we still need to ensure that the ALE coordinates given by $\Phi$ satisfy the assumptions of Lemma \ref{lem: vol red not Ricci flat} for the metrics $g_t$, in other words, we need to verify that $\mathbf{o}(\partial_r,\partial_r) = 0$ for any $\mathbf{o}\in\mathbf{O}_0(g_b)$. 

    \begin{lemma}
        Let $\mathbf{o}\in\mathbf{O}_0(g_b)$, then $\mathbf{o}(\partial_r,\partial_r) = o(r^{-4})$. On the other hand, one has
	\begin{equation}\label{eq: radial term o1}
	\lim_{r\to\infty} r^4\mathbf{o}_1(\partial_r,\partial_r) = \frac{|\Gamma|}{|\mathbb{S}^3|} \|\mathbf{o}_1\|_{L^2(g_b)}^2.
\end{equation}	        
    \end{lemma}
    \begin{proof}
        According to the decomposition of harmonic $2$-tensors that are $O(r^{-4})$ at infinity of \cite[Proposition 2.6]{Biq-Hein}, we find that for a constant $a'\in\mathbb{R}$,
        $\mathbf{o}(\partial_r,\partial_r) = \frac{a'}{r^4}+o(r^{-4})$, and our goal is to show that $a'=0$.

        Let us write $\mathbf{o}_1 := (\mathcal{L}_Xg_b)^\circ$ for $X =\frac{1}{2} \nabla^{g_b} u$. Applying the formula \cite[(21)]{ozu4} to $T = \mathbf{o}$, and $X$ as above, one obtains
        \begin{equation}\label{eq:ipp Lie}
            \int_{\partial U_r} \mathbf{o}(X,\nu_r) = \frac{1}{2}\int_{ U_r} \langle\mathbf{o},\mathbf{o}_1\rangle_{g_b}\,d\mu_{g_b}.
        \end{equation}

        Now, taking a limit of \eqref{eq:ipp Lie} as $r\to +\infty$, one finds
        \begin{equation}
            \frac{|\mathbb{S}^3|}{|\Gamma|} a' = \int_N\langle\mathbf{o},\mathbf{o}_1\rangle_{g_b}\,d\mu_{g_b},
        \end{equation}
        and consequently, if $\mathbf{o}\in\mathbf{O}_0(g_b)$, $a' = 0$ as announced, and we similarly obtain the formula for $\mathbf{o}_1$.
    \end{proof}

    This concludes the proof of \eqref{eq: volume cst in O0}. \\

\textbf{Moduli space of Ricci-flat modulo obstructions ALE metrics}
    
	As constructed in \cite[Appendix B]{ozu4}, close to any Ricci-flat ALE metric, there is a real-analytic set of metrics $g_{b,v}$ parametrized by $v$ in the unit ball of the vector space $\mathbf{O}_0(g_b)$ endowed with the $L^2(g_b)$ norm. 

\begin{defn}	\label{defn: RFALE mod obst}
	The metrics $g_{b,v}$ are uniquely determined as satisfying:
	\begin{enumerate}
		\item $\Ric(g_{b,v}) \in \pi_{\ker B_{g_{b,v}}}\mathbf{O}(g_b)$, and $\|\Ric(g_{b,v})\|_{L^2(g_b)}\leq \|v\|_{L^2(g_b)}^2$, where $\pi_{\ker B_{g_{b,v}}}$ is the projection on the kernel of $B_{g_{b,v}}$ as in \cite[(45)]{Biq-1},
		\item $B_{g_b}g_{b,v}=0$,
            \item $\mathcal{V}(g_{b,v}) = \mathcal{V}(g_{b})$, and $g_{b,v}-g_{b}\perp_{L^2(g_b)} \mathbf{O}_0(g_b)$, and
		\item $|\nabla^{g_b,\,k}(g_{b,v}-g_b)|_{g_b}\leq C_k \|v\|_{L^2(g_b)}^2 (1+r)^{-4-k}$ for $C_k>0$.
	\end{enumerate}
    There are also deformations of $g_b$ corresponding to the infinitesimal deformation $\mathbf{o}_1$: define 
    \begin{enumerate}
        \item $ \partial_s g_{b}^s = g_{b}^s-\frac{1}{2}\mathcal{L}_{\nabla^{g_b^s}u^s}g_{b}^s =: \mathbf{o}_1^s$, 
        \item where $u^s$ is the only solution of $\Delta_{g_b^s}u^s = 4$ with $u^s = \frac{r^2}{2}+o(1)$ which is also characterized by $\Delta_{g_{b}^s}\partial_s u^s = |g_{b}^s-\frac{1}{2}\mathcal{L}_{\nabla^{g_b^s}u^s}g_{b}^s|^2_{g_b^s}$ from Lemma \ref{lem:expansion variation u}.
    \end{enumerate}
    The metrics $g_b^s$ are all homothetic since $\partial_s (e^{-s}g_{b}^s) =- \frac{1}{2}\mathcal{L}_{\nabla^{g_b^s}u^s}(e^{-s}g_{b}^s)$, hence there exists a diffeomorphism $\phi_s$ so that $g_b^s = e^s\phi_s^*g_b$. We may then define the metrics $g^s_{b,v} = e^s\phi_s^*g_{b,v}$.
\end{defn}
 
\subsubsection{Linear (in)stability of orbifolds with constant curvature}

Thanks to the previous section, we can interpret the obstructions against the infinitesimal deformation $\mathbf{o}_1$ as the way to measure how the scale of a Ricci-flat ALE metric evolves along Ricci flow. \textbf{We will only treat the simplest case of an Einstein orbifold with vanishing Weyl curvature at its singular points (such as spherical and hyperbolic orbifolds).}

Let $g_o$ be an Einstein orbifold with $\Ric(g_o) = \Lambda g_o$ and vanishing Weyl curvature at a singular point $p$ where one has the development:
\begin{equation}\label{taylor-g0-devt-2}
g_o = \mathbf{e} + H_2+O_{\mathbf{e}}(r_o^4),
\end{equation}
where $H_2$ is a quadratic $2$-tensor in Bianchi gauge with respect to the flat metric $\mathbf{e}$, and $r_o$ is the distance to zero for $\mathbf{e}$. 
The following result echoes Propositions  \ref{prop-h2-H2} and \ref{prop-h2-H2-kro}.

\begin{prop}\label{prop: obst general ALE}
 Up to the addition of an element of $\mathbf{O}(g_b)$, there exists a unique solution $h_{2,\zeta}$ to: 
    \begin{equation*}
    \begin{split}
        d_{g_b}\Ric (h_{2,\zeta})- \Lambda g_b  &\in \mathbf{O}(g_b),\\
        B_{g_b}(h_{2,\zeta}) &= 0      \\
        h_{2,\zeta} - \chi H_2 &  \in C^{2,\alpha}_{\beta-2},\quad\text{ for all $\beta>0$,}
        \end{split}
    \end{equation*}
    where $\chi$ is a cut-off function supported in a neighborhood of infinity, $C^{2,\alpha}_{\beta-2}$ is defined in Example \ref{ex: holder a poids ale}, and where one has more precisely for $c = c(\Gamma)>0$,
    \begin{equation}\label{hello-beautyy}
    d_{g_b}\Ric (h_{2})-\Lambda g_b = -c\Lambda\mathbf{o}_1.
    \end{equation}


\end{prop}
\begin{proof}
	Let us first prove that the obstruction vanishes against an element $\mathbf{o}_0\in\mathbf{O}_0(g_b)$. For this, we use the development of a spherical or hyperbolic orbifold satisfying in Bianchi gauge with respect to the flat cone:
    \begin{equation}
        g_o = \mathbf{e} + H_2 +O(r_o^4),
    \end{equation}
    where $H_2$ is a multiple of $r^2(3dr^2+2r^2g_{\mathbb{S}^3})$ if $\mathbf{e} = dr^2+r^2g_{\mathbb{S}^3}$.

    The $2$-tensor $\mathbf{o}_0$ admits an expansion at infinity: for $O^4_0$, the homogeneous harmonic leading term of $\mathbf{o}_0$ in $r^{-4}$,
    \begin{equation}
        \mathbf{o}_0 = O^4_0+ O(r^{-5}).
    \end{equation}
    The obstruction to solving the equation for $h_2$ against $\mathbf{o}_0$ from \cite[Proposition 5.3, Remarque 5.5]{ozuthese} is then a multiple of
    \begin{equation}
        \int_{\mathbb{S}^3/\Gamma} \langle O^4_0, H_2 \rangle_{\mathbf{e}} \,d\sigma_{\mathbb{S}^3/\Gamma}.
    \end{equation}
    Since $O^4_0$ is $\mathbf{e}$-traceless and has a vanishing $(\partial_r,\partial_r)$ component, it is pointwise orthogonal to the $2$-tensor $H_2=r^2(3dr^2+2r^2g_{\mathbb{S}^3})$, and the above integral vanishes.
	\\
	
	As computed in \eqref{eq: radial term o1}, we have $O^4_1(\partial_r,\partial_r) = \frac{|\Gamma|}{|\mathbb{S}^3|} \frac{\|\mathbf{o}_1\|_{L^2(g_b)}^2}{r^4}$ for $\mathbf{o}_1 = O^4_1+o(r^{-4})$ with leading homogeneous harmonic term $O^4_1$. Consequently, since $O^4_1$ is traceless, there is a positive numerical constant $c_2>0$ so that 
    \begin{equation}
        \int_{\mathbb{S}^3/\Gamma} \langle O^4_1, H_2 \rangle_{\mathbf{e}} \,d\sigma_{\mathbb{S}^3/\Gamma} = c_2\Lambda\frac{|\Gamma|}{|\mathbb{S}^3|} \|\mathbf{o}_1\|_{L^2(g_b)}^2.
    \end{equation}	
	
	Consequently, since the projection on $\mathbf{o}_1$ compensates the coefficient $\|\mathbf{o}_1\|_{L^2(g_b)}^2$ we can solve our equation with explicit obstruction: for $c = c(\Gamma)>0$,
	\begin{equation}\label{hello-beautyyy}
    d_{g_b}\Ric (h_{2}) - \Lambda g_b = -c \Lambda \mathbf{o}_1.
    \end{equation}
\end{proof}
\begin{rk}
    If the orbifold had nonvanishing Weyl curvature at its singular point, we would see additional obstructions couplings of the $\pm$-selfdual Weyl curvature of the orbifold at its singular point, and of the $\mp$-selfdual Weyl curvature of the ALE metric at infinity, see \cite{ozu4}.
\end{rk}

\textbf{An approximate Ricci flow.} 

Notice that $t\mapsto g^{2c\Lambda t}_{b,0}$ solves $\partial_t g^{2c\Lambda t}_{b,0} = 2c\Lambda\mathbf{o}_1$. From the equation \eqref{hello-beautyy} we deduce as in Lemma \ref{lemma-time-der-metric} that an approximate renormalized Ricci flow is obtained by gluing a metric $ g^{2c\Lambda t}_{b,0} $ to a spherical or hyperbolic orbifold $(M_o,g_o)$. 

In particular, the bubbles grow when glued to spherical orbifolds and shrink when glued to hyperbolic orbifolds, at an exponential rate. This is consistent with the case of hyperkähler Ricci-flat ALE metrics. Similar observations hold with orbifold solitons with vanishing Weyl curvature at their singular point, e.g. quotients of Bryant soliton or cylinders $\mathbb{R}\times \mathbb{S}^3$. 

\subsubsection{Constructing ancient and immortal Ricci flows from general Ricci-flat ALE metrics}

Let us list the properties that we used along the article in order to construct immortal and ancient Ricci flows from hyperkähler ALE bubbles which might not be satisfied by general Ricci-flat ALE metrics.
\begin{enumerate}
    \item \textit{Stability}: hyperkähler metrics satisfy Proposition \ref{linear-sta-instantons}.
    \item \textit{Explicit obstruction}: the obstructions can be computed as in Proposition \ref{prop-h2-H2-kro}.
    \item \textit{Integrability}: for any infinitesimal Ricci-flat ALE deformation $\mathbf{o}$, there is an associated \textit{curve} of Ricci-flat ALE metrics with initial speed $\mathbf{o}$.
    \item \textit{Canonical parametrization}: Kronheimer's period map gives a natural, analytic parametrization of the moduli space.
\end{enumerate}

We believe that none of these properties specific to the hyperkähler ALE metrics are essential to our construction, except that of \textit{stability}, which ensures that the Ricci-flat ALE metric satisfies an appropriate Liouville theorem. 
\begin{enumerate}
    \item \textit{Stability}: intuitively, Ricci-flat ALE metrics appearing ``generically'' along the Ricci flow should be \textit{stable}, so it is reasonable to assume this property. By Remark \ref{rk:ALE instable}, there is a good chance that it is not necessary.
    \item \textit{Explicit obstruction}: Proposition \ref{prop: obst general ALE} provides explicit obstructions when the orbifold has vanishing Weyl curvature at its singular point.
    \item \textit{Integrability}: even if the smooth moduli space of nearby Ricci-flat ALE metrics is not smooth, there is a smooth moduli space of Ricci-flat ALE modulo obstruction metrics $g^s_{b,v}$.
    \begin{enumerate}
        \item everything is solved modulo obstructions everywhere along the article until the compensation of obstructions through an unaffected fixed point setup, and
        \item the obstructions are of very low order with respect to the other infinitesimal deformations, as in the case of hyperkähler ALE metrics. Indeed, the obstruction of Proposition \ref{prop: obst general ALE} only involves $\mathbf{o}_1$, and the nonintegrability, only induces quadratic errors in the parameter $v$, which is already negligible on its own.
    \end{enumerate}
    \item \textit{Canonical parametrization}: this was not necessary at all, and merely links our evolution equation to parameters used in the literature.
\end{enumerate}

\newpage
\appendix

\section{Real-analytic variations of tensors solving equations}\label{sec:analytic}

In this section, we present a way to precisely analyze the dependence of the metrics and special tensors as the parameter $\zeta$ of Kronheimer's period map varies. This leads to analytic controls of these tensors seen as functions of $\zeta$ with specific decays.

\begin{rk}
	In the particular case of the explicit Eguchi-Hanson metric, all of the estimates below can be obtained through explicit but painful computations.
\end{rk}

The analytic dependence will be proved by implicit function theorems for analytic maps between Banach spaces.

\subsection{Real-analyticity of operations}

We prove that our maps are analytic to start with.

\subsubsection{Real-analytic Mappings Between Banach Spaces}
Let us recall some basic definitions and theorems about real-analytic mappings between Banach spaces and real-analytic subsets that will be useful to us.

Let \( a_k : V^k \to W \), \( k \in \mathbb{N} \) be continuous symmetric \( k \)-linear forms on \( V \) with values in \( W \). The \emph{power series} denoted \( x \in V \mapsto \sum_k a_k x^k \in W \) from a Banach space \( V \) to a Banach space \( W \) is defined as the sum of \( a_k(x, \ldots, x) \). We say that it \textit{converges} if the real sum \( \sum_k \|a_k(x, \ldots, x)\|_W \) converges. From this, we derive the definition of an real-analytic map.

\begin{defn}[Real-analytic maps]\label{def:real analytic}
    Let \( V \) and \( W \) be Banach spaces, and \( U \) an open set in \( V \). A map \( f: U \to W \) is said to be \emph{real-analytic} if for every point \( x \in U \), \( f \) is equal to a convergent power series in a neighborhood of \( x \).
\end{defn}

This definition is adapted to the usual framework of analysis, and in particular, the implicit function theorem is valid for this regularity.

\begin{lemma}[{\cite[Implicit Function Theorem, p. 1081]{Whittlesey1965}}]\label{th implicit functions}
    Let \( X, Y, \) and \( Z \) be Banach spaces, and \( A \) an open set in \( X \times Y \). Suppose a real-analytic map \( f : A \to Z \) satisfies for \( (x_0, y_0) \in A \),
    \begin{enumerate}
        \item \( f(x_0, y_0) = 0, \)
        \item \( d_{(x_0, y_0)}f(0,.) : Y \to Z \) is a homeomorphism.
    \end{enumerate}
    Then,
    \begin{enumerate}
        \item there exists \( N(x_0) \), an open neighborhood of \( x_0 \) in \( X \), and a unique continuous map \( g : N(x_0) \to Y \) such that \( g(x_0) = y_0 \), \( (x, g(x)) \in A \) for \( x \in N(x_0) \), and
        \( f(x, g(x)) = 0, \)
        \item \( g \) is analytic,
        \item \( d_xg = -\big(d_{(x, g(x))}f(0,.)\big)^{-1} \circ \big(d_{(x, g(x))} f (.,0)\big) \).
    \end{enumerate}
\end{lemma}

\subsubsection{Real-analytic operations between weighted Hölder spaces}

It will be crucial to note that operations between tensors having only a local \( H^s \) regularity for \( s \) large enough depending on the dimension or having a \( C^{k, \alpha} \) regularity for \( k \) larger than the number of derivations are in fact analytic. This may be obtained by the proof of \cite[Lemma 13.7]{Koiso1983}, but we give a simpler proof in the case of (weighted) Hölder spaces. The main tool is the theory developed in \cite{Palais1968} and the crucial property is that for these regularities, the set of functions equipped with pointwise multiplication forms a Banach algebra. 

This leads us to define classes of function spaces well adapted to our analysis. For a tensor $s$, a point $x$, $\alpha>0$ and a metric $g$, the pointwise Hölder semi-norm $ [s]_{C^\alpha(g)}(x)$  is defined as
$$ [s]_{C^\alpha(g)}(x) := \sup_{y\in B_g(x,\,\inj_g(x)/2)}  \frac{\left|s(x)-(P^g_{x,y})^{\ast}s(y)\right|_g}{d_g(x,y)^\alpha},$$
where $P_{x,y}^g$ denotes the parallel transport with respect to $g$ along the unique minimizing geodesic joining $x$ to $y$.
\begin{defn}[Weighted Hölder spaces and algebra]
    Let $w$ be a smooth function called a ``weight'', and $\rho$ be a positive function called a ``scale'' on a Riemannian manifold $(M,g)$, and let $0<\alpha<1$. For any tensor $s$ on $M$, we define 
    \begin{align*}
        \|s\|_{C^{k,\alpha}_{w,\rho}(g)} &:= \sup_{M}w\left(\sum_{i=0}^k \rho^{i}|\nabla^{g,\,i} s|_{g} + \rho^{k+\alpha}[\nabla^{g,\,k}s]_{C^\alpha(g)}\right).
    \end{align*}
    We will call such a function space an \emph{algebra} if $w\geq 1$.
\end{defn}

\begin{rk}
    Typically, $\rho$ is chosen as equivalent to the injectivity or conjugate radius, while $w$ is a function of $\rho$.
\end{rk}

\begin{exmp}\label{ex: holder a poids ale}
    Let $\beta\in \mathbb{R}$, $k\in\mathbb{N}$, $0<\alpha<1$ and let $(N,g_b)$ be an ALE orbifold. If we define for any tensor $s$ on $N$,
    \begin{align*}
        \| s \|_{C^{k,\alpha}_{\beta}(g_b)} &:= \sup_{N}r_b^{\beta}\left(\sum_{i=0}^k r_b^{i}|\nabla^{g_b,\,i}s|_{g_b} + r_b^{k+\alpha}[\nabla^{g_b,\,k}s]_{C^\alpha(g_b)}\right),
    \end{align*}
    then, this is a weighted Hölder algebra as long as $\beta\geq 0$. Here $r_b$ denotes some smoothing of the maximum of $1$ and the distance from an arbitrary point in $N$ defined with respect to $g_b$.
\end{exmp}

Let us now explain which operations are analytic in these weighted Hölder algebras. 

\begin{lemma}\label{lem norm multilinear}
    Let $ s_1,\dots,s_m $ be tensors on $M$, and associated norms ${C^{k_1,\alpha_1}_{w_1,\rho}(g)},\dots,{C^{k_m,\alpha_m}_{w_m,\rho}(g)}$. Then, we have:
    \begin{enumerate}
        \item $\|s_1\|_{C^{k_1,\alpha_1}_{w_1,\rho}(g)}\leq\|s_1\|_{C^{k_1',\alpha_1'}_{w_1',\rho}(g)}$ for $k_1\leq k'_1$, $\alpha_1\leq\alpha'_1$ and $w_1\leq w'_1$,
        \item $\|\nabla^{g,\,l}s_1\|_{C^{k_1-l,\alpha_1}_{w_1,\rho}(g)}\leq \|s_1\|_{C^{k_1,\alpha_1}_{\rho^{-l}w_1,\rho}(g)}$ for $l\leq k_1$,
        \item if $w\geq 1$, and if $(s_1,\dots,s_m)\mapsto q( s_1,\dots,s_m)$ is both
        \begin{itemize}
            \item at each point a \emph{pointwise} $m$-linear operation on the tensors $s_1,\dots,s_m$, i.e. a fiberwise $m$-linear form with a control $|q( s_1,\dots,s_m)|_g<C|s_1|_g\dots|s_m|_g$ for $C>0$ uniform on the manifold,
            \item  defined through contractions of the tensors $s_i$ with the metric $g$,
        \end{itemize}
          then, there exists $C'>0$ depending on $q$ and the parameters $k_i,\alpha_1,w_i,\rho$ and the metric, so that $$ \| q(s_1,\dots,s_m) \|_{C^{k,\alpha}_{w,\rho}}\leq C' \|s_1\|_{C^{k_1,\alpha_1}_{w_1,\rho}}\dots \|s_m\|_{C^{k_m,\alpha_m}_{w_m,\rho}},$$
        where $k:=\min(k_1,\dots,k_m)$, $\alpha:=\min(\alpha_1,\dots,\alpha_m)$, $w:=w_1\dots w_m$. 
        \item If $w\geq 1$, and if $h$ is a $2$-tensor and $\| h \|_{C^{k_1,\alpha_1}_{w_1,\rho}}<\varepsilon$ for some $\varepsilon>0$ depending only on the dimension, then the map $h\mapsto (g+h)^{-1}$ from $(0,2)-$tensors to $(2,0)-$tensors is real analytic as long as $w_1>1$.
    \end{enumerate}
\end{lemma}
\begin{proof}
    The first two points are obvious from the definitions of the norms, and the third point may be checked pointwise using a uniform bound on $q$ and its derivatives. The assumption $w\geq 1$ is crucial and is simply used to ensure that $w^m\geq w$ to control multilinear terms.

    The last point is proved through the local formula for the inverse: in a $g$-orthonormal basis, the $i,j$ component of $(g+h)^{-1}$ is $(g+h)^{ij} = \sum_{k\in\mathbb{N}}(-h)^k_{ij} $, where $(-h)^k$ is the composition of $k$ times of the matrices of $-h$ in the chosen basis--which are multilinear operations in the sense of point $3$. Indeed, the composition $k$-times is a continuous multilinear operation $q_k$ which satisfies for a dimensional constant $C$, $|(-h)^k|= |q_k(h,\dots,h)|<C^k|h|^k$ pointwise. For $|h|$ small enough the series $(g+h)^{ij} = \sum_{k\in\mathbb{N}}(-h)^k_{ij}$ is a converging power series in the sense of Banach spaces, Definition \ref{def:real analytic}.
\end{proof}

The maps we are interested in only involve various contractions and differentiations and will therefore be analytic as long as we choose the function spaces accordingly. 

\subsection{Real-analytic dependence of solutions of parametrized equations}

The goal is now to prove the analytic dependence of our metrics and tensors with respect to their parameter $\zeta$. Note that we will not need precise controls over our operators or their radii of convergence. Indeed, up to the action of $\mathbb{R}^+\times \operatorname{SO}(3)$ described in \eqref{eq homothety kro}, we will be able to work close to $\zeta$ of norm $1$ (which fixes the renormalized volume of \cite{Biq-Hein}) away from the forbidden set $\mathcal{D}$ of codimension $3$, corresponding to orbifold metrics.
 \\

Recall that the linearization of the map $g\rightarrow \Ric(g)-\frac{1}{2}\mathcal{L}_{B_{g_b}(g)}(g)$ at $g_b$, denoted by $-\frac{1}{2} \Delta_{L,g_b}: C^{2,\alpha}_\beta\mapsto C^{0,\alpha}_{\beta+2}$ is a Fredholm operator for $\alpha\in(0,1)$ and $\beta\in (0,4)$ and that its kernels and cokernels are equal to a space denoted by $\mathbf{O}(g_b)$ of symmetric, traceless, divergence-free $2$-tensors decaying like $r^{-4}$, \cite{ozuthese}. It is therefore possible to find parametrizations in which the metric depends in a real-analytic way (even if they are not hyperkähler). This yields the family of metrics of Definition \ref{defn: RFALE mod obst} at a general Ricci-flat ALE metric, see \cite[Appendix B]{ozu4}. 

However, we are presently given a specific parametrization through Kronheimer's period map and the associated parameter $\zeta$. We want to show the analyticity with respect to this specific parameter.
\\

We will only consider hyperkähler spaces classified in \cite{Kronheimer1989Torelli}. They are parametrized by $\zeta\in\mathbb{R}^{3k}$, for $k$ the dimension of the $L^2$-cohomology of the ALE space, a topological quantity. By \cite{Hausel-Hunsicker-Mazzeo}, for any ALE metric $g$, the space $\mathcal{H}^2_{L^2}(g)$ of $L^2$-harmonic $2$-forms is a well-defined $k$-dimensional space of $2$-forms. 
\begin{lemma}\label{lem cohomology HK}
    For a hyperkähler metric $g_b$, the $2$-forms in $\mathcal{H}^2_{L^2}(g_b)$ are anti-selfdual. 
\end{lemma}
\begin{proof}
    Consider $\Omega\in \mathcal{H}^2_{L^2}(g_b)$ and $\Omega^\pm := \frac{1}{2}\left(\Omega\pm\ast\Omega\right)$ its corresponding selfdual and anti-selfdual parts. The $2$-forms $\Omega^\pm$ are still $L^2$ since the selfdual and anti-selfdual $2$-forms are \textit{pointwise} orthogonal, and they are still harmonic since the Hodge star commutes with the Hodge Laplacian. 

    Now, the classical Weitzenböck formula for the Hodge Laplacian on selfdual $2$-forms of \cite{bourguignon} applied to $g_b$ with zero scalar curvature and zero selfdual Weyl curvature rewrites:
    \begin{equation*}
     0=(dd^*+d^*d)_{g_b}\Omega^+ = \nabla^*\nabla^{g_b}\Omega^+, 
     \end{equation*}
    and an integration by parts shows that $\Omega^+$ is parallel and $L^2$, hence identically zero. 
\end{proof}

The following proposition provides a noncompact Hodge theory in weighted Hölder spaces.  
\begin{prop}\label{prop: inv linearizations Hodge}
    For any $\beta\in(0,4)$ and $\alpha\in(0,1)$, the following operators are invertible:
    \begin{enumerate}
        \item $d_{-,g_b}\oplus d^*_{g_b}: \Lambda^1\cap C^{3,\alpha}_{\beta-1}\to (\Lambda^-_{g_b}\cap (\mathcal{H}^-_{L^2}(g_b))^\perp\cap C^{2,\alpha}_{\beta})\oplus (\Lambda^0\cap C^{2,\alpha}_{\beta})$,
        \item $d_{+,g_b}\oplus d^*_{g_b}: \Lambda^1\cap C^{3,\alpha}_{\beta-1}\to (\Lambda^+_{g_b}\cap C^{2,\alpha}_{\beta})\oplus (\Lambda^0\cap C^{2,\alpha}_{\beta})$,
        \item $ (d^*_{g_b},d)(\eta^-,f):=d^*_{g_b}\eta^- + df $ seen as an operator from $(\Lambda^-_{g_b}\cap (\mathcal{H}^-_{L^2}(g_b))^\perp\cap C^{2,\alpha}_{\beta})\oplus (\Lambda^0\cap C^{2,\alpha}_{\beta})$ to $\Lambda^1\cap C^{1,\alpha}_{\beta+1}$, and
        \item $ (d^*_{g_b},d)(\eta^+,f):=d^*_{g_b}\eta^+ + df $ seen as an operator from $(\Lambda^+_{g_b}\cap C^{2,\alpha}_{\beta})\oplus (\Lambda^0\cap C^{2,\alpha}_{\beta})$ to $\Lambda^1\cap C^{1,\alpha}_{\beta+1}$.
    \end{enumerate}
    More precisely, there exists $C>0$ depending on $g_b$ so that: for any $\gamma\in\Lambda^1\cap C^{3,\alpha}_{\beta-1}$,
    \begin{enumerate}
        \item $ \|\gamma\|_{C^{3,\alpha}_{\beta-1}} < C(\|d_{-,g_b}\gamma\|_{C^{2,\alpha}_\beta} + \|d^*_{g_b}\gamma\|_{C^{2,\alpha}_\beta} )$, 
        \item $ \|\gamma\|_{C^{3,\alpha}_{\beta-1}} < C(\|d_{+,g_b}\gamma\|_{C^{2,\alpha}_\beta} + \|d^*_{g_b}\gamma\|_{C^{2,\alpha}_\beta} )$,
        \item $\|\eta^-\|_{C^{2,\alpha}_\beta} + \|f\|_{C^{2,\alpha}_\beta}<C\|d_{g_b}^*\eta^- + df\|_{C^{1,\alpha}_{\beta+1}}$ if $\eta^-\perp_{L^2(g_b)} \mathcal{H}^-_{L^2}(g_b)$, and
        \item $\|\eta^+\|_{C^{2,\alpha}_\beta} + \|f\|_{C^{2,\alpha}_\beta}<C\|d_{g_b}^*\eta^+ + df\|_{C^{1,\alpha}_{\beta+1}}$.
    \end{enumerate}
\end{prop}
\begin{rk}
Observe that any element in $\mathcal{H}^2_{L^2}(g_b)$ decays as fast as $r^{-4}$ at infinity according to [\eqref{dec-oij}, Section \ref{exp-inf-def-obs}]. This fact implies that the subspace $(\mathcal{H}^-_{L^2}(g_b))^\perp\cap C^{2,\alpha}_{\beta}$ in $C^{2,\alpha}_{\beta}$ is closed for $\beta>0$, i.e. $(\mathcal{H}^-_{L^2}(g_b))^\perp\cap C^{2,\alpha}_{\beta}$ is a Banach space with respect to the norm induced by that of $C^{2,\alpha}_{\beta}$.
\end{rk}
\begin{proof}
    We only prove the isomorphism property for the operator $$d_{-,g_b}\oplus d^*_{g_b}: \Lambda^1\,\cap\, C^{3,\alpha}_{\beta-1}\to (\Lambda^-_{g_b}\,\cap\, (\mathcal{H}^-_{L^2}(g_b))^\perp\,\cap \,C^{2,\alpha}_{\beta})\oplus (\Lambda^0\,\cap\, C^{2,\alpha}_{\beta}),$$ the case of $d_{+,g_b}\oplus d^*_{g_b}$ is similar and simpler.

    The operator $d_{-,g_b}\oplus d^*_{g_b}$ is elliptic as part of the classical elliptic complex
    \begin{equation}
        0\xrightarrow[]{}\Lambda^0\xrightarrow[]{d}\Lambda^1\xrightarrow[]{d_-}\Lambda^{-}\xrightarrow[]{}0.
    \end{equation}
    Between weighted Hölder spaces, $d_{-,g_b}\oplus d^*_{g_b}$ is therefore Fredholm as long as a discrete set of exceptional values of $\beta$ is avoided. Let us determine a range in which it is injective and surjective up to $\mathcal{H}^-_{L^2}(g_b)$.

    Let us start with the injectivity, and let us omit the references to the metric $g_b$ for simplicity. Assume that $d_{-}\gamma = 0$ and $d^*\gamma = 0$, then in particular, since $$2d^*d_-\gamma = -*d*(d-*d)\gamma =-*d*d\gamma = d^*d\gamma, $$
    we see that $(dd^*+d^*d)\gamma = 0$, which implies that $\gamma = 0$ if it lies in $C^{3,\alpha}_{\beta-1}$ with $\beta\in(-1,4)$. Indeed, for $\beta>3$, the only solution is $\gamma=0$ since $\mathcal{H}^1_{L^2}(g_b)=0$, and there are no exceptional values for $(dd^*+d^*d)$ in the interval $ (-1,4) $ following an adaptation of the argument of \cite[Lemme 4.30]{ozuthese}.
    
    Let us now prove the surjectivity. The adjoint of $d_{-,g_b}\oplus d^*_{g_b}:\Lambda^1\cap C^{3,\alpha}_{\beta-1}\to (\Lambda^-_{g_b}\cap (\mathcal{H}^-_{L^2}(g_b))^\perp\cap C^{3,\alpha}_{\beta})\oplus (\Lambda^0\cap C^{3,\alpha}_{\beta})$ is the operator $ (d^*_{g_b},d)(\eta,f):=d^*_{g_b}\eta + df $ seen as an operator from $(\Lambda^-_{g_b}\cap (\mathcal{H}^-_{L^2}(g_b))^\perp\cap C^{4,\alpha}_{4-\beta})\oplus (\Lambda^0\cap C^{4,\alpha}_{4-\beta})$ to $\Lambda^1\cap C^{3,\alpha}_{5-\beta}$. 

    Assume that $d^*\eta =0$ and $df =0$ with $(\eta,f)\in(\Lambda^-_{g_b}\cap (\mathcal{H}^-_{L^2}(g_b))^\perp\cap C^{4,\alpha}_{4-\beta})\oplus (\Lambda^0\cap C^{4,\alpha}_{4-\beta})$. If $4-\beta>0$, then $df = 0$ implies that $f=0$ since it is constant and decaying at infinity. For $\eta$, we note that there are no exceptional weights in $(0,4)$ for $2d_-d^* = dd^* + d^*d$ on $\Lambda^-$. Since for $C^{4,\alpha}_{3}\subset L^2(g_b)$, the kernel is $\mathcal{H}^-_{L^2}(g_b)$, we obtain surjectivity for all $4-\beta\in(0,4)$, i.e. $\beta\in(0,4)$.
\end{proof}



In order to keep track of the dependence on the parameter $\zeta\in\RR^{3k}$, we will construct geometric tensors on the manifold under consideration that vary analytically in $\zeta$. Recall the notation $\omega_i^+(\mathbf{e})$ from \eqref{def omegai euclidean} on Euclidean space, and denote by $(\Omega_j^-)_j$ the basis of anti-selfdual $2$-forms from Section \ref{sec-inf-def-hyp}.
\begin{prop}\label{prop analytic dependence tensors}
    Let $g_b$ be a hyperkähler ALE metric and let $\chi$ be a fixed cut-off function supported in an ALE chart of $g_b$ (in the sense of Definition \ref{def orb ale}). Then, for any small enough symmetric $2$-tensor $h\in C^{2,\alpha}_\beta$ with $\beta\in(0,4)$, there exist: 
    \begin{enumerate}
        \item a basis $(\Omega_1^-(h),\dots,\Omega_k^-(h))$ that is anti-selfdual with respect to $g_b+h$ and which lies in $\mathcal{H}^2_{L^2}(g_b+h)$ with $\Omega_i^-(h)-\Omega_i^-\perp_{g_b} \Omega_j^-$ for all $i, j\in \{1,\dots,k\}$,
        
        \item a unique family $(\omega^+_1(g),\omega^+_2(g),\omega^+_3(g))$ of $g=(g_b+h)$-harmonic selfdual $2$-forms with $\omega^+_i(g)- \omega_i^+(g_b)\in C^{2,\alpha}_\beta$,
        \item for $2<\beta<4$, a unique function $u$ satisfying $\Delta_{g_b+h} u(h)=4$ with $u(h)-\chi r^2/2 \in C^{2,\alpha}_{\beta-2}$.
    \end{enumerate}
    Moreover the following maps are real-analytic: 
    \begin{enumerate}
        \item $h\mapsto \Omega_i^-(h)$ between $C^{2,\alpha}_\beta$ and $C^{2,\alpha}_{\beta}$,
        \item $h\mapsto \omega_i^+(g_b+h)-\omega_i^+(g_b)$ between $C^{2,\alpha}_\beta$ and $C^{2,\alpha}_{\beta}$, and
        \item if $2<\beta<4$, then $h\mapsto u(h)-\chi r^2/2$ between $C^{2,\alpha}_\beta$ and $C^{2,\alpha}_{\beta-2}$.
    \end{enumerate}
    
    As a consequence, the map $h\mapsto \mathbf{o}_{ij}(h):= \Omega_j^-(h)\circ_{g_b+h} \omega_i^+(g_b+h)$ is also real-analytic between $C^{2,\alpha}_\beta$ and $C^{2,\alpha}_{\beta}$.
\end{prop}
\begin{proof}
    We prove all of the above results by using the real-analytic implicit function Theorem \ref{th implicit functions} to appropriate maps.
    First, let $\mathcal{B}^{2,\alpha}_{\beta}$ be the ball of radius $1/2$ centered at $0$ in $\operatorname{Sym}^2(M)\cap C^{2,\alpha}_\beta$ so that from Lemma \ref{lem norm multilinear}, it is clear that the following maps are real-analytic:
    \begin{itemize}
        \item $\Phi_1: \mathcal{B}^{2,\alpha}_{\beta}\times (\Lambda^-_{g_b}\cap C^{2,\alpha}_\beta) \times (\Lambda^0\cap C^{2,\alpha}_\beta)\rightarrow(\Lambda^1\cap C^{1,\alpha}_{\beta+1})\times \mathbb{R}^k$, where 
        \begin{equation*}
        \begin{split}
        \Phi_1&(h,\eta_i^-,f_i):= \\
        &\left(d^*_{g_b+h}\pi_{g_b+h}^-(\Omega_i^-+\eta_i^-)+df_i,\left(\int_{\Sigma_j}\pi_{g_b+h}^-(\Omega_i^-+\eta_i^-)-\int_{\Sigma_j}\Omega_i^-\right)_j\right).
        \end{split}
        \end{equation*} 
        \item $\Phi_2:\mathcal{B}^{2,\alpha}_{\beta}\times (\Lambda^+_{g_b}\cap C^{2,\alpha}_\beta) \times (\Lambda^0\cap C^{2,\alpha}_\beta)\rightarrow\Lambda^1\cap C^{1,\alpha}_{\beta+1}$ where
        \begin{equation*}
        \Phi_2(h,\eta_i^+,\phi_i):= (d^*_{g_b+h}\pi_{g_b+h}^+(\chi\omega_i(\mathbf{e})^++\eta_i^+)+d\phi_i).
        \end{equation*}
        \item $\Phi_3:\mathcal{B}^{2,\alpha}_{\beta}\times (\Lambda^0\cap C^{2,\alpha}_\beta)\rightarrow\Lambda^0\cap C^{0,\alpha}_{\beta+2}$ where
        \begin{equation*}
        \Phi_3(h,v):=\Delta_{g_b+h}(\chi r^2/2 +v) - 4.
        \end{equation*}
Here $\chi$ a cut-off function supported in a neighborhood of infinity.
    \end{itemize}
    Given solutions to $\Phi_1(h,\eta_i^-,f_i) =0$, $\Phi_2(h,\eta_i^+,\phi_i) = 0$ and $\Phi_3(h,v)=0$, we define:
    \begin{itemize}
        \item $\Omega_i^-(h) := \pi_{g_b+h}^-(\Omega_i^-+\eta_i^-)$,
        \item $\omega_i^+(g_b+h):=\pi_{g_b+h}^+(\chi\omega_i(\mathbf{e})^++\eta_i^+)$, and
        \item $u(h):=\chi r^2/2 +v$.
    \end{itemize}
    They satisfy the desired properties.
    
    Note that the functions $(f_i)_i$ and $(\phi_i)_i$ are simply auxiliary functions simplifying the proof of the invertibility of our operators.

    Our goal is to parametrize zeros of the maps $\Phi_1$, $\Phi_2$ and $\Phi_3$ by the perturbation $h$ by applying the implicit function Theorem \ref{th implicit functions}. For this we need the existence of solutions at $h=0$, and the invertibility of the linearization at $h=0$ in the direction of the second variable of the maps $\Phi_i$ for $i=1,2,3$.

    See the discussion preceding this proposition for $\Omega_i^- = \Omega_i^-(0)$ and for the existence of $u(0)$, see Section \ref{sec-can-rad-vec-field}. As for the existence of $\omega_i^+(g_b)$, it can be proved at the same time as the invertibility of the linearization of the affine map $\Phi_2$ at $h=0$. Indeed, $\Phi_2(0,0,0)\in \Lambda^1\cap C^{1,\alpha}_\beta$, and the surjectivity of the linearization lets one obtain the existence of $\omega_i^+(g_b)$.  We give a proof for completeness. Notice that $(dd^*+d^*d)_{g_b}(\chi \pi_{g_b}^+(\omega_i^+(\mathbf{e})))\in C^{0,\alpha}_{\beta+2}$ and that it is selfdual. Now, $(dd^*+d^*d)_{g_b}: C^{2,\alpha}_\beta\to C^{0,\alpha}_{\beta+2}$ is Fredholm and invertible because of the Weitzenböck formula $(dd^*+d^*d)_{g_b}= \nabla^*\nabla^{g_b}$ on selfdual $2$-forms as recalled in the proof of Lemma \ref{lem cohomology HK}. Consequently, there exists a unique selfdual $2$-form $\eta_{i}^+\in C^{2,\alpha}_\beta$ such that $(dd^*+d^*d)_{g_b}\eta_i^+ = - (dd^*+d^*d)_{g_b}(\pi_{g_b}^+(\chi \omega_i^+(\mathbf{e})))$, and we define 
    \begin{equation*}
    \omega_i^+(g_b):= \pi_{g_b}^+(\chi \omega_i(\mathbf{e})^+) + \eta_i^+.
    \end{equation*}

    Let us now investigate the invertibility of the linearizations of the maps $\Phi_i$, $i=1,2,3$ at $h=0$.

    The linearization of the affine map $(\eta^-,f)\mapsto \Phi_1(0,\eta^-,f)$ is  $(\eta^-,f)\mapsto  (d^*_{g_b}\eta^-+df,(\int_{\Sigma_j}\eta^-)_j)$, where one recognizes on the first factor the operator $d^*\oplus d$ which is surjective between our spaces of interest by Proposition \ref{prop: inv linearizations Hodge}. It is however not injective: its kernel being $\mathcal{H}^-_{L^2}(g_b)$ on which $\eta^-\mapsto (\int_{\Sigma_j}\eta^-)_j$ is invertible, this allows one to conclude about the invertibility of the differential of $\Phi_1$ with respect to the last two variables.

    Similarly, the linearization of $\Phi_2$ is invertible by Proposition \ref{prop: inv linearizations Hodge}.

    Finally, regarding the invertibility of the linearization of the map $\Phi_3$ at $g_b$ with respect to the second variable, it is a well-known fact that $\Delta_{g_b}:C^{2,\alpha}_{\beta-2}\to C^{0,\alpha}_{\beta}$ is invertible for $2<\beta<4$: see for instance \cite{Bartnik1986}.

Applying the implicit function Theorem \ref{th implicit functions} lets us conclude for each map $\Phi_i$, $i=1,2,3$.

\end{proof}

By applying the same technique as that of the proof of Proposition \ref{prop analytic dependence tensors} applied to the map:
\begin{equation*}
 \Phi_4(h,k,(l_{ij})_{ij}):=\Delta_{L,g_b+h}(h_2(0)+k) - \sum_{i,j}l_{ij}\mathbf{o}_{ij}(h)-\Lambda (g_b+h),
\end{equation*}
seen as a map from $C^{2,\alpha}_\beta\times C^{2,\alpha}_{\beta-2}\times \mathbb{R}^{3k}$ with values into $C^{0,\alpha}_\beta$, where $\mathbf{o}_{ij}(h) = \omega_i^+(g_b+h)\circ_{g_b+h}\Omega_j^-(h)$ we may prove the following result. 
Notice that $\Phi_4$ is well-defined provided there exists $(h_2(0),l_{ij}(0))\in C^{2,\alpha}_{-2}\times \mathbb{R}^{3k}$ such that $\Phi_4(0,0,l_{ij}(0))=0$.
\begin{coro}
    Let $(h_2(0),l_{ij}(0))\in C^{2,\alpha}_{-2}\times \mathbb{R}^{3k}$ be solutions of $\Delta_{L,g_b}h_2(0) - \sum_{ij}l_{ij}(0)\mathbf{o}_{ij}(0)-\Lambda g_{b}=0$ as in Proposition \ref{prop-h2-H2-kro}. Then, for any $\beta\in(2,4)$ and  for every $h\in C^{2,\alpha}_\beta$ small enough, there exists unique $(k(h),l_{ij}(h))\in C^{2,\alpha}_{\beta-2}\times \mathbb{R}^{3k}$ such that:
    \begin{itemize}
        \item $\Delta_{L,g_b+h}(h_2(0)+k(h)) - \sum_{ij}l_{ij}(h)\mathbf{o}_{ij}(h)=\Lambda (g_{b}+h)$, and
        \item $h_2(h) = \lambda (g_b+h) + \sum_{i}\phi_i^-\circ\omega^+_i(g_b+h)$, where $\phi_i^-$ are $(g_b+h)$-anti-selfdual $2$-forms such that $\int_{\Sigma_j}\phi_i^-=0$.
    \end{itemize}
    Moreover, the map $h\mapsto (h_2(h)-h_2(0)):=k(h),l_{ij}(h))$ is real-analytic as a map from $C^{2,\alpha}_\beta$ with values into $C^{2,\alpha}_{\beta-2}\times \mathbb{R}^{3k}$.
\end{coro}

Let us also compute the linearization of the map $h\mapsto \omega_i^+(g_b+h)$ which we will use shortly after.

\begin{lemma}\label{lem var omegai}
    The linearization of the map $h\mapsto \omega_i^+(g_b+h)$ for $h= \lambda g_b + \sum_i\phi_i^-\circ\omega_i^+(g_b)$ at $h=0$ is:
    $$ -\phi_i^-+S_{g_b}(\phi_i^-), $$
    for some linear operator $S_{g_b}: C^{2,\alpha}_\beta\cap \Lambda^-\to C^{2,\alpha}_\beta\cap \Lambda^+$ whose kernel is $\mathcal{H}^2_{L^2}(g_b)$. In particular, if $h = \sum_{ij}h_{ij}\Omega_j^-\circ \omega_i^+(g_b)$ for $h_{ij}\in\mathbb{R}$, then the linearization is exactly
    $$  -\sum_{ij}h_{ij}\Omega_j^-. $$
\end{lemma}
\begin{proof}
    This is the issue of the variation of the set of selfdual harmonic $2$-forms addressed in \cite{Donaldson1986}, see also \cite{no1}. Denote for $t\in(-1,1)$ $g_t := g_b +t h$ for $h = \lambda g+ \sum_i\phi_i^-\circ \omega_i^+(g_b)$, and let $\phi:\Lambda^+\to\Lambda^-$ be the linear map defined by $\phi(\omega_i^+(g_b))=\phi_i^-$, and let $\alpha_t$ be one of the $\omega_i^+(g_b+th)$.





    We will generally denote with a ``dot'' the derivatives at $t=0$. Our goal here is to understand $\dot{\alpha}$ as a function of $\dot{g}$ and $\alpha_0=\alpha$. For this, we decompose it into its selfdual and anti-selfdual parts.
    \\

    Differentiating the equation $\pi^t_-\alpha_t:=\frac12 (\alpha_t-*_{g_t}\alpha_t)=0$, for at $t=0$ thanks to the formulas for the variation of the projection \cite{Biq-2,no1} gives:
\begin{equation}
 \pi_-\dot{\alpha} = -\phi(\alpha) = -\phi_i^-.
\end{equation} 
We will then study the selfdual part $\dot{\alpha}^+$.

Following \cite{Donaldson1986} or from the implicit function theorem, denote $G_+$, the inverse for $d_+^*d=\frac{1}{2}(d^*_{g_b}d+dd^*_{g_b}):C^{2,\alpha}_\beta\to C^{0,\alpha}_{\beta+2}$ from $\Lambda^+_{g_b}$ to itself (only), where $d_+^* := \pi_{g_b}^+d^*_{g_b} :\Lambda^3\to \Lambda^+_{g_b}$. Define $S_{g_b}:\Lambda^-_{g_b}\to\Lambda^+_{g_b}$ as $G_+\circ d_+^*d$ (note that $d_+^*d=d^*_{g_b}d-dd^*_{g_b}$ on $\Lambda^-_{g_b}$). Maybe more concretely, if $\beta_-\in \Lambda^-$ and $\beta_+\in \Lambda^+$ satisfy $d\beta_+ = d\beta_-$ and are in $C^{2,\alpha}_\beta$, then $ S_g(\beta_-)=\beta_+$. 

This operator lets us understand the deformation $\dot{\alpha}^+$. Indeed, since $0=d\dot{\alpha}^++d\dot{\alpha}^- =d\dot{\alpha}^+-d\phi(\alpha)$. This tells us that, $d\dot{\alpha}^+ = d\phi(\alpha)$, $\dot{\alpha}^+\in C^{2,\alpha}_\beta$ and $\phi(\alpha)\in\Lambda^-$, and consequently, $\dot{\alpha}^+ = S(\phi(\alpha))$ and finally:
$$ \partial_{t|t=0}\alpha_t = -\phi_i^-+S_{g_b}(\phi_i^-). $$
\end{proof}

As explained in Section \ref{sec-kro}, the elements $\Omega_j^-$ may be chosen so that their Poincaré duals of the $\Sigma_j$, the $\Sigma_j$ are smooth embedded $2$-spheres. Recall that we defined $\zeta = (\zeta_i^j)_{ij}\in \mathbb{R}^{3k}$ as:
$$ \zeta_i^j:= \int_{\Sigma_j}{\omega}_i^+(g_\zeta).$$
This is Kronheimer's period map described in Section \ref{sec-kro}.
Let us now explain how to precise the above analytic dependence of the metrics, even after their first term of the expansion in $r^{-4}$. The key is to use the following norms in order to separate the explicit $r^{-4}$-term $H^4_\zeta$ from the rest of the expansion of the nearby metrics. Consider $\chi: \mathbb{R}^+\to [0,1]$ a cut-off function supported in $[1,+\infty)\subset \mathbb{R}^+$. 

\begin{defn}[$C^{2,\alpha}_{\beta,**}$-norm on a ALE orbifold]
    Let $(N,g_b)$ be an ALE orbifold, and let $h$ be a symmetric $2$-tensor on $N$, and assume that $h = H^4 + O(r_b^{-4-\beta})$ for $\beta>0$ and $H^4$ a homogeneous harmonic symmetric $2$-tensor with respect to $\mathbf{e}$ with $|H^4|\sim r_\mathbf{e}^{-4}$. We define its $C^{2,\alpha}_{\beta,**}$-norm by
    $$\|h\|_{C^{2,\alpha}_{\beta,**}}:= \sup r_\mathbf{e}^{4}|H^4|_{\mathbf{e}} + \big\|(1+r_b)^{4}(h-\chi(|\zeta| r_b^{2}) H^4)\big\|_{C^{2,\alpha}_{\beta,*}}.$$
\end{defn}

\begin{prop}\label{prop analytic dependence}
    Let $g_b = g_{\zeta_0}$  be a hyperkähler ALE metric on $N$ for $\zeta_0\in\mathbb{R}^{3k}\setminus \mathcal{D}$, where $\mathcal{D}$ is the ``forbidden set'' introduced in Section \ref{sec-kro} and let $(\Sigma_j)_j$ be the associated elements of $H^2_c(N)$. 
Let $\beta\in(0,4)$.
    Then, for every $\zeta\in\mathbb{R}^{3k}$ close enough to $\zeta_0$, there exists a unique hyperkähler metric $g_\zeta=g_b+\tilde{h}^5_{\zeta} = g_b + \chi_\zeta H^4_\zeta + \tilde{h}^5_{\zeta}$ such that $\tilde{h}^5_{\zeta}\in C^{2,\alpha}_{4+\beta}$ for any $0<\beta<1$ and:
    \begin{enumerate}
        \item $B_{g_b}(g_\zeta-g_b)=0$, and
        \item $\int_{\Sigma_j}\omega_i^+(g_\zeta) = \zeta_i^j$.
    \end{enumerate}
    Moreover, the map $\zeta\mapsto g_\zeta-g_{\zeta_0}-\chi_\zeta H^4_\zeta$ is real-analytic as a map from a neighborhood of $\zeta_0\subset \RR^{3k}$ with values into $C^{2,\alpha}_{\beta}$.
\end{prop}
\begin{proof}
    This is again an application of the implicit function Theorem \ref{th implicit functions}. We apply it to the map: for $(\zeta,h)\in\mathbb{R}^3\times C^{2,\alpha}_{\beta,**}$,
    \begin{equation*}
    \begin{split}
    \Phi_5(\zeta,h)&:= (\Phi_5^1(\zeta,h),\Phi_5^2(h))\in \mathbb{R}^{3k}\times  (\ker_{L^2}\Delta_{L,g_b})^\perp\cap C^{0,\alpha}_{\beta+6},\\
    \Phi_5^1(\zeta,h)_{ij} &:= \int_{\Sigma_j}\omega_i^+(g_b+h)-\zeta_i^j,\quad \Phi_5^2(h) := \pi_{(\ker_{L^2}\Delta_{L,g_b})^\perp}\left(\Ric(g_b+h)-\frac{1}{2}\Li_{B_{g_b}(h)}(g_b+h)\right),
    \end{split}
    \end{equation*}
    where $\omega_i^+(g_b+h)$ is defined in Proposition \ref{prop analytic dependence tensors}, and where $\pi_{(\ker_{L^2}\Delta_{L,g_b})^\perp}$ denotes the orthogonal projection on $(\ker_{L^2}\Delta_{L,g_b})^\perp$.

    We already have a solution to $\Phi_5(\zeta,h)=(0,0)$ if $h=0$ by assumption since $\Phi_5(\zeta_0,0)=(0,0)$. We need to show that the linearization of $\Phi_5$ with respect to the second variable is invertible. 
    
    Indeed, it is well-known (and reproved in the present function spaces in \cite{ozuthese}) that the infinitesimal Ricci-flat deformations of $g_b$ are \textit{integrable}, meaning that solving 
\begin{equation*}
\pi_{(\ker_{L^2}\Delta_{L,g_b})^\perp}\left(\Ric(g_b+h)-\frac{1}{2}\Li_{B_{g_b}(h)}(g_b+h)\right)=0,
\end{equation*}
 is equivalent to solving $\Ric(g_b+h)-\frac{1}{2}\Li_{B_{g_b}(h)}(g_b+h)=0$, which is also equivalent to solving both $\Ric(g_b+h)=0$ and $B_{g_b}(h)=0$ by Bianchi's identity together with the invertibility of $B_{g_b}(\Li_{V}(g_b)) = \Delta_{g_b}V$ on a Ricci flat metric $g_b$ for any $C^2_{loc}$ vector field $V$.
 
Therefore it suffices to show that the linearization of $\Phi_5$ with respect to the second variable is invertible. From \cite[Lemma 5.5]{ozu2}, the linearization of $\Phi_5^2$ at $0$ is surjective between $C^{2,\alpha}_{\beta,**}$ and $(\ker_{L^2}\Delta_{L,g_b})^\perp\cap C^{0,\alpha}_{\beta+6}$, and its kernel is $\ker_{L^2}\Delta_{L,g_b}\subset C^{2,\alpha}_{\beta,**}$. Now, the $C^{2,\alpha}_{\beta,**}$-closed space of $2$-tensors $h=\frac{\tr_{g_b} h}{4} g_b + \sum_i \phi_i^-\circ \omega_i^+(g_b)$ with $\int_{\Sigma_j}\phi_i^-=0$ for all $i,j$ is a complement of $\ker_{L^2}\Delta_{L,g_b}$, see for instance \cite{Biq-2}. 
    
    Any $2$-tensor $h\in C^{2,\alpha}_{\beta,**}$ splits as $h= h_\perp + h_T =h_\perp  + \sum_{ij} h_{ij} \Omega_j^-\circ \omega_i^+(g_b)\in S\oplus \ker_{L^2}\Delta_{L,g_b}$ where $S$ is the complement of $\ker_{L^2}\Delta_{L,g_b}$ composed of symmetric $2$-tensors $h_\perp=\frac{\tr_{g_b} h}{4} g_b + \sum_i \phi_i^-\circ \omega_i^+(g_b)$  with $\int_{\Sigma_j}\phi_i^-=0$. The linearization of $\Phi_5^2$ is invertible on $S$ by the above discussion. 
    
    Consequently, we are left with proving that the linearization of $h\mapsto \Phi_5^1(\zeta,h)$ computed in Lemma \ref{lem var omegai} is invertible at $(\zeta,h)=(0,0)$ on $\ker_{L^2}\Delta_{L,g_b}$. By Lemma \ref{lem var omegai}, the infinitesimal variation of $h\rightarrow \int_{\Sigma_k}\omega_i^+(g_b+h)$ associated to the infinitesimal deformation $\sum_{ij} h_{ij} \Omega_j^-\circ \omega_i^+(g_b)$ of the metric is $ -\sum_{ij} h_{ij} \int_{\Sigma_k}\Omega_j^- = \sum_{ij}a_{jk}h_{ij} $, where the $a_{jk}$ are the coefficients of the intersection form associated to the Poincaré duals $(\Sigma_j)_j$. Recall $A = (a_{jk})_{jk}$ defines an invertible matrix as explained in Section \ref{period-map-homology}. Consequently, the linearization of $h\mapsto \Phi^1_5(\zeta,h)$ is invertible at $(\zeta,h)=(0,0)$ since it is
    \begin{equation*}
     \sum_{ij} h_{ij} \Omega_j^-\circ \omega_i^+(g_b) \mapsto \Big(\sum_ja_{jk}h_{ij}\Big). 
\end{equation*}

\end{proof}

We may therefore parametrize every tensor of interest by $\zeta$ in a real-analytic way, and obtain controls on their variations at every order as follows.

\begin{prop}\label{prop-analy}
	The following maps are real-analytic for $\zeta$ in a neighborhood of $\zeta_0\in\mathbb{R}^{3k}\backslash\mathcal{D}$: for any $l\in\mathbb{N}$, $\alpha\in(0,1)$ and $\beta\in(2,4)$,
	\begin{enumerate}
		\item $\zeta\mapsto g_\zeta = g_{\zeta_0} + h_\zeta = g_{\zeta_0} + \chi_\zeta H^4_\zeta-\chi_{\zeta_0}H^4_{\zeta_0} + \tilde{h}^5_\zeta$ with the notation of Proposition \ref{prop analytic dependence} where $h_\zeta$ has values into $ C^{l,\alpha}_{5^-} $ symmetric $2$-tensors,
		\item $\zeta\mapsto \mathbf{o}_{ij}(h_\zeta)=:\mathbf{o}_{ij}(\zeta)$ with values into $ C^{l,\alpha}_\beta $ symmetric traceless-2-tensors,
		\item $\zeta\mapsto h_2(h_\zeta)-h_2(h_{\zeta_0})=:h_2(\zeta)-h_2(\zeta_0)$ with values into $C^{l,\alpha}_{\beta-2} $ symmetric $2$-tensors,
		\item $\zeta\mapsto u(h_\zeta)-u(h_{\zeta_0})=:u(\zeta)-u(\zeta_0)$ with values into $ C^{l,\alpha}_\beta $ functions.
	\end{enumerate}
	
	As a consequence, we also have the following controls: for any $m\in \mathbb{N}$, denote $d^{(m)}_\zeta$ the $m$-th derivative of a tensor with respect to the $\zeta$ variable, as $m$-linear operators from $(\mathbb{R}^{3k})^m$, we find the following pointwise controls. For any $l\in\mathbb{N}$ and $\kappa>0$
	\begin{enumerate}
		\item 
  \begin{itemize}
      \item $\left|\nabla^{g_{\zeta_0},\,l}(d^{(m)}_\zeta g(\zeta_1)-d^{(m)}_\zeta g(\zeta_2))\right| \leq C\frac{ |\zeta_1-\zeta_2|}{|\zeta_0|^{1+m}} |\zeta_0|^{2-\kappa}(\varepsilon_0+r_o)^{-4+2\kappa-l},$
      \item $\left|\nabla^{g_{\zeta_0},\,l}(d^{(m)}_\zeta \tilde{h}^5(\zeta_1)-d^{(m)}_\zeta \tilde{h}^5(\zeta_2))\right| \leq C\frac{ |\zeta_1-\zeta_2|}{|\zeta_0|^{1+m}} |\zeta_0|^{5-\kappa}(\varepsilon_0+r_o)^{-5+2\kappa-l},$
  \end{itemize}
		\item $\left|\nabla^{g_{\zeta_0},\,l}(d^{(m)}_\zeta \mathbf{o}_{ij}(\zeta_1)-d^{(m)}_\zeta \mathbf{o}_{ij}(\zeta_2))\right| \leq C \frac{ |\zeta_1-\zeta_2|}{|\zeta_0|^{1+m}} |\zeta_0|^{2-\kappa}(\varepsilon_0+r_o)^{-4+2\kappa-l},$
		\item $\left|\nabla^{g_{\zeta_0},\,l}(d^{(m)}_\zeta h_2(\zeta_1)-d^{(m)}_\zeta h_2(\zeta_2))\right| \leq C \frac{ |\zeta_1-\zeta_2|}{|\zeta_0|^{1+m}} |\zeta_0|^{2-\kappa}(\varepsilon_0+r_o)^{-2+2\kappa-l},$
		\item $\left|\nabla^{g_{\zeta_0},\,l}(d^{(m)}_\zeta u(\zeta_1)-d^{(m)}_\zeta u(\zeta_2))\right| \leq C\frac{ |\zeta_1-\zeta_2|}{|\zeta_0|^{1+m}} |\zeta_0|^{2-\kappa}(\varepsilon_0+r_o)^{-2+2\kappa-l}.$
	\end{enumerate}
\end{prop}
\begin{proof}
	The analyticity at $\zeta_0$ with $|\zeta_0| = 1$ away from the codimension $3$ forbidden set $\mathcal{D}\subset \mathbb{R}^{3k}$ provides the estimates: for $r$ depending on $\zeta_0$ and some constant $C$ and for $\zeta_1,\zeta_2$ with $|\zeta_1-\zeta_0|<r$ and $|\zeta_2-\zeta_0|<r$, the estimates hold.
	The estimates for others $\zeta = (\varepsilon,\phi)^*\zeta_0$ are obtained by rescaling thanks to \eqref{eq homothety kro}. For instance, one has
	\begin{align*}
		|g_{(\varepsilon,\phi)^*\zeta_1}-g_{(\varepsilon,\phi)^*\zeta_2}|_{g_{(\varepsilon,\phi)^*\zeta_0}} =& \,(s_{(\varepsilon,\phi)})_*|g_{\zeta_1}-g_{\zeta_2}|_{g_{\zeta_0}}\\
  \leq&\, C|\zeta_1-\zeta_2|\cdot (s_{(\varepsilon,\phi)})_*|\zeta_0|^{2-\kappa}(|\zeta_0|^\frac{1}{2}+r_o)^{-4+2\kappa}\\
  =&\,  C\frac{|\zeta_1-\zeta_2|}{\varepsilon^2}|\varepsilon^2\zeta_0|^{2-\kappa}(|\varepsilon^2|\zeta_0||^\frac{1}{2}+r_o)^{-4+2\kappa}.
	\end{align*}
	
	As a sanity check, one verifies that these estimates are true on the Eguchi-Hanson metric. 
\end{proof}

\newpage

\section{Heat flows modulo obstructions}\label{sec-flot-mod-obst}


In this section, we state and give a proof of the existence of solutions to parabolic equations ``modulo obstructions'' as used in \cite[Corollary 6.6]{Bre-Kap}, and in Proposition \ref{prop-not-so-easy-bis} of the present paper. Let us mention some standard simplifying steps before stating and proving the core result.
We only discuss the ancient case, the immortal case can be adapted easily.
As discussed in the proof of Proposition \ref{prop-not-so-easy-bis}, we only need to show the existence of solutions on intervals $[T',T]$. 
Define:
\begin{itemize} 
\item $I:=[T',T]$,
        \item $(g_t)_{t\in I}$ to be the curve $(\tilde{g}(t))_t$ restricted to $I$ as defined in \eqref{defn-tilde-g},
        \item $X:=C^{2,\alpha}_{\gamma,\sigma,T}$ and $Y:=C^{0,\alpha}_{\gamma,\sigma+2,T}$ restricted to $(T',T]$ as defined in Section \ref{subsec-holder},
        \item $(P_t)_{t\in I}$ to be the linearized operator $ (\Delta_{L,\tilde{g}(t)} +2\Lambda)_{t\in I}$ ,
        \item $(V_t)_{t\in I}$ the spaces spanned by $((\tilde{\oi}_i(t))_{i},\tilde{g}(t),\tilde{c}(t))_{t\in I}$ from $\tilde{\mathbf{O}}(t)$ as defined in Definition \ref{defn-app-kernel}.
     \item  $W_t$ the $L^2_{g_t}$-orthogonal complement of $V_t$ for each $t\in I$.
\end{itemize} 
For a family of symmetric $2$-tensors $(h_t)_{t\in I}$ in $X$, we will denote the corresponding norm by $\|h_t\|_X$ with a slight abuse of notation.

\begin{prop}\label{prop-flot-mod-obst}
   With the notations introduced above, assume that: there exist $C>0$ and $\varepsilon>0$ such that, 
\begin{enumerate}
\item\label{item-1} the operator $\partial_t-P_t:X\rightarrow Y$ is continuous: $\|(\partial_t -P_t)h_t\|_Y\leq C \|h_t\|_X $ for all $(h_t)_{t\in I}\in X$,

	\item \label{item-2} For any $ (\psi_t)_{t\in I}\in Y $, there exists a unique $(u_t)_{t\in I}\in X$ such that $(\partial_t -P_t)u_t = \psi_t$ and $u_0=0$,
    with 
    $$\|u_t\|_X\leq C \|\psi_t\|_Y,$$
         \item \label{item-3} the projections are continuous: $\|\pi_{V_t}(h_t)\|_X\leq C\|h_t\|_X$, and $\|\pi_{W_t}(h_t)\|_X\leq C\|h_t\|_X$,
    \item \label{item-4} if $h_t\in V_t$, then $\|\pi_{W_t}(\partial_t-P_t) h_t\|_Y\leq \varepsilon \|h_t\|_X$.
\end{enumerate}

    Then, if $\varepsilon>0$ is small enough compared to $C>0,$ for any family of $2$-tensors on $M$, $(\psi_t)_{t\in I}\in Y$ and any initial $2$-tensor $h_0\in W_0$, then there exists a family of $2$-tensors $(h_t)_{t\in I}\in X$  so that for every $t\in I$, one has
    \begin{enumerate}
        \item $(\partial_t-P_t)h_t-\psi_t \in V_t$,
        \item $h_0=0$, and
        \item $h_t\in W_t$.
    \end{enumerate}
\end{prop}

\begin{rk}\label{rk-dep-T'}
The only constant that may depend on $T'$ in the assumptions of Proposition \ref{prop-flot-mod-obst} is that of item \eqref{item-2}: see the discussion after the proof of Proposition \ref{prop-flot-mod-obst} below.
\end{rk}

\begin{proof}
    Define the following ``box'' operator acting on family of symmetric $2$-tensors $(h_t)_{t\in I}\in X$,
\begin{equation*}
\Box_th_t:=\left(\partial_t-P_t \right)h_t.
\end{equation*}
Along the proof, we will use $C>0$ to denote various constants that may change from line to line, but which are independent of $\varepsilon>0$.

We prove the property by an iteration starting with $h^0_t\in X$ the only solution of $\Box_t h^0_t = \psi_t$ with initial condition $h^0_0=0$, which exists and is controlled as $\|h^0_t\|_X\leq C \|\psi_t\|_Y$ from Assumption \eqref{item-2}. Of course, it does not satisfy $h^0_t\in W_t$ a priori.

We define a map $F:X\to X$ as follows. Given $h_t$ solving $\Box_th_t-\psi_t \in V_t$ with initial condition $h_0 =0$, define 
$u$ as the unique solution of 
$$\Box_t u_t = \pi_{W_t}\Box_t \pi_{V_t}h_t, \quad u_0 = 0,$$
which, by Assumption \eqref{item-2}, \eqref{item-3} and \eqref{item-4} is controlled as 
\begin{align}
    \|u_t\|_X &\leq C\|\pi_{W_t}\Box_t \pi_{V_t}h_t\|_Y,&\text{using \eqref{item-2},}\nonumber\\
    &\leq C\varepsilon \|\pi_{V_t}h_t\|_X,&\text{using \eqref{item-4},}\label{eq:control u}\\
    &\leq C\varepsilon\|h_t\|_X,&\text{using \eqref{item-3}.}\nonumber
\end{align}
Define 
\begin{equation}\label{eq:def F}
    F(h_t):= \pi_{W_t}h_t + u_t,
\end{equation}
which by construction satisfies $F(h_0)=0$. Here $F(h_t)$ is a slight abuse of notation for $F((h_s)_{s\in I})_t$. The family $(F(h_t))_{t\in I}$ solves an equation of the desired type:
\begin{equation}\label{eq: box F(h)}
    \Box_tF(h_t) -\psi_t= \Box_t h_t - \Box_t \pi_{V_t}h_t +\pi_{W_t}\Box_t \pi_{V_t}h_t-\psi_t = \Box_t h_t-\psi_t - \pi_{V_t}\Box_t \pi_{V_t}h_t \in V_t.
\end{equation}
Now, on the one hand let us show that $\pi_{V_t}F(h_t)$ is smaller than $\pi_{V_t}h_t$ was  in $X$:
\begin{align*}
    \|\pi_{V_t}F(h_t)\|_X &=\|\pi_{V_t}u_t\|_X,&\text{by \eqref{eq:def F},}\\
    &\leq C\|u_t\|_X,&\text{ using \eqref{item-3},}\\
    &\leq C\varepsilon\|\pi_{V_t}h_t\|_X,&\text{by \eqref{eq:control u}}.
\end{align*}
On the other hand, since $\pi_{W_t}F(h_t) =\pi_{W_t}h_t+ \pi_{W_t}u_t$,
\begin{align*}
    \|\pi_{W_t}(F(h_t)-h_t)\|_X &= \|\pi_{W_t}u_t\|_X,&\text{by \eqref{eq:def F},}\\
    &\leq C\|u_t\|_X,&\text{ using \eqref{item-3},}\\
    &\leq C\varepsilon\|\pi_{V_t}h_t\|_X,&\text{by \eqref{eq:control u}}.
\end{align*}

Defining $h^{k+1}_t := F(h^k_t)$, we obtain a sequence of solutions to $\Box_th_t^k-\psi_t \in V_t$ satisfying $\|\pi_{V_t}h^{k+1}_t\|_X \leq C\varepsilon \|\pi_{V_t}h^{k}_t\|_X$ and $\|\pi_{W_t}(h^{k+1}_t-h^k_t)\|_X \leq C\varepsilon \|\pi_{V_t}h^{k}_t\|_X$. Choosing $\varepsilon>0$ small enough depending on $C>0$ implies that $(h_t^k)_k$ is a Cauchy sequence in the Banach space $X$. This fact alone not only gives the existence of a limit $h_t^{\infty}\in X$ of the sequence $(h_t^k)_k$ in $X$ but also guarantees that $h^\infty_t\in W_t$ (since $\pi_{V_t}(h^{\infty}_t)=0$) and $\Box_th_t^\infty-\psi_t \in V_t$ for each $t\in I$ as desired.
\end{proof}

The next discussion justifies and checks the properties assumed in the previous Proposition \ref{prop-flot-mod-obst} in the setting of the present paper. 

\begin{enumerate}
  \item Assumption \eqref{item-1} from Proposition \ref{prop-flot-mod-obst} is straightforward by definition of the function spaces $X$ and $Y$ and that of the operator $P_t$. 
	\item Assumption \eqref{item-2} from Proposition \ref{prop-flot-mod-obst} follows from standard existence and uniqueness results of solutions to inhomogeneous heat equations with initial condition in the setting of H\"older spaces which can be proved along the same lines as those of the proof of Proposition \ref{Schauder nos normes}. Notice that the constant $C$ from the corresponding Schauder estimate may depend on $T'$ here as highlighted in Remark \ref{rk-dep-T'}.
 
    \item Assumption \eqref{item-3} from Proposition \ref{prop-flot-mod-obst} is slightly trickier to check as it is essentially equivalent to showing that the norms $\|\cdot\|_X$ (respectively $\|\cdot\|_Y$) and $\|h_t\|^*_X:=\|\pi_{W_t}h_t\|_X+\|\pi_{V_t}h_t\|_X$ (respectively $\|h_t\|^*_Y:=\|\pi_{W_t}h_t\|_Y+\|\pi_{V_t}h_t\|_Y$) are equivalent:
    \begin{lemma}\label{equiv-norms-X}
    There exists $C>0$ such that $\|h_t\|^*_X\leq C\|h_t\|_X$ for all $(h_t)_{t\in I}\in X$ and $\|h_t\|^*_Y\leq C\|h_t\|_Y$ for all $(h_t)_{t\in I}\in Y$. Moreover, if $\pi_{V_t}h_t=:\sum_iL_i(t)\tilde{\oi}_i(t)+M(t)\tilde{g}(t)+N(t)\tilde{c}(t)$ then 
    \begin{equation*}
    \sum_i\|L_i\cdot\tilde{\oi}_i\|_X+\|M\cdot \tilde{g}\|_X+\|N\cdot \tilde{c}\|_X\leq C\|h_t\|_X, \quad (h_t)_{t\in I}\in X,
    \end{equation*}
     for some uniform positive constant $C$. The same holds if $(h_t)_{t\in I}\in Y$:
     \begin{equation*}
    \sum_i\|L_i\cdot\tilde{\oi}_i\|_Y+\|M\cdot \tilde{g}\|_Y+\|N\cdot \tilde{c}\|_Y\leq C\|h_t\|_Y, \quad (h_t)_{t\in I}\in Y.
    \end{equation*}

    \end{lemma}
    \begin{proof}
  We give a proof for the space $Y=C^{0,\alpha}_{\gamma,\sigma+2,T}$ only. The corresponding proof for the space $X$ is analogous.
  
  Thanks to [\eqref{est-inverse-Gram-pre}, Lemma \ref{est-inverse-Gram-pre}], 
  \begin{equation*}
  \begin{split}
\sum_i&\|L_i\cdot\tilde{\oi}_i\|_{C^0_{\gamma,\sigma+2,T}}+\|M\cdot \tilde{g}\|_{C^0_{\gamma,\sigma+2,T}}+\|N\cdot \tilde{c}\|_{C^0_{\gamma,\sigma+2,T}}\leq\\
& C\Big( \sum_i\|\overline{L_i}\cdot\tilde{\oi}_i\|_{C^0_{\gamma,\sigma+2,T}}+\|\overline{M}\cdot \tilde{g}\|_{C^0_{\gamma,\sigma+2,T}}+\|\overline{N}\cdot \tilde{c}\|_{C^0_{\gamma,\sigma+2,T}}\Big),
\end{split}
\end{equation*}
where the coefficients $\overline{L_i}$, $\overline{M}$ and $\overline{N}$ are defined by:
\begin{equation*}
\overline{L_i}(t):=\frac{\langle h(t),\tilde{\oi}_i(t)\rangle_{L^2(\tilde{g}(t))}}{\|\tilde{\oi}_{i}(t)\|^2_{L^2(\tilde{g}(t))}},\quad\overline{M}(t):=\frac{\langle h(t),\tilde{g}(t)\rangle_{L^2(\tilde{g}(t))}}{\|\tilde{g}(t)\|^2_{L^2(\tilde{g}(t))}},\quad \overline{N}(t):=\frac{\langle h(t),\tilde{c}(t)\rangle_{L^2(\tilde{g}(t))}}{\|\tilde{c}(t)\|^2_{L^2(\tilde{g}(t))}}.
\end{equation*}

  The following intermediate conclusion holds by invoking Lemmata \ref{lemma-easy-peasy}, \ref{lemma-easy-peasy-bis} and \ref{lemma-easy-peasy-ter}:
   \begin{equation}
  \begin{split}\label{intermed-gram-cont-proj}
\sum_i&\|L_i\cdot\tilde{\oi}_i\|_{C^0_{\gamma,\sigma+2,T}}+\|M\cdot \tilde{g}\|_{C^0_{\gamma,\sigma+2,T}}+\|N\cdot \tilde{c}\|_{C^0_{\gamma,\sigma+2,T}}\leq C\|h\|_{C^0_{\gamma,\sigma+2,T}}.
\end{split}
\end{equation}
As for the estimate involving the semi-norm $[\,\cdot\,]_{\alpha,t,x}$, we apply Proposition \ref{prop-holder-semi-norm-app-ker}  to $f:=L_i$ and $\tilde{\oi}_i$ and observe by the above estimate \eqref{intermed-gram-cont-proj} that it suffices to estimate $[L_i]_{\alpha,t,\rho}$ from above. To do so, we consider difference quotients of [\eqref{link-G-L-scal}, Section \ref{sec-gram-app-ker}] to get:

\begin{equation}\label{step-gram-1}
\begin{pmatrix}[L_1]_{\alpha,t,\rho} \\ [L_2]_{\alpha,t,\rho}
\\ \vdots \\ [M]_{\alpha,t,\rho} \\ [N]_{\alpha,t,\rho}
\end{pmatrix}
\leq G(t)^{-1}\begin{pmatrix}\left[\overline{L_1}\cdot\|\tilde{\oi}_1\|^2_{L^2(\tilde{g})}\right]_{\alpha,t,\rho}\\ \left[\overline{L_2}\cdot\|\tilde{\oi}_2\|^2_{L^2(\tilde{g})}\right]_{\alpha,t,\rho}\\\vdots\\\left[\overline{M}\cdot\|\tilde{g}\|^2_{L^2(\tilde{g})}\right]_{\alpha,t,\rho}\\\left[\overline{N}\cdot\|\tilde{c}\|^2_{L^2(\tilde{g})}\right]_{\alpha,t,\rho}\end{pmatrix}+G(t)^{-1}[G]_{\alpha,t,\rho}\begin{pmatrix}\sup_{t-\rho^2\,\leq\,t'\,\,\leq\,t}|L_1(t')| \\ \sup_{t-\rho^2\,\leq\,t'\,\,\leq\,t}|L_2(t')|
\\ \vdots \\ \sup_{t-\rho^2\,\leq\,t'\,\,\leq\,t}|M(t')| \\ \sup_{t-\rho^2\,\leq\,t'\,\,\leq\,t}|N(t')|\end{pmatrix},
\end{equation}
where the previous inequality is meant line by line and where the $[G]_{\alpha,t,\rho}$ denotes the matrix whose coefficients are given by the semi-norm $[\,\cdot\,]_{\alpha,t,x}$ of the coefficients of the Gram matrix $G$. 

On the one hand, Lemmata \ref{gal-est-h-holder} and \ref{gal-est-h-holder-g-c} ensure with $\sigma+2$ in lieu of $\sigma$ that:
\begin{equation}\label{step-gram-2}
\begin{pmatrix}\left[\overline{L_1}\cdot\|\tilde{\oi}_1\|^2_{L^2(\tilde{g})}\right]_{\alpha,t,\rho}\\ \left[\overline{L_2}\cdot\|\tilde{\oi}_2\|^2_{L^2(\tilde{g})}\right]_{\alpha,t,\rho}\\\vdots\\\left[\overline{M}\cdot\|\tilde{g}\|^2_{L^2(\tilde{g})}\right]_{\alpha,t,\rho}\\\left[\overline{N}\cdot\|\tilde{c}\|^2_{L^2(\tilde{g})}\right]_{\alpha,t,\rho}\end{pmatrix}\leq \gamma(t)^{-1}\|h\|_{C^{0,\alpha}_{\gamma,\sigma+2,T}}\begin{pmatrix}\rho^{2\alpha}\varepsilon(t)^{2-\sigma-2\alpha}\\\rho^{2\alpha}\varepsilon(t)^{2-\sigma-2\alpha}\\\vdots\\\vdots\\1\\1\end{pmatrix}.
\end{equation}
On the other hand, Proposition \ref{prop-easy-l2-proj-holder} gives:

\begin{equation}
\begin{split}\label{est-holder-Gram-matrix}
&[G]_{\alpha,t,\rho}=\\
&\rho^2\begin{pmatrix}
O(\varepsilon(t)^{4}) & O(\varepsilon(t)^8\delta(t)^{-4}) & \dots  &O(\varepsilon(t)^{4^-}\delta(t)^{2^+})&O(\varepsilon(t)^4) \\
    O(\varepsilon(t)^8\delta(t)^{-4}) & O(\varepsilon(t)^4)  &   \dots &O(\varepsilon(t)^{4^-}\delta(t)^{2^+}) &O(\varepsilon(t)^4)\\
    \vdots & \vdots & \ddots & \vdots&\vdots \\
    O(\varepsilon(t)^{4^-}\delta(t)^{2^+}) & O(\varepsilon(t)^{4^-}\delta(t)^{2^+}) & \dots & O(\varepsilon(t)^8\delta(t)^{-8})&0\\
    O(\varepsilon(t)^4)&O(\varepsilon(t)^4)&\dots&0&O(\varepsilon(t)^{4-}\delta(t)^{-4^-})\end{pmatrix}.
    \end{split}
\end{equation}
In particular, one can estimate as follows:
\begin{equation}
\begin{split}\label{est-holder-inv-Gram-matrix}
&G(t)^{-1}[G]_{\alpha,t,\rho}\begin{pmatrix}\sup_{t-\rho^2\,\leq\,t'\,\,\leq\,t}|L_1(t')| \\ \sup_{t-\rho^2\,\leq\,t'\,\,\leq\,t}|L_2(t')|
\\ \vdots \\ \sup_{t-\rho^2\,\leq\,t'\,\,\leq\,t}|M(t')| \\ \sup_{t-\rho^2\,\leq\,t'\,\,\leq\,t}|N(t')|\end{pmatrix}=\rho^2\gamma(t)^{-1}\|h\|_{C^{0,\alpha}_{\gamma,\sigma+2,T}}\begin{pmatrix}
O(\varepsilon(t)^{-2-\sigma})\\O(\varepsilon(t)^{-2-\sigma})\\\vdots\\ O(\varepsilon(t)^{2^{-}-\sigma}\delta(t)^{2^+})\\O(\varepsilon(t)^{2-\sigma})
\end{pmatrix}.
    \end{split}
\end{equation}

Combining \eqref{step-gram-1}, \eqref{step-gram-2} and \eqref{est-holder-inv-Gram-matrix} gives:
\begin{equation*}
\begin{split}
&\sup_{t\,\leq\,T}\gamma(t)\varepsilon(t)^{4}\sup_{\rho\,\leq\,2\delta(t)}\rho^{-2+\sigma}[L_i]_{\alpha,t,\rho}\leq \\
&C\|h\|_{C^{0,\alpha}_{\gamma,\sigma+2,T}}\sup_{t\,\leq\,T}\varepsilon(t)^{4}\sup_{\rho\,\leq\,2\delta(t)}\left((\rho/\varepsilon(t))^{2\alpha}\varepsilon(t)^{-2-\sigma}+1\right)\rho^{-2+\sigma}\\
&+C\|h\|_{C^{0,\alpha}_{\gamma,\sigma+2,T}}\sup_{t\,\leq\,T}\varepsilon(t)^{4}\sup_{\rho\,\leq\,2\delta(t)}\left(\rho^2\varepsilon(t)^{-2-\sigma}\right)\rho^{-2+\sigma}\\
&\leq C\|h\|_{C^{0,\alpha}_{\gamma,\sigma+2,T}}.
\end{split}
\end{equation*}
The corresponding estimates for $\|M\cdot \tilde{g}\|_{C^{0,\alpha}_{\gamma,\sigma+2,T}}$ and $\|N\cdot \tilde{c}\|_{C^{0,\alpha}_{\gamma,\sigma+2,T}}$ can be proved along the same lines.
    \end{proof}
    \item Assumption \eqref{item-4} from Proposition \ref{prop-flot-mod-obst} essentially follows from the estimates from Section \ref{sec-first-approx} and the previous Lemma \ref{equiv-norms-X}:
    \begin{lemma}
    There exists $\varepsilon(T)>0$ independent of $T'$ such that for all $h\in X$,
    \begin{equation*}
    \left\|\pi_{W_t}\left(\left(\partial_t-\Delta_{L,\tilde{g}(t)}-2\Lambda\right)\pi_{V_t}h_t\right)\right\|_{Y}\leq o_T(1)\|\pi_{V_t}h_t\|_{X},
    \end{equation*}
    where $\lim_{T\rightarrow-\infty }\varepsilon(T)=0.$
    \end{lemma}
    \begin{proof}
    We first check that the following holds.
    
     If $\pi_{V_t}h_t=:L_i(t)\tilde{\oi}_i(t)+M(t)\tilde{g}(t)+N(t)\tilde{c}(t)$ then for each index $i$:
    \begin{equation}
    \begin{split}\label{list-est-X-Y}
    \left\|\pi_{W_t}\left(\left(\partial_t-\Delta_{L,\tilde{g}(t)}-2\Lambda\right)L_i(t)\tilde{\oi}_i(t)\right)\right\|_{Y}&=\left\|L_i(t)\pi_{W_t}\left(\left(\partial_t-\Delta_{L,\tilde{g}(t)}\right)\tilde{\oi}_i(t)\right)\right\|_{Y}\\
    &\leq o_T(1)\|L_i\cdot\tilde{\oi}_i\|_{X},\\
     \left\|\pi_{W_t}\left(M(t)\left(\partial_t-\Delta_{L,\tilde{g}(t)}-2\Lambda\right)\tilde{g}(t)\right)\right\|_{Y}&=\left\|M(t)\pi_{W_t}\left(\partial_t\tilde{g}(t)\right)\right\|_{Y}\\
     &\leq o_T(1)\|M\cdot\tilde{g}\|_{X},\\
      \left\|\pi_{W_t}\left(N(t)\left(\partial_t-\Delta_{L,\tilde{g}(t)}-2\Lambda\right)\tilde{c}(t)\right)\right\|_{Y}&=\left\|N(t)\pi_{W_t}\left(\left(\partial_t-\Delta_{L,\tilde{g}(t)}-4\right)\tilde{c}(t)\right)\right\|_{Y}\\
      &\leq o_T(1)\|N\cdot\tilde{c}\|_{X}.
    \end{split}
    \end{equation}
  Let us prove the first estimate involving the tensor $\tilde{\oi}_i(t)$.
  By definition of $Y:=C^{0,\alpha}_{\gamma,\sigma+2,T}$, Lemma \ref{estimates variations et laplacien o1} and since $\delta<2/3$ by \eqref{cond-param},
  \begin{equation*}
  \begin{split}
  \left\|L_i\cdot\left(\partial_t-\Delta_{L,\tilde{g}}\right)\tilde{\oi}_i\right\|_{C^0_{\gamma,\sigma+2,T}}&\leq C\sup_{t\,\leq \,T}\gamma(t)|L_i(t)|\left(\delta(t)^{2+\sigma}+\varepsilon(t)^4\delta(t)^{\sigma-4}\right)\\
  &\leq C\sup_{t\,\leq \,T}\gamma(t)\delta(t)^{2+\sigma}|L_i(t)|,
  \end{split}
  \end{equation*}
  for some uniform positive constant $C$. Moreover, according to Lemmata \ref{lemma-semi-norm-diff-scales}, \ref{lemma-prod-rule-semi-norm}, Lemma \ref{estimates variations et laplacien o1} and Propositions \ref{prop-semi-norm-lap-app-ker} and \ref{prop-holder-time-app-ker-oi},
  \begin{equation}
  \begin{split}\label{est-above-Y}
&\sup_{M\times(-\infty,T]}\gamma(t)\rho(x,t)^{2+\sigma}[L_i\cdot\left(\partial_t-\Delta_{L,\tilde{g}}\right)\tilde{\oi}_i]_{\alpha,x,t}\\
&\leq C\sup_{M\times(-\infty,T]}\gamma(t)\rho(x,t)^{2+\sigma}\left(\sup_{P_{\rho}(x,t)}[L_i]_{\alpha,t,\rho}|\left(\partial_t-\Delta_{L,\tilde{g}}\right)\tilde{\oi}_i|_{\tilde{g}(t)}\right)\\
&+\sup_{M\times(-\infty,T]}\gamma(t)\rho(x,t)^{2+\sigma}\left(\sup_{t-\rho(x,t)^2\,\leq\, s\,\leq \,t}|L_i(s)|[\left(\partial_t-\Delta_{L,\tilde{g}}\right)\tilde{\oi}_i]_{\alpha,x,t}\right)\\
&\leq C\sup_{t\,\leq\, T}\gamma(t)\sup_{\rho\,\leq\,2\delta(t)}\left(\varepsilon(t)^4\rho^{-2+\sigma}+\varepsilon(t)^4\delta(t)^{-4+\sigma}\right)[L_i]_{\alpha,t,\rho}\\
&\quad+C\sup_{t\,\leq\, T}\gamma(t)\sup_{\rho\,\leq\,2\delta(t)}\left(\varepsilon(t)^{4^--2\alpha}\rho^{-2^-+\sigma+2\alpha}+\varepsilon(t)^{4^-}\delta(t)^{-2^{-}+\sigma}\right)\sup_{t-\rho(x,t)^2\,\leq\, s\,\leq \,t}|L_i(s)|\\
&\leq C\sup_{t\,\leq\, T}\gamma(t)\varepsilon(t)^{4}\delta(t)^{-4+\sigma}\sup_{\rho\,\leq\,2\delta(t)}[L_i]_{\alpha,t,\rho}\\
&\quad+C\left(\varepsilon(T)^{2}+\varepsilon(T)^{4^--\sigma}\delta(T)^{-2^{-}+\sigma}\right)\sup_{t\,\leq\,T}\gamma(t)\varepsilon(t)^{\sigma}|L_i(t)|,
\end{split}
\end{equation}
where we have used that $\delta<2/3<1$ in the penultimate line together with $(4-\sigma)\delta>2-\sigma$ since $\delta>1/2$ thanks to \eqref{cond-param} in the last line.
  
  Now,
  \begin{equation}\label{lower-bd-X}
  \|L_i\cdot \tilde{\oi}_i\|_{X}\geq \|L_i\cdot \tilde{\oi}_i\|_{C^{0}_{\gamma,\sigma,T}}\geq C^{-1}\sup_{t\,\leq\,T}\gamma(t)\varepsilon(t)^{\sigma}|L_i(t)|,
 \end{equation} 
 where we have chosen a point in $(N,g_{\zeta(t)})$ such that $|\oi_i(\zeta(t))|_{g_{\zeta(t)}}\geq C^{-1}$ uniformly in time at that point. In particular, we get as an intermediate result:
 \begin{equation}\label{first-upp-X-Y}
\left\|L_i\cdot\left(\partial_t-\Delta_{L,\tilde{g}}\right)\tilde{\oi}_i\right\|_{C^0_{\gamma,\sigma+2,T}}\leq C\varepsilon(T)^{\delta(2+\sigma)-\sigma}\|L_i\cdot \tilde{\oi}_i\|_{X}\leq o_T(1)\|L_i\cdot \tilde{\oi}_i\|_{X},
\end{equation}
where we have used that $\delta(2+\sigma)-\sigma>0$ thanks to \eqref{cond-param}.

Now, 
    \begin{equation*}
    \begin{split}
  \|L_i\cdot \tilde{\oi}_i\|_{X}&\geq \sup_{M\times(-\infty,T]}\gamma(t)\rho(x,t)^{\sigma}[L_i\cdot\tilde{\oi}_i]_{\alpha,x,t}\geq \sup_{M\times(-\infty,T]}\gamma(t)\rho(x,t)^{\sigma}[L_i\cdot\tilde{\oi}_i]_{\alpha,t,\rho}\\
  &\geq \sup_{t\,\leq \,T}\gamma(t)\sup_M\rho^{\sigma}[L_i]_{\alpha,t,\rho}|\tilde{\oi}_i(t)|_{\tilde{g}(t)}-\sup_{t\,\leq \,T}\gamma(t)\sup_M\rho^{\sigma}[\tilde{\oi}_i]_{\alpha,t,\rho}\sup_{t-\rho^2\,\leq \,t'\,\leq\,t}|L_i(t')|\\
  &\geq C^{-1}\sup_{t\,\leq \,T}\gamma(t)\varepsilon(t)^{\sigma}\sup_{\rho\,\leq\, 2\delta(t)}[L_i]_{\alpha,t,\rho}-\sup_{t\,\leq \,T}\gamma(t)\sup_M\rho^{\sigma}[\tilde{\oi}_i]_{\alpha,t,\rho}\sup_{t-\rho^2\,\leq \,t'\,\leq\,t}|L_i(t')|\\
  &\geq C^{-1}\sup_{t\,\leq \,T}\gamma(t)\varepsilon(t)^{\sigma}\sup_{\rho\,\leq\, 2\delta(t)}[L_i]_{\alpha,t,\rho}-C\varepsilon(T)^{2+\sigma}\sup_{t\,\leq\,T}\gamma(t)\varepsilon(t)^{\sigma}|L_i(t)|\\
  &\geq C^{-1}\sup_{t\,\leq \,T}\gamma(t)\varepsilon(t)^{\sigma}\sup_{\rho\,\leq\, 2\delta(t)}[L_i]_{\alpha,t,\rho}-C\varepsilon(T)^{2+\sigma} \|L_i\cdot \tilde{\oi}_i\|_{X}.
    \end{split}
   \end{equation*} 
   Here we have invoked [\eqref{semi-holder-oi}, proof of Proposition \ref{prop-est-semi-norm-time-der-g}] in the penultimate line and \eqref{lower-bd-X} in the last line.
   
   In particular, we get:
    \begin{equation*}
    \begin{split}
  \|L_i\cdot \tilde{\oi}_i\|_{X}&\geq C^{-1}\sup_{t\,\leq \,T}\gamma(t)\varepsilon(t)^{\sigma}\sup_M[L_i]_{\alpha,t,\rho}.
    \end{split}
   \end{equation*} 
   Thanks to \eqref{est-above-Y}, we therefore obtain in combination with the previous lower bound on $  \|L_i\cdot \tilde{\oi}_i\|_{X}$:
    \begin{equation}
    \begin{split}\label{last-upp-X-Y}
\sup_{M\times(-\infty,T]}&\gamma(t)\rho(x,t)^{2+\sigma}[L_i\cdot\left(\partial_t-\Delta_{L,\tilde{g}}\right)\tilde{\oi}_i]_{\alpha,x,t}\leq \\
&C\varepsilon(T)^{4-\sigma}\delta(T)^{-4+\sigma}\sup_{t\,\leq \,T}\gamma(t)\varepsilon(t)^{\sigma}\sup_M[L_i]_{\alpha,t,\rho} +o_T(1)\|L_i\cdot \tilde{\oi}_i\|_{X}\\
&\leq o_T(1)\|L_i\cdot \tilde{\oi}_i\|_{X}.
    \end{split}
   \end{equation} 
   Estimates \eqref{first-upp-X-Y} and \eqref{last-upp-X-Y} give the first desired estimate of \eqref{list-est-X-Y}.
   
    The proof of the remaining estimates follow the same lines.
    
    On the basis of \eqref{list-est-X-Y}, one concludes thanks to Lemma \ref{equiv-norms-X} and the triangular inequality for the norm $\|\cdot\|_Y$.
    
    \end{proof}
\end{enumerate}

\newpage
\section{Some nonlinear estimates}\label{sec:lin curv}

\begin{lemma}\label{gal-lemma-diff-curv}
Let $g_1$ and $g_2$ be two smooth Riemannian metrics on a smooth manifold $M$. Then:
\begin{equation*}
\Rm(g_2)-\Rm(g_1)=g_2^{-1}\ast\nabla^{g_1,2}(g_2-g_1)+g_2^{-1}\ast g_2^{-1}\ast \nabla^{g_1}(g_2-g_1)\ast \nabla^{g_1}(g_2-g_1),
\end{equation*}
if $\Rm(g_i)$, $i=1,2$, is interpreted as a $(3,1)$-tensor.
\end{lemma}
A proof of this lemma can be found in [Chapter 2, \cite{Cho-Lu-Ni-Boo}].

\begin{lemma}\label{app-C-monster}
For $g$ a Riemannian metric and $k$ a symmetric $2$-tensor such that $|k|_g\leq 1/2$, defined as in \eqref{prop-Q-0}:
\begin{equation*}
	\begin{split}
	Q_{g}(k)&=g^{-1}\ast k\ast \nabla^{g,\,2}k+g^{-1}\ast g^{-1}\ast \nabla^{g}k\ast\nabla^{g}k+g^{-1}\ast \Rm(g)\ast k\ast k.
	\end{split}
	\end{equation*}

Then one has:
	\begin{equation}\label{prop-Q}
	|Q_{g}(k)|_{g}\leq C\left(|k|_{g}|\nabla^{g,2}k|_{g}+|\nabla^{g}k|^2_{g}+|\Rm(g)|_{g}|k|^2_{g}\right).
	\end{equation}
	Moreover, for a couple $(g_i,k_i)$, $i=1,2$ as above, then:
	 \begin{equation}\label{prop-Q-diff}
	 \begin{split}
	Q_{g_2}(k_2)-Q_{g_1}(k_1)&=g_2^{-1}\ast (g_2-g_1)\ast g_1^{-1}\ast k_2\ast \nabla^{g_2,2}k_2+g_1^{-1}\ast (k_2-k_1)\ast \nabla^{g_2,2}k_2\\
	&\quad+g_1^{-1}\ast k_1\ast\nabla^{g_1,2}(k_2-k_1)\\
	&\quad+g_1^{-1}\ast k_1^{*2}\ast \left[g_2^{-1}\ast\nabla^{g_1,2}(g_2-g_1)+g_2^{-1}\ast g_2^{-1}\ast\nabla^{g_1}(g_2-g_1)^{*2} \right]\\
	&\quad+g_1^{-1}\ast g_2^{-1}\ast \nabla^{g_1}(g_2-g_1)\ast \nabla^{g_1}k_1\ast k_1\\
	&\quad+g_2^{-1}\ast (g_2-g_1)\ast g_1^{-1}\ast \left(g_2^{-1}+g_1^{-1}\right)\ast \nabla^{g_2}k_2\ast\nabla^{g_2}k_2\\
	&\quad+g_1^{-1}\ast g_1^{-1}\ast g_2^{-1}\ast \nabla^{g_1}(g_2-g_1)\ast k_2\ast(\nabla^{g_2}k_2+\nabla^{g_1}k_2)\\
	&\quad +g_1^{-1}\ast g_1^{-1}\ast \nabla^{g_1}(k_2-k_1)\ast \nabla^{g_1}(k_2+k_1)\\
	&\quad+\left(g_2^{-1}\ast (g_2-g_1)\ast g_1^{-1}\ast\Rm(g_2)+g_1^{-1}\ast g_2^{-1}\ast\nabla^{g_1,2}(g_2-g_1)\right)\ast k_2\ast k_2\\
	&\quad+g_1^{-1}\ast g_2^{-1}\ast g_2^{-1}\ast \nabla^{g_1}(g_2-g_1)\ast \nabla^{g_1}(g_2-g_1)\ast k_2\ast k_2\\
	&\quad+g_1^{-1}\ast \Rm(g_1)\ast (k_2-k_1)\ast (k_2+k_1).
	\end{split}
	\end{equation}
	In particular, if $|g_2-g_1|_{g_1}\leq 1/2$ then there exists $C>0$ such that:
	\begin{equation*}
	\begin{split}
	&|Q_{g_2}(k_2)-Q_{g_1}(k_1)|_{g_1}\leq |k_2-k_1|_{g_1}\left(| \nabla^{g_2,2}k_2|_{g_1}+|\Rm(g_1)|_{g_1}(|k_2|_{g_1}+|k_1|_{g_1})\right)\\
	&+|\nabla^{g_1}(k_2-k_1)|_{g_1}|\nabla^{g_1}(k_2+k_1)|_{g_1}+|k_1|_{g_1}|\nabla^{g_1,2}(k_2-k_1)|_{g_1}\\
	&+\left(|\nabla^{g_1,2}(g_2-g_1)|_{g_1}+|\nabla^{g_1}(g_2-g_1)|_{g_1}^{2}+\left(|\Rm(g_1)|_{g_1}+|\Rm(g_2)|_{g_2}\right)|g_2-g_1|_{g_1}\right)\left(|k_1|^2_{g_1}+|k_2|^2_{g_1}\right)\\
	&+|\nabla^{g_1}(g_2-g_1)|_{g_1}\left(|\nabla^{g_1}k_1|_{g_1}+|\nabla^{g_1}k_2|_{g_1}\right)(|k_1|_{g_1}+|k_2|_{g_1})\\
	&+| g_2-g_1|_{g_1} \left(|\nabla^{g_2}k_2|_{g_1}^2+|k_2|_{g_1} |\nabla^{g_2,2}k_2|_{g_1}\right).
	\end{split}
	\end{equation*}
	\end{lemma}
	
	We end this section by recalling how to compare scalar products of two tensors defined with respect to two equivalent Riemannian metrics:
\begin{lemma}\label{conv-lemma-diff-scal-prod}
If $S$ and $T$ are two tensors of the same type on $M^n$ and if $g_1$ and $g_2$ are two Riemannian metrics such that $|g_2-g_1|_{g_1}\leq \frac{1}{2}$ then there exists a constant $c(n)>0$ such that:
\begin{equation*}
\left|\langle S,T\rangle_{g_2}-\langle S,T\rangle_{g_1}\right|\leq c(n)|g_2-g_1|_{g_1}|S|_{g_1}|T|_{g_1},
\end{equation*}
and,
\begin{equation*}
\left|\langle S,T\rangle_{L^2(g_2)}-\langle S,T\rangle_{L^2(g_1)}\right|\leq c(n)\int_M|g_2-g_1|_{g_1}|S|_{g_1}|T|_{g_1}\,d\mu_{g_1}.
\end{equation*}
\end{lemma}

We end with the following straightforward but crucial lemma:
\begin{lemma}\label{easy-time-der-abs-2-tensors-prelim}
Let $(M,g(t))_t$ be a one-parameter $C^1$ family of $C^2$ Riemannian metrics on a closed manifold $M$ and let $(S(t))_t$ and $(T(t))_{t}$ be two families of $C^2$ tensors on $M$. Then:
 \begin{equation*}
\begin{split}
 \frac{d}{dt}&\langle S(t),T(t)\rangle_{L^2(g(t))}=\int_M\left\langle \partial_tS(t),T(t)\right\rangle_{g(t)}\,d\mu_{g(t)}\\
&+ \int_M\left\langle S(t),\left(\partial_t+\frac{\tr_{g(t)}\partial_tg(t)}{2}\operatorname{Id}-\Sym(\partial_tg(t)\circ \cdot)\right)T(t)\right\rangle_{g(t)}\,d\mu_{g(t)},
 \end{split}
 \end{equation*}
 {where in coordinates, one has $\Sym(a\circ b)_{ij} = g^{pq}a_{ip}b_{qj} + g^{pq}b_{ip}a_{qj}$.}
 \end{lemma}

\newpage
\section{H\"older estimates}\label{sec: annexe holder}
\begin{defn}\label{defn-semi-norm-holder-fct}
Let $f\in(-\infty,T)\rightarrow \RR$ be a function and let $\alpha\in(0,1/2)$. The H\"older semi-norm $[f]_{\alpha,t,r}$ is defined as:
\begin{equation*}
[f]_{\alpha,t,r}:=r^{2\alpha}\sup_{t'\neq t\,,\,t-r^2\,\leq\,t'\,\leq \,t}\frac{|f(t')-f(t)|}{|t'-t|^{\alpha}},\quad r\in\left(\inf_{x\in M}\rho_0(x,t),\sup_{x\in M}\rho_0(x,t)\right),
\end{equation*}
where $\rho_0(x,t):=\iota_0(\varepsilon(t)+r_o(x))$ as defined in Section \ref{subsec-holder}.
\end{defn}
A straightforward application of the mean value theorem gives:
\begin{lemma}\label{lemma-basic-interp}
Let $f\in(-\infty,T)\rightarrow \RR$ be a $C^1$ function and let $\alpha\in(0,1/2)$. Then for $r\in(\inf_{ M}\rho_0(t),\sup_{M}\rho_0(t))$,
\begin{equation*}
[f]_{\alpha,t,r}\leq r^2\sup_{t-r^2\,\leq\,\tau\,\leq \,t}|f'(\tau)|.
\end{equation*}
\end{lemma}
Observe also that we have the weighted interpolation:
\begin{equation*}
[f]_{\alpha,t,r}\leq \left(2\sup_{t-r^2\,\leq\,t'\,\leq\,t}|f(t')|\right)^{1-\alpha}\left(r^2\sup_{t-r^2\,\leq\,\tau\,\leq \,t}|f'(\tau)|\right)^{\alpha}.
\end{equation*}

\begin{defn}\label{defn-semi-norm-holder-tensor}
We define similarly a H\"older semi-norm at time $t\leq T$ for tensors at a point $x\in M$:
\begin{equation*}
[h]_{\alpha,t,r}(x):=r^{2\alpha}\sup_{t'\neq t\,,\,t-r^2\,\leq\,t'\,\leq \,t}\frac{|h(x,t')-h(x,t)|_{\tilde{g}(t)}}{|t'-t|^{\alpha}}.
\end{equation*}
\end{defn}
For ease of reading, we will omit the dependency on $x$ from now on.
Observe that if $h$ is time-differentiable then:
\begin{equation}\label{semi-holder-diff-tens}
[h]_{\alpha,t,r}(x)\leq r^{2}\sup_{t-r^2\,\leq\,\tau\,\leq \,t}|\partial_th(x,\tau)|_{\tilde{g}(t)}.
\end{equation}

We also compare the semi-norm $[\,\cdot\,]_{\alpha,t,r}$ for different scales:

\begin{lemma}\label{lemma-semi-norm-diff-scales}
For $r>0$ and $\rho>0$, $t\leq T$ and a tensor $S$ on $M\times(-\infty,T]$,
\begin{equation*}
[S]_{\alpha,t,r}(x)\leq \max\left\{1\,;\,\left(\frac{r}{\rho}\right)^{2\alpha}\right\}\left(2\sup_{t-r^2\leq\,s\,\leq\, t}|S(x,s)|_{\tilde{g}(t)}+[S]_{\alpha,t,\rho}(x)\right).
\end{equation*}
\end{lemma}
\begin{proof}
If $r\leq \rho$ then $[S]_{\alpha,t,r}(x)\leq [S]_{\alpha,t,\rho}(x)$ by definition of the semi-norm $[\,\cdot\,]_{\alpha,t,r}$. If $r>\rho$, let us split the supremum over $t'\in(t-r^2,t-\rho^2]$ and $(t-\rho^2,t]$ and write:
\begin{equation*}
\begin{split}
[S]_{\alpha,t,r}&\leq r^{2\alpha}\sup_{t'\in(t-\rho^2,\,t]}\frac{|S(x,t')-S(x,t)|_{\tilde{g}(t)}}{|t'-t|^{\alpha}}+ r^{2\alpha}\sup_{t'\in(t-r^2,\,t-\rho^2]}\frac{|S(x,t')-S(x,t)|_{\tilde{g}(t)}}{|t'-t|^{\alpha}}\\
&\leq \left(\frac{r}{\rho}\right)^{2\alpha}[S]_{\alpha,t,\rho}(x)+r^{2\alpha}\sup_{t'\in(t-r^2,\,t-\rho^2]}\frac{|S(x,t')-S(x,t)|_{\tilde{g}(t)}}{\rho^{2\alpha}},
\end{split}
\end{equation*}
which leads to the desired estimate by invoking the triangular inequality in the last line.

\end{proof}
We gather some results on the use of the previously defined H\"older semi-norms $[\,\cdot\,]_{\alpha,t,r}$ whose proofs are omitted.

\begin{lemma}\label{lemma-holder-prod}
There exists $C>0$ such that for $t\leq T$,
\begin{equation*}
\begin{split}
[S\,\ast\, T]_{\alpha,t,r}(x)&\leq C[S]_{\alpha,t,r}(x)\sup_{t-r^2\,\leq\,t'\,\leq \,t}|T(x,t')|_{\tilde{g}(t)}+C\sup_{t-r^2\,\leq\,t'\,\leq \,t}|S(x,t')|_{\tilde{g}(t)}[T]_{\alpha,t,r}(x)\\
&\quad+Cr^2\sup_{t-r^2\,\leq\,t'\,\leq \,t}|\partial_t\tilde{g}(x,t')|_{\tilde{g}(t)}|S(x,t')|_{\tilde{g}(t)}|T(x,t')|_{\tilde{g}(t)}.
\end{split}
\end{equation*}
\end{lemma}

The last term on the righthand side of Lemma \ref{lemma-holder-prod} will turn out to be negligible in our estimates.

Let us estimate the $[\,\cdot\,]_{\alpha,x,t}$-semi-norm (introduced in Section \ref{subsec-holder}) of a product $\lambda\cdot T$ in terms of the $[\,\cdot\,]_{\alpha,t,r}$-semi-norm of time-dependent function $\lambda$ and the $[\,\cdot\,]_{\alpha,x,t}$-semi-norm of a tensor $T$:
\begin{lemma}\label{lemma-prod-rule-semi-norm}
There exists $C>0$ such that for $t\leq T$, $\lambda:(-\infty,T]\rightarrow\RR$ and $S$ a tensor on $M\times(-\infty,T]$:
\begin{equation*}
\begin{split}
[\lambda\cdot S]_{\alpha,x,t}\leq C[\lambda]_{\alpha,t,\rho}\sup_{P_{\rho}(x,t)}|S|_{\tilde{g}(t)}+C\sup_{t-\rho_0(x,t)^2\,\leq\, s\,\leq \,t}|\lambda(s)|[S]_{\alpha,x,t}.
\end{split}
\end{equation*}
\end{lemma}
Observe the following straightforward interpolation lemma:
\begin{lemma}\label{lemma-interpo-semi-norm}
There exists $C>0$ such that for $t\leq T$ and any $C^1_{loc}$-tensor $S$ on $M\times(-\infty,T]$,
\begin{equation*}
[S]_{\alpha,x,t}\leq C\rho(x,t)\sup_{P_{\rho}(x,t)}|\nabla^{\tilde{g}(t)}S|_{\tilde{g}(t)}+C\rho(x,t)^2\sup_{P_{\rho}(x,t)}|\partial_{\tau}S|_{\tilde{g}(t)}.
\end{equation*}
\end{lemma}

The semi-norm $[h]_{\alpha,t,r}$ also shows up naturally when one estimates a H\"older semi-norm of $L^2$-scalar products:
\begin{lemma}\label{lemma-scal-prod-holder}
If $S(t)$ and $T(t)$ are two tensors of the same type for each $t\leq T$ then there exists $C>0$ such that for each $t\leq T$ and $r>0$:
\begin{equation*}
\begin{split}
\left[\langle S(\cdot),T(\cdot)\rangle_{L^2(\tilde{g}(\cdot))}\right]_{\alpha,t,r}&\leq C\int_M[S]_{\alpha,t,r}\sup_{t-r^2\,\leq\,t'\,\leq \,t}|T(t')|_{\tilde{g}(t)}+\sup_{t-r^2\,\leq\,t'\,\leq \,t}|S(t')|_{\tilde{g}(t)}[T]_{\alpha,t,r}\,d\mu_{\tilde{g}(t)}\\
&\quad+C\int_Mr^2\sup_{t-r^2\,\leq\,t'\,\leq \,t}|\partial_t\tilde{g}(t')|_{\tilde{g}(t)}|S(t')|_{\tilde{g}(t)}|T(t')|_{\tilde{g}(t)}\,d\mu_{\tilde{g}(t)}.
\end{split}
\end{equation*}
\end{lemma}

The next lemma estimates the H\"older semi-norm $[\cdot]_{\alpha,t,r}$ of the first and second covariant derivatives of time-dependent tensor along a one-parameter family of metrics $\tilde{g}(t)$.

\begin{lemma}\label{lemma-holder-cov-tensor}
Let $(T(t))_{t\in I}$ be a one-parameter family of tensors on $M$. There exists $C$ such that for $[t-r^2,t]\subset I$,
\begin{equation*}
\begin{split}
[\nabla^{\tilde{g}}T]_{\alpha,t,r}&\leq Cr^2\sup_{t-r^2\,\leq\,t'\,\leq \,t}|\nabla^{\tilde{g}(t)}\partial_t\tilde{g}(t')|_{\tilde{g}(t)}|T(t')|_{\tilde{g}(t)}\\
&\quad+r^{2\alpha}\sup_{t'\neq t\,,\,t-r^2\,\leq\,t'\,\leq \,t}\frac{|\nabla^{\tilde{g}(t)}(T(t')-T(t))|_{\tilde{g}(t)}}{|t'-t|^{\alpha}}.
\end{split}
\end{equation*}
Moreover,
\begin{equation*}
\begin{split}
[\nabla^{\tilde{g},2}T]_{\alpha,t,r}&\leq Cr^2\sup_{t-r^2\,\leq\,t'\,\leq \,t}|\nabla^{\tilde{g}(t),2}\partial_t\tilde{g}(t')|_{\tilde{g}(t)}|T(t')|_{\tilde{g}(t)}\\
&\quad+Cr^4\sup_{t-r^2\,\leq\,t'\,\leq \,t}|\nabla^{\tilde{g}(t)}\partial_t\tilde{g}(t')|^2_{\tilde{g}(t)}|T(t')|_{\tilde{g}(t)}\\
&\quad+Cr^2\sup_{t-r^2\,\leq\,t'\,\leq \,t}|\nabla^{\tilde{g}(t)}\partial_t\tilde{g}(t')|_{\tilde{g}(t)}|\nabla^{\tilde{g}(t)}T(t')|_{\tilde{g}(t)}\\
&\quad+r^{2\alpha}\sup_{t'\neq t\,,\,t-r^2\,\leq\,t'\,\leq \,t}\frac{|\nabla^{\tilde{g}(t),2}(T(t')-T(t))|_{\tilde{g}(t)}}{|t'-t|^{\alpha}}.
\end{split}
\end{equation*}
\end{lemma}
Finally, the following lemma estimates the difference of the scalar products of localized and unlocalized tensors on the bubble: 
\begin{lemma}\label{lemma-diff-compact-scal-prod}
Let $S(t)$ and $T(t)$ two tensors on $M$ of the same type for each $t\leq T\leq 0$. Then,
\begin{equation*}
\begin{split}
&\langle \chi_{\delta(t)} S(t),\chi_{\delta(t)} T(t)\rangle_{L^2(\tilde{g}(t))}-\langle  S(t),T(t)\rangle_{L^2(g_{\zeta(t)})}=\left\langle(1-\chi_{\delta(t)}^2)S(t),T(t)\right\rangle_{L^2(g_{\zeta(t)})}\\
 &+\int_{M}\left(\left\langle \chi_{\delta(t)}S(t),\chi_{\delta(t)}T(t)\right\rangle_{\tilde{g}(t)}\frac{d\mu_{\tilde{g}(t)}}{d\mu_{g_{\zeta(t)}}}-\left\langle\chi_{\delta(t)}S(t),\chi_{\delta(t)}T(t)\right\rangle_{g_{\zeta(t)}}\right)\,d\mu_{\tilde{g}(t)}.
\end{split}
\end{equation*}
\end{lemma}
\newpage
\section{Notations}\label{sec:notations}

\begin{tabular}{l l}

\textbf{Symbol} & \textbf{Meaning} \\
$\operatorname{Sym}^2(M)$& Symmetric $2$-tensors on a manifold $M$.\\
$\operatorname{Sym}^2_0(M)$& Traceless symmetric $2$-tensors on a Riemannian manifold $(M,g)$.\\
$d\alpha$ & Exterior derivative of a form $\alpha$. \\
$B_g(p,r)$ & Geodesic ball centered at a point $p$ with radius $r$ in metric $g$. \\
$S_g(p,r)$ & Geodesic sphere centered at a point $p$ with radius $r$ in metric $g$. \\
$|\cdot|_g$, $\langle\cdot,\cdot\rangle_g$ & Norm and scalar product on tensors by metric $g$. \\
$\langle\cdot,\cdot\rangle_{L^2(g)}$ & $L^2$-scalar product induced by metric $g$. \\
$\inj_g(p)$ & Injectivity radius at point $p$ in metric $g$. \\
$d\mu_{g}$ & Riemannian measure associated to metric $g$. \\
$\vol_g(A)$ & Volume of a subset $A$ in metric $g$. \\
$d\sigma_{A}$ & $(n-1)$-Hausdorff measure restricted to subset $A$. \\
$\ast_g$ & Hodge star operator associated to metric $g$. \\
$\Lambda^2$& $2$-forms on a manifold $M$.\\
$\Lambda^{+}$, $\Lambda^{-}$ (sometimes $\Lambda^{\pm}_g$)&   selfdual and anti-selfdual $2$-forms on $(M^4,g)$.\\
$\nabla^g$ & Levi-Civita connection of metric $g$. \\
$\nabla^{\ast}$ & $L^2(g)$-adjoint of $\nabla^g$ (metric omitted). \\
$\nabla^{g,\,k}$ & Covariant derivatives of order $k$ by metric $g$. \\
$\Hess_gf$, $\nabla^{g,\,2}f$ & Hessian of function $f$ with respect to metric $g$. \\
$\mathcal{L}_X(T)$ & Lie derivative of tensor $T$ along vector field $X$. \\
$\Rm(g)$ & Curvature operator of metric $g$. \\
$\mathbf{R}^{+}(g)$, $\mathbf{R}^{-}(g)$ & (Anti)-selfdual part of curvature tensor of metric $g$. \\
$W(g)$, $W^{+}(g)$, $W^{-}(g)$ & Weyl tensor and its (anti)-selfdual parts of metric $g$. \\
$\Ric(g)$ & Ricci curvature of metric $g$. \\
$\R_g$ & Scalar curvature of metric $g$. \\
$\Delta_g$, $\Delta_{L,g}$ & Rough and Lichnerowicz Laplacians associated to metric $g$. \\
$\mathcal{H}^k_{L^2(g)}$ & Harmonic $k$-forms on $(M,g)$ in $L^2(g)$.\\
$\mathcal{H}^{\pm}_{L^2(g)}$ & (Anti-)Selfdual harmonic $2$-forms on $(M4,g)$ in $L^2(g)$.\\
$|A|\leq C B^{a\pm}$, $A = O(B^{a\pm})$ & Notation for ``for any $c>0$, there exists $C(c)>0$ so that $|A|\leq C B^{a\pm c}$''.
\end{tabular}

\newpage
\bibliographystyle{amsalpha}

\bibliography{ref-desing}
\end{document}